\documentclass[10pt]{amsart}

\usepackage{color}
\usepackage{graphicx, tikz}
\usepackage{xcolor}
\usepackage{cancel}
\usepackage{amssymb,amsfonts,amsthm,amsmath,calligra, enumerate}
\usepackage{extarrows}
\usepackage[utf8]{inputenc}
\usepackage{accents}
\usepackage{slashed}
\usepackage{mathrsfs,pifont}
\usepackage{mathtools}

\usepackage[bookmarksnumbered=true, bookmarksopen=true]{hyperref}
\usepackage[margin=1in]{geometry}
\usepackage[all]{xy}

\theoremstyle{definition}

\newtheorem{Theorem}{Theorem}[section]
\newtheorem{Proposition}[Theorem]{Proposition}
\newtheorem{Lemma}[Theorem]{Lemma}
\newtheorem{Corollary}[Theorem]{Corollary}
\newtheorem{Definition}[Theorem]{Definition}

\newtheorem{Conjecture}[Theorem]{Conjecture}

\newtheorem{Remark}[Theorem]{Remark}
\newtheorem{Example}[Theorem]{Example}

\newtheorem{Assumption}[Theorem]{Assumption}

\newtheorem{Notation}[Theorem]{Notation}

\numberwithin{equation}{section}

\def\ben{\begin{eqnarray*}}
\def\een{\end{eqnarray*}}

\newcommand{\HallEnv}{\mathrm{HallEnv}}

\newcommand{\sQ}{\mathrm{Q}}

\newcommand{\oO}{\mathcal{O}}

\newcommand{\der}{\mathrm{der}}

\newcommand{\bA}{\mathbb{A}}

\newcommand{\C}{\mathbb{C}}
\newcommand{\GL}{\mathrm{GL}}

\newcommand{\Crit}{\mathrm{Crit}}

\newcommand{\bC}{\mathbb{C}}

\newcommand{\bG}{\mathbb{G}}

\newcommand{\bK}{\mathbb{K}}

\newcommand{\bN}{\mathbb{N}}
\newcommand{\bP}{\mathbb{P}}
\newcommand{\bQ}{\mathbb{Q}}
\newcommand{\bR}{\mathbb{R}}

\newcommand{\bZ}{\mathbb{Z}}

\newcommand{\ba}{\mathbf{a}}

\newcommand{\bv}{\mathbf{v}}

\newcommand{\bd}{\mathbf{d}}

\newcommand{\bPsi}{\mathbf{\Psi}}
\newcommand{\dR}{\mathbf{R}}

\newcommand{\cC}{\mathcal{C}}

\newcommand{\cE}{\mathcal{E}}

\newcommand{\cF}{\mathcal{F}}

\newcommand{\cH}{\mathcal{H}}

\newcommand{\cL}{\mathcal{L}}
\newcommand{\cM}{\mathcal{M}}
\newcommand{\cN}{\mathcal{N}}
\newcommand{\cO}{\mathcal{O}}

\newcommand{\cS}{\mathcal{S}}

\newcommand{\cU}{\mathcal{U}}
\newcommand{\cV}{\mathcal{V}}
\newcommand{\cW}{\mathcal{W}}

\newcommand{\fC}{\mathfrak{C}}
\newcommand{\fg}{\mathfrak{g}}

\newcommand{\fM}{\mathfrak{M}}

\newcommand{\fX}{\mathfrak{X}}

\newcommand{\loc}{\mathrm{loc}}

\newcommand{\vir}{\mathrm{vir}}

\newcommand{\sA}{\mathsf{A}}
\newcommand{\sE}{\mathsf{E}}

\newcommand{\sS}{\mathsf{S}}
\newcommand{\sT}{\mathsf{T}}

\newcommand{\sW}{\mathsf{W}}
\newcommand{\sw}{\mathsf{w}}

\newcommand{\ad}{\operatorname{ad}}
\newcommand{\Attr}{\operatorname{Attr}}
\newcommand{\Aut}{\operatorname{Aut}}

\newcommand{\can}{\operatorname{can}}

\newcommand{\codim}{\operatorname{codim}}

\newcommand{\diag}{\operatorname{diag}}

\newcommand{\End}{\operatorname{End}}
\newcommand{\ev}{\operatorname{ev}}

\newcommand{\Frac}{\operatorname{Frac}}
\newcommand{\Gr}{\operatorname{Gr}}
\newcommand{\Hilb}{\operatorname{Hilb}}
\newcommand{\Hom}{\operatorname{Hom}}

\newcommand{\im}{\operatorname{im}}

\newcommand{\Lie}{\operatorname{Lie}}

\newcommand{\pr}{\operatorname{pr}}

\newcommand{\pt}{\operatorname{pt}}

\newcommand{\Res}{\operatorname{Res}}
\newcommand{\rk}{\operatorname{rk}}

\newcommand{\Span}{\operatorname{Span}}
\newcommand{\Spec}{\operatorname{Spec}}
\newcommand{\Stab}{\operatorname{Stab}}
\newcommand{\Sym}{\operatorname{Sym}}

\newcommand{\tr}{\operatorname{tr}}
\newcommand{\str}{\operatorname{Str}}

\newcommand{\Supp}{\operatorname{Supp}}

\newcommand{\id}{\operatorname{id}}

\newcommand{\In}{\operatorname{in}}
\newcommand{\Out}{\operatorname{out}}

\makeatletter
\newcommand{\ostar}{\mathbin{\mathpalette\make@circled\star}}
\newcommand{\make@circled}[2]{%
  \ooalign{$\m@th#1\smallbigcirc{#1}$\cr\hidewidth$\m@th#1#2$\hidewidth\cr}%
}
\newcommand{\smallbigcirc}[1]{%
  \vcenter{\hbox{\scalebox{0.77778}{$\m@th#1\bigcirc$}}}%
}
\makeatother

\newcommand{\yc}[1]{\textcolor{blue}{  #1 }}

\newcommand{\yh}[1]{\textcolor{red}{$[$ Yehao: #1 $]$}}
 
\newcommand{\yl}[1]{\textcolor{blue}{$[$ Yalong: #1 $]$}}

\title{Shifted quantum groups via critical stable envelopes}

\author{Yalong Cao}
\address{Morningside Center of Mathematics, Institute of Mathematics \& State Key Laboratory of Mathematical Sciences, Academy of Mathematics and Systems Sciences, Chinese Academy of Sciences, Beijing, China}
\email{yalongcao@amss.ac.cn}
\author{Andrei Okounkov} 
\address{Department of Mathematics, Columbia University, New York, U.S.A.}
\email{okounkov@math.columbia.edu} 
\author{Yehao Zhou}
\address{\parbox{\linewidth}{Center for Mathematics and Interdisciplinary Sciences, Fudan University, Shanghai 200433, China\\
Shanghai Institute for Mathematics and Interdisciplinary Sciences (SIMIS), Shanghai 200433, China}}
\email{yehao.zhou@simis.cn}
\author{Zijun Zhou}
\address{School of Mathematical Sciences, Shanghai Jiao Tong University, Shanghai, China}
\email{zijun.zhou@sjtu.edu.cn}

\subjclass[2020]{
Primary
14N35,  % (2000-now) Gromov-Witten invariants, quantum cohomology, Gopakumar-Vafa invariants, Donaldson-Thomas invariants (algebro-geometric aspects) [See also 53D45]
17B37. % (1991-now) Quantum groups (quantized enveloping algebras) and related deformations [See also 16T20, 20G42, 81R50, 82B23]
Secondary 
16G20, % Representations of quivers and partially ordered sets
22E99} % None of the above, but in MSC2010 section 22Exx Lie groups
\keywords{Stable envelopes, $R$-matrices, shifted Yangians, quivers with potentials, quantum multiplication by divisors}

\begin{document}

\begin{abstract}

Given a symmetric quiver with potential, 
we develop a geometric construction of shifted Yangians acting on the critical cohomologies of antidominantly framed quiver varieties with extended potentials, using the $R$-matrices constructed from critical stable envelopes. 
We relate such Reshetikhin type Yangians to Drinfeld type Yangians arising from critical cohomological Hall algebras.
Several detailed examples, including the trivial, Jordan, and tripled Jordan quivers are explicitly computed. 

For symmetric quiver varieties with potentials, by using the smallness property of their affinization maps, 
we derive explicit formulas for quantum multiplication by divisors in terms of Casimir elements of the associated Lie (super)algebras, extending 
results from Nakajima quiver varieties to the critical setting. A similar formula in the antidominantly framed case is also obtained, which includes Hilbert schemes of points on $\mathbb C^3$ as examples.

%Several detailed examples, including the trivial, Jordan, and tripled Jordan quivers, illustrate the general theory and its connections to shifted Yangians of $\mathfrak{gl}_{1|1}$, $\mathfrak{gl}_2$, and the affine Yangian of $\mathfrak{gl}_1$ respectively. 
%We also study their representations, including 
% higher spin, dual Verma and prefundamental modules in the $\mathfrak{gl}_2$ case, Kirillov-Reshetikhin modules,
% MacMahon modules, and several modules arising from spiked instantons of Nekrasov in the affine $\mathfrak{gl}_1$ case. 
%Our results provide a unified geometric framework for shifted quantum groups associated to symmetric quivers with potentials and point toward a broader relationship with BPS Lie algebras and quantum group doubles of cohomological Hall algebras.

\end{abstract}

\maketitle

\setcounter{tocdepth}{1}
\tableofcontents

\setlength{\parskip}{1ex}

\allowdisplaybreaks

\vspace{1cm}

\section{Introduction}

\subsection{Overview}

This paper is the second in our sequence of papers devoted to
the theory and applications of stable envelopes for critical
loci. The opening paper \cite{COZZ} of the sequence contains
a detailed introduction with which we assume the reader is
familiar.

The bulk of this paper is dedicated to the construction
of the \textit{shifted quantum group} actions on the \textit{critical
cohomology} along the lines described
in \cite[\S 1.8, \S 1.9]{COZZ}\,\footnote{The case using critical $K$-theory will be pursued in \cite{COZZ2}.}, see \S \ref{sec double of COHA} and \S \ref{sec shifted Yangian}. 
It builds on the ideas that go back to \cite{FRT}, in a
special case of a classical Lie algebra in its defining representation, and to \cite{Resh90}, in general. 

Concretely, the critical stable envelopes 
$$\Stab_\fC\colon H^{\sT}(X^\sA,\sw) \to H^{\sT}(X,\sw), \quad   \Stab^{\mathsf s}_{\fC}\colon K^{\sT}(X^\sA,\sw) \to K^{\sT}(X,\sw) $$
make the relevant critical
cohomology or $K$-theory groups into a tensor
category with a meromorphic braiding, called the
$R$-\textit{matrix}, see \S \ref{sect on rmatrix}. After that, the operators of the quantum group
appear as matrix elements of the
$R$-matrix.

The full quantum group structure also involves the \textit{antipode operation},
related to the notions of left and right duals, ${}^*V$ and $V^*$, for
a representation $V$ of a quantum group. 
Over a field, which means equivariant
localization to the fixed locus of an automorphism $g$ in the geometric context,
the underlying vector space of both ${}^*V$ and $V^*$ is
the ordinary vector space dual of $V$. Here $V$ stands for a localized
critical cohomology group and the choice of $g$ incorporates both the
value of the continuous parameters for the shifted Yangian and
the parameters of the representation, as is the common practice in
many other geometric representation theory contexts, see \cite{CG} for
a comprehensive discussion. We identify the dual of $V$
with $V$ itself using the standard \textit{bilinear form} on critical cohomology,
which exists and is nondegenerate when the $g$-fixed critical locus $\Crit(\sw)^g$ is 
proper, see \S \ref{subsec bilinear form}. This induces a transposition
\textit{anti-automorphism} on the Yangian, discussed in \S \ref{sect on anti-auto}.
In Reshetikhin's reconstruction \cite{Resh90}, $R$-matrices of the
dual representations appear as the inverse of the transpose of $R$ in one
of the factors. This completes the construction of the quantum group. 

An important application of the invariant bilinear form is a geometric
criterion of the \textit{irreducibility} of the Yangian modules. Recall that
we package the continuous parameters of the algebra and of the module
into a choice of an automorphism $g$ and work in equivariant cohomology
localized at $g$. Under fairly
general hypotheses, we prove in Theorem \ref{thm vac cogen_general} that our critical
cohomology modules are \emph{cogenerated} by the vacuum under
the action of the Yangian $\mathsf Y_\mu$. In fact, it is enough to consider the subalgebra
$\mathsf Y_\mu^{\leqslant} \subset \mathsf Y_\mu$ generated by 
the lowering
operators $\mathsf Y_\mu^-$ and the operators $\mathsf Y_\mu^0$ of the
cup product with tautological classes. If, in addition, $\Crit(\sw)^g$ is
proper, and hence a nondegenerate invariant bilinear form exists,
one can conclude that the Yangian module is irreducible, see Proposition 
\ref{prop on irr mod}. 

In \S \ref{sec quant mult}, we prove a formula for the \textit{quantum
multiplication by divisors}, first for a symmetric
quiver variety with potential, and then in the antidominantly
shifted case, see \cite[\S 1.11]{COZZ} and
below for additional introductory material.
We recall from \cite{COZZ} that an antidominant shift means that 
we allow asymmetry in the framing dimension vectors,
but require $\mathbf{d}_\textup{in} \geqslant \mathbf{d}_\textup{out}$.
The sign of the inequality here is correlated with the
specific stability condition on which we focus. Namely,
we consider a framed quiver representation stable if the
the maps out of the framing spaces and the maps between
the unframed spaces generate everything. It remains
an interesting question to explore what happens for
other shifts and other stability conditions. 

\S \ref{sec gl(1|1)}, \S \ref{sect on higher spin}, and \S \ref{sect on hilbc3} are devoted to making our general
theory explicit and application-ready in several
fundamental examples, which we explain in more details below.

\subsection{Reshetikhin Yangian v.s.\ Drinfeld Yangian}

We recall that in the formalism of \cite{Resh90}, quantum group
operators appear as matrix elements of $R$-matrices and
the relations among them come from the morphisms in the
corresponding category, that is, from maps that intertwine
$R$-matrices. For example, the Yang-Baxter equation itself
results in the RTT=TTR relation of \cite{FRT}.

For Yangians
of finite-dimensional Lie algebras, this recovers Drinfeld
Yangians. In its ``new'' presentation, Drinfeld Yangian is generated
by currents corresponding to the simple positive and negative
roots and polynomial loops into the Cartan subalgebra. In the
geometric situation, the latter correspond to operators of the
cup product by the tautological characteristic classes. The 
former have to do with the locus of extensions by a representation
of the unframed quiver  in which all maps are zero and the dimension
vector is a $\delta$-function. 

We adapt the construction of the
Drinfeld type Yangian to our setting in \S \ref{sec double of COHA}, building on
the works of Nakajima \cite{Nak1}, Negu\c t \cite{N2}, 
Varagnolo and Vasserot \cite{VV1, VV2}, and Yang and Zhao \cite{YZ},
among other works in the really vast literature that constructs
Yangian-like algebras using Hall algebras of Schiffmann-Vasserot \cite{SV0} and
Kontsevich-Soibelman type \cite{CoHA}. 

%To summarize, we have:

\begin{Theorem}(Theorems \ref{cor shifted yangian action}, \ref{thm e f h as R matrix elements}, \ref{thm drinfeld yangian map to rtt yangian})
Let $Q$ be a symmetric quiver with potential $\sW$, and $\mu\in \bZ^{Q_0}$.  
\begin{itemize}
    \item There is a Drinfeld type $\mu$\textit{-shifted Yangian} $\mathcal D\widetilde{\mathcal{SH}}_\mu(Q,\sW)$ defined by the double of spherical nilpotent critical CoHAs, which acts on the equivariant critical cohomology of framed quiver varieties with potentials by Hecke type correspondences specified in Theorem \ref{cor shifted yangian action}. 
\item When $\mu\in \bZ_{\leqslant 0}^{Q_0}$, there is a Reshetikhin type $\mu$\textit{-shifted Yangian} $\mathsf Y_\mu(Q,\sW)$ defined by geometric $R$-matrices
which come from critical stable envelopes. It naturally acts on equivariant critical cohomology of  
$\mu =\mathbf{d}_\textup{out} - \mathbf{d}_\textup{in}$ antidominant-framed quiver varieties with potentials. 
\item When $\mu\in \bZ_{\leqslant 0}^{Q_0}$, the action of $\mathcal D\widetilde{\mathcal{SH}}_\mu(Q,\sW)$ factors through $\mathsf Y_\mu(Q,\sW)$ via a natural map 
$$\mathcal D\widetilde{\mathcal{SH}}_\mu(Q,\sW)\to \mathsf Y_\mu(Q,\sW). $$ 
\end{itemize}
\end{Theorem}
From the technical point of view, to define Reshetikhin type shifted Yangians for symmetric quivers with potentials, we need to consider a so-called 
\textit{auxiliary data set} $\mathcal C$, which specifies framings, extensions of torus actions and potentials to the framed quivers, see Definition \ref{def on adm aux data}. For a choice of $\mathcal C$, one can define $\mathsf Y_\mu(Q,\sW,\mathcal C)$
(Definition \ref{def of rtt yang}) and the Reshetikhin type $\mu$-shifted Yangian is defined with respect to the union of all auxiliary data sets: 
$$\mathsf Y_\mu(Q,\sW):=\mathsf Y_\mu(Q,\sW,\bigcup_{\mathcal C}\mathcal C).$$
Apriori, this seems not possible to manipulate due to the possible infinite number of auxiliary data. But remarkably, we show that 
if $\mathcal C$ is \textit{admissible}, which requires certain torus weight condition on the framings, see Definition \ref{def on adm aux data}, we have 
the following.
\begin{Theorem}(Theorem \ref{thm admissible})
If $\mathcal C$ is admissible, then for any $\mu\in \bZ_{\leqslant 0}^{Q_0}$, we have $\mathsf Y_\mu(Q,\sW,\mathcal C)=\mathsf Y_\mu(Q,\sW)$.
\end{Theorem}
This makes computations manageable as one can for example choose a finite set of auxiliary data labelled by the vertex of quiver $Q$.
We constantly used it in computations throughout the paper. 

\subsection{PBW type theorem and Lie subalgebra $\mathfrak{g}_{Q,\sW}$}

%We show in Theorem \ref{thm e f h as R matrix elements} that the Drinfeld Yangian is naturally a subalgebra in the Reshetikhin Yangian. 
To control the size of our zero shifted Yangian $\mathsf Y_0(Q,\sW)$, we need the following \textit{PBW} type theorem. 
\begin{Theorem}(Proposition \ref{prop g(Q,w)}, Theorems \ref{thm grY}, \ref{thm compare with MO yangian})
Let $Q$ be a symmetric quiver with potential $\sW$.  
\begin{itemize}
    \item There is a Lie superalgebra $\mathfrak{g}_{Q,\sW} \subset \mathsf Y_0(Q,\sW)$ defined from the classical $R$-matrix. 
    \item $\mathsf Y_0(Q,\sW)$ is generated by $\mathfrak{g}_{Q,\sW}$ and multiplication by tautological classes.
Moreover, 
\begin{align*}
    \mathrm{gr}\: \mathsf Y_0(Q,\sW)\cong \mathcal U(\mathfrak{g}_{Q,\sW}[u])\,,
\end{align*}
with respect to the filtration by degree in $u$.
\item When $Q$ is a tripled quiver and $\sW$ is the canonical cubic potential, $\mathsf Y_0(Q,\sW)$ recovers the Maulik-Okounkov Yangian 
and $\mathfrak{g}_{Q,\sW}$ recovers the Maulik-Okounkov Lie algebra \cite{MO}. 
\end{itemize}
\end{Theorem}
To study the nonzero shifted Yangians, one can use the (parametrized) shifted homomorphisms in \S \ref{sec shift map} 
to relate them to the zero shifted case.

Note that in general, the Reshetikhin Yangian is much larger than the Drinfeld type Yangian.  
We conjecture
that, on the CoHA side, the right object to consider is
the BPS Lie algebra of Davison--Meinhardt \cite{DM}. Namely, in
parallel with a conjecture made by one of us in the context of
Nakajima varieties, and recently proven by Botta and Davison in
\cite{BD} and Schiffmann and Vasserot in \cite{SV2}, we conjecture the Davison--Meinhardt  Lie algebra is
the positive half of the Lie algebra $\mathfrak{g}_{Q,\sW}$. 
%We recall that, by construction, $\mathfrak{g}_{Q,\sW}$ is the 0th filtration piece
%of $\mathsf Y_0(Q,\sW)$, and that the whole Yangian is a deformation of
%$\mathcal U(\mathfrak{g}_{Q,\sW}[u])$.

If true, this conjecture provides
a natural collection of quantum group doubles for the cohomological Hall
algebra, indexed by the shift parameter $\mu$. In general, the lack
of properness is the reason why only half of the quantum
group operations are defined in the CoHA setting. From the beginning,
stable envelopes have served as particularly useful canonical extensions of
attracting manifolds to proper correspondences, and it would be
logical for them to have a role here, too.

\subsection{Coproducts}

Compared to coproducts studied by many people from the CoHA side, we have a 
coproduct on $\mathsf Y_\mu(Q,\sW)$ defined by splitting the $R$-matrix in \S \ref{sect on coprod} which should correspond to the \textit{standard coproduct}.
This is justified by the following. 

\begin{Theorem}(Remark \ref{rmk on copr compa}, Theorem \ref{thm on copr compa})
Let $A$ be a generalized Cartan matrix with symmetrizer $D$, $\mathfrak{g}_{A,D}$ be the symmetrizable Kac-Moody Lie algebra associated to the pair $(A,D)$ and $Y_{0}(\mathfrak{g}_{A,D})$ be its Yangian defined by explicit generators with relations 
(Example \ref{ex on gen symi}). Let $(\widetilde{Q}, \sW)$ be the associated tripled quiver with potential of Geiss-Leclerc-Schr\"oer \cite{GLS} (see Appendix \ref{sec Y(sym KM)}).
Then there is an algebra homomorphism 
\begin{equation}\label{intro equ on bialg map for sym km}\varrho\colon Y_{0}(\mathfrak{g}_{A,D})_{\loc}\xrightarrow{\text{Ex.\ref{ex on gen symi}}} \mathcal D\widetilde{\mathcal{SH}}_0(\widetilde{Q},\sW)_{\loc}\xrightarrow{\text{Thm.\ref{thm drinfeld yangian map to rtt yangian}}} \mathsf Y_0(\widetilde{Q},\sW). \end{equation}
When $\mathfrak{g}_{A,D}$ is a finite dimensional simple Lie algebra, $Y_{0}(\mathfrak{g}_{A,D})_{\loc}$ has a standard coproduct due to Drinfeld, then the map 
\eqref{intro equ on bialg map for sym km}
is an injective bialgebra homomorphism. Moreover, there is a bialgebra isomorphism
\begin{align*}
\widetilde{\varrho}\colon Y_{0}(\mathfrak{g}_{A,D})_{\loc}\otimes \bC[\mathsf d_{i,k}: i\in \widetilde{Q}_0,\, k\in \bZ_{\geqslant 0}]\cong \mathsf Y_0(\widetilde{Q},\sW).
\end{align*}
Here $\mathsf d_{i,k}$ are primitive for all $i$ and $k$, $\widetilde{\varrho}$ on the first component is $\varrho$, and $\widetilde{\varrho}(\mathsf d_{i,k})$ is the multiplication by $\mathrm{ch}_{k+1}(\mathsf D_{\In,i}-\mathsf D_{\Out,i})$ operator when acting on a state space with framing vector spaces $\mathsf D_{\In}, \mathsf D_{\Out}$.
\end{Theorem}
%The proof uses the coproduct formula obtained in \S \ref{expl formula of coprod}.
%Here the coproduct on $Y_{0}(\mathfrak{g}_{A,D})_{\loc}$ 
We remark that since Drinfeld \cite{Dri} wrote down the standard coproduct formula in the case of finite dimensional simple Lie algebras, a proof has never appeared in the literature until the work of Guay-Nakajima-Wendlandt \cite{GNW}. The above theorem can be seen as another proof using geometric methods (critical stable envelopes). 
In fact, the coproduct formula of \cite{GNW} also works for more general affine Kac-Moody Lie algebras. We expect 
the map \eqref{intro equ on bialg map for sym km} preserves coproduct in those cases too. 

\subsection{Smallness for symmetric quiver varieties}

The theory developed in this sequence of papers is a far-reaching
generalization of what was previously known for Nakajima varieties and 
other symplectic resolutions. Historically, the
symplectic form played an important technical role in
various geometric arguments, and a different technical ingredient had
to be found in the more general situation treated here.
The property we use, and which may be of independent interest,
is the following.

Let $X=\cM_\theta(\bv,\bd)$ be a symmetric quiver variety. 
We prove in a companion paper \cite{COZZ3} that $X$ has
a Poisson structure such that:
\begin{itemize}
\item[(1)] the Poisson center $\cO_B \subset H^0(\cO_X)$ is a polynomial
  algebra,
\item[(2)] the fibers of the map $X\to B$ have symplectic
  singularities, in particular, finitely many symplectic leaves,
\item[(3)] the generic fiber of $X\to B$ is affine. 
\end{itemize}
Moreover, if a torus $\sA$ acts on $X$ in a symmetry-preserving way
then $\sA$ acts trivially on $B$.

As a corollary, the affinization map
$X\to X_0 = \cM_0(\bv,\bd)$ is birational and \emph{small}. By
definition, that means that $X\times_{X_0} X$ is the union of the
diagonal and components of dimension \emph{less} than $\dim X$. 
The components of dimension $\dim X-1$ appear as
\emph{Steinberg correspondences} and have a very important place in
the theory, see \S \ref{subsec steinberg corr}.

\textit{Smallness} plays a crucial role in our analysis of the
quantum multiplication by divisors for symmetric quiver varieties with 
potentials. There, one needs to understand the Steinberg
correspondence $Q_\beta$ formed, in the  sense of the virtual fundamental
class, by pairs of points lying on a rational curve of degree
$\beta$. For any choice of attracting directions $\fC$, we prove in
Theorem \ref{thm on stab and Q}
the commutativity of the diagram
\begin{equation}\label{commQStab}
\xymatrix{
H^{\sT}(F,\sw)  \ar[rr]^{\Stab_\fC} \ar[d]_-{\pm \sQ_\beta^{F}}   && H^{\sT}(X,\sw) \ar[d]^-{\sQ_\beta}  \\
H^{\sT}(F,\sw)   && \ar[ll]_{\,\, (-)|_{F}\circ\Stab_\fC^{-1}} H^{\sT}(X,\sw)\,, 
} 
\end{equation}
for any fixed component $F\subseteq X^\sA$ with a precise determination of the sign and up to an equivariant scalar trivial on $\sA$. In the symplectic
situation, the proof of the parallel result in \cite{MO} rests on the
analysis of broken/unbroken contributions, originated from \cite{OP}. We replace that analysis by the smallness argument in broader generality of symmetric quivers with potentials.

\subsection{Quantum multiplication by divisors}

As in the symplectic case, it is convenient to get rid of the signs in
\eqref{commQStab} and elsewhere by shifting the origin in the
torus of degree-counting variables $z$
to a certain point of order $2$. This is called the
\textit{modified quantum product}, see Definition \ref{def modified quant mult}, and is denoted
by $\widetilde{\star}$ in our paper.

\begin{Theorem}(Theorem \ref{thm on qm div for sym}, Corollary \ref{cor qm div for sym_sp inj})
%Let $X$ be a symmetric quiver variety with a torus $\sT$-action and a $\sT$-invariant function $\sw$.
Let $X$ be a symmetric quiver variety with a potential $\sw$. 
Suppose that the specialization map $\mathrm{sp}\colon H^\sT(X,\sw^{})\to H^\sT(X)$ is injective, then the modified quantum multiplication by divisor $\lambda$ on 
$H^\sT(X,\sw^{})$ is given by
\begin{equation}
c_1(\lambda)\,\widetilde{\star}\,\cdot  =c_1(\lambda)\cup\,\cdot 
-\sum_{\alpha>0}\alpha(\lambda)\frac{z^\alpha}{(-1)^{|\alpha|}-z^\alpha}\mathrm{Cas}_\alpha+\cdots
\,,\label{eq:modified product}
\end{equation}
where dots stand for a scalar operator, 
the sum is over the positive roots $\alpha$ of the Lie superalgebra
$\mathfrak{g}_{Q,\sW} \subset \mathsf Y_0(Q,\sW)$ underlying the Yangian, and
the Casimir element 
$\mathrm{Cas}_\alpha \in \fg_\alpha \otimes \fg_{-\alpha} \subset
U(\mathfrak{g}_{Q,\sW})$ is the canonical tensor.

\end{Theorem}

The shape of the 
formula \eqref{eq:modified product} is essentially the same as in the
symplectic case, the main difference is the presence of the parity
$(-1)^{|\alpha|}$ of a root in a superalgebra. In \eqref{eq:modified
  product} we sum over the positive roots because of our choice of
the stability condition, discussed above. For a general
stability condition $\theta$ on a symmetric quiver variety the
corresponding range of summation should be modified to
$\theta \cdot \alpha > 0$. 

In general, the formula \eqref{eq:modified product} holds on the
image of the specialization map from the critical cohomology to
the ordinary cohomology. In the eventuality that the specialization
map is not injective, it is the image group that is the natural
object to consider in the enumerative setup.

For quiver varieties defined from symmetric quivers with \textit{antidominant framing} 
$\mu =\mathbf{d}_\textup{out} - \mathbf{d}_\textup{in}\in \bZ_{\leqslant 0}^{Q_0}$ (which include Hilbert schemes of points on $\C^3$ as examples, see \S \ref{sect on qm by div on hilb}), the formula 
\eqref{eq:modified product} is modified as follows.
\begin{Theorem}(Theorem \ref{thm qm div for asym_sp inj}) 
We have 
\begin{equation}
c_1(\lambda)\,\widetilde{\star}\,  =c_1(\lambda)\cup 
-\sum_{\substack{\alpha>0\\ \alpha\cdot
    \mu=0}}\alpha(\lambda)\frac{z^\alpha}{(-1)^{|\alpha|}-z^\alpha}\mathrm{Cas}_\alpha
-\sum_{\substack{\alpha>0\\ \alpha\cdot \mu\neq 0}}(-1)^{|\alpha|}\alpha(\lambda)\,z^\alpha\,\mathrm{Cas}_{\alpha,\mu}
+\cdots
\,,\label{eq:modified product2}
\end{equation}
where $\mathrm{Cas}_{\alpha,\mu}$ is a certain shifted Casimir element
in the shifted Yangian $\mathsf Y_\mu(Q,\sW)$, defined as follows. 
As a special case of the \emph{parametrized shift homomorphisms} defined
in \S \ref{sec shift map}, we have a map
$$
\mathsf Y_0(Q,\sW)\to \mathsf Y_\mu(Q,\sW)[z],
$$
which sends $\mathrm{Cas}_\alpha$ to a polynomial of $z$ of degree
$-\mu \cdot \alpha$. The shifted Casimir element
$\mathrm{Cas}_{\alpha,\mu}$ is the
coefficient of $z^{-\mu \cdot \alpha}$ in this polynomial, see
Definition \ref{def of shifted Casimir}.

\end{Theorem}

\subsection{Examples}

We apply our general theory to specific examples, and find presentations of our Reshetikhin type shifted Yangians using generators and relations. 

\begin{Theorem}(Theorems \ref{thm ex trivial quiver_main}, \ref{thm ex Jordan_main}, \ref{thm ex tripled Jordan_main})
For an arbitrary $\mu\in \bZ_{\leqslant 0}$, we have 
\begin{itemize}
    \item if $Q$ is the trivial quiver, there is a natural algebra isomorphism
\begin{align*}
    Y_{\mu}(\mathfrak{gl}_{1|1})\otimes_{\bC[\hbar]} \bC(\hbar)\cong \mathsf Y_{\mu}(Q,0)\:,
\end{align*}
where $Y_{\mu}(\mathfrak{gl}_{1|1})$ is the 
$\mu$-shifted Yangian of $\mathfrak{gl}_{1|1}$ (Definition \ref{def of shifted Y(gl(1|1))}).
    \item if $Q$ is the Jordan quiver, there is a natural algebra isomorphism
\begin{align*}
    Y_{\mu}(\mathfrak{gl}_2)\otimes_{\bC[\hbar]} \bC(\hbar)\cong \mathsf Y_{\mu}(Q,0)\:,
\end{align*}
where $Y_{\mu}(\mathfrak{gl}_2)$ is the Frassek-Pestun-Tsymbaliuk 
$\mu$-shifted Yangian of $\mathfrak{gl}_2$ (Definition \ref{def of shifted Y(gl_2)}).
\item 
if $Q$ is the tripled Jordan quiver, equipped with the canonical cubic potential $\sW$,  there is a surjective algebra homomorphism
\begin{align*}
    \varrho_\mu\colon Y_{\mu}(\widehat{\mathfrak{gl}}_1)\twoheadrightarrow \mathsf Y_{\mu}(Q,\sW),
\end{align*}
where $Y_{\mu}(\widehat{\mathfrak{gl}}_1)$ is the $\mu$-shifted affine Yangian of $\mathfrak{gl}_1$ (Definition \ref{def of shifted affine Y(gl_1)}). Moreover, $\varrho_0$ is a bialgebra isomorphism, with respect to the coproduct on $Y_{0}(\widehat{\mathfrak{gl}}_1)$ defined by Schiffmann-Vasserot \cite{SV2}, and the coproduct on $\mathsf Y_{\mu}(Q,\sW)$ defined in \S \ref{sect on coprod} using critical stable envelopes.
\end{itemize}
\end{Theorem} 

In the above examples, we also study their representations, including 
higher spin, dual Verma and prefundamental modules in the $\mathfrak{gl}_2$ case (\S \ref{sec bases of H_N, V, F}, Theorem \ref{thm fusion}),  Kirillov-Reshetikhin modules, vacuum MacMahon modules, and several modules arising from spiked instantons of Nekrasov in the affine $\mathfrak{gl}_1$ case which are related to the vertex algebra at the corner $\mathcal Y_{L,M,N}$ of Gaiotto and Rap\v{c}\'{a}k (Theorem \ref{prop on 3 modules of gl1hat}, Proposition \ref{prop on more ex}). We perform explicit computations of stable envelopes and $R$-matrices, including braiding between higher spin modules, dual Verma and prefundamental modules \S \ref{sect on comp stab of gl2}, \S \ref{sect on comp rmatrix of gl2}, \S \ref{sect on fusion gl2}, as well as some examples in Lie superalgebras, see \S \ref{sec fund R-mat gl(1|1)}, \S \ref{sec Lax mat Cl mod}, and Appendix \ref{sect on Lie superalgebra}. 

%\subsection{Background}

%\subsection*{Related works and future directions}

%\subsection*{Notations and Conventions}

\subsection*{Acknowledgments}
This work benefits from helpful discussions and communications from many people, including
Mina Aganagic, Daniel Halpern-Leistner, Ryosuke Kodera, Yixuan Li, Yuan Miao, Hiraku Nakajima, Andrei Negu\c{t}, Tudor P\u{a}durariu, Spencer Tamagni, Yukinobu Toda, Yaping Yang, Gufang Zhao, Tianqing Zhu. A.O. would like to thank SIMIS for hospitality. We would like to thank Kavli IPMU for bringing us together.

\subsection*{Statements and Declarations}
We have no conflicts of interest to disclose.

\iffalse
\yc{\begin{enumerate}
\item Critical Yangian v.s. MO Yangian. Expectation for general case: non-simply laced Lie alge or more general Lie alg. Add example on $C_2$ and super $sl(2|1)$. 
\item shifted quantum gp, examples: $\Hilb(\C^3)$, spiked instantons on $\C^3$ (expected to be useful in ``AGT"), DST lax. 
\item Higher spin representations. KR-module (compared with VV2's method) 
\item d-critical flops and stab. Example given by VGIT, use \S \ref{sec:na stab}.
\end{enumerate}} 
\fi

\vspace{1cm}

\section{\texorpdfstring{$R$}{R}-matrices}\label{sect on rmatrix}

In this section, we define $R$-matrices in the context of critical cohomology on quiver varieties with potentials, using corresponding 
Hall (stable) envelopes developed in \cite[\S 9]{COZZ}. We study their properties and deduce Yang-Baxter equations. For later use, we also analyze the classical $R$-matrices, and study Steinberg correspondences and their compatibility with stable envelopes.

\subsection{Hall envelopes and stable envelopes}
Let $Q=(Q_0,Q_1)$ be a quiver and $\bv,\bd_{\In},\bd_{\Out}\in \bQ^{Q_0}$, $\underline{\bd}=(\bd_{\In},\bd_{\Out})$. Denote the \textit{space of framed representations} of $Q$ with gauge dimension $\mathbf{v}$ and in-coming framing dimension $\mathbf{d}_{\mathrm{in}}$ 
and out-going framing dimension $\mathbf d_{\mathrm{out}}$ by 
\begin{equation}\label{equ rvab}
R(\mathbf{v},\underline{\bd}):=R(Q,\mathbf{v},\underline{\bd})=\bigoplus_{a\in Q_1}\underbrace{\Hom(\bC^{\mathbf v_{t(a)}},\bC^{\mathbf v_{h(a)}})}_{X_a}\oplus \bigoplus_{i\in Q_0}\left(\underbrace{\Hom(\bC^{\bd_{\mathrm{in},i}},\bC^{\mathbf v_i})}_{A_i}\oplus \underbrace{\Hom(\bC^{\mathbf v_i},\bC^{\bd_{\Out,i}})}_{B_i}\right).
\end{equation}
Let $\sT\subseteq \Aut_G(R(\bv,\underline{\bd}))$ be a subtorus in the flavour group, and $\sA\subseteq \sT$ be a subtorus. We fix a cyclic stability $\theta\in \bQ_{<0}^{Q_0}$ and define the \textit{quiver variety} as the GIT quotient:
\begin{align}\label{equ on qv}
    X=\cM_\theta(\bv,\underline{\bd}):=R(\mathbf{v},\underline{\bd})/\!\!/_{\theta} G=R(\mathbf{v},\underline{\bd})^s/G,
\end{align}
where the \textit{gauge group} $G=\prod_{i\in Q_0}\GL(\mathbf v_i)$ naturally acts on $R(\mathbf{v},\underline{\bd})$ by compositions with maps. 
As there is no relation imposed on the quiver, the above quiver variety is a smooth quasi-projective variety. 

Let $\fC$ be a chamber associated with $\sA$-action on $\cM_\theta(\bv,\underline{\bd})$, $\sw\colon X\to \C$ be a $\sT$-invariant regular function.
Then Hall envelopes 
\begin{align}\label{equ on hall env}
\begin{gathered}
\HallEnv_\fC\colon H^\sT(\cM_\theta(\bv,\underline{\bd})^{\sA},\sw)\longrightarrow H^\sT(\cM_\theta(\bv,\underline{\bd}),\sw),
\end{gathered}
\end{align}
are defined as the composition of interpolation maps \cite[Def.~9.1]{COZZ}, Hall multiplications and restrictions to the stable loci \cite[Def.~9.4]{COZZ}. These can be defined for any quiver with potential. 

\begin{Definition}\label{def of sym qv}
A quiver $Q$ is \textit{symmetric} if its adjacency matrix $ (\mathsf Q_{ij})_{i,j\in Q_0}$ is symmetric. 
A framing $\underline{\bd}$ is \textit{symmetric} if $\bd_{\In}=\bd_{\Out}=\bd$. In the latter case, we simplify the notations as 
\begin{equation}\label{equ on sym qu}R(\bv,\bd)=R(Q,\bv,\bd)= R(Q,\mathbf{v},\underline{\bd}), \quad 
\cM_\theta(\bv,\bd)=\cM_\theta(Q,\bv,\bd)=\mathcal M_{\theta}(Q,\mathbf v,\underline{\bd}). \end{equation} 
When both $Q$ and $\underline{\bd}$ are symmetric, we call $\cM_\theta(\bv,\bd)$ a \textit{symmetric quiver variety}.  
We say that a torus action $\sA$ on the symmetric quiver variety $\cM_\theta(\mathbf v,\mathbf{d})$ is \textit{self-dual} if it is induced from a self-dual action of $\sA$ on $R(\mathbf v,\mathbf{d})$.
\end{Definition}

\begin{Definition}\label{def symmetrization}
Let $Q$ be a symmetric quiver and $\bv,\bd_{\In},\bd_{\Out}\in \bN^{Q_0}$ be dimension vectors. The \textit{symmetrization} of $R(\bv,\underline{\bd})$ is defined to be $R(\bv,\mathbf c)$ \eqref{equ on sym qu} where $\mathbf c_i=\max\{\bd_{\In,i},\bd_{\Out,i}\}$. We identify $R(\bv,\underline{\bd})$ as a $G$-subrepresentation of 
\begin{align*}
    R(\bv,\mathbf c)=R(\bv,\underline{\bd})\oplus \bigoplus_{\begin{subarray}{c}i\in Q_0 \\ \mathrm{s.t.}\, \bd_{\In,i}>\bd_{\Out,i}   \end{subarray}}
     \Hom(\bC^{\bv_i},\bC^{\bd_{\In,i}-\bd_{\Out,i}})\oplus \bigoplus_{\begin{subarray}{c}i\in Q_0 \\ \mathrm{s.t.}\, \bd_{\In,i}<\bd_{\Out,i}   \end{subarray}} \Hom(\bC^{\bd_{\Out,i}-\bd_{\In,i}},\bC^{\bv_i}).
\end{align*}
\end{Definition}

\begin{Definition}\label{def pseudo-self-dual}
Let $Q$ be a symmetric quiver and $\bv,\bd_{\In},\bd_{\Out}\in \bN^{Q_0}$ be dimension vectors and $\sA$ be a subtorus of flavour group $\Aut_G R(\bv,\underline{\bd})$. We say that the $\sA$-action on $R(\bv,\underline{\bd})$ is \textit{pseudo-self-dual} if it extends to a self-dual $\sA$-action on the symmetrization $R(\bv,\mathbf c)$. 
A torus $\sA$ action on $\cM_\theta(\bv,\underline{\bd})$ is \textit{pseudo-self-dual} if it is induced from a \textit{pseudo-self-dual} $\sA$-action on $R(\bv,\underline{\bd})$.
% the following two conditions hold:
% \begin{enumerate}
%     \item the induced $\sA$ action on the subspace of unframed representations $R(\bv,\mathbf{0},\mathbf{0})$ is self-dual,
%     \item for all $i\in Q_0$, the virtual $\sA$-character $\bC^{\bd_{\In,i}}-\bC^{\bd_{\Out,i}}\in K_\sA(\pt)$ is 
% \end{enumerate}
\end{Definition}

\begin{Example}
\label{ex doubled vs tripled}
Given a quiver $Q=(Q_0,Q_1)$, we define the \textit{doubled quiver} to be $\overline{Q}:=(Q_0,Q_1\sqcup Q_1^*)$, where $Q_1^*$ is the same as $Q_1$ but with reversed arrow direction. We also define the \textit{tripled quiver} $\widetilde{Q}:=(Q_0,Q_1\sqcup Q_1^*\sqcup Q_0)$, which adds one edge loop to each node on top of $\overline{Q}$. 

Fix $\mathbf{v},\mathbf{d}\in \bN^{Q_0}$, and a cyclic stability $\theta\in \bQ_{<0}^{Q_0}$, we consider two \textit{symmetric} quiver varieties:
\begin{enumerate}
    \item $\cM_\theta(\overline{Q},\mathbf{v},\mathbf{d})$ associated with the doubled quiver $\overline{Q}$;
    \item $\cM_\theta(\widetilde{Q},\mathbf{v},\mathbf{d})$ associated with the tripled quiver $\widetilde{Q}$.
\end{enumerate}
The quotient map 
$$R(\overline{Q},\mathbf{v},\mathbf{d})^{s}\to \cM_\theta(\overline{Q},\mathbf{v},\mathbf{d})=R(\overline{Q},\mathbf{v},\mathbf{d})^{s}/G$$ is a principal $G$-bundle.
Denote the associated adjoint bundle to be $\mathcal G\to \cM_\theta(\overline{Q},\mathbf{v},\mathbf{d})$. Then the \textit{moment map} for $G\curvearrowright R(\overline{Q},\mathbf{v},\mathbf{d})$ descends to a section $\mu\in \Gamma\left(\cM_\theta(\overline{Q},\mathbf{v},\mathbf{d}),\mathcal G^\vee\right)$. The total space $\mathrm{Tot}(\mathcal G)$ is naturally identified with an open subscheme $\overset{\bullet}{\cM}_\theta(\widetilde{Q},\mathbf{v},\mathbf{d})\subseteq {\cM}_\theta(\widetilde{Q},\mathbf{v},\mathbf{d})$ consisting of $G$-equivalence classes in $R(\widetilde{Q},\mathbf{v},\mathbf{d})$, which are $\theta$-stable when restricted to $\overline{Q}$. 
The zero locus of $\mu$:
$$\cN_{\theta}(\overline{Q},\mathbf{v},\mathbf{d}):=Z(\mu)$$
is the \textit{Nakajima quiver variety} \cite{Nak1,Nak2} (which we assume to be nonempty).  
Note that $\cN_{\theta}(\overline{Q},\mathbf{v},\mathbf{d})$ is a smooth subvariety in $\cM_{\theta}(\overline{Q},\mathbf{v},\mathbf d)$ with codimension equals to $\mathcal G$, and $Z(\mu)=Z^\der(\mu)$, coincides with the derived zero locus. 

Denote $\bC^*_\hbar$ the torus that acts on $R(\overline{Q},\bv,\bd)$ by assigning weights $R(\overline{Q},\bv,\bd)=M\oplus\hbar^{-1}M^\vee$, where
\begin{align*}
    M=\bigoplus_{a\in Q_1}\Hom(\bC^{\bv_{t(a)}},\bC^{\bv_{t(a)}})\oplus\bigoplus_{i\in Q_0}\Hom(\bC^{\bd_{i}},\bC^{\bv_{i}}).
\end{align*}
We choose torus $\sT$ which contains $\bC^*_\hbar$ such that $\sT=(\sT/\bC^*_\hbar)\times \bC^*_\hbar$\,, and let $\sA\subseteq \sT/\bC^*_\hbar$ be a subtorus. Choose a function $\phi\colon \cM_\theta(\overline{Q},\mathbf{v},\mathbf{d})\to \bC$ and a section $\mu\in \Gamma(\cM_\theta(\overline{Q},\bv,\bd),\mathcal \hbar^{-1}\mathcal G^\vee)$ and appropriate $\sT$ action such that $\phi$ and $\mu$ are $\sT$-invariant. Let us define the $\sT$ action on edge loops by scaling with weight $\hbar$, then $\sT$ fixes 
%the pairing $\langle e,\mu\rangle$. 
the \textit{canonical cubic potential}: 
\begin{equation}
\label{equ on can cub pot}\widetilde{\sw}=\langle e,\mu\rangle=\sum_{i\in Q_0}\tr(\varepsilon_i\mu_i), \end{equation} 
where $\varepsilon_i$ is the edge loop on the node $i$, and $\mu_i$ is the $i$-th component of moment map.

Then according to \cite[Ex.~6.13]{COZZ}, the following
$$\left(\overset{\bullet}{\cM}_\theta(\widetilde{Q},\mathbf{v},\mathbf{d}),\cM_\theta(\overline{Q},\mathbf{v},\mathbf{d}),\mu, \phi \right), \quad \left(\cN_{\theta}(\overline{Q},\mathbf{v},\mathbf{d}),\cN_{\theta}(\overline{Q},\mathbf{v},\mathbf{d}),0, \phi|_{\cN_{\theta}(\overline{Q},\mathbf{v},\mathbf{d})}\right)$$ are compatible dimensional reduction data in the sense of \cite[Def.~6.2]{COZZ}. In particular, for functions
$$\sw:=\langle e,\mu\rangle+\phi\colon \overset{\bullet}{\cM}_\theta(\widetilde{Q},\mathbf{v},\mathbf{d})\to \C, \quad \sw':=\phi|_{\cN_{\theta}(\overline{Q},\mathbf{v},\mathbf{d})}\colon \cN_{\theta}(\overline{Q},\mathbf{v},\mathbf{d})\to \C,
$$
we have (ref.~\cite[Props.~6.5,~6.9]{COZZ}):
\begin{equation}\label{equ on dim redd}\Crit(\sw)\cong \Crit(\sw'), \quad 
H^{\sT}(\cN_{\theta}(\overline{Q},\mathbf{v},\mathbf{d}),\sw')\cong H^{\sT}(\overset{\bullet}{\cM}_\theta(\widetilde{Q},\mathbf{v},\mathbf{d}),\sw). \end{equation}
% Similar isomorphism holds for critical $K$-theory with more conditions (e.g.~when $\phi=0$).
\end{Example}

\begin{Theorem}
$($\cite[Thms.~9.20, 4.22, 5.5]{COZZ}$)$
\label{thm AFSQV}
If the quiver $Q$ is symmetric, $\bd_{\In,i}\geqslant \bd_{\Out,i}$ for all $i\in Q_0$, 
and $\sA$ is a pseudo-self-dual subtorus of $\sT$, then the Hall envelopes \eqref{equ on hall env}
are compatible with Hall operations (\cite[Def.~9.4]{COZZ}). In particular, they satisfy the triangle lemma (\cite[Lem.~9.6]{COZZ}). 

When the framing is moreover symmetric, we have equalities
$$\HallEnv_\fC=\Stab_\fC, $$
with (critical) stable envelopes $\Stab_\fC$ defined in \cite[Defs.~3.4]{COZZ}. 
\end{Theorem}
\begin{Remark}\label{rmk on hallequstab}
%The above equalities also hold in some cases in case the framing is not symmetric. 

We refer to \cite[Props.~9.23, 9.24]{COZZ} when the equalities also hold but the framing is not symmetric. 
This is applicable to the case when the $\sA$-action splits the framing to $u\underline{\bd}'+\underline{\bd}''$, including all root
$R$-matrices used in this paper, see for example eqn.~\eqref{op E(m)}.
\end{Remark}

\begin{Remark}
The condition $\bd_{\In,i}\geqslant \bd_{\Out,i}$ is necessary for the triangle lemma to hold, see \cite[Ex.~9.25,~Prop.~9.26]{COZZ}.
\end{Remark}

\subsection{$R$-matrices}

\begin{Definition}
For two chambers $\fC_1,\fC_2$ in $\Lie(\sA)_\bR$, and a generic slope $\mathsf s\in \mathrm{Char}(G)\otimes_\bZ\bR$, the \textit{R-matrices} are
\begin{align*}
    R_{\fC_2,\fC_1}&:=(\HallEnv_{\fC_2})^{-1}\circ \HallEnv_{\fC_1}\in \End\left(H^\sT(X^\sA,\sw )_{\mathrm{loc}}\right),
    % R^{\mathsf s}_{\fC_2,\fC_1}&:=(\HallEnv^{\mathsf s}_{\fC_2})^{-1}\circ \HallEnv^{\mathsf s}_{\fC_1}\in \End\left(K^\sT(X^\sA,\sw )_{\mathrm{loc}}\right).
\end{align*}
Here we consider $\sT$-localized cohomology, so inverses of Hall envelopes exist (see \cite[Rmk.~9.5]{COZZ}).
\end{Definition}

\begin{Remark}
The $R$ matrix admits a Gauss type decomposition 
\begin{align}\label{gauss decomp_general}
    R_{\fC_2,\fC_1}=F_{\fC_2}\cdot G_{\fC_2,\fC_1}\cdot E_{\fC_1}
\end{align}
where 
\begin{align*}
    F_{\fC_2}=\HallEnv_{\fC_2}^{-1}\cdot \: e^{\sT}(N^-_{\fC_2}),\quad G_{\fC_2,\fC_1}=\frac{e^{\sT}(N^-_{\fC_1})}{e^{\sT}(N^-_{\fC_2})},\quad E_{\fC_1}=e^{\sT}(N^-_{\fC_1})^{-1}\cdot \HallEnv_{\fC_1}\:.
\end{align*}
Note that $G_{\fC_2,\fC_1}$ is a diagonal operator, $F_{\fC_2}$ and $E_{\fC_1}$ are triangular with unit diagonal with respect to the partial order given by $\fC_2$ and $\fC_1$ respectively.
\end{Remark}

By construction, for a sequence of chambers $\fC_1,\fC_2,\ldots,\fC_n$, $R$-matrices satisfy the following \textit{braid relations}:
\begin{align}\label{braid}
    R_{\fC_1,\fC_n} R_{\fC_n,\fC_{n-1}}\cdots R_{\fC_2,\fC_1}=\id.
    % \quad R^{\mathsf s}_{\fC_1,\fC_n} R^{\mathsf s}_{\fC_n,\fC_{n-1}}\cdots R^{\mathsf s}_{\fC_2,\fC_1}=\id.
\end{align}

\begin{Example}\label{ex R-matrix for quiver}
Consider a decomposition of framing vector spaces  
\begin{align*}
    \underline{\bd}=\sum_{i=1}^n \underline{\bd}^{(i)}.
\end{align*}
This gives an $\sA\cong (\bC^*)^n$ action by $\underline{\bd}=\sum_{i=1}^n z_i\,\underline{\bd}^{(i)}$, where $K_\sA(\pt)=\bQ[z_i^\pm]_{i=1}^n$. Define
\begin{align*}
    \cM_\theta(\underline{\bd}):=\bigsqcup_{\bv\in \bN^{Q_0}} \cM_\theta(\bv,\underline{\bd}),
\end{align*}
and denote its critical cohomology as well as $\bN^{Q_0}$-graded components by
\begin{align}\label{state space}
\mathcal H_{\underline{\bd}}^{\sw}:=H^\sT(\cM_\theta(\underline{\bd}),\sw),\quad \mathcal H_{\underline{\bd}}^{\sw}(\bv):=H^\sT(\cM_\theta(\bv,\underline{\bd}),\sw),
% \\
% &\mathcal K_{\underline{\bd}}^{\sw}:=K^\sT(\cM_\theta(\underline{\bd}),\sw),\quad \mathcal K_{\underline{\bd}}^{\sw}(\bv):=K^\sT(\cM_\theta(\bv,\underline{\bd}),\sw).
\end{align}
We have an induced $\sA$-action such that 
\begin{align*}
    \cM_\theta(\underline{\bd})^{\sA}\cong\prod_{i=1}^n \cM_\theta(\underline{\bd}^{(i)}).
\end{align*}
Assume that 
\begin{align}\label{eq potential splits}
    \sw|_{\cM_\theta(\underline{\bd})^{\sA}}=\boxplus_{i=1}^n\sw_i,\quad \sw_i\colon \cM_\theta(\underline{\bd}^{(i)})\to \bC,
\end{align}
then by Thom-Sebastiani isomorphism, we have
\begin{align*}
    H^\sT(\cM_\theta(\underline{\bd})^\sA,\sw)_{\mathrm{loc}}\cong \bigotimes_{i=1}^n \mathcal H^{\sw_i}_{\underline{\bd}^{(i)},\mathrm{loc}},\quad \text{and}\quad R_{\fC_2,\fC_1}\in \End\left(\bigotimes_{i=1}^n \mathcal H^{\sw_i}_{\underline{\bd}^{(i)},\mathrm{loc}}\right),
\end{align*}
where the second tensor products are taken over the base field $\Frac(H_\sT(\pt))$. 
\end{Example}
\begin{Remark}\label{rmk purity and tensor prod}
If the derived global sections $\mathrm{R}\Gamma(\cM_\theta(\underline{\bd}^{(i)}),\sw_i)$ are pure for all $i$, then the Thom-Sebastiani isomorphism holds without localization (e.g.~\cite[Thm.~ A.1.10]{Zhu}), that is,
\begin{align*}
    H^\sT(\cM_\theta(\underline{\bd})^\sA,\sw )\cong \mathcal H^{\sW_1}_{\underline{\bd}^{(1)}}\otimes_{H_\sT(\pt)}\cdots\otimes_{H_\sT(\pt)}\mathcal H^{\sw_n}_{\underline{\bd}^{(n)}}\:,
\end{align*}
and in fact each $\mathcal H^{\sw_i}_{\underline{\bd}^{(i)}}$ is a free $H_\sT(\pt)$-module. For example, $\mathrm{R}\Gamma(\cM_\theta(\underline{\bd}^{(i)}),\sw_i)$ is pure when $\sw_i=0$, or $\sw_i$ admits dimensional reduction to Borel-Moore homology of a variety that is pure (e.g.~Nakajima quiver varieties).
\end{Remark}

\subsection{Root \texorpdfstring{$R$}{R}-matrices}
A general $R$-matrix is a product of $R$-matrices corresponding to pairs of chambers separated by a wall. Namely, let $\fC,\fC'$ be two chambers separated by a wall $\alpha=0$. Here $\alpha$ is a root and we may assume $\alpha(\fC)>0$. Let $\sA_\alpha\subset \sA$ be the subtorus with Lie algebra $\ker(\alpha)$. We denote
\begin{align*}
    X^\alpha=X^{\sA_\alpha}.
\end{align*}
We note that
\begin{align*}
    X^\alpha=\bigsqcup_{(\phi\colon \sA_\alpha\to G)/\sim} X^{\alpha,\phi},\,\,\, \text{ where }\, X^{\alpha,\phi}=R(\bv,\underline{\bd})^{\sA_\alpha,\phi}/\!\!/_\theta G^{\phi(\sA_\alpha)}.
\end{align*}
For the $\sA/\sA_\alpha$-action on $X^\alpha$, there are two chambers, namely $\alpha\gtrless 0$. We denote the (cohomological) \textit{root} $R$-\textit{matrix}: 
\begin{align*}
    R_\alpha=R_{<0,>0}\in \End(H^{\sT/\sA_\alpha}(X^\sA,\sw)_{\mathrm{loc}}).
\end{align*}
% For a collection $\mathsf{\lambda}=\left(\lambda_\phi\in \mathrm{Char}( G^{\phi(\sA_\alpha)})\otimes_\bZ \bR\right)_{\phi\colon \sA_\alpha\to G}$ of generic slopes, we denote
% the ($K$-theoretic) \textit{root} $R$-\textit{matrix}: 
% \begin{align*}
%     R^{\lambda}_\alpha=\bigoplus_{(\phi\colon \sA_\alpha\to G)/\sim} R^{\lambda_\phi}_{<0,>0}\in \End(K^{\sT/\sA_\alpha}(X^{\sA},\sw )_{\mathrm{loc}}).
% \end{align*}

\begin{Proposition}\label{prop root R-matrix}
Under the same assumptions as Theorem \ref{thm AFSQV}, we have
\begin{align}
    R_{\fC',\fC}=R_\alpha.
\end{align}
\end{Proposition}

\begin{proof}
We have $R(\bv,\underline{\bd})^{\sA_\alpha,\phi}=R(Q_\phi,\bv_\phi,\underline{\bd_\phi})$ for a quiver $Q_\phi$ with $\bv_\phi,\bd_{\phi,\In},\bd_{\phi,\Out}\in \bN^{\bQ_{\phi,0}}$ determined by the homomorphism $\phi\colon \sA_\alpha\to G$. The assumptions in Theorem \ref{thm AFSQV} imply that $Q_\phi$ is symmetric and  $\bd_{\phi,\In,i}\geqslant\bd_{\phi,\Out,i}$ for all $i\in Q_{\phi,0}$ and the induced $\sA/\sA_\alpha$ action on $R(Q_\phi,\bv_\phi,\underline{\bd_\phi})$ is pseudo-self-dual. Then it follows from the triangle lemma  that 
\begin{align*}
    R_{\fC',\fC}&=\HallEnv_{\fC'}^{-1}\circ \HallEnv_{\fC}\\
    &=\HallEnv_{<0}^{-1}\circ \HallEnv_{\fC_\alpha}^{-1}\circ \HallEnv_{\fC_\alpha}\circ \HallEnv_{>0}\\
    &=\HallEnv_{<0}^{-1}\circ  \HallEnv_{>0}\\
    &=R_\alpha,
\end{align*}
where $\HallEnv_{\fC_\alpha}\colon H^\sT(X^{\alpha},\sw )\to H^\sT(X,\sw)$ is the Hall envelope for the torus $\sA_\alpha$ in the chamber $\fC_\alpha$.
\end{proof}

\begin{Example}\label{ex root R-matrix}
Take $n=2$ in Example \ref{ex R-matrix for quiver}, and pick chambers $\fC_\pm=\{z_1\gtrless z_2\}$, then $R_{-,+}(z_1,z_2)$ is a root $R$-matrix corresponding to $\sA/\bC^*_{\mathrm{diag}}$-action where $\bC^*_{\mathrm{diag}}$ is the diagonal torus $\{z_1=z_2\}$, in particular, it only depends on the difference of spectral parameters:
\begin{align*}
    R_{-,+}(z_1,z_2)=R_{-,+}(z_1-z_2).
    % \quad R^{\mathsf s}_{-,+}(z_1,z_2)=R^{\mathsf s}_{-,+}(z_1/z_2).
\end{align*}
Assume moreover that $\bd^{(i)}_{\In,j}\geqslant \bd^{(i)}_{\Out,j}$ for all $j\in Q_0$ and $i\in \{1,2\}$, then the $\sA$ action is pseudo-self-dual (see Definition \ref{def pseudo-self-dual}). According to \cite[Prop.~9.23]{COZZ}, Hall envelopes agree with stable envelopes in this case (here $\sA/\bC^*_{\mathrm{diag}}$ is one dimensional, the existence of stable envelopes is guaranteed by \cite[Prop.~3.23]{COZZ}), and we have
\begin{align*}
 R_{-,+}(z_1-z_2)=(\Stab_{-})^{-1}\circ \Stab_+\,.
 % \quad R^{\mathsf s}_{-,+}(z_1/z_2)=(\Stab^{\mathsf s}_{-})^{-1}\circ \Stab^{\mathsf s}_+.
\end{align*}
\end{Example}

\subsection{Covers and factorizations}

Assume that $\widehat Q\to Q$ is a free abelian cover of quivers such that $\widehat Q$ is symmetric (so $Q$ is also symmetric). Let $\Gamma\cong \bZ^r$ be the covering group. We fix a lift of vertices
\begin{align*}
    Q_0\hookrightarrow \widehat Q_0, \qquad i\mapsto \widehat i,
\end{align*}
and set the image to be $\Gamma$-degree zero, then any other vertex in $\widehat Q$ can be uniquely written as $\gamma\widehat i$ for $\gamma\in \Gamma$ and $i\in Q_0$. These data determine an embedding of groups
\begin{align*}
    (\bC^*)^r\cong \Gamma^{\wedge}\hookrightarrow G_{\mathrm{edge}}=\prod_{i,j\in Q_0}\GL(\mathsf Q_{ij}),
\end{align*}
such that $\Gamma^\wedge$-weight $\gamma_a$ for an arrow $a\in Q_1$ is determined by requiring a lift of $a:i\to j$ being
\begin{align}
    \widehat{i}\to \gamma_a^{-1} \widehat{j}.
\end{align}
We have
\begin{align*}
    \cM_{\theta}(Q,\underline{\bd})^{\Gamma^{\wedge}}=\cM_{\theta}(\widehat Q,\underline{\widehat\bd}),
\end{align*}
where $\underline{\widehat\bd}$ is the degree-zero lift of framing vector spaces $\underline{\bd}$, namely
\begin{align*}
    \underline{\widehat\bd}_{\gamma\widehat i}=\begin{cases}
        \underline{\bd}_{i}, & \gamma=1,\\
        \underline{\mathbf 0}, & \gamma \neq 1.
    \end{cases}
\end{align*}
Consider a decomposition of framing vector spaces:
\begin{align*}
    \underline{\bd}= u\,\underline{\bd}^{(1)}+\underline{\bd}^{(2)},
\end{align*}
where $u$ is the equivariant parameter in $K_{\sA_1}(\pt)\cong \bQ[u^\pm]$ for a torus $\sA_1\cong \bC^*$. Denote 
$$R_{\underline{\bd}^{(1)},\underline{\bd}^{(2)}}(u)
% \quad R^{\mathsf s}_{\underline{\bd}^{(1)},\underline{\bd}^{(2)}}(u)
$$ the root $R$-matrix associated to the $\sA_1$ action. Here we regard $u$ as an additive variable in $H_{\sA_1}(\pt)\cong \bQ[u]$.

We take $\sA_2= \Gamma^{\wedge}\times \sA_1$. Note that $\sA_2$ action on $R(Q,\bv, \underline{\bd})$ is pseudo-self-dual. Choose a generic vector $t\in \Lie(\Gamma^{\wedge})_\bR$ and represent \textit{root} $R$-\textit{matrices in the stable basis}:
\begin{align*}
    R_{\underline{\bd}^{(1)},\underline{\bd}^{(2)}}(u)_{\mathrm{stab}}:=
    (\HallEnv_{-t})^{-1}R_{\underline{\bd}^{(1)},\underline{\bd}^{(2)}}(u)\,\HallEnv_{t}\,.
    % \quad R^{\mathsf s}_{\underline{\bd}^{(1)},\underline{\bd}^{(2)}}(u)_{\mathrm{stab}}:=(\HallEnv^{\mathsf s_-}_{-t})^{-1}R^{\mathsf s}_{\underline{\bd}^{(1)},\underline{\bd}^{(2)}}(u)\,\HallEnv^{\mathsf s_+}_{t},
\end{align*}
% where 
% \begin{align*}
% \mathsf s_+\big|_{\cM_{\theta}(Q,\bv^{(1)},\underline{\bd}^{(1)})\times \cM_{\theta}(Q,\bv^{(2)},\underline{\bd}^{(2)})}=\mathsf s+\sum_{i,j\in Q_0}\left[\left(\frac{\mathsf C_{ij}}{4}\bv^{(2)}_j-\frac{\bd^{(2)}_{\Out,i}}{2}\right)\det(V^{(1)}_i)-\left(\frac{\mathsf C_{ij}}{4}\bv^{(1)}_i-\frac{\bd^{(1)}_{\In,j}}{2}\right)\det(V^{(2)}_j)\right],\\
% \mathsf s_-\big|_{\cM_{\theta}(Q,\bv^{(1)},\underline{\bd}^{(1)})\times \cM_{\theta}(Q,\bv^{(2)},\underline{\bd}^{(2)})}=\mathsf s+\sum_{i,j\in Q_0}\left[\left(\frac{\mathsf C_{ij}}{4}\bv^{(1)}_j-\frac{\bd^{(1)}_{\Out,i}}{2}\right)\det(V^{(2)}_i)-\left(\frac{\mathsf C_{ij}}{4}\bv^{(2)}_i-\frac{\bd^{(2)}_{\In,j}}{2}\right)\det(V^{(1)}_j)\right].
% \end{align*}
% Here $(\mathsf C_{ij})_{i,j\in Q_0}$ is the Cartan matrix of $Q$, with $\mathsf C_{ij}=2\delta_{ij}-\#(i\to j)-\#(j\to i)$.

% Denote $\widehat{\mathsf C}$ the Cartan matrix of $\widehat Q$.

\begin{Proposition}
Assume $\bd^{(i)}_{\In,j}\geqslant \bd^{(i)}_{\Out,j}$ for all $j\in Q_0$ and $i\in \{1,2\}$, then in the stable basis we have
\begin{align}
    R_{\underline{\bd}^{(1)},\underline{\bd}^{(2)}}(u)_{\mathrm{stab}}=\prod_{\gamma\in \Gamma}^{\longleftarrow} \widehat R_{\gamma\underline{ \widehat\bd}^{(1)},\underline{ \widehat\bd}^{(2)}}(u-\gamma),
    % \qquad R^{\mathsf s}_{\underline{\bd}^{(1)},\underline{\bd}^{(2)}}(u)_{\mathrm{stab}}=\prod_{\gamma\in \Gamma}^{\longleftarrow}  \widehat R^{\mathsf s_\gamma}_{\gamma\underline{ \widehat\bd}^{(1)},\underline{ \widehat\bd}^{(2)}}(u\gamma^{-1}),
\end{align}
where the product is ordered such that $\gamma(t)$ is increasing from left to right.
Here $\widehat R_{\gamma\underline{ \widehat\bd}^{(1)},\underline{ \widehat\bd}^{(2)}}(u)$ is the root $R$-matrix for the quiver $\widehat Q$ with decomposition of framing 
$u\gamma\underline{\widehat\bd}^{(1)}+ \underline{\widehat\bd}^{(2)}$.
% and the slope is
% \begin{align*}
%     \mathsf s_\gamma\bigg|_{\cM_{\theta}\left(\widehat Q,\widehat\bv,\:\gamma\underline{\widehat\bd}^{(1)}+ \:\underline{\widehat\bd}^{(2)}\right)}=\mathsf s+\sum_{i,j\in \widehat Q_0}\left(\frac{\widehat{\mathsf C}^-_{ij}-\widehat{\mathsf C}^+_{ij}}{4}\widehat\bv_j+\frac{\widehat{\bd}^+_{\In,i}-\widehat{\bd}^-_{\Out,i}}{2}\right)\det(\widehat V_i)
% \end{align*}
% where
% \begin{align*}
%     \widehat{\mathsf C}^\pm_{i,j}:=\sum_{\substack{g\in \Gamma\\ g(t)<0}} \widehat{\mathsf C}_{i,g^{\pm 1}j},\quad \underline{\widehat{\bd}}^\pm:=\sum_{\substack{g\in \Gamma\\ g(t)<0}} g^{\pm 1}\: (\gamma \widehat{\bd}^{(1)}+\widehat{\bd}^{(2)}).
% \end{align*}
\end{Proposition}

\begin{proof}
This follows from Theorem \ref{thm AFSQV} and \cite[Lem.~9.6]{COZZ}.
\end{proof}

% \subsubsection{Unitarity}
% Let
% \begin{align*}
%     R_{-\alpha}=R_{>0,<0}\quad\text{and}\quad  R^{\lambda}_{-\alpha}=\bigoplus_{(\phi\colon \sA_\alpha\to G)/\sim} R^{\lambda_\phi}_{>0,<0}
% \end{align*}
% be the wall R-matrices in the opposite chamber.

% \begin{Proposition}
% Let equivariant parameters be $H^{\sA/\sA_{\alpha}}(\pt)=\bQ[u]$ and $K^{\sA/\sA_{\alpha}}(\pt)=\bQ[t^{\pm}]$. The root R-matrices are unitary in the sense that
% \begin{align}
%     R_\alpha(u)^{-1}=R_{-\alpha}(-u),\quad R^{\lambda}_{\alpha}(t)^{-1}=R^{\lambda}_{-\alpha}(t^{-1}).
% \end{align}
% \end{Proposition}

\subsection{Yang-Baxter equations}

Take $n=3$ in Example \ref{ex R-matrix for quiver}, and assume moreover that $\bd^{(i)}_{\In,j}\geqslant \bd^{(i)}_{\Out,j}$ for all $j\in Q_0$ and $i\in \{1,2,3\}$. Then the $\sA\cong (\bC^*)^3$ action given by $\underline{\bd}=\sum_{i=1}^3 z_i\,\underline{\bd}^{(i)}$ is pseudo-self-dual (see Definition \ref{def pseudo-self-dual}). The roots are given by 
\begin{align*}
    \alpha_{ij}=\{z_i=z_j\},\quad i\neq j\in \{1,2,3\},
\end{align*}
and the chambers are
\begin{align*}
    \fC_{ijk}=\{z_i>z_j>z_k\}, \quad \{i,j,k\}=\{1,2,3\}.
\end{align*}
Assuming \eqref{eq potential splits}, i.e. $$\sw|_{\cM_\theta(\underline{\bd})^{\sA}}=\boxplus_{i=1}^3\sw_i.$$
  
% \subsubsection{Cohomological case}

The root $R$-matrix is 
\begin{align*}
    R_{ij}(z_i-z_j):=R_{\alpha_{ij}}(z_i-z_j)\in \id_{\mathcal H^{\sw_k}_{\underline{\bd}^{(k)}}}\otimes\End\left(\mathcal H^{\sw_i}_{\underline{\bd}^{(i)}}\otimes\mathcal H^{\sw_j}_{\underline{\bd}^{(j)}}\right)_\loc,\qquad k\notin\{i,j\}.
\end{align*}
The braid relation \eqref{braid} implies the following \textit{Yang-Baxter equation}.
\begin{Proposition}
Notations as above, we have 
\begin{align}
    R_{12}(z_1-z_2)R_{13}(z_1-z_3)R_{23}(z_2-z_3)=R_{23}(z_2-z_3)R_{13}(z_1-z_3)R_{12}(z_1-z_2).
\end{align}
\end{Proposition}

\subsection{Classical \texorpdfstring{$R$}{R}-matrices}
We start with a root $R$-matrix $R(u)=R_{-,+}(u)$ (Example \ref{ex root R-matrix}) in the cohomology setting with the additional assumption that the framings are symmetric: 
$$\bd^{(i)}_{\In}=\bd^{(i)}_{\Out}. $$ 
Here we choose the torus such that $\sA=\bC^*_u$, $\sT=\sT_0\times \sA$, so that the framing decomposes as $u\bd^{(1)}+\bd^{(2)}$. 

For the following discussions, we also include the  normalizer $\epsilon$ given by \eqref{normalizer_yangian} in the stable envelopes \cite[Def.~3.5]{COZZ}. Furthermore, we invoke a sign twist $\Sigma$ such that it acts on $H^{\sT}(\cM_\theta(\bd^{(1)})\times\cM_\theta(\bd^{(2)}),\sw^{(1)}\boxplus\sw^{(2)})$ by the sign $(-1)^{|\bv^{(1)}|\cdot |\bv^{(2)}|}$, where $|\cdots|$ is the $\bZ/2$-grading \eqref{equ on gra}. 

Applying the Gauss decomposition \eqref{gauss decomp_general} to $\Sigma\cdot R(u)$ and we get 
$$\Sigma\cdot R(u)=F(u)\,G(u)\,E(u). $$ Moreover, each factor is of form $\id+O(u^{-1})$, so $\Sigma\cdot R(u)$ is of the following form:
\begin{align}\label{expansion_classical r}
    \Sigma\cdot R(u)=\id+\frac{\pmb r}{u}+O(u^{-2}),
\end{align}
where $\pmb r$ is a linear operator acting on $H^{\sT}(\cM_\theta(\bd^{(1)})\times\cM_\theta(\bd^{(2)}),\sw^{(1)}\boxplus\sw^{(2)})$. 
We call $\pmb r$ the \textit{classical} $R$-\textit{matrix}. It follows from Yang-Baxter equation that $\pmb r$ satisfies the \textit{classical Yang-Baxter equations}:
\begin{align}\label{CYBE}
    [\pmb r_{12},\pmb r_{13}+\pmb r_{23}]=0,\qquad [\pmb r_{23},\pmb r_{12}+\pmb r_{13}]=0.
\end{align}
Expanding $E(u)$, $G(u)$, and $F(u)$ to the order $O(u^{-1})$, we find
\begin{align*}
    \pmb r=\pmb r_{\diag}+\pmb r_{\text{off-diag}}.
\end{align*}
Here 
\begin{itemize}
    \item the diagonal part $\pmb r_{\diag}$ is 
    \begin{align}\label{eq classical r}
        \sum_{i\in Q_0}\left(c_1(\mathsf D_{\In,i}^{(1)})-c_1(\mathsf D_{\Out,i}^{(1)})\right)\otimes \bv^{(2)}_i+\sum_{i\in Q_0}\bv^{(1)}_i\otimes \left(c_1(\mathsf D_{\In,i}^{(2)})-c_1(\mathsf D_{\Out,i}^{(2)})\right)-\sum_{a:i\to j}t_a\left(\bv^{(1)}_i\otimes \bv^{(2)}_j+\bv^{(1)}_j\otimes \bv^{(2)}_i\right),
    \end{align}
       where $t_a$ is the $\sT_0$-equivariant weight of the arrow $a:i\to j$, and $c_1(-)$ is the $\sT_0$-equivariant 1st Chern class,
    \item the off-diagonal part $\pmb r_{\text{off-diag}}$ is given by critical convolution using $\can(\pmb \rho)$, where
    \begin{align*}
        \pmb \rho\in H^{\sT_0}\left(X^{\sA}\times_{X^{\sA}_0} X^{\sA}\right)\;\text{ for }\;X^{\sA}=\cM_\theta(\bd^{(1)})\times\cM_\theta(\bd^{(2)})\; \text{ and }\; X^{\sA}_0=\cM_0(\bd^{(1)})\times\cM_0(\bd^{(2)})
    \end{align*}
    is an integral class of homological degree $\dim F+\dim F'-2$ on the connected component $F'\times_{X^{\sA}_0} F\subset X^{\sA}\times_{X^{\sA}_0} X^{\sA}$. $\pmb \rho$ is determined by
    \begin{equation}\label{r mat corr}
    \begin{gathered}
    (F'\times F\hookrightarrow X\times F)^![\overline{\Attr}_+(\Delta_F)]=\pm u^{\rk N^-_{F'/X}-1}\pmb \rho+\text{lower degree terms}\;,\;\text{ if $F\prec F'$,}\\
        (F'\times F\hookrightarrow X\times F)^![\overline{\Attr}_-(\Delta_F)]=\pm u^{\rk N^+_{F'/X}-1}\pmb \rho+\text{lower degree terms}\;,\;\text{ if $F'\prec F$,}
    \end{gathered}
    \end{equation}
    where the signs are determined by the normalizer $\epsilon$ and sign twist $\Sigma$.
    In particular, $\pmb \rho$ does not depend on the choice of function $\sw^{(1)}$ or $\sw^{(2)}$.
\end{itemize}

\begin{Remark}
In the case of a tripled quiver $\widetilde{Q}$ with canonical cubic potential $\widetilde{\sW}$ \eqref{equ on can cub pot}, we have $\sT_0=\bC^*_{\hbar}$ and the diagonal part reads
\begin{align*}
    \pmb r_{\diag}=\hbar\sum_{i\in Q_0} \left(\bd_{i}^{(1)}\otimes \bv^{(2)}_i+\bv_{i}^{(1)}\otimes \bd^{(2)}_i\right)-\hbar\sum_{i,j\in Q_0}\mathsf C_{i,j}\bv^{(1)}_{i}\otimes \bv^{(2)}_{j},
\end{align*}
where $\mathsf C_{i,j}$ is the Cartan matrix for the original quiver ${Q}$. This recovers \cite[(4.21)]{MO} by dimensional reduction.
\end{Remark}

\subsection{Action of Steinberg correspondence}\label{subsec steinberg corr}
Let $X=\cM_\theta(\bv,\bd)$ be a symmetric quiver variety. It is known \cite[Thm.~1.1]{COZZ3} that
\begin{itemize}
    \item the affinization map $X\to X_0=\cM_0(\bv,\bd)$ is birational and small,
    \item there exists a flat $\sA$-equivariant morphism $\pi\colon X\to B$ to an affine space $B$ (with trivial $\sA$-action) such that every fiber of $\pi$ is an irreducible normal variety with symplectic singularities,
    \item there exists an $\sA$-invariant open dense subset $U\subset B$ with $\pi|_{\pi^{-1}(U)}\colon \pi^{-1}(U)\to U$ smooth and affine.
\end{itemize}
Since $X\to X_0$ is birational and small, we have $\dim X\times_{X_0}X=\dim X$ and the diagonal $\Delta_X$ is the unique irreducible component of $X\times_{X_0}X$ with dimension $\dim X$. It follows that there exists a decomposition
\begin{align*}
    H^{\sT}_{2\dim X-2}(X\times_{X_0}X)\cong  H^{\sT}_{2\dim X-2}(\Delta_X)\oplus \left(\bigoplus_i \bC\cdot [Z_i]\right),
\end{align*}
where $\{Z_i\}$ are irreducible components of $X\times_{X_0} X$ of dimension $\dim X-1$.

\begin{Definition}\label{def of Steinberg corr}
A \textit{Steinberg correspondence} is a class $\mathcal L\in H^{\sT}_{2\dim X-2}(X\times_{X_0}X)$ of the form 
\begin{align*}
    \mathcal L=\lambda\,[\Delta_X]+\sum_ia_i[Z_i],\;\text{ with }\; \lambda\in H_\sT^2(\pt),\;a_i\in \bC.
\end{align*}
\end{Definition}

We give some criteria for a class in $H^{\sT}_{2\dim X-2}(X\times_{X_0}X)$ being a Steinberg correspondence.

\begin{Lemma}\label{lem action on [X]}
A class $\mathcal L\in H^{\sT}_{2\dim X-2}(X\times_{X_0}X)$ is a Steinberg correspondence if and only if the induced convolution on $H^\sT(X)$ maps the fundamental class $[X]$ to $\lambda[X]$ for some $\lambda\in H_\sT^2(\pt)$.
\end{Lemma}

\begin{proof}
The convolution induced by $[Z_i]$ maps $[X]$ to $\pr_{1*}([Z_i])$, where $\pr_1\colon X\times_{X_0}X\to X$ is the projection to the first factor. We claim that the image of $Z_i$ in $X$ has dimension $\leqslant \dim X-2$, so $\pr_{1*}([Z_i])=0$, and the lemma then follows.
To prove the claim, we note that $\pi\colon \pi^{-1}(U)\to U$ is a smooth affine morphism and the affinization map is an isomorphism when restricted to $\pi^{-1}(U)$. This implies that $Z_i\subseteq X\times_{X_0} X\times_B (B\setminus U)$. As $B$ is affine, we have $$X\times_{X_0} X\times_B \{b\}\cong \pi^{-1}(b)\times_{X_0} \pi^{-1}(b). $$ 
For a generic $b\in B\setminus U$, $Z_i\times_{B}\{b\}\subseteq \pi^{-1}(b)\times_{X_0}\pi^{-1}(b)$ does not lie in the diagonal, as otherwise $Z_i$ would lie in the diagonal of $X\times_{X_0}X$. This implies that the image of $Z_i\times_{B}\{b\}$ in $\pi^{-1}(b)$ has dimension
$<\pi^{-1}(b)$. Therefore the image of $Z_i$ in $X$ has 
dimension $<\dim \pi^{-1}(b)+\dim(B\setminus U)\leqslant\dim X-1$.
\end{proof}

\begin{Lemma}\label{lem image in H^A}
A class $\mathcal L\in H^{\sT}_{2\dim X-2}(X\times_{X_0}X)$ is a Steinberg correspondence if and only its image in $H^{\sA}_{2\dim X-2}(X\times_{X_0}X)$ is a Steinberg correspondence.
\end{Lemma}

\begin{proof}
By assumption and Lemma \ref{lem action on [X]}, the image of the class $\mathcal L([X])$ in $H^\sA(X)$ is $\lambda[X]$ for some $\lambda\in H^2_\sA(\pt)$. By the freeness\footnote{The quiver variety $X$ has a $\C^*$-action scaling all arrows of the quiver, whose fixed locus is smooth and proper. By an induction argument on the Bialynicki-Birula strata, one can show the freeness by reducing to the fixed locus and using \cite[Thm.~ A.1.10]{Zhu}.}:
$$H^\sT(X)\cong H(X)\otimes H_\sT(\pt),$$ we conclude that $\mathcal L([X])=\tilde\lambda[X]$ in $H^\sT(X)$ for some $\sT$-weight $\tilde\lambda$.
\end{proof}

\begin{Lemma}\label{lem away from affine locus}
Let $\mathcal L\in H^{\sT}_{2\dim X-2}(X\times_{X_0}X)$ be a class, and assume that the image of $\mathcal L$ in $H^{\sA}(X\times_{X_0}X)$ is in the image of the pushforward map $H^{\sA}(X\times_{X_0}X\times_B(B\setminus U))\to H^{\sA}(X\times_{X_0}X)$. Then $\mathcal L$ is a Steinberg correspondence.
\end{Lemma}

\begin{proof}
By Lemma \ref{lem image in H^A}, it is enough to show that every class in $H^{\sA}_{2\dim X-2}(X\times_{X_0}X\times_B(B\setminus U))$ induces a Steinberg correspondence after pushing forward to $H^{\sA}(X\times_{X_0}X)$. The irreducible components $\{Z_i\}$ of $X\times_{X_0}X$ with dimension $\dim X-1$ are contained in $X\times_{X_0}X\times_B(B\setminus U)$. Let codimension one (in $B$) irreducible components of $B\setminus U$ be $C_1,C_2,\ldots$, then irreducible components of $\Delta_X\times_B(B\setminus U)\cong \pi^{-1}(B\setminus U)$ with dimension $\dim X-1$ are exactly $\pi^{-1}(C_1),\pi^{-1}(C_2),\ldots$, because $\pi$ is flat with irreducible fibers. Since $\sA$ acts on $B$ trivially, we have $H^2_\sA(B)=0$, and it follows that $[\pi^{-1}(C_1)]=\pi^*[C_1],[\pi^{-1}(C_2)]=\pi^*[C_2],\ldots$ are zero in $H^{\sA}(\Delta_X)$. Then a class in the image of the pushforward map $H^{\sA}_{2\dim X-2}(X\times_{X_0}X\times_B(B\setminus U))\to H^{\sA}_{2\dim X-2}(X\times_{X_0}X)$ can be written as $\sum_i a_i [Z_i]$ with $a_i\in \bC$. In particular, it is a Steinberg correspondence.
\end{proof}

The main result of this subsection is the following compatibility result between stable envelopes and Steinberg correspondences.

\begin{Theorem}\label{thm Stab and Steinberg}
Given a Steinberg correspondence $\mathcal L$, 
then there exists an integral class $\mathcal L_\sA\in H^\sT(X^\sA\times_{X_0} X^\sA)$ such that for any chamber $\fC$, the diagram 
\begin{equation*}
\xymatrix{
H^{\sT}(X^\sA,\sw ) \ar[r]^{\,\,\Stab_\fC} \ar[d]_{(-1)^\sharp\mathcal L_\sA} &  H^{\sT}(X,\sw) \ar[d]^{\mathcal L} \\
H^{\sT}(X^\sA,\sw ) \ar[r]^{\,\,\Stab_\fC} &  H^{\sT}(X,\sw) 
}
\end{equation*}
commutes after $\sA$-localization,
where vertical arrows are convolutions induced by $\mathcal L_\sA$ (up to sign) and by $\mathcal L$ respectively. For a pair of connected components $F_1,F_2\subset X^\sA$, the sign $(-1)^\sharp$ for $H^{\sT}(F_2,\sw)\to H^{\sT}(F_1,\sw)$ is determined by 
\begin{align*}
    (-1)^\sharp=\frac{\prod\text{$\sA$-weights in $N^-_{F_1/X}$}}{\sqrt{\prod\text{$\sA$-weights in $N_{F_1/X}$}}}\times \frac{\prod\text{$\sA$-weights in $N^+_{F_2/X}$}}{\sqrt{\prod\text{$\sA$-weights in $N_{F_2/X}$}}}\,,
\end{align*}
where we choose a square root and fix it for all chambers.
\end{Theorem}

\begin{proof}
By linearity, it is enough to prove for $\mathcal L=\lambda[\Delta_X]$ and for $\mathcal L=[Z_i]$. In the first case, we take $\mathcal L_\sA=\lambda[\Delta_{X^\sA}]$. Let us show the existence of $[Z_i]_\sA$ in the second case. Define 
\begin{align*}
    \mathcal L':=[\Stab_{-\fC}]^{\mathrm{t}}\circ [Z_i]\circ [\Stab_{\fC}],
\end{align*}
and by a similar argument as \cite[Thm.~4.6.1]{MO}, the pushforward along $X\times X$ in the definition of $\mathcal L'$ is proper, so $ \mathcal L'$ is an integral class in $H^{\sT}(X^\sA\times_{X_0} X^\sA)$. Note that for a pair of connected components $F_1,F_2\subset X^\sA$, the homological degree of $\mathcal L'$ on the component $F_1\times_{X_0} F_2$ is $\dim(F_1\times F_2)-2$. By \cite[Lem.~3.29]{COZZ}, $\mathcal L'$ makes the diagram commute for chamber $\fC$.
Next we will show that $\mathcal L'$ is up to a sign independent of the choice of $\fC$.

According to \cite{COZZ3}, we have
\begin{enumerate}
    \item $F_1\to X_0$ is $\sT$-equivariant, and birational and small onto its image (ref.~\cite[Prop.~3.21]{COZZ3}),
    \item $\dim F_1\times_{X_0} F_2\leqslant \frac{1}{2}\dim (F_1\times F_2)-1$ if $F_1\neq F_2$ (ref.~\cite[Rmk.~3.24]{COZZ3}).
    % $\dim F_1\times_{X_0} F_1=\dim F_1$ and the diagonal $\Delta_{F_1}$ is the unique irreducible component of $F_1\times_{X_0} F_1$ with dimension $\dim F_1$.
\end{enumerate}
If $F_1\neq F_2$, then $\mathcal L'$ on the component $F_1\times_{X_0} F_2$ is a top-dimensional cycle class. If $F_1=F_2$, then $\mathcal L'$ on the component $F_1\times_{X_0} F_1$ is of the form (see Lemma \ref{lem on A vs T loc})
\begin{align*}
\mu\cup[\Delta_{F_1}]+\sum_{j=1}^n a_j[H_j],
\end{align*}
where $\mu\in H_{\sT}^2(F_1)$, $a_j\in \bC$, and $\{H_j\}_{j=1}^n$ are irreducible components of $F_1\times_{X_0} F_1$ with dimension $\dim F_1-1$. 
By the freeness $H^\sT(F_1)\cong H(F_1)\otimes H_{\sT}(\pt)$ (see the proof of Lemma \ref{lem on A vs T loc}), there is a decomposition
\begin{align*}
H_{\sT}^2(F_1)\cong H^2(F_1)\oplus H^2_{\sT}(\pt).
\end{align*}
The above decomposition is noncanonical, but the subspace $H^2_{\sT}(\pt)\subset H_{\sT}^2(F_1)$ is canonical. 

We claim that $\mu\in H^2_{\sT/\sA}(\pt)$. The set-theoretic convolution is disjoint from $F_1\times_{X_0} F_1\times_B U$, so $\mu\cup[\Delta_{F_1}]$ in the $\sA$-equivariant cohomology can be written as $\sum_{\alpha=1}^m b_\alpha [\Delta_{F_1\cap\pi^{-1}(D_\alpha)}]$, where $b_\alpha\in \bC$ and $\{D_\alpha\}_{\alpha=1}^m$ are codimension $1$ components of $B\setminus U$. Since $B$ is an affine space with trivial $\sA$ action, $[D_\alpha]=0$ in $H^\sA(B)$. So $[F_1\cap\pi^{-1}(D_\alpha)]=\pi^*([D_\alpha])=0$ in $H^\sA(F_1)$, and this implies that $\mu\cup[\Delta_{F_1}]$ vanishes in $H^\sA(\Delta_{F_1})$. This proves the claim.

From the above discussions, $\mathcal L'$ is independent of $\sA$-equivariant parameters $\mathsf a$, so we can compute it by sending $\mathsf a$ to infinity in any direction. By \cite[Prop.~3.23]{COZZ3}, we have 
$$\dim (F_1\times F_2)\cap Z_i\leqslant \frac{1}{2}\dim (F_1\times F_2)-1\,.$$
Denote by $L_1,L_2,\cdots$ to be the $\frac{1}{2}\dim (F_1\times F_2)-1$ dimensional components of $(F_1\times F_2)\cap Z_i$, then
\begin{align*}
(F_1\times F_2\hookrightarrow X\times X)^*[Z_i]=\sum_{r} f_r [L_r]+\cdots\,,
\end{align*}
where $f_r\in H_{\sA}(\pt)$ is a homogeneous polynomial of degree $\frac{1}{2}\codim_{X\times X}(F_1\times F_2)$, and dots stand for terms of smaller degree. Applying $\sA$-localization and we get 
\begin{align*}
    \mathcal L'=\lim_{\mathsf a\to \infty}\sum_{r}\frac{f_r}{\varepsilon_1\varepsilon_2}[L_r],
\end{align*}
where $\varepsilon_1$ is the product of $\sA$-weights in $N^{-}_{F_1/X}$ and $\varepsilon_2$ is the product of $\sA$-weights in $N^{+}_{F_2/X}$. Since $\lim_{\mathsf a\to \infty}\frac{f_r}{\varepsilon_1\varepsilon_2}$ does not depend on the direction of taking the limit $\mathsf a\to \infty$, we must have $f_r\propto \varepsilon_1\varepsilon_2$. We fix square roots 
\begin{align*}
    \varsigma_1=\sqrt{\prod\text{$\sA$-weights in $N_{F_1/X}$}},\quad \varsigma_2=\sqrt{\prod\text{$\sA$-weights in $N_{F_2/X}$}}
\end{align*}
and define 
\begin{align*}
    \mathcal L_\sA:=\lim_{\mathsf a\to \infty}\sum_{r}\frac{f_r}{\varsigma_1\varsigma_2}[L_r],
\end{align*}
Then $\mathcal L'=(-1)^\sharp \mathcal L_\sA$ and the theorem follows.
\end{proof}

In the proof of Theorem \ref{thm Stab and Steinberg}, we have used the following lemma.

\begin{Lemma}\label{lem on A vs T loc}
Suppose that $Y$ is a symmetric quiver variety with a $\sT$-action, and let $Y\to V$ be a $\sT$-equivariant morphism to an affine $\sT$-variety $V$, which is birational and small onto its image. Let $\gamma$ be a class in $H^\sT_{2\dim Y-2}(Y\times_{V}Y)$ such that $\gamma=0$ in $H^\sA(Y\times_{V}Y)$, then
\begin{align*}
    \gamma=\lambda\,[\Delta_Y]\;\text{ for some $\sT/\sA$-weight $\lambda$}\:.
\end{align*}
\end{Lemma}

\begin{proof}
Since $Y\to V$ is birational and small, $\Delta_Y$ is an irreducible component of $Y\times_{V}Y$ and other components have dimension $<\dim Y$. It follows that there exists a decomposition
\begin{align*}
    H^{\sT}_{2\dim Y-2}(Y\times_{V}Y)\cong  H^{\sT}_{2\dim Y-2}(\Delta_Y)\oplus \left(\bigoplus_i \bC\cdot [H_i]\right)
\end{align*}
where $\{H_i\}$ are irreducible components of $Y\times_{V} Y$ of dimension $\dim Y-1$. Write $$\gamma=c+\sum_{i}a_i[H_i], $$ 
where $c\in H^{\sT}_{2\dim Y-2}(\Delta_Y)$, $a_i\in \bC$. Then $\gamma=0$ in $H^\sA(Y\times_{V}Y)$ implies that $a_i=0$ for all $i$, and $c=0$ in $H^{\sA}_{2\dim Y-2}(\Delta_Y)$. By the freeness $H^\sT(Y)\cong H(Y)\otimes H_\sT(\pt)$, we conclude that $c=\lambda[\Delta_Y]$ for some $\sT/\sA$-weight $\lambda$.
\end{proof}

% \begin{Corollary}\label{cor steinberg corr comm R}
% Let $\mathcal L\in H^{\sT}_{2\dim X-2}(X\times_{X_0}X)$ be a Steinberg correspondence, and $\epsilon\in \{\pm\}^{\Fix_\sA(X)}$ be a normalizer, then 
% \end{Corollary}

\section{Drinfeld type shifted Yangian action on critical cohomology of quiver varieties}\label{sec double of COHA}

In this section, we define 
(Drinfeld type) shifted Yangians from spherical nilpotent critical cohomological Hall algebras (Definition \ref{def shifted yangian}). 
Motivated by works of Nakajima \cite{Nak1}, Negu\c{t} \cite{N2} on Nakajima quiver varieties, we define their action on critical cohomology of quiver varieties with potentials (Theorem \ref{thm borel yangian action}, Theorem \ref{cor shifted yangian action}). This generalizes various works of Varagnolo-Vasserot \cite{VV1, VV2} and Yang-Zhao \cite{YZ3} (Examples \ref{ex on gen symi}, \ref{ex on vv ex}).

We also describe generators of the (Drinfeld type) shifted super Yangians via matrix elements of $R$-matrices (Theorem \ref{thm e f h as R matrix elements}).

%\yl{mention that on tripled quiver with canonical cubic potential, it has dim reduction to previous construction}

We start with the following data:
\begin{itemize}
    \item a quiver $Q=(Q_0,Q_1)$, and denote its path algebra by $\bC Q$,
    \item a torus $\sT_0$ that scales the arrows in $Q$,
    \item a \textit{potential}, i.e. an element $\sW\in \bC Q/[\bC Q,\bC Q]$, such that $\sW$ is $\sT_0$-invariant.
\end{itemize}
For any dimension vector $\bv\in \bZ_{\geqslant 0}^{Q_0}$, $$\sw=\tr(\sW)$$ becomes a $\sT_0\times G$-invariant function on the space of $\bv$-dimensional unframed representations $R(\bv,\mathbf 0)$ ($\mathbf 0$ means framing dimension is $\mathbf 0$), where $G=\prod_{i\in Q_0}\GL(\bv_i)$ is the gauge group that acts on $R(\bv,\mathbf 0)$ naturally. By $G$-invariance, $\sw$ descends to a $\sT_0$-invariant function on the quotient stack $\fM(\bv,\mathbf 0)=[R(\bv,\mathbf 0)/G]$, still denoted by $\sw$.

Given a pair of framing dimension vectors $\underline{\bd}=(\bd_{\In},\bd_{\Out})\in \bZ_{\geqslant 0}^{Q_0}\times \bZ_{\geqslant 0}^{Q_0}$, define the Crawley-Boevey quiver $Q^{\underline{\bd}}$ to be obtained from $Q$ by adding one node $\infty$, and $\bd_{\In ,i}$ many arrows from $\infty$ to $i$ and $\bd_{\Out ,i}$ many arrows from $i$ to $\infty$, for every $i\in Q_0$. We choose an extension of $\sT_0$ action from $Q$ to $Q^{\underline{\bd}}$, then choose a \textit{framed potential} $\sW^{\mathrm{fr}}\in \bC Q^{\underline{\bd}}/[\bC Q^{\underline{\bd}},\bC Q^{\underline{\bd}}]$ such that $\sW^{\mathrm{fr}}$ is $\sT_0$-invariant and its restriction to $Q$ is $\sW$. The function 
\begin{align*}
    \sw^{\mathrm{fr}}=\tr(\sW^{\mathrm{fr}})
\end{align*}
on the space of $\underline{\bd}$-framed representation $R(\bv,\underline{\bd})$ descends to the quotient stack $\fM(\bv,\underline{\bd})=[R(\bv,\underline{\bd})/G]$ as well as the stable locus $\cM_\theta(\bv,\underline{\bd})=R(\bv,\underline{\bd})^{\theta\emph{-}s}/G$, and we denote the induced functions by $\sw^{\mathrm{fr}}$. Starting from now, we renew the notation in \eqref{state space} to
\begin{align*}
\mathcal H_{\underline{\bd}}^{\sW^{\mathrm{fr}}}:=H^\sT(\cM_\theta(\underline{\bd}),\sw^{\mathrm{fr}}),\quad \mathcal H_{\underline{\bd}}^{\sW^{\mathrm{fr}}}(\bv):=H^\sT(\cM_\theta(\bv,\underline{\bd}),\sw^{\mathrm{fr}}).
\end{align*}
For simplicity, we fix a cyclic stability $\theta\in \bQ_{<0}^{Q_0}$ throughout this section.

\subsection{The general case}

% For a general quiver $Q$ with potential, i.e. an element $\sW\in \bC Q/[\bC Q,\bC Q]$, we equip it with a general framing $\underline{\bd}=(\bd_{\In},\bd_{\Out})\in \bZ_{\geqslant 0}^{Q_0}\times \bZ_{\geqslant 0}^{Q_0}$ and extend the cycle $\sW$ to a cycle $\sW^{\mathrm{fr}}$ on the associated Crawley-Boevey quiver, and take $\sw^{\mathrm{fr}}=\tr(\mathsf m^{\mathrm{fr}})$. Take appropriate $\sT_0$ action such that $\sw^{\mathrm{fr}}$ is $\sT_0$-invariant.

Let $\mathcal{H}^{\mathrm{nil}}_{Q,\sW}$ denote the $\sT_0$-\textit{equivariant} 
\textit{nilpotent critical CoHA} associated to $(Q,\sW)$,~i.e.
\begin{align}\label{equ on HQsw}
    \mathcal{H}^{\mathrm{nil}}_{Q,\sW}=\bigoplus_{\bv\in \bN^{Q_0}}\mathcal{H}^{\mathrm{nil}}_{Q,\sW}(\bv),\quad \mathcal{H}^{\mathrm{nil}}_{Q,\sW}(\bv):=H^{\sT_0}(\fM(\bv,\mathbf 0),\sw)_{\fM(\bv,\mathbf 0)^{\mathrm{nil}}},
\end{align}
decorated with cohomological Hall algebra structure of Kontsevich and Soibelman \cite{CoHA}, where $\fM(\bv,\mathbf 0)^{\mathrm{nil}}\subset \fM(\bv,\mathbf 0)$ is the moduli stack of nilpotent unframed representations of $Q$ with dimension vector $\bv$. 

$\mathcal{H}^{\mathrm{nil}}_{Q,\sW}$ is a $\bZ/2$-graded algebra with the grading given by
\begin{align}\label{equ on gra}
    |\bv|:=\sum_{i\in Q_0}|i|\cdot\bv_i,\;\,\, \text{where } |i|=1+\# (i\to i)\pmod{2}.
\end{align}
%\begin{cases}
%     0, & \# (i\to i)\text{ is odd},\\
%     1, & \# (i\to i)\text{ is even}.
% \end{cases}
Throughout this section, we impose the following assumption for simplicity.
\begin{Assumption}
The $\sT_0$-weight on every edge loop of $Q$ is nontrivial.
\end{Assumption}
This assumption implies that $\sT_0$-fixed point of $\mathrm{Rep}(Q,\delta_i,\mathbf 0)\cong \bC^{\mathsf g_i}$ is $\{0\}$, where $\mathsf g_i$ is the number of edge loops at node $i$. Note that the locus of nilpotent representations $\mathrm{Rep}(Q,\delta_i,\mathbf 0)^{\mathrm{nil}}=\{0\}$. By \cite[Prop.~2.4]{COZZ}, we get
\begin{align*}
    \mathcal{H}^{\mathrm{nil}}_{Q,\sW}(\delta_i)=H^{\sT_0}(\fM(\delta_i,\mathbf 0),\sw)_{\fM(\delta_i,\mathbf 0)^{\mathrm{nil}}}\cong H^{\sT_0}([\{0\}/\bC^*]).
\end{align*}
Let $\mathcal L_i$ be the universal line bundle on $\fM(\delta_i,\mathbf 0)\cong [\bC^{\mathsf g_i}/\bC^*]$, and denote a basis of $\mathcal{H}^{\mathrm{nil}}_{Q,\sW}(\delta_i)$ as $\bC[\mathsf t_0]$-module by  $$e_{i,r}:=c_1(\mathcal L_{i})^{r}\cap\left[\fM(\delta_i,\mathbf 0)^{\mathrm{nil}}\right], \quad \forall \, r\in \bZ_{\geqslant 0}, $$  
where $c_1(-)$ denotes the $\sT_0$-equivariant first Chern class. 

The \textit{spherical subalgebra}
$\mathcal{SH}^{\mathrm{nil}}_{Q,\sW}$ is the $\bC[\mathsf t_0]$-subalgebra of $\mathcal{H}^{\mathrm{nil}}_{Q,\sW}$ generated by $\{\mathcal{H}^{\mathrm{nil}}_{Q,\sW}(\delta_i)\}_{i\in Q_0}$.
Let $\mathcal{SH}^{\mathrm{nil},\mathrm{op}}_{Q,\sW}$ denote the $\bZ/2$-graded opposite algebra with product 
$$f\star^{\mathrm{op}}g=(-1)^{|f|\cdot |g|}g\star f,$$ for homogeneous $f,g$. Denote $\left\{f_{i,r}\right\}_{i\in Q_0}^{r\in \bZ_{\geqslant 0}}$ to be the generators of $\mathcal{SH}^{\mathrm{nil},\mathrm{op}}_{Q,\sW}$.

A generalization of \cite[Prop. 3.4]{VV2}, \cite{Nak1} shows that 
\begin{align}\label{sph coha act}
    e_i(u):=\sum_{r\geqslant 0}e_{i,r}u^{-r-1}\mapsto \frac{(-1)^{\bd_{\Out,i}}}{u-c_1(\mathcal L_i)}\,[\overline{\mathfrak{P}}(\bv+\delta_i,\bv,\underline{\bd})]
\end{align}
% \begin{align}\label{sph coha act}
%     e_i(u):=\sum_{r>0}e^{(r)}_iu^{-r}\mapsto \frac{(-1)^{\bd_{\Out,i}+\sum_{i\to j}\bv_j}}{u-c_1(\mathcal L_i)}[\overline{\mathfrak{P}}(\bv+\delta_i,\bv,\underline{\bd})]
% \end{align}
gives rise to an action of $\mathcal{SH}^{\mathrm{nil}}_{Q,\sW}$ on $\mathcal H_{\underline{\bd}}^{\sW^{\mathrm{fr}}}$. Here 
$$\overline{\mathfrak{P}}(\bv+\delta_i,\bv,\underline{\bd})\subset {\mathfrak{P}}(\bv+\delta_i,\bv,\underline{\bd})
\subset \cM(\bv+\delta_i,\underline{\bd})\times \cM(\bv,\underline{\bd})$$ are certain closed subvarieties in the product, where ${\mathfrak{P}}(\bv+\delta_i,\bv,\underline{\bd})$ is the locus consisting of framed quiver representations $(V^+=V\oplus\bC^{\delta_i},D_{\In },D_{\Out})$, $(V,D_{\In },D_{\Out})$ such that there exists a surjective map between framed quiver representations $(V^+,D_{\In },D_{\Out})\twoheadrightarrow(V,D_{\In },D_{\Out})$ which is identity on the framing vector spaces. $\mathcal L_i$ is the line bundle on ${\mathfrak{P}}(\bv+\delta_i,\bv,\underline{\bd})$ whose fiber is the kernel of $(V^+,D_{\In },D_{\Out})\twoheadrightarrow(V,D_{\In },D_{\Out})$. $\overline{\mathfrak{P}}(\bv+\delta_i,\bv,\underline{\bd})$ is the locus in ${\mathfrak{P}}(\bv+\delta_i,\bv,\underline{\bd})$ such that the kernel of $(V^+,D_{\In },D_{\Out})\twoheadrightarrow(V,D_{\In },D_{\Out})$ is a nilpotent unframed $Q$-representation.

Similarly, the map 
\begin{align}\label{sph coha op act}
    f_i(u):=\sum_{r\geqslant 0}f_{i,r}u^{-r-1}\mapsto \frac{(-1)^{|i|\cdot |\bv|}}{u-c_1(\mathcal L_i)}\,[\overline{\mathfrak{P}}(\bv,\bv-\delta_i,\underline{\bd})]^{\mathrm{t}}
\end{align}
% \begin{align}\label{sph coha op act}
%     f_i(u):=\sum_{r>0}f^{(r)}_iu^{-r}\mapsto \frac{(-1)^{|i|\cdot |\bv|+\sum_{j\to i}\bv_j}}{u-c_1(\mathcal L_i)}[\overline{\mathfrak{P}}(\bv,\bv-\delta_i,\underline{\bd})]^{\mathrm{t}},
% \end{align}
gives rise to an action of $\mathcal{SH}^{\mathrm{nil},\mathrm{op}}_{Q,\sW}$ on $\mathcal H_{\underline{\bd}}^{\sW^{\mathrm{fr}}}$. 

Define the free commutative algebra
\begin{align*}
    \mathcal H^0_{Q}:=\bC\left[h_{i,s}\:\big|\: i\in Q_0,s\in \bZ\right],
\end{align*}
and set 
$$h_i(u)=\sum_{s\in \bZ}h_{i,s}u^{-s-1}. $$
% Define the free commutative algebra
% \begin{align*}
%     \mathcal H^0_{Q,\mu}:=\bC\left[h_{i}^{(s_i)}\right]_{i\in Q_0}^{s_i\in \bZ_{>-\mu_i}},
% \end{align*}
% and let it act on $\mathcal H_{\underline{\bd}}^{\sW^{\mathrm{fr}}}$ by tautological classes:
% \begin{align}\label{cartan act}
%     h_i(u):=u^{\mu_i}+\sum_{s>-\mu_i}h_{i}^{(s)}u^{-s}\mapsto u^{\mu_i}c_{-1/u}\left(\sum_{a:i\to j}t_a\mathsf V_j-\sum_{b:j\to i}t_{b}^{-1}\mathsf V_j+\mathsf D_{\Out,i}-\mathsf D_{\In,i}\right),
% \end{align}
% where $\{t_a\}_{a\in Q_1}$ are the $\sT_0$-weights of arrows. 

\begin{Definition}
\label{def borel yangian}
Define the \textit{unmodified Cartan doubled Yangian} $\mathcal{DSH}_\infty(Q,\sW)$ to be the $\bC[\mathsf t_0]$-algebra which is obtained from the free product $\mathcal{SH}^{\mathrm{nil}}_{Q,\sW}*\mathcal{SH}^{\mathrm{nil},\mathrm{op}}_{Q,\sW}*\mathcal H^0_{Q}$ subject to the relations:
\begin{align}\label{ef rel}
    e_{i,r} f_{j,s}-(-1)^{|i|\cdot|j|}f_{j,s}e_{i,r} =\delta_{ij}\gamma_i h_{i,r+s},\,\,\;\text{where }\,\gamma_i:=\prod_{e:i\to i}t_e,
\end{align}
\begin{align}\label{he and hf rel}
h_i(z)e_j(w)=\zeta^{\circ}_{ij}(z)e_j(w)h_i(z),\quad h_i(z)f_j(w)=\zeta^{\circ}_{ij}(z)^{-1}f_j(w)h_i(z).
\end{align}
Here $\zeta^{\circ}_{ij}(z)$ is the ``bond factor'':
\begin{align*}
    \zeta^{\circ}_{ij}(z)=(-1)^{\delta_{ij}+|i|\cdot|j|}\frac{\prod_{a:i\to j}(\sigma_j+t_a-z)}{\prod_{b:j\to i}(z+t_b-\sigma_j)},\;\text{where $\sigma_j$ is the operator } e_{j,r}\mapsto e_{j,r+1},\; f_{j,r}\mapsto f_{j,r+1},\;(r\in \bZ_{\geqslant0}).
\end{align*}
\end{Definition}
It is easy to see that the multiplication map
$$\mathcal{SH}^{\mathrm{nil}}_{Q,\sW}\otimes_{\bC[\mathsf t_0]}\mathcal{SH}^{\mathrm{nil},\mathrm{op}}_{Q,\sW}\otimes\mathcal H^0_{Q}\to \mathcal{DSH}_\infty(Q,\sW)$$ is surjective, but 
the injectivity is not clear.

% \begin{Remark}
% The relation \eqref{ef rel} is equivalent to
% \begin{align*}
%     e_i(z)f_j(w)-(-1)^{|i|\cdot|j|}f_j(w)e_i(z)=\delta_{ij}\gamma_i\left[\frac{h_i(z)-h_i(w)}{z-w}\right]_{<0}.
% \end{align*}
% The subscript $<0$ means first expanding the rational function in $\mathcal Y_\mu(Q,\sW)(\!(w^{-1})\!)(\!(z^{-1})\!)$ then taking truncation to $\mathcal Y_\mu(Q,\sW)\cdot z^{-1}w^{-1}[\![w^{-1},z^{-1}]\!]$.
% \end{Remark}

Let $\mathcal H^0_{Q}$ act on $\mathcal H_{\underline{\bd}}^{\sW^{\mathrm{fr}}}$ by tautological classes:
\begin{align}\label{cartan act}
    h_i(z)\mapsto (-1)^{1+\bv_i+|i|\cdot |\bv|+\sum_{i\to j}\bv_j}z^{\rk \mathsf U_i}c_{-1/z}\left(\mathsf U_i\right),\;\text{ for }\; \mathsf U_i=\sum_{a:i\to j}t_a\mathsf V_j-\sum_{b:j\to i}t_{b}^{-1}\mathsf V_j+\mathsf D_{\Out,i}-\mathsf D_{\In,i}\:,
\end{align}
where $\{t_a\}_{a\in Q_1}$ are $\sT_0$-weights of arrows, $c_{-1/z}(-)$ denotes the $\sT_0$-equivariant Chern polynomial with variable $-1/z$. 
\begin{Theorem}
\label{thm borel yangian action}
The assignments \eqref{sph coha act}, \eqref{sph coha op act}, and \eqref{cartan act} give rise to a $\mathcal{DSH}_\infty(Q,\sW)$ action on $\mathcal H_{\underline{\bd}}^{\sW^{\mathrm{fr}}}$.
\end{Theorem}

\begin{proof}
The relations \eqref{he and hf rel} are easy to check. For the relation \eqref{ef rel}, we have
\begin{equation}\label{[e,f] corr}
\begin{split}
& \quad \,\, e_i(z)f_j(w)-(-1)^{|i|\cdot|j|}f_j(w)e_i(z)\\
    &=(-1)^{\sharp}\left(\frac{[\overline{\mathfrak{P}}(\bv+\delta_i-\delta_j,\bv-\delta_j,\underline{\bd})]}{z-c_1(\mathcal L_i)}\frac{[\overline{\mathfrak{P}}(\bv-\delta_j,\bv,\underline{\bd})]}{w-c_1(\mathcal L_j)}-\frac{[\overline{\mathfrak{P}}(\bv-\delta_j+\delta_i,\bv+\delta_i,\underline{\bd})]}{w-c_1(\mathcal L_j)}\frac{[\overline{\mathfrak{P}}(\bv+\delta_i,\bv,\underline{\bd})]}{z-c_1(\mathcal L_i)}\right).
\end{split}
\end{equation}
If $i\neq j$, then the correspondence on the RHS of\eqref{[e,f] corr} is trivial; if $i=j$, then the correspondence on the RHS of\eqref{[e,f] corr} is supported on the diagonal, and we have
\begin{align*}
    e_i(z)f_i(w)-(-1)^{|i|}f_i(w)e_i(z)=\Delta_*(\beta_i(z,w))\in H^{\sT_0}\left(\cM(\bv,\underline{\bd})^{\times 2},\sw^{\mathrm{fr}}\boxminus \sw^{\mathrm{fr}}\right)_{\Delta}.
\end{align*}
Note that $\beta_i(z,w)$ is completely determined by the intersection theory on $\cM(\underline{\bd})$ which is irrelevant to the potential. Thus we can set the potential to be zero, then let $e_i(z)f_i(w)-(-1)^{|i|}f_i(w)e_i(z)$ act on $1\in H_{\sT_0}(\cM(\bv,\underline{\bd}))$, and get:
\begin{align*}
    \beta_i(z,w)=e_i(z)f_i(w)(1)-(-1)^{|i|}f_i(w)e_i(z)(1).
\end{align*}
We compute the above in two steps, following similar computations on $K$-theory of Nakajima quiver varieties by Negu\c{t} \cite{N2}.

\textbf{Step 1.} Define two complexes on $\cM(\bv,\underline{\bd})$ in cohomological degrees $0$ and $1$:
\begin{align}\label{two complexes}
    \mathcal C^+_i=\left[\mathsf V_i\xrightarrow{(B_i, X_a)}\mathsf D_{\Out,i}\oplus\bigoplus_{a:i\to j}t_a\mathsf V_j\right],\qquad \mathcal C^-_i=\left[\mathsf D_{\In,i}\oplus\bigoplus_{b:j\to i}t_b^{-1}\mathsf V_j\xrightarrow{(A_i, X_b)}\mathsf V_i\right].
\end{align}
Note that $H^1(\mathcal C^-_i)=0$ by the cyclic stability condition. The projection 
$$\pi^+\colon \overline{\mathfrak{P}}(\bv,\bv-\delta_i,\underline{\bd})\to \cM(\bv,\underline{\bd})$$ is $\sT_0$-equivariantly isomorphic to $\bP((\mathcal C^+_i)^\vee)$, and $\mathcal L_i$ is $\sT_0$-equivariantly isomorphic to $\mathcal O(-1)$. The projection 
$$\pi^-\colon \overline{\mathfrak{P}}(\bv+\delta_i,\bv,\underline{\bd})\to \cM(\bv,\underline{\bd})$$ is $\sT_0$-equivariantly isomorphic to $\bP(\mathcal C^-_i)$, and $\mathcal L_i$ is $\sT_0$-equivariantly isomorphic to $\mathcal O(1)$. 

Let $\mathsf V^+,\mathsf V,\mathsf V^-$ denote universal bundles on $\cM(\bv+\delta_i,\underline{\bd}),\cM(\bv,\underline{\bd}),\cM(\bv-\delta_i,\underline{\bd})$ respectively.  Lemma \ref{lem segre} implies 
\begin{align*}
    e_{i,r}(1)=\pi^+_*\left(\frac{(-1)^{\bd_{\Out,i}}}{z-c_1(\mathcal O(-1))}\right)_{z^{-r-1}}=(-1)^{\bv_i+\sum_{i\to j}\bv^+_j}\left[z^{\bd_{\Out,i}-\bv_i^+ +\sum_{i\to j}\bv_j^+}c_{-1/z}\left(\mathsf D_{\Out,i}+\sum_{a:i\to j}t_a\mathsf V^+_j-\mathsf V_i^+\right)\right]_{z^{-r-1}},
\end{align*}
\begin{align*}
    f_{i,s}(1)=\pi^-_*\left(\frac{(-1)^{|i|\cdot |\bv|}}{w-c_1(\mathcal O(1))}\right)_{w^{-s-1}}=(-1)^{|i|\cdot |\bv|}\left[ w^{\bv_i^- -\bd_{\In,i}-\sum_{j\to i}\bv_j^-}c_{-1/w}\left(\mathsf V_i^- -\mathsf D_{\In,i}-\sum_{b:j\to i}t_b^{-1}\mathsf V^-_j\right)\right]_{w^{-s-1}},
\end{align*}
where we take coefficients of $z^{-r}$ and $w^{-s}$ in the power series expansions respectively.

\textbf{Step 2.} We have
\begin{align*}
    &f_{i,s}e_{i,r}(1)=(-1)^{\bv^+_i+|i|\cdot |\bv|+\sum_{i\to j}\bv_j}\pi^-_*\left(\frac{z^{\bd_{\Out,i}-\bv_i^+ +\sum_{i\to j}\bv_j^+}\pi^{+*}c_{-1/z}\left(\mathsf D_{\Out,i}+\sum_{a:i\to j}t_a\mathsf V^+_j-\mathsf V_i^+\right)}{w-c_1(\mathcal O(1))}\right)_{z^{-r-1},w^{-s-1}}\\
    =&(-1)^{\bv^+_i+|i|\cdot |\bv|+\sum_{i\to j}\bv_j}z^{\bd_{\Out,i}-\bv_i +\sum_{i\to j}\bv_j}c_{-1/z}\left(\mathsf D_{\Out,i}+\sum_{a:i\to j}t_a\mathsf V_j-\mathsf V_i\right)\pi^-_*\left(\frac{\prod_{a:i\to i}(z-t_a-c_1(\mathcal O(1)))}{(z-c_1(\mathcal O(1)))(w-c_1(\mathcal O(1)))}\right)_{z^{-r-1},w^{-s-1}}.
\end{align*}
Using Lemma \ref{lem segre} again, we get:
\begin{align*}
    f_{i,s}e_{i,r}(1)=&(-1)^{\bv^+_i+|i|\cdot |\bv|+\sum_{i\to j}\bv_j}z^{\bd_{\Out,i}-\bv_i +\sum_{i\to j}\bv_j}w^{\bv_i -\bd_{\In,i}-\sum_{j\to i}\bv_j}c_{-1/z}\left(\mathsf D_{\Out,i}+\sum_{a:i\to j}t_a\mathsf V_j-\mathsf V_i\right)\\
    &\times \frac{\prod_{a:i\to i}(z-t_a-w)}{z-w} c_{-1/w}\left(\mathsf V_i -\mathsf D_{\In,i}-\sum_{b:j\to i}t_b^{-1}\mathsf V_j\right)\Bigg|_{z^{-r-1},w^{-s-1}}.
\end{align*}
Similarly, 
\begin{align*}
    &\, e_{i,r}f_{i,s}(1)=(-1)^{|i|\cdot |\bv|+\bd_{\Out,i}}\pi^+_*\left(\frac{w^{\bv_i^- -\bd_{\In,i}-\sum_{j\to i}\bv_j^-}\pi^{-*}c_{-1/w}\left(\mathsf V_i^- -\mathsf D_{\In,i}-\sum_{b:j\to i}t_b^{-1}\mathsf V^-_j\right)}{z-c_1(\mathcal O(-1))}\right)_{w^{-s-1},z^{-r-1}}\\
    =&(-1)^{|i|\cdot |\bv|+\bd_{\Out,i}}w^{\bv_i -\bd_{\In,i}-\sum_{j\to i}\bv_j}c_{-1/w}\left(\mathsf V_i -\mathsf D_{\In,i}-\sum_{b:j\to i}t_b^{-1}\mathsf V_j\right)\pi^+_*\left(\frac{\prod_{b:i\to i}(w+t_b-c_1(\mathcal O(-1)))}{(w-c_1(\mathcal O(-1)))(z-c_1(\mathcal O(-1)))}\right)_{w^{-s-1},z^{-r-1}}\\
    =&(-1)^{\bv^+_i+|i|\cdot |\bv|+\sum_{i\to j}\bv_j}z^{\bd_{\Out,i}-\bv_i +\sum_{i\to j}\bv_j}w^{\bv_i -\bd_{\In,i}-\sum_{j\to i}\bv_j}c_{-1/w}\left(\mathsf V_i -\mathsf D_{\In,i}-\sum_{b:j\to i}t_b^{-1}\mathsf V_j\right)\\
    &\times \frac{\prod_{b:i\to i}(w+t_b-z)}{w-z} c_{-1/z}\left(\mathsf D_{\Out,i}+\sum_{a:i\to j}t_a\mathsf V_j-\mathsf V_i\right)\Bigg|_{w^{-s-1},z^{-r-1}}.
\end{align*}
Then, $e_{i,r}f_{i,s}(1)$ and $(-1)^{|i|}f_{i,s}e_{i,r}(1)$ are given by the coefficients of $w^{-s-1}z^{-r-1}$ in the same rational function:
\begin{multline*}
(-1)^{\bv^+_i+|i|\cdot |\bv|+\sum_{i\to j}\bv_j}z^{\bd_{\Out,i}-\bv_i +\sum_{i\to j}\bv_j}w^{\bv_i -\bd_{\In,i}-\sum_{j\to i}\bv_j}\\
\cdot c_{-1/w}\left(\mathsf V_i -\mathsf D_{\In,i}-\sum_{b:j\to i}t_b^{-1}\mathsf V_j\right)c_{-1/z}\left(\mathsf D_{\Out,i}+\sum_{a:i\to j}t_a\mathsf V_j-\mathsf V_i\right)\frac{\prod_{b:i\to i}(w+t_b-z)}{w-z}, 
\end{multline*}
but expanded in different orders. Namely, for $e_{i,r}f_{i,s}(1)$, we expand as 
$$\frac{1}{w-z}=\frac{1}{w}+\frac{z}{w^2}+\cdots, $$
and for $(-1)^{|i|}f_{i,s}e_{i,r}(1)$, we expand as 
$$\frac{1}{w-z}=-\frac{1}{z}-\frac{w}{z^2}+\cdots. $$ 
The difference is given by coefficient of $z^{-r-s-1}$ in the residue of the rational function at $w=z$, that is,
\begin{align*}
e_{i,r}f_{i,s}(1)-(-1)^{|i|}f_{i,s}e_{i,r}(1)=
(-1)^{1+\bv_i+|i|\cdot |\bv|+\sum_{i\to j}\bv_j}\left(\prod_{b:i\to i}t_b\right)z^{\rk\mathsf U_i}c_{-1/z}\left(\mathsf U_i\right)\Bigg|_{z^{-r-s-1}}=\gamma_i h_{i,r+s}.
\end{align*}
This finishes the proof.
\end{proof}
In the above proof, we have used the following lemma, whose proof is left to the interested readers. 
\begin{Lemma}\label{lem segre}
Let $X$ be a variety with a complex of vector bundles $[\mathcal V\to \mathcal U]$ in cohomology degrees $-1$ and $0$, $p\colon \bP([\mathcal V\to \mathcal U])\to X$ denote the projection, then 
\begin{align*}
    p_*\left(\frac{1}{z-c_1(\mathcal O(1))}\right)=\left[z^{\rk\mathcal V-\rk\mathcal U}\frac{c_{-1/z}(\mathcal V)}{c_{-1/z}(\mathcal U)}\right]_{<0},
\end{align*}
where the subscript $<0$ means taking the negative power part of the power series.
\end{Lemma}

\subsection{The case of symmetric quivers}\label{sec shifted yangian_sym quiver}

When $Q$ is symmetric, we can ``twist'' the multiplications in $\mathcal{DSH}_\infty(Q,\sW)$ in order to make comparison with known Yangian type algebras. 

%\textbf{Convention.} 
\begin{Notation}
We say $i\in Q_0^{\mathrm b}$ (resp.~$i\in Q_0^{\mathrm f}$) if $|i|=0$ (resp.~$|i|=1$), where $|\cdot |$ is given by \eqref{equ on gra}. Here $b$ stands for \textit{bosonic} and $f$ stands for \textit{fermionic}.
Fix a splitting $Q_1=\mathcal A\sqcup \mathcal A^*\sqcup \mathcal E$ such that
\begin{itemize}
    \item there is an isomorphism $\iota\colon \mathcal A\cong \mathcal A^*$, $(a:i\to j)\mapsto (a^*:j\to i)$,
    \item $\mathcal E$ is a set of edge loops, such that the tail map $t\colon \mathcal E\to Q_0$ identifies $\mathcal E$ with $Q_0^{\mathrm b}$.
\end{itemize}
Moreover, we fix a total order $\prec$ on $Q_0^{\mathrm f}$.
\end{Notation}
Given the above data, we endow $\mathcal{H}^{\mathrm{nil}}_{Q,\sW}$ \eqref{equ on HQsw} with the \textit{twisted Hall product}:
\begin{align}\label{twi prod}
    f\ostar g=(-1)^{(\bv^{(1)}|\bv^{(2)})} f\star g,\,\,\,\;\text{for }\,\,f\in \mathcal{H}^{\mathrm{nil}}_{Q,\sW}(\bv^{(1)}),\,\: g\in \mathcal{H}^{\mathrm{nil}}_{Q,\sW}(\bv^{(2)}),
\end{align}
where $f\star g$ is the usual Hall product, and
\begin{align}\label{twist hall product}
    (\bv^{(1)}|\bv^{(2)}):=\sum_{(i\to j)\in \mathcal A}\bv^{(1)}_i\cdot\bv^{(2)}_j+\sum_{k\prec l\in Q_0^{\mathrm{f}}}\bv^{(1)}_k\cdot\bv^{(2)}_l.
\end{align}
% \begin{align*}
%     (\bv^{(1)}|\bv^{(2)}):=\sum_{(i\to j)\in \mathcal A}\bv^{(1)}_i\cdot\bv^{(2)}_j-\sum_{i\in Q_0}\bv^{(1)}_i\cdot\bv^{(2)}_i+\sum_{k\prec l\in Q_0^{\mathrm{f}}}\bv^{(1)}_k\cdot\bv^{(2)}_l.
% \end{align*}
Let $\widetilde{\mathcal{H}}^{\mathrm{nil}}_{Q,\sW}$ denote the CoHA with the above twisted product,  $\widetilde{\mathcal{SH}}^{\mathrm{nil}}_{Q,\sW}$ denote its spherical subalgebra, and $\widetilde{\mathcal{SH}}^{\mathrm{nil},\mathrm{op}}_{Q,\sW}$ denote the $\bZ/2$-graded opposite algebra. Note that the opposite product is given by
\begin{align*}
    f\ostar^{\mathrm{op}} g= (-1)^{(\bv^{(2)}|\bv^{(1)})+|\bv^{(1)}|\cdot |\bv^{(2)}|} g\star f,\,\,\,\;\text{for }\,f\in\mathcal{H}^{\mathrm{nil}}_{Q,\sW}(\bv^{(1)}),\,\: g\in \mathcal{H}^{\mathrm{nil}}_{Q,\sW}(\bv^{(2)}).
\end{align*}

\begin{Definition}
\label{def shifted yangian}
When $Q$ is symmetric, the \textit{Cartan doubled Yangian} $\mathcal D\widetilde{\mathcal{SH}}_\infty(Q,\sW)$ is the $\bC[\mathsf t_0]$-algebra obtained from the free product $\widetilde{\mathcal{SH}}^{\mathrm{nil}}_{Q,\sW}*\widetilde{\mathcal{SH}}^{\mathrm{nil},\mathrm{op}}_{Q,\sW}*\mathcal H^0_{Q}$ subject to the relations: 
\begin{align}\label{ef rel_sym}
    e_{i,r} f_{j,s}-(-1)^{|i|\cdot|j|}f_{j,s}e_{i,r}=(-1)^{(\delta_i|\delta_i)}\delta_{ij}\gamma_i h_{i,r+s},\,\,\;\text{where }\,\gamma_i:=\prod_{e:i\to i}t_e,
\end{align}
\begin{align}\label{he and hf rel_sym}
h_i(z)e_j(w)=\zeta_{ij}(z)e_j(w)h_i(z),\quad h_i(z)f_j(w)=\zeta_{ij}(z)^{-1}f_j(w)h_i(z),
\end{align}
where $\zeta_{ij}(z)$ is the ``bond factor'':
\begin{align}\label{bond fac}
    \zeta_{ij}(z)=\frac{\prod_{a:i\to j}(z-t_a-\sigma_j)}{\prod_{b:j\to i}(z+t_b-\sigma_j)},\;\text{where $\sigma_j$ is the operator } e_{j,r}\mapsto e_{j,r+1},\; f_{j,r}\mapsto f_{j,r+1},\;(r\in \bZ_{\geqslant0}).
\end{align}
For $\mu\in \bZ^{Q_0}$,  the \textit{Drinfeld type} $\mu$\textit{-shifted Yangian} $\mathcal D\widetilde{\mathcal{SH}}_\mu(Q,\sW)$ is the quotient of $\mathcal D\widetilde{\mathcal{SH}}_\infty(Q,\sW)$ by the relations
\begin{align}\label{h shift}
    h_{i,t}=0\:\text{ for }\:t<-\mu_i-1,\quad h_{i,-\mu_i-1}=1.
\end{align}
\end{Definition}

\begin{Theorem}\label{cor shifted yangian action}
Let $\mu:=\bd_{\Out}-\bd_{\In}\in \bZ^{Q_0}$. The following assignments 
\begin{align}\label{sph coha act_sym}
    e_i(u)\mapsto \frac{(-1)^{(\delta_i|\bv+\delta_i)}}{u-c_1(\mathcal L_i)}\,[\overline{\mathfrak{P}}(\bv+\delta_i,\bv,\underline{\bd})],
\end{align}
\begin{align}\label{sph coha op act_sym}
    f_i(u)\mapsto \frac{(-1)^{1+\bd_{\Out,i}+|i|\cdot |\bv|+(\bv-\delta_i|\delta_i)}}{u-c_1(\mathcal L_i)}\,[\overline{\mathfrak{P}}(\bv,\bv-\delta_i,\underline{\bd})]^{\mathrm{t}},
\end{align}
\begin{align}\label{cartan act_sym}
    h_i(u)\mapsto u^{\mu_i}c_{-1/u}\left(\mathsf U_i\right),\;\text{ for }\; \mathsf U_i=\sum_{a:i\to j}t_a\mathsf V_j-\sum_{b:j\to i}t_{b}^{-1}\mathsf V_j+\mathsf D_{\Out,i}-\mathsf D_{\In,i}\:,
\end{align}
give rise to a $\mathcal D\widetilde{\mathcal{SH}}_\mu(Q,\sW)$ action on $\mathcal H_{\underline{\bd}}^{\sW^{\mathrm{fr}}}$.
\end{Theorem}

\begin{proof}
This follows directly from Theorem \ref{thm borel yangian action}.
\end{proof}
\begin{Remark}\label{rmk on sm q}
To have a well-defined notion of \textit{shifted} Yangian, it is important to assume the quiver to be  symmetric, so the rank of $\mathsf U_i$ in \eqref{cartan act_sym}
only depends on the dimensions of framings. 
\end{Remark}

\begin{Example}\label{ex on gen symi}
Let the pair $(A,D)$ be a \textit{generalized Cartan matrix with symmetrizer}, one can construct a quiver $Q$, whose associated tripled quiver $\widetilde{Q}$ has a potential $\sW$, see Appendix \ref{sec Y(sym KM)}.

Let $Y^+(\mathfrak{g}_{A,D})$ be the \textit{positive} part of the Yangian of the \textit{symmetrizable Kac-Moody Lie algebra} $\mathfrak{g}_{A,D}$ associated to $(A,D)$ \cite{GTL,YZ3}, 
see Appendix \ref{sec Y(sym KM)} for a review of its definition. $Y^+(\mathfrak{g}_{A,D})$ is generated by $\{x^+_{i,r}\}_{i\in Q_0}^{r\in \bZ_{\geqslant 0}}$, and by Proposition \ref{prop Y^+(sym KM) to COHA}, there is a surjective map of $\bC(\hbar)$-algebras:
$$Y^+(\mathfrak{g}_{A,D})\otimes_{\bC[\hbar]}\bC(\hbar)\twoheadrightarrow \widetilde{\mathcal{SH}}^{\mathrm{nil}}_{\widetilde{Q},\sW}\otimes_{\bC[\hbar]}\bC(\hbar),\quad x^+_i(z):=\sum_{r\geqslant 0}x^+_{i,r}z^{-r-1}\mapsto e_i(z).$$
Given $\mu\in \bZ^{Q_0}$, the \textit{shifted Yangian} $Y_\mu(\mathfrak{g}_{A,D})$ is the $\bC[\hbar]$-algebra obtained from the free product $$Y^+(\mathfrak{g}_{A,D})* Y^+(\mathfrak{g}_{A,D})^{\mathrm{op}}*\bC[\psi_{i,s}\:|\: i\in Q_0,s\in \bZ]$$ subject to relations
\begin{align*}
    [x^+_{i,r},x^-_{j,s}]=\delta_{ij}\psi_{i,r+s},
\end{align*}
\begin{align}\label{relation for shifted yangian of KM alg}
    \psi_i(z)x^+_j(w)=\zeta_{ij}(z)x^+_j(w)\psi_i(z),\quad \psi_i(z)x^-_j(w)=\zeta_{ij}(z)^{-1}x^-_j(w)\psi_i(z),
\end{align}
\begin{align*}
    \psi_{i,r}=0\:\text{ for }\:r<-\mu_i-1,\quad \psi_{i,-\mu_i-1}=1.
\end{align*}
Here $\psi_i(z)=\sum_{r\in\bZ}\psi_{i,r}z^{-r-1}$, and $\zeta_{ij}(z)$ is the bond factor
\begin{align*}
    \zeta_{ij}(z)=\frac{z+d_ia_{ij}\hbar/2-\sigma_j}{z-d_ia_{ij}\hbar/2-\sigma_j},\;\text{where $\sigma_j$ is the operator } x^\pm_{j,r}\mapsto x^\pm_{j,r+1},\;(r\in \bZ_{\geqslant0}).
\end{align*}
Then there is a \textit{surjective} map of $\bC(\hbar)$-algebras:
\begin{gather*}
Y_\mu(\mathfrak{g}_{A,D})\otimes_{\bC[\hbar]}\bC(\hbar)\twoheadrightarrow \mathcal D\widetilde{\mathcal{SH}}_\mu(\widetilde{Q},\sW)\otimes_{\bC[\hbar]}\bC(\hbar),\\
x^+_i(z)\mapsto e_i(z),\quad x^-_i(z)\mapsto \frac{1}{d_i\hbar}f_i(z),\quad \psi_i(z)\mapsto h_i(z).
\end{gather*}
\end{Example}

\begin{Example}\label{ex on vv ex}
Let $m,n\in \bN$ with $\max\{mn,m,n\}\geqslant3$, and set $Q_0=\bZ/(m+n)\bZ$. A \textit{parity sequence} of type $(m,n)$ is a tuple $\mathbf s=(s_1,s_2,\ldots,s_{m+n})$ of $\pm 1$ 
with $m$ entries equal to $1$.
Given a parity sequence $\mathbf s$, one can define a symmetric quiver $Q$ with potential $\sW$, see Appendix \ref{sec super affine Yangian}.

Let $Y^+(\widehat{\mathfrak{sl}}_{n|m})$ be the \textit{positive} part of the \textit{affine Yangian} of $\mathfrak{sl}_{n|m}$ defined in \cite{U,VV1}, which we review in 
Appendix \ref{sec super affine Yangian}. Then by Proposition \ref{prop Y^+(aff sl(n|m)) to COHA}, there is a surjective $\bC(\hbar,\varepsilon)$-algebra map
\begin{align*}
    Y^+(\widehat{\mathfrak{sl}}_{n|m})\otimes_{\bC[\hbar,\varepsilon]}\bC(\hbar,\varepsilon)\twoheadrightarrow \widetilde{\mathcal{SH}}^{\mathrm{nil}}_{Q,\sW}\otimes_{\bC[\hbar,\varepsilon]}\bC(\hbar,\varepsilon),\quad x^+_i(z)\mapsto e_i(z),
\end{align*}
For $\mu\in \bZ^{Q_0}$, the \textit{shifted affine Yangian} $Y_\mu(\widehat{\mathfrak{sl}}_{n|m})$ is the $\bC[\hbar,\varepsilon]$-algebra obtained from the free product $$Y^+(\widehat{\mathfrak{sl}}_{n|m})* Y^+(\widehat{\mathfrak{sl}}_{n|m})^{\mathrm{op}}*\bC[\psi_{i,s}\:|\: i\in Q_0,s\in \bZ]$$ subject to relations
\begin{align*}
    [x^+_{i,r},x^-_{j,s}]=\delta_{ij}\psi_{i,r+s},
\end{align*}
\begin{align}\label{relation for shifted affine yangian of sl(n|m)}
    \psi_i(z)x^+_j(w)=\zeta_{ij}(z)x^+_j(w)\psi_i(z),\quad \psi_i(z)x^-_j(w)=\zeta_{ij}(z)^{-1}x^-_j(w)\psi_i(z),
\end{align}
\begin{align*}
    \psi_{i,r}=0\:\text{ for }\:r<-\mu_i-1,\quad \psi_{i,-\mu_i-1}=1.
\end{align*}
Here $\psi_i(z)=\sum_{r\in\bZ}\psi_{i,r}z^{-r-1}$, and $\zeta_{ij}(z)$ is the bond factor
\begin{align*}
    \zeta_{ij}(z)=\frac{z-a_{ij}\hbar-m_{ij}\varepsilon-\sigma_j}{z+a_{ij}\hbar-m_{ij}\varepsilon-\sigma_j},\;\text{where $a_{ij}$ and $m_{ij}$ are defined in \eqref{aij and mij}, and $\sigma_j$ is the operator } x^\pm_{j,r}\mapsto x^\pm_{j,r+1},\;(r\in \bZ_{\geqslant0}).
\end{align*}
Then there is a \textit{surjective} map of $\bC(\hbar)$-algebras:
\begin{gather*}
Y_\mu(\widehat{\mathfrak{sl}}_{n|m})\otimes_{\bC[\hbar,\varepsilon]}\bC(\hbar,\varepsilon)\twoheadrightarrow \mathcal D\widetilde{\mathcal{SH}}_\mu(Q,\sW)\otimes_{\bC[\hbar,\varepsilon]}\bC(\hbar,\varepsilon),\\
x^+_i(z)\mapsto e_i(z),\quad x^-_i(z)\mapsto \begin{cases}
    \frac{-1}{2s_i\hbar\:}f_i(z), & i\in Q_0^{\mathrm{b}}\\
    f_i(z), & i\in Q_0^{\mathrm{f}}
\end{cases},\quad \psi_i(z)\mapsto h_i(z).
\end{gather*}
\end{Example}

\begin{Example}\label{ex on dd ex}
Let $Q$ and $Q^\circ$ with potential $\sW$ and $\sW^\circ$ be given in Appendix \S\ref{sec quiver for D(2,1)} and \S\ref{sec quiver for D'(2,1)}, respectively.

% $Q=(Q_0,Q_1)$ be the quiver with $Q_0=\bZ/3\bZ$ and $Q_1=Q_1^+\sqcup Q_1^-$ where 
% \begin{align*}
%     Q_1^+=\{X_i:i\to i+1\:|\: i\in Q_0\},\quad Q_1^-=\{Y_i:i+1\to i\:|\: i\in Q_0\}.
% \end{align*}
% Let $\bC^*_{\hbar_1}\times \bC^*_{\hbar_2}\times \bC^*_{\hbar_3}$ acts on $Q$ by
% \begin{align*}
%     (t_1,t_2,t_3)\in \bC^*_{\hbar_1}\times \bC^*_{\hbar_2}\times \bC^*_{\hbar_3}: (X_i,Y_i)\mapsto (t_iX_i,t_iY_i).
% \end{align*}
% Take the potential $\sw$ to be
% \begin{align*}
%     \sw=\tr\left(X_3X_2X_1+Y_1Y_2Y_3\right),
% \end{align*}
% which is fixed by the subtorus $\sT_0=\{(t_1,t_2,t_3)\in \bC^*_{\hbar_1}\times \bC^*_{\hbar_2}\times \bC^*_{\hbar_3}\:|\: t_1t_2t_3=1\}$. 
% We denote $\bC[\mathsf t_0]:=H_{\sT_0}(\pt)\cong \bC[\hbar_1,\hbar_2,\hbar_3]/(\hbar_1+\hbar_2+\hbar_3)$.

Let $Y^+(D_\lambda)$ and $Y^+(D^\circ_\lambda)$ be the two versions of \textit{positive} part of the Yangians of the \textit{exceptional Lie superalgebra} $D(2,1;\lambda)$ defined in \cite{LZ}, see Appendix \S\ref{sec Yangian for D(2,1;a)} for a review. Then by Proposition \ref{prop Y^+(D(2,1;a)) to COHA} and \ref{prop Y^+(D'(2,1;a)) to COHA}, there are surjective $\bC$-algebra maps
\begin{align*}
Y^+(D_\lambda)\otimes_{\bC[\hbar,\lambda]}\bC(\hbar,\lambda)\twoheadrightarrow \widetilde{\mathcal{SH}}^{\mathrm{nil}}_{Q,\sW}\otimes_{\bC[\mathsf t_0]}\bC(\mathsf t_0),\quad Y^+(D^\circ_\lambda)\otimes_{\bC[\hbar,\lambda]}\bC(\hbar,\lambda)\twoheadrightarrow \widetilde{\mathcal{SH}}^{\mathrm{nil}}_{Q^\circ,\sW^\circ}\otimes_{\bC[\mathsf t_0]}\bC(\mathsf t_0),
\end{align*}
given by
\begin{align*}
x^+_i(z)\mapsto e_i(z), \quad \hbar\mapsto 2\hbar_1,\quad \lambda\mapsto \hbar_2/\hbar_1.
\end{align*}
For $\mu\in \bZ^{Q_0}$, the \textit{shifted Yangian} $Y_\mu(\mathfrak{g})$ for $\mathfrak{g}=D_\lambda$ or $D^\circ_\lambda$ is the $\bC[\hbar,\lambda]$-algebra obtained from the free product $$Y^+(\mathfrak{g})* Y^+(\mathfrak{g})^{\mathrm{op}}*\bC[\psi_{i,s}\:|\: i\in Q_0\text{ or }Q^\circ_0,s\in \bZ]$$ subject to relations
\begin{align*}
    [x^+_{i,r},x^-_{j,s}]=\delta_{ij}\psi_{i,r+s},
\end{align*}
\begin{align}\label{relation for shifted affine yangian of D(2,1)}
    \psi_i(z)x^+_j(w)=\zeta_{ij}(z)x^+_j(w)\psi_i(z),\quad \psi_i(z)x^-_j(w)=\zeta_{ij}(z)^{-1}x^-_j(w)\psi_i(z),
\end{align}
\begin{align*}
    \psi_{i,r}=0\:\text{ for }\:r<-\mu_i-1,\quad \psi_{i,-\mu_i-1}=1.
\end{align*}
Here $\psi_i(z)=\sum_{r\in\bZ}\psi_{i,r}z^{-r-1}$, and $\zeta_{ij}(z)$ is the bond factor
\begin{align*}
    \zeta_{ij}(z)=\frac{z-a_{ij}\hbar-\sigma_j}{z+a_{ij}\hbar-\sigma_j},\;\text{where $a_{ij}$ is defined in \eqref{Cartan matrix for D(2,1)}, and $\sigma_j$ is the operator } x^\pm_{j,r}\mapsto x^\pm_{j,r+1},\;(r\in \bZ_{\geqslant0}).
\end{align*}
Then there are \textit{surjective} maps of $\bC$-algebras:
\begin{gather*}
Y_\mu(D_\lambda)\otimes_{\bC[\hbar,\lambda]}\bC(\hbar,\lambda)\twoheadrightarrow \mathcal D\widetilde{\mathcal{SH}}_\mu(Q,\sW)\otimes_{\bC[\mathsf t_0]}\bC(\mathsf t_0),\quad Y_\mu(D^\circ_\lambda)\otimes_{\bC[\hbar,\lambda]}\bC(\hbar,\lambda)\twoheadrightarrow \mathcal D\widetilde{\mathcal{SH}}_\mu(Q^\circ,\sW^\circ)\otimes_{\bC[\mathsf t_0]}\bC(\mathsf t_0)
\end{gather*}
given by 
\begin{align*}
x^+_i(z)\mapsto e_i(z),\quad x^-_i(z)\mapsto f_i(z),\quad \psi_i(z)\mapsto h_i(z),\quad\hbar\mapsto 2\hbar_1,\quad\lambda\mapsto \hbar_2/\hbar_1.
\end{align*}
Quivers with potentials related to affinizations $\widehat{D}_\lambda$ and $\widehat{D}^\circ_\lambda$ have been considered in \cite{Bao}.
\end{Example}

\begin{Remark}\label{rmk on quiver Yangian}
More generally, for a symmetric quiver with potential, we expect that the (shifted) quiver Yangian introduced by Galakhov-Li-Yamazaki in \cite{GLY, LY} also has an algebra map to 
the Drinfeld type Yangian $\mathcal D\widetilde{\mathcal{SH}}_\mu(Q,\sW)$ (which is defined only when the quiver is symmetric, see Remark \ref{rmk on sm q}). Moreover, we expect that the crystal representations of (shifted) quiver Yangian studied in \cite{GLY, LY, L} can be identified with $\mathcal D\widetilde{\mathcal{SH}}_\mu(Q,\sW)$ representations $\cH^{\sW^{\mathrm{fr}}}_{\underline{\bd}}$ such that the crystals correspond to the torus fixed points of $\Crit_{\cM_\theta(\underline{\bd})}(\sw^{\mathrm{fr}})$.
\end{Remark}

\subsection{Raising/lowering operators via stable envelopes}

\begin{Definition}\label{def on adm pot}
We call a framing $\underline{\bd}$ \textit{minuscule symmetric} if $\bd_{\In}=\bd_{\Out}=\delta_i$ for some $i\in Q_0$; in the case of a minuscule symmetric framing $\underline{\delta_i}$, we say a framed potential $\sW^{\mathrm{fr}}$ \textit{admissible} if 
\begin{align*}
    \mathsf W^{\mathrm{fr}}\big|_{X_a=0,\,\forall\,a\in Q_1}=0.
\end{align*}
% \begin{align*}
%     \dim_{\bC(\mathsf t_0)}H^{\sT_0}(\cM(\delta_i,\delta_i),\sw^{\mathrm{fr}})\otimes_{\bC[\mathsf t_0]} \bC(\mathsf t_0)=1,
% \end{align*}
% \textit{minuscule} if \yh{unclear for fermionic node}
% \begin{align*}
%     \mathsf m^{\mathrm{fr}}=\mathsf m+B_i\Phi_i A_i\:.
% \end{align*}
\end{Definition}

\begin{Example}
A nonadmissible example is $\sW^{\mathrm{fr}}=\sW+A_iB_i$ where $A_i$ (resp.\,$B_i$) is the in-coming (resp.\,out-going) framing at the node $i$. In this case, for all $\bv\neq \mathbf 0$, we have $\Crit_{\cM(\bv,\delta_i)}(\sw^{\mathrm{fr}})=\emptyset$, and in particular $H^{\sT_0}(\cM(\bv,\delta_i),\sw^{\mathrm{fr}})=0$. The admissibility condition ensures that the cohomology is nontrivial at least for $\bv=\delta_i$, shown in the next lemma.
\end{Example}

\begin{Lemma}\label{lem 1st excitation}
Assume that the framed potential 
$\sW^{\mathrm{fr}}$ on $\cM(\delta_i)$ is admissible, and that $\sT_0$-equivariant weights on every edge loop at $i\in Q_0$ is nonzero (equivalently $\gamma_i=\prod_{e:i\to i}t_e$ is nonzero), then 
\begin{align*}
\dim_{\bC(\mathsf t_0)}H^{\sT_0}\left(\cM(\delta_i,\delta_i),\sw^{\mathrm{fr}}\right)\otimes_{\bC[\mathsf t_0]} \bC(\mathsf t_0)=1.
\end{align*}
\end{Lemma}

\begin{proof}
We have $\cM(\delta_i,\delta_i)\cong \bC^{\mathsf g_i+1}$, where $\mathsf g_i=\#(i\to i\in Q_1)$ is the number of loops at $i$. Decompose $\bC^{\mathsf g_i+1}=\bC^{\mathsf g_i}\times \bC$, then $\sT_0$-weights on $\bC^{\mathsf g_i}$ are nonzero, and $\sw^{\mathrm{fr}}\big|_{\bC}=0$. 
By equivariant localization, we have 
\begin{equation*}
    H^{\sT_0}\left(\cM(\delta_i,\delta_i),\sw^{\mathrm{fr}}\right)\otimes_{\bC[\mathsf t_0]} \bC(\mathsf t_0)\cong H^{\sT_0}(\bC)\otimes_{\bC[\mathsf t_0]}\bC(\mathsf t_0)=\bC(\mathsf t_0). \qedhere
\end{equation*}
\end{proof}

\begin{Notation}\label{ground and 1st excitation}
In the following discussions, we use the notations:
$$|\delta_i\rangle:=\can([\bC])\in H^{\sT_0}\left(\cM(\delta_i,\delta_i),\sw^{\mathrm{fr}}\right),\quad |\mathbf 0\rangle:=[\pt]\in H^{\sT_0}\left(\cM(\mathbf 0,\delta_i),\sw^{\mathrm{fr}}\right)=H_{\sT_0}(\pt),$$ where $\bC$ is the component in the decomposition $\cM(\delta_i,\delta_i)\cong \bC^{\mathsf g_i}\times \bC$ in the proof of Lemma \ref{lem 1st excitation}.
\end{Notation}

Consider a framing vector space $u\underline{\delta_i}+\underline{\bd}$, where $u$ is the equivariant variable of $\sA=\bC^*_u$. The next result is an explicit description of the first off-diagonal term of stable envelope $\Stab_{u>0}$\footnote{This coincides with Hall envelope \eqref{equ on hall env} by \cite[Prop.~9.23]{COZZ}.}. Let \begin{equation}
\label{equ onFF'}
F=\cM(\mathbf{0},\delta_i)\times \cM(\bv+\delta_i,\underline{\bd}),\quad F'=\cM(\delta_i,\delta_i)\times\cM(\bv,\underline{\bd})
\end{equation}be two connected components in $\cM(\bv+\delta_i,\underline{\delta_i}+\underline{\bd})^{\bC^*_u}$, and $$\langle \mathbf 0|\Stab_{u>0}|\delta_i\rangle\in \Hom_{\bC[\mathsf t_0]}\left(\cH^{\sW^{\mathrm{fr}}}_{\underline{\bd}}(\bv),\cH^{\sW^{\mathrm{fr}}}_{\underline{\bd}}(\bv+\delta_i)\right)$$ be the matrix element of $\Stab_{u>0}\big|_{F\times F'}$ in the basis $|\delta_i\rangle\otimes \cH^{\sW^{\mathrm{fr}}}_{\underline{\bd}}(\bv)$ and $|\mathbf 0\rangle\otimes \cH^{\sW^{\mathrm{fr}}}_{\underline{\bd}}(\bv+\delta_i)$.

\begin{Proposition}
\label{prop e as stab element}
Assume that the framed potential on $\cM(\delta_i)$ is admissible, and assume $\bd_{\In, i}\geqslant \bd_{\Out, i}$ for all $i\in Q_0$, then the matrix element $\langle \mathbf 0|\Stab_{u>0}|\delta_i\rangle$ is induced by the correspondence:
\begin{align*}
   (-u)^{\bv_i} c_{-1/u}(\mathsf V_i)\cap[\overline{\mathfrak{P}}(\bv+\delta_i,\bv,\underline{\bd})].
\end{align*}
Here $c_{-1/u}(-)$ denotes the $\sT_0$-equivariant Chern polynomial with variable $-1/u$.
\end{Proposition}

The proof of Proposition \ref{prop e as stab element} follows from similar arguments for $K$-theory of Nakajima quiver varieties by Negu\c{t} \cite[Prop.~3.16]{N2}, which we now explain. 
First of all, we have the following analog of \cite[Prop.~2.19]{N2}.

\begin{Lemma}\label{lem locus of quiver rep hom}
Fix dimension vectors $\bv,\widetilde{\bv},\underline{\bd},\underline{\widetilde{\bd}}$, and  linear maps $p_i\colon \widetilde{D}_{\In,i}\to {D}_{\In,i}$, $q_i\colon \widetilde{D}_{\Out,i}\to {D}_{\Out,i}$ for all $i\in Q_0$. Consider the following complex of vector bundles on $\cM(\widetilde{\bv},\underline{\widetilde{\bd}})\times \cM(\bv,\underline{\bd})$:
\begin{align*}
    \bigoplus_{i\in Q_0}\Hom(\widetilde{\mathsf V}_i,\mathsf V_i)\xrightarrow{f} \bigoplus_{(a\colon i\to j)\in Q_1}\Hom(\widetilde{\mathsf V}_i,\mathsf V_j)\oplus\bigoplus_{i\in Q_0}\left(\Hom(\widetilde{\mathsf D}_{\In,i},\mathsf V_i)\oplus\Hom(\widetilde{\mathsf V}_i,\mathsf D_{\Out,i})\right),
\end{align*}
where the cohomological degrees are $-1$ and $0$, and the map $f$ is given by
\begin{align*}
    f(\{\psi_i\}_{i\in Q_0})=(X_a\psi_{t(a)}-\psi_{h(a)}X_a,\psi_i\widetilde{A}_i,B_i\psi_i),
\end{align*}
where $A_i$, $B_i$, $X_a$ are maps in \eqref{equ rvab} for $\cM(\bv,\underline{\bd})$, similarly for $\widetilde{A}_i$ in $\cM(\widetilde{\bv},\underline{\widetilde{\bd}})$. 

Then the map $f$ is fiberwise injective, so the cohomology of the complex defines a vector bundle $\mathcal F$. Moreover, the assignment
\begin{align}\label{eq section for quiver rep hom}
    s\left((\widetilde{X}_a,\widetilde{A}_i,\widetilde{B}_i),(X_a,A_i,B_i)\right)=(0,A_ip_i,q_i\widetilde{B}_i)
\end{align}
induces a section $s\in \Gamma(\cM(\widetilde{\bv},\underline{\widetilde{\bd}})\times \cM(\bv,\underline{\bd}),\mathcal F)$ whose zero locus is the set of pairs of framed quiver representations $(\widetilde{V}_i,\widetilde{D}_{\In,i},\widetilde{D}_{\Out,i}),(V_i,D_{\In,i},D_{\Out,i})$ together with a framed quiver representation homomorphism 
\begin{align*}
(\widetilde{V}_i,\widetilde{D}_{\In,i},\widetilde{D}_{\Out,i})\xrightarrow{(\psi_i,p_i,q_i)} (V_i,D_{\In,i},D_{\Out,i}).
\end{align*}
$\{\psi_i\}_{i\in Q_0}$ is uniquely determined if it exists. If $p_i$ are surjective for all $i\in Q_0$, then so are $\psi_i$.
\end{Lemma}
\begin{proof}
The proof is similar to that of \cite[Prop.~2.19]{N2} and we omit the details.
\end{proof}
For dimension vector $\widetilde{\bv}=\delta_i+\bv$, 
by applying Lemma \ref{lem locus of quiver rep hom} to the projection $$\bC^{\underline{\delta_i}}\oplus\bC^{\underline{\bd}}\to \bC^{\underline{\bd}},$$ 
we get a vector bundle $\mathcal F$ on $\cM(\bv+\delta_i,\underline{\delta_i}+\underline{\bd})\times\cM(\bv,\underline{\bd})$ together with a section $$s\in \Gamma( \cM(\bv+\delta_i,\underline{\delta_i}+\underline{\bd})\times\cM(\bv,\underline{\bd}),\mathcal F).$$ 
By applying Lemma \ref{lem locus of quiver rep hom} to the identity map $$\bC^{\underline{\bd}}\xrightarrow{\id} \bC^{\underline{\bd}},$$ 
we get a vector bundle $\overline{\mathcal F}$ on $\cM(\bv+\delta_i,\underline{\bd})\times\cM(\bv,\underline{\bd})$ together with a section 
$$\bar s\in \Gamma( \cM(\bv+\delta_i,\underline{\bd})\times\cM(\bv,\underline{\bd}),\overline{\mathcal F}).$$ Then we have the following analogue of \cite[Claim 3.17]{N2}, 
which we omit the proof.
\begin{Lemma}\label{lem attr closure is lci}
The attracting set closure $\overline{\Attr}_{u>0}(\Delta_{F'})\subset \cM(\bv+\delta_i,\underline{\delta_i}+\underline{\bd})\times \cM(\delta_i,\delta_i)\times \cM(\bv,\underline{\bd})$ is a local complete intersection, cut out by the section $s\in \Gamma( \cM(\bv+\delta_i,\underline{\delta_i}+\underline{\bd})\times\cM(\bv,\underline{\bd}),\mathcal F)$, together with $(\mathsf g_i+1)$ extra equations, where $\mathsf g_i$ is the number of loops at node $i$.
\end{Lemma}
%\begin{proof}
%The proof is similar to that of \cite[Claim 3.17]{N2} and we omit the details.\end{proof}
The key ingredient of the proof of Proposition \ref{prop e as stab element} is the following description of the matrix element $\Stab_{u>0}\big|_{F\times F'}$ to the product of connected components $F$, $F'$\eqref{equ onFF'}.

\begin{Lemma}\label{lem stab matrix element}
The intersection $\overline{\Attr}_{u>0}(\Delta_{F'})\bigcap\,(F\times F')$ is smooth, cut out by the section $\bar s\in \Gamma( \cM(\bv+\delta_i,\underline{\bd})\times\cM(\bv,\underline{\bd}),\overline{\mathcal F})$, together with $(\mathsf g_i+1)$ extra equations.
\end{Lemma}

\begin{proof}
By Lemma \ref{lem attr closure is lci}, $\overline{\Attr}_{u>0}(\Delta_{F'})\bigcap\,(F\times F')$ is cut out by the restriction of $s\in \Gamma( \cM(\bv+\delta_i,\underline{\delta_i}+\underline{\bd})\times\cM(\bv,\underline{\bd}),\mathcal F)$ to $F\times \cM(\bv,\underline{\bd})\cong \cM(\bv+\delta_i,\underline{\bd})\times \cM(\bv,\underline{\bd})$ together with $(\mathsf g_i+1)$ extra equations. Notice that
\begin{align}\label{res of F}
    \mathcal F\big|_{\cM(\bv+\delta_i,\underline{\bd})\times \cM(\bv,\underline{\bd})}\cong \overline{\mathcal F}\oplus \Hom(u\,\bC,\mathsf V_i)
\end{align}
and under which we have 
\begin{align*}
    s\big|_{\cM(\bv+\delta_i,\underline{\bd})\times \cM(\bv,\underline{\bd})}=(\bar s,0).
\end{align*}
Thus $\overline{\Attr}_{u>0}(\Delta_{F'})\bigcap\,(F\times F')$ is cut out by $\bar s$ together with $(\mathsf g_i+1)$ extra equations. 

The smoothness is the consequence of the following observation: the projection $$\cM(\bv+\delta_i,\underline{\bd})\times \cM(\delta_i,\delta_i)\times \cM(\bv,\underline{\bd})\to \cM(\bv+\delta_i,\underline{\bd})\times \cM(\bv,\underline{\bd})$$ induces an isomorphism \begin{equation}\label{equ on attinterff}\overline{\Attr}_{u>0}(\Delta_{F'})\bigcap\,(F\times F')\cong{\mathfrak{P}}(\bv+\delta_i,\bv,\underline{\bd}).\end{equation} In fact, the zero locus of $\bar s$ on $\cM(\bv+\delta_i,\underline{\bd})\times \cM(\bv,\underline{\bd})$ is exactly $${\mathfrak{P}}(\bv+\delta_i,\bv,\underline{\bd})=\left\{\substack{\text{surjective maps between framed quiver representations }\\\text{ $(V_j^+,D_{\In ,j},D_{\Out,j})\twoheadrightarrow (V_j,D_{\In ,j},D_{\Out,j})$ which are identity on framing}}\right\},$$
and the extra $(\mathsf g_i+1)$ equations on $\cM(\bv+\delta_i,\underline{\bd})\times \cM(\delta_i,\delta_i)\times \cM(\bv,\underline{\bd})\to \cM(\bv+\delta_i,\underline{\bd})\times \cM(\bv,\underline{\bd})$ are the following:
\begin{itemize}
    \item $B_i$ vanishes on the framed quiver representation $(\bC^{\delta_i},\bC^{\delta_i},\bC^{\delta_i})$,
    \item for every loop $a:i\to i$, $X_a$ action on the gauge vector space $\bC^{\delta_i}$ agrees with the $X_a$ action on the kernel of $(V_j^+,D_{\In ,j},D_{\Out,j})\twoheadrightarrow (V_j,D_{\In ,j},D_{\Out,j})$.
\end{itemize}
This proves \eqref{equ on attinterff}. Note that ${\mathfrak{P}}(\bv+\delta_i,\bv,\underline{\bd})\cong \overline{\mathfrak{P}}(\bv+\delta_i,\bv,\underline{\bd})\times \bC^{\mathsf g_i}$, and that $$\overline{\mathfrak{P}}(\bv+\delta_i,\bv,\underline{\bd})\cong \bP(H^0\mathcal C^-_i),$$
where $H^0\mathcal C^-_i$ is the zeroth cohomology of the complex $\mathcal C^-_i$ on $\cM(\bv,\underline{\bd})$ defined in \eqref{two complexes}. The smoothness 
%of $\overline{\Attr}_{u>0}(\Delta_{F'})\bigcap\,(F\times F')$ 
follows from the local freeness of $H^0\mathcal C^-_i$.
\end{proof}

\begin{proof}[Proof of Proposition \ref{prop e as stab element}]
In the decomposition \eqref{res of F}, $\Hom(u\bC, \mathsf V_i)$ is the excess vector bundle for the intersection between $\overline{\Attr}_{u>0}(\Delta_{F'})$ and $F\times F'$, so $\Stab_{u>0}\big|_{F\times F'}$ is induced by the correspondence $$(-u)^{\bv_i}c_{-1/u}(\mathsf V_i)\cap\can\left[\overline{\Attr}_{u>0}(\Delta_{F'})\bigcap \,(F\times F')\right].$$ 
Recall that $\cM(\delta_i,\delta_i)\cong \bC^{\mathsf g_i}\times \bC$, where $\bC$ accounts for the map $B_i$ on the framed representation $(\bC^{\delta_i},\bC^{\delta_i},\bC^{\delta_i})$. Then 
\begin{align*}
    &\quad \,\,\overline{\Attr}_{u>0}(\Delta_{F'})\bigcap\, (F\times F')\bigcap\left(\cM(\bv+\delta_i,\underline{\bd})\times \{0\}\times \bC\times \cM(\bv,\underline{\bd})\right)\\
    &=\left\{\substack{\text{surjective maps between framed quiver representations }\\\text{ $(V_j^+,D_{\In ,j},D_{\Out,j})\twoheadrightarrow (V_j,D_{\In ,j},D_{\Out,j})$ which are identity on framing}\\ \text{and for all edge loops $a:i\to i$, $X_a$ action on the kernel is trivial}}\right\}
    =\overline{\mathfrak{P}}(\bv+\delta_i,\bv,\underline{\bd}).
\end{align*}
And $\langle \mathbf 0|\Stab_{u>0}|\delta_i\rangle$ is induced by the correspondence:
\begin{align*}
    (-u)^{\bv_i} c_{-1/u}(\mathsf V_i)\cap[\overline{\mathfrak{P}}(\bv+\delta_i,\bv,\underline{\bd})].
\end{align*}
This finishes the proof.
\end{proof}

With the same convention as in Section \ref{sec shifted yangian_sym quiver}, we define a virtual vector bundle on $\cM(\bv,\underline{\bd})$:
\begin{align}\label{par pol on M(v,d)}
    \mathsf P=\sum_{i\in Q_0}\Hom(\mathsf V_i,\mathsf D_{\Out,i})+\sum_{(i\to j) \in \mathcal A^*}\Hom(\mathsf V_i,\mathsf V_j)+\sum_{k\prec l\in Q_0^{\mathrm{f}}}\Hom(\mathsf V_l,\mathsf V_k),
\end{align}
% \begin{align*}
%     \mathsf P=\sum_{i\in Q_0}\Hom(\mathsf V_i,\mathsf D_{\Out,i})+\sum_{i\to j \in \mathcal A^*}\Hom(\mathsf V_i,\mathsf V_j)-\sum_{i\in Q_0}\End(\mathsf V_i)+\sum_{k\prec l\in Q_0^{\mathrm{f}}}\Hom(\mathsf V_l,\mathsf V_k),
% \end{align*}
and define stable envelope with a normalizer
\begin{align}\label{normalizer_yangian}
    \Stab_{\fC,\epsilon}=\Stab_{\fC}\cdot \,\epsilon,\; \text{with }\;\epsilon|_F=(-1)^{\rk \mathsf P|_{F}^+},
\end{align}
where $F$ is a fixed component and $\mathsf P|_{F}^+$ is the $\fC$-attracting part of the restriction of $\mathsf P$ to $F$.

\begin{Theorem}\label{thm e f h as R matrix elements}
Let $R(u)=\Stab_{u<0,\epsilon}^{-1}\Stab_{u>0,\epsilon}\in \End^{\bZ/2}\left(\cH^{\sW^{\mathrm{fr}}}_{\underline{\delta_i},\mathrm{loc}}\otimes\cH^{\sW^{\mathrm{fr}}}_{\underline{\bd},\mathrm{loc}}\right)(u)$ be the root $R$-matrix
and $$R^{\mathrm{sup}}(u):=\Sigma\cdot R(u), $$ 
where $\Sigma$ is the diagonal operator that acts on $\cH^{\sW^{\mathrm{fr}}}_{\underline{\delta_i}}(\bv')\otimes\cH^{\sW^{\mathrm{fr}}}_{\underline{\bd}}(\bv'')$ by the scalar $(-1)^{|\bv'|\cdot |\bv''|}$. 

Assume that $\sw^{\mathrm{fr}}$ on $\cM(\delta_i)$ is admissible, then the first $2$ by $2$ block of $R^{\mathrm{sup}}(u)$ admits Gauss decomposition:
\begin{align}\label{gauss decomp of R}
\begin{pmatrix}
\langle\mathbf 0|R^{\mathrm{sup}}(u)|\mathbf 0\rangle & \langle\mathbf 0|R^{\mathrm{sup}}(u)|\delta_i\rangle \\
\langle\delta_i|R^{\mathrm{sup}}(u)|\mathbf 0\rangle & \langle\delta_i|R^{\mathrm{sup}}(u)|\delta_i\rangle
\end{pmatrix}=
\begin{pmatrix}
1 & 0 \\
\frac{(-1)^{|i|} t_i}{\gamma_i}f_i(u) & 1
\end{pmatrix}
\begin{pmatrix}
g_i(u) & 0 \\
0 & g_i(u)h_i(u) 
\end{pmatrix}
\begin{pmatrix}
1 & e_i(u)\\
0 & 1 
\end{pmatrix},
\end{align}
% the Gauss decomposition of $R^{\mathrm{sup}}(u)$ takes the form:
% \begin{align*}
%     R^{\mathrm{sup}}(u)=F(u)G(u)E(u),
% \end{align*}
% where $F(u)$ is lower-unipotent, $E(u)$ is upper-unipotent, and $G(u)$ is diagonal, with first $2$ by $2$ blocks given by:
% \begin{align*}
% F(u)=\begin{pmatrix}
% 1 & 0 & \cdots\\
% f_i(u) & 1 & \\
% \vdots & & \ddots
% \end{pmatrix},\quad
% G(u)=\begin{pmatrix}
% g_i(u) & 0 & \cdots\\
% 0 & g_i(u)h_i(u) & \\
% \vdots & & \ddots
% \end{pmatrix},\quad
% E(u)=\begin{pmatrix}
% 1 & e_i(u) & \cdots\\
% 0 & 1 & \\
% \vdots & & \ddots
% \end{pmatrix},
% \end{align*}
where $e_i(u)$, $f_i(u)$, $h_i(u)$ are the linear operators in Theorem \ref{cor shifted yangian action} acting on $\cH^{\sW^{\mathrm{fr}}}_{\underline{\bd}}$, $\gamma_i=\prod_{e:i\to i}t_e$ for $t_e=$ $\sT_0$-weight of the edge $e$,  $t_i$ is the $\sT_0$-weight of the out-going framing $B_i$ in $\cM(\delta_i)$, and $g_i(u)$ is the following diagonal operator (with equivariant parameter written multiplicatively):
\begin{equation}
\label{equ on gi}
    g_i(u)=c_{-1/u}((1-t_i^{-1})\mathsf V_i).
\end{equation}
% Here $t_i$ is the $\sT_0$-weight on the out-going framing $B_i$ in $\cM(\delta_i)$.
\end{Theorem}

\begin{proof}
By Proposition \ref{prop e as stab element}, the first $2$ by $2$ block of $\Stab_{u>0}$ reads:
\begin{align*}
\begin{pmatrix}
\langle\mathbf 0|\Stab_{u>0}|\mathbf 0\rangle & \langle\mathbf 0|\Stab_{u>0}|\delta_i\rangle \\
\langle\delta_i|\Stab_{u>0}|\mathbf 0\rangle & \langle\delta_i|\Stab_{u>0}|\delta_i\rangle
\end{pmatrix}
=
\begin{pmatrix}
e^{\sT_0\times \bC^*_u}(N^-_{F/X}) & 0 \\
0 & e^{\sT_0\times \bC^*_u}(N^-_{F'/X}) 
\end{pmatrix}
\begin{pmatrix}
1 & \frac{[\overline{\mathfrak{P}}(\bv+\delta_i,\bv,\underline{\bd})]}{c_1(\mathcal L_i)-u}\\
0 & 1 
\end{pmatrix}.
\end{align*}
By \cite[Lem.~3.29]{COZZ}, $\Stab_{u<0}^{-1}$ is induced by the transpose correspondence $[\Stab_{u>0}]^{\mathrm{t}}$, so the first $2$ by $2$ block of $\Stab_{u<0}^{-1}\Stab_{u>0}$ admits Gauss decomposition
\begin{align*}
\begin{pmatrix}
1 & 0\\
\frac{t_i}{\gamma_i}\frac{[\overline{\mathfrak{P}}(\bv+\delta_i,\bv,\underline{\bd})]^{\mathrm{t}}}{c_1(\mathcal L_i)-u} & 1 
\end{pmatrix}
\begin{pmatrix}
\frac{e^{\sT_0\times \bC^*_u}(N^-_{F/X})}{e^{\sT_0\times \bC^*_u}(N^+_{F/X})} & 0 \\
0 & \frac{e^{\sT_0\times \bC^*_u}(N^-_{F'/X})}{e^{\sT_0\times \bC^*_u}(N^+_{F'/X})} 
\end{pmatrix}
\begin{pmatrix}
1 & \frac{[\overline{\mathfrak{P}}(\bv+\delta_i,\bv,\underline{\bd})]}{c_1(\mathcal L_i)-u}\\
0 & 1 
\end{pmatrix}.
\end{align*}
Taking normalizers into account, explicit computation gives \eqref{gauss decomp of R}.
\end{proof}

\begin{Remark}\label{rmk super YBE}
The $R$ matrices $R^{\mathrm{sup}}(u)$ satisfy the \textit{super Yang-Baxter equation} in the following sense. Let 
$$R_{\alpha\beta}(u)=\Stab_{u<0,\epsilon}^{-1}\Stab_{u>0,\epsilon}\in \End^{\bZ/2}\left(\cH^{\sW^{\mathrm{fr}}}_{\underline{\bd}^{(\alpha)},\mathrm{loc}}\otimes \cH^{\sW^{\mathrm{fr}}}_{\underline{\bd}^{(\beta)},\mathrm{loc}}\right)(u)$$ be the root $R$ matrix, and $\Sigma$ be the linear map that acts on $\cH^{\sW^{\mathrm{fr}}}_{\underline{\bd}^{(\alpha)}}(\bv')\otimes \cH^{\sW^{\mathrm{fr}}}_{\underline{\bd}^{(\beta)}}(\bv'')$ by the scalar $(-1)^{|\bv'|\cdot |\bv''|}$. Define $R^{\mathrm{sup}}_{\alpha\beta}(u):=\Sigma\cdot R_{\alpha\beta}(u)$, then the equation
\begin{align}\label{super YBE}
R^{\mathrm{sup}}_{12}(u-v)R^{\mathrm{sup}}_{13}(u)R^{\mathrm{sup}}_{23}(v)=R^{\mathrm{sup}}_{23}(v)R^{\mathrm{sup}}_{13}(u)R^{\mathrm{sup}}_{12}(u-v)
\end{align}
holds in $\End^{\bZ/2}\left(\cH^{\sW^{\mathrm{fr}}}_{\underline{\bd}^{(1)},\mathrm{loc}}\otimes \cH^{\sW^{\mathrm{fr}}}_{\underline{\bd}^{(2)},\mathrm{loc}}\otimes \cH^{\sW^{\mathrm{fr}}}_{\underline{\bd}^{(3)},\mathrm{loc}}\right)(u,v)$. Here endomorphisms are $\bZ/2$ graded since they are taken in the category of super vector spaces, and $\cH^{\sW^{\mathrm{fr}}}_{\underline{\bd}}(\bv)$ is of homogeneous degree $|\bv|$ \eqref{equ on gra}. In dealing with endomorphisms on tensor products of super vector spaces, we have 
\begin{align*}
    \left(\id^{\otimes a}\otimes X\otimes \id^{\otimes b}\right)\left(\bigotimes_{i=1}^{a+b+1}v_i\right)=(-1)^{|X|\cdot\sum_{i=1}^a|v_i|}\left(\bigotimes_{i=1}^{a}v_i\right)\otimes Xv_i\:\otimes \left(\bigotimes_{j=a+1}^{a+b+1}v_j\right)
\end{align*}
for homogeneous endomorphism $X$ and homogeneous elements $v_1,\ldots,v_{a+b+1}$.
% $$(A\otimes B)\cdot(C\otimes D)=(-1)^{|B|\cdot|C|}AC\otimes BD. $$
\end{Remark}

\section{Reshetikhin type shifted Yangians for symmetric quivers with potentials}\label{sec shifted Yangian}

In this section, given a symmetric quiver $Q$ with potential $\sW$, and an auxiliary set $\mathcal C$ of framing data, we construct Reshetikhin type (anti-dominant) shifted  Yangians 
$\mathsf Y_\mu(Q,\sW,\mathcal C)$ using the matrix elements of the $R$-matrices, which naturally act on the critical cohomology of (framed) quiver varieties with potentials. 

We introduce the concept of admissibility of $\mathcal C$ (Definition \ref{def on adm aux data}), and show that $\mathsf Y_\mu(Q,\sW,\mathcal C)$ is independent of those admissible choices (Theorem \ref{thm admissible}), and will be simply denoted $\mathsf Y_\mu(Q,\sW)$.
This is remarkable as one can use a `small' (finite) admissible auxiliary set to generate the shifted Yangian assoicated 
to a `large' (infinite) admissible auxiliary set.
We also construct natural maps
from Drinfeld type shifted Yangians of the previous section to the Reshetikhin type shifted Yangians (Theorem \ref{thm drinfeld yangian map to rtt yangian}),
define coproducts by splitting $R$-matrices (\S \ref{sect on coprod}) and shift homomorhisms relating Yangians 
for different shifts (\S \ref{sec shift map}). 

In the zero shift case, we construct Maulik-Okounkov (MO) type Lie superalgebra $\mathfrak{g}_{Q,\sW}$ out of $(Q,\sW)$ and study their properties (Proposition \ref{prop g(Q,w)}). This recovers the MO Lie algebra for tripled quiver with canonical cubic potential (Theorem \ref{thm compare with MO yangian}). We conjecture that the positive part of $\mathfrak{g}_{Q,\sW}$ equals to the BPS Lie algebra of Davison-Meinhardt \cite{DM} (Conjecture \ref{conj: bps lie}). This generalizes a conjecture of the second author, proven recently by Botta-Davison \cite{BD}.

We prove a PBW basis theorem for $\mathsf Y_0(Q,\sW)$,~i.e.~there is a filtration on $\mathsf Y_0(Q,\sW)$ such that the associated graded is the enveloping algebra of $\mathfrak{g}_{Q,\sW}[u]$ (Theorem \ref{thm grY}).
This generalizes a theorem of Maulik-Okounkov \cite[Thm.\,5.5.1]{MO} from Nakajima quivers to symmetric quivers with potentials.

Several explicit formulas of coproducts are presented in \S \ref{expl formula of coprod}. For quivers with potentials assoicated with symmetrizable Kac-Moody Lie algebras (Appendix \ref{sec Y(sym KM)}), we expect the coproduct constructed using algebraic method by Guay-Nakajima-Wendlandt \cite{GNW} is compatible with the coproduct constructed from critical stable envelopes (Remark \ref{rmk on copr compa}), which we prove for all finite dimensional simple Lie algebras (Theorem \ref{thm on copr compa}).

We also study the Casimir operators (Definition \ref{def of Casimir}) and show that they are given by 
Steinberg correspondences (Proposition \ref{prop Casimir}). The corresponding shifted version is studied in Proposition \ref{prop shifted Casimir}.

\subsection{Reshetikhin type shifted Yangians}

\subsubsection{Definitions}

We fix a symmetric quiver $Q$ with $\sT_0$-action and $\sT_0$-invariant potential $\sW$, as in the beginning of \S\ref{sec double of COHA}, and we assume that
\begin{itemize}
    \item $\sT_0$-equivariant weights on every edge loop at $i\in Q_0$ is nonzero, equivalently $\gamma_i=\prod_{e:i\to i}t_e$ is nonzero for every $i\in Q_0$, where $t_a$ is the $\sT_0$-equivariant weight of the arrow $a\in Q_1$.
    \item For any dimension vectors $\bv,\bd_{\In},\bd_{\Out}\in \bN^{Q_0}$, the quiver variety $\cM(\bv,\bd)$ is defined with respect to the cyclic stability. 
\end{itemize}

\begin{Definition}\label{def on adm aux data}
An \textit{auxiliary data set} $\mathcal C=\{(\bd,\sT_0\curvearrowright\cM(\bd),\sW^{\mathrm{fr}})\}$ is a (possibly infinite) set of the following triples: (1) a framing dimension $\bd$, (2) a choice of $\sT_0$-action on the symmetric quiver variety 
$$\cM(\bd)=\bigsqcup_{\bv\in \bN^{Q_0}} \cM(\bv,\bd),$$ 
extending the $\sT_0$-action on the unframed quiver, (3) a $\sT_0$-invariant framed potential $\sW^{\mathrm{fr}}$ on $\cM(\bd)$ extending $\sW$.

For an element $\mathfrak{c}=(\bd,\sT_0\curvearrowright\cM(\bd),\sW^{\mathrm{fr}})\in \mathcal C$, the \textit{auxiliary space} is the equivariant critical cohomology:
\begin{equation}\label{equ on aux sp}\mathcal H_{\mathfrak{c}}:=H^{\sT_0}(\cM(\bd),\sw^{\mathrm{fr}}). \end{equation} 
%denote the equivariant critical cohomology $H^{\sT_0}(\cM(\bd),\sw^{\mathrm{fr}})$.
An auxiliary data set $\mathcal C$ is said to be \textit{admissible} if for every $i\in Q_0$, there exists an element $\mathfrak{c}_i=(\delta_i,\sT_0\curvearrowright\cM(\delta_i),\sW_i)\in \mathcal C$ such that $\sT_0$-weight on the out-going framing $B_i$ in $\cM(\delta_i)$ is nonzero (here we may assume $\sT_0$-weight of in-coming framing $A_i$ in $\cM(\delta_i)$ is zero without loss of generality).
In this case, $\{\mathcal H_{\mathfrak{c}_i}\}_{i\in Q_0}$ is called a collection of \textit{admissible auxiliary spaces}.
% this is by using gauge group action
\end{Definition}

\begin{Remark}
If an auxiliary datum $(\delta_i,\sT_0\curvearrowright\cM(\delta_i),\sW_i)$ is admissible, then $\sw_i$ is automatically admissible in the sense of Definition \ref{def on adm pot}. In fact, if we write $\sW_i\big|_{X_a=0,\:\forall a\in Q_1}=\sum_{k\geqslant 1}a_k (B_iA_i)^k$, then the $\sT_0$-invariance condition forces all $a_k$ to be zero.
\end{Remark}

\begin{Remark}
Admissible auxiliary data set always exists, since for each $i\in Q_0$, we can take the framed potential to be the trivial extension of $\sW$, and choose a nonzero $\sT_0$ weight on $B_i$ ($\sW$ is $\sT_0$-invariant automatically).
\end{Remark}

\begin{Definition}\label{def state space}
A \textit{state space} is the equivariant critical cohomology: 
\begin{equation}\label{equ on state sp}
    \cH^{\sW^{\mathrm{fr}}}_{\underline{\bd},\sA^{\mathrm{fr}}}:=H^{\sT_0\times \sA^{\mathrm{fr}}}(\cM(\underline{\bd}),\sw^{\mathrm{fr}}),
\end{equation}
where the $\sT_0$ action on $\cM(\underline{\bd})$ extends the $\sT_0$ action on the unframed quiver, $\sA^{\mathrm{fr}}$ acts on $\cM(\underline{\bd})$ as a framing torus, and $\sw^{\mathrm{fr}}=\tr(\sW^{\mathrm{fr}})$ for a $\sT_0\times \sA^{\mathrm{fr}}$ invariant framed potential $\sW^{\mathrm{fr}}$. 
% We also write $\sT_0^{\mathrm{fr}}:=\sT_0\times \sA^{\mathrm{fr}}$.
\end{Definition}
\begin{Remark}
The spaces \eqref{equ on aux sp}, \eqref{equ on state sp} have natural $\mathbb{Z}^{Q_0}$-grading, whose degree zero part is spanned by 
$$ [\pt]\in H^{\sT_0\times \sA^{\mathrm{fr}}}\left(\cM(\mathbf 0,{\bd}),\sw^{\mathrm{fr}}\right)=H_{\sT_0\times \sA^{\mathrm{fr}}}(\pt), \quad 
 [\pt]\in H^{\sT_0\times \sA^{\mathrm{fr}}}\left(\cM(\mathbf 0,\underline{\bd}),\sw^{\mathrm{fr}}\right)=H_{\sT_0\times \sA^{\mathrm{fr}}}(\pt).$$
\end{Remark}

\begin{Definition}\label{def of rtt yang}
Let $\mathcal C$ be an auxiliary data set. For a given $\mu\in \bZ_{\leqslant 0}^{Q_0}$, the \textit{Reshetikhin type} $\mu$-\textit{shifted Yangian} with respect to $\mathcal C$ is the $\bC(\mathsf t_0)$-subalgebra
\begin{align}\label{rtt yangian def}
    \mathsf Y_\mu(Q,\sW,\mathcal C)\subset \prod_{\underline{\bd},\sW^{\mathrm{fr}},\sA^{\mathrm{fr}}}\End^{\bZ/2}_{\bC(\mathsf t_0)[\mathsf a^{\mathrm{fr}}]}\left(\cH^{\sW^{\mathrm{fr}}}_{\underline{\bd},\mathsf A^{\mathrm{fr}},\mathrm{loc}}\right),
\end{align}
generated by endomorphisms of the form
\begin{align}\label{op E(m)}
    \mathsf E(m):=\underset{u\to \infty}{\Res}\str_{\cH_{\mathfrak{c},\loc}}\left((m\otimes \id)\cdot R^{\mathrm{sup}}_{\cH_{\mathfrak{c}},\cH^{\sW^{\mathrm{fr}}}_{\underline{\bd},\sA^{\mathrm{fr}}}}(u)\right),\quad m\in \End^{\bZ/2}_{\bC(\mathsf t_0)}(\cH_{\mathfrak{c}})[u],\quad \mathfrak{c}\in \mathcal C,
\end{align}
where the product in \eqref{rtt yangian def} is taken for all state spaces with $\bd_{\Out}-\bd_{\In}=\mu$, and the subscript ``loc'' means $-\otimes_{\bC[\mathsf t_0]}\bC(\mathsf t_0)$. The superscript $\bZ/2$ means endowing state spaces with super vector space structures in the sense of Remark \ref{rmk super YBE}, and ``$\str$'' stands for the super trace. 
Here the $R$-matrix can also be defined using the stable envelope as it is equal to the Hall envelope (Remark \ref{rmk on hallequstab}).

We define the (\textit{big}) \textit{Reshetikhin type} $\mu$-\textit{shifted Yangian} of $(Q,\sW)$ to be 
\begin{equation}\label{equ on rtt yangian}
    \mathsf Y_\mu(Q,\sW):=\mathsf Y_\mu(Q,\sW,\bigcup_{\mathcal C}\mathcal C)\subset \prod_{\underline{\bd},\sW^{\mathrm{fr}},\sA^{\mathrm{fr}}}\End^{\bZ/2}_{\bC(\mathsf t_0)[\mathsf a^{\mathrm{fr}}]}(\cH^{\sW^{\mathrm{fr}}}_{\underline{\bd},\sA^{\mathrm{fr}},\mathrm{loc}}),
\end{equation}
where the union is taken over all auxiliary data sets.
\end{Definition}

By definition, $\mathsf Y_\mu(Q,\sW,\mathcal C)$ is a $(\bZ^{Q_0}\times \bZ/2)$-graded $\bC(\mathsf t_0)$-algebra, where $\bZ^{Q_0}$ grading comes from the $\bZ^{Q_0}$ grading on $m\in \End^{\bZ/2}_{\bC(\mathsf t_0)}(\cH_{\mathfrak{c},\loc})$.

The following is obvious by definition.
\begin{Lemma}
If $\mathcal C\subset \mathcal C'$, then $\mathsf Y_\mu(Q,\sW,\mathcal C)\subseteq \mathsf Y_\mu(Q,\sW,\mathcal C')$.
\end{Lemma}

\begin{Theorem}
\label{thm drinfeld yangian map to rtt yangian}
If $\mathcal C$ is admissible, then for any $\mu\in \bZ_{\leqslant 0}^{Q_0}$, the image of the map $$\mathcal D\widetilde{\mathcal{SH}}_\mu(Q,\sW)\to \prod_{\underline{\bd},\sW^{\mathrm{fr}},\sA^{\mathrm{fr}}}\End^{\bZ/2}_{\bC(\mathsf t_0)[\mathsf a^{\mathrm{fr}}]}(\cH^{\sW^{\mathrm{fr}}}_{\underline{\bd},\sA^{\mathrm{fr}},\mathrm{loc}})$$ induced by the action in Theorem \ref{cor shifted yangian action} is contained in $\mathsf Y_\mu(Q,\sW,\mathcal C)$. In particular, the image of the above map is contained in $\mathsf Y_\mu(Q,\sW)$.
\end{Theorem}

\begin{proof}
Using the Gauss decomposition \eqref{gauss decomp of R}, we have
\begin{align*}
\sum_{k\in \bZ}z^{-k-1}
\begin{pmatrix}\mathsf E(|\mathbf 0\rangle\langle\mathbf 0|u^k) & \mathsf E(|\delta_i\rangle\langle\mathbf 0|u^k) \\
\mathsf E(|\mathbf 0\rangle\langle\delta_i|u^k) & \mathsf E(|\delta_i\rangle\langle\delta_i|u^k)
\end{pmatrix}
=
\begin{pmatrix}
g_i(z) & (-1)^{|i|}g_i(z)e_i(z)\\
\frac{(-1)^{|i|} t_i}{\gamma_i}f_i(z)g_i(z) & \frac{t_i}{\gamma_i}f_i(z)g_i(z)e_i(z)+(-1)^{|i|}g_i(z)h_i(z)
\end{pmatrix}\;,
\end{align*}
where $\langle\mathbf 0|$ denotes the dual vector of $|\mathbf 0\rangle$ in the ($\mathbb{Z}^{Q_0}$-graded) dual space, 
$|\mathbf 0\rangle\langle\mathbf 0|u^k\in \End^{\bZ/2}_{\bC(\mathsf t_0)}(\cH_{\mathfrak{c}})[u]$ for some $\mathfrak{c}$,
and $g_i(z)$ is given by \eqref{equ on gi}.
It follows that the images of $e_i(z),f_i(z),h_i(z)$ are contained in $\mathsf Y_\mu(Q,\sW,\mathcal C)$.
\end{proof}

\begin{Remark}
\eqref{op E(m)} makes sense for any rational function $m(u)$ of $u$, in particular, 
\begin{align*}
    \mathsf E(m\: u^{-k})=\begin{cases}
        \str_{\cH_{\mathfrak{c},\loc}}\left(m\cdot P_{Q_0\setminus \Supp(\mu)}\right)\:, & k=1\:,\\
        0\:, & k>1\:,
    \end{cases}
\end{align*}
where $P_{Q_0\setminus \Supp(\mu)}$ is the projector from $\cH_{\mathfrak{c}}$ to the subspace 
\begin{align*}
    \bigoplus_{\mu\cdot\bv=0}\cH_{\mathfrak{c}}(\bv).
\end{align*}
\end{Remark}

% \begin{Remark}\label{rmk taut class subalg}
% The diagonal part of the Gauss decomposition \eqref{gauss decomp of R} generates all tautological classes of $\mathsf V_i$, and $\mathsf D_{\Out,i}-\mathsf D_{\In,i}$. In particular, if $\mathcal C$ is admissible, then $\mathsf Y_{\mu}(Q,\sW,\mathcal C)$ contains a subalgebra generated by these tautological classes.
% \end{Remark}

\subsubsection{Triangular decomposition}

Write the Gauss decompositions $R^{\mathrm{sup}}(u)=F(u)G(u)E(u)$, where 
\begin{align}\label{gauss decomp}
    F(u)=\pm\Stab^{-1}_{u<0,\epsilon}\cdot \:e^{\sT_0\times \bC^*_u}(N^+),\quad G(u)=\pm\frac{e^{\sT_0\times \bC^*_u}(N^-)}{e^{\sT_0\times \bC^*_u}(N^+)},\quad E(u)=\pm e^{\sT_0\times \bC^*_u}(N^-)^{-1}\cdot\Stab_{u>0,\epsilon}\:,
\end{align}
where signs are determined by $\mathsf P$. Matrices elements of $F(u)$, $G(u)$, and $E(u)$ can be written in terms of matrix elements of $R^{\mathrm{sup}}(u)$ by Gauss elimination.

\begin{Definition}
Let $\mathcal C$ be an auxiliary data set. For $\mathfrak{c}\in \mathcal C$, $m\in \End^{\bZ/2}_{\bC(\mathsf t_0)}(\cH_{\mathfrak{c}})[u]$, define the following elements in $\mathsf Y_\mu(Q,\sW,\mathcal C)$:
\begin{align*}
    \mathsf E^+(m):=\underset{u\to \infty}{\Res}\str_{\cH_{\mathfrak{c},\loc}}\left((m\otimes \id)\cdot E(u)\right),\\
    \mathsf E^0(m):=\underset{u\to \infty}{\Res}\str_{\cH_{\mathfrak{c},\loc}}\left((m\otimes \id)\cdot G(u)\right),\\
    \mathsf E^-(m):=\underset{u\to \infty}{\Res}\str_{\cH_{\mathfrak{c},\loc}}\left((m\otimes \id)\cdot F(u)\right).
\end{align*}
Define $\mathsf Y^+_\mu(Q,\sW,\mathcal C)$, $\mathsf Y^0_\mu(Q,\sW,\mathcal C)$, and $\mathsf Y^-_\mu(Q,\sW,\mathcal C)$ to be $\bC(\mathsf t_0)$-subalgebras of $\mathsf Y_\mu(Q,\sW,\mathcal C)$ generated by $\mathsf E^+(m)$, $\mathsf E^0(m)$, and $\mathsf E^-(m)$ respectively. And define $\mathsf Y^{\geqslant }_\mu(Q,\sW,\mathcal C)$ (resp. $\mathsf Y^{\leqslant}_\mu(Q,\sW,\mathcal C)$) to be $\bC(\mathsf t_0)$-subalgebra of $\mathsf Y_\mu(Q,\sW,\mathcal C)$ generated by $\mathsf Y^+_\mu(Q,\sW,\mathcal C)$ (resp. $\mathsf Y^-_\mu(Q,\sW,\mathcal C)$) and $\mathsf Y^0_\mu(Q,\sW,\mathcal C)$. 
\end{Definition}

\begin{Lemma}\label{lem Y^0 generators}
If $\mathcal C$ is admissible, then $\mathsf Y^0_\mu(Q,\sW,\mathcal C)$ is generated by the multiplication by tautological classes
\begin{align}\label{taut class Y^0}
    \left\{\mathrm{ch}_k(\mathsf V_i),\:\mathrm{ch}_{k+1}(\mathsf D_{\In,i}-\mathsf D_{\Out,i})\right\}_{i\in Q_0}^{k\in \bZ_{\geqslant 0}}.
\end{align}
Here Chern characters are $(\sT_0\times \sA^{\text{fr}})$-equivariant.
\end{Lemma}

\begin{proof}
Observe that $\frac{e^{\sT_0\times \bC^*_u}(N^-)}{e^{\sT_0\times \bC^*_u}(N^+)}$ is a power series in $u$ whose coefficients on the state space are given by Chern classes of $\mathsf V_i$ and $\mathsf D_{\In,i}-\mathsf D_{\Out,i}$. Thus $\mathsf Y^0_\mu(Q,\sW,\mathcal C)$ is contained in the algebra generated by multiplication by tautological classes \eqref{taut class Y^0}. For an admissible $\mathfrak{c}\in \mathcal C$, take the elements $|\mathbf 0\rangle,|\delta_i\rangle\in \mathcal H_{\mathfrak{c}}$, and we have
\begin{align*}
\frac{1}{k!}\mathsf E^0(|\mathbf 0\rangle\langle\mathbf 0|\:u^k)&=-t_i\sum_{i\in Q_0}\mathrm{ch}_k(\mathsf V_i)+\cdots\:,\\
\frac{(-1)^{|i|}}{k!}\mathsf E^0(|\delta_i\rangle\langle\delta_i|\:u^k)&=-t_i\sum_{i\in Q_0}\mathrm{ch}_k(\mathsf V_i)+\sum_{i\in Q_0}\mathrm{ch}_{k+1}(\mathsf D_{\In,i}-\mathsf D_{\Out,i})-\sum_{i,j\in Q_0}\pmb Q_{ij}\:\mathrm{ch}_k(\mathsf V_j)+\cdots\:,
\end{align*}
for $\pmb Q_{ij}=\sum_{a:i\to j}t_a+\sum_{b:j\to i}t_b$, where $t_a$ is the $\sT_0$-weight of an arrow $a$. Here the dot term stands for lower degree tautological classes. Inductively on $k$, we see that all elements in \eqref{taut class Y^0} are contained in $\mathsf Y^0_\mu(Q,\sW,\mathcal C)$.
% Let us introduce the following notation: $\mathsf v_i$ is the linear functional whose value on $\cH^{\sW^{\mathrm{fr}}}_{\underline{\bd}}(\bv)$ is $\bv_i$, $\mathsf d_i$ is the linear functional whose value on $\cH^{\sW^{\mathrm{fr}}}_{\underline{\bd}}(\bv)$ is $c_1(\mathsf D_{\In,i})-c_1(\mathsf D_{\Out,i})$. For a matrix $m$ constant in $u$, 
% direct computation shows that
% \begin{multline*}
% \frac{1}{k!}\mathsf E^0(m\:u^k)=\sum_{i\in Q_0}(m,\mathsf d_i)\:\mathrm{ch}_k(\mathsf V_i)+\sum_{i\in Q_0}(m,\mathsf v_i)\:\mathrm{ch}_{k+1}(\mathsf D_{\In,i}-\mathsf D_{\Out,i})-\sum_{i,j\in Q_0}\pmb Q_{ij}\:(m,\mathsf v_i)\:\mathrm{ch}_k(\mathsf V_j)+\cdots,
% \end{multline*}
% for $\pmb Q_{ij}=\sum_{a:i\to j}t_a+\sum_{b:j\to i}t_b$, where $t_a$ is the $\sT_0$-weight of an arrow $a$. Here paring with $\mathsf v_i$ and $\mathsf d_i$ is the trace pairing, and the dot term stands for lower degree tautological classes.  
\end{proof}

% \begin{Proposition}\label{prop coeff of stab}
% If $\mathcal C$ is admissible, then $\mathsf Y^{\geqslant}_\mu(Q,\sW,\mathcal C)$ and $\mathsf Y^{\leqslant}_\mu(Q,\sW,\mathcal C)$ are generated by endomorphisms of the form
% \begin{align*}
%     \mathsf E^{\geqslant}(m):=\underset{u\to \infty}{\Res}\str_{\cH_{\mathfrak{c},\loc}}\left((m\otimes \id)\cdot \Stab_{u>0,\epsilon}\right),\quad m\in \End^{\bZ/2}_{\bC(\mathsf t_0)}(\cH_{\mathfrak{c}})[u^{\pm}],\quad \mathfrak{c}\in \mathcal C,
% \end{align*}
% and 
% \begin{align*}
%     \mathsf E^{\leqslant}(m):=\underset{u\to \infty}{\Res}\str_{\cH_{\mathfrak{c},\loc}}\left((m\otimes \id)\cdot \Stab_{u<0,\epsilon}\right),\quad m\in \End^{\bZ/2}_{\bC(\mathsf t_0)}(\cH_{\mathfrak{c}})[u^{\pm}],\quad \mathfrak{c}\in \mathcal C,
% \end{align*}
% respectively.
% \end{Proposition}

It follows from the Gauss decomposition and the definition that $\mathsf Y^+_\mu(Q,\sW,\mathcal C)$, $\mathsf Y^0_\mu(Q,\sW,\mathcal C)$, and $\mathsf Y^-_\mu(Q,\sW,\mathcal C)$ generate $\mathsf Y_\mu(Q,\sW,\mathcal C)$.  When $\mathcal C$ is admissible, 
the following stronger statement holds.

\begin{Proposition}\label{prop gen by deg 0 and Cartan}
Assume $\mathcal C$ is admissible, then $\mathsf Y_\mu(Q,\sW,\mathcal C)$ is generated by $\{\mathsf E^{\pm}(m)\:|\: m\text{ is constant in }u\}$ and $\mathsf Y^0_\mu(Q,\sW,\mathcal C)$.
\end{Proposition}

\begin{proof}
$\mathsf Y_\mu(Q,\sW,\mathcal C)$ is filtered by degree in $u$: for $k\in \bZ_{\geqslant 0}$, $F_k\mathsf Y_\mu(Q,\sW,\mathcal C)$ is spanned by monomials $\mathsf E(m_1u^{n_1})\cdots \mathsf E(m_ru^{n_r})$ with $m_i\in \End^{\bZ/2}(\cH_{\mathfrak{c}_i})$, $\mathfrak{c}_i\in \mathcal C$, and $\sum n_i\leqslant k$. Note that in the Gauss decomposition 
$$R^{\mathrm{sup}}(u)=F(u)G(u)E(u),$$ $F(u)$ and $E(u)$ are triangular with $u$-expansions of form $\id+O(u^{-1})$, and $G(u)$ is diagonal with $u$-expansion of form 
\begin{align*}
    u^{\mu\cdot\bv_1}+\text{lower degree in }u
\end{align*}
on the component $\cH_{\mathfrak{c}}(\bv_1)\otimes\cH^{\sW^\mathrm{fr}}_{\underline{\bd},\sA^{\mathrm{fr}}}(\bv_2)$. As a consequence of the $u$-expansion being nonpositive, for a $\bZ^{Q_0}$-homogeneous $m$ constant in $u$, if $\mathsf E(m)\neq 0$, then $\deg(m)$ must be either positive or negative or zero, and there must be equality
\begin{align*}
    \mathsf E(m)=\begin{cases}
        \mathsf E^+(m) , & \deg(m)>0,\\
        \mathsf E^0(m) , & \deg(m)=0,\\
        \mathsf E^-(m) , & \deg(m)<0.\\
    \end{cases}
\end{align*}
Then it is clear that $F_0\:\mathsf Y_\mu(Q,\sW,\mathcal C)$ is contained in the subalgebra generated by $\{\mathsf E^{\pm}(m)\:|\: m\text{ is constant in }u\}$ and $\mathsf Y^0_\mu(Q,\sW,\mathcal C)$.

$\mathsf Y_\mu(Q,\sW,\mathcal C)$ has another filtration labelled by $\bZ_{\geqslant 0}^{Q_0}\times \bZ_{\geqslant 0}^{Q_0}$. Define $G_{\bv_1,\bv_2}\mathsf Y_\mu(Q,\sW,\mathcal C)$ to be the subalgebra generated by $\mathsf E(m)$ for $m\in \Hom(\cH_{\mathfrak{c}}(\bv'_1),\cH_{\mathfrak{c}}(\bv'_2))[u]$ with 
\begin{align*}
    \bv_1\geqslant \bv'_1\geqslant\mathbf 0,\;\text{ and }\: \bv_2\geqslant \bv'_2\geqslant\mathbf 0.
\end{align*}
Note that $G_{\mathbf 0,\mathbf 0}\mathsf Y_\mu(Q,\sW,\mathcal C)$ is generated by multiplication by tautological classes $\{\mathrm{ch}_k(\mathsf V_i)\}_{i\in Q_0}^{k\in \bZ_{\geqslant 0}}$, and in particular $G_{\mathbf 0,\mathbf 0}\mathsf Y_\mu(Q,\sW,\mathcal C)\subset \mathsf Y^0_\mu(Q,\sW,\mathcal C)$. We also define $G^<_{\bv_1,\bv_2}\mathsf Y_\mu(Q,\sW,\mathcal C)$ to be the subalgebra generated by 
\begin{align*}
    \bigcup_{\substack{\bv_1\geqslant \bv'_1,\bv_2\geqslant \bv'_2\\ (\bv_1,\bv_2)\neq (\bv'_1,\bv'_2)}} \:G_{\bv'_1,\bv'_2}\mathsf Y_\mu(Q,\sW,\mathcal C).
\end{align*}
We prove that $\mathsf Y_\mu(Q,\sW,\mathcal C)$ is generated by $\mathsf Y^0_\mu(Q,\sW,\mathcal C)$ and $\{\mathsf E^{\pm}(m)\:|\: m\text{ is constant in }u\}$ by using double induction on the filtrations $G_{\bullet,\bullet}$ and $F_\bullet$. The initial cases $G_{\mathbf 0,\mathbf 0}\mathsf Y_\mu(Q,\sW,\mathcal C)$ and $F_0\:\mathsf Y_\mu(Q,\sW,\mathcal C)$ are known.

% for arbitrary $(\bv_1,\bv_2)\in \bZ_{\geqslant 0}^{Q_0}\times \bZ_{\geqslant 0}^{Q_0}$ and take arbitrary $X\in \Hom(\cH_{\mathfrak{c}}(\bv_1),\cH_{\mathfrak{c}}(\bv_2))$ and arbitrary $k\in \bZ_{\geqslant 0}$, $\mathsf E(X\: u^k)$ can be

Assume that $G^<_{\bv_1,\bv_2}\mathsf Y_\mu(Q,\sW,\mathcal C)$ can be generated. Take arbitrary $X\in \Hom(\cH_{\mathfrak{c}}(\bv_1),\cH_{\mathfrak{c}}(\bv_2))$ and arbitrary $k\in \bZ_{\geqslant 0}$, we need to show that $\mathsf E(X\: u^k)$ can be generated.

If $\deg(X)=\bv_2-\bv_1=0$, then \cite[(4.34)]{MO} still applies to the shifted case and we get
\begin{align*}
    \mathsf E(X\: u^k)\equiv \mathsf E^{0}(X\: u^k)\mod G^<_{\bv_1,\bv_2}\mathsf Y_\mu(Q,\sW,\mathcal C),
\end{align*}
so $\mathsf E(m\: u^k)$ can be generated.

If $\deg(X)_i=\bv_{2,i}-\bv_{1,i}\neq 0$ for some $i\in Q_0$, then we prove by induction on $k$ that $\mathsf E(X\: u^k)$ can be generated. The initial case $k=0$ is known from the above discussions. Take the vacuum state $|\mathbf 0\rangle$ in an admissible auxiliary space $\cH^{\sW_i}_{\underline{\delta_i}}$, and let $t_i$ be the $\sT_0$ weight of the out-going arrow on the node $i$ for $\cH^{\sW_i}_{\underline{\delta_i}}$, which is nontrivial by admissibility. Then 
\begin{align*}
    -t_i\deg(X)_i\:\mathsf E(X\:u^k)&=[\mathsf E(|\mathbf 0\rangle\langle\mathbf 0|),\mathsf E(X\:u^k)]\\
    \text{\tiny by Lemma \ref{lem fund induction}}\quad&=[\mathsf E(|\mathbf 0\rangle\langle\mathbf 0|\:u),\mathsf E(X\: u^{k-1})]+\mathsf E(X'\: u^{k-1})+\mathsf E(X''\: u^{k-1})\cdot\mathsf E(|\mathbf 0\rangle\langle\mathbf 0|)+\text{term in }G^<_{\bv_1,\bv_2}\mathsf Y_\mu(Q,\sW,\mathcal C),
\end{align*}
for certain $X',X''\in \Hom(\cH_{\mathfrak{c}}(\bv_1),\cH_{\mathfrak{c}}(\bv_2))$. Here we have used Lemma \ref{lem fund induction} which is proven below. Note that for any $n$, $\mathsf E(|\mathbf 0\rangle\langle\mathbf 0|\:u^n)\in \mathsf Y^0_\mu(Q,\sW,\mathcal C)$. Then $\mathsf E(X\:u^k)$ can be generated by induction. This finishes the proof.
\end{proof}

% We postpone the proof of the above proposition until Section \ref{sec Y(Q,C)=Y(Q,w)}.

\begin{Lemma}\label{lem fund induction}
Let $\mathfrak{c}_1,\mathfrak{c}_2$ be arbitrary auxiliary data, $|\mathbf 0\rangle$ be the vacuum of $ \cH_{\mathfrak{c}_1}(\mathbf 0)$, $m\in \Hom(\cH_{\mathfrak{c}_2}(\bv_1),\cH_{\mathfrak{c}_2}(\bv_2))$. Then for any $k\in \bZ_{\geqslant 0}$, we have
\begin{align}\label{fund induction}
[\mathsf E(|\mathbf 0\rangle\langle\mathbf 0|\: u),\mathsf E(m\: u^{k})]-[\mathsf E(|\mathbf 0\rangle\langle\mathbf 0|),\mathsf E(m\: u^{k+1})]&=\mathsf E(m'\: u^{k})+\mathsf E(m''\: u^{k})\:\mathsf E(|\mathbf 0\rangle\langle\mathbf 0|)\\ \nonumber
&\quad +\sum_{\bv_2\geqslant\bv> 0}\mathsf E(f_{\bv}\: u^{k})\mathsf E(g_{\bv})+\sum_{\bv_1\geqslant\bv> 0}\mathsf E(f'_{\bv}\: u^{k}) \mathsf E(g'_{\bv}),
\end{align}
for certain $m',m''\in \Hom(\cH_{\mathfrak{c}_2}(\bv_1),\cH_{\mathfrak{c}_2}(\bv_2))$, $g_{\bv}\in \Hom(\cH_{\mathfrak{c}_1}(\mathbf 0),\cH_{\mathfrak{c}_1}(\bv))$, $f_{\bv}\in \Hom(\cH_{\mathfrak{c}_2}(\bv_1),\cH_{\mathfrak{c}_2}(\bv_2-\bv))$, $g'_{\bv}\in \Hom(\cH_{\mathfrak{c}_1}(\bv),\cH_{\mathfrak{c}_1}(\mathbf 0))$, $f'_{\bv}\in \Hom(\cH_{\mathfrak{c}_2}(\bv_1-\bv),\cH_{\mathfrak{c}_2}(\bv_2))$.
\end{Lemma}

\begin{proof}
By the RTT=TTR equation, we have
\begin{align*}
[\mathsf E(|\mathbf 0\rangle\langle\mathbf 0|\: v),\mathsf E(m\: u^{k})]=\underset{v\to \infty}{\Res}\underset{u\to \infty}{\Res}\str_{\cH_{\mathfrak{c}_1}\otimes \cH_{\mathfrak{c}_2}}\left[\left(R_{12}(v-u)(|\mathbf 0\rangle\langle\mathbf 0|\: v\:\otimes m\: u^{k})R_{12}(v-u)^{-1}- |\mathbf 0\rangle\langle\mathbf 0|\: v\:\otimes m\: u^{k}\right)T_2(u)T_{1}(v)\right],
\end{align*}
where $T_i$ $(i=1,2)$ is the $R$-matrix that braids $\cH_{\mathfrak{c}_i}$ with a state space. By taking expansions 
$$R_{12}(v-u)=\id+\frac{\pmb r_{12}}{v-u}+\frac{\pmb s}{(v-u)^2}+O((v-u)^{-3}),$$ 
$$R_{12}(v-u)^{-1}=\id-\frac{\pmb r_{12}}{v-u}+\frac{\pmb s'}{(v-u)^2}+O((v-u)^{-3}),$$ 
we get
\begin{multline*}
[\mathsf E(|\mathbf 0\rangle\langle\mathbf 0|\: v),\mathsf E(m\: u^{k})]=\underset{v\to \infty}{\Res}\underset{u\to \infty}{\Res}\str_{\cH_{\mathfrak{c}_1}\otimes \cH_{\mathfrak{c}_2}}\\
\left[\left([\pmb r_{12},|\mathbf 0\rangle\langle\mathbf 0| \otimes m](u^k+v^{-1}u^{k+1})+\pmb s\cdot(|\mathbf 0\rangle\langle\mathbf 0| \otimes m)v^{-1}u^k+(|\mathbf 0\rangle\langle\mathbf 0| \otimes m)\cdot \pmb s'\:v^{-1}u^k\right)T_2(u)T_{1}(v)\right].
\end{multline*}
We ignore the higher order in $u/v$ terms in the above expansion because $T_1(v)$ has no pole as $v\to \infty$. The same argument shows that
\begin{align*}
[\mathsf E(|\mathbf 0\rangle\langle\mathbf 0|\: v),\mathsf E(m\: u^{k})]=\underset{v\to \infty}{\Res}\underset{u\to \infty}{\Res}\str_{\cH_{\mathfrak{c}_1}\otimes \cH_{\mathfrak{c}_2}}
\left[[\pmb r_{12},|\mathbf 0\rangle\langle\mathbf 0| \otimes m]v^{-1}u^{k+1}\:T_2(u)T_{1}(v)\right].
\end{align*}
It follows that
\begin{multline*}
\text{LHS of \eqref{fund induction}}=\underset{v\to \infty}{\Res}\underset{u\to \infty}{\Res}\str_{\cH_{\mathfrak{c}_1}\otimes \cH_{\mathfrak{c}_2}}
\\
\left[\left([\pmb r_{12},|\mathbf 0\rangle\langle\mathbf 0| \otimes m]\:u^{k}+\pmb s\cdot(|\mathbf 0\rangle\langle\mathbf 0| \otimes m)v^{-1}u^k+(|\mathbf 0\rangle\langle\mathbf 0| \otimes m)\cdot \pmb s'\:v^{-1}u^k\right)T_2(u)T_{1}(v)\right].
\end{multline*}
Since the nonvanishing terms in $\lim_{v\to \infty}T_1(v)$ are concentrated in the diagonal, we have 
\begin{align*}
    \underset{v\to \infty}{\Res}\underset{u\to \infty}{\Res}\str_{\cH_{\mathfrak{c}_1}\otimes \cH_{\mathfrak{c}_2}}
\left[\left(\pmb s\cdot(|\mathbf 0\rangle\langle\mathbf 0| \otimes m)v^{-1}u^k+(|\mathbf 0\rangle\langle\mathbf 0| \otimes m)\cdot \pmb s'\:v^{-1}u^k\right)T_2(u)T_{1}(v)\right]=\mathsf E\left(m'\:u^k \right),
\end{align*}
where $m'=\langle\mathbf 0|\pmb s |\mathbf 0\rangle \cdot m+m\cdot\langle\mathbf 0|\pmb s' |\mathbf 0\rangle$ is a homomorphism from $\cH_{\mathfrak{c}_2}(\bv_1)$ to $\cH_{\mathfrak{c}_2}(\bv_2)$. On the other hand, we can write the other term in the LHS of \eqref{fund induction} in homogeneous components and get 
\begin{align*}
    \underset{v\to \infty}{\Res}\underset{u\to \infty}{\Res}\str_{\cH_{\mathfrak{c}_1}\otimes \cH_{\mathfrak{c}_2}}
\left[\pmb r_{12}\:(|\mathbf 0\rangle\langle\mathbf 0| \otimes m)u^{k}\:T_2(u)T_{1}(v)\right]=\sum_{\bv_2\geqslant\bv\geqslant 0}\mathsf E(f_{\bv}\: u^{k})\mathsf E(g_{\bv}),
\end{align*}
where $g_{\bv}\otimes f_{\bv}$ is the short-hand notation for the $\Hom(\cH_{\mathfrak{c}_1}(\mathbf 0)\otimes \cH_{\mathfrak{c}_2}(\bv_1),\cH_{\mathfrak{c}_1}(\bv)\otimes \cH_{\mathfrak{c}_2}(\bv_2-\bv))$ component of $\pmb r_{12}\:(|\mathbf 0\rangle\langle\mathbf 0| \otimes m)$. $g_{\bv}\otimes f_{\bv}$ should be actually a linear combination of this form but we suppress the summation for simplicity. Similarly,
\begin{align*}
    \underset{v\to \infty}{\Res}\underset{u\to \infty}{\Res}\str_{\cH_{\mathfrak{c}_1}\otimes \cH_{\mathfrak{c}_2}}
\left[(|\mathbf 0\rangle\langle\mathbf 0| \otimes m)\:\pmb r_{12}\:u^{k}\:T_2(u)T_{1}(v)\right]=\sum_{\bv_2\geqslant\bv\geqslant 0}\mathsf E(f'_{\bv}\: u^{k})\mathsf E(g'_{\bv})
\end{align*}
for $g'_{\bv}\otimes f'_{\bv}$ being the $\Hom(\cH_{\mathfrak{c}_1}(\bv)\otimes \cH_{\mathfrak{c}_2}(\bv_1-\bv),\cH_{\mathfrak{c}_1}(\mathbf 0)\otimes \cH_{\mathfrak{c}_2}(\bv_2))$ component of $(|\mathbf 0\rangle\langle\mathbf 0| \otimes m)\:\pmb r_{12}$. Combine the above three equations and let $\mathsf E(m''\: u^{k})\:\mathsf E(|\mathbf 0\rangle\langle\mathbf 0|)=\mathsf E(f_{\mathbf 0}\: u^{k})\mathsf E(g_{\mathbf 0})+\mathsf E(f'_{\mathbf 0}\: u^{k})\mathsf E(g'_{\mathbf 0})$, we get \eqref{fund induction}.
\end{proof}

\subsection{Coproducts}\label{sect on coprod}
Suppose there is a decomposition of framing torus $\sA^{\mathrm{fr}}=\sA^{\mathrm{fr}}_1\times \sA^{\mathrm{fr}}_2$ with corresponding decomposition of framing: $\underline{\bd}=\underline{\bd}'+\underline{\bd}''$. Consider the $+$ chamber $\{\mathsf a^{\mathrm{fr}}_1> \mathsf a^{\mathrm{fr}}_2\}$ and the corresponding stable envelope
\begin{align*}
    \Stab_{+,\epsilon}\colon \cH^{\sW^{\mathrm{fr}}}_{\underline{\bd}',\sA^{\mathrm{fr}}_1,\loc}\otimes \cH^{\sW^{\mathrm{fr}}}_{\underline{\bd}'',\sA^{\mathrm{fr}}_2,\loc}\to \cH^{\sW^{\mathrm{fr}}}_{\underline{\bd},\sA^{\mathrm{fr}}_1\times \sA^{\mathrm{fr}}_2,\loc}.
\end{align*}
The triangle lemma implies that
\begin{align*}
\Stab_{+,\epsilon}^{-1}R^{\mathrm{sup}}_{\underline{\bd}}\Stab_{+,\epsilon}=R^{\mathrm{sup}}_{\underline{\bd}''}\cdot R^{\mathrm{sup}}_{\underline{\bd}'},
\end{align*}
where $R^{\mathrm{sup}}_{\underline{\bd}}$ stands for the $R$-matrix braiding on auxiliary space $\cH_{\mathfrak{c}}$ with $\cH^{\sW^{\mathrm{fr}}}_{\underline{\bd},\sA^{\mathrm{fr}}_1\times \sA^{\mathrm{fr}}_2}$, and similarly for $R_{\underline{\bd}'}$ and $R_{\underline{\bd}''}$.
It follows that $\mathsf E(m\: u^k)$'s of $\mathsf Y_\mu(Q,\sW,\mathcal C)$ transform as 
\begin{align}\label{stab induces coproduct}
    \Stab_{+,\epsilon}^{-1}\mathsf E(m\: u^k)\Stab_{+,\epsilon}=\sum_{i=-1}^k\sum (-1)^{|m_1|\cdot |m_2|}\mathsf E(m_2\:u^i)\otimes \mathsf E(m_1\:u^{k-1-i})\:,
\end{align}
where $\sum m_1\otimes m_2$ is the image of $m$ under the dual multiplication map $\End^{\bZ/2}(\cH_{\mathfrak{c},\loc})\to \End^{\bZ/2}(\cH_{\mathfrak{c},\loc})\otimes \End^{\bZ/2}(\cH_{\mathfrak{c},\loc})$. Explicitly, $\sum m_1\otimes m_2$ is determined by requiring
\begin{align*}
    \str(m\cdot x \cdot y)=\sum\str(m_1\cdot x)\: \str(m_2\cdot y)\:,
\end{align*}
for all $x,y\in \End^{\bZ/2}(\cH_{\mathfrak{c},\loc})$.

Fix $\mu_1,\mu_2\in \bZ_{\leqslant 0}^{Q_0}$. Collecting all decompositions $\underline{\bd}=\underline{\bd}'+\underline{\bd}''$ with $$\bd'_{\Out}-\bd'_{\In}=\mu_1, \quad \bd''_{\Out}-\bd''_{\In}=\mu_2,$$ and all framed potentials $\sw^{\mathrm{fr}}$, we obtain a map
\begin{align*}
    \prod\End^{\bZ/2}(\cH^{\sW^{\mathrm{fr}}}_{\underline{\bd},\sA^{\mathrm{fr}}_1\times \sA^{\mathrm{fr}}_2,\loc})\to \prod \End^{\bZ/2}(\cH^{\sW^{\mathrm{fr}}}_{\underline{\bd}',\sA^{\mathrm{fr}}_1,\loc})\otimes \End^{\bZ/2}(\cH^{\sW^{\mathrm{fr}}}_{\underline{\bd}'', \sA^{\mathrm{fr}}_2,\loc}),
\end{align*}
which maps the subalgebra $\mathsf Y_{\mu_1+\mu_2}(Q,\sW,\mathcal C)$ to 
\begin{align*}
    \mathsf Y_{\mu_1}(Q,\sW,\mathcal C)\widehat{\otimes}\mathsf Y_{\mu_2}(Q,\sW,\mathcal C)\subset \prod \End^{\bZ/2}(\cH^{\sW^{\mathrm{fr}}}_{\underline{\bd}',\sA^{\mathrm{fr}}_1,\loc})\otimes \End^{\bZ/2}(\cH^{\sW^{\mathrm{fr}}}_{\underline{\bd}'', \sA^{\mathrm{fr}}_2,\loc}).
\end{align*}
This induces a \textit{coproduct}:
\begin{align*}
    \Delta_{\mu_1,\mu_2}\colon \mathsf Y_{\mu_1+\mu_2}(Q,\sW,\mathcal C)\to \mathsf Y_{\mu_1}(Q,\sW,\mathcal C)\widehat{\otimes}\mathsf Y_{\mu_2}(Q,\sW,\mathcal C)\:.
\end{align*}
When $\mu_1=\mu_2=0$, we also write $\Delta:=\Delta_{0,0}$. 

By the triangle lemma, we have: 
\begin{Proposition}
The coproduct is \textit{coassociative}, that is,
\begin{align*}
    (\Delta_{\mu_1,\mu_2}\otimes \id)\circ\Delta_{\mu_1+\mu_2,\mu_3}=(\id\otimes \Delta_{\mu_2,\mu_3})\circ\Delta_{\mu_1,\mu_2+\mu_3},\quad \mu_{1,2,3}\in \bZ_{\leqslant 0}^{Q_0}\:.
\end{align*}
\end{Proposition}

\begin{Remark}\label{rmk counit}
($\Delta$ is counital): Let $\cH^{\sW}_{\underline{\mathbf 0}}=\bC[\mathsf t_0]$ be a state space with zero shift, and $\varepsilon\colon \mathsf Y_{0}(Q,\sW,\mathcal C)\to \bC(\mathsf t_0)$ be the map induced by projection from $\prod\End^{\bZ/2}(\cH^{\sW^{\mathrm{fr}}}_{\underline{\bd},\sA^{\mathrm{fr}},\loc})$ to $\End^{\bZ/2}(\cH^{\sW}_{\underline{\mathbf 0},\loc})=\bC(\mathsf t_0)$. Then 
\begin{align*}
    (\varepsilon\otimes \id)\circ \Delta_{0,\mu}=\id,\quad (\id\otimes\, \varepsilon)\circ \Delta_{\mu,0}=\id.
\end{align*}
\end{Remark}

The following is obvious from definition:
\begin{Lemma}\label{lem closed under coproduct}
If $\mathcal C\subset\mathcal C'$, then the subalgebra $\mathsf Y_{\mu}(Q,\sW,\mathcal C)\subseteq \mathsf Y_{\mu}(Q,\sW,\mathcal C')$ is closed under coproduct, that is, the diagram commutes
\begin{equation*}
\xymatrix{
\mathsf Y_{\mu_1+\mu_2}(Q,\sW,\mathcal C) \ar[r]^-{\Delta_{\mu_1,\mu_2}} \ar@{^{(}->}[d] & \mathsf Y_{\mu_1}(Q,\sW,\mathcal C)\widehat{\otimes} \mathsf Y_{\mu_2}(Q,\sW,\mathcal C) \ar@{^{(}->}[d]\\
\mathsf Y_{\mu_1+\mu_2}(Q,\sW,\mathcal C') \ar[r]^-{\Delta_{\mu_1,\mu_2}} & \mathsf Y_{\mu_1}(Q,\sW,\mathcal C')\widehat{\otimes} \mathsf Y_{\mu_2}(Q,\sW,\mathcal C').
}
\end{equation*}
\end{Lemma}

\begin{Remark}\label{rmk opposite coproduct}
Define the stable envelope with opposite normalizer $\bar\epsilon$ be
\begin{align*}
    \Stab_{\fC,\bar\epsilon}:=\Stab_{\fC}\cdot\, \bar\epsilon,\;\text{ with }\; \bar\epsilon|_F=(-1)^{\rk N^+_{F/\cM(\bv,\underline{\bd})}+\rk \mathsf P|_F^-}\;,
\end{align*}
then $R^{\mathrm{sup}}(u)$ can be written as $\Stab_{-,\bar\epsilon}^{-1}\circ \Stab_{+,\epsilon}$.
The triangle lemma implies that
\begin{align*}
\Stab_{-,\bar\epsilon}^{-1}R^{\mathrm{sup}}_{\underline{\bd}}\Stab_{-,\bar\epsilon}=R^{\mathrm{sup}}_{\underline{\bd}'}\cdot R^{\mathrm{sup}}_{\underline{\bd}''}.
\end{align*}
Then $\mathsf E(m\: u^k)$'s of $\mathsf Y_\mu(Q,\sW,\mathcal C)$ transform as 
\begin{align}
    \Stab_{-,\bar\epsilon}^{-1}\mathsf E(m\: u^k)\Stab_{-,\bar\epsilon}=\sum_{i=-1}^k\sum \mathsf E(m_1\:u^i)\otimes \mathsf E(m_2\:u^{k-1-i})\:,
\end{align}
where $\sum m_1\otimes m_2$ is determined by requiring $\str(m\cdot x \cdot y)=\sum\str(m_1\cdot x)\: \str(m_2\cdot y)$ for all $x,y\in \End^{\bZ/2}(\cH_{\mathfrak{c},\loc})$.
Then $\Stab_{-,\bar\epsilon}^{-1}(\cdots)\Stab_{-,\bar\epsilon}$ induces a map 
\begin{align*}
    \Delta_{\mu_1,\mu_2}^{\mathrm{op}}\colon \mathsf Y_{\mu_1+\mu_2}(Q,\sW,\mathcal C)\to \mathsf Y_{\mu_1}(Q,\sW,\mathcal C)\widehat{\otimes}\mathsf Y_{\mu_2}(Q,\sW,\mathcal C)\:.
\end{align*}
which we call the \textit{opposite coproduct}. $\Delta_{\mu_1,\mu_2}^{\mathrm{op}}$ is coassociative and counital, and it is related to $\Delta_{\mu_1,\mu_2}$ by 
\begin{align*}
    R^{\mathrm{sup}}_{\underline{\bd}',\underline{\bd}''}\cdot \Delta_{\mu_1,\mu_2}(x)\cdot (R^{\mathrm{sup}}_{\underline{\bd}',\underline{\bd}''})^{-1}=\Delta_{\mu_1,\mu_2}^{\mathrm{op}}(x),\quad\forall\, x\in \mathsf Y_{\mu_1+\mu_2}(Q,\sW,\mathcal C),
\end{align*}
where $R^{\mathrm{sup}}_{\underline{\bd}',\underline{\bd}''}=\Stab_{-,\bar\epsilon}^{-1}\Stab_{+,\epsilon}$ is the $R$-matrix braiding $\cH^{\sW^{\mathrm{fr}}}_{\underline{\bd}',\sA^{\mathrm{fr}}_1}$ with $\cH^{\sW^{\mathrm{fr}}}_{\underline{\bd}'',\sA^{\mathrm{fr}}_2}$.
\end{Remark}

\begin{Remark}\label{rmk steinberg comm R}
Suppose that $\bd'_{\In}=\bd'_{\Out}$, $\bd''_{\In}=\bd''_{\Out}$, 
and let $\mathcal L$ be a Steinberg correspondence on $\cM(\bv,\underline{\bd})\times_{\cM_0(\bv,\underline{\bd})} \cM(\bv,\underline{\bd})$ (see Definition \ref{def of Steinberg corr}), then Theorem \ref{thm Stab and Steinberg} implies that
\begin{align*}
\Stab_{+,\epsilon}^{-1}\circ\mathcal L\circ\Stab_{+,\epsilon} = \Stab_{-,\bar\epsilon}^{-1}\circ\mathcal L\circ\Stab_{-,\bar\epsilon}
\end{align*}
equivalently,
\begin{align*}
\left[R^{\mathrm{sup}}_{\underline{\bd}',\underline{\bd}''}\;,\;\Stab_{+,\epsilon}^{-1}\circ\mathcal L\circ\Stab_{+,\epsilon}\right]=0
\end{align*}
\end{Remark}

\subsection{The Lie superalgebra \texorpdfstring{$\mathfrak{g}_{Q,\sW}$}{g(Q,w)}} 

Throughout this subsection, we focus on the symmetric framing $\bd_{\In}=\bd_{\Out}$, or equivalently the zero shift $\mu=0$ case.

\begin{Definition}
Let $\mathfrak{g}_{Q,\sW}\subset \mathsf Y_0(Q,\sW)$ be $\Span_{\bC(\mathsf t_0)}\{\mathsf E(m_0)\:|\: m_0\text{ is constant in }u\}$.
\end{Definition}

Using the expansion \eqref{expansion_classical r}, we see that $\mathfrak{g}_{Q,\sW}$ consists of matrix elements of the classical $R$-matrix $\pmb r$. This implies the following.
\begin{Proposition}
Every element $x\in\mathfrak{g}_{Q,\sW}$ is primitive:
\begin{align*}
    \Delta(x)=x\otimes 1+1\otimes x.
\end{align*}
In particular, $\Delta(x)$ commutes with $R$ matrices:
\begin{align*}
    [\Delta(x),R^{\mathrm{sup}}(u)]=0,\quad [\Delta(x),\pmb r]=0.
\end{align*}
\end{Proposition}

\begin{Lemma}
$\mathfrak{g}_{Q,\sW}$ is a Lie superalgebra with the Lie brackets being the super commutators in $\mathsf Y_0(Q,\sW)$.
\end{Lemma}

\begin{proof}
The same argument as \cite[Prop.~5.2.1]{MO} shows that
\begin{align}\label{commutator}
    [\mathsf E(m\:u^{i}),\mathsf E(m'\:u^{j})]=\mathsf E\left((\tr\otimes\id)\:[\pmb r,m\otimes m']\: u^{i+j}\right)+\text{lower degree in }u.
\end{align}
The lemma follows by plugging $i=j=0$ into the above equation.
\end{proof}

The Lie superalgebra 
 $\mathfrak{g}_{Q,\sW}$ inherits the $\bZ^{Q_0}$-grading from $\mathsf Y_0(Q,\sW)$:
\begin{align*}
    \mathfrak{g}_{Q,\sW}=\bigoplus_{\alpha\in \bZ^{Q_0}} \mathfrak{g}_{\alpha},\quad [\mathfrak{g}_{\alpha},\mathfrak{g}_{\beta}]\subseteq \mathfrak{g}_{\alpha+\beta}\:.
\end{align*}

\begin{Definition}\label{def of cartan sub}
We call $\mathfrak{g}_0$ the \emph{Cartan subalgebra}. A nonzero vector $\eta$ with $\mathfrak{g}_\eta\neq 0$ is called a \emph{root} of $\mathfrak{g}_{Q,\sW}$ and $\mathfrak{g}_\eta$ is called the \emph{root space}. A root $\eta$ is a \emph{positive root} if $\eta\in \bZ_{\geqslant 0}^{Q_0}$; $\eta$ is a \emph{negative root} if $-\eta$ is a positive root. 
\end{Definition}

\begin{Proposition}
\label{prop g(Q,w)}
We have the following facts for $\mathfrak{g}_{Q,\sW}$.
\begin{enumerate}
\item The Cartan subalgebra decomposes as 
\begin{align*}
    \mathfrak{g}_0=\mathfrak{v}\oplus\mathfrak{d},\quad \mathfrak{v}=\bigoplus_{i\in Q_0}\bC(\mathsf t_0)\cdot\mathsf v_i,\quad \mathfrak{d}=\bigoplus_{i\in Q_0}\bC(\mathsf t_0)\cdot\mathsf d_i,
\end{align*}
where $\mathsf v_i$ is the linear functional whose value on $\cH^{\sW^{\mathrm{fr}}}_{\underline{\bd}}(\bv)$ is $\bv_i$, $\mathsf d_i$ is the linear functional whose value on $\cH^{\sW^{\mathrm{fr}}}_{\underline{\bd}}(\bv)$ is $c_1(\mathsf D_{\In,i})-c_1(\mathsf D_{\Out,i})$.
\item $\mathfrak{g}_0$ is a maximal super commutative subalgebra of $\mathfrak{g}_{Q,\sW}$.
\item All root spaces are finite dimensional.
\item All roots are either positive or negative. 
\item There exists a nondegenerate super symmetric (adjoint action) invariant bilinear form $$(\cdot,\cdot)_{\mathfrak{g}}\colon \mathfrak{g}_{Q,\sW}\times \mathfrak{g}_{Q,\sW}\to \bC(\mathsf t_0)\:. $$ 
With respect to this form, we have $\mathfrak{g}_{-\eta}\cong \mathfrak{g}_\eta^{\vee}$.
\item The invariant tensor of $(\cdot,\cdot)_{\mathfrak{g}}$ is the classical $R$-matrix $\pmb r$, explicitly
\begin{align}\label{explicit r matrix}
    \pmb r=\pmb r_{\diag}+\sum_{\alpha > 0}\sum_s \left(e_{\alpha}^{(s)}\otimes e_{-\alpha}^{(s)}+(-1)^{|\alpha|}e_{-\alpha}^{(s)}\otimes e_{\alpha}^{(s)}\right),
\end{align}
where the sum is taken for all positive roots $\alpha$,
and for each $\alpha$,
$s$ is summed over an index finite set depending on $\alpha$, and $|\alpha|$ is given as \eqref{equ on gra}.
The elements $e_{\alpha}^{(s)}$ form a basis in the root space $\mathfrak{g}_{\alpha}$ such that for $\alpha>0$
\begin{align*}
    (e_{\beta}^{(r)},e_{\alpha}^{(s)})=\delta_{\beta,-\alpha}\:\delta_{r,s}\:.
\end{align*}
\item The restriction of $(\cdot,\cdot)_{\mathfrak{g}}$ to $\mathfrak{g}_0$ yields a nondegenerate symmetric bilinear form on $\mathfrak{g}_0=\mathfrak{v}\oplus\mathfrak{d}$, which is explicitly given by the matrix 
\begin{align}\label{bilinear form}
\begin{pmatrix}
0 & \id\\
\id & \pmb Q
\end{pmatrix}\quad\text{for}\quad \pmb Q_{ij}=\sum_{a:i\to j}t_a+\sum_{b:j\to i}t_b\in \bC(\mathsf t_0)\:.
\end{align}
\item Let $h_\eta\in \mathfrak{g}_0$ be the coroot for a root $\eta$ given by
\begin{align*}
    h_{\eta}=\sum_{i}\eta_i\:\mathsf d_i-\sum_{i,j}\pmb Q_{ij}\:\eta_i\:\mathsf v_j\:.
\end{align*}
Then for $\alpha>0$
\begin{align*}
    [e_{\alpha}^{(s)},e_{-\alpha}^{(t)}]=(-1)^{|\alpha|}\:\delta_{s,t}\: h_{\alpha}\:.
\end{align*}
\end{enumerate}
\end{Proposition}

\begin{proof}
(1) follows from the formula \eqref{eq classical r} for the diagonal part of classical $R$-matrix. Note that the vacuum component of $\pmb r_{\diag}$ is $-t_i\cdot 1\otimes \bv_i$, and $t_i\neq 0$ for an admissible auxiliary datum $\mathfrak{c}$, so $\mathsf v_i\in \mathfrak{g}_0$. It also follows from the admissibility that $\cH_{\mathfrak{c}}(\delta_i)\neq 0$ (Lemma \ref{lem 1st excitation}), so we have $\mathsf d_i\in \mathfrak{g}_0$. By \eqref{eq classical r}, $\mathfrak{g}_0$ is spanned by $\{\mathsf v_i\}_{i\in Q_0}$ and $\{\mathsf d_i\}_{i\in Q_0}$, and the linear independence is obvious. 

$\mathfrak{g}_\eta$ is the eigenspace of adjoint action of $\mathfrak{g}_0$ with weight $\eta$. This implies (2).

(3) and (4) follow from the same argument as \cite[\S 5.3.7]{MO}. 

(5) and (6) follow from the same argument as \cite[\S 5.3.8]{MO}, and we shall not repeat here.

(7) follows by a direct computation using (6) and \eqref{eq classical r}. 

(8) follows from the invariance of the bilinear form $(\cdot,\cdot)_{\mathfrak{g}}$.
\end{proof}

\begin{Remark}
The coroot $h_\eta$ for a root $\eta$ can be written as
\begin{align*}
h_\eta=\sum_{i}\eta_i h_{i,0}
\end{align*}
where $h_{i,0}$ is the image of the lowest mode of $h_i(z)=\sum_{r\geqslant 0}h_{i,r}z^{-r-1}$ under the map in Theorem \ref{thm drinfeld yangian map to rtt yangian}.
\end{Remark}

In \cite{DM} (see also \cite[pp.\,18]{Dav3}), a Lie superalgebra $\mathfrak{g}^{\mathrm{BPS},\mathsf T_0}_{Q,\sW}$ over the base ring $\bC[\mathsf t_0]$ is defined as the lowest component in the perverse filtration of CoHA $\cH_{Q,\sW}$, whose Lie brackets are supercommutators induced by a sign-twisted CoHA multiplication. 

In the case of a tripled quiver $\widetilde{Q}$ with canonical cubic potential $\widetilde{\sW}$ (Example \ref{ex doubled vs tripled}), it was conjectured by the second author and proven in \cite{BD} that $$\mathfrak{g}^{\mathrm{BPS},\mathsf T_0}_{\widetilde{Q},\widetilde \sW}\cong \mathfrak{g}_Q^{\mathrm{MO},+},$$ where $\mathfrak{g}_Q^{\mathrm{MO},+}$ is the positive part of Maulik-Okounkov Lie algebra $\mathfrak{g}_Q^{\mathrm{MO}}$ \cite[\S 5.3]{MO}. Analogously, we define
\begin{align*}
    \mathfrak{g}^+_{Q,\sW}=\bigoplus_{\eta>0}\mathfrak{g}_{\eta},
\end{align*}
and make the following conjecture. 

\begin{Conjecture}
\label{conj: bps lie}
There is a $\bC(\mathsf t_0)$-Lie algebra isomorphism $$\mathfrak{g}^+_{Q,\sW}\cong \mathfrak{g}^{\mathrm{BPS},\mathsf T_0}_{Q,\sW}\otimes_{\bC[\mathsf t_0]}\bC(\mathsf t_0)$$ that intertwines their actions on any state space $\cH^{\sW^{\mathrm{fr}}}_{\underline{\bd},\mathsf A^{\mathrm{fr}}}$.
\end{Conjecture}

In the case of a tripled quiver $\widetilde{Q}$ with canonical cubic potential $\widetilde{\sW}$, we will show in Theorem \ref{thm compare with MO yangian} that the above conjecture follows from the main result of \cite{BD}.

\subsection{Filtration and PBW basis of \texorpdfstring{$\mathsf Y_0(Q,\sW)$}{Y_0(Q,W)}}\label{subsec filtration and PBW}
$\mathsf Y_0(Q,\sW)$ has a \textit{filtration} by the degree in $u$: for $k\in \bZ_{\geqslant 0}$, $F_k\mathsf Y_0(Q,\sW)$ is spanned by monomials $\mathsf E(m_1u^{n_1})\cdots \mathsf E(m_ru^{n_r})$ with $m_i$ constant in $u$ and $\sum n_i\leqslant k$. The following is a generalization of \cite[Thm.\,5.5.1]{MO}.

\begin{Theorem}\label{thm grY}
$\mathsf Y_0(Q,\sW)$ is generated by $\mathfrak{g}_{Q,\sW}$ and multiplication by the following tautological classes 
\begin{align}\label{taut class}
    \left\{\mathrm{ch}_k(\mathsf V_i),\:\mathrm{ch}_{k+1}(\mathsf D_{\In,i}-\mathsf D_{\Out,i})\right\}_{i\in Q_0}^{k\in \bZ_{\geqslant 1}},
\end{align}
where Chern characters are $(\sT_0\times \sA^{\text{fr}})$-equivariant. 
Moreover, 
\begin{align*}
    \mathrm{gr}\: \mathsf Y_0(Q,\sW)\cong \mathcal U(\mathfrak{g}_{Q,\sW}[u])
\end{align*}
with respect to the filtration by degree in $u$.
\end{Theorem}

\begin{proof}
The idea is the same as \cite[Thm.\,5.5.1]{MO}. The surjectivity of the natural map $\mathcal U(\mathfrak{g}_{Q,\sW}[u])\to \mathsf Y_0(Q,\sW)$ is deduced from the following generalization of 
\cite[Prop.\,5.5.2]{MO}.
\begin{itemize}
    \item[] Claim: if $\mathsf E(m)=0$, then $\mathsf E(m\:u^k)\in F_{k-1}\mathsf Y_0(Q,\sW)$.
\end{itemize}
The proof of the above claim is similar to \textit{loc.\,cit.}. In fact, we can assume that $m$ is an eigenvector of $\mathfrak{g}_0$ of weight $\lambda$. If $\lambda\neq 0$ then for $h=\mathsf E(\tilde h)\in \mathfrak{g}_0$, we have
\begin{align}\label{commutator 2}
    \lambda(h)\mathsf E(m\:u^k)=[\mathsf E(\tilde h),\mathsf E(m\:u^k)]=[\mathsf E(\tilde h\:u^k),\mathsf E(m)]+\text{lower degree in }u
\end{align}
by \eqref{commutator}. If $\lambda=0$, then $\mathsf E(m\:u^k)$ is a linear combination of diagonal matrix elements of the $R$-matrix. Using the Gauss decomposition $R^{\mathrm{sup}}(u)=F(u)G(u)E(u)$ and the fact that $E(u)$ and $F(u)$ are of the form
\begin{align*}
    \id+\text{ strict triangular of order }O(u^{-1}),
\end{align*}
we have
\begin{multline*}
\frac{1}{k!}\mathsf E(m\:u^k)=\sum_{i\in Q_0}(m,\mathsf d_i)\:\mathrm{ch}_k(\mathsf V_i)+\sum_{i\in Q_0}(m,\mathsf v_i)\:\mathrm{ch}_{k+1}(\mathsf D_{\In,i}-\mathsf D_{\Out,i})-\sum_{i,j\in Q_0}\pmb Q_{ij}\:(m,\mathsf v_i)\:\mathrm{ch}_k(\mathsf V_j)
+\text{lower degree in }u,
\end{multline*}
where $\pmb Q_{ij}$ is the matrix in Proposition \ref{prop g(Q,w)} (7), and the pairing with $\mathsf v_i,\mathsf d_i\in \mathfrak{g}_0$ is the trace pairing. $\mathsf E(m)=0$ implies that $(m,\mathsf v_i)=0$ and $(m,\mathsf d_i)=0$; therefore $\mathsf E(m\:u^k)\in F_{k-1}\mathsf Y_0(Q,\sW)$. This proves the claim and the surjectivity. 

The injectivity is proven the same way as \cite[\S 5.5.3]{MO}, and we shall not repeat. Finally, \eqref{commutator 2} implies that $\mathsf Y_0(Q,\sW)$ is generated by $\mathfrak{g}_{Q,\sW}$ and multiplication by tautological classes \eqref{taut class}.
\end{proof}

\subsection{Shift homomorphisms}\label{sec shift map}

Fix $\mu'\leqslant \mu\in \bZ_{\leqslant 0}^{Q_0}$. Pick $\underline{\bd}$, $\underline{\bd}'$ such that $$\bd_{\In}=\bd'_{\In},\,\,\,\bd_{\Out}-\bd_{\In}=\mu,\,\,\, \bd'_{\Out}-\bd'_{\In}=\mu'.$$ 
Then $\cM(\underline{\bd})$ is a vector bundle on $\cM(\underline{\bd}')$ by adding $\bd_{\Out}-\bd'_{\Out}$ many copies of out-going framings to $\cM(\underline{\bd}')$, that is $\bigoplus_{i\in Q_0}\Hom(\mathsf V_i,\bC^{\bd_{\Out,i}-\bd'_{\Out,i}})$. Let $\pi\colon \cM(\underline{\bd})\to \cM(\underline{\bd}')$ be the bundle map. Take a framed potential $\sw^{\mathrm{fr}}$ on $\cM(\underline{\bd}')$ and its pullback to $\cM(\underline{\bd})$ is still denoted by $\sw^{\mathrm{fr}}$. 
Then we have the pullback isomorphism: 
\begin{align*}
    \pi^*\colon \cH^{\sW^{\mathrm{fr}}}_{\underline{\bd}',\sA^{\mathrm{fr}}}\xrightarrow{\cong} \cH^{\sW^{\mathrm{fr}}}_{\underline{\bd},\sA^{\mathrm{fr}}}.
\end{align*}
%is an isomorphism. 
Consider the following auxiliary $\bC^*_z$ action: it acts trivially on $\cM(\underline{\bd}')$ and scales the fibers of the vector bundle $\cM(\underline{\bd})$ with weight $1$. Then $$H^{\sT_0\times \sA^{\mathrm{fr}}\times \bC^*_z}(\cM(\underline{\bd}'),\sw^{\mathrm{fr}})\cong \cH^{\sW^{\mathrm{fr}}}_{\underline{\bd}',\sA^{\mathrm{fr}}}\otimes\bC[z]=: \cH^{\sW^{\mathrm{fr}}}_{\underline{\bd}',\sA^{\mathrm{fr}}}[z], $$  
$$H^{\sT_0\times \sA^{\mathrm{fr}}\times \bC^*_z}(\cM(\underline{\bd}),\sw^{\mathrm{fr}})\cong \cH^{\sW^{\mathrm{fr}}}_{\underline{\bd},\sA^{\mathrm{fr}}}\otimes\bC[z]=: \cH^{\sW^{\mathrm{fr}}}_{\underline{\bd},\sA^{\mathrm{fr}}}[z],$$ and $\pi^*\colon \cH^{\sW^{\mathrm{fr}}}_{\underline{\bd}',\sA^{\mathrm{fr}}}[z]\to \cH^{\sW^{\mathrm{fr}}}_{\underline{\bd},\sA^{\mathrm{fr}}}[z]$ is an isomorphism whose inverse  is $i^*$, where $i\colon \cM(\underline{\bd}')\hookrightarrow \cM(\underline{\bd})$ is the zero section.

Let $m\in \End^{\bZ/2}(\cH_{\mathfrak{c},\loc})[u]$, and $\mathsf E(m)_z$ and $\mathsf E(m)'_z$ be the corresponding operators on $\cH^{\sW^{\mathrm{fr}}}_{\underline{\bd},\sA^{\mathrm{fr}},\loc}[z]$ and $\cH^{\sW^{\mathrm{fr}}}_{\underline{\bd}',\sA^{\mathrm{fr}},\loc}[z]$ respectively. We note that $R$-matrix on $\cH_{\mathfrak{c},\loc}\otimes\cH^{\sW^{\mathrm{fr}}}_{\underline{\bd}',\sA^{\mathrm{fr}},\loc}[z]$ is constant in $z$ because $\bC^*_z$ acts trivially; therefore $$\mathsf E(m)'_z=\mathsf E(m)'_{z=0}\:.$$
% In general, $i^*\mathsf E(m)_z\pi^*$ is not equal to $\mathsf E(m)'_z$, instead, we have the following.
\begin{Lemma}\label{lem shift map}
If $\mathcal C$ is admissible, then $i^*\mathsf E(m)_z\pi^*\in \prod \End^{\bZ/2}_{\bC(\mathsf t_0)[\mathsf a^{\mathrm{fr}}]}(\cH^{\sW^{\mathrm{fr}}}_{\underline{\bd}',\sA^{\mathrm{fr}},\loc})[z]$ is in the subalgebra $\mathsf Y_{\mu'}(Q,\sW,\mathcal C)[z]$. 
% Moreover, there exists a polynomial $f$ over the base field $\bC(\mathsf t_0)$ in noncommutative variables $\left\{A_z\in \mathsf Y_{\mu'}(Q,\sW,\mathcal C)[z]\right\}$ such that $i^*\mathsf E(m)_z\pi^*=f(\{A_z\})$.
\end{Lemma}

\begin{proof}
It is enough to prove the lemma for $i^*\mathsf E^\pm(m)_z\pi^*$ and $i^*\mathsf E^0(m)_z\pi^*$. Take $\mathfrak{c}=(\ba,\sT_0\curvearrowright\cM(\ba),\sw^{\mathrm{fr}})\in \mathcal C$. Let $$R^{\mathrm{sup}}(u)_z=\Sigma\cdot\Stab_{u<0,\epsilon}^{-1}\Stab_{u>0,\epsilon},\,\,\,R^{\mathrm{sup}}(u)'_z=\Sigma\cdot(\Stab'_{u<0,\epsilon})^{-1}\Stab'_{u>0,\epsilon}$$ be the $R$-matrices acting on $\cH_{\mathfrak{c},\loc}\otimes \cH^{\sW^{\mathrm{fr}}}_{\underline{\bd},\sA^{\mathrm{fr}},\loc}[z]$ and $\cH_{\mathfrak{c},\loc}\otimes \cH^{\sW^{\mathrm{fr}}}_{\underline{\bd}',\sA^{\mathrm{fr}},\loc}[z]$ respectively. Note that the vector bundle $\mathcal V=\bigoplus_{i\in Q_0}\Hom(\mathsf V_i,\bC^{\bd_{\Out,i}-\bd'_{\Out,i}})$ is repelling in the chamber $u>0$. By \cite[Props.~8.2,~8.3]{COZZ}, we have
\begin{align*}
i^*\Stab_{u<0,\epsilon}\pi^*=\Stab'_{u<0,\epsilon}\cdot \:(-1)^{\sharp},\quad i^*\Stab_{u>0,\epsilon}\pi^*=e^{\sT_0\times\bC^*_u\times \bC^*_z}(\mathcal V)\cdot\Stab'_{u>0,\epsilon}\cdot \:e^{\sT_0\times \bC^*_z}(\mathcal V^{\mathrm{fix}})^{-1},
\end{align*}
where $(-1)^{\sharp}$ on the component $\cM(\bv_1,\ba)\times \cM(\bv_2,\underline{\bd})$ is $(-1)^{\bv_1\cdot(\mu-\mu')}$, and $\mathcal V^{\mathrm{fix}}$ is the $\bC^*_u$-fixed part of the restriction of $\cV$ to $\bC^*_u$-fixed components.

Write the Gauss decompositions $R^{\mathrm{sup}}(u)_z=F(u)_zG(u)_zE(u)_z$ and $R^{\mathrm{sup}}(u)'_z=F(u)'_zG(u)'_zE(u)'_z$, as in \eqref{gauss decomp} with extra $\bC^*_z$ equivariance. Then we have
\begin{align*}
    i^*F(u)_z\pi^*=(-1)^{\sharp}\cdot F(u)'_z\cdot(-1)^{\sharp},\quad i^*G(u)_z\pi^*=G(u)'_z\cdot D(u)_z,\quad i^*E(u)_z\pi^*=H_z\cdot E(u)'_z\cdot H^{-1}_z,
\end{align*}
where $D(u)$ and $H$ are the following diagonal operators:
\begin{align*}
    D(u)_z\bigg|_{\cM(\bv_1,\ba)\times\cM(\bv_2,\underline{\bd})}&=(u-z)^{\bv_1\cdot(\mu-\mu')}c_{-1/(u-z)}\left(\bigoplus_{i\in Q_0}\Hom(\mathsf V_{1,i},\bC^{\bd_{\Out,i}-\bd'_{\Out,i}})\right),\\
    H_z\bigg|_{\cM(\bv_1,\ba)\times\cM(\bv_2,\underline{\bd})}&=z^{\bv_2\cdot(\mu-\mu')}c_{1/z}\left(\bigoplus_{i\in Q_0}\Hom(\mathsf V_{2,i},\bC^{\bd_{\Out,i}-\bd'_{\Out,i}})\right),
\end{align*}
where $c_{(-)}(\cdots)$ denotes the $(\sT_0\times \sA^{\text{fr}})$-equivariant Chern polynomial with variable $(-)$. 

Obviously matrix elements of $i^*F(u)_z\pi^*$ are contained in $\mathsf Y_{\mu'}(Q,\sW,\mathcal C)[z]$.
$D(u)_z$ only acts on $\cH_{\mathfrak{c}}$, so the matrix elements of $i^*G(u)_z\pi^*$ are contained in $\mathsf Y_{\mu'}(Q,\sW,\mathcal C)[z]$. $H_z$ is a polynomial in $z$ with coefficients given by characteristic classes of tautological bundles $\{\mathsf V_j\}_{j\in Q_0}$, so $H_z\in\mathsf Y^0_{\mu'}(Q,\sW,\mathcal C)[z]$ by admissibility of $\mathcal C$ and Lemma \ref{lem Y^0 generators}. Then it follows that the matrix elements of $H_z\cdot E(u)'_z\cdot H^{-1}_z$ are contained in $\mathsf Y_{\mu'}(Q,\sW,\mathcal C)(\!(z^{-1})\!)$. Since $$H_z\cdot E(u)'_z\cdot H^{-1}_z=i^*E(u)_z\pi^*,$$ and the latter is a polynomial in $z$, the matrix elements of $i^*E(u)_z\pi^*$ must be in $\mathsf Y_{\mu'}(Q,\sW,\mathcal C)[z]$. 
% This proves the first statement. 
% In the above argument, writing matrix elements of $i^*G(u)_z\pi^*$ in terms of elements of $\mathsf Y_{\mu'}(Q,\sW,\mathcal C)[z]$ only involves operations on $\cH_{\mathfrak{c}}$, and writing matrix elements of $i^*E(u)_z\pi^*$ in terms of elements of $\mathsf Y_{\mu'}(Q,\sW,\mathcal C)[z]$ only involves conjugation by modes of $\{g_j(u)\}_{j\in Q_0}$ where $g_j(u)$ is given in Corollary in \ref{thm e f h as R matrix elements}. Both of the operations are given by polynomials in $\mathsf Y_{\mu'}(Q,\sW,\mathcal C)[z]$. This proves the second statement.
\end{proof}

\begin{Definition}
Fix $\mu'\leqslant \mu\in \bZ_{\leqslant 0}^{Q_0}$.
Assume that $\mathcal C$ is admissible. Define \textit{parametrized shift homomorphism} $$\sS_{\mu,\mu';z}\colon \mathsf Y_{\mu}(Q,\sW,\mathcal C)\to \mathsf Y_{\mu'}(Q,\sW,\mathcal C)[z]$$ by the assignment $\mathsf E(m)\mapsto \left(i^*\mathsf E(m)_{z}\pi^*\right)$ in Lemma \ref{lem shift map}, and define \textit{shift homomorphism} $$\sS_{\mu,\mu'}:=\sS_{\mu,\mu';z}\big|_{z=0}\colon \mathsf Y_{\mu}(Q,\sW,\mathcal C)\to \mathsf Y_{\mu'}(Q,\sW,\mathcal C).$$
Define $\mathsf Y_{\mu,\mu'}(Q,\sW,\mathcal C)$ to be the image of $\mathsf Y_{\mu}(Q,\sW,\mathcal C)$ under the projection
\begin{align*}
    \prod\End^{\bZ/2}_{\bC(\mathsf t_0)[\mathsf a^{\mathrm{fr}}]}(\cH^{\sW^{\mathrm{fr}}}_{\underline{\bd},\mathsf A^{\mathrm{fr}},\mathrm{loc}})\twoheadrightarrow{\prod}'\End^{\bZ/2}_{\bC(\mathsf t_0)[\mathsf a^{\mathrm{fr}}]}(\cH^{\sW^{\mathrm{fr}}}_{\underline{\bd},\mathsf A^{\mathrm{fr}},\mathrm{loc}}),
\end{align*}
where the first product is taken for all state spaces with $\bd_{\Out}-\bd_{\In}=\mu$, and the second product is taken for those state spaces coming from adding out-going arrows to those with $\bd'_{\Out}-\bd'_{\In}=\mu'$.
\end{Definition}

\begin{Lemma}
$\mathsf Y_{\mu,\mu'}(Q,\sW,\mathcal C)$ is isomorphic to the image of $\sS_{\mu,\mu'}\colon \mathsf Y_{\mu}(Q,\sW,\mathcal C)\to \mathsf Y_{\mu'}(Q,\sW,\mathcal C)$.
\end{Lemma}

\begin{proof}
Since $M\mapsto i^*M\pi^*$ induces an isomorphism $ \End^{\bZ/2}_{\bC(\mathsf t_0)[\mathsf a^{\mathrm{fr}}]}(\cH^{\sW^{\mathrm{fr}}}_{\underline{\bd},\mathsf A^{\mathrm{fr}},\mathrm{loc}})\cong \End^{\bZ/2}_{\bC(\mathsf t_0)[\mathsf a^{\mathrm{fr}}]}(\cH^{\sW^{\mathrm{fr}}}_{\underline{\bd}',\mathsf A^{\mathrm{fr}},\mathrm{loc}})$, we can rewrite the map $\sS_{\mu,\mu'}$ as the composition (restricted to $\mathsf Y_{\mu}(Q,\sW,\mathcal C)$) 
\begin{equation}\label{equ on endsurj}
    \prod\End^{\bZ/2}_{\bC(\mathsf t_0)[\mathsf a^{\mathrm{fr}}]}(\cH^{\sW^{\mathrm{fr}}}_{\underline{\bd},\mathsf A^{\mathrm{fr}},\mathrm{loc}})\twoheadrightarrow{\prod}'\End^{\bZ/2}_{\bC(\mathsf t_0)[\mathsf a^{\mathrm{fr}}]}(\cH^{\sW^{\mathrm{fr}}}_{\underline{\bd},\mathsf A^{\mathrm{fr}},\mathrm{loc}})\cong \prod \End^{\bZ/2}_{\bC(\mathsf t_0)[\mathsf a^{\mathrm{fr}}]}(\cH^{\sW^{\mathrm{fr}}}_{\underline{\bd}',\mathsf A^{\mathrm{fr}},\mathrm{loc}})\:.
\end{equation}
Then it follows that $\mathsf Y_{\mu,\mu'}(Q,\sW,\mathcal C)$ is isomorphic to the image of $\sS_{\mu,\mu'}$.
\end{proof}

\begin{Definition}\label{def of wrong-way shift}
Define the \textit{wrong-way shift homomorphism }
$$\sS^{\mu,\mu'}_z\colon \mathsf Y_{\mu'}(Q,\sW,\mathcal C)\to \mathsf Y_{\mu,\mu'}(Q,\sW,\mathcal C)(\!(z^{-1})\!)$$ 
by the assignment $\mathsf E(m)'\mapsto \pi^*\mathsf E(m)'i^*$, where we use the identification in \eqref{equ on endsurj}.
\end{Definition}

\begin{Remark}\label{rmk form of shift}
The explicit formula for $\sS_{\mu,\mu';z}$ and $\sS^{\mu,\mu'}_z$ on the triangular generators can be read from the proof of Lemma \ref{lem shift map} 
(below all Chern characters/polynomials are  $(\sT_0\times \sA^{\text{fr}})$-equivariant):
\begin{align*}
    \sS_{\mu,\mu';z}(\mathsf E^-(m))&=(-1)^{\deg(m)\cdot (\mu-\mu')}\mathsf E^-(m)', & \sS^{\mu,\mu'}_z(\mathsf E^-(m)')&=(-1)^{\deg(m)\cdot (\mu-\mu')}\mathsf E^-(m)_z,\\
    \sS_{\mu,\mu';z}(\mathsf E^+(m))&=z^{\deg(m)\cdot(\mu-\mu')}\widetilde{H}_z\:\mathsf E^+(m)'\:\widetilde{H}_z^{-1}, & \sS^{\mu,\mu'}_z(\mathsf E^+(m)')&=z^{\deg(m)\cdot (\mu'-\mu)}\widetilde{H}_z^{-1}\:\mathsf E^+(m)_z\widetilde{H}_z,\\
    \sS_{\mu,\mu';z}(\mathrm{ch}_k(\mathsf V_i))&=\mathrm{ch}_k(\mathsf V_i), &  \sS^{\mu,\mu'}_z(\mathrm{ch}_k(\mathsf V_i))&=\mathrm{ch}_k(\mathsf V_i)\,,\\
    \sS_{\mu,\mu';z}(\mathrm{ch}_{k}(\mathsf D_{\mathsf{diff},i}&))=\mathrm{ch}_{k}(\mathsf D_{\mathsf{diff},i})-\frac{(\mu_i-\mu'_i)z^k}{k!}, & \sS^{\mu,\mu'}_z(\mathrm{ch}_{k}(\mathsf D_{\mathsf{diff},i}&))=\mathrm{ch}_{k}(\mathsf D_{\mathsf{diff},i})+\frac{(\mu_i-\mu'_i)z^k}{k!}\,,
\end{align*}
where $\mathsf D_{\mathsf{diff},i}:=\mathsf D_{\In,i}-\mathsf D_{\Out,i}$, $\deg(m)$ is the $\bZ^{Q_0}$ grading of $m$, $$\widetilde{H}_z=\prod_{i\in Q_0}c_{-1/z}(\mathsf V_i)^{\mu_i-\mu'_i}\:,$$ and $(\cdots)_{z^{n}}$ means taking $z^{n}$ coefficient in the power series. For example, the operators $g_i(u),h_i(u),e_i(u),f_i(u)$ in the Gauss decomposition \eqref{gauss decomp of R} are mapped under $\sS_{\mu,\mu';z}$ as follows:
\begin{align*}
   g_i(u)\mapsto g_i(u),\quad h_i(u)\mapsto (u-z)^{\mu_i-\mu'_i}h_i(u),\quad e_i(u)\mapsto [(z-u)^{\mu_i-\mu'_i}e_i(u)]_{u^{<0}},\quad f_i(u)\mapsto (-1)^{\mu_i-\mu'_i}f_i(u),
\end{align*}
where $[\cdots]_{u^{<0}}$ means taking the negative power in $u$ part in the power series expansion. On the other hand, $s_z^{\mu,\mu'}$ maps these generators as follows
\begin{align*}
    g_i(u)\mapsto g_i(u),\quad h_i(u)\mapsto (u-z)^{\mu'_i-\mu_i}h_i(u),\quad e_i(u)\mapsto [(z-u)^{\mu'_i-\mu_i}e_i(u)]_{u^{<0}},\quad f_i(u)\mapsto (-1)^{\mu'_i-\mu_i}f_i(u).
\end{align*}
Here we expand $(z-u)^{\mu'_i-\mu_i}e_i(u)$ in the region $1<|u|<|z|$, and expand $(u-z)^{\mu'_i-\mu_i}h_i(u)$ in the region $1<|z|<|u|$.
\end{Remark}

\begin{Lemma}\label{lem shift generation lemma}
Let $\mathcal C$ be admissible, and $\mu'\leqslant \mu\in \bZ_{\leqslant 0}^{Q_0}$. Then $\mathsf Y_{\mu'}(Q,\sW,\mathcal C)$ is generated by $$\bigcup_{a\in \bC}\sS_{\mu,\mu';z}(\mathsf Y_{\mu}(Q,\sW,\mathcal C))\big|_{z=a}$$ as a $\bC(\mathsf t_0)$-algebra.
\end{Lemma}

\begin{proof}
From the explicit formula of $\sS_{\mu,\mu';z}$ in Remark \ref{rmk form of shift}, we see that $\sE^-(m)'$, $\mathrm{ch}_k(\mathsf V_i)$, and $\mathrm{ch}_k(\mathsf D_{\In,i}-\mathsf D_{\Out,i})$ is contained in the image of $\sS_{\mu,\mu';z}\big|_{z=0}$. Moreover, $\sE^+(m)'$ appears as the leading coefficient of $\sS_{\mu,\mu';z}(\sE^+(m))$, so it can be generated from $\bigcup_{a\in \bC}\sS_{\mu,\mu';z}(\sE^+(m))\big|_{z=a}$. Since $\sE^\pm(m)'$, $\mathrm{ch}_k(\mathsf V_i)$, and $\mathrm{ch}_k(\mathsf D_{\In,i}-\mathsf D_{\Out,i})$ generate $\mathsf Y_{\mu'}(Q,\sW,\mathcal C)$ by Proposition \ref{prop gen by deg 0 and Cartan}, the lemma follows.
\end{proof}

\begin{Lemma}\label{lem wrong-way shift inj}
Let $\mathcal C$ be admissible, and $\mu'\leqslant \mu\in \bZ_{\leqslant 0}^{Q_0}$. Then the wrong-way shift homomorphism $\sS^{\mu,\mu'}_z$ is injective.
\end{Lemma}

\begin{proof}
This follows from the fact that
\begin{align*}
    {\prod}'\End^{\bZ/2}_{\bC(\mathsf t_0)[\mathsf a^{\mathrm{fr}}]}(\cH^{\sW^{\mathrm{fr}}}_{\underline{\bd},\mathsf A^{\mathrm{fr}},\mathrm{loc}})(\!(z^{-1})\!)\longrightarrow \prod\End^{\bZ/2}_{\bC(\mathsf t_0)[\mathsf a^{\mathrm{fr}}]}(\cH^{\sW^{\mathrm{fr}}}_{\underline{\bd}',\mathsf A^{\mathrm{fr}},\mathrm{loc}})(\!(z^{-1})\!),\qquad X\mapsto \pi^*Xi^*
\end{align*}
is an isomorphism.
\end{proof}

The following is obvious from definition:
\begin{Lemma}\label{lem closed under shift}
If $\mathcal C\subset\mathcal C'$, then the subalgebra $\mathsf Y_{\mu}(Q,\sW,\mathcal C)\subseteq \mathsf Y_{\mu}(Q,\sW,\mathcal C')$ is closed under parameterized shift and wrong-way shift homomorphisms, that is, the following diagrams commute
\begin{equation*}
\xymatrix{
\mathsf Y_{\mu}(Q,\sW,\mathcal C) \ar[r]^-{\sS_{\mu,\mu';z}} \ar@{^{(}->}[d] & \mathsf Y_{\mu'}(Q,\sW,\mathcal C)[z] \ar@{^{(}->}[d]\\
\mathsf Y_{\mu}(Q,\sW,\mathcal C') \ar[r]^-{\sS_{\mu,\mu';z}} & \mathsf Y_{\mu'}(Q,\sW,\mathcal C')[z]
}
\qquad
\xymatrix{
\mathsf Y_{\mu'}(Q,\sW,\mathcal C) \ar[r]^-{\sS^{\mu,\mu'}_z} \ar@{^{(}->}[d] & \mathsf Y_{\mu,\mu'}(Q,\sW,\mathcal C)(\!(z^{-1})\!) \ar@{^{(}->}[d]\\
\mathsf Y_{\mu'}(Q,\sW,\mathcal C') \ar[r]^-{\sS^{\mu,\mu'}_z} & \mathsf Y_{\mu,\mu'}(Q,\sW,\mathcal C')(\!(z^{-1})\!)
}
\end{equation*}
\end{Lemma}

% We choose the sign twist $x\mapsto (-1)^{\deg(x)\cdot(\mu-\mu')}x$ to cancel some signs in the second equation 

\begin{Proposition}\label{prop shift compatibility}
Shift homomorphisms are associative, that is,
\begin{align*}
    \sS_{\mu',\mu''}\circ \sS_{\mu,\mu'}=\sS_{\mu,\mu''},\quad \mu''\leqslant \mu'\leqslant \mu\in \bZ_{\leqslant 0}^{Q_0}\:.
\end{align*}
Shift homomorphisms are compatible with coproducts in the following sense
\begin{align*}
    \Delta_{\mu'_1,\mu_2}\circ \sS_{\mu_1+\mu_2,\mu'_1+\mu_2}=(\sS_{\mu_1,\mu'_1}\otimes \tau^{\mu_1-\mu'_1}) \circ\Delta_{\mu_1,\mu_2}\:.
\end{align*}
Here $\tau^\nu$ $(\nu\in \bZ^{Q_0})$ is the automorphism  
\begin{equation}\label{equ on taunv}\tau^\nu\colon \mathsf Y_{\mu_2}(Q,\sW,\mathcal C)\to \mathsf Y_{\mu_2}(Q,\sW,\mathcal C), \end{equation} given by $x\mapsto (-1)^{\deg(x)\cdot \nu}x$ on homogeneous element $x$, where $\deg(x)$ is the $\bZ^{Q_0}$ grading on $\mathsf Y_{\mu_2}(Q,\sW,\mathcal C)$.
\end{Proposition}

\begin{proof}
The associativity of shift homomorphisms follows from associativity of pullback isomorphisms. Let $\underline{\bd_1},\underline{\bd'_1},\underline{\bd_2}$ be framing dimension vectors such that the differences between out-going and in-coming dimensions are $\mu_1,\mu'_1,\mu_2$ respectively. Consider the vector bundle projections constructed in the beginning of this subsection: 
$$\pi\colon \cM(\underline{\bd_1})\to \cM(\underline{\bd'_1}), \quad \widetilde{\pi}\colon \cM(\underline{\bd_1}+\underline{\bd_2})\to \cM(\underline{\bd'_1}+\underline{\bd_2}).$$  The vector bundle corresponding to $\widetilde{\pi}$ is attracting in the chamber $a>0$ under the splitting $a\underline{\bd_1}+\underline{\bd_2}$, then according to \cite[Prop.~8.2]{COZZ}, we have
\begin{align*}
    \Stab_{+,\epsilon}\circ \:(\pi^*\otimes\id) = \widetilde{\pi}^*\circ\Stab'_{+,\epsilon}\cdot \:(-1)^{\flat},
\end{align*}
where $(-1)^{\flat}$ on the component $\cM(\bv_1,\underline{\bd'_1})\times \cM(\bv_2,\underline{\bd_2})$ is $(-1)^{\bv_2\cdot (\mu_1-\mu'_1)}$. Then the compatibility with coproduct follows from the above equation.
\end{proof}

\begin{Proposition}\label{prop shift map inj}
The shift homomorphism $\sS_{\mu,\mu'}$ is injective. As a result, $\mathsf Y_{\mu,\mu'}(Q,\sW,\mathcal C)$ is isomorphic to $\mathsf Y_{\mu}(Q,\sW,\mathcal C)$.
\end{Proposition}

The proof strategy is to first show that $\sS_{0,\mu}$ is injective (Lemma \ref{lem shift map inj}), based on the technical Lemma \ref{lem sufficient state spaces}, then to deduce the injectivity for general $\sS_{\mu,\mu'}$ using a formal argument.

\begin{Lemma}\label{lem sufficient state spaces}
Assume that $\mathcal C$ is admissible. Let $\cH_0$ be a state space with zero shift, and let $\left\{\cH_i:=\cH^{\sW_i}_{\underline{\delta_i}}\right\}_{i\in Q_0}$ be a collection of admissible auxiliary spaces (regarded as state spaces). 
% with zero shift labeled by nodes of $Q$ such that $\cH_i(\delta_i)_{\loc}\neq 0$ and the value of $\mathsf d_i$ on $\cH_i$ is nonzero. 
Then the natural map
\begin{align*}
    \mathsf Y_0(Q,\sW,\mathcal C)\longrightarrow \prod_{i_1,\cdots,i_n}\End^{\bZ/2}_{\bC(\mathsf t_0)[a_{0},a_{i_1},\ldots,a_{i_n}]}\left(\cH_0[a_0]\otimes\cH_{i_1}[a_{i_1}]\otimes\cdots\otimes \cH_{i_n}[a_{i_n}]\right)
\end{align*}
is injective. Here $a_0,a_{i_1},\ldots,a_{i_n}$ are equivariant parameters of torus $\mathsf A^{\mathrm{fr}}\cong (\bC^*)^{n+1}$, and $\cH_{i_j}[a_{i_j}]$ stands for $\mathsf A^{\mathrm{fr}}$ acting on $\cH_{i_j}$ with weight $a_{i_j}$.
\end{Lemma}

\begin{proof}
Let $\overline{\mathsf Y}$ be the image of $\mathsf Y_0(Q,\sW,\mathcal C)$ in 
$$\prod\End^{\bZ/2}_{\bC(\mathsf t_0)[a_{0},a_{i_1},\ldots,a_{i_n}]}\left(\cH_0[a_0]\otimes\cH_{i_1}[a_{i_1}]\otimes\cdots\otimes \cH_{i_n}[a_{i_n}]\right).$$ 
Using Theorem \ref{thm grY}, it is enough to show that $\cU(\mathfrak{g}_{Q,\sW}[u])\to \mathrm{gr}\:\overline{\mathsf Y}$ is injective. According to \cite[\S 5.5.3]{MO}, it boils down to showing that the linear map $\mathfrak{g}_{Q,\sW}\to \overline{\mathsf Y}$ is injective. Since $\mathfrak{g}_{Q,\sW}\to \overline{\mathsf Y}$ is $\bZ^{Q_0}$-graded, we need to show injectivity for every homogeneous subspace $\mathfrak{g}_\eta$. Assume that $\mathsf x=\sum a_i\mathsf v_i+\sum b_i\mathsf d_i$ is in the kernel. Then $\mathsf x$ acts on $\cH_i(\mathbf 0)$ trivially, so $b_i=0$ by the fact that the value of $\mathsf d_i$ on $\cH_j$ is nonzero multiple of $\delta_{ij}$. $\mathsf x$ acts on $\cH_i(\delta_i)$ trivially, so $a_i=0$ by the fact that the value of $\mathsf v_i$ on $\cH_j(\delta_j)$ is $\delta_{ij}$. This shows injectivity for $\mathfrak{g}_0$. For $\eta$ with $\eta_i>0$, the same argument as \cite[Prop.~5.3.4]{MO} shows that the $\mathfrak{g}_{Q,\sW}$-action maps
\begin{align*}
    \mathfrak{g}_\eta\to \cH_i(\eta),\quad \mathfrak{g}_{-\eta}\to \cH_i(\eta)^{\vee}
\end{align*}
that take $\xi\in \mathfrak{g}_\eta$ to $\xi\:|\mathbf 0\rangle$ and dually for $\mathfrak{g}_{-\eta}$, are injective. This finishes the proof.
\end{proof}

\begin{Lemma}\label{lem shift map inj}
Assume that $\mathcal C$ is admissible, then the shift homomorphism $\sS_{0,\mu}\colon \mathsf Y_0(Q,\sW,\mathcal C)\to \mathsf Y_\mu(Q,\sW,\mathcal C)$ is injective. As a result, the image $\mathsf Y_{0,\mu}(Q,\sW,\mathcal C)$ of $\sS_{0,\mu}$ is isomorphic to $\mathsf Y_{0}(Q,\sW,\mathcal C)$.
\end{Lemma}

\begin{proof}
By Proposition \ref{prop shift compatibility}, we have $$\Delta_{\mu,0}\circ \sS_{0,\mu}=(\sS_{0,\mu}\otimes\tau^{-\mu})\circ \Delta,$$ 
and it is enough to show that $(\sS_{0,\mu}\otimes\id)\circ \Delta$ is injective since $\tau^{-\mu}$ is an automorphism. The image of $(\sS_{0,\mu}\otimes\id)\circ \Delta$ is contained in the subalgebra $\mathsf Y_{0,\mu}(Q,\sW,\mathcal C)\widehat{\otimes} \mathsf Y_{0}(Q,\sW,\mathcal C)$. Applying Lemma \ref{lem sufficient state spaces} to $\cH_0$ coming from adding $-\mu$ framing arrows to a state space with shift $\mu$, and $\{\cH_i\}_{i\in Q_0}$ 
a collection of admissible auxiliary spaces (regarded as state spaces), we see that the induced map 
$$\mathsf Y_{0}(Q,\sW,\mathcal C)\to \mathsf Y_{0,\mu}(Q,\sW,\mathcal C)\widehat{\otimes} \mathsf Y_{0}(Q,\sW,\mathcal C)$$ is injective. This implies the injectivity of $(\sS_{0,\mu}\otimes\id)\circ \Delta$.
% By Remark \ref{rmk counit}
% For $i\in Q_0$, let $n_i=1-\mu_i$. For the $n_i$ out-going framing of $\cM(n_i\delta_i)$, equip one of them with a nontrivial $\sT_0$-weight and $\sT_0$ acts trivially on the others. In particular, there is a $\sT_0$-equivariant vector bundle $\cM(n_i\delta_i)\to \cM(\underline{\bd}^{(i)})$ forming by adding $-\mu_i$ out-going arrows to the node $i$, where $\bd^{(i)}_{\In}=n_i\delta_i$, $\bd^{(i)}_{\Out}=\delta_i$. Take the potential to be $\sw=\tr(\mathsf m)$ (same as unframed quiver). The collection $\left\{\cH_i:=H^{\sT_0}(\cM(n_i\delta_i),\sw)\right\}_{i\in Q_0}$ satisfies the assumption in Lemma \ref{lem sufficient state spaces}. Moreover, each $\cH_i$ also appears in the 
\end{proof}
\begin{proof}[Proof of Proposition \ref{prop shift map inj}]
By Proposition \ref{prop shift compatibility}, there is commutative diagram 
\begin{equation*}
\xymatrix{
\mathsf Y_\mu(Q,\sW,\mathcal C) \ar[d]_{\sS_{\mu,\mu'}} \ar[rr]^-{\Delta_{0,\mu}} & & \mathsf Y_0(Q,\sW,\mathcal C)\widehat{\otimes} \mathsf Y_\mu(Q,\sW,\mathcal C) \ar[d]^{\sS_{0,\mu'-\mu}\otimes \tau^{\mu-\mu'}} \\
\mathsf Y_{\mu'}(Q,\sW,\mathcal C) \ar[rr]^-{\Delta_{\mu'-\mu,\mu}} & & \mathsf Y_{\mu'-\mu}(Q,\sW,\mathcal C)\widehat{\otimes} \mathsf Y_\mu(Q,\sW,\mathcal C).
}
\end{equation*}
The right vertical arrow is injective since $\tau^{\mu-\mu'}$ \eqref{equ on taunv} is an automorphism and $\sS_{0,\mu'-\mu}$ is injective (Lemma \ref{lem shift map inj}). The top horizontal arrow is injective because $(\varepsilon\otimes\id)\circ \Delta_{0,\mu}=\id$ (Remark \ref{rmk counit}). Thus $(\sS_{0,\mu'-\mu}\otimes \tau^{\mu-\mu'})\circ \Delta_{0,\mu}$ is injective, and it follows from the 
commutativity of the above diagram 
that $\sS_{\mu,\mu'}$ is injective.
\end{proof}

We deduce the following generalization of Lemma \ref{lem sufficient state spaces}, 
which will be used in \S \ref{sect on anti-auto}.
\begin{Proposition}\label{prop sufficient state spaces}
Assume that $\mathcal C$ is admissible. Let $\cH_0$ be a state space with shift $\bd_{\Out}-\bd_{\In}=\mu$, and let $\left\{\cH_i:=\cH^{\sW_i}_{\underline{\delta_i}}\right\}_{i\in Q_0}$ be a collection of admissible auxiliary spaces (regarded as state spaces). Then the natural map
\begin{align*}
    \mathsf Y_\mu(Q,\sW,\mathcal C)\longrightarrow \prod_{i_1,\cdots,i_n}\End^{\bZ/2}_{\bC(\mathsf t_0)[a_{0},a_{i_1},\ldots,a_{i_n}]}\left(\cH_0[a_0]\otimes\cH_{i_1}[a_{i_1}]\otimes\cdots\otimes \cH_{i_n}[a_{i_n}]\right)
\end{align*}
is injective.
\end{Proposition}
\begin{proof}
This follows from Lemmata \ref{lem wrong-way shift inj} and \ref{lem sufficient state spaces}.
\end{proof}

\subsection{Explicit formulas for coproducts}\label{expl formula of coprod}
In this subsection, we write down explicit formulas for some elements under the coproduct $\Delta\colon \mathsf Y_0(Q,\sW)\to \mathsf Y_0(Q,\sW)\widehat{\otimes}\mathsf Y_0(Q,\sW)$.

By the construction of the coproduct in \eqref{stab induces coproduct}, we immediately see that
\begin{align}\label{coproduct e f h lowest}
    \Delta(X_{i,0})=X_{i,0}\otimes 1+1\otimes X_{i,0},\;\text{ for }\; X=e,f,h\,.
\end{align}
Here $X_{i,r}$ for $X=e,f,h$ are images of the corresponding generators of $\mathcal{D}\widetilde{\mathcal{SH}}_0(Q,\sW)$ under the map in Theorem \ref{thm drinfeld yangian map to rtt yangian}.

The coproduct of the Chern characters of the framing vector spaces is also primitive:
\begin{align}
    \Delta(\mathrm{ch}_k(\mathsf D_{\In,i}-\mathsf D_{\Out,i}))=\mathrm{ch}_k(\mathsf D_{\In,i}-\mathsf D_{\Out,i})\otimes 1+1\otimes\mathrm{ch}_k(\mathsf D_{\In,i}-\mathsf D_{\Out,i}).
\end{align}

The coproduct of the Chern characters of the tautological bundles are complicated, and we have the following explicit formula for the coproduct of the first Chern class.

\begin{Lemma}
For $\lambda\in \bC^{Q_0}$, let $c_1(\lambda)=\sum\lambda_i \:c_1(\mathsf V_i)$, then
\begin{align}\label{coproduct for 1st Chern class}
    \Delta c_1(\lambda)=c_1(\lambda)\otimes 1+1\otimes c_1(\lambda)-\sum_{\alpha>0}\alpha(\lambda)\sum_s (-1)^{|\alpha|}\: e_{-\alpha}^{(s)}\otimes e_{\alpha}^{(s)},
\end{align}
where the sum is taken for all positive roots, and $\left\{e_\alpha^{(s)}\right\}$ is the basis of $\mathfrak{g}_{\alpha}$ in Proposition \ref{prop g(Q,w)} (6). 
\end{Lemma}

\begin{proof}
First, the Gauss decomposition \eqref{gauss decomp of R} yields
\begin{align*}
    c_1(\delta_i)=\frac{-1}{t_i}\mathsf E(|\mathbf 0\rangle\langle \mathbf 0|\: u)+t_i\frac{\mathsf v_i(\mathsf v_i+1)}{2}\:,
\end{align*}
where $\delta_i\in \bC^{Q_0}$ is the unit vector for the $i$-th coordinate, $|\mathbf 0\rangle$ is the vector defined in Notation \ref{ground and 1st excitation} for an admissible auxiliary datum $(\delta_i,\sT_0\curvearrowright\cM(\delta_i),\sw_i)$. Then \eqref{coproduct for 1st Chern class} for $\lambda=\delta_i$ follows from the same computation as \cite[\S 10.1.2]{MO}, namely,
\begin{align*}
\Delta\:\mathsf E(|\mathbf 0\rangle\langle \mathbf 0|\: u)&=\mathsf E(|\mathbf 0\rangle\langle \mathbf 0|\: u)\otimes 1+1\otimes \mathsf E(|\mathbf 0\rangle\langle \mathbf 0|\: u)+\sum_{j\in Q_0}\langle \mathbf 0|\mathsf d_j^2|\mathbf 0\rangle\:\mathsf v_j\otimes\mathsf v_j+
\sum_{\alpha,\beta>0}\sum_{s,t}(-1)^{|\beta|}\langle \mathbf 0|e^{(s)}_{-\beta}\:e^{(t)}_{\alpha}|\mathbf 0\rangle\: e^{(t)}_{-\alpha}\otimes e^{(s)}_{\beta}\\
&=\mathsf E(|\mathbf 0\rangle\langle \mathbf 0|\: u)\otimes 1+1\otimes \mathsf E(|\mathbf 0\rangle\langle \mathbf 0|\: u)+t_i^2 \mathsf v_i\otimes\mathsf v_i+t_i\sum_{\alpha>0}\alpha_i\sum_s (-1)^{|\alpha|} e_{-\alpha}^{(s)}\otimes e_{\alpha}^{(s)}\:,
\end{align*}
and 
\begin{align*}
    \Delta\:\frac{\mathsf v_i(\mathsf v_i+1)}{2}=\frac{\mathsf v_i(\mathsf v_i+1)}{2}\otimes 1+1\otimes\frac{\mathsf v_i(\mathsf v_i+1)}{2}+\mathsf v_i\otimes\mathsf v_i\:.
\end{align*}
For a general $\lambda$, \eqref{coproduct for 1st Chern class} follows by linearity.
\end{proof}

If we set
\begin{align*}
    \pmb r_+=\sum_{\alpha>0}\sum_s (-1)^{|\alpha|}\: e_{-\alpha}^{(s)}\otimes e_{\alpha}^{(s)}\,,
\end{align*}
then \eqref{coproduct for 1st Chern class} can be written as
\begin{align*}
\Delta c_1(\lambda)=c_1(\lambda)\otimes 1+1\otimes c_1(\lambda)+\left[\sum\lambda_i\mathsf v_i\otimes 1\,,\,\pmb r_+\right].
\end{align*}
Using the commutation relations $$[c_1(\mathsf V_i),e_{i,r}]=e_{i,r+1}, \quad [c_1(\mathsf V_i),f_{i,r}]=-f_{i,r+1},$$ 
and \eqref{ef rel_sym}, the formulas \eqref{coproduct e f h lowest} and \eqref{coproduct for 1st Chern class} determine $\Delta(X_{i,r})$ for $X=e,f,h$ and all $r>0$. We compute $\Delta(X_{i,1})$ explicitly as follows.
\begin{Proposition}
The coproducts of $X_{i,1}$ for $X=e,f,h$ are given by
\begin{equation}\label{coproduct e f h 1st nontrivial}
\begin{split}
\Delta(e_{i,1})&=e_{i,1}\otimes 1+1\otimes e_{i,1}+h_{i,0}\otimes e_{i,0}-\left[1\otimes e_{i,0}\,,\,\pmb r_+\right],\\
\Delta(f_{i,1})&=f_{i,1}\otimes 1+1\otimes f_{i,1}+f_{i,0}\otimes h_{i,0}+[f_{i,0}\otimes 1\,,\, \pmb r_+],\\
\Delta(h_{i,1})&=h_{i,1}\otimes 1+1\otimes h_{i,1}+h_{i,0}\otimes h_{i,0}+[h_{i,0}\otimes 1\,,\, \pmb r_+]\\
&=h_{i,1}\otimes 1+1\otimes h_{i,1}+h_{i,0}\otimes h_{i,0}+\sum_{\alpha>0}(\delta_i,\alpha)_{\pmb Q}\sum_s (-1)^{|\alpha|}\: e_{-\alpha}^{(s)}\otimes e_{\alpha}^{(s)},
\end{split}
\end{equation}
where $\delta_i\in \bC^{Q_0}$ is the unit vector for the $i$-th coordinate, $(\cdot,\cdot)_{\pmb Q}$ is the bilinear form induced by the matrix $\pmb Q$ in \eqref{bilinear form}, that is, $(a,b)=\sum a_k\,\pmb Q_{kl}\,b_l$.
\end{Proposition}

\begin{proof}
Since $\Delta$ is an algebra homomorphism, we have
\begin{align*}
\Delta(e_{i,1})=[\Delta(c_1(\mathsf V_i)),\Delta(e_{i,0})]=\left[\square(c_1(\mathsf V_i)) +\left[\mathsf v_i\otimes 1\,,\,\pmb r_+\right]\,,\, \square(e_{i,0})\right]\,,
\end{align*}
where $\square(x)=x\otimes 1+1\otimes x$. It is easy to see that $[\square(c_1(\mathsf V_i)),\square(e_{i,0})]=\square(e_{i,1})$ and $[\mathsf v_i\otimes 1,\square(e_{i,0})]=e_{i,0}\otimes 1$. 

We claim that
\begin{align*}
[\pmb r_+,\square(e_{i,0})]=-h_{i,0}\otimes e_{i,0}\,,\quad [\pmb r_+,\square(f_{i,0})]=f_{i,0}\otimes h_{i,0}\,.
\end{align*}
In fact, the classical Yang-Baxter equation \eqref{CYBE} implies that $[\pmb r,\square(x)]=0$ for any $x\in \mathfrak{g}_{Q,\sW}$. Write $\pmb r=\pmb r_{\mathrm{diag}}+\pmb r_++\pmb r_-$, then we have
\begin{align*}
[\pmb r_+,\square(e_{i,0})]\in \mathfrak{g}_{\leqslant0}\otimes \mathfrak{g}_{>0}\,,\quad [\pmb r_{\mathrm{diag}},\square(e_{i,0})]=h_{i,0}\otimes e_{i,0}+e_{i,0}\otimes h_{i,0}\,,\quad [\pmb r_-,\square(e_{i,0})]\in \mathfrak{g}_{>0}\otimes \mathfrak{g}_{\leqslant0}\,,
\end{align*}
where $\mathfrak{g}_{\gtrless 0}=\bigoplus_{\eta \gtrless 0}\mathfrak{g}_\eta$, $\mathfrak{g}_{\geqslant 0}=\mathfrak{g}_{> 0}\oplus \mathfrak{g}_{0}$, and $\mathfrak{g}_{\leqslant 0}=\mathfrak{g}_{< 0}\oplus \mathfrak{g}_{0}$. Comparing degrees, we get $[\pmb r_+,\square(e_{i,0})]=-h_{i,0}\otimes e_{i,0}$. A similar argument shows $[\pmb r_+,\square(f_{i,0})]=f_{i,0}\otimes h_{i,0}$. This proves the claim.

Then
\begin{align*}
\Delta(e_{i,1})&=\square(e_{i,1})-\left[\mathsf v_i\otimes 1\,,\,h_{i,0}\otimes e_{i,0}\right]+\left[e_{i,0}\otimes 1\,,\,\pmb r_+\right]=\square(e_{i,1})+\left[e_{i,0}\otimes 1\,,\,\pmb r_+\right]\\
&=\square(e_{i,1})+h_{i,0}\otimes e_{i,0}-\left[1\otimes e_{i,0}\,,\,\pmb r_+\right],
\end{align*}
and 
\begin{align*}
\Delta(f_{i,1})&=\left[\square(f_{i,0})\,,\, \square(c_1(\mathsf V_i)) +\left[\mathsf v_i\otimes 1\,,\,\pmb r_+\right]\right]=\square(f_{i,1})+\left[f_{i,0}\otimes 1\,,\,\pmb r_+\right]-\left[\mathsf v_i\otimes 1\,,\,f_{i,0}\otimes h_{i,0}\right]\,\\
&=\square(f_{i,1})+\left[f_{i,0}\otimes 1\,,\,\pmb r_+\right]+f_{i,0}\otimes h_{i,0}\,.
\end{align*}
Using the relation $[e_{i,0},f_{i,1}]=(-1)^{(\delta_i|\delta_i)}\gamma_i h_{i,1}$, we get
\begin{align*}
\Delta(h_{i,1})&=(-1)^{(\delta_i|\delta_i)}\gamma_i^{-1}[\Delta(e_{i,0}),\Delta(f_{i,1})]\\
&=\square(h_{i,1})+h_{i,0}\otimes h_{i,0}+[h_{i,0}\otimes 1\,,\, \pmb r_+]+(-1)^{(\delta_i|\delta_i)+|i|}\gamma_i^{-1}\left([f_{i,0}\otimes 1\,,\, [\square(e_{i,0}),\pmb r_+]]+f_{i,0}\otimes [e_{i,0},h_{i,0}]\right).
\end{align*}
The term inside the last bracket vanishes since $[f_{i,0}\otimes 1\,,\, [\square(e_{i,0}),\pmb r_+]]=[f_{i,0}\otimes 1\,,\, h_{i,0}\otimes e_{i,0}]=-\pmb Q_{ii}f_{i,0}\otimes e_{i,0}$ and $f_{i,0}\otimes [e_{i,0},h_{i,0}]=\pmb Q_{ii}f_{i,0}\otimes e_{i,0}$. Finally, direct computation shows that \begin{equation*}[h_{i,0}\otimes 1\,,\, \pmb r_+]=\sum_{\alpha>0}(\delta_i,\alpha)_{\pmb Q}\sum_s (-1)^{|\alpha|}\: e_{-\alpha}^{(s)}\otimes e_{\alpha}^{(s)}. \qedhere  \end{equation*}
\end{proof}

\begin{Remark}\label{rmk on copr compa}
The coproduct formulas \eqref{coproduct e f h lowest} and \eqref{coproduct e f h 1st nontrivial} resemble those coproduct formulas in \cite[Def.\,4.6]{GNW} for a symmetrizable Kac-Moody Lie algebra $\mathfrak{g}_{A,D}$ associated to a generalized Cartan matrix $A$ with symmetrizer $D$ (see Appendix \S \ref{sec Y(sym KM)} for definition). We expect that the composition of algebra homomorphisms 
\begin{equation}\label{equ on bialg map for sym km}\varrho\colon Y_{0}(\mathfrak{g}_{A,D})\otimes_{\bC[\hbar]}\bC(\hbar)\xrightarrow{\text{Ex.\ref{ex on gen symi}}} \mathcal D\widetilde{\mathcal{SH}}_0(\widetilde{Q},\sW)\otimes_{\bC[\hbar]}\bC(\hbar)\xrightarrow{\text{Thm.\ref{thm drinfeld yangian map to rtt yangian}}} \mathsf Y_0(\widetilde{Q},\sW) \end{equation} is a bialgebra homomorphism, where $(\widetilde{Q},\sW)$ is the quiver with potential in Appendix \S \ref{sec Y(sym KM)}. 

We note that the induced map $\bar\varrho\colon \mathfrak{g}_{A,D}\to \mathfrak{g}_{\widetilde{Q},\sW}$ is always injective\footnote{This can be proven as follows. The Cartan subalgebra of $\mathfrak{g}_{A,D}$ is spanned by $\{\psi_{i,0}\}_{i\in \widetilde{Q}_0}$ in Example \ref{ex on gen symi}, and $\bar\varrho(\psi_{i,0})=h_{i,0}$, so $\bar\varrho$ is injective on the Cartan subalgebra. Since $\bar\varrho$ is $\bZ^{\widetilde{Q}_0}$-graded, so is $\ker(\bar\varrho)$. Suppose $x\in\ker(\bar\varrho)$ is a nontrivial element in the $\alpha$ root space, then by \cite[Thm.~2.2]{Kac} there exists $y$ in $-\alpha$ root space of $\mathfrak{g}_{A,D}$ such that $[x,y]\neq 0$. Since $[x,y]$ is in the Cartan subalgebra and $\bar\varrho([x,y])=0$, $[x,y]=0$, a contradiction.}. By comparing \eqref{coproduct e f h lowest} \eqref{coproduct e f h 1st nontrivial} with \cite[(4.7)]{GNW}, $\varrho$ is a bialgebra homomorphism if and only if any root space $\mathfrak{g}_\eta\subset \mathfrak{g}_{\widetilde{Q},\sW}$ with $(\delta_i,\eta)_{\pmb Q}\neq 0$ for some $i\in \widetilde{Q}_0$ is contained in the image of $\bar{\varrho}$. This is the case for the tripled affine type $A^{(1)}_{n-1}$ ($n>2$) quiver with canonical cubic potential \cite{Jin}, so the natural map
\begin{equation}\label{equ on bialg map for affine type A}\varrho\colon Y_{0}(\widehat{\mathfrak{sl}}_n)\otimes_{\bC[\hbar,\varepsilon]}\bC(\hbar,\varepsilon)\xrightarrow{\text{Ex.\ref{ex on vv ex}}} \mathcal D\widetilde{\mathcal{SH}}_0(Q,\sW)\otimes_{\bC[\hbar,\varepsilon]}\bC(\hbar,\varepsilon)\xrightarrow{\text{Thm.\ref{thm drinfeld yangian map to rtt yangian}}} \mathsf Y_0(Q,\sW) \end{equation}
is a bialgebra homomorphism, where $(Q,\sW)$ is the quiver with potential associated to the parity sequence $(s_i=1:1\leqslant i\leqslant n)$ in Appendix \S \ref{sec super affine Yangian}.

Besides the affine type $A^{(1)}_{n-1}$ ($n>2$) cases, we also have the compatibility of coproducts in the finite type cases.

% In fact, assuming that $\mathfrak{g}_{\widetilde{Q},\sW}^{+}$ is generated by $\{e_{i,0}\}_{i\in \widetilde{Q}_0}$ (this is equivalent to assuming that the induced map 
% $$\bar\varrho\colon \mathfrak{g}_{A,D}\to \mathfrak{g}_{\widetilde{Q},\sW}$$ maps $\mathfrak{g}_{A,D}^{\pm}$ surjectively onto $\mathfrak{g}_{\widetilde{Q},\sW}^{\pm}$, because $\mathfrak{g}_{A,D}^{\pm}$ is generated by $\{x^{\pm}_{i,0}\}_{i\in \widetilde{Q}_0}$ respectively), 
% then $\bar\varrho$ induces isomorphisms $$\mathfrak{g}_{\widetilde{Q},\sW}^{\pm}\cong \mathfrak{g}_{A,D}^{\pm}\,,$$ 
% and \eqref{coproduct e f h lowest} \eqref{coproduct e f h 1st nontrivial} agree with \cite[(4.7)]{GNW}, it follows that $\varrho\colon Y_{0}(\mathfrak{g}_{A,D})\to \mathsf Y_0(\widetilde{Q},\sW)$ preserves the coproduct.
\end{Remark}

\begin{Theorem}\label{thm on copr compa}
When $A$ is of finite type and indecomposable,~i.e.~$\mathfrak{g}_{A,D}$ is a finite dimensional simple Lie algebra, then the map 
\eqref{equ on bialg map for sym km}:
$$\varrho\colon Y_{0}(\mathfrak{g}_{A,D})_{\loc}\to \mathsf Y_0(\widetilde{Q},\sW)$$
is an injective bialgebra homomorphism, where $Y_{0}(\mathfrak{g}_{A,D})_{\loc}=Y_{0}(\mathfrak{g}_{A,D})\otimes_{\bC[\hbar]}\bC(\hbar)$. Moreover, there is a bialgebra isomorphism
\begin{align*}
\widetilde{\varrho}\colon Y_{0}(\mathfrak{g}_{A,D})_{\loc}\otimes \bC[\mathsf d_{i,k}: i\in \widetilde{Q}_0,\, k\in \bZ_{\geqslant 0}]\cong \mathsf Y_0(\widetilde{Q},\sW).
\end{align*}
Here $\mathsf d_{i,k}$ are primitive for all $i$ and $k$, $\widetilde{\varrho}$ on the first component is $\varrho$, and $\widetilde{\varrho}(\mathsf d_{i,k})$ is the multiplication by $\mathrm{ch}_{k+1}(\mathsf D_{\In,i}-\mathsf D_{\Out,i})$ operator when acting on a state space with framing vector spaces $\mathsf D_{\In}, \mathsf D_{\Out}$.
\end{Theorem}
\begin{proof}
We claim that $\mathfrak{g}_{\widetilde{Q},\sW}$ is finite dimensional. Consider the dimension vector $\delta_i$ and symmetric quiver variety $\cM(\bv,\delta_i)$ with framed potential $\sW_i=\sW+B\Phi_iA$ where $A$ is the in-coming framing, $\Phi_i$ is the edge loop of node $i$, and $B$ is the out-going framing. According to the proof of \cite[Cor.~5.2]{VV3} and \cite[Rmk.~5.3]{VV3}, there are only finitely many $\bv$ such that $\Crit_{\cM(\bv,\delta_i)}(\tr\sW_i)$ is nonempty. In particular, $\cH^{\sW_i}_{\underline{\delta_i},\loc}=\bigoplus_{\bv}H^{\sT_0}(\cM(\bv,\delta_i),\tr\sW_i)_\loc$ is a finite dimensional $\bC(\hbar)$ vector space. The proof of Lemma \ref{lem sufficient state spaces} implies that the $\mathfrak{g}_{\widetilde{Q},\sW}$-action map
\begin{align*}
\mathfrak{g}_{\widetilde{Q},\sW}^+\to \bigoplus_{i\in \widetilde{Q}_0}\cH^{\sW_i}_{\underline{\delta_i},\loc}\,,\quad \xi\in \mathfrak{g}_{\widetilde{Q},\sW}^+\mapsto (\xi\,|\mathbf 0\rangle_i)_{i\in \widetilde{Q}_0}
\end{align*}
is injective. This implies that $\mathfrak{g}_{\widetilde{Q},\sW}^+$ is finite dimensional. So $\dim\mathfrak{g}_{\widetilde{Q},\sW}=2\dim \mathfrak{g}_{\widetilde{Q},\sW}^++2|Q_0|<\infty$.

Since $\mathfrak{g}_{\widetilde{Q},\sW}$ is a finite dimensional $\bC(\hbar)$ Lie algebra with a nondegenerate symmetric bilinear form, it is reductive, together with a triangular decomposition $\mathfrak{g}_{\widetilde{Q},\sW}=\mathfrak{g}_{\widetilde{Q},\sW}^+\oplus \mathfrak{g}_{0}\oplus \mathfrak{g}_{\widetilde{Q},\sW}^-$. The root grading on $\mathfrak{g}_{\widetilde{Q},\sW}$ is the same as the $\widetilde{Q}_0$ grading. It is well-known that the nilradical of the Borel subalgebra of a reductive Lie algebra is generated by the simple root vectors, so $\mathfrak{g}_{\widetilde{Q},\sW}^+$ is generated by $\{e_{i,0}\}_{i\in \widetilde{Q}_0}$. This implies that the induced map $\bar\varrho\colon \mathfrak{g}_{A,D}\to \mathfrak{g}_{\widetilde{Q},\sW}$ maps $\mathfrak{g}_{A,D}^{\pm}$ surjectively onto $\mathfrak{g}_{\widetilde{Q},\sW}^{\pm}$, and by Remark \ref{rmk on copr compa}, $\bar\varrho$ induces isomorphisms $\mathfrak{g}_{A,D}^{\pm}\cong\mathfrak{g}_{\widetilde{Q},\sW}^{\pm}$ and $\varrho$ is a bialgebra homomorphism.

Consider a bialgebra $\mathfrak{D}:=\bC[\mathsf d_{i,k}: i\in \widetilde{Q}_0,\, k\in \bZ_{\geqslant 0}]$ which we require that $\mathsf d_{i,k}$ are primitive for all $i$ and $k$. $\mathfrak{D}$ is a sub-bialgebra of $\mathsf Y_0(\widetilde{Q},\sW)$ via the identification 
$$\mathsf d_{i,k}\mapsto \mathrm{multiplication}\, \,\mathrm{by}\,\, \mathrm{ch}_{k+1}(\mathsf D_{\In,i}-\mathsf D_{\Out,i})\,\,\mathrm{operator}. $$ 
Then the map 
\begin{align*}
\widetilde{\varrho}\colon Y_{0}(\mathfrak{g}_{A,D})_{\loc}\otimes \bC[\mathsf d_{i,k}: i\in \widetilde{Q}_0,\, k\in \bZ_{\geqslant 0}]\to \mathsf Y_0(\widetilde{Q},\sW)
\end{align*}
induced by $\varrho$ and the inclusion $\mathfrak{D}\subset \mathsf Y_0(\widetilde{Q},\sW)$ is a bialgebra homomorphism. The source of $\widetilde{\varrho}$ is filtered by setting $\deg \mathsf d_{i,k}=k$ and $\deg X_{r}=r$ for $X=x^\pm_i,\psi_i$, the target of $\widetilde{\varrho}$ is also filtered as in \S \ref{subsec filtration and PBW}. It is easy to see that $\widetilde{\varrho}$ preserves the filtration, so it induces
\begin{align*}
\mathrm{gr}\,\widetilde{\varrho}\colon \mathrm{gr}\,Y_{0}(\mathfrak{g}_{A,D})_{\loc}\otimes \mathrm{gr}\,\mathfrak{D}\to \mathrm{gr}\,\mathsf Y_0(\widetilde{Q},\sW).
\end{align*}
The source of $\mathrm{gr}\,\widetilde{\varrho}$ is isomorphic to $\mathcal U(\mathfrak{g}_{A,D}[u]\oplus \mathfrak{d}[u])$, where $\mathfrak{d}=\bigoplus_i \bC(\hbar)\cdot\mathsf d_i$ (see Proposition \ref{prop g(Q,w)}) which we identify $\mathsf d_i=\mathsf d_{i,0}$, and the target of $\mathrm{gr}\,\widetilde{\varrho}$ is isomorphic to $\mathcal U(\mathfrak{g}_{\widetilde{Q},\sW}[u])$. Moreover, $\mathrm{gr}\,\widetilde{\varrho}$ is induced from the Lie algebra homomorphism $\mathfrak{g}_{A,D}\oplus \mathfrak{d}\xrightarrow{(\bar{\varrho},\iota)} \mathfrak{g}_{\widetilde{Q},\sW}$, where $\iota$ is the natural embedding of $\mathfrak{d}$ in $\mathfrak{g}_{\widetilde{Q},\sW}$.
Direct computation shows that $\pmb Q=-\hbar \,DA$. In particular, $\pmb Q$ is invertible. It follows that $\mathsf v_i$ is in the linear span of $\mathsf d_i$ and $\{h_{j,0}\}_{j\in \widetilde{Q}_0}$. Thus, $(\bar{\varrho},\iota)$ is surjective and therefore must be an isomorphism since the two sides have the same dimension. So $\mathrm{gr}\,\widetilde{\varrho}$ is an isomorphism. Therefore, $\widetilde{\varrho}$ is an isomorphism.
\end{proof}

% The assumption that $\mathfrak{g}_{\widetilde{Q},\sW}^{+}$ is generated by $\{e_{i,0}\}_{i\in \widetilde{Q}_0}$ is usually not satisfied, e.g. the tripled affine type $A^{(1)}_n$ quiver with canonical cubic potential \cite{Jin}. A weaker assumption that any root space $\mathfrak{g}_\eta\subset \mathfrak{g}_{\widetilde{Q},\sW}$ with $(\delta_i,\eta)_{\pmb Q}\neq 0$ for some $i\in \widetilde{Q}_0$ is contained in the image of $\bar{\varrho}$ will suffice to imply the compatibility of coproducts. In fact, by comparing \cite[(4.7)]{GNW} , it is enough to show that the coproducts for $h_{i,1}$ are compatible. 

\subsection{Any admissible \texorpdfstring{$\mathcal C$}{C} is enough}\label{sec Y(Q,C)=Y(Q,w)}

\begin{Theorem}
\label{thm admissible}
If $\mathcal C$ is admissible, then for any $\mu\in \bZ_{\leqslant 0}^{Q_0}$, we have $\mathsf Y_\mu(Q,\sW,\mathcal C)=\mathsf Y_\mu(Q,\sW)$.
\end{Theorem}

\begin{proof}
We first show that $\mathsf Y_0(Q,\sW,\mathcal C)=\mathsf Y_0(Q,\sW)$. Notice that $c_1(\lambda)\in \mathsf Y_0(Q,\sW,\mathcal C)$ by Lemma \ref{lem Y^0 generators}, so the coproduct $\Delta c_1(\lambda)$ is contained in $\mathsf Y_0(Q,\sW,\mathcal C)\widehat{\otimes} \mathsf Y_0(Q,\sW,\mathcal C)$ by Lemma \ref{lem closed under coproduct}. Then \eqref{coproduct for 1st Chern class} implies that $e_{\alpha}^{(s)}\in \mathsf Y_0(Q,\sW,\mathcal C)$ for all $\alpha$ and $s$; thus $\mathfrak{g}_\alpha\subset \mathsf Y_0(Q,\sW,\mathcal C)$. The proof of Proposition \ref{prop g(Q,w)}\,(1) shows that $\mathfrak{g}_0\subset \mathsf Y_0(Q,\sW,\mathcal C)$. Combine the above two inclusions, and we get $\mathfrak{g}_{Q,\sW}\subset \mathsf Y_0(Q,\sW,\mathcal C)$. Moreover, by Theorem \ref{thm grY}, we inductively see that $\mathrm{ch}_{k}(\mathsf D_{\In,i}-\mathsf D_{\Out,i})$ is contained in $\mathsf Y_0(Q,\sW,\mathcal C)$ for all $k$, and this implies $\mathsf Y_0(Q,\sW,\mathcal C)=\mathsf Y_0(Q,\sW)$.

By Lemma \ref{lem shift generation lemma}, $\mathsf Y_\mu(Q,\sW)$ is generated by the subset $\bigcup_{a\in \bC}\sS_{0,\mu;z}(\mathsf Y_{0}(Q,\sW))\big|_{z=a}$, which is the same as $$\bigcup_{a\in \bC}\sS_{0,\mu;z}(\mathsf Y_{0}(Q,\sW,\mathcal C))\big|_{z=a}\,,$$ by Step 1, and the latter generates $\mathsf Y_\mu(Q,\sW,\mathcal C)$ by Lemma \ref{lem shift generation lemma} again. Thus, $\mathsf Y_\mu(Q,\sW,\mathcal C)=\mathsf Y_\mu(Q,\sW)$.
\end{proof}

\subsection{Comparison with Maulik-Okounkov Yangians}\label{sec compare MO Yangian}

Let $Q$ be a quiver, $\overline{Q}$ be its doubled quiver, and $\widetilde{Q}$ be its tripled quiver, equipped with canonical cubic potential \eqref{equ on can cub pot}:
$$\widetilde{\sW}=\sum_{i\in Q_0}\varepsilon_i\mu_i, $$ 
where $\varepsilon_i$ is the edge loop on the node $i$, and $\mu_i$ is the $i$-th component of moment map. We take $\sT_0=\sA_0\times \bC^*_{\hbar}$, where $\sA_0$ acts on $\overline{Q}$ by scaling pair of arrows $(X_a,X_{a^*})$ with opposite weights, $\bC^*_{\hbar}$ fixes $X_a$ and scales $X_{a^*}$ with weight $-1$ and scales edge loop $\Phi_i$ with weight $1$.

\begin{Theorem}\label{thm compare with MO yangian}
Conjecture \ref{conj: bps lie} is true for a tripled quiver $\widetilde{Q}$ with canonical cubic potential $\widetilde{\sW}$. Moreover, there is a natural $\bC(\mathsf t_0)$-bialgebra isomorphism $$\mathsf Y_0(\widetilde{Q},\widetilde{\sW})\cong \mathsf Y^{\mathrm{MO}}_Q\otimes_{\bC[\mathsf t_0]}\bC(\mathsf t_0).$$
\end{Theorem}

\begin{proof}
Let us take $\mathcal C=\{(\delta_i,\sT_0\curvearrowright\cM(\delta_i),\widetilde{\sW}_i)\}_{i\in Q_0}$, where $(a,q)\in \sA_0\times \bC^*_\hbar=\sT_0$ acts on the framing $(A_i,B_i)$ by
\begin{align*}
    (A_i,B_i)\mapsto (A_i,q^{-1}B_i),
\end{align*}
and the framed potentials are given by
\begin{align*}
    \widetilde{\sW}_i=\widetilde{\sW}+B_i\Phi_iA_i\:.
\end{align*}
It is clear that $\mathcal C$ is admissible. Then it follows from Theorem \ref{thm admissible} that $\mathsf Y_0(\widetilde{Q},\widetilde{\sW})\cong\mathsf Y_0(\widetilde{Q},\mathcal C)$. 

Plug $\cH_0=\cH^{\widetilde{\sW}}_{\underline{\mathbf 0}}=\bC[\mathsf t_0]$ and $\mathcal H_i=\cH^{\widetilde{\sW}_i}_{\underline{\delta_i}}$ into Lemma \ref{lem sufficient state spaces}, and the natural map
\begin{align*}
    \mathsf Y_0(\widetilde Q,\widetilde{\sW})=\mathsf Y_0(\widetilde Q,\widetilde{\sW},\mathcal C)\longrightarrow \prod_{i_1,\cdots,i_n}\End_{\bC(\mathsf t_0)[a_{i_1},\ldots,a_{i_n}]}\left(\cH_{i_1}[a_{i_1}]\otimes\cdots\otimes \cH_{i_n}[a_{i_n}]\right)
\end{align*}
is injective. By dimensional reduction, we have 
\begin{align*}
    H^{\sT_0}(\cM(\delta_i),\widetilde\sw_i)\cong H^{\sT_0}(\cN(\delta_i)),
\end{align*}
where $\cN(\delta_i)$ is the Nakajima variety associated to $\overline{Q}$ (Example \ref{ex doubled vs tripled}). The above isomorphisms are compatible with stable envelopes (\cite[Thm.~6.10]{COZZ}) and therefore compatible with $R$-matrices. Thus the image of $\mathsf Y_0(\widetilde Q,\widetilde{\sW})$ in the product of endomorphisms is generated by matrices elements of the $R$-matrices for Nakajima varieties, which agrees with $\mathsf Y^{\mathrm{MO}}_Q$ after localization (since we define $\mathsf Y_0(\widetilde Q,\widetilde{\sW})$ only after localization), that is,
$\mathsf Y_0(\widetilde{Q},\widetilde{\sW})\cong \mathsf Y^{\mathrm{MO}}_Q\otimes_{\bC[\mathsf t_0]}\bC(\mathsf t_0)$. This is a bialgebra isomorphism because the counit is induced by restriction to trivial representation $\cH_{0}$, and the coproduct is induced by stable envelope which is compatible with dimensional reduction 
(\cite[Thm.~6.10]{COZZ}).

The Yangian isomorphism is compatible with filtration by degree in $u$, so we get an induced isomorphism on the $F_0$ part:
\begin{align*}
    \mathcal U(\mathfrak{g}_{\widetilde{Q},\widetilde{\sW}})\cong \mathcal U(\mathfrak{g}_{Q}^{\mathrm{MO}}\otimes_{\bC[\mathsf t_0]}\bC(\mathsf t_0)),
\end{align*}
from which we deduce $\mathfrak{g}_{\widetilde{Q},\widetilde{\sW}}\cong \mathfrak{g}_{Q}^{\mathrm{MO}}\otimes_{\bC[\mathsf t_0]}\bC(\mathsf t_0)$.
Then \cite[Thm.\,B]{BD} yields
\begin{equation*}
\mathfrak{g}^+_{\widetilde{Q},\widetilde{\sW}}\cong \mathfrak{g}_{\widetilde{Q},\widetilde{\sW}}^{\mathrm{BPS},\sT_0}\otimes_{\bC[\mathsf t_0]}\bC(\mathsf t_0). \qedhere \end{equation*}
\end{proof}

\begin{Corollary}\label{cor MO yangian crit module}
Let $\cN(\bd)=\bigsqcup_{\bv}\cN(\bv,\bd)$ be the Nakajima variety associated to a doubled quiver $\overline{Q}$, and let $\sT_0=\bC^*_\hbar\times \mathsf{S}$ act on $\mathrm{Rep}_{Q}(\bv,\bd)$ by flavour symmetry (i.e. commutes with $G=\prod_{i\in Q_0}\GL(\bv_i)$ action), then $\sT_0$ act on $\cN(\bd)$ using the identification $T^*\mathrm{Rep}_{Q}(\bv,\bd)=\mathrm{Rep}_{Q}(\bv,\bd)\oplus \hbar\:\mathrm{Rep}_{Q}(\bv,\bd)^{\vee}$. Then for any $\sT_0$-invariant function $\phi=\tr(\mathsf m)$ ($\mathsf m$ is a potential of the framed doubled quiver), $H^{\sT_0}(\cN(\bd),\phi)_{\loc}$ has a natural $\mathsf Y^{\mathrm{MO}}_Q\otimes_{\bC[\mathsf t_0]}\bC(\mathsf t_0)$ module structure. 

Moreover, there is a specialization map \cite[Ex.~7.19]{COZZ}:
\begin{align}\label{equ on spci map}
    \mathsf{sp}\colon H^{\sT_0}(\cN(\bd),\phi)_{\loc}\to H^{\sT_0}(\cN(\bd))_{\loc},
\end{align}
which is a $\mathsf Y^{\mathrm{MO}}_Q\otimes_{\bC[\mathsf t_0]}\bC(\mathsf t_0)$ module map.
\end{Corollary}

\begin{proof}
Let $\cM(\bd)=\bigsqcup_{\bv}\cM(\bv,\bd)$ be the symmetric quiver variety associated with the tripled quiver $\widetilde{Q}$, then we have the deformed dimensional reduction isomorphism \eqref{equ on dim redd}: 
$$H^{\sT_0}(\cM(\bd),\widetilde\sw+\phi)\cong H^{\sT_0}(\cN(\bd),\phi),$$ and the former is naturally a $\mathsf Y_0(\widetilde Q,\widetilde{\sW})\cong\mathsf Y^{\mathrm{MO}}_Q\otimes_{\bC[\mathsf t_0]}\bC(\mathsf t_0)$ module. Since  \eqref{equ on spci map} is compatible with stable envelope \cite[Ex.~7.19]{COZZ}, it commutes with $R$-matrices and therefore induces a $\mathsf Y^{\mathrm{MO}}_Q\otimes_{\bC[\mathsf t_0]}\bC(\mathsf t_0)$ module map.
\end{proof}

\subsection{Casimir operator}

\begin{Definition}\label{def of Casimir}
For any $\alpha\gtrless 0$, define the \textit{Casimir operator}
\begin{align*}
    \mathrm{Cas}_\alpha:=\sum_{s}e^{(s)}_\alpha e^{(s)}_{-\alpha}\in U(\mathfrak{g}_{Q,\sW})\subset \mathsf Y_0(Q,\sW)
\end{align*}
Note that $\mathrm{Cas}_\alpha-(-1)^{|\alpha|}\mathrm{Cas}_{-\alpha}$ is in the Cartan subalgebra $\mathfrak{g}_0$ by Proposition \ref{prop g(Q,w)}(8).
\end{Definition}

The Casimir operator, when acting on a state space, is induced by a Steinberg correspondence (Definition \ref{def of Steinberg corr}).

\begin{Proposition}\label{prop Casimir}
Let $\bd_{\In} =\bd_{\Out}$ be a symmetric framing, and denote $X=\cM(\bv,{\bd}), X_0=\cM_0(\bv,{\bd})$, and $\sT=\sT_0\times \sA^{\mathrm{fr}}$. Then for any $\alpha>0$, there exists a Steinberg correspondence $\mathsf C_\alpha\in H^{\sT}_{2\dim X-2}(X\times_{X_0} X)$, such that for any potential $\sw^{\mathrm{fr}}$, the action of $\mathrm{Cas}_\alpha$ on $H^\sT(X,\sw^{\mathrm{fr}})$ is induced by $\mathsf C_\alpha$.
\end{Proposition}

\begin{proof}
Let $i\in Q_0$ be a node such that $\alpha_i>0$. Take an auxiliary symmetric quiver variety by setting the framing dimension to be $\delta_i$, let $\sT_0$ act on the framing vector spaces trivially, and introduce an extra torus $\bC^*_t$ such that it only acts on the out-going framing with weight $-1$ (and fixes all other arrows). Consider the composition 
\begin{multline*}
H^{\sT_0\times \bC^*_t}(\cM(\mathbf 0,\delta_i),\sw)\otimes H^{\sT}(X,\sw^{\mathrm{fr}})\xrightarrow{\pmb r_{\alpha,-\alpha}} H^{\sT_0\times \bC^*_t}(\cM(\alpha,\delta_i),\sw)\otimes H^{\sT}(\cM(\bv-\alpha,\bd),\sw^{\mathrm{fr}})\\
\xrightarrow{\pmb r_{-\alpha,\alpha}} H^{\sT_0\times \bC^*_t}(\cM(\mathbf 0,\delta_i),\sw)\otimes H^{\sT}(X,\sw^{\mathrm{fr}})\,,
% \scriptstyle{
% }
\end{multline*}
where $\pmb r_{\alpha,-\alpha}$ and $\pmb r_{-\alpha,\alpha}$ are the corresponding components of the classical $R$-matrix \eqref{explicit r matrix}. By explicit computations,
\begin{align}\label{r^2 explicit}
    \pmb r_{-\alpha,\alpha}\circ \pmb r_{\alpha,-\alpha}=\alpha_i\,t\,\sum_{s}e^{(s)}_\alpha e^{(s)}_{-\alpha}=\alpha_i\,t\,\mathrm{Cas}_\alpha\,.
\end{align}
On the other hand, off-diagonal component of $\pmb r$ is induced by critical convolution using $\can(\pmb \rho)$, where $\pmb\rho$ is a class in $H^{\sT\times \bC^*_t}\left((\cM(\delta_i)\times \cM(\underline{\bd}))\times_{\cM_0(\delta_i)\times \cM_0(\underline{\bd})}(\cM(\delta_i)\times \cM(\underline{\bd}))\right)$ given by \eqref{r mat corr}. Therefore $\pmb r_{-\alpha,\alpha}\circ \pmb r_{\alpha,-\alpha}$ is induced by the class $\pmb \rho_{-\alpha,\alpha}\circ \pmb \rho_{\alpha,-\alpha}\in H^{\sT\times \bC^*_t}(X\times_{X_0} X)$ of homological degree $2\dim X-4$. Since $\bC^*_t$ acts on $X\times_{X_0} X$ trivially, 
$$H^{\sT\times \bC^*_t}(X\times_{X_0} X)\cong H^{\sT}(X\times_{X_0} X)\otimes_\bC H_{\bC^*_t}(\pt),$$ 
and accordingly
\begin{align}\label{r^2 convolution}
    \pmb \rho_{-\alpha,\alpha}\circ \pmb \rho_{\alpha,-\alpha}=\sum_{k=0}^2 a_k t^k\,,\quad a_k\in H^{\sT}_{2\dim X-2k}(X\times_{X_0} X).
\end{align}
Compare \eqref{r^2 explicit} with \eqref{r^2 convolution}, and we see that $\mathrm{Cas}_\alpha$ is induced by the class $a_1/\alpha_i\in H^{\sT}_{2\dim X-2}(X\times_{X_0} X)$. Denote $Y=\cM(\bv,\underline{\delta_i}+\underline{\bd})$, let $\pi\colon Y\to B$ be the flat morphism to the affine space $B$ in \cite{COZZ3}\footnote{Denoted $\tau\colon X\to \cH/\Gamma$ in \cite{COZZ3}.}, and let $U\subset B$ be the dense open subset such that $\pi|_{\pi^{-1}(U)}\colon \pi^{-1}(U)\to U$ is an affine morphism. Since $\pi^{-1}(U)$ is quasi-affine, so any attracting set is closed in $\pi^{-1}(U)$. 
Therefore for two distinct $\bC^*_u$-fixed components $F',F\subset Y^{\bC^*_u}$, we have $$\overline{\Attr}_\pm(\Delta_F)\cap (\pi^{-1}(U)\times \pi^{-1}(U))=\Attr_\pm(\Delta_F)\cap (\pi^{-1}(U)\times \pi^{-1}(U)),  $$ 
so the intersection $\overline{\Attr}_\pm(\Delta_F)\cap (F'\times F)$ is disjoint from $\pi^{-1}(U)\times \pi^{-1}(U)$.
We conclude that, after passing to $\bC^*_t$-equivariant cohomology, the class $\pmb \rho_{-\alpha,\alpha}\circ \pmb \rho_{\alpha,-\alpha}$ as well as its $t$-expansion coefficients $a_0,a_1,a_2$, are in the image of the pushforward map $H^{\bC^*_u}(X\times_{X_0} X\times_B(B\setminus U))\to H^{\bC^*_u}(X\times_{X_0} X)$. In particular, $a_1$ is a Steinberg correspondence by Lemma \ref{lem away from affine locus}. $\pmb\rho$ is independent of potential by construction \eqref{r mat corr}, so is $a_1$. 
\end{proof}

\subsection{Shifted Casimir operator}

Let $\mu<0$ be an antidominant shift. Pick $\underline{\bd}$ with $\bd_{\Out}-\bd_{\In}=\mu$, then the symmetric quiver variety $\cM(\bd_{\In})$ is the total space of a vector bundle $E=\bigoplus_{i\in Q_0}\Hom(\mathsf V_i,\bC^{-\mu_i})$ on $\cM(\underline{\bd})$. Let $\pi\colon \cM(\bd_{\In})\to \cM(\underline{\bd})$ be the bundle map. Take a framed potential $\sw^{\mathrm{fr}}$ on $\cM(\underline{\bd})$ and its pullback to $\cM(\bd_{\In})$ is still denoted by $\sw^{\mathrm{fr}}$. 
% Then we have the pullback isomorphism: 
% \begin{align*}
%     \pi^*\colon \cH^{\sW^{\mathrm{fr}}}_{\underline{\bd}',\sA^{\mathrm{fr}}}\xrightarrow{\cong} \cH^{\sW^{\mathrm{fr}}}_{\underline{\bd},\sA^{\mathrm{fr}}}.
% \end{align*}
%is an isomorphism. 
Consider the following auxiliary $\bC^*_z$ action on $\cM(\bd_{\In})$: it acts trivially on $\cM(\underline{\bd})$ and scales the fibers of $E$ with weight $1$. Then  
\begin{align*}
H^{\sT\times \bC^*_z}(\cM(\underline{\bd}),\sw^{\mathrm{fr}})\cong H^{\sT}(\cM(\underline{\bd}),\sw^{\mathrm{fr}})\otimes_\bC\bC[z],
\end{align*}
and $\pi^*\colon H^{\sT\times \bC^*_z}(\cM(\underline{\bd}),\sw^{\mathrm{fr}})\to H^{\sT\times \bC^*_z}(\cM(\bd_{\In}),\sw^{\mathrm{fr}})$ is an isomorphism whose inverse is $i^*$, where $i\colon \cM(\underline{\bd})\hookrightarrow \cM(\bd_{\In})$ is the zero section.

By Lemma \ref{lem shift map}, the assignment $\mathsf E(m)\mapsto i^*\sE(m)_z\pi^*$ induces an algebra homomorphism $\mathsf Y_0(Q,\sW)\to \mathsf Y_\mu(Q,\sW)[z]$.

\begin{Proposition}\label{prop shifted Casimir}
Let $\alpha>0$. Under the map $\mathsf Y_0(Q,\sW)\to \mathsf Y_\mu(Q,\sW)[z]$ sending $\mathsf E(m)$ to $i^*\sE(m)_z\pi^*$, the operator $\mathrm{Cas}_{\alpha}$ is mapped to a polynomial in $z$ of degree $\leqslant -\alpha\cdot \mu$. 

Moreover, denote $X=\cM(\bv,\underline{\bd}), X_0=\cM_0(\bv,\underline{\bd})$, then there exists a class $\mathsf C_{\alpha,\mu}\in H^{\sT}_{}(X\times_{X_0} X)$, such that for any potential $\sW^{\mathrm{fr}}$, the action of the coefficient of $z^{-\alpha\cdot \mu}$ in $i^*(\mathrm{Cas}_{\alpha})_z\pi^*$ on $H^\sT(X,\sw^{\mathrm{fr}})$ is induced by $\mathsf C_{\alpha,\mu}$.
\end{Proposition}

\begin{proof}
Let $i\in Q_0$ be a node such that $\alpha_i>0$, and $\underline{\delta_i}:=(\delta_i,\delta_i)\in \bZ_{\geqslant 0}^{Q_0}\times \bZ_{\geqslant 0}^{Q_0}$. 
Denote
\begin{align*}
Y&=\cM(\bv,\underline{\delta_i}+\underline{\bd}), & F&=\cM(\mathbf 0,\delta_i)\times \cM(\bv,\underline{\bd}),& F'&=\cM(\alpha,\delta_i)\times \cM(\bv-\alpha,\underline{\bd}),\\
\widetilde{Y}&=\cM(\bv,\delta_i+\bd_{\In}),& \widetilde{F}&=\cM(\mathbf 0,\delta_i)\times \cM(\bv,\bd_{\In}),& \widetilde{F}'&=\cM(\alpha,\delta_i)\times \cM(\bv-\alpha,\bd_{\In}).
\end{align*}
Then $\widetilde{Y}$ is the total space of the vector bundle $E=\bigoplus_{i\in Q_0}\Hom(\mathsf V_i, \bC^{-\mu_i})$. $F,F'$ belongs to the set of connected components of torus fixed point set $Y^{\bC^*_u}$, where $\bC^*_u$ act on the framing by weight decomposition $u\,\underline{\delta_i}+\underline{\bd}$. This $\bC^*_u$ action naturally extends to $\widetilde{Y}$, and under the bundle projection map $\pi\colon \widetilde{Y}\to Y$, $\widetilde{F}, \widetilde{F}'$ are the connected components of $\widetilde{Y}^{\bC^*_u}$ that map to $F,F'$, respectively. Denote the stable envelopes in the $\pm$-chambers on $Y$ (resp. on $\widetilde{Y}$) by $\Stab_{\pm,Y}$ (resp. $\Stab_{\pm,\widetilde{Y}}$).

As in the proof of Proposition \ref{prop Casimir}, we let $\bC^*_t$ act on the out-going part of the framing $\underline{\delta_i}$ with weight $-1$ and fix all the other arrows. Then $\mathrm{Cas_\alpha}$ is determined by the classical $R$-matrix:
\begin{equation*}
\xymatrix{
H^{\sT\times \bC^*_z\times \bC^*_t}(\widetilde{F}, \sw^{\mathrm{fr}})\ar[r]^{\pmb r_{\alpha,-\alpha}} \ar@/^2pc/[rr]^{\alpha_i\,t\,\mathrm{Cas}_\alpha}_{\rotatebox{90}{=}} & H^{\sT\times \bC^*_z\times \bC^*_t}(\widetilde{F}', \sw^{\mathrm{fr}})
\ar[r]^{\pmb r_{-\alpha,\alpha}} & H^{\sT\times \bC^*_z\times \bC^*_t}(\widetilde{F}, \sw^{\mathrm{fr}})\,.
}
\end{equation*}
Let $i\colon Y\to \widetilde{Y}$ be the zero section map. And we let $\bC^*_z$ act on $\widetilde{Y}$ such that it fixes $Y$ and scale the fibers of $E$ with weight $1$. We claim that
\begin{enumerate}
    \item $\Stab_{-,Y}\big|_{F'\times F}=\pm u^{\rk N^+_{F'/Y}-1}\cdot i^*\pmb r_{\alpha,-\alpha}\pi^*+\text{lower order in $u$}$, where the sign is determined by the normalizer $\epsilon$,
    \item $z^{-\alpha\cdot\mu}\Stab_{+,Y}\big|_{F\times F'}+\text{lower order in $z$}=\pm u^{\rk N^-_{F/Y}-1}\cdot i^*\pmb r_{-\alpha,\alpha}\pi^*+\text{lower order in $u$}$, where the sign is determined by the normalizer $\epsilon$.
\end{enumerate}
Let us assume these claims for now. By \cite[Thm.~8.9]{COZZ}, the stable envelopes for $Y$ are induced by $\overline{\Attr}_{\pm}(\Delta_{Y^{\bC^*_u}})$. Write
\begin{gather*}
(F'\times F\hookrightarrow Y\times F)^![\overline{\Attr}_{-}(\Delta_{F})]=u^{\rk N^+_{F'/Y}-1}\cdot C_{F',F}+\text{lower order in $u$},\\
(F\times F'\hookrightarrow Y\times F')^![\overline{\Attr}_{+}(\Delta_{F'})]=u^{\rk N^-_{F/Y}-1}\cdot C_{F,F'}+\text{lower order in $u$},
\end{gather*}
where $C_{F',F}\in H^{\sT\times \bC^*_t}(F'\times_{Y_0} F)$ and $C_{F,F'}\in H^{\sT\times \bC^*_t}(F\times_{Y_0} F')$ with $Y_0=\cM_0(\bv,\underline{\delta_i}+\underline{\bd})$. It follows that
\begin{align*}
i^*\pmb r_{\alpha,-\alpha}\pi^*=\pm C_{F',F},\quad i^*\pmb r_{-\alpha,\alpha}\pi^*=\pm z^{-\alpha\cdot\mu}C_{F,F'}+\text{lower order in $z$}.
\end{align*}
Therefore,
\begin{align*}
\alpha_i\,t\cdot i^*(\mathrm{Cas}_{\alpha})_z\pi^*=\pm z^{-\alpha\cdot\mu}C_{F,F'}\circ C_{F',F}+\text{lower order in $z$}.
\end{align*}
This shows that $i^*(\mathrm{Cas}_{\alpha})_z\pi^*$ is a polynomial in $z$ of degree $\leqslant -\alpha\cdot \mu$.

Finally, since $\bC^*_t$ acts on $F$ trivially, $H^{\sT\times \bC^*_t}(F\times_{Y_0} F)=H^{\sT}(F\times_{Y_0} F)\otimes H_{\bC^*_t}(\pt)$. Expand $C_{F,F'}\circ C_{F',F}$ into a polynomial in $t$: 
\begin{align*}
C_{F,F'}\circ C_{F',F}=\sum_{n}b_n t^t,\;\text{ with }\;b_n\in H^{\sT}(F\times_{Y_0} F),
\end{align*}
then we see that the coefficient of $z^{-\alpha\cdot \mu}$ in $i^*(\mathrm{Cas}_{\alpha})_z\pi^*$ is induced by $\pm b_1/\alpha_i$. Obviously $b_1$ does not depend on the potential.

Let us prove the claims. The claim (1) is a consequence of \cite[Prop.~8.2]{COZZ}, namely, the bundle $E$ is attracting in the ``$-$'' chamber, so we have
\begin{align*}
\Stab_{-,Y}\big|_{F'\times F}=i^*\Stab_{-,\widetilde Y}\big|_{\widetilde F'\times \widetilde F}\;\pi^*=\pm u^{\rk N^+_{\widetilde F'/\widetilde Y}-1}\cdot i^*\pmb r_{\alpha,-\alpha}\pi^*+\text{lower order in $u$}.
\end{align*}
Since $\rk N^+_{\widetilde F'/\widetilde Y}=\rk N^+_{F'/Y}$, claim (1) follows. The claim (2) is a consequence of \cite[Prop.~8.3]{COZZ}, namely, the bundle $E$ is repelling in the ``$+$'' chamber, so we have
\begin{equation}\label{shifted Casimir_+ chamber}
\begin{split}
e^{\sT\times \bC^*_z\times \bC^*_t}(E|_F)\cdot\Stab_{+,Y}\big|_{F\times F'}\;\cdot e^{\sT\times \bC^*_z\times \bC^*_t}(E|_{F'}^{\mathrm{fixed}})^{-1}&= i^*\Stab_{+,\widetilde Y}\big|_{\widetilde F\times \widetilde F'}\;\pi^*\\
&=\pm u^{\rk N^-_{\widetilde F/\widetilde Y}-1}\cdot i^*\pmb r_{-\alpha,\alpha}\pi^*+\text{lower order in $u$}.
\end{split}
\end{equation}
Note that $\rk N^-_{\widetilde F/\widetilde Y}=\rk N^-_{F/Y}$. The equivariant Euler classes are given by
\begin{align*}
e^{\sT\times \bC^*_z\times \bC^*_t}(E|_F)=z^{-\bv\cdot \mu}+\text{lower order in $z$},\quad e^{\sT\times \bC^*_z\times \bC^*_t}(E|_{F'}^{\mathrm{fixed}})=z^{-(\bv-\alpha)\cdot\mu}+\text{lower order in $z$\,,}
\end{align*}
and they are independent of $u$. Thus, the LHS of \eqref{shifted Casimir_+ chamber} in the $z$-expansion is $z^{-\alpha\cdot\mu}\Stab_{+,Y}\big|_{F\times F'}+\text{lower order in $z$}$, and the claim (2) follows.
\end{proof}

\begin{Definition}\label{def of shifted Casimir}
For $\alpha>0$, we define the \textit{shifted Casimir operator} $\mathrm{Cas}_{\alpha,\mu}\in \mathsf Y_\mu(Q,\sW)$ to be the coefficient of $z^{-\alpha\cdot \mu}$ in $(-1)^{\alpha\cdot\mu}\,i^*(\mathrm{Cas}_{\alpha})_z\pi^*$, equivalently,
\begin{align*}
    \mathrm{Cas}_{\alpha,\mu}=\lim_{z\to \infty} (-z)^{\alpha\cdot \mu}\cdot i^*(\mathrm{Cas}_{\alpha})_z\pi^*\,.
\end{align*}
\end{Definition}

% Denote $X=\cM(\bv,\underline{\bd}), X_0=\cM_0(\bv,\underline{\bd})$. Then for any $\alpha>0$, there exists a class $\mathsf C_{\alpha,\mu}\in H^{\sT}_{}(X\times_{X_0} X)$, such that for any potential $\sw^{\mathrm{fr}}$, the action of $\mathrm{Cas}_{\alpha,\mu}$ on $H^\sT(X,\sw^{\mathrm{fr}})$ is induced by $\mathsf C_{\alpha,\mu}$.

\section{Applications to quantum multiplication by divisors}\label{sec quant mult}

Building on techniques developed in the previous section, we 
deduce the formula of quantum multiplication by divisors for symmetric quiver varieties with potentials (Theorem \ref{thm on qm div for sym}), 
and certain asymmetric cases (Theorem \ref{thm qm div for asym_sp inj}).

\subsection{Quantum multiplication by divisors}

Let $X$ be a quiver variety \eqref{equ on qv} with a $\sT$-invariant function $\sw$, and the affinization map $X\to X_0$. Consider the moduli stack  of 
genus $0$, 3-pointed stable maps with proper evaluation map 
$$\ev:=\ev_1\times \ev_2\times \ev_3\colon \overline{M}_{0,3}(X,\beta)\to X\times X\times X. $$ 
As any domain curve of a stable map contracts to a point in the affine space $X_0$, we have 
$$\ev(\overline{M}_{0,3}(X,\beta))\subseteq (X\times Z(\sw\boxminus \sw))\cap (X\times_{X_0}X\times_{X_0}X). $$
As $X$ is smooth, there is a $\sT$-equivariant virtual class 
$$[\overline{M}_{0,3}(X,\beta)]_{\sT}^{\vir}\in H^{\sT}(\overline{M}_{0,3}(X,\beta)). $$
\begin{Definition}
Given any $\gamma\in H_{\sT}^*(X)$, we define 
\begin{equation}\label{equ on Mgamma}M_{\beta}(\gamma):=\mathrm{can}\circ (\ev_{2}\times \ev_3)_*([\overline{M}_{0,3}(X,\beta)]_{\sT}^{\vir}\cap \ev_1^*\gamma)
\in H^{\sT}(X^2,\sw\boxminus \sw). \end{equation}
This induces the \textit{quantum multiplication} by classes in $\gamma \in H_{\sT}^*(X)$: 
\begin{equation}\label{equ on Mgammabeta as map}
\gamma \star_{\beta} \cdot \, \colon H^{\sT}(X,\sw)\to H^{\sT}(X,\sw), \quad \alpha\mapsto M_{\beta}(\gamma)\circ \alpha,
\end{equation}
where $\circ$ is the convolution product \cite[\S 3.4.1]{COZZ}. Summing over $\beta$, we define 
\begin{equation}\label{equ on Mgamma as map}
\gamma \star \cdot:=\sum_{\beta}(\gamma \star_{\beta} \cdot)\,z^\beta\, \colon H^{\sT}(X,\sw)[\![z]\!]\to H^{\sT}(X,\sw)[\![z]\!].
\end{equation}
%M_{\beta}\colon H_{\sT}^*(X)\otimes H^{\sT}(X,\sw)\to H^{\sT}(X,\sw), \quad 
%(\gamma,\alpha)\mapsto M_{\beta}(\gamma)\circ \alpha=: \gamma \star . 
%Summing over all $\beta$, we define 
%$$M(\gamma,-):=\sum_{\beta}M_{\beta}(\gamma,-)z^\beta\colon H^{\sT}(X,\sw)[\![z]\!]\to H^{\sT}(X,\sw)[\![z]\!], $$
%$$M:=\sum_{\beta}M_{\beta}z^\beta\colon H_{\sT}^*(X)\otimes H^{\sT}(X,\sw)[\![z]\!]\to H^{\sT}(X,\sw)[\![z]\!]. $$ 
We define 
$$\sQ_{\beta}:=\mathrm{can}\circ \ev_*([\overline{M}_{0,2}(X,\beta)]_{\sT}^{\vir})\in H^{\sT}(X^2,\sw\boxminus \sw)_{X\times_{X_0} X}. $$
%\yl{consider the version using $\QM_{0,2}^{}(X,\beta)$, here stability implies all domain curves have $\omega_{\log}=\oO$.}
By abuse of notations, its induced map (via convolution) is denoted by 
\begin{equation}\label{equ on Qbeta}\sQ_{X,\beta}=\sQ_{\beta}\colon H^{\sT}(X,\sw) \to H^{\sT}(X,\sw).  \end{equation}
Summing over $\beta$, we define 
\begin{equation}\label{equ on Q map}
\sQ_{X}=\sQ:=\sum_{\beta}\sQ_{\beta}\,z^\beta\, \colon H^{\sT}(X,\sw)[\![z]\!]\to H^{\sT}(X,\sw)[\![z]\!].
\end{equation}
\end{Definition}

\begin{Lemma}
The map $\gamma\mapsto \gamma\star\cdot$ is unital and associative with respect to the quantum product on $H^*_\sT(X)$, i.e. $$1\star\alpha=\alpha, \quad \gamma_1\star(\gamma_2\star\alpha)=(\gamma_1\star\gamma_2)\star\alpha, \quad \forall\,\gamma_1, \gamma_2\in H^*_\sT(X), \,\alpha\in H^{\sT}(X,\sw)[\![z]\!]. $$ 
In other words, $\gamma\mapsto \gamma\star\cdot$ induces a quantum cohomology $QH^*_\sT(X)$ module structure on $H^{\sT}(X,\sw)[\![z]\!]$.
\end{Lemma}

\begin{proof}
$\ev_{2}\times \ev_3\colon \overline{M}_{0,3}(X,\beta)\to X^2$ factors through the universal curve map $\overline{M}_{0,3}(X,\beta)\to \overline{M}_{0,2}(X,\beta)$, so $$(\ev_{2}\times \ev_3)_*([\overline{M}_{0,3}(X,\beta)]_{\sT}^{\vir})=0.$$
It follows that $1\star\cdot=\id$. 

By a standard degeneration argument, we have the following identity in $H^{\sT}(X^2,\sw\boxminus \sw)$:
\begin{align*}
    \mathrm{can}\circ (\ev_{3}\times \ev_4)_*([\overline{M}_{0,4}(X,\beta)]_{\sT}^{\vir}\cap \ev_1^*\gamma_1\cap \ev_2^*\gamma_2)=\sum_{\beta_1+\beta_2=\beta}M_{\beta_1}(\gamma_1)\circ M_{\beta_2}(\gamma_2).
\end{align*}
Degeneration in another way gives
\begin{multline*}
(\ev_{3}\times \ev_4)_*([\overline{M}_{0,4}(X,\beta)]_{\sT}^{\vir}\cap \ev_1^*\gamma_1\cap \ev_2^*\gamma_2)\\=\sum_{\beta'_1+\beta'_2=\beta}\pr_{45*} i_3^*\left(\ev_{3*}([\overline{M}_{0,3}(X,\beta_1')]_{\sT}^{\vir}\cap \ev_1^*\gamma_1\cap \ev_2^*\gamma_2)\boxtimes(\ev_{3}\times \ev_4\times \ev_5)_*([\overline{M}_{0,3}(X,\beta_2')]_{\sT}^{\vir})\right),
\end{multline*}
where $i_3\colon X_3\times X_4\times X_5\hookrightarrow X_3\times X_3\times X_4\times X_5$ is induced by the diagonal of $X_3$, and $\pr_{45}\colon X_3\times X_4\times X_5\to  X_4\times X_5$ is the projection. Applying canonical map and summing over $\beta$ gives
\begin{equation*}
\gamma_1\star\cdot\circ\gamma_2\star\cdot=\gamma_1\cap\gamma_2\cap\cdot+\sum_{\beta_2'}\can\circ(\ev_4\times \ev_5)_*([\overline{M}_{0,3}(X,\beta_2')]_{\sT}^{\vir}\cap \ev_{3}^*(\gamma_1\star\gamma_2))=(\gamma_1\star\gamma_2)\star\cdot\;.\qedhere
\end{equation*}
\end{proof}

\begin{Lemma}\label{lem 2-pt Steinberg}
The class $\ev_*([\overline{M}_{0,2}(X,\beta)]_{\sT}^{\vir})\in H^\sT_{2\dim X-2}(X\times_{X_0}X)$ is a Steinberg correspondence in the sense of Definition \ref{def of Steinberg corr}.
\end{Lemma}

\begin{proof}
By Lemma \ref{lem away from affine locus}, it suffices to show that the image of $\ev\colon \overline{M}_{0,2}(X,\beta)\to X\times_{X_0} X$ is contained in $X\times_{X_0} X\times_B (B\setminus U)$, and this follows from the fact that $B$ is affine and $\pi|_{\pi^{-1}(U)}\colon \pi^{-1}(U)\to U$ is an affine morphism (so any proper curve is mapped to $\pi^{-1}(B\setminus U)$).
\end{proof}

\begin{Remark}\label{rmk on two qbet}
If $\gamma$ is a divisor class, $\beta$ is a nonzero curve class, divisor equation implies that 
$$M_{\beta}(\gamma)=(L\cdot \beta)\cdot \sQ_{\beta}. $$ 
In general, there is a conjecture of Joyce, see e.g.~\cite{Joy, CZ, CTZ}, asserting that by counting certain `twisted' maps with $r$-marked points (assuming $\Crit(\sw)^\sT$ is proper), one obtains an 
$r$-pointed function:
\begin{equation}\label{equ on jc}H^{\sT}(X,\sw_{})^{\otimes r} \to H_\sT(\pt)_{\loc}. \end{equation}
In the two marked point case ($r=2$), we can choose the twist to be trivial and 
the twisted maps in \textit{loc.\,cit.} become ordinary maps from rational curves $C$ such that $\omega_{\log,C}\cong \oO_C$.
Under Verdier duality, 
the provisional two pointed function \eqref{equ on jc} becomes an operator as \eqref{equ on Qbeta}, we expect it coincides with $\sQ_{\beta}$ up to a scalar operator. 
\end{Remark}

\subsection{Stable envelopes v.s.\,operator $\sQ$}
By \cite[Lem.~3.29]{COZZ}, $\Stab_\fC^{-1}$ is induced by the transpose of the stable envelope correspondence for the opposite chamber, i.e. 
\begin{align*}
\Stab_\fC^{-1}=[\Stab_{-\fC}]^{\mathrm{t}}\circ \,\colon H^{\sT}(X,\sw)_\loc\to H^{\sT}(X^\sA,\sw)_\loc\,,
\end{align*}
where $(-)_\loc$ stands for $(-)\otimes_{H_\sA(\pt)}\Frac H_\sA(\pt)$. Note that the convolution $[\Stab_{-\fC}]^{\mathrm{t}}\circ\ev_*[\overline{M}_{0,2}(X,\beta)]^{\vir}\circ[\Stab_\fC]$ is a class in $H^{\sT}(X^\sA\times_{X_0}X^\sA)$ by the proof of \cite[Lem.~3.29]{COZZ}, so $\Stab_\fC^{-1}\circ\, \sQ_\beta\circ \Stab_\fC$ is defined without localization.

\begin{Theorem}\label{thm on stab and Q}
Let $X$ be a symmetric quiver variety (Definition \ref{def of sym qv}) with a torus $\sT$-action, and a $\sT$-invariant function $\sw$. Let 
$\sA\subset \sT$ be a self-dual subtorus, $\fC$ be a chamber. 

For any $\sA$-fixed component $F\subseteq X^{\sA}$, and $\beta\in H_2(F,\mathbb{Z})$ a nonzero effective class, we have a commutative diagram:
\begin{equation}\label{equ on comm diag on stab and Q}
\xymatrix{
H^{\sT}(F,\sw)  \ar[rr]^-{\Stab_\fC} \ar[d]_-{(-1)^{(c_1(N^-_{F/X}),\beta)}\cdot\sQ_\beta} &  & H^{\sT}(X,\sw) \ar[d]^-{\sQ_\beta}  \\
H^{\sT}(F,\sw) & & H^{\sT}(X,\sw), \ar[ll]_-{\,\, (-)|_{F}\circ\Stab_\fC^{-1}}    
} \quad \quad \quad \quad 
\end{equation}
up to a scalar operator of $\sT/\sA$-weight. Here $(-)|_{F}\colon H^{\sT}(X^{\sA},\sw)\to H^{\sT}(F,\sw)$ is the projection to component $F$.
\end{Theorem}
\begin{proof}
By Lemma \ref{lem on A vs T loc}, it suffices to show that 
\begin{equation}\label{equ on stabevstab}[\Stab_{-\fC}]^{\mathrm{t}}\circ\ev_*[\overline{M}_{0,2}(X,\beta)]^{\vir}\circ[\Stab_\fC]=(-1)^{(c_1(N^-_{F/X}),\beta)}\cdot\ev_*[\overline{M}_{0,2}(F,\beta)]^{\vir}\in H^{\sA}(F\times_{X_0}F).\end{equation} 
Below we use the following properties of a symmetric quiver variety $X$: 
\begin{enumerate}
\item the affinization map $X\to X_0$ is birational and small (ref.~\cite[Thm.~1.1]{COZZ3}), 
\item there is an $\sA$-equivariant flat map $\pi\colon X\to B$ to an affine space such that generic fibers are affine (ref.~\cite[Thm.~1.1]{COZZ3}), 
\item the virtual dimension satisfies $\mathrm{vdim}(\overline{M}_{0,2}(X,\beta))=\dim X-1$.
\end{enumerate}
By (3) and the definition of convolution product, both sides of \eqref{equ on stabevstab} have homological degree $(2\dim F-2)$ in $ H^{\sA}(F\times_{X_0}F)$. By (1), $\dim F\times_{X_0}F\leqslant \dim F$ and there is a unique irreducible component of $F\times_{X_0}F$ with dimension equal to $\dim F$. By (2), both sides of \eqref{equ on stabevstab} are supported away from an $\sA$-invariant dense open subset $U$ of the diagonal $\Delta_F$. So both sides of \eqref{equ on stabevstab} are $\bQ$-linear combinations of top-dimensional cycles in $(F\times_{X_0}F)\setminus U$, in particular, they are independent of the $\sA$-equivariant parameters. Then it is enough to prove \eqref{equ on stabevstab} in $\sA$-localized cohomology with the $\sA$-equivariant parameters $\mathsf a$ taken to infinity in some direction.

We recall the notion of (un-)broken curves in \cite[\S 7.3]{MO}, \cite[\S 3.8.2]{OP}. By deformation theory, a connected component in $\overline{M}_{0,2}(X,\beta)^{\sA}$ is either broken or unbroken. By Proposition \ref{prop on sym var van} and Lemma \ref{lem on vani of delta cpn} below, only unbroken components for which two marked points do not lie in a contracted connected subcurve contribute to the $\sA$-localization in the $\mathsf a\to \infty$ limit.
% If we take a unbroken components whose elements consist of curves for which both marked points lie on a contracted component attached to an unbroken chain, then by Lemma \ref{lem on vani of delta cpn}, such component contributes zero to the $\sA$-localization in the $\mathsf a\to \infty$ limit. 
We also do not have unbroken chains connecting two points in the same component of $X^\sA$ by \cite[Lem.~7.3.3]{MO}.
Hence the only nonzero contribution comes from the component whose elements consist of stable maps to $F$. In the $\mathsf a\to \infty$ limit, we have
\begin{align*}
[\Stab_{-\fC}]^{\mathrm{t}}\circ\ev_*[\overline{M}_{0,2}(X,\beta)]^{\vir}\circ[\Stab_\fC]&\sim e^\sA(N_{F/X})\cdot\ev_*\left(\frac{[\overline{M}_{0,2}(F,\beta)]^{\vir}}{e^\sA\left(N^{\vir}_{\overline{M}_{0,2}(F,\beta)/\overline{M}_{0,2}(X,\beta)}\right)}\right)\\
&\sim (-1)^{(c_1(N^-_{F/X}),\beta)}\mathsf a^{\rk N_{F/X}-\rk N^{\vir}}\ev_*[\overline{M}_{0,2}(F,\beta)]^{\vir}
\end{align*}
where the sign comes from the fact that $\sA$ is a self-dual torus, 
$\rk N^{\vir}$ is the rank of virtual normal bundle of $\overline{M}_{0,2}(F,\beta)$ in $\overline{M}_{0,2}(X,\beta)$, which equals to 
\begin{align*}
\chi(C,f^*N_{F/X})&=\chi(C,f^*N^+_{F/X})+\chi(C,f^*N^-_{F/X})=\chi(C,f^*N^+_{F/X})+\chi(C,f^*(N^+_{F/X})^\vee)\\
&=2\rk N^+_{F/X}=\rk N_{F/X}\:.
\end{align*}
Therefore, $[\Stab_{-\fC}]^{\mathrm{t}}\circ\ev_*[\overline{M}_{0,2}(X,\beta)]^{\vir}\circ[\Stab_\fC]\sim (-1)^{(c_1(N^-_{F/X}),\beta)}\ev_*[\overline{M}_{0,2}(F,\beta)]^{\vir}$ in the $\mathsf a\to \infty$ limit. This finishes the proof.
% The proof of \cite[Thm.~7.2.1]{MO} on smooth symplectic varieties uses broken curve calculations where holomorphic symplectic 
% forms are involved. This can be replaced by Proposition \ref{prop on sym var van} below. Other argument in \textit{loc.\,cit.} uses only 
% the fact that the repelling part and attracting part of normal bundle $N_{F/X}$ is dual to each other. 
% \yl{revise}
% All together they conclude that the diagram 
% commutes after modulo scalar operator on $H^{\sT}(F,\sw)$. 
% And the scalar operator comes from components of $\overline{M}_{0,2}(X,\beta)^{\sA}$ whose elements consist of curves for which both marked points lie
% on a contracted component attached to an unbroken chain. 
% Lemma \ref{lem on vani of delta cpn} completes the proof. 
\end{proof}

\begin{Proposition}\label{prop on sym var van}
% In the Setting of Theorem \ref{thm on stab and Q}, every map in a given connected component of $\overline{M}_{0,2}(X,\beta)^\sA$ is either broken or unbroken.
In the $\mathsf a\to \infty$ limit, a broken component for which two marked points do not lie in a contracted connected subcurve contributes trivially to the left hand side of \eqref{equ on stabevstab}.
\end{Proposition}

\begin{proof}
Let $\overline{M}$ be a broken components for which two marked points do not lie in contracted connected subcurve, and $f\colon C\to X$ be a stable map in $\overline{M}$, then
there are two possibilities:
\begin{enumerate}
\item[(a)] $C=C_L\cup C_R$, such that $C_L$ and $C_R$ are two connected genus zero (possibly reducible) curves intersecting at a node $p$ with $f(p)\in F'$ ($F'$ is a connected component of $X^\sA$), and $T_{C_L,p}\otimes T_{C_R,p}$ has nonzero $\sA$ weight, and the marked points satisfy $v_1\in C_L$, $v_2\in C_R$.
\item[(b)] $C=C_0\cup C_1\cup \cdots \cup C_n$ with $n\geqslant 1$ such that every $C_i$ is a connected genus zero (possibly reducible) curve, $f(C_0)\subseteq F$ with $\deg f|_{C_0}\neq 0$, and the two marked points $v_1$ and $v_2$ lie in $C_0$, and for every $i>0$, $C_i$ intersects with $C_0$ at a node $p_i$, and the tangent vector $T_{C_i,p_i}$ has nonzero $\sA$ weight.
% \item $C=C_0\cup C_1\cup \cdots \cup C_n$ with $n>1$ such that every $C_i$ is a connected genus zero curve, $f$ contracts $C_0$ to a point in $F$, and the two marked points $v_1$ and $v_2$ lie in $C_0$, and for every $i>0$, $C_i$ intersects with $C_0$ at a node $p_i$, and the tangent vector $T_{C_i,p_i}$ has nonzero $\sA$ weight.
\end{enumerate}
By deformation theory, only one of (a) or (b) can happen in this component. 

% We will show that, in the $\mathsf a\to \infty$ limit, 
% \begin{itemize}
%     \item every component of type (a) contributes trivially,
%     \item the sum of contributions from type (b) components vanishes.
% \end{itemize}

\textbf{Case (a).} In this situation, $\overline{M}$ is of the shape
\begin{align*}
    \overline{M}\cong \overline{M}_L\times_{F'} \overline{M}_R\,,
\end{align*}
where $\overline{M}_L$ (resp.\,$\overline{M}_R$) is the stable map moduli stack corresponding to $C_L$ (resp.\,to $C_R$), and the maps $\overline{M}_L\to F'$ and $\overline{M}_R\to F'$ are evaluations at the node $p$. 

The normalization sequence for the universal domain at the node $p$ is 
$$0\to \oO_{C} \to  \oO_{C_L}\oplus \oO_{C_R}\to \oO_{p} \to 0\,, $$
and it follows that the tangent complex of $\overline{M}$ is related to the tangent complexes of $\overline{M}_L$ and $\overline{M}_R$ by the exact triangle
\begin{align}\label{tangent cplx splits}
    \mathbb T_{\overline{M}}\to \mathrm{pr}_L^*\mathbb T_{\overline{M}_L}\oplus \mathrm{pr}_R^* \mathbb T_{\overline{M}_R}\to \ev_p^* \mathbb T_{F'}\to \mathbb T_{\overline{M}}[1],
\end{align}
$\mathrm{pr}_{L\text{ or }R}$ is the projection from $\overline{M}$ to $\overline{M}_{L\text{ or }R}$. 

The virtual normal contribution $N^\vir_{\overline{M}/\overline{M}_{0,2}(X,\beta)}$ breaks into
\begin{align}\label{vir normal splits}
    N^\vir_{\overline{M}/\overline{M}_{0,2}(X,\beta)}\big|_{[f\colon C\to X]}=N_L\big|_{[f|_{C_L}]}\oplus N_R\big|_{[f|_{C_R}]}-T_X|_{f(p)}^{\mathrm{moving}}+T_{C_L,p}\otimes T_{C_R,p}\,, 
\end{align}
where $N_{L\text{ or }R}\big|_{[f|_{C_{L\text{ or }R}}]}=H^*(C_{L\text{ or }R},f^*T_X)^{\mathrm{moving}}$.

It follows from \eqref{tangent cplx splits} and \eqref{vir normal splits} that the contribution of $\overline{M}$ breaks into sum of convolution products:
\begin{align*}
e^{\sA}\left(N_{F/X}^+\right)\ev_*\left(\frac{[\overline{M}]^{\vir}}{e^{\sA}\left(N^{\vir}_{\overline{M}/\overline{M}_{0,2}(X,\beta)}\right)} \right)e^{\sA}\left(N_{F/X}^-\right)=\sum_{\gamma_L,\gamma_R}\pr_*\Delta_{F'}^*\left(\ev_*(\gamma_L\cdot [\overline{M}_L]^{\vir})\boxtimes \ev_*(\gamma_R\cdot [\overline{M}_R]^{\vir})\right),
% \left(\underset{v\in V_1}{\circ}\ev_*(\gamma_v\cdot [\overline{M}_v]^{\vir})\right)\circ \left(\underset{e\in E}{\circ}\ev_*(\gamma_e\cdot [\overline{M}_e]^{\vir})\right)
\end{align*}
where $\Delta_{F'}\colon F\times F'\times F\hookrightarrow F\times F'\times F'\times F$ is given by the diagonal embedding $F'\hookrightarrow F'\times F'$, $\pr\colon F\times F'\times F\to F\times F$ is the projection to the first and last factor, $\gamma_L,\gamma_R$ are cohomology classes on $\overline{M}_L$ and $\overline{M}_R$ respectively which come from equivariant Euler classes of $-T_X|_{f(p)}^{\mathrm{moving}}+T_{C_L,p}\otimes T_{C_R,p}$.

Let $\pi\colon X\to B$ be the flat map to the affine space $B$ as mentioned above. According to \cite[Lem.~3.22]{COZZ3}, for any $b\in B$,
$$\dim F\times_{X_0} F'\times_{B}\{b\}\leqslant \frac{1}{2}\left(\dim F+\dim F'\right)-\dim B.$$ 
Note that $\pi\colon \pi^{-1}(U)\to U$ is affine for an open dense subset $U\subset B$, so $\ev(\overline M_L)$ is away from $F\times_{X_0} F'\times_{B}U$, and we see that $\ev(\overline M_L)$ is contained in a closed subset of $F\times_{X_0} F'$ of dimension $\leqslant \frac{\dim F+\dim F'}{2}-1$. Therefore
\begin{align*}
\ev(\gamma_L\cdot[\overline M_L]^\vir)\in H_{\leqslant \dim F+\dim F'-2}(F\times_{X_0} F')\otimes \Frac H_\sA(\pt).
\end{align*}
Similarly,
\begin{align*}
\ev(\gamma_R\cdot[\overline M_R]^\vir)\in H_{\leqslant \dim F'+\dim F-2}(F'\times_{X_0} F)\otimes \Frac H_\sA(\pt).
\end{align*}
After taking convolution, we have
\begin{align*}
e^{\sA}\left(N_{F/X}^+\right)\ev_*\left(\frac{[\overline{M}]^{\vir}}{e^{\sA}\left(N^{\vir}_{\overline{M}/\overline{M}_{0,2}(X,\beta)}\right)} \right)e^{\sA}\left(N_{F/X}^-\right)\in H_{\leqslant 2\dim F-4}(F\times_{X_0} F)\otimes \Frac H_\sA(\pt).
\end{align*}
The $\sA$-degree in the above is negative (since the LHS has homological degree $2\dim F-2$ as the LHS of \eqref{equ on stabevstab}), and hence 
vanishes when $\mathsf a\to \infty$.

\bigskip

\textbf{Case (b).} 
% We say that a component $\overline{M}$ is of type $\vec{\beta}=(\beta_0,\ldots,\beta_n)\in H_2(X,\bZ)^{}_{}\times\cdots\times H_2(X,\bZ)_{}$ if a stable map $f\colon C\to X$ in $\overline{M}$ is of type (b) with $C=C_0\cup C_1\cup \cdots \cup C_n$ such that $\deg f|_{C_i}=\beta_i$ (they are nonzero effective classes). 
In this situation, $\overline{M}$ is of the shape
\begin{align*}
    \overline{M}\cong \overline{M}_0\times_{F^n} \prod_{i=1}^n \overline{M}_i\,,
\end{align*}
where $\overline{M}_i$ is the stable map moduli stack corresponding to $C_i$ for $0\leqslant i\leqslant n$, and the maps $\overline{M}_0\to F^n$ and $\overline{M}_i\to F$ are evaluations at the nodes $p_1,\ldots,p_n$. 
Notice that the $\sA$-action on $\overline{M}_0$ is trivial as $f(C_0)\subseteq F$. 

The same splitting curve argument as in case (a) shows that the tangent complex of $\overline{M}$ is related to the tangent complexes of $\overline{M}_i$ by the exact triangle
\begin{align}\label{tangent cplx splits_b}
    \mathbb T_{\overline{M}}\to \bigoplus_{i=0}^n\mathrm{pr}_i^*\mathbb T_{\overline{M}_i}\to \bigoplus_{j=1}^n\ev_{p_j}^* \mathbb T_{F}\to \mathbb T_{\overline{M}}[1],
\end{align}
where 
$\mathrm{pr}_{i}$ is the projection from $\overline{M}$ to $\overline{M}_{i}$. The virtual normal contribution $N^\vir_{\overline{M}/\overline{M}_{0,2}(X,\beta)}$ breaks into
\begin{align}\label{vir normal splits_b}
    N^\vir_{\overline{M}/\overline{M}_{0,2}(X,\beta)}\big|_{[f\colon C\to X]}=\sum_{i=0}^n N_i\big|_{[f|_{C_i}]}-\sum_{j=1}^n T_X|_{f(p_j)}^{\mathrm{moving}}+\sum_{j=1}^n T_{C_0,p_j}\otimes T_{C_j,p_j}\,, 
\end{align}
where $N_{i}\big|_{[f|_{C_{i}}]}=H^*(C_{i},f^*T_X)^{\mathrm{moving}}$. Let $\mathsf a_i$ denote the $\sA$-weight of $T_{C_j,p_j}$. 

% It follows from \eqref{tangent cplx splits_b} and \eqref{vir normal splits_b} that
% the contribution of $\overline{M}$ is of form
% \begin{equation}\label{split into factor_revised}
% e^{\sA}\left(N_{F/X}^+\right)\ev_*\left(\frac{[\overline{M}]^{\vir}}{e^{\sA}\left(N^{\vir}_{\overline{M}/\overline{M}_{0,2}(X,\beta)}\right)} \right)e^{\sA}\left(N_{F/X}^-\right)=\sum_{\gamma_0,\ldots,\gamma_n}\pr_{*}\Delta_{F^n}^*\left(\ev_*(\gamma_0\cdot [\overline{M}_0]^\vir)\boxtimes \boxtimes_{i=1}^n \ev_{i*}(\gamma_i\cdot [\overline{M}_i]^{\vir})\right),
% \end{equation}
% where $\Delta_{F^n}\colon F\times F\times F^n\hookrightarrow F\times F\times F^n\times F^n$ is given by the diagonal embedding $F^n\hookrightarrow F^n\times F^n$, $\pr\colon F\times F\times F^n\to F\times F$ is the projection to the first and second factor,
% $\ev=\ev_{v_1}\times \ev_{v_2}\times \prod_{i=1}^n\ev_{i}\colon \overline{M}_0\to F^{n+2}$, $\ev_i\colon \overline{M}_i \to F$ are evaluation maps, $\gamma_i$ are cohomology classes on $\overline{M}_i$ which come from equivariant Euler classes of $-\sum_{j=1}^n T_X|_{f(p_j)}^{\mathrm{moving}}+\sum_{j=1}^n T_{C_0,p_j}\otimes T_{C_j,p_j}$.

It follows from \eqref{tangent cplx splits_b} and \eqref{vir normal splits_b} that
the contribution of $\overline{M}$ is of form
\begin{equation}\label{split into factor}
\begin{split}
&\quad \,\,e^{\sA}\left(N_{F/X}^+\right)\ev_*\left(\frac{[\overline{M}]^{\vir}}{e^{\sA}\left(N^{\vir}_{\overline{M}/\overline{M}_{0,2}(X,\beta)}\right)} \right)e^{\sA}\left(N_{F/X}^-\right)\\
&=\sum_{\substack{r_1,\ldots,r_n\in \bZ_{\geqslant 0}\\s_1,\ldots,s_n\in \bZ_{\geqslant 0}}}\pr_{*}\Delta_{F^n}^*\left(\ev_*(\psi_1^{r_1}\cdots\psi_n^{r_n}\alpha_0)\boxtimes \left(\boxtimes_{i=1}^n e^\sA(N_{F/X})\cdot \ev_{i*}(\psi_i^{s_i}\cdot [\overline{M}_i]^{\vir})\right)\right)\prod_{i=1}^n\mathsf a_i^{-r_i-s_i-1}\binom{r_i+s_i}{r_i},
\end{split}
\end{equation}
where $\Delta_{F^n}\colon F\times F\times F^n\hookrightarrow F\times F\times F^n\times F^n$ is given by the diagonal embedding $F^n\hookrightarrow F^n\times F^n$, $\pr\colon F\times F\times F^n\to F\times F$ is the projection to the first and second factor,
$$\ev=\ev_{v_1}\times \ev_{v_2}\times \prod_{i=1}^n\ev_{i}\colon \overline{M}_0\to F\times_{X_0} F\times_{X_0}\cdots \times_{X_0} F, \quad 
\ev_i\colon \overline{M}_i \to F$$
are evaluation maps,  
\begin{align*}
\alpha_0=\frac{e^{\sA}\left(\ev_{v_1}^*N_{F/X}^+\right)e^{\sA}\left(\ev_{v_2}^*N_{F/X}^-\right)}{e^{\sA}\left(N_0\right)} [\overline{M}_0]^{\vir}\in H_*(\overline{M}_0)\otimes \Frac H_\sA(\pt),
\end{align*}
and $\psi_i$ is the $\psi$-class corresponding to $p_i$.

Notice that
\begin{align*}
\ev_*(\psi_1^{r_1}\cdots\psi_n^{r_n}\alpha_0)\in H_{\leqslant 2\cdot\mathrm{vdim}\overline{M}_0}(\underbrace{F\times_{X_0} F\times_{X_0}\cdots \times_{X_0} F}_{n+2\text{ copies}})\otimes \Frac H_\sA(\pt).
\end{align*}
On the other hand, the image of $\ev_{i}\colon \overline{M}_i\to F$ is contained in the closed subset $F\cap \pi^{-1}(B\setminus U)$ (notation is the same as in case (a)), because the restriction of $\pi\colon X\to B$ to $\pi^{-1}(U)$ is an affine morphism therefore the image of any stable map $f\colon C\to X$ is in $\pi^{-1}(B\setminus U)$. Therefore, 
\begin{align*}
e^\sA(N_{F/X})\cdot \ev_{i*}(\psi_i^{s_i}\cdot [\overline{M}_i]^{\vir})\in H_{\leqslant 2(\dim F-1)}(F\cap \pi^{-1}(B\setminus U))\otimes \Frac H_\sA(\pt).
\end{align*}
We claim that the image of $H_{2(\dim F-1)}(F\cap \pi^{-1}(B\setminus U))$ in $H_{2(\dim F-1)}(F)$ is zero. Since $F$ is a symmetric quiver variety, there is a flat morphism $\pi'\colon F\to B'$ to an affine space $B'$ with irreducible fibers by \cite[Thm.~1.1]{COZZ3}. By construction in \cite{COZZ3}, there is a finite morphism $p\colon B'\to B$ \footnote{$B=\prod_j\mathfrak{gl}_{n_j}/\!\!/\GL_{n_j}$, $B'=\prod_j\mathfrak{gl}_{n'_j}/\!\!/\GL_{n'_j}\times \mathfrak{gl}_{n_j''}/\!\!/\GL_{n_j''}$ with $n_j'+n_j''=n_j$. The map $p\colon B'\to B$ is induced by the diagonal embedding $\mathfrak{gl}_{n_j'}\times \mathfrak{gl}_{n_j''}\hookrightarrow \mathfrak{gl}_{n_j}$.} such that $p\circ \pi'=\pi|_F$. Let $U':=p^{-1}(U)$, and let $C_1,C_2,\ldots$ be codimension one (in $B'$) irreducible components of $B'\setminus U'$. Since $\pi'$ is flat with irreducible fibers, codimension one
(in $F$) irreducible components of $F\cap \pi'^{-1}(B'\setminus U')$ are exactly $\pi'^{-1}(C_1),\pi'^{-1}(C_2),\ldots$. Thus $H_{2(\dim F-1)}(F\cap \pi^{-1}(B\setminus U))$ is spanned by $[\pi'^{-1}(C_1)]=\pi'^*[C_1],[\pi'^{-1}(C_2)]=\pi'^*[C_2],\ldots$. Since $B'$ is an affine space, $H^2(B')=0$, so $[C_1]=[C_2]=\cdots=0$, and the claim follows. Then we have
\begin{align*}
e^\sA(N_{F/X})\cdot \ev_{i*}(\psi_i^{s_i}\cdot [\overline{M}_i]^{\vir})\in H_{< 2(\dim F-1)}(F)\otimes \Frac H_\sA(\pt).
\end{align*}
After taking convolution, we see that
\begin{align*}
\text{LHS of \eqref{split into factor}}\in H_{<2(\mathrm{vdim}\overline{M}_0-n)}(F\times_{X_0} F)\otimes \Frac H_\sA(\pt).
\end{align*}
Since $\mathrm{vdim}\overline{M}_0=\dim F+n-1$ and the homological degree of LHS of \eqref{split into factor} is $2\dim F-2$, the $\sA$-degree of LHS of \eqref{split into factor} is negative, and hence vanishes when $\mathsf a\to \infty$.
\end{proof}

\begin{Lemma}\label{lem on vani of delta cpn}
Let $\overline{M}_{\delta}\subseteq \overline{M}_{0,2}(X,\beta)^{\sA}$ be a connected component 
whose elements consist of curves for which both marked points lie
on a contracted component.
%Then it does not contribute to $\Stab_\fC^{-1}\, \sQ_\beta\, \Stab_\fC$.
Then it contributes zero to the left hand side of \eqref{equ on stabevstab}. 
\end{Lemma}
\begin{proof}
A stable map in $\overline{M}_{\delta}$ is of form $f\colon C=C_0\cup C_1\cup\cdots\cup C_n\to X$, where two marked points lie in $C_0$ which is contracted to a point in a connected component $F\subseteq X^\sA$. Here $F$ is also a symmetric quiver variety. Then 
$$\ev(\overline{M}_{\delta})\subseteq \Delta_F. $$
As mentioned above, there is a flat map $\pi\colon F\to B$ to an affine space such that generic fibers are affine, and 
$\sA$-action on the base is trivial. Hence 
there is a closed subscheme $S\subsetneq B$ such that 
\begin{equation}\ev(\overline{M}_{\delta})\subseteq \Delta_{ \pi^{-1}(S)}.  \nonumber \end{equation}
Therefore, we have 
\begin{equation}\label{equ on evm02}\dim(\ev(\overline{M}_{\delta}))\leqslant \dim F-1. \end{equation}
By virtual localization \cite{GP} with respect to $\sA$, 
$\Stab_\fC^{-1}\, \sQ_\beta\, \Stab_\fC$ receives contribution from component $F$ by 
\begin{equation}\label{equ on sandw of ev2}e^{\sA}\left(N_{F/X}^+\right)\ev_*\left(\frac{[\overline{M}_{\delta}]^{\vir}}{e^{\sA}\left(N^{\vir}_{\overline{M}_{\delta}/\overline{M}_{0,2}(X,\beta)}\right)} \right)e^{\sA}\left(N_{F/X}^-\right). \end{equation}
We may expand the middle term $\ev_*(-)$ into a sum 
$$\sum_{i}a_i[C_i], $$ 
where $a_i$'s are equivariant parameters and $[C_i]$'s are cycles in $\ev(\overline{M}_{\delta})$.  
\begin{itemize} 
\item If the power of $a_i$ is $>-(\dim X-\dim F)$, then $\dim C_i> \dim F-1$, making $[C_i]$ to be zero by \eqref{equ on evm02}. 
\item If the power of $a_i$ is $=-(\dim X-\dim F)$, then $\dim C_i= \dim F-1$, so $[C_i]$ is a linear combination of $(\dim F-1)$-dimensional 
cycles in $\Delta_{ \pi^{-1}(S)}$. Then $[C_i]=0$ as $A_{n}(B)=0$ for any $n\neq \dim B$.
\item If the power of $a_i$ is $<-(\dim X-\dim F)=-\deg_{\sA} (e^{\sA}(N_{F/X}))$, by taking equivariant parameters to infinity, 
we know it contributes zero to \eqref{equ on sandw of ev2}.
\end{itemize} 
To sum up, we conclude that \eqref{equ on sandw of ev2} vanishes in $H^\sA(X\times_{X_0}X)$.
% By Lemma \ref{lem on A vs T loc}, we are done.
\end{proof}

We note also that the same idea in the proof of Proposition \ref{prop on sym var van} implies the following analog of \cite[Prop.~7.5.1]{MO}.
\begin{Proposition}\label{prop on tens of Q}
Let $X=X_1\times X_2$ be the product of two symmetric quiver varieties. Then 
$$\sQ_{X}=\sQ_{X_1}\otimes 1+1\otimes \sQ_{X_2}. $$
\end{Proposition}

\subsection{Formula on symmetric quiver varieties}
%Let $X$ be a symmetric quiver variety (Definition \ref{def of sym qv}) with a torus $\sT$-action, and a $\sT$-invariant function $\sw$. Let 
%$\sA\subset \sT$ be a self-dual subtorus. 

Given a symmetric quiver $Q$ with potential $\sW$ preserved by a torus $\sT_0$, recall that there is a Lie superalgebra $\mathfrak{g}_{Q,\sW}$ 
(Definition \ref{def of cartan sub}) with a decomposition 
$$\mathfrak{g}_{Q,\sW}=\bigoplus_{\alpha\in \bZ^{Q_0}} \mathfrak{g}_{\alpha}. $$
Let $\{e_{\alpha}^{(s)}\}$ denote the basis in the root space $\mathfrak{g}_{\alpha}$ for each positive root $\alpha$ that appearing in Proposition \ref{prop g(Q,w)}.

By choosing a gauge dimension vector $\bv$ and a framing dimension vector $\bd$, we have a symmetric quiver variety 
$$X=\cM(\bv,\bd). $$ 
We take a $\sT_0$ action on $\cM(\bv,\bd)$ extending the $\sT_0$ action on the unframed quiver.
Let $\sA^{\mathrm{fr}}$ act on $\cM(\bv,\bd)$ as a framing torus, and $\sW^{\mathrm{fr}}$ be a $\sT=\sT_0\times \sA^{\mathrm{fr}}$ invariant framed potential. Denote $\sw^{\mathrm{fr}}=\tr(\sW^{\mathrm{fr}})$. 
Then $\mathsf Y_0(Q,\sW)$ (and hence $\mathfrak{g}_{Q,\sW}$) acts on $H^\sT(X,\sw^{\mathrm{fr}})$.
For $\lambda\in \bC^{Q_0}$, we define 
$$c_1(\lambda)=\sum_{i\in Q_0}\lambda_{i} \:c_1(\mathsf V_i), $$ 
where $\mathsf V_i$ is the $i$-th tautological bundle of $X$. 
\begin{Definition}\label{def modified quant mult}
We define the \textit{modified quantum multiplication} 
%$\gamma\widetilde{\star}_{\beta%} \cdot$，
$\gamma\, \widetilde{\star}\,\cdot$ by substitution
\begin{align*}
    z^\beta\mapsto (-1)^{(c_1(\mathsf P),\beta)+|\bv|\cdot |\beta|} z^\beta \;\text{ in } \eqref{equ on Mgamma as map} \text{ for }\text{each}\,\,\beta\,\, \text{on}\,\,\cM(\bv,\bd)\,,
\end{align*}
where $\mathsf P$ is given by \eqref{par pol on M(v,d)}, $|\bv|$ is given by \eqref{equ on gra}, and $|\beta|:=\sum_{i\in Q_0^{\mathrm{f}}}\beta_i\pmod 2$
is the parity of $\beta$ with $\beta_i=(c_1(\mathsf V_i),\beta)$. 
\end{Definition}

\begin{Theorem}\label{thm on qm div for sym}
%Let $X$ be a symmetric quiver variety with a torus $\sT$-action and a $\sT$-invariant function $\sw$.
Notations as above, then for the zero potential ($\sW^{\mathrm{fr}}=0$), we have
$$c_1(\lambda)\,\widetilde{\star}\, \cdot =c_1(\lambda)\cup \cdot -\sum_{\alpha>0}\alpha(\lambda)\frac{z^\alpha}{(-1)^{|\alpha|}-z^\alpha}\mathrm{Cas}_\alpha+\cdots \,,$$
where $\mathrm{Cas}_\alpha=\sum_{s} e_{\alpha}^{(s)} e_{-\alpha}^{(s)}$ is the Casimir operator for $\mathfrak{g}_{Q,\sW=0}$, $|\alpha|$ is given as \eqref{equ on gra}, and dots denote a scalar operator.
\end{Theorem}

The scalar operator is fixed by the requirement that the purely quantum
part of $c_1(\lambda)\,\widetilde{\star}\, \cdot$ annihilates the identity.

\begin{proof}
This is parallel to the proof of \cite[Thm.~10.2.1]{MO}, which we sketch as follows.  
By \eqref{coproduct for 1st Chern class}, we have 
\begin{equation}\label{equ on del c1}
\Delta c_1(\lambda)=c_1(\lambda)\otimes 1+1\otimes c_1(\lambda)-\sum_{\alpha>0} (-1)^{|\alpha|}\alpha(\lambda)\sum_s\: e_{-\alpha}^{(s)}\otimes e_{\alpha}^{(s)}.
\end{equation}
Combining with Remark \ref{rmk opposite coproduct}, we have 
\begin{equation}\label{equ on sand del c1}
R^{\mathrm{sup}}(u)\Delta c_1(\lambda) R^{\mathrm{sup}}(u)^{-1}=\Delta^{\mathrm{op}} c_1(\lambda)=\Delta c_1(\lambda)+
\sum_{\alpha>0}\alpha(\lambda)\sum_s \left((-1)^{|\alpha|}e_{-\alpha}^{(s)}\otimes e_{\alpha}^{(s)}-e_{\alpha}^{(s)}\otimes e_{-\alpha}^{(s)} \right).
\end{equation}
We denote the modified quantum multiplication operator by 
$$\widetilde{\sQ}(\lambda):=c_1(\lambda)+\sum_{\beta>0}(-1)^{(c_1(\mathsf P),\beta)+|\bv|\cdot |\beta|}\lambda(\beta)\,\sQ_{\beta}\,z^\beta. $$ 
%We claim that 
%$$Q(\lambda)=c_1(\lambda)-\sum_{\alpha>0}(-1)^{|\alpha|} \alpha(\lambda)\frac{z^\alpha}{1-z^\alpha}\sum_{s} e_{\alpha}^{(s)} e_{-\alpha}^{(s)}+\cdots. $$
This acts on critical cohomology for any $\cM(\bv,\bd)$ and any extended potential $\sw^{\mathrm{fr}}$. 
We consider the pullback 
\begin{equation}\label{equ on deltatiltQ}\Delta \widetilde{\sQ}=\Stab_{+,\epsilon}^{-1}\widetilde{\sQ}\Stab_{+,\epsilon} \end{equation} 
of $\sQ$ under the stable envelope map 
\begin{align*}
H^{\sT}(\cM(\bd'))\otimes H^{\sT}(\cM(\bd''))\to H^{\sT}(\cM(\bd'+\bd'')). 
\end{align*}
% $$H^{\sT}(\cM(\bd'),\sw^{\mathrm{fr}})\otimes H^{\sT}(\cM(\bd''),\sw^{\mathrm{fr}})\to H^{\sT}(\cM(\bd'+\bd''),\sw^{\mathrm{fr}}). $$
Consider a decomposition 
$$\Delta \widetilde{\sQ}=\sum_{\alpha} \Delta_\alpha \widetilde{\sQ}, $$
according to the weight of the Cartan action
$$[h\otimes 1,\Delta_\alpha \widetilde{\sQ}]=\alpha(h) \Delta_\alpha \widetilde{\sQ}. $$
By \eqref{equ on del c1}, Theorem \ref{thm on stab and Q} and Proposition \ref{prop on tens of Q}, we have \footnote{The modification $z^\beta\mapsto (-1)^{(c_1(\mathsf P),\beta)+|\bv|\cdot |\beta|} z^\beta$ is introduced to cancel the sign in \eqref{equ on comm diag on stab and Q}.}
\begin{equation}\label{equ on Delta0Q}\Delta_0 \widetilde{\sQ}=\widetilde{\sQ}\otimes 1+1\otimes \widetilde{\sQ}+\text{a scalar operator} \end{equation}
By an analog of \cite[Thm.~9.3.1]{MO}, which uses only the symmetric property of quiver varieties,   
$\Stab_{+,\epsilon}\,(\tilde z^{\mathsf{v}}\otimes 1)\,\Stab_{-,\bar\epsilon}^{-1}$ equals to the minuscule shift operator $\mathsf S(\sigma)$, where
$\tilde z^{\mathsf{v}}$ is the diagonal operator that acts on $H^{\sT}(\cM(\bv,\bd),\sw^{\mathrm{fr}})$ by the scalar $(-1)^{|\bv|}\cdot z^{\bv}$. By a standard argument \cite[Prop.~8.2.1, Lem.~8.3.1]{MO} which holds for any smooth $X$ with an action $\sigma\colon \bC^*\to \Aut(X)$, shift operator $\mathsf S(\sigma)$ commutes with modified quantum multiplications; therefore 
$$[(\tilde z^{\mathsf{v}}\otimes 1)\,R^{\mathrm{sup}}(u),\Delta \widetilde{\sQ}]=0,\quad \mathrm{i.e.}\,\,\, R^{\mathrm{sup}}(u)\,\Delta \widetilde{\sQ} \,R^{\mathrm{sup}}(u)^{-1}=\sum_{\alpha}(-1)^{|\alpha|}z^{-\alpha} \Delta_{\alpha} \widetilde{\sQ}. $$
We also have 
\begin{align*}
R^{\mathrm{sup}}(u)\,\Delta \widetilde{\sQ} \,R^{\mathrm{sup}}(u)^{-1}&=R^{\mathrm{sup}}(u)\,\Delta c_1(\lambda) \,R^{\mathrm{sup}}(u)^{-1}+
\sum_{\beta>0}(-1)^{(c_1(\mathsf P),\beta)+|\bv|\cdot |\beta|}\lambda(\beta)z^\beta\, R^{\mathrm{sup}}(u)\,\Delta \sQ_{\beta}\,R^{\mathrm{sup}}(u)^{-1} \\
&=\Delta^{\mathrm{op}} c_1(\lambda)+\sum_{\beta>0}(-1)^{(c_1(\mathsf P),\beta)+|\bv|\cdot |\beta|}\lambda(\beta)z^\beta\,\Delta \sQ_{\beta} \\
&=\Delta \widetilde{\sQ} +(\Delta^{\mathrm{op}}-\Delta)(c_1(\lambda)).
\end{align*}
where in the second equality, we use \eqref{equ on sand del c1}, Lemma \ref{lem 2-pt Steinberg}, Remark \ref{rmk steinberg comm R} and \eqref{equ on deltatiltQ}.
This implies that 
\begin{align}\label{equ on ind equ}
\sum_{\alpha}(1-(-1)^{|\alpha|}z^{-\alpha})\Delta_{\alpha} \widetilde{\sQ}
=\Delta c_1(\lambda)- \Delta^{\mathrm{op}} c_1(\lambda) =\sum_{\alpha>0}\alpha(\lambda)\sum_s \left(e_{\alpha}^{(s)}\otimes e_{-\alpha}^{(s)}-(-1)^{|\alpha|}e_{-\alpha}^{(s)}\otimes e_{\alpha}^{(s)} \right),
\end{align}
where the last equality uses \eqref{equ on sand del c1}. 
Equation \eqref{equ on ind equ} uniquely determines $\Delta_{\alpha}\widetilde{\sQ}$ for $\alpha\neq 0$, given by 
\begin{equation}\label{equ on deltaalpha}\Delta_{\alpha}\widetilde{\sQ}=
\begin{cases}
-\alpha(\lambda)\frac{z^\alpha}{(-1)^{|\alpha|}-z^\alpha}\sum_{s} e_{\alpha}^{(s)}\otimes e_{-\alpha}^{(s)}, \quad \mathrm{if}\,\, \alpha>0\\
\alpha(\lambda)\frac{1}{(-1)^{|\alpha|}-z^{-\alpha}}\sum_{s} e_{\alpha}^{(s)}\otimes e_{-\alpha}^{(s)}, \quad \mathrm{if}\,\, \alpha< 0
\end{cases} \end{equation}
We define 
$$ \widetilde{\sQ}_{\mathrm{remainder}}:= \widetilde{\sQ}-\left( c_1(\lambda)-\sum_{\alpha>0} \alpha(\lambda)\frac{z^\alpha}{(-1)^{|\alpha|}-z^\alpha}\sum_{s} 
e_{\alpha}^{(s)} e_{-\alpha}^{(s)}\right). $$
By \eqref{equ on del c1}, \eqref{equ on Delta0Q}, \eqref{equ on deltaalpha}, 
we have 
\begin{equation}\label{equ on deltaQremind}\Delta \widetilde{\sQ}_{\mathrm{remainder}}=\widetilde{\sQ}_{\mathrm{remainder}}\otimes 1+1\otimes \widetilde{\sQ}_{\mathrm{remainder}}+\text{a scalar operator}. \end{equation}
$\Delta\widetilde{\sQ}_{\mathrm{remainder}}$ preserves the weight space decomposition
\begin{align*}
    \bigoplus_{\bv',\bv''}H^{\sT}(\cM(\bv',\bd'))\otimes H^{\sT}(\cM(\bv'',\bd''))\,,
\end{align*}
so
\begin{align*}
\left[|\mathbf 0\rangle\langle\mathbf 0|\otimes \id\,,\, \Delta\widetilde{\sQ}_{\mathrm{remainder}}\right]=0\,.
\end{align*}
By Lemma \ref{lem 2-pt Steinberg}, Proposition \ref{prop Casimir}, $\widetilde{\sQ}_{\mathrm{remainder}}$ is induced by a Steinberg correspondence, so by Remark \ref{rmk steinberg comm R}, we have
\begin{align*}
\left[R^{\mathrm{sup}}_{\bd',\bd''}\,,\, \Delta\widetilde{\sQ}_{\mathrm{remainder}}\right]=0.
\end{align*}
Combining the above two equations and \eqref{equ on deltaQremind}, we get
\begin{align*}
\left[\tr_{\cH^{\sW=0}_{\bd'}}^{\bZ/2}\left((|\mathbf 0\rangle\langle\mathbf 0|\otimes \id)R^{\mathrm{sup}}_{\bd',\bd''}(u)\right)\,,\, \widetilde{\sQ}_{\mathrm{remainder}}\right]=0\,.
\end{align*}
The computation in Lemma \ref{lem Y^0 generators} shows that $\widetilde{\sQ}_{\mathrm{remainder}}$ commutes with all classical multiplications, so $\widetilde{\sQ}_{\mathrm{remainder}}$ itself is a classical multiplication. Since $\widetilde{\sQ}_{\mathrm{remainder}}$ has cohomological degree $2$, it is of form $c_1(\mu)\cup\cdot+a\cdot$ for some $\mu\in \bZ^{Q_0}$ and $a\in H^2_\sT(\pt)$. Since $\Delta_\alpha\widetilde{\sQ}_{\mathrm{remainder}}=0$ for nonzero $\alpha$, we read from the coproduct formula \eqref{coproduct for 1st Chern class} that $\mu$ must be zero. Thus $\widetilde{\sQ}_{\mathrm{remainder}}$ is a scalar multiplication.
% Let us consider the zero potential case, then an analog of \cite[Prop.~10.2.2]{MO} implies that $\widetilde{\sQ}_{\mathrm{remainder}}$ as an operator acting on $H^\sT(\cM(\bv,\bd))$ is a scalar multiplication. Denote this scalar by $q$, then $\widetilde{\sQ}_{\mathrm{remainder}}-q[\Delta_{\cM(\bv,\bd)}]$ is a class in $H^\sT_{2d-2}(\cM(\bv,\bd)\times_{\cM_0(\bv,\bd)}\cM(\bv,\bd))$ with $d=\dim \cM(\bv,\bd)$, such that the convolution map induced by $\widetilde{\sQ}_{\mathrm{remainder}}-q[\Delta_{\cM(\bv,\bd)}]$ on $H^\sT(\cM(\bv,\bd))$ is zero. Then by Lemma \ref{op vs corr} \yh{False statement! Revise}, $\widetilde{\sQ}_{\mathrm{remainder}}=q[\Delta_{\cM(\bv,\bd)}]$, so $\widetilde{\sQ}_{\mathrm{remainder}}$ acts on critical cohomology $H^\sT(\cM(\bv,\bd),\sw^{\mathrm{fr}})$ is the scalar multiplication by $q$.
\end{proof}

\begin{Corollary}\label{cor qm div for sym_sp inj}
Suppose that the specialization map $\mathrm{sp}\colon H^\sT(\cM(\bv,\bd),\sw^{\mathrm{fr}})\to H^\sT(\cM(\bv,\bd))$ is injective, then the modified quantum multiplication by divisor $\lambda$ on $H^\sT(\cM(\bv,\bd),\sw^{\mathrm{fr}})$ is given by
$$c_1(\lambda)\,\widetilde{\star}\, \cdot =c_1(\lambda)\cup \cdot -\sum_{\alpha>0}\alpha(\lambda)\frac{z^\alpha}{(-1)^{|\alpha|}-z^\alpha}\mathrm{Cas}_\alpha+\cdots \,,$$
where $\mathrm{Cas}_\alpha=\sum_{s} e_{\alpha}^{(s)} e_{-\alpha}^{(s)}$ is the Casimir operator for $\mathfrak{g}_{Q,\sW}$, and dots denote a scalar operator.
\end{Corollary}

\begin{proof}
Since $\sQ_\beta$ is induced by a correspondence, $c_1(\lambda)\,\widetilde{\star}\,\cdot$ commutes with specialization \cite[Prop.~7.17]{COZZ}. By Proposition \ref{prop Casimir}, there exists a Steinberg correspondence $\mathsf C_\alpha$ that induces the operator  $\mathrm{Cas}_\alpha$ for $H^\sT(\cM(\bv,\bd),\sw^{\mathrm{fr}})$ and for $H^\sT(\cM(\bv,\bd))$ simultaneously. In particular, the restriction of Casimir operator for $\mathfrak{g}_{Q,0}$ to the image of specialization map gives the Casimir operator for $\mathfrak{g}_{Q,\sW}$. Then the corollary follows from Theorem \ref{thm on qm div for sym}.
\end{proof}

% \yh{The following statement is incorrect. Counter example: resolved conifold $Y=\mathrm{Tot}_{\bP^1}(\mathcal O(-1)^{\oplus 2})$.}

% \begin{Lemma}\label{op vs corr}
% Suppose that $Y$ is a symmetric quiver variety with a $\sT$-action, and let $Y\to Y_0$ be the affinization map. Let $\gamma$ be a class in $H^\sT_{2\dim Y-2}(Y\times_{Y_0}Y)$ such that the convolution map $[\gamma]\colon H^\sT(Y)\to H^\sT(Y)$ is zero, then $\gamma=0$.
% \end{Lemma}

\subsection{Asymmetric cases}

For a not necessarily symmetric quiver variety, one may associate certain symmetrization of it and use Theorem \ref{thm on qm div for sym} and the 
following lemma to deduce the formula of quantum multiplication by divisors. 

\begin{Proposition}\label{prop on cmp corr}
Let $X$ be a quiver variety \eqref{equ on qv} with affinization $X\to X_0$, $p\colon Y=\mathrm{Tot}_X(E)\to X$ be the total space of a $\sT$-equivariant 
vector bundle $E$ on $X$. Let $[Z]\in H^{\sT}(X\times X, \sw \boxminus \sw)_{X\times_{X_0}X}$ and denote  
$[Z]_{X}\colon H^{\sT}(X,\sw)\to H^{\sT}(X,\sw)$ to be the convolution induced by $[Z]$, and 
$[Z]_{Y}\colon H^{\sT}(Y,p^*\sw)\to H^{\sT}(Y,p^*\sw)$ to be the convolution induced by the image of $[Z]$ under the pushforward by the zero section
$X\times X\to Y\times Y$. 

Then we have a commutative diagram 
\begin{equation}\label{equ on comm diag on corr}
\xymatrix{
H^{\sT}(X,\sw)  \ar[rr]^{p^* \,\,\,}_{\cong \quad  } \ar[d]_{[Z]_{X}} &  & H^{\sT}(Y,p^*\sw) \ar[d]^{[Z]_{Y}}  \\
H^{\sT}(X,\sw) \ar[rr]^{p^*(e(E)\cdot (-)) \,\,\,}_{\cong \quad  }  & & H^{\sT}(Y,p^*\sw),
} 
\end{equation}
where $e(E)\colon H^{\sT}(X,\sw)\to H^{\sT}(X,\sw)$ is the Euler class operator, see e.g.~\cite[Eqn.~(2.10)]{COZZ}. 
\end{Proposition}
\begin{proof}
We have the following commutative diagrams 
$$
\xymatrix{
Y\times Y  \ar[r]^{ } \ar@/^1.2pc/[rr]^{\widetilde{\pr}_2}  \ar@/_2pc/[dd]_{\widetilde{\pr}_1}  \ar[d]_{ } & X\times Y   \ar[d]^{ }  \ar[r]^{ }  & Y  \ar[d]^{p}  \\
Y\times X \ar[r]^{ }  \ar[d]_{ }  & X\times X  \ar[r]^{\quad \pr_2} \ar[d]^{\pr_1 }  & X \\
Y  \ar[r]^{p}   & X
} 
$$
where all squares are Cartesian. Denote $0\times 0\colon X\times X\to Y\times Y$ to be the zero section. 

%Recall the product $\otimes$ \eqref{equ on otimes k theory}. 
By base change, for any $\alpha\in K^{\sT}(X,\sw)$, we have 
$$((0\times 0)_*[F])\otimes (\widetilde{\pr}_2^*\, p^*\alpha)=(0\times 0)_*([F]\otimes \pr_2^*\alpha). $$
Therefore, we have 
$$\widetilde{\pr}_{1*}\left((0\times 0)_*[F]\otimes (\widetilde{\pr}_2^*\, p^*\alpha)\right)=\widetilde{\pr}_{1*}(0\times 0)_*([F]\otimes \pr_2^*\alpha)=0_*\pr_{1*}([F]\otimes \pr_2^*\alpha). $$
By the definition of convolution, the above is exactly 
$$[F]_Y\circ (p^*\alpha)=0_*([F]_X\circ \alpha)\in K^{\sT}(Y,p^*\sw). $$
Applying the Gysin pullback $0^*$, we obtain 
$$0^*([F]_Y\circ (p^*\alpha))=e_K(E)\cdot ([F]_X\circ \alpha)  $$
Since $0^*$ is the inverse of the smooth pullback $p^*$, we obtain the claim: 
\begin{equation*} [F]_Y\circ (p^*\alpha)=p^*(e_K(E)\cdot ([F]_X\circ \alpha)). \qedhere \end{equation*}
\end{proof}

\begin{Theorem}\label{thm qm div for asym_sp inj}
Take a state space $H^{\sT}(\cM(\bv,\underline{\bd}),\sw^{\mathrm{fr}})$ with $\bd_{\Out}-\bd_{\In}=\mu\in \bZ_{\leqslant 0}^{Q_0}$, where $\sT=\sT_0\times \sA^{\mathrm{fr}}$. Suppose that the specialization map $\mathrm{sp}\colon H^\sT(\cM(\bv,\underline{\bd}),\sw^{\mathrm{fr}})\to H^\sT(\cM(\bv,\underline{\bd}))$ is injective, then the modified quantum multiplication by divisor $\lambda$ on $H^{\sT}(\cM(\bv,\underline{\bd}),\sw^{\mathrm{fr}})$ is given by
$$c_1(\lambda)\,\widetilde{\star}\,\cdot =c_1(\lambda)\cup \cdot -\sum_{\substack{\alpha>0\\ \alpha\cdot \mu=0}}\alpha(\lambda)\frac{z^\alpha}{(-1)^{|\alpha|}-z^\alpha}\mathrm{Cas}_{\alpha,\mu}-\sum_{\substack{\alpha>0\\ \alpha\cdot \mu\neq 0}}(-1)^{|\alpha|}\alpha(\lambda)\,z^\alpha\,\mathrm{Cas}_{\alpha,\mu}+\cdots \,,$$
where $\mathrm{Cas}_{\alpha,\mu}$ is the shifted Casimir operator in $\mathsf Y_{\mu}(Q,\sW)$, and dots denote a scalar operator.
\end{Theorem}

\begin{proof}
It is enough to prove the zero potential case and the nonzero potential follows from an argument similar to Corollary \ref{cor qm div for sym_sp inj} (with Proposition \ref{prop Casimir} replaced by Proposition \ref{prop shifted Casimir}).
Let $\underline{\bd}'$ be a symmetric framing vector such that $\bd'_{\In}=\bd'_{\Out}=\bd_{\In}$. Denote $X=\cM(\bv,\underline{\bd})$, $Y=\cM(\bv,\underline{\bd}')$, then $Y$ is the total space of the vector bundle $E=\bigoplus_{i\in Q_0}\Hom(\mathsf V_i,\bC^{\bd_{\In,i}-\bd_{\Out,i}})$ on $X$. Let the torus $\bC^*_\hbar$ act on $Y$ by fixing the base $X$ and scaling the fibers of $E$ with weight $1$.

By the definition of torus equivariant virtual classes, we have 
\begin{equation}\label{equ on m02X_1}
[\overline{M}_{0,2}(X,\beta)]_{\sT}^{\vir}=[\overline{M}_{0,2}(X,\beta)]_{\sT\times \bC^*_\hbar}^{\vir}, \end{equation}
\begin{align}\label{equ on m02xy_1}
[\overline{M}_{0,2}(Y,\beta)]_{\sT\times \bC^*_\hbar}^{\vir}&=i_*([\overline{M}_{0,2}(X,\beta)]_{\sT\times \bC^*_\hbar}^{\vir}\cdot e(-\pi_*\ev^*E)) \\ \nonumber 
&=i_*[\overline{M}_{0,2}(X,\beta)]_{\sT\times \bC^*_\hbar}^{\vir}\cdot (\hbar^{-\rk E-\beta\cdot\mu}+\mathrm{l.o.t.}), \end{align}
where $i\colon \overline{M}_{0,2}(X,\beta) \stackrel{\cong}{\to}   \overline{M}_{0,2}(Y,\beta)^{\bC^*_\hbar}$ is induced by the zero section $X\hookrightarrow Y$, 
$\pi\colon \mathcal{C}\to \overline{M}_{0,2}(X,\beta)$ is the universal curve and $\ev\colon  \mathcal{C}\to X$ is the evaluation map.
Denote 
$$[Z]_X:=\ev_*([\overline{M}_{0,2}(X,\beta)]_{\sT\times \bC^*_\hbar}^{\vir}), \,\,\,  
[Z]_Y:=i_*\circ \ev_*([\overline{M}_{0,2}(X,\beta)]_{\sT\times \bC^*_\hbar}^{\vir}), 
 \,\,\, [Z]_Y':= \ev_*([\overline{M}_{0,2}(Y,\beta)]_{\sT\times \bC^*_\hbar}^{\vir}). $$ 
By Proposition \ref{prop on cmp corr}, we have 
\begin{align*}[Z]_X&=\frac{(p^{*})^{-1}\circ [Z]_Y\circ  p^*(-)}{e(E)}  \\
&=\frac{1}{(\hbar^{\rk E}+\mathrm{l.o.t.})(\hbar^{-\rk E-\beta\cdot\mu}+\mathrm{l.o.t.})} \left((p^{*})^{-1}\circ [Z]_Y'\circ  p^*(-)\right) \\ 
&=-(-1)^{(c_1(\mathsf P_Y),\beta)+|\bv|\cdot |\beta|} \sum_{\substack{\alpha>0\\ \alpha |\beta}}\frac{\alpha(\lambda)\,(-1)^{|\beta|}}{\hbar^{-\beta\cdot\mu}} \left((p^{*})^{-1}\circ \mathrm{Cas}_{\alpha,0}\circ  p^*(-)\right)+\mathrm{l.o.t.}+\text{scalar operator}.
\end{align*}
Here the second equality uses \eqref{equ on m02xy_1}, the third uses Theorem \ref{thm on qm div for sym}. By \eqref{equ on m02X_1}, $[Z]_X$ is independent of the choice of $\hbar$ , so we can take the limit $\hbar\to \infty$. Proposition \ref{prop shifted Casimir} implies that
\begin{align*}
\lim_{\hbar\to \infty}\frac{1}{(-\hbar)^{-\beta\cdot\mu}} \left((p^{*})^{-1}\circ \mathrm{Cas}_{\alpha,0}\circ  p^*(-)\right)=
\begin{cases}
\mathrm{Cas}_{\alpha,\mu}\,, & \text{if $\beta\cdot\mu=0$},\\
\mathrm{Cas}_{\alpha,\mu}\,, & \text{if $\beta\cdot\mu\neq 0$, $\alpha=\beta$}, \\
\,\,\,\,\,0\,, & \text{otherwise},
\end{cases}
\end{align*}
where $\mathrm{Cas}_{\alpha,\mu}$ is the shifted Casimir operator (see Definition \ref{def of shifted Casimir}).
The result then follows.
\end{proof}

We apply Theorem \ref{thm qm div for asym_sp inj} to compute the formula of quantum multiplication by divisors for $\Hilb(\C^3)$ in Section \ref{sect on qm by div on hilb}.
Corresponding $K$-theoretic analog will be pursued in a future work \cite{COZZ2}. 

\section{Bilinear form and anti-automorphism}

In this section, under certain properness assumption on the critical loci, we discuss an bilinear form on the shifted Yangian $\mathsf Y_{\mu}(Q,\sW)$ modules, and the induced anti-automorphism on the shifted Yangian (Theorem \ref{thm anti-auto RTT Yangian}). These are geometric analogues of the Shapovalov form and the Chevalley anti-involution. We keep using notations from \S \ref{sec shifted Yangian}.

\subsection{Properness of torus fixed loci}

\begin{Definition}
We say that an auxiliary datum $\mathfrak{c}=(\bd,\sT_0\curvearrowright\cM(\bd),\sW^{\mathrm{fr}})$ is \textit{compact} if $\Crit_{\cM(\bv,\bd)}(\sw^{\mathrm{fr}})^{\sT_0}$ is a proper variety for arbitrary $\bv$. In this case, we also call the auxiliary space $\cH_{\mathfrak{c}}$ \textit{compact}. 
An auxiliary data set $\mathcal C$ is \textit{compact} if every auxiliary datum in $\mathcal C$ is compact. Similarly, a state space $\cH^{\sW^{\mathrm{fr}}}_{\underline{\bd},\sA^{\mathrm{fr}}}=H^{\sT_0\times \sA^{\mathrm{fr}}}(\cM(\underline{\bd}),\sw^{\mathrm{fr}})$ is \textit{compact}  if $\Crit_{\cM(\bv,\underline{\bd})}(\sw^{\mathrm{fr}})^{\sT_0\times \sA^{\mathrm{fr}}}$ is a proper variety for arbitrary $\bv$.
\end{Definition}

The following lemma gives a criterion for abundant compact auxiliary data as well as state spaces.

\begin{Lemma}\label{lem abundant compact}
Suppose that $\left(\Crit_{\mathrm{Rep}(\bv,\mathbf 0)}(\sw)/\!\!/ \GL(\bv)\right)^{\sT_0}$ is proper (equivalent to finite), then $\Crit_{\cM(\bv,\underline{\bd})}(\sw^{\mathrm{fr}})^{\sT_0\times \sA^{\mathrm{fr}}}$ is proper if the following genericity condition is satisfied:
\begin{itemize}
    \item[] for every pair of $\sT_0\times \sA^{\mathrm{fr}}$ weights $\lambda$ on $\mathsf D_{\In}$ and $\lambda'$ on $\mathsf D_{\Out}$, the difference $\lambda-\lambda'\notin \mathrm{Char}(\sT_0)$.
\end{itemize}
% all $\sT_0\times \sA^{\mathrm{fr}}$ weights on out-going framing $\mathsf D_{\Out}$ are in the space $\mathrm{Char}(\sT_0\times \sA^{\mathrm{fr}})\otimes_{\bZ}\bQ\setminus \mathrm{Char}(\sT_0)$. Here we assume that $\sT_0$-weights on the unframed quiver are integral, and $\mathrm{Char}(\sT_0)$ is a subspace of $\mathrm{Char}(\sT_0\times \sA^{\mathrm{fr}})$ via the projection $\sT_0\times \sA^{\mathrm{fr}}\to \sT_0$.
\end{Lemma}

\begin{proof}
$\cM(\bv,\underline{\bd})^{\sT_0\times \sA^{\mathrm{fr}}}$ is the disjoint union of quiver varieties 
\begin{align*}
    \bigsqcup_{\phi/\sim}\cM(Q',\bv_{\phi},\underline{\bd}_{\phi}),
\end{align*}
for a certain quiver $Q'$ which has a natural $\mathrm{Char}(\sT_0\times \sA^{\mathrm{fr}})$-grading on its nodes, and $\phi\colon \sT_0\times \sA^{\mathrm{fr}}\to \GL(\bv)$ is a group homomorphism. Since the stability is cyclic, nonempty $\cM(Q',\bv_{\phi},\underline{\bd}_{\phi})$ must satisfy that the support of $\bv_{\phi}$ lies in the subspace $\left\{\lambda+\mathrm{Char}(\sT_0)\:|\:\lambda\text{ is a weight of $\mathsf D_{\In}$}\right\}$. Then the genericity condition implies that the tail of any out-going arrow on a nonempty $\cM(Q',\bv_{\phi},\underline{\bd}_{\phi})$ must have gauge dimension zero. It follows that $\cM(\bv,\underline{\bd})^{\sT_0\times \sA^{\mathrm{fr}}}\subseteq \cM(\bv,\underline{\bd}')^{\sT_0}$ where $\bd'_{\In}=\bd_{\In}$, $\bd'_{\Out}=\mathbf 0$. From this we deduce that $$\Crit_{\cM(\bv,\underline{\bd})}(\sw^{\mathrm{fr}})^{\sT_0\times \sA^{\mathrm{fr}}}\subseteq \Crit_{\cM(\bv,\underline{\bd}')}(\sw)^{\sT_0}.$$
Since the Jordan-H\"older map $\Crit_{\cM(\bv,\underline{\bd}')}(\sw)\to \Crit_{\mathrm{Rep}(\bv,\underline{\bd}')}(\sw)/\!\!/ \GL(\bv)$ is projective, it is enough to show that the latter has finite $\sT_0$ fixed points. Note that $\bC\left[\Crit_{\mathrm{Rep}(\bv,\underline{\bd}')}(\sw)\right]^{\GL(\bv)}=\bC\left[\Crit_{\mathrm{Rep}(\bv,\mathbf 0)}(\sw)\right]^{\GL(\bv)}$ because there is no out-going arrow to form cycle. Thus $\Crit_{\mathrm{Rep}(\bv,\underline{\bd}')}(\sw)/\!\!/ \GL(\bv)$ is isomorphic to $ \Crit_{\mathrm{Rep}(\bv,\mathbf 0)}(\sw)/\!\!/ \GL(\bv)$ which has finite $\sT_0$ fixed points by assumption.
\end{proof}

\begin{Remark}\label{rmk fractional weights}
Let us assume that $\sT_0$-weights on the unframed quiver are integral, and we can take $\sT_0\times \sA^{\mathrm{fr}}$ weights on $\mathsf D_{\In}$, $\mathsf D_{\Out}$ to be fractional, then the genericity condition in Lemma \ref{lem abundant compact} is $\lambda-\lambda'\in \mathrm{Char}(\sT_0\times \sA^{\mathrm{fr}})\otimes_{\bZ}\bQ\setminus \mathrm{Char}(\sT_0)$. In particular, $\lambda-\lambda'\in \mathrm{Char}(\sT_0)\otimes_{\bZ}\bQ\setminus \mathrm{Char}(\sT_0)$ is allowed.
\end{Remark}

\begin{Example}
Let $\widetilde{Q}$ be the tripled quiver of $Q$, $\widetilde{\sW}$ be the canonical cubic potential, $\sT_0=\bC^*_\hbar$ be the torus that acts on $\widetilde{Q}$ as in Example \ref{ex doubled vs tripled}. Denote $\widetilde{\sw}=\tr(\widetilde{\sW})$. Then $\left(\Crit_{\mathrm{Rep}(\widetilde{Q},\bv,\mathbf 0)}(\widetilde{\sw})/\!\!/ \GL(\bv)\right)^{\bC^*_\hbar}$ consists of a unique point $0$ corresponding to the trivial representation. In fact, by direct computation, we have $\Crit_{\mathrm{Rep}(\widetilde{Q},\bv,\mathbf 0)}(\widetilde{\sw})=\{M\in \mathrm{Rep}(\Pi_{\overline{Q}},\bv), \cE\in \End_{\mathrm{Rep}(\Pi_{\overline{Q}},\bv)}(M)\}$, where $\Pi_{\overline{Q}}$ is the preprojective algebra of $\overline{Q}$. Under the natural map $$\Crit_{\mathrm{Rep}(\widetilde{Q},\bv,\mathbf 0)}(\widetilde{\sw})/\!\!/ \GL(\bv)\longrightarrow \left(\mathrm{Rep}(\Pi_{\overline{Q}},\bv)/\!\!/ \GL(\bv)\right)\times \left(\mathfrak{gl}(\bv)/\!\!/\GL(\bv)\right)$$
the $\sT_0$ fixed locus of the left-hand-side maps to the $\sT_0$ fixed locus of the right-hand-side, which is $\{0\}\times \{0\}$. This implies that the preimage of $\left(\Crit_{\mathrm{Rep}(\widetilde{Q},\bv,\mathbf 0)}(\widetilde{\sw})/\!\!/ \GL(\bv)\right)^{\bC^*_\hbar}$ under the quotient map is contained in the locus where $M$ is a nilpotent $\Pi_{\overline{Q}}$-representation, and $\cE$ is an nilpotent endomorphism. Semisimplification of $\{M\text{ nilpotent}, \cE\text{ nilpotent} \}$ as $\widetilde{Q}$-representation is the trivial one; thus $\left(\Crit_{\mathrm{Rep}(\widetilde{Q},\bv,\mathbf 0)}(\widetilde{\sw})/\!\!/ \GL(\bv)\right)^{\bC^*_\hbar}=\{0\}$.
\end{Example}

\subsection{Duality and adjoint operators}\label{subsec bilinear form}

For an algebraic variety $X$ with a regular function $f\colon X\to \bA^1$, we have Verdier duality
\begin{align*}
    H^{i}(X,\varphi_{f}\omega_X)^{\vee}\cong H^{-i}_c(X,\varphi_{f}\bC_X).
\end{align*}
If $X$ is smooth and $\Crit(f)$ is proper, then the RHS of the above isomorphism is isomorphic to $H^{-2d-i}(X,\varphi_{f}\omega_X)$ for $d=\dim X$, and the duality is induced by the nondegenerate pairing:
\begin{align}\label{pairing on crit coh}
(\cdot,\cdot)_X:
H(X,f)\otimes H(X,f)\cong H(X\times X,f\boxminus f)\xrightarrow{\Delta^!} H^{\mathrm{BM}}(X)_{\Crit(f)}\to \bC\:.
\end{align}
Here the first isomorphism uses $H^{*}(X,f)\cong H^{*}(X,-f)$ and Thom-Sebastiani isomorphism, the second map is Gysin pullback to diagonal, and the third map is pushforward to the point. 

Consider a pair of smooth varieties $X$ and $Y$ with functions $f_X$ and $f_Y$, and assume that both $\Crit(f_X)$ and $\Crit(f_Y)$ are proper. A class $\alpha\in H(X\times Y,f_X\boxminus f_Y)$ induces a convolution map \cite[(3.3)]{COZZ}: 
$$\alpha\colon H(Y,f_Y)\to H(X,f_X). $$ 
Using the isomorphism $H(X\times Y,f_X\boxminus f_Y)\cong H(Y\times X,f_Y\boxminus f_X)$, the transpose of $\alpha$ induces a convolution map $\alpha^{\mathrm{t}}\colon H(X,f_X)\to H(Y,f_Y)$. $\alpha$ and $\alpha^{\mathrm{t}}$ are adjoint to each other in the sense that
\begin{align*}
    (\alpha (u),v)_X=(u,\alpha^{\mathrm{t}}(v))_Y, \text{ for }u\in H(Y,f_Y), \: v\in H(X,f_X).
\end{align*}
Here $(\cdot,\cdot)_X$ and $(\cdot,\cdot)_Y$ are the pairing \eqref{pairing on crit coh} on $X$ and on $Y$ respectively.

When $X$ has a torus $\sT$ action, we can replace the properness by equivariant properness, that is, if $X$ is smooth and $\Crit(f)^{\sT}$ is proper, then we have nondegenerate pairing:
\begin{align}\label{pairing on eq crit coh}
  (\cdot,\cdot)\colon  H^\sT(X,f)_{\loc}\otimes H^\sT(X,f)_\loc\cong H^\sT(X\times X,f\boxminus f)_\loc\xrightarrow{\Delta^!} H^{\mathrm{BM}}(X)_{\Crit(f),\loc}\to \bC(\mathsf t)=\Frac H_\sT(\pt).
\end{align}
Here ``loc'' means $-\otimes_{\bC[\mathsf t]}\bC(\mathsf t)$. The discussions on transpose correspondence and adjoint operators naturally generalizes to the equivariant setting.

For our purpose, we introduce the following sign rule for pairing \eqref{pairing on eq crit coh} on cohomologies $\cH^{\sW^{\mathrm{fr}}}_{\underline{\bd},\sA^{\mathrm{fr}}}$.

\begin{Definition}
Consider the quiver variety $\cM(\bv,\underline{\bd})$ with $\sT_0\times \sA^{\mathrm{fr}}$ action and invariant potential $\sw^{\mathrm{fr}}$.
Suppose that $\Crit(\sw^{\mathrm{fr}})^{\sT_0\times \sA^{\mathrm{fr}}}$ is proper, we define the \textit{bilinear form} 
on $\cH^{\sW^{\mathrm{fr}}}_{\underline{\bd},\sA^{\mathrm{fr}}}(\bv)=H^{\sT_0\times \sA^{\mathrm{fr}}}(\cM(\bv,\underline{\bd}),\sw^{\mathrm{fr}})$: 
\begin{equation}\label{equ on bilinear form}(\cdot,\cdot)\colon \cH^{\sW^{\mathrm{fr}}}_{\underline{\bd},\sA^{\mathrm{fr}}}(\bv)\otimes \cH^{\sW^{\mathrm{fr}}}_{\underline{\bd},\sA^{\mathrm{fr}}}(\bv)\to  \Frac H_{\sT_0\times \sA^{\mathrm{fr}}}(\pt) \end{equation}
to be
\begin{align*}
    (-1)^{\rk \mathsf P}\times \eqref{pairing on eq crit coh}\text{ for $X=\cM(\bv,\underline{\bd})$, $\sT=\sT_0\times \sA^{\mathrm{fr}}$, $f=\sw^{\mathrm{fr}}$}.
\end{align*}
Here $\mathsf P$ is the virtual vector bundle \eqref{par pol on M(v,d)} on $\cM(\bv,\underline{\bd})$. 

For a homomorphism $M\colon \cH^{\sW^{\mathrm{fr}}}_{\underline{\bd},\sA^{\mathrm{fr}}}(\bv)\to \cH^{\sW^{\mathrm{fr}}}_{\underline{\bd},\sA^{\mathrm{fr}}}(\bv')$, we denote its $\bZ/2$-graded (super) adjoint operator  
\begin{equation}\label{equ on z2adj op}
M^{\tau}\colon \cH^{\sW^{\mathrm{fr}}}_{\underline{\bd},\sA^{\mathrm{fr}}}(\bv')\to \cH^{\sW^{\mathrm{fr}}}_{\underline{\bd},\sA^{\mathrm{fr}}}(\bv)
\end{equation}
with respect to the above pairing by 
\begin{align*}
    (Mx,y)=(-1)^{|M|\cdot|\bv|}(x,M^\tau y),\;\text{where }x\in\cH^{\sW^{\mathrm{fr}}}_{\underline{\bd},\sA^{\mathrm{fr}}}(\bv),\;y\in \cH^{\sW^{\mathrm{fr}}}_{\underline{\bd},\sA^{\mathrm{fr}}}(\bv').
\end{align*}
Note that $M^{\tau\tau}=(-1)^{|M|}M$.

For the tensor product of cohomologies, we invoke the super vector space sign rule:
\begin{align*}
    (a\otimes b,c\otimes d)_{\cH\otimes\cH'}=(-1)^{|b|\cdot|c|}(a,c)_{\cH}\cdot( b,d)_{\cH'}\:.
\end{align*}
\end{Definition}

\begin{Proposition}[Anti-automorphism for $\mathcal D\widetilde{\mathcal{SH}}_{\mu}(Q,\sW)$]
For any $\mu\in \bZ^{Q_0}$, the following assignment for the generators $e_i(z),f_i(z),h_i(z)$ in Definition \ref{def shifted yangian}:
\begin{align}\label{anti-inv}
    e_i(z)\mapsto (-1)^{|i|}f_i(z),\quad f_i(z)\mapsto e_i(z),\quad h_i(z)\mapsto h_i(z)
\end{align}
induces an anti-automorphism 
\begin{equation}\label{equ on anitmophism}\vartheta\colon \mathcal D\widetilde{\mathcal{SH}}_{\mu}(Q,\sW)\to \mathcal D\widetilde{\mathcal{SH}}_{\mu}(Q,\sW), \end{equation} 
that is, $\vartheta$ is a $\bC[\mathsf t_0]$-module automorphism such that 
$$\vartheta(XY)=(-1)^{|X|\cdot|Y|}\vartheta(Y)\vartheta(X), \,\,\, \mathrm{where}\,\,\, X,Y\in \mathcal D\widetilde{\mathcal{SH}}_{\mu}(Q,\sW).$$
Moreover, for any compact state space $\cH^{\sW^{\mathrm{fr}}}_{\underline{\bd},\sA^{\mathrm{fr}}}$ with $\bd_{\Out}-\bd_{\In}=\mu$, the bilinear form $(\cdot,\cdot)$ on $\cH^{\sW^{\mathrm{fr}}}_{\underline{\bd},\sA^{\mathrm{fr}}}$ is $\mathcal D\widetilde{\mathcal{SH}}_{\mu}(Q,\sW)$-invariant, that is,
\begin{align*}
    (a\cdot x,y)=(-1)^{|a|\cdot|x|}(x,\vartheta(a)\cdot y),\text{ for $a\in \mathcal D\widetilde{\mathcal{SH}}_{\mu}(Q,\sW)$, $x,y\in \cH^{\sW^{\mathrm{fr}}}_{\underline{\bd},\sA^{\mathrm{fr}}}$}\:.
\end{align*}
\end{Proposition}

\begin{proof}
The map $e_i(z)\mapsto (-1)^{|i|}f_i(z)$ is an anti-isomorphism $\widetilde{\mathcal{SH}}_{Q,\sW}^{\mathrm{nil}}\to \widetilde{\mathcal{SH}}_{Q,\sW}^{\mathrm{nil},\mathrm{op}}$, and $f_i(z)\mapsto e_i(z)$ is an anti-isomorphism from the opposite direction. Moreover, \eqref{anti-inv} respects the defining relations in Definition \ref{def shifted yangian}; therefore it induces an anti-automorphism \eqref{equ on anitmophism}. Finally, the $\mathcal D\widetilde{\mathcal{SH}}_{\mu}(Q,\sW)$-invariance follows by direct computations.
\end{proof}

\subsection{Anti-automorphisms for Reshetikhin type shifted Yangians}\label{sect on anti-auto}

\begin{Lemma}\label{lem adj op}
Let $\cH_{\mathfrak{c}}$ be an auxiliary space, and $\cH^{\sW^{\mathrm{fr}}}_{\underline{\bd},\sA^{\mathrm{fr}}}$ be a state space with $\bd_{\In}-\bd_{\Out}\in \bZ_{\geqslant 0}^{Q_0}$. Assume that both $\cH_{\mathfrak{c}}$ and $\cH^{\sW^{\mathrm{fr}}}_{\underline{\bd},\sA^{\mathrm{fr}}}$ are compact, then for any $m\in \End^{\bZ/2}_{\bC[\mathsf t_0]}(\cH_{\mathfrak{c}})[u]$, with respect to \eqref{equ on z2adj op}, we have
\begin{align*}
    \mathsf E(m)^\tau=\mathsf E(m^\tau).
\end{align*}
\end{Lemma}

\begin{proof}
According to the definition of root $R$-matrix and \cite[Lem.~3.29]{COZZ}, $$R^{\mathrm{sup}}(u)_{F,F'}=(-1)^{\mathrm{sgn}(F)}R(u)_{F,F'}=(-1)^{\mathrm{sgn}(F)}(-1)^{\rk \mathsf P|_F^-}\left([\Stab_{u>0}]^{\mathrm{t}}\circ[\Stab_{u>0}]\right)_{F,F'}(-1)^{\rk \mathsf P|_{F'}^+},$$ 
where $(-1)^{\mathrm{sgn}(F)}=(-1)^{|\bv|\cdot|\bv'|}$ on the component $\cH_{\mathfrak{c}}(\bv)\otimes \cH^{\sW^{\mathrm{fr}}}_{\underline{\bd},\sA^{\mathrm{fr}}}(\bv')$, which comes from the super vector space sign rule on the tensor product. It follows that
\begin{align*}
    R^{\mathrm{sup}}(u)^\tau_{F',F}&=(-1)^{\mathrm{sgn}(F')}(-1)^{\rk\mathsf P|_{F'}^{\text{fixed}}}(-1)^{\rk \mathsf P|_{F'}^+}\left([\Stab_{u>0}]^{\mathrm{t}}\circ[\Stab_{u>0}]\right)_{F',F}(-1)^{\rk \mathsf P|_F^-}(-1)^{\rk\mathsf P|_{F}^{\text{fixed}}}\\
    &=(-1)^{\mathrm{sgn}(F')}(-1)^{\rk \mathsf P|_{F'}^-}\left([\Stab_{u>0}]^{\mathrm{t}}\circ[\Stab_{u>0}]\right)_{F',F}(-1)^{\rk \mathsf P|_{F}^+}\\
    &= R^{\mathrm{sup}}(u)_{F',F}\:.
\end{align*}
Therefore 
\begin{align*}
    \str\left(m\cdot R^{\mathrm{sup}}(u)\right)^\tau=\str\left(m^\tau\cdot R^{\mathrm{sup}}(u)^\tau\right)=\str\left(m^\tau\cdot R^{\mathrm{sup}}(u)\right),
\end{align*}
and the lemma follows.
\end{proof}

\begin{Theorem}[Anti-automorphism for $\mathsf Y_{\mu}(Q,\sW)$]\label{thm anti-auto RTT Yangian}
Let $\mu\in \bZ_{\leqslant 0}^{Q_0}$. Assume that there exists a set of compact and admissible auxiliary data $\mathcal C$. If $\mu\neq 0$, we assume moreover that there exists a compact state space with $\bd_{\Out}-\bd_{\In}=\mu$. Then the map
\begin{align*}
    \mathsf E(m)\mapsto \mathsf E(m^{\tau}),\;\text{for }m\in \End^{\bZ/2}_{\bC[\mathsf t_0]}(\cH_{\mathfrak{c}})[u],\;\mathfrak{c}\in \mathcal C
\end{align*}
gives rise to an anti-automorphism $$\varsigma\colon \mathsf Y_{\mu}(Q,\sW)\cong \mathsf Y_{\mu}(Q,\sW).$$ 
Moreover, the bilinear form $(\cdot,\cdot)$ \eqref{equ on bilinear form} on any compact state space $\cH^{\sW^{\mathrm{fr}}}_{\underline{\bd},\sA^{\mathrm{fr}}}$ with $\bd_{\Out}-\bd_{\In}=\mu$ is $\mathsf Y_{\mu}(Q,\sW)$-invariant.
\end{Theorem}

\begin{proof}
By Theorem \ref{thm admissible}, $\mathsf Y_{\mu}(Q,\sW)\cong \mathsf Y_{\mu}(Q,\sW,\mathcal C)$. By Proposition \ref{prop sufficient state spaces}, the natural map
\begin{align}\label{embed into compact end}
    \mathsf Y_\mu(Q,\sW,\mathcal C)\longrightarrow \prod_{i_1,\cdots,i_n}\End^{\bZ/2}_{\bC(\mathsf t_0)[a_{0},a_{i_1},\ldots,a_{i_n}]}\left(\cH_0[a_0]\otimes\cH_{i_1}[a_{i_1}]\otimes\cdots\otimes \cH_{i_n}[a_{i_n}]\right)
\end{align}
is injective, where $\cH_0$ is a compact state space with shift $\bd_{\Out}-\bd_{\In}=\mu$, and $\left\{\cH_i:=\cH^{\sW_i}_{\underline{\delta_i}}\right\}_{i\in Q_0}$ is a collection of compact admissible auxiliary spaces (regarded as state spaces). The image of the map \eqref{embed into compact end} is invariant under taking adjoint operator, as $\mathsf E(m)^\tau=\mathsf E(m^\tau)$ by Lemma \ref{lem adj op}. Since taking adjoint operator is an algebra anti-automorphism, the map $\mathsf E(m)\mapsto \mathsf E(m^{\tau})$ gives an anti-automorphism $$\varsigma\colon \mathsf Y_{\mu}(Q,\sW)\cong \mathsf Y_{\mu}(Q,\sW).$$ Let $(\cdot,\cdot)$ be the bilinear form on a compact state space $\cH^{\sW^{\mathrm{fr}}}_{\underline{\bd},\sA^{\mathrm{fr}}}$ with $\bd_{\Out}-\bd_{\In}=\mu$, then
\begin{align*}
    (\mathsf E(m)\cdot x,y)=(-1)^{|m|\cdot|x|}(x,\mathsf E(m)^{\tau}\cdot y)\xlongequal{\scriptscriptstyle \text{Lemma \ref{lem adj op}}} (-1)^{|m|\cdot|x|}(x,\mathsf E(m^{\tau})\cdot y),
\end{align*}
in other words, $(\cdot,\cdot)$ is $\mathsf Y_{\mu}(Q,\sW)$-invariant.
\end{proof}

\begin{Remark}
If $\left(\Crit_{\mathrm{Rep}(\bv,\mathbf 0)}(\sw)/\!\!/ \GL(\bv)\right)^{\sT_0}$ is finite for all $\bv\in \bZ_{\geqslant 0}^{Q_0}$, then Lemma \ref{lem abundant compact} implies that there always exists a set of compact and admissible auxiliary data $\mathcal C$, and there exists a compact state space with $\bd_{\Out}-\bd_{\In}=\mu$ for arbitrary $\mu\in \bZ_{\leqslant 0}^{Q_0}$. In particular, assumptions in Theorem \ref{thm anti-auto RTT Yangian} are satisfied in this case.
\end{Remark}

% \begin{Definition}\label{def dual module}
% For a state space $\cH^{\sW^{\mathrm{fr}}}_{\underline{\bd},\sA^{\mathrm{fr}}}=H^{\sT_0\times \sA^{\mathrm{fr}}}(\cM(\underline{\bd}),\sw^{\mathrm{fr}})$ as in Definition \ref{def state space}, we define the compactly supported state space 
% \begin{align*}
% \prescript{}{c}{\cH}^{\sW^{\mathrm{fr}}}_{\underline{\bd},\sA^{\mathrm{fr}}}:=H^{\sT_0\times \sA^{\mathrm{fr}}}(\cM(\underline{\bd}),\sw^{\mathrm{fr}})_{\cM(\underline{\bd})^{\mathrm{nil}}}=\bigoplus_{\bv\in \bZ_{\geqslant0}^{Q_0}} H^{\sT_0\times \sA^{\mathrm{fr}}}(\cM(\bv,\underline{\bd}),\sw^{\mathrm{fr}})_{\cM(\bv,\underline{\bd})^{\mathrm{nil}}}
% \end{align*}
% where $\cM(\bv,\underline{\bd})^{\mathrm{nil}}\subseteq \cM(\bv,\underline{\bd})$ is the locus of nilpotent representations, that is, $\cM(\bv,\underline{\bd})^{\mathrm{nil}}=\mathsf{JH}^{-1}(0)$ for $\mathsf{JH}\colon \cM(\bv,\underline{\bd})\to \cM_0(\bv,\underline{\bd})$ the semisimplification map.
% \end{Definition}

\section{Vacuum cogeneration and irreducibility}\label{sec property_module}

In this section, we discuss two key properties of shifted Yangian $\mathsf Y_{\mu}(Q,\sW)$ modules: vacuum cogeneration (Theorem \ref{thm vac cogen_general}) and irreducibility (Proposition \ref{prop on irr mod}). We keep using the same notations from \S \ref{sec shifted Yangian}.

\subsection{Vacuum cogeneration}

\begin{Definition}
For an algebra $A$ defined over a field $\bK$, we say that an $A$-module $M$ is \textit{cogenerated} by a vector $v\in M^{\vee}$ if the map
\begin{align}\label{equ on AtoM}
    A\cdot v\to \Hom_{\bK}(M,\bK)
\end{align}
given by precomposing with the action of $A$, is surjective. In particular, if $v$ is the projection to the vacuum vector (identified with ground field $\bK$) in $M=\cH^{\sW^{\mathrm{fr}}}_{\underline{\bd},\sA^{\mathrm{fr}}}$, and $M$ is cogenerated by $v$, we say that $M$ is \textit{cogenerated by vacuum}.

\end{Definition}
\begin{Remark}
To show the surjectivity of \eqref{equ on AtoM}, one is left to show for any $m\in M$, there exists $a\in A$ such that $v(a\cdot m)\neq 0$. 
\end{Remark}
\begin{Theorem}\label{thm vac cogen_general}
Let $\underline{\bd}$ be a framing such that $\mu:=\bd_{\Out}-\bd_{\In}\in \bZ_{\leqslant0}^{Q_0}$. If the following two conditions are satisfied:
\begin{enumerate}
    \item[(i)] the elements $\{\mathrm{ch}_k(\mathsf D_{\In,i})\}_{i\in Q_0,k\in \bZ_{\geqslant 1}}\subseteq H_{\sT_0\times \sA^{\mathrm{fr}}}(\pt)$\footnote{Because the framing bundle $\mathsf D_{\In,i}$ is trivial, it is pulled back from a point. So we treat their Chern characters as classes on a point.} can be generated from $\{\mathrm{ch}_k(\mathsf D_{\In,i}-\mathsf D_{\Out,i})\}_{i\in Q_0,k\in \bZ_{\geqslant 1}}$ over the base field $\Frac H_{\sT_0}(\pt)$,
    \item[(ii)] one of the following specialization maps (as constructed in \cite[Ex.~7.19]{COZZ}):
    $$\mathsf{sp}_{\cM}\colon H^{\sT_0\times \sA^{\mathrm{fr}}}(\cM(\underline{\bd}),\sw^{\mathrm{fr}})\to H^{\sT_0\times \sA^{\mathrm{fr}}}(\cM(\underline{\bd})), \quad
   \mathsf{sp}_{\fM}\colon H^{\sT_0\times \sA^{\mathrm{fr}}}(\fM(\underline{\bd}),\sw^{\mathrm{fr}})\to H^{\sT_0\times \sA^{\mathrm{fr}}}(\fM(\underline{\bd}))$$ 
  is injective after tensoring $\Frac H_{\sT_0\times \sA^{\mathrm{fr}}}(\pt)$ over the base ring $H_{\sT_0\times \sA^{\mathrm{fr}}}(\pt)$,
\end{enumerate}
then $\cH^{\sW^{\mathrm{fr}}}_{\underline{\bd},\sA^{\mathrm{fr}}}$ is cogenerated by vacuum as a $\mathsf Y_{\mu}^{\leqslant}(Q,\sW)$-module over the base field $\Frac H_{\sT_0\times \sA^{\mathrm{fr}}}(\pt)$.
\end{Theorem}

\begin{proof}
Let us fix a gauge dimension $\bv$, and consider the following auxiliary datum: $(\bv,\sT_0\curvearrowright\cM(\bv),\sw)$ such that $\sT_0$ on out-going framing at node $i$ with a nontrivial weight $t_i$, and the potential is given by the unframed potential $\sw$. Take $G'=\GL(\bv)$ and let it act on the (both in-coming and out-going) framing vector spaces $\bC^{\bv}$ by fundamental representation, and pick a maximal torus of $G'$ to be $T'\cong\prod_{i\in Q_o}(\bC^*)^{\bv_i}$. Consider the stable envelope followed by restriction to $\cM(\bv,\bv)\times \cM(\mathbf 0,\underline{\bd})$:
\begin{align*}
    \Stab\colon H^{\sT_0\times \sA^{\mathrm{fr}}\times G'}\left(\bcancel{\cM(\mathbf 0,\bv)}\times \cM(\bv,\underline{\bd}),\sw\boxplus \sw^{\mathrm{fr}}\right)\to H^{\sT_0\times \sA^{\mathrm{fr}}\times G'}\left(\cM(\bv,\bv)\times \bcancel{\cM(\mathbf 0,\underline{\bd})},\sw\boxplus \sw^{\mathrm{fr}}\right)
\end{align*}
where the canceled spaces are points. In taking the stable envelope, we use the positive chamber for the $\bC^*$ framing torus that scales the framings of the auxiliary space with weight $-1$ and fixes the framings of the state space. Let 
$$\id\otimes 1\colon H^{\sT_0\times \sA^{\mathrm{fr}}}\left( \cM(\bv,\underline{\bd}),\sw^{\mathrm{fr}}\right)\to H^{\sT_0\times \sA^{\mathrm{fr}}\times G'}\left( \cM(\bv,\underline{\bd}),\sw^{\mathrm{fr}}\right)$$ be the extension by scalar map, where $1\in H_{G'}(\pt)$ is the unit. Consider the operator 
\begin{equation}\label{F op and stab}
\begin{gathered}
e^{\sT_0\times \sA^{\mathrm{fr}}\times G'}\left(N^-_{\cM(\bv,\bv)\times\cM(\mathbf 0,\underline{\bd})}\right)^{-1}\cdot\Stab\circ (\id\otimes 1)\colon H^{\sT_0\times \sA^{\mathrm{fr}}}\left( \cM(\bv,\underline{\bd}),\sw^{\mathrm{fr}}\right)\\
\longrightarrow H^{\sT_0\times G'}\left(\cM(\bv,\bv),\sw\right)\otimes_{H_{\sT_0\times G'}(\pt)}\Frac H_{\sT_0\times \sA^{\mathrm{fr}}\times G'}(\pt),
\end{gathered}
\end{equation}
which is a matrix element of the operator $F(u)^{-1}$ in the Gauss decomposition \eqref{gauss decomp} up to a sign. Composing \eqref{F op and stab} with any $H^{\sT_0\times G'}\left(\pt\right)$-linear functional $$H^{\sT_0\times G'}\left(\cM(\bv,\bv),\sw\right)\to H^{\sT_0\times G'}\left(\pt\right),$$ and taking a coefficient in the spectral parameter $u\in H_{G'}(\pt)$ as $u\to \infty$ in some direction (all directions are equivalent by the Weyl group symmetry), the result gives a map $$H^{\sT_0\times \sA^{\mathrm{fr}}}\left( \cM(\bv,\underline{\bd}),\sw^{\mathrm{fr}}\right)\to \Frac H^{\sT_0\times \sA^{\mathrm{fr}}}\left(\pt\right),$$ which lives in the algebra $\mathsf Y^-_\mu(Q,\sW)$. Write 
$$N^-_{\cM(\bv,\bv)\times\cM(\mathbf 0,\underline{\bd})}=\sum_{i\in Q_0} \Hom(\mathsf D_{\In,i},\mathsf V_i'),$$ where $\mathsf V_i'$ is the tautological vector bundle on $\cM(\bv,\bv)$, then by the condition (i), in terms of spectral parameter $u\in H_{G'}(\pt)$, coefficients of $e^{\sT_0\times \sA^{\mathrm{fr}}\times G'}\left(N^-_{\cM(\bv,\bv)\times\cM(\mathbf 0,\underline{\bd})}\right)$ are given by elements in the algebra $\mathsf Y^0_\mu(Q,\sW)$. 

We conclude that the composition
\begin{align}\label{cogen key form}
    f\circ \Stab\circ (\id\otimes 1)\colon H^{\sT_0\times \sA^{\mathrm{fr}}}\left(\cM(\bv,\underline{\bd}),\sw^{\mathrm{fr}}\right)\to H^{\sT_0\times \sA^{\mathrm{fr}}\times G'}\left(\pt\right)
\end{align}
for a linear functional $f\in \Hom_{H^{\sT_0\times G'}\left(\pt\right)}\left(H^{\sT_0\times G'}\left(\cM(\bv,\bv),\sw\right), H^{\sT_0\times G'}\left(\pt\right)\right)$, followed by taking a coefficient in the spectral parameter $u\in H_{G'}(\pt)$ as $u\to \infty$ in some direction, gives a map $$\cH^{\sW^{\mathrm{fr}}}_{\underline{\bd},\sA^{\mathrm{fr}}}(\bv)=H^{\sT_0\times \sA^{\mathrm{fr}}}\left( \cM(\bv,\underline{\bd}),\sw^{\mathrm{fr}}\right)\to H^{\sT_0\times \sA^{\mathrm{fr}}}\left(\pt\right)\cong \cH^{\sW^{\mathrm{fr}}}_{\underline{\bd},\sA^{\mathrm{fr}}}(\mathbf 0),$$ which lives in the algebra $\mathsf Y^{\leqslant}_\mu(Q,\sW)$. To prove the theorem, it is enough to show that $u$-coefficients of \eqref{cogen key form} cogenerates $\cH^{\sW^{\mathrm{fr}}}_{\underline{\bd},\sA^{\mathrm{fr}}}(\bv)$ for some $f$.

Consider the following choice of $f$. Let $j\colon \mathring{\cM}(\bv,\bv)\hookrightarrow\cM(\bv,\bv)$ be the natural inclusion of the open locus $\mathring{\cM}(\bv,\bv)$ consisting of those framed quiver representations such that in-coming framings are isomorphisms. Note that there is a $G'$-equivariant isomorphism 
$$\mathring{\cM}(\bv,\bv)\cong R(\bv,\mathbf 0)\times \prod_{i\in Q_0}\mathfrak{gl}(\bv_i),$$ where $G'$ acts on $R(\bv,\mathbf 0)$ as gauge group and acts on $\prod_{i\in Q_0}\mathfrak{gl}(\bv_i)$ as adjoint representation. In particular, $$H^{\sT_0\times G'}\left(\mathring{\cM}(\bv,\bv)\right)\cong H_{\sT_0\times G'}(\pt).$$ 
Consider specialization map \cite[Ex.~7.19]{COZZ}:
$$\mathsf{sp}_{\bv}\colon H^{\sT_0\times G'}\left(\mathring{\cM}(\bv,\bv),\sw\right)\to H^{\sT_0\times G'}\left(\mathring{\cM}(\bv,\bv)\right)\cong H_{\sT_0\times G'}(\pt).$$
We define the linear functional 
\begin{align*}
    f:=\mathsf{sp}_{\bv}\circ j^*\colon H^{\sT_0\times G'}\left(\cM(\bv,\bv),\sw\right)\to H_{\sT_0\times G'}(\pt).
\end{align*}
We claim that $u$-coefficients of \eqref{cogen key form} cogenerates $\cH^{\sW^{\mathrm{fr}}}_{\underline{\bd},\sA^{\mathrm{fr}}}(\bv)$ for the above $f$.

%Compare the construction of $f$ with the definition of nonabelian stable envelope in and 
By comparison, we find that 
\begin{align*}
    f\circ \Stab\circ (\id\otimes 1)=\mathsf{sp}_{\bv}\circ i^*\circ e^{\sT_0\times \sA^{\mathrm{fr}}\times G'}\left(\sum_{i\in Q_0}t_i\cdot\mathfrak{gl}(\bv_i)\right)\cdot \bPsi_H,
\end{align*}
where $\bPsi_H\colon H^{\sT_0\times \sA^{\mathrm{fr}}}\left(\cM(\bv,\underline{\bd}),\sw^{\mathrm{fr}}\right)\to H^{\sT_0\times \sA^{\mathrm{fr}}}\left(\fM(\bv,\underline{\bd}),\sw^{\mathrm{fr}}\right)$ is the interpolation map \cite[Def.\,9.1]{COZZ}, and $i\colon \fM(\bv,\mathbf 0)\hookrightarrow \fM(\bv,\underline{\bd})$ is the closed embedding induced by extension by trivial framing. 

Since specialization commutes with Gysin pullback, multiplication by Euler classes, and stable envelope, we have
\begin{align*}
    f\circ \Stab\circ (\id\otimes 1)=i^*\circ e^{\sT_0\times \sA^{\mathrm{fr}}\times G'}\left(\sum_{i\in Q_0}t_i\cdot\mathfrak{gl}(\bv_i)\right)\cdot \mathsf{sp}_{\fM}\circ  \bPsi_H=i^*\circ e^{\sT_0\times \sA^{\mathrm{fr}}\times G'}\left(\sum_{i\in Q_0}t_i\cdot\mathfrak{gl}(\bv_i)\right)\cdot \bPsi_H\circ \mathsf{sp}_{\cM}.
\end{align*}
Note that $i^*$ on the RHS of the above equation is actually an isomorphism: 
$$i^*\colon H^{\sT_0\times \sA^{\mathrm{fr}}}\left(\fM(\bv,\underline{\bd})\right)\cong H^{\sT_0\times \sA^{\mathrm{fr}}}\left(\fM(\bv,\mathbf 0)\right)\cong H^{\sT_0\times \sA^{\mathrm{fr}}\times G'}\left(\pt\right).$$ 
According to \cite[Lem.~9.2]{COZZ}, $\bPsi_H$ is injective. Then condition (ii) implies that 
$$\bPsi_H\circ \mathsf{sp}_{\cM}=\mathsf{sp}_{\fM}\circ  \bPsi_H$$ is injective after localization. So it suffices to show that for any nonzero $x\in H^{\sT_0\times \sA^{\mathrm{fr}}\times G'}\left(\pt\right)$, there exists a nonzero $u$-coefficient of the expansion of $e^{\sT_0\times \sA^{\mathrm{fr}}\times G'}\left(\sum_{i\in Q_0}t_i\cdot\mathfrak{gl}(\bv_i)\right)\cdot x$ in $u\to \infty$. This is equivalent to proving that $$e^{\sT_0\times \sA^{\mathrm{fr}}\times G'}\left(\sum_{i\in Q_0}t_i\cdot\mathfrak{gl}(\bv_i)\right)\cdot x\neq 0, $$ 
which follows from the facts that $t_i$ is nontrivial and $H^{\sT_0\times \sA^{\mathrm{fr}}\times G'}\left(\pt\right)$ is an integral domain. We are done.
\end{proof}

\begin{Corollary}\label{cor cogen_nak}
Let $\cN(\bd)=\bigsqcup_{\bv}\cN(\bv,\bd)$ be the Nakajima variety associated to a doubled quiver $\overline{Q}$, and let $\mathsf{S}$ be a torus that acts on $\mathrm{Rep}_{Q}(\bv,\bd)$ by flavour symmetry (i.e.~commutes with $G=\prod_{i\in Q_0}\GL(\bv_i)$ action), then $\sT_0=\bC^*_\hbar\times \mathsf{S}$ act on $\cN(\bd)$ using the identification $T^*\mathrm{Rep}_{Q}(\bv,\bd)=\mathrm{Rep}_{Q}(\bv,\bd)\oplus \hbar\:\mathrm{Rep}_{Q}(\bv,\bd)^{\vee}$. Assume that $\mathsf S$-fixed locus satisfies $\mathrm{Rep}_{Q}(\bv,\bd)^{\mathsf S}\subset\mathrm{Rep}_{Q}(\mathbf 0,\bd)$ where the latter is regarded as subspace of the former via extension by zero. Take a torus $\sA^{\mathrm{fr}}$ that acts on the framing vector space, then $H^{\sT_0\times \sA^{\mathrm{fr}}}(\cN(\bd))$ is cogenerated by vacuum as a $\mathsf Y_0^{\leqslant}(\widetilde Q,\widetilde{\sW})$-module over the base field $\Frac H_{\sT_0\times \sA^{\mathrm{fr}}}(\pt)$.
\end{Corollary}

\begin{proof}
Let $\cM(\bv,\bd)=\cM(\widetilde{Q},\bv,\bd)$ be the symmetric quiver variety associated to the tripled quiver $\widetilde{Q}$, and let $\sT_0$ act on $\cM(\bv,\bd)$ by scaling the added loops with weight $\hbar^{-1}$. Denote the standard cubic potential by $\widetilde{\sw}$. 
By dimensional reduction, we have an isomorphism of $\mathsf Y_0^{\leqslant}(\widetilde Q,\widetilde{\sW})$-modules:
$$H^{\sT_0\times \sA^{\mathrm{fr}}}(\cN(\bd))\cong H^{\sT_0\times \sA^{\mathrm{fr}}}(\cM(\bd),\widetilde{\sw}). $$ 
So it is enough to prove the vacuum cogeneration for the latter. We claim that
\begin{enumerate}
\item[(i)] the elements $\{\mathrm{ch}_k(\mathsf D_{\In,i})\}_{i\in Q_0,k\in \bZ_{\geqslant 1}}\subseteq H_{\sT_0\times \sA^{\mathrm{fr}}}(\pt)$ can be generated from $\{\mathrm{ch}_k(\mathsf D_{\In,i}-\mathsf D_{\Out,i})\}_{i\in Q_0,k\in \bZ_{\geqslant 1}}$ over the base field $\Frac H_{\sT_0}(\pt)$,
\item[(ii)] the specialization maps $\mathsf{sp}_{\fM}\colon H^{\sT_0\times \sA^{\mathrm{fr}}}(\fM(\underline{\bd}),\sw^{\mathrm{fr}})\to H^{\sT_0\times \sA^{\mathrm{fr}}}(\fM(\underline{\bd}))$ is injective.
\end{enumerate}
Claim (i) follows from the fact that $\mathrm{ch}(\mathsf D_{\Out,i})=e^\hbar\mathrm{ch}(\mathsf D_{\In,i})$. By \cite[Rem.~7.16]{COZZ}, we have commutative diagram
\begin{equation*}
\xymatrix{
H^{\sT_0\times \sA^{\mathrm{fr}}}(Z(\widetilde{\sw})) \ar[d]_{\can} \ar[dr] & \\
H^{\sT_0\times \sA^{\mathrm{fr}}}(\fM(\bv,\bd),\widetilde{\sw}) \ar[r]^{\mathsf{sp}} & H^{\sT_0\times \sA^{\mathrm{fr}}}(\fM(\bv,\bd)),
}
\end{equation*}
where $Z(\widetilde{\sw})$ is the zero locus of $\widetilde{\sw}$ in $\fM(\bv,\bd)$. Note that the zero locus of the moment map $\mu\colon \fM(\bv,\bd)\to [\mathfrak{g}/G]$ is a closed substack of $Z(\widetilde{\sw})$, and the composition 
$$\delta:=\can\circ\: (Z(\mu)\hookrightarrow Z(\widetilde{\sw}))_*\colon H^{\sT_0\times \sA^{\mathrm{fr}}}(Z(\mu))\to H^{\sT_0\times \sA^{\mathrm{fr}}}(\fM(\bv,\bd),\widetilde{\sw})$$ is an isomorphism by dimensional reduction. Therefore, for $j\colon Z(\mu)\hookrightarrow \fM(\bv,\bd)$, we have $$\mathsf{sp}=j_*\circ \delta^{-1}.$$ 
We claim that the pushforward map $j_*$ is injective. Consider an additional torus $\bC^*_a$ that scales the framing vector space by weight $1$ and fixes other quiver data. This $\bC^*_a$ action on $\mathrm{Rep}_{\widetilde{Q}}(\bv,\bd)$ can be equivalently obtained by first mapping $\bC^*_a$ diagonally into the center $\prod_{i\in Q_0}\bC^*$ of gauge group $\prod_{i\in Q_0}\GL(\bv_i)$ and then taking the inverse. Then it follows that $H^{\sT_0\times \sA^{\mathrm{fr}}\times \bC^*_a}(Z(\mu))$ and $H^{\sT_0\times \sA^{\mathrm{fr}}\times \bC^*_a}(\fM(\bv,\bd))$ are free $H_{\bC^*_a}(\pt)$-modules, and we have 
\begin{equation*}
H^{\sT_0\times \sA^{\mathrm{fr}}}(Z(\mu))\otimes H_{\bC^*_a}(\pt)\cong H^{\sT_0\times \sA^{\mathrm{fr}}\times \bC^*_a}(Z(\mu)),\; H^{\sT_0\times \sA^{\mathrm{fr}}}(\fM(\bv,\bd))\otimes H_{\bC^*_a}(\pt)\cong H^{\sT_0\times \sA^{\mathrm{fr}}\times \bC^*_a}(\fM(\bv,\bd)).
\end{equation*}
Then our choice of $\mathsf S\times \sA^{\mathrm{fr}}\times \bC^*_a$ action satisfies the genericity condition in \cite[Thm.\,10.2]{Dav2}, which implies that $$j_*\colon H^{\sT_0\times \sA^{\mathrm{fr}}\times \bC^*_a}(Z(\mu))\to H^{\sT_0\times \sA^{\mathrm{fr}}\times \bC^*_a}(\fM(\bv,\bd))$$ is injective. Since $\sT_0\times \sA^{\mathrm{fr}}\times \bC^*_a$ equivariant pushforward is obtained from the $\sT_0\times \sA^{\mathrm{fr}}$ by extension by scalar, $j_*\colon H^{\sT_0\times \sA^{\mathrm{fr}}}(Z(\mu))\to H^{\sT_0\times \sA^{\mathrm{fr}}}(\fM(\bv,\bd))$ must be injective as well. This proves Claim (ii). Then the corollary follows from the above two claims and Theorem \ref{thm vac cogen_general}.
\end{proof}

\begin{Corollary}\label{cor vac cogen_tripled}
Under the same assumptions as Corollary \ref{cor cogen_nak}, and let $\phi=\tr(\mathsf m)$ be a $(\sT_0\times \sA^{\mathrm{fr}})$-invariant potential on $\cN(\bd)$, the following statements are equivalent:
\begin{enumerate}
    \item the specialization map $\mathsf{sp}\colon H^{\sT_0\times \sA^{\mathrm{fr}}}(\cN(\bd),\phi)\to H^{\sT_0\times \sA^{\mathrm{fr}}}(\cN(\bd))$ in \cite[Ex.\,\,7.19]{COZZ} is injective after tensoring with $\Frac H_{\sT_0\times \sA^{\mathrm{fr}}}(\pt)$ over the base ring $ H_{\sT_0}(\pt)$,
    \item $H^{\sT_0\times \sA^{\mathrm{fr}}}(\cN(\bd),\phi)$ is cogenerated by vacuum as a $\mathsf Y_0(\widetilde Q,\widetilde{\sW})$-module over the base field $\Frac H_{\sT_0\times \sA^{\mathrm{fr}}}(\pt)$.
    \item $H^{\sT_0\times \sA^{\mathrm{fr}}}(\cN(\bd),\phi)$ is cogenerated by vacuum as a $\mathsf Y_0^{\leqslant}(\widetilde Q,\widetilde{\sW})$-module over the base field $\Frac H_{\sT_0\times \sA^{\mathrm{fr}}}(\pt)$.
\end{enumerate}
\end{Corollary}

\begin{proof}
$(3)\Longrightarrow(2)$ is tautological. 

$(2)\Longrightarrow(1)$. Suppose $\ker(\mathsf{sp})$ is nontrivial after localization, then $$\ker(\mathsf{sp})_\loc \subseteq H^{\sT_0\times \sA^{\mathrm{fr}}}(\cN(\bd),\phi)_\loc$$ is a $\mathsf Y_0(\widetilde Q,\widetilde{\sW})$-submodule, since $\mathsf{sp}$ commutes with stable envelopes and $R$-matrices. By construction, we have $$\mathsf{sp}\colon H^{\sT_0\times \sA^{\mathrm{fr}}}(\cN(\mathbf 0,\bd),\phi)\xrightarrow{\cong}  H^{\sT_0\times \sA^{\mathrm{fr}}}(\cN(\mathbf 0,\bd)).$$ Denote $\bK=\Frac H_{\sT_0\times \sA^{\mathrm{fr}}}(\pt)$, $|0\rangle\in H^{\sT_0\times \sA^{\mathrm{fr}}}(\cN(\mathbf 0,\bd),\phi)=H^{\sT_0\times \sA^{\mathrm{fr}}}(\pt)$ to be a generator. Then the map $$\mathsf Y_0(\widetilde Q,\widetilde{\sW})\cdot\Hom_{\bK}(\bK\cdot |0\rangle,\bK)\to \Hom_{\bK}(H^{\sT_0\times \sA^{\mathrm{fr}}}(\cN(\bd),\phi)_\loc,\bK)$$ factors through the proper subspace $\Hom_{\bK}(\im(\mathsf{sp})_\loc,\bK)$, which contradicts with (2). Thus $\ker(\mathsf{sp})_\loc$ is trivial.

$(1)\Longrightarrow(3)$. $\mathsf{sp}\colon H^{\sT_0\times \sA^{\mathrm{fr}}}(\cN(\bd),\phi)\to H^{\sT_0\times \sA^{\mathrm{fr}}}(\cN(\bd))$ commutes with stable envelopes and is therefore a $\mathsf Y_0^{\leqslant}(\widetilde Q,\widetilde{\sW})$-module map. By Corollary \ref{cor cogen_nak}, $H^{\sT_0\times \sA^{\mathrm{fr}}}(\cN(\bd))$ is cogenerated by vacuum as a $\mathsf Y_0^{\leqslant}(\widetilde Q,\widetilde{\sW})$-module over the base field $\Frac H_{\sT_0\times \sA^{\mathrm{fr}}}(\pt)$, so is its submodule $H^{\sT_0\times \sA^{\mathrm{fr}}}(\cN(\bd),\phi)$.
\end{proof}

\subsection{Irreducibility}

\begin{Proposition}\label{prop on irr mod}
Let $\mu\in \bZ_{\leqslant 0}^{Q_0}$, $\underline{\bd}=(\bd_{\In},\bd_{\Out})$ with $\bd_{\Out}-\bd_{\In}=\mu$. Assume that there exists a set of compact and admissible auxiliary data $\mathcal C$. If the state space $\cH^{\sW^{\mathrm{fr}}}_{\underline{\bd},\sA^{\mathrm{fr}}}$ is compact and cogenerated by vacuum over the base field $\Frac H_{\sT_0\times \sA^{\mathrm{fr}}}(\pt)$, then $\cH^{\sW^{\mathrm{fr}}}_{\underline{\bd},\sA^{\mathrm{fr}}}$ is an irreducible $\mathsf Y_{\mu}(Q,\sW)$-module over the base field $\Frac H_{\sT_0\times \sA^{\mathrm{fr}}}(\pt)$.
\end{Proposition}

\begin{proof}
Since $\cH^{\sW^{\mathrm{fr}}}_{\underline{\bd},\sA^{\mathrm{fr}}}$ is compact, it has a nondegenerate bilinear form $(\cdot,\cdot)$ \eqref{equ on bilinear form}, which is $\mathsf Y_{\mu}(Q,\sW)$-invariant by Theorem \ref{thm anti-auto RTT Yangian}. By duality, cogeneration by vacuum implies that $\cH^{\sW^{\mathrm{fr}}}_{\underline{\bd},\sA^{\mathrm{fr}}}$ is also generated by the vacuum vector $|\mathbf 0\rangle$ over the base field $\Frac H_{\sT_0\times \sA^{\mathrm{fr}}}(\pt)$. Cogeneration by vacuum implies that for arbitrary $x\in \cH^{\sW^{\mathrm{fr}}}_{\underline{\bd},\sA^{\mathrm{fr}}}$ there exists $a\in \mathsf Y_{\mu}(Q,\sW)\otimes_{H_{\sT_0\times \sA^{\mathrm{fr}}}(\pt)}\Frac H_{\sT_0\times \sA^{\mathrm{fr}}}(\pt)$ such that $a\cdot x=|\mathbf 0\rangle$, so $x$ generates $\cH^{\sW^{\mathrm{fr}}}_{\underline{\bd},\sA^{\mathrm{fr}}}$. Therefore $\cH^{\sW^{\mathrm{fr}}}_{\underline{\bd},\sA^{\mathrm{fr}}}$ is an irreducible $\mathsf Y_{\mu}(Q,\sW)$-module over the base field $\Frac H_{\sT_0\times \sA^{\mathrm{fr}}}(\pt)$.
\end{proof}

\section{Example: trivial quiver and modules of shifted Yangians of \texorpdfstring{$\mathfrak{gl}_{1|1}$}{gl(1|1)}}\label{sec gl(1|1)}

In this section, we consider the simplest example: the trivial quiver with zero potential:
\begin{equation*}
Q\colon \quad
\begin{tikzpicture}[x={(1cm,0cm)}, y={(0cm,1cm)}, baseline=0cm]
  % Nodes
  \node[draw,circle,fill=white] (Gauge1) at (0,0){$\phantom{n}$}; %{$\bv$}  {$\phantom{n}$};
  % Edges

  % Loop
\end{tikzpicture}
\;,
\quad\sW=0\,.
\end{equation*}
Let $\sT_0=\bC^*_\hbar$ be the torus that acts on $Q$ trivially ($\sT_0$ action on the framing vector spaces will be specified later). Fix a cyclic stability $\theta<0$ when we consider associated quiver varieties.  

%We refer to Example \ref{ex doubled vs tripled} for the definition of tripled quivers.
%$$\cM(n,\underline{\bd})=\cM_\theta(n,\underline{\bd}), \quad \cM(\underline{\bd})=\bigsqcup_{n\in \bN} \cM(n,\underline{\bd}). $$
%Recall that for symmetric framing $\mathbf d_{\In}=\mathbf d_{\Out}=\bd$, we write 
%$$\cM(n,\bd)=\cM(n,\underline{\bd}), \quad \cM(\bd)=\cM(\underline{\bd}).$$ 

The main theorem of this section is an \textit{explicit} description of the shifted Yangian $\mathsf Y_{\mu}(Q,0)$ for $Q$ \eqref{equ on rtt yangian}. 
\begin{Theorem}\label{thm ex trivial quiver_main}
Let $Q$ be the trivial quiver. For arbitrary $\mu\in \bZ_{\leqslant 0}$, there is a natural algebra isomorphism
\begin{align*}
    \varrho_\mu\colon Y_{\mu}(\mathfrak{gl}_{1|1})\otimes_{\bC[\hbar]} \bC(\hbar)\cong \mathsf Y_{\mu}(Q,0)\:,
\end{align*}
where $Y_{\mu}(\mathfrak{gl}_{1|1})$ is the 
$\mu$-shifted Yangian of $\mathfrak{gl}_{1|1}$ (Definition \ref{def of shifted Y(gl(1|1))}).
\end{Theorem}

In this section, we also perform \textit{explicit} computations of cohomological stable envelopes and root $R$-matrices in \S \ref{sec fund R-mat gl(1|1)} and 
\S \ref{sec Lax mat Cl mod}. These computations supersede the results in \cite[\S 7.2]{IMZ}, but the method there is different from ours. We deduce the explicit formula for quantum multiplication by divisors on the framed quiver varieties (Proposition \ref{prop on qmbd on gl(1|1)}). This in particular recovers the well-known result on $\Gr(n,r)$.

There are also quivers with potentials associated to $\mathfrak{gl}_{m|n}$, see Example \ref{ex on vv ex} and Appendix \S \ref{sec super affine Yangian}. In Appendix \S \ref{sect on Lie superalgebra}, we compute some $R$-matrices for the quivers with potentials associated to $\mathfrak{gl}_{2|1}$.

\subsection{The fundamental $R$-matrices}\label{sec fund R-mat gl(1|1)}
Consider the framed quiver
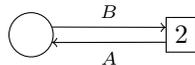
\begin{equation*}
\begin{tikzpicture}[x={(1cm,0cm)}, y={(0cm,1cm)}, baseline=0cm]
  % Nodes
  \node[draw,circle,fill=white] (Gauge) at (0,0){$\phantom{n}$}; %{$\bv$}  {$\phantom{n}$};
  \node[draw,rectangle,fill=white] (Framing) at (2,0) {$2$};
  % Edges
  \draw[<-] (Gauge.340) -- (Framing.205) node[midway,below] {\scriptsize $A$};
  \draw[->] (Gauge.20) -- (Framing.155) node[midway,above] {\scriptsize $B$};

  % Loop
\end{tikzpicture}
\end{equation*}
and the corresponding symmetric quiver variety $\cM(\mathbf 2)$. Take a torus $\sT=\bC^*_{t_1}\times \bC^*_{t_2}\times \bC^*_u$ such that $(s_1,s_2,a)\in \bC^*_{t_1}\times \bC^*_{t_2}\times \bC^*_u$ acts on the quiver data by
\begin{align*}
    (A_1,A_2,B_1,B_2)\mapsto (a^{-1}A_1,A_2,s_1aB_1,s_2B_2).
\end{align*}
Let $\sA=\bC^*_u$, then $\cM(\mathbf 2)^{\sA}=\cM(\mathbf 1)\times \cM(\mathbf 1)$. Note that 
$$\cM(\mathbf 1)=\cM(0,\mathbf 1)\sqcup \cM(1,\mathbf 1), \quad \cM(0,\mathbf 1)=\pt, \quad \cM(1,\mathbf 1)\cong \bC. $$ 
We take the set of fundamental classes of $\cM(n,\mathbf 1)$ ($n=0,1$) as a basis of $H^{\sT}(\cM(\mathbf1))$, and denote them
\begin{align*}
b_0=[\cM(0,\mathbf 1)],\quad b_1=[\cM(1,\mathbf 1)].
\end{align*}
\subsubsection{Computation of stable envelopes}
Let us consider stable envelopes for $\cM(1,\mathbf 2)$ in both ``$+$'' and ``$-$'' chambers. Let the equivariant variable of the gauge group $\bC^*$ be $x$,~i.e.~$H_{\bC^*}(\pt)=\bC[x]$. Using \cite[Cor.~4.25]{COZZ}, we have 
\begin{align*}
\Stab_+(b_0\otimes b_1)=x-u,\quad \Stab_+(b_1\otimes b_0)=t_2-x,\quad \Stab_-(b_0\otimes b_1)=u+t_1-x,\quad \Stab_-(b_1\otimes b_0)=x.
\end{align*}
The matrix forms of stable envelopes in the basis $\{b_0\otimes b_1,b_1\otimes b_0\}$ are
\begin{align*}
\Stab_+=\begin{pmatrix}
-u & t_2\\
0 & t_2-u
\end{pmatrix}\;,
\qquad
\Stab_-=\begin{pmatrix}
u+t_1 & 0\\
t_1 & u
\end{pmatrix}\;.
\end{align*}

\subsubsection{The R-matrices and YBE}
The $R$-matrix for $\cM(1,\mathbf 2)$ is given by
\begin{align*}
R(u)\big|_{\cM(1,\mathbf 2)}=\Stab_-^{-1}\Stab_+=\frac{1}{u+t_1}\begin{pmatrix}
-u & t_2\\
t_1 & t_2-t_1-u
\end{pmatrix}\;.
\end{align*}
The $R$-matrix for $\cM(0,\mathbf 2)$ and for $\cM(2,\mathbf 2)$ are
\begin{align*}
R(u)\big|_{\cM(0,\mathbf 2)}=1,\quad R(u)\big|_{\cM(2,\mathbf 2)}=\frac{t_2-u}{u+t_1}.
\end{align*}
Taking the normalizer \eqref{normalizer_yangian} and the sign twist $\Sigma$ in Theorem \ref{thm e f h as R matrix elements} into account, the $R^{\mathrm{sup}}(u)$ matrix is given by 
\begin{align}\label{R sup for gl(1|1)}
R^{\mathrm{sup}}(u,t_1,t_2)=\frac{1}{u+t_1}\begin{pmatrix}
u+t_1 & & & \\
 &u & t_2 & \\
 & t_1 & u+t_1-t_2 &\\
 & & & u-t_2
\end{pmatrix}\;.
\end{align}
The Yang-Baxter-equation
\begin{align*}
R^{\mathrm{sup}}_{12}(u,t_1,t_2)R^{\mathrm{sup}}_{13}(u+v,t_1,t_3)R^{\mathrm{sup}}_{23}(v,t_2,t_3)=R^{\mathrm{sup}}_{23}(v,t_2,t_3)R^{\mathrm{sup}}_{13}(u+v,t_1,t_3)R^{\mathrm{sup}}_{12}(u,t_1,t_2)
\end{align*}
holds in $\End(\bC^{1|1}\otimes \bC^{1|1}\otimes\bC^{1|1})(u,v,t_1,t_2,t_3)$, where $\bC^{1|1}=\Span_\bC\{b_0,b_1\}$ with $b_0$ even and $b_1$ odd.

When $t_1=t_2=\hbar$, \eqref{R sup for gl(1|1)} can be written in a compact form:
\begin{align}\label{fund geom R-mat 1|1}
R^{\mathrm{sup}}(u,\hbar)=\frac{u+\hbar P}{u+\hbar},\;\text{ where }\;P(b_i\otimes b_j)=(-1)^{ij}b_j\otimes b_i\;\text{ is the super permutation operator}.
\end{align}

\subsection{Lax matrix associated to a Clifford module}\label{sec Lax mat Cl mod}
Consider the framing dimension $(\bd_{\In},\bd_{\Out})=(2,1)$, depicted as follows
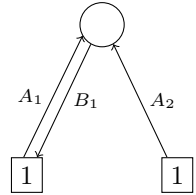
\begin{equation*}
\begin{tikzpicture}[x={(1cm,0cm)}, y={(0cm,1cm)}, baseline=0cm]
  % Nodes
  \node[draw,circle,fill=white] (Gauge) at (0,0) {$\phantom{n}$};
  \node[draw,rectangle,fill=white] (Framing1) at (-1,-2) {1};
  % \node (Framingdot) at (0,-2) {$\cdots$};
  % \node (Arrowdot) at (0,-1.5) {$\cdots$};
  \node[draw,rectangle,fill=white] (Framingk) at (1,-2) {1};
  % Edges
  \draw[<-] (Gauge.210) -- (Framing1.100) node[midway,left] {\scriptsize $A_1$};
  \draw[->] (Gauge.240) -- (Framing1.60) node[midway,right] {\scriptsize $B_1$};

  %\draw[->] (Gauge.330) -- (Framingk.80) node[midway,right] {\scriptsize $B_2$};
  \draw[<-] (Gauge.300) -- (Framingk.120) node[midway,right] {\scriptsize $A_2$};

  % Loop
\end{tikzpicture}
\end{equation*}
and the corresponding quiver variety $\cM((2,1))$. Take a torus $\sT=\bC^*_{\hbar}\times \bC^*_u$ such that $(s,a)\in \bC^*_{\hbar}\times \bC^*_u$ acts on the quiver data by
\begin{align*}
    (A_1,A_2,B)\mapsto (a^{-1}A_1,A_2,saB_1).
\end{align*}
Let $\sA=\bC^*_u$, then $\cM((2,1))^{\sA}=\cM(\mathbf 1)\times \cM((1,0))$. As in the previous subsection, we use $b_n$ (resp.~$c_n$) to denote the fundamental class of $\cM(n,\mathbf 1)$ (resp.~$\cM(n,(1,0))$) for $n=0,1$, then $H^\sT(\cM((2,1))^{\sA})$ is a free $H_\sT(\pt)$-module with a basis $\{b_i\otimes c_j\::\: 0\leqslant i,j\leqslant 1\}$.

Let the equivariant variable of the gauge group $\bC^*$ be $x$. Using \cite[Cor.~4.25]{COZZ}, we have 
\begin{align*}
\Stab_+(b_0\otimes c_1)=x-u,\quad \Stab_+(b_1\otimes c_0)=1,\quad \Stab_-(b_0\otimes c_1)=u+\hbar-x,\quad \Stab_-(b_1\otimes c_0)=x,
\end{align*}
with the matrix forms:
\begin{align*}
\Stab_+=\begin{pmatrix}
-u & 1\\
0 & 1
\end{pmatrix}\;,
\qquad
\Stab_-=\begin{pmatrix}
u+\hbar & 0\\
\hbar & u
\end{pmatrix}\;.
\end{align*}
Let $L(u):=R^{\mathrm{sup}}(u)$ for $\cM((2,1))$, then 
\begin{align*}
L(u,\hbar)=\frac{1}{u+\hbar}\begin{pmatrix}
u+\hbar & & & \\
 &u & 1 & \\
 & -\hbar & 1 &\\
 & & & 1
\end{pmatrix}\;.
\end{align*}
Let $\mathsf{Cl}$ be the $\bC$-algebra with two odd generators $\psi$ and $\psi^*$, subject to relation $\psi\psi^*+\psi^*\psi=1$. Then $\Span_\bC\{c_0,c_1\}$ can be endowed with a $\mathsf{Cl}$-module structure by setting
\begin{align*}
    \psi^*c_0=c_1,\quad \psi c_1=c_0.
\end{align*}
Let $\mathcal F$ denote this module. Then $L(u,\hbar)$ can be represented as an element in $\End(\bC^{1|1})\otimes \mathsf{Cl}$:
\begin{align*}
L(u,\hbar)=\frac{1}{u+\hbar}\begin{pmatrix}
u+\hbar\psi\psi^* & \psi^*\\
-\hbar\psi & 1
\end{pmatrix}\;,
\end{align*}
and it satisfies the RLL equation:
\begin{align*}
R^{\mathrm{sup}}_{12}(u-v,\hbar)L_1(u,\hbar)L_2(v,\hbar)=L_2(v,\hbar)L_1(u,\hbar)R^{\mathrm{sup}}_{12}(u-v,\hbar)
\end{align*}
when acting on $\bC^{1|1}\otimes \bC^{1|1}\otimes\mathcal F$.

\subsection{Shifted Yangian of \texorpdfstring{$\mathfrak{gl}_{1|1}$}{gl(1|1)}}\label{sec yangian of gl(1|1)}

Let 
\begin{equation}\label{equ on R(z) super}R(z)=z\id+\hbar P\in \End(\bC^{1|1})^{\otimes 2}[z,\hbar]. \end{equation}
Here $P(a\otimes b)=(-1)^{|a|\cdot |b|}b\otimes a$ is the super permutation operator.

\begin{Definition}\label{def of shifted Y(gl(1|1))}
Fix $\mu\in \bZ_{\leqslant 0}$. The $\mu$-\textit{shifted Yangian} $Y_{\mu}(\mathfrak{gl}_{1|1})$ is a $\bC[\hbar]$ algebra generated by $$\left\{e_{r},f_r,g_{1,r},g_{2,s}\right\}_{r\in \bZ_{\geqslant0},s\in \bZ_{\geqslant-\mu}},$$ 
subject to relations
\begin{align}\label{RTT gl(1|1)}
    R(u-v)\,T_1(u)\,T_2(v)=T_2(v)\,T_1(u)\,R(u-v),
\end{align}
with $T(z)\in Y_{\mu}(\mathfrak{gl}_{1|1})[\![z,z^{-1}]\!]\otimes \End(\bC^{1|1})$ defined via
\begin{align*}
    T(z)=F(z)\,G(z)\,E(z),\,\,\, \text{where}\,\,\, E(z)=\id_{1|1}+e(z)\otimes E_{12},\,\,\,  F(z)=\id_{1|1}+f(z)\otimes E_{21},\,\,\,  G(z)=\sum_{i=1}^2 g_i(z)\otimes E_{ii},
\end{align*}
and 
\begin{align}\label{gen current gl(1|1)}
    e(z)=\sum_{r\geqslant 0} e_r z^{-r-1},\quad f(z)=\sum_{r\geqslant 0} f_r z^{-r-1},\quad g_1(z)=1+\sum_{r\geqslant 0}g_{1,r}z^{-r-1},\quad g_2(z)=z^{\mu}+\sum_{s\geqslant-\mu}g_{2,s}z^{-s-1}\,.
\end{align}
Here \eqref{RTT gl(1|1)} is an equation in $Y_{\mu}(\mathfrak{gl}_{1|1})[\![u,v,u^{-1},v^{-1}]\!]\otimes \End(\bC^{1|1})^{\otimes 2}$. 
\end{Definition}

\begin{Remark}
The special case $\mu=0$ gives the definition of unshifted Yangian $Y(\mathfrak{gl}_{1|1})$ in \cite{G}. In the below, we also write $Y(\mathfrak{gl}_{1|1})=Y_{0}(\mathfrak{gl}_{1|1})$.
\end{Remark}

Let $\mathfrak{c}$ be the auxiliary datum such that the framing dimension is one, potential is trivial, and $\sT_0=\bC^*_\hbar$ fixes the in-coming framing and scales the out-going framing with weight one. The auxiliary data set $\mathcal C=\{\mathfrak{c}\}$ is admissible (Definition \ref{def on adm aux data}); therefore $$\mathsf Y_\mu(Q,0)=\mathsf Y_\mu(Q,0,\cC).$$ 
By \eqref{fund geom R-mat 1|1}, the $R$-matrix that braids $\cH_{\mathfrak{c}}\otimes \cH_{\mathfrak{c}}$ is the $R$-matrix in \eqref{equ on R(z) super} up to a scalar constant. It follows that for any state space $\cH^{\sW^{\mathrm{fr}}}_{\underline{\bd},\sA^{\mathrm{fr}}}$ with $\bd_{\Out}-\bd_{\In}=\mu$, the $R$-matrix that braids $\cH_{\mathfrak{c}}\otimes \cH^{\sW^{\mathrm{fr}}}_{\underline{\bd},\sA^{\mathrm{fr}}}$, denoted $L(u)$, satisfies the equation 
$$R(u-v)\,L_1(u)\,L_2(v)=L_2(v)\,L_1(u)\,R(u-v).$$ 
Since $L(u)$ has the Gauss decomposition of the form \eqref{gauss decomp of R}, it induces a $\bC(\hbar)$-algebra homomorphism:
\begin{align*}
Y_\mu(\mathfrak{gl}_{1|1})\otimes_{\bC[\hbar]}\bC(\hbar)\to \prod_{\underline{\bd},\sW^{\mathrm{fr}},\sA^{\mathrm{fr}}}\End^{\bZ/2}_{\bC(\hbar)[\mathsf a^{\mathrm{fr}}]}(\cH^{\sW^{\mathrm{fr}}}_{\underline{\bd},\sA^{\mathrm{fr}}}),
\end{align*}
given by coefficients of $L(u)$. By definition of $\mathsf Y_\mu(Q,0,\cC)$, the above map induces a surjective $\bC(\hbar)$-algebra homomorphism
\begin{align}
\varrho_\mu\colon Y_\mu(\mathfrak{gl}_{1|1})\otimes_{\bC[\hbar]}\bC(\hbar)\twoheadrightarrow \mathsf Y_\mu(Q,0,\cC)=\mathsf Y_\mu(Q,0).
\end{align}

\begin{Lemma}\label{lem rho_0 isom}
The map $\varrho_0\colon Y(\mathfrak{gl}_{1|1})\otimes_{\bC[\hbar]}\bC(\hbar)\to\mathsf Y_0(Q,0)$ is an isomorphism.
\end{Lemma}

\begin{proof}
It remains to show that $\varrho_0$ is injective.
Viewing $\cH_{\mathfrak{c}}$ as a state space, its $Y(\mathfrak{gl}_{1|1})$ module structure is given by 
\begin{align*}
T(z)\mapsto \frac{z+\hbar P}{z+\hbar}.
\end{align*}
As a $\mathfrak{gl}_{1|1}$ module, $\cH_{\mathfrak{c}}$ is the fundamental representation, so $\mathfrak{gl}_{1|1}\to \End_{\bC(\hbar)}^{\bZ/2}(\cH_{\mathfrak{c}})$ is injective. In particular, $\varrho_0$ maps $\mathfrak{gl}_{1|1}$ injectively into $\mathsf Y_0(Q,0)$. We note that $\varrho_0(\mathfrak{gl}_{1|1})\subseteq \mathfrak{g}_{Q,0}$.

As in \cite[(1.6)]{G}, $Y(\mathfrak{gl}_{1|1})$ is a filtered algebra by setting $\deg X_{r}=r$ for $X=e,f,g_1,g_2$. The map $\varrho_0$ preserves the filtrations: $\varrho_0(F_kY(\mathfrak{gl}_{1|1})\otimes_{\bC[\hbar]}\bC(\hbar))\subseteq F_k\mathsf Y_0(Q,0)$, so it induces a map between the associated graded algebras: $\bar\varrho_0\colon \mathrm{gr}\, Y(\mathfrak{gl}_{1|1})\otimes_{\bC[\hbar]}\bC(\hbar)\to \mathrm{gr}\, \mathsf Y_0(Q,0)$. By the PBW theorem for $Y(\mathfrak{gl}_{1|1})$ \cite[Cor.~2.1]{G}, there is a natural isomorphism $$\mathrm{gr}\, Y(\mathfrak{gl}_{1|1})\otimes_{\bC[\hbar]}\bC(\hbar)\cong \mathcal U(\mathfrak{gl}_{1|1}[u]).$$ By Theorem \ref{thm grY}, we have a natural isomorphism $$\mathrm{gr}\, \mathsf Y_0(Q,0)\cong \mathcal U(\mathfrak{g}_{Q,0}[u]).$$ Moreover, the induced map $\bar\varrho_0\colon \mathcal U(\mathfrak{gl}_{1|1}[u])\to \mathcal U(\mathfrak{g}_{Q,0}[u])$ is obtained by joining the polynomial variable $u$ to the natural map $\mathfrak{gl}_{1|1}\to \mathfrak{g}_{Q,0}$ followed by taking enveloping algebras. As we have shown, $\mathfrak{gl}_{1|1}\to \mathfrak{g}_{Q,0}$ is injective, so $\bar\varrho_0\colon \mathrm{gr}\, Y(\mathfrak{gl}_{1|1})\otimes_{\bC[\hbar]}\bC(\hbar)\to \mathrm{gr}\, \mathsf Y_0(Q,0)$ is also injective. Hence $\varrho_0$ is injective.
\end{proof}

% It follows that 
% \begin{align*}
% \cU(\mathfrak{gl}_{1|1})\to \prod_{n\geqslant 1} \End_{\bC(\hbar)[a_1,\ldots,a_n]}^{\bZ/2}(\cH_{\mathfrak{c}}[a_1]\otimes \cdots \otimes \cH_{\mathfrak{c}}[a_n])
% \end{align*}
% is injective, where $\cH_{\mathfrak{c}}[a]$ is the $Y_0(\mathfrak{gl}_{1|1})$ module given by $T(z)\mapsto \frac{z-a+\hbar P}{z-a+\hbar}$, and the product is taken for all $n\geqslant 1$.

\begin{Remark}
By Theorem \ref{thm e f h as R matrix elements}, the generating series of $Y_\mu(\mathfrak{gl}_{1|1})$ generators are mapped to operators acting on $\cH^{\sW^{\mathrm{fr}}}_{\underline{\bd},\sA^{\mathrm{fr}}}$ as follows:
\begin{align*}
g_1(z)\mapsto c_{-1/z}((1-\hbar^{-1})\mathsf V)\,, \quad & g_2(z)/g_1(z)\mapsto z^{\mu}c_{-1/z}(\mathsf D_{\Out}-\mathsf D_{\In})\,,\\
e(z)\mapsto \frac{[\overline{\mathfrak{P}}(n+1,n,\underline{\bd})]}{z-c_1(\mathcal L)}\,,\quad & f(z)\mapsto \frac{(-1)^{\bd_{\Out}+n}\,\hbar}{z-c_1(\mathcal L)}[\overline{\mathfrak{P}}(n,n-1,\underline{\bd})]^{\mathrm{t}}\,.
\end{align*}
In the case of $\mu=0$, we have
\begin{align}\label{r-mat gl(1|1)}
-g_{1,0}\otimes g_{2,0}-g_{2,0}\otimes g_{1,0}+2g_{1,0}\otimes g_{1,0}+e_0\otimes f_0-f_0\otimes e_0\mapsto \hbar\,\pmb r,
\end{align}
where $\pmb r$ is the classical 
$R$-matrix \eqref{explicit r matrix}.
\end{Remark}

\subsection{Proof of Theorem \ref{thm ex trivial quiver_main}}
\begin{Definition}
Let $\mathfrak{Y}=\mathfrak{Y}^0\oplus \mathfrak{Y}^1$ be the $\bC(\hbar)$ Lie superalgebra with even part $\mathfrak{Y}^0=\Span_{\bC(\hbar)}\{\mathfrak{a}_r,\mathfrak{h}_r\}_{r\in \bZ_{\geqslant0}}$ and odd part $\mathfrak{Y}^1=\Span_{\bC(\hbar)}\{\mathfrak{e}_r,\mathfrak{f}_r\}_{r\in \bZ_{\geqslant0}}$, and Lie brackets on $\mathfrak{Y}$ are given by
\begin{equation}\label{bracket of frak Y}
\begin{gathered}
\mathfrak{h}_{r}\text{ is central for all $r$},\\
[\mathfrak{e}_n,\mathfrak{e}_m]=[\mathfrak{f}_n,\mathfrak{f}_m]=[\mathfrak{a}_n,\mathfrak{a}_m]=0,\\
[\mathfrak{e}_n,\mathfrak{f}_m]=-\hbar\,\mathfrak{h}_{n+m},\\
[\mathfrak{a}_n,\mathfrak{e}_m]=\mathfrak{e}_{n+m},\quad [\mathfrak{a}_n,\mathfrak{f}_m]=-\mathfrak{f}_{n+m}.
\end{gathered}
\end{equation}
Let $\mathcal U(\mathfrak{Y})$ be the enveloping algebra (over $\bC(\hbar)$) of $\mathfrak{Y}$. Define $\mathcal U_0(\mathfrak{Y})=\mathcal U(\mathfrak{Y})$, and for $\mu<0$, define $\mathcal U_\mu(\mathfrak{Y})$ to be the quotient of $\mathcal U(\mathfrak{Y})$ by the two-sided ideal generated by $\mathfrak{h}_{i}$ for all $i<-\mu-1$ and $\mathfrak{h}_{-\mu-1}-1$.
\end{Definition}

\begin{Remark}
$\mathfrak{Y}$ is isomorphic to $\mathfrak{gl}_{1|1}[u]$, we give the above explicit generators and Lie brackets and use notation $\mathfrak{Y}$ instead of $\mathfrak{gl}_{1|1}[u]$ because we will provide an explicit algebra isomorphism $\mathcal U(\mathfrak{Y})\cong Y(\mathfrak{gl}_{1|1})\otimes_{\bC[\hbar]}\bC(\hbar)$, and we would like to reserve the notation $\mathcal U(\mathfrak{gl}_{1|1}[u])$ that is isomorphic to $\mathrm{gr}\,Y(\mathfrak{gl}_{1|1})\otimes_{\bC[\hbar]}\bC(\hbar)$.
\end{Remark}

We denote $X(z)=\sum_{n\geqslant 0}X_nz^{-n-1}$ for $X=\mathfrak{e},\mathfrak{f},\mathfrak{a},\mathfrak{h}$. 

\begin{Lemma}\label{lem from frak Y to Y(gl(1|1))}
For any $\mu\in \bZ_{\leqslant 0}$, the map 
\begin{align}\label{from frak Y to Y(gl(1|1))}
\mathfrak{e}(z)\mapsto e(z),\quad \mathfrak{f}(z)\mapsto f(z),\quad \exp\left(\sum_{n\geqslant 0}\mathfrak{a}_n\psi_n(z)\right)\mapsto g_1(z),\quad \mathfrak{h}(z)\mapsto g_1(z)^{-1}g_2(z),
\end{align}
uniquely determines a surjective $\bC(\hbar)$ algebra homomorphism $\varphi_\mu\colon \mathcal U_\mu(\mathfrak{Y})\twoheadrightarrow Y_\mu(\mathfrak{gl}_{1|1})\otimes_{\bC[\hbar]}\bC(\hbar)$. Here $\{\psi_n(z)\}_{n\in \bZ_{\geqslant}}$ is a sequence of power series that is uniquely determined by the equation 
\begin{align*}
    \exp\left(\sum_{n\geqslant 0}a^n\psi_n(z)\right)=\frac{z-a}{z+\hbar-a}
\end{align*}
for a formal variable $a$.
\end{Lemma}

\begin{proof}
We first show that \eqref{from frak Y to Y(gl(1|1))} generates a surjective $\bC(\hbar)$ algebra homomorphism from $\mathcal U(\mathfrak{Y})$ to $Y_\mu(\mathfrak{gl}_{1|1})\otimes_{\bC[\hbar]}\bC(\hbar)$. From the RTT relation \eqref{RTT gl(1|1)} we deduce that
\begin{equation}\label{rel in Y(gl(1|1))}
\begin{split}
[g_i(u),g_j(v)]&=0\text{ for all }1\leqslant i,j\leqslant 2\,,\\
[e(u),e(v)]&=0=[f(u),f(v)]\,,\\
(u-v)[e(u),f(v)]&=\hbar\,(g_1(u)^{-1}g_2(u)-g_1(v)^{-1}g_2(v))\,,\\
(u-v)[g_i(u),e(v)]&=\hbar\, g_i(u)(e(u)-e(v))\,,\;\text{ for }\; 1\leqslant i\leqslant 2\,,\\
(u-v)[g_i(u),f(v)]&=\hbar\, (f(v)-f(u))g_i(u)\,,\;\text{ for }\; 1\leqslant i\leqslant 2\,.
\end{split}
\end{equation}
Then all relations except the last line in \eqref{bracket of frak Y} are apparently respected by the map \eqref{from frak Y to Y(gl(1|1))}. 

Notice that $\ad_{\mathfrak{a}_n}(\mathfrak{e}_m)=\ad_{\mathfrak{a}_1}^n(\mathfrak{e}_m)$.
Let $\mathfrak{A}(z)=\exp\left(\sum_{n\geqslant 0}\mathfrak{a}_n\psi_n(z)\right)$, then we have
\begin{align*}
\mathfrak{A}(z)\mathfrak{e}(w)\mathfrak{A}(z)^{-1}=\exp\left(\sum_{n\geqslant 0}\ad_{\mathfrak{a}_n}\psi_n(z)\right)\mathfrak{e}(w)=\exp\left(\sum_{n\geqslant 0}\ad_{\mathfrak{a}_1}^n\psi_n(z)\right)\mathfrak{e}(w)=\frac{z-\ad_{\mathfrak{a}_1}}{z+\hbar-\ad_{\mathfrak{a}_1}}e(w).
\end{align*}
Using the relation $\ad_{\mathfrak{a}_1}\mathfrak{e}(w)=w\,\mathfrak{e}(w)-\mathfrak{e}_0$, we can rewrite the above equation as follows
\begin{align*}
(z-w)[\mathfrak{A}(z),\mathfrak{e}(w)]=-\hbar\,\mathfrak{A}(z)\mathfrak{e}(w)+[\mathfrak{e}_0,\mathfrak{A}(z)].
\end{align*}
Taking $w=z$, we get $[\mathfrak{e}_0,\mathfrak{A}(z)]=\hbar\,\mathfrak{A}(z)\mathfrak{e}(z)$, plug it back to the above and we have
\begin{align}\label{A e commutator}
(z-w)[\mathfrak{A}(z),\mathfrak{e}(w)]=\hbar\,\mathfrak{A}(z)(\mathfrak{e}(z)-\mathfrak{e}(w))\,.
\end{align}
A similar argument gives
\begin{align}\label{A f commutator}
(z-w)[\mathfrak{A}(z),\mathfrak{f}(w)]=\hbar\,(\mathfrak{f}(w)-\mathfrak{f}(z))\,\mathfrak{A}(z)\,.
\end{align}
If we write $\mathfrak{A}(z)=1+\sum_{n\geqslant 0}\mathfrak{A}_n z^{-n-1}$, then for all $n\geqslant0$, $\mathfrak{A}_n=-\hbar\,\mathfrak{a}_n+f_n(\{\mathfrak{a}_{i}\}_{i<n})$ for some polynomial $f_n$ (we set $\mathfrak{a}_{-1}:=-1$). So \eqref{A e commutator} and \eqref{A f commutator} are equivalent to the last two relations in \eqref{bracket of frak Y} respectively. Compare with the last two equations in \eqref{rel in Y(gl(1|1))}, and we see that \eqref{from frak Y to Y(gl(1|1))} also respects the last two relations in \eqref{bracket of frak Y}, so \eqref{from frak Y to Y(gl(1|1))} generates a $\bC(\hbar)$ algebra homomorphism from $\mathcal U(\mathfrak{Y})$ to $Y_\mu(\mathfrak{gl}_{1|1})$. It is obviously surjective. When $\mu<0$, $\{\mathfrak{h}_i\}_{i<-\mu-1}\cup\{\mathfrak{h}_{-\mu-1}-1\}$ are in the kernel. This proves the lemma.
\end{proof}

Define $\widetilde{\varrho}_\mu$ to be the composition $\mathcal U_\mu(\mathfrak{Y})\xrightarrow{\varphi_\mu}Y_\mu(\mathfrak{gl}_{1|1})\otimes_{\bC[\hbar]}\bC(\hbar)\xrightarrow{{\varrho}_\mu}\mathsf Y_\mu(Q,0)$. It is surjective since both ${\varrho}_\mu$ and $\varphi_\mu$ are surjective. Theorem \ref{thm ex trivial quiver_main} follows from the next proposition.

\begin{Proposition}
For all $\mu\in \bZ_{\leqslant 0}$, the map $\widetilde{\varrho}_\mu\colon \mathcal U_\mu(\mathfrak{Y})\to\mathsf Y_\mu(Q,0)$ is injective.
\end{Proposition}

\begin{proof}
$\mu=0$ case. $\mathfrak{Y}$ is graded by setting $\deg X_n=n$ for $X=\mathfrak{e},\mathfrak{f},\mathfrak{a},\mathfrak{h}$. Let $F_\bullet\,\mathcal U(\mathfrak{Y})$ be the filtration induced by the grading on $\mathfrak{Y}$, then the map \eqref{from frak Y to Y(gl(1|1))} preserves the filtrations, i.e. $\varphi_0(F_k\,\mathcal U(\mathfrak{Y}))\subseteq F_k\,Y(\mathfrak{gl}_{1|1})\otimes_{\bC[\hbar]}\bC(\hbar)$. So $\varphi_0$ induces a natural map between associated graded algebras $\bar\varphi_0\colon \mathcal U(\mathfrak{Y})\to \mathrm{gr}\, Y(\mathfrak{gl}_{1|1})\otimes_{\bC[\hbar]}\bC(\hbar)$. Note that $\mathcal U(\mathfrak{Y})$ is canonically isomorphic to $\mathrm{gr}\,\mathcal U(\mathfrak{Y})$ since the filtration is induced by a grading. $\mathrm{gr}\, Y(\mathfrak{gl}_{1|1})\otimes_{\bC[\hbar]}\bC(\hbar)$ is isomorphic to $\mathcal U(\mathfrak{gl}_{1|1}[u])$. Let $\bar x_n$ be the image of $x_n$ in $\mathrm{gr}\, Y(\mathfrak{gl}_{1|1})$ for $x=e,f,g_1,g_2$. Then $\mathfrak{gl}_{1|1}[u]$ has a basis $\{\bar e_n,\bar f_n,\bar g_{1,n},\bar g_{2,n}\}_{n\in \bZ_{\geqslant 0}}$. It is straightforward to compute from \eqref{from frak Y to Y(gl(1|1))} that $\bar\varphi_0(\mathfrak{e}_n)=\bar e_n$, $\bar\varphi_0(\mathfrak{f}_n)=\bar f_n$, $\bar\varphi_0(\mathfrak{a}_n)=-\bar g_{1,n}/\hbar$, and $\bar\varphi_0(\mathfrak{h}_n)=\bar g_{2,n}-\bar g_{1,n}$. In particular, $\bar\varphi_0$ induces an isomorphism. Thus $\varphi_0$ is injective. Then the injectivity of $\widetilde{\varrho}_0$ follows from Lemma \ref{lem rho_0 isom}.

$\mu<0$ case. Define a $\bC(\hbar)$ algebra homomorphism $\mathfrak{S}_z^\mu\colon \mathcal U_\mu(\mathfrak{Y})\to \mathcal U(\mathfrak{Y})(\!(z^{-1})\!)$ as follows:
\begin{align*}
\mathfrak{S}_z^\mu(\mathfrak{e}_n)=(-1)^{\mu}\sum_{k=0}^{\infty}(-z)^{\mu-k}\binom{\mu}{k}\mathfrak{e}_{n+k},\quad \mathfrak{S}_z^\mu(\mathfrak{f}_n)=(-1)^\mu \mathfrak{f}_n,\quad \mathfrak{S}_z^\mu(\mathfrak{a}_n)=\mathfrak{a}_n,\quad \mathfrak{S}_z^\mu(\mathfrak{h}_n)=\sum_{k=0}^{n+\mu+1}(-z)^k\binom{\mu}{k}\mathfrak{h}_{n+\mu-k},
\end{align*}
where we set $\mathfrak{h}_i=0$ for $i<-1$ and $\mathfrak{h}_{-1}=0$ on the right hand side of the last equation. The proof that $\mathfrak{S}_z^\mu$ is an algebra homomorphism is straightforward and is omitted. 

We claim that $\mathfrak{S}_z^\mu$ is injective. Let $\mathfrak{E}_\mu,\mathfrak{F}_\mu,\mathfrak{A}_\mu,\mathfrak{H}_\mu$ be the subalgebras of $\mathcal U_\mu(\mathfrak{Y})$ which are generated by $\{\mathfrak{e}_r\}_{r\in \bZ_{\geqslant0}}$, $\{\mathfrak{f}_r\}_{r\in \bZ_{\geqslant0}}$, $\{\mathfrak{a}_r\}_{r\in \bZ_{\geqslant0}}$, and $\{\mathfrak{h}_s\}_{s\in \bZ_{\geqslant-\mu}}$ respectively. The PBW theorem for $\mathcal U(\mathfrak{Y})$ implies that $\mathcal U_\mu(\mathfrak{Y})$ has a basis (over $\bC(\hbar)$) consisting of ordered super monomials in $\{\mathfrak{e}_r,\mathfrak{f}_r,\mathfrak{a}_r,\mathfrak{h}_s\}_{r\in \bZ_{\geqslant0},s\in \bZ_{\geqslant-\mu}}$. In particular, $\mathcal U_\mu(\mathfrak{Y})\cong \mathfrak{E}_\mu\otimes\mathfrak{F}_\mu\otimes\mathfrak{A}_\mu\otimes\mathfrak{H}_\mu$ where the tensor product is over $\bC(\hbar)$. Then it is enough to show that the restrictions:
\begin{align*}
\mathfrak{S}_{z,\mathfrak{E}}^\mu\colon \mathfrak{E}_\mu\to \mathfrak{E}_0 [\![z^{-1}]\!],\quad \mathfrak{S}_{z,\mathfrak{F}}^\mu\colon \mathfrak{F}_\mu\to \mathfrak{F}_0,\quad \mathfrak{S}_{z,\mathfrak{A}}^\mu\colon \mathfrak{A}_\mu\to \mathfrak{A}_0,\quad \mathfrak{S}_{z,\mathfrak{H}}^\mu\colon \mathfrak{H}_\mu\to \mathfrak{H}_0[z],
\end{align*}
are injective. $\mathfrak{S}_{z,\mathfrak{F}}^\mu$ and $\mathfrak{S}_{z,\mathfrak{A}}^\mu\colon \mathfrak{A}_\mu\to \mathfrak{A}_0$ are obviously isomorphisms. The composition $\mathfrak{H}_\mu\xrightarrow{\mathfrak{S}_{z,\mathfrak{H}}^\mu} \mathfrak{H}_0[z]\xrightarrow{\mod z}\mathfrak{H}_0$ is given by $\mathfrak{h}_n\mapsto \mathfrak{h}_{n+\mu}$, which induces an isomorphism; thus $\mathfrak{S}_{z,\mathfrak{H}}^\mu$ is injective. $\mathfrak{S}_{z,\mathfrak{E}}^\mu$ decomposes into $$\mathfrak{S}_{z,\mathfrak{E}}^\mu=\underbrace{\{z^{-1}\}\circ\cdots\circ\{z^{-1}\}}_{-\mu\text{ copies}}\circ \overline{\mathfrak{S}}_{z,\mathfrak{E}}^\mu,$$ where $\overline{\mathfrak{S}}_{z,\mathfrak{E}}^\mu\colon \mathfrak{E}_\mu\to \mathfrak{E}_0 [\![z^{-1}]\!]$ is given by $\overline{\mathfrak{S}}_z^\mu(\mathfrak{e}_n)=(-1)^{\mu}\sum_{k=0}^{\infty}(-z)^{-k}\binom{\mu}{k}\mathfrak{e}_{n+k}$, $\{z^{-1}\}\colon \mathfrak{E}_0 [\![z^{-1}]\!]\to \mathfrak{E}_0 [\![z^{-1}]\!]$ is the endomorphism generated by $\mathfrak{e}_n\mapsto z^{-1}\mathfrak{e}_n$. The composition $\mathfrak{E}_\mu\xrightarrow{\overline{\mathfrak{S}}_{z,\mathfrak{E}}^\mu} \mathfrak{E}_0[z]\xrightarrow{\mod z}\mathfrak{E}_0$ is given by $\mathfrak{e}_n\mapsto (-1)^{\mu}\mathfrak{e}_{n}$, which induces an isomorphism; thus $\overline{\mathfrak{S}}_{z,\mathfrak{E}}^\mu$ is injective. On the other hand, $\mathfrak{E}_0$ is the free exterior algebra $\bigwedge^\bullet\left(\Span_{\bC(\hbar)}\{\mathfrak{e}_i\}_{i\in \bZ_{\geqslant 0}}\right)$. Suppose that
\begin{align*}
f=\sum_{m=0}^{\infty}\sum_{r\geqslant 0}f_{m,r}z^{-m}\in \ker(\{z^{-1}\}),
\end{align*}
where the second summation is finite for any fixed $m$, and $f_{m,r}\in \bigwedge^r\left(\Span_{\bC(\hbar)}\{\mathfrak{e}_i\}_{i\in \bZ_{\geqslant 0}}\right)$. Then we have
\begin{align*}
0=\{z^{-1}\}(f)=\sum_{m=0}^{\infty}\sum_{r\geqslant 0}f_{m,r}z^{-m-r},\;\text{ equivalently }\;\sum_{m+r=n}f_{m,r}=0\; \text{ for any }\; n\in \bZ_{\geqslant0}.
\end{align*}
$\bigwedge^r\left(\Span_{\bC(\hbar)}\{\mathfrak{e}_i\}_{i\in \bZ_{\geqslant 0}}\right)$ are linearly independent for different $r$, so $f_{m,r}=0$ for all $m$ and $r$. Thus $\ker(\{z^{-1}\})=0$. This shows that $\mathfrak{S}_{z,\mathfrak{E}}^\mu$ is injective and the claim follows.

Finally, straightforward comparison shows that $\mathfrak{S}_z^\mu$ is compatible with the wrong-way shift homomorphism (Definition \ref{def of wrong-way shift}):
\begin{align*}
    \widetilde{\varrho}_0\circ\mathfrak{S}_z^\mu=\sS^{0,\mu}_z\circ \widetilde{\varrho}_\mu.
\end{align*}
The left hand side is injective by what we have shown, so $\widetilde{\varrho}_\mu$ is injective.
\end{proof}

It follows from the proof that $\varphi_\mu\colon \mathcal U_\mu(\mathfrak{Y})\to Y_\mu(\mathfrak{gl}_{1|1})\otimes_{\bC[\hbar]}\bC(\hbar)$ is an isomorphism.

\begin{Corollary}\label{cor PBW Y(gl(1|1))}
$Y_\mu(\mathfrak{gl}_{1|1})$ is a free $\bC[\hbar]$-module with a basis given by ordered super monomials in 
$$\left\{e_{r},f_r,g_{1,r},g_{2,s}\right\}_{r\in \bZ_{\geqslant0},s\in \bZ_{\geqslant-\mu}}. $$
\end{Corollary}

\begin{proof}
The relations \eqref{rel in Y(gl(1|1))} implies that ordered super monomials in $\left\{e_{r},f_r,g_{1,r},g_{2,s}\right\}_{r\in \bZ_{\geqslant0},s\in \bZ_{\geqslant-\mu}}$ generate $Y_\mu(\mathfrak{gl}_{1|1})$ as a $\bC[\hbar]$-module. To show that these monomials are $\bC[\hbar]$-independent, it suffices to show that they are $\bC(\hbar)$ linear independent after localization. Note that ordered super monomials in $\left\{e_{r},f_r,g_{1,r},g_{2,s}\right\}_{r\in \bZ_{\geqslant0},s\in \bZ_{\geqslant-\mu}}$ can be $\bC(\hbar)$ linearly represented by ordered super monomials in $\left\{\rho_\mu(\mathfrak{e}_{r}),\rho_\mu(\mathfrak{f}_{r}),\rho_\mu(\mathfrak{a}_{r}),\rho_\mu(\mathfrak{h}_{s})\right\}_{r\in \bZ_{\geqslant0},s\in \bZ_{\geqslant-\mu}}$ and vice versa, and the latter set of monomials is a basis of $\mathcal U_\mu(\mathfrak{gl}_{1|1})$ by the PBW theorem for $\mathcal U(\mathfrak{gl}_{1|1})$. This proves the linear independence.
\end{proof}

\subsection{Quantum multiplication by divisors}\label{sec qm by div_trivial}

Consider the framing dimension $\underline{\bd}=(r_1,r_2)$ with $r_1\geqslant r_2$ in the following framed quiver with zero potential:
\begin{equation*}
\begin{tikzpicture}[x={(1cm,0cm)}, y={(0cm,1cm)}, baseline=0cm]
  % Nodes
  \node[draw,circle,fill=white] (Gauge) at (0,0) {$n$};
  \node[draw,rectangle,fill=white] (Framing1) at (-2,0) {$r_1$};
  % \node (Framingdot) at (0,-2) {$\cdots$};
  % \node (Arrowdot) at (0,-1.5) {$\cdots$};
  \node[draw,rectangle,fill=white] (Framingk) at (2,0) {$r_2$};

  % Edges
  \draw[<-] (Gauge.180) -- (Framing1.0) node[midway,above] {\scriptsize $A$};
  % \draw[->] (Gauge.240) -- (Framing1.60) node[midway,right] {\scriptsize $B_1$};

  \draw[->] (Gauge.0) -- (Framingk.180) node[midway,above] {\scriptsize $B$};
  % \draw[<-] (Gauge.300) -- (Framingk.120) node[midway,left] {\scriptsize $A_2$};

  % Loop
\end{tikzpicture}
\qquad
\sW=0. 
\end{equation*} 
Denote $X_{n,r_1,r_2}=\cM(n,\underline{\bd})$. Apply Theorem \ref{thm qm div for asym_sp inj} to $X_{n,r_1,r_2}$ and we get an explicit formula of modified quantum multiplication by divisor on $X_{n,r_1,r_2}$. 
Note that $X_{r_1,r_1,r_2}$ is an affine space, so we shall focus on the $r_1>n$ cases.

\begin{Proposition}\label{prop on qmbd on gl(1|1)}
Assume $r_1>n$. Let $\mathcal O(1)$ be the determinant of tautological bundle $\mathsf V$ on $X_{n,r_1,r_2}$, then 
%the modified quantum multiplication by $c_1(\mathcal O(1))$ is the following
\begin{align*}
c_1(\mathcal O(1))\,\widetilde{\star}\,\cdot =
\begin{cases}
c_1(\mathcal O(1))\cup \cdot + \hbar^{-1}\frac{z}{1+z}e_0f_0, &\text{ if $r_1=r_2$ },\\
c_1(\mathcal O(1))\cup \cdot +\hbar^{-1}z\,e_0f_0, &\text{ if $r_1>r_2$ },
\end{cases}
\end{align*}
where $z$ is the parameter of curve classes, $e_0,f_0$ are part of the generators of $Y_{r_2-r_1}(\mathfrak{gl}_{1|1})$ in Definition \ref{def of shifted Y(gl(1|1))}.
\end{Proposition}

\begin{proof}
By Theorem \ref{thm compare with MO yangian} and the formula for the classical $R$-matrix \eqref{r-mat gl(1|1)}, the Casimir operators for $\mathsf Y_0(Q,0)$ are $$\mathrm{Cas}_{1}=\hbar^{-1}e_0f_0, \quad \mathrm{Cas}_{i}=0, \,\, \mathrm{if}\, \, i>1.$$ 
Then by Theorem \ref{thm on qm div for sym}, we have 
$$c_1(\mathcal O(1))\,\widetilde{\star}\,\cdot=c_1(\mathcal O(1))\cup \cdot + \hbar^{-1}\frac{z}{1+z}e_0f_0+\mathrm{const}. $$ 
The constant is given by $-\hbar^{-1}\frac{z}{1+z}e_0f_0$ acting on $1\in H^{\sT_0}(X_{n,r_1,r_2})$. According to Theorem \ref{cor shifted yangian action}, $f_0(1)$ is given by the pushforward of the fundamental class $[\overline{\mathfrak{P}}(n,n-1,\underline{\bd})]^{\mathrm{t}}$ to $X_{n-1,r_1,r_2}$. It follows from \eqref{two complexes} that $\overline{\mathfrak{P}}(n,n-1,\underline{\bd})$ is a $\bP^{r_1-n}$-bundle over $X_{n-1,r_1,r_2}$, so $f_0(1)=0$ by our assumption that $r_1>n$. This proves the $r_1=r_2$ case.

For the $r_1>r_2$ case, it suffices to show that 
$$\mathrm{Cas}_{1,r_2-r_1}= \hbar^{-1}e_0f_0, $$ 
where $e_0,f_0$ are part of the generators of $Y_{r_2-r_1}(\mathfrak{gl}_2)$. We put bar on the generators of $Y_{r_2-r_1}(\mathfrak{gl}_2)$ to distinguish from their counterparts of $Y_{0}(\mathfrak{gl}_2)$. From the proof of Lemma \ref{lem shift map}, we have equations of operators: 
$$ i^*e(u)\,p^{*}=H_{z}\,\bar e(u)\,H_{z}^{-1}=\frac{(z-c_1(\mathcal L))^{r_1-r_2}}{u-c_1(\mathcal L)}\,[\overline{\mathfrak{P}}(n+1,n,\underline{\bd})]
=z^{r_1-r_2}\bar e(u)+\text{lower order in $z$}, $$
$$i^*f(u)\,p^{*}=(-1)^{r_1-r_2}\bar f(u).$$
By expanding in the $u$-power, we obtain 
$$i^*e_r\,p^{*}=z^{r_1-r_2}\, \bar e_r+\text{lower order in $z$}, \quad i^*f_r\,p^{*}=(-1)^{r_1-r_2}\bar f_r. $$ 
This implies that 
\begin{equation*}
i^*\mathrm{Cas}_1\,p^{*}=\hbar^{-1}i^*e_0f_0\,p^{*}=(-z)^{r_1-r_2}\,\hbar^{-1}\, \bar e_0\bar f_0+\text{lower order in $z$},
\end{equation*}
therefore $\mathrm{Cas}_{1,r_2-r_1}=\lim_{z\to \infty}(-z)^{r_2-r_1}i^*\mathrm{Cas}_1\,p^{*}=\hbar^{-1}\bar e_0\bar f_0$ by Definition \ref{def of shifted Casimir}.
\end{proof}

\section{Example: Jordan quiver and modules of shifted Yangians of \texorpdfstring{$\mathfrak{gl}_2$}{gl(2)}}\label{sect on higher spin}
In this section, we consider the example of the Jordan quiver consisting of one node and one edge loop:
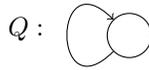
\begin{equation*}
Q:
\begin{tikzpicture}[x={(1cm,0cm)}, y={(0cm,1cm)}, baseline=0cm]
  % Nodes
  \node[draw,circle,fill=white] (Gauge) at (0,0) {$\phantom{n}$};
  % \node[draw,rectangle,fill=white] (Framing) at (1,0) {};
  % \node (Z) at (-.73,0) {};
  % Edges
 % \draw[<-] (Gauge.0) -- (Framing.180) node[midway,below] {\scriptsize $\varphi$};

  % Loop
  \draw[->,looseness=7] (Gauge.225) to[out=225,in=135] (Gauge.135);
\end{tikzpicture}
\end{equation*}
We regard $Q$ as the tripled quiver (Example \ref{ex doubled vs tripled}) associated to the $A_1$ quiver (with one node with no edge). 
Let $\sT_0=\bC^*_\hbar$ be the torus that scales the loop with weight $-1$. Fix a cyclic stability $\theta<0$ when we consider associated quiver varieties.  

%We refer to Example \ref{ex doubled vs tripled} for the definition of tripled quivers.
%$$\cM(n,\underline{\bd})=\cM_\theta(n,\underline{\bd}), \quad \cM(\underline{\bd})=\bigsqcup_{n\in \bN} \cM(n,\underline{\bd}). $$
%Recall that for symmetric framing $\mathbf d_{\In}=\mathbf d_{\Out}=\bd$, we write 
%$$\cM(n,\bd)=\cM(n,\underline{\bd}), \quad \cM(\bd)=\cM(\underline{\bd}).$$ 

The main theorem of this section is an \textit{explicit} description of the shifted Yangian $\mathsf Y_{\mu}(Q,0)$ for $Q$ \eqref{equ on rtt yangian}. 
\begin{Theorem}\label{thm ex Jordan_main}
Let $Q$ be the Jordan quiver. For arbitrary $\mu\in \bZ_{\leqslant 0}$, there is a natural algebra isomorphism
\begin{align*}
    Y_{\mu}(\mathfrak{gl}_2)\otimes_{\bC[\hbar]} \bC(\hbar)\cong \mathsf Y_{\mu}(Q,0)\:,
\end{align*}
where $Y_{\mu}(\mathfrak{gl}_2)$ is the Frassek-Pestun-Tsymbaliuk 
$\mu$-shifted Yangian of $\mathfrak{gl}_2$ (Definition \ref{def of shifted Y(gl_2)}).
%and $\mathsf Y_{-m}(Q,0)$ is the RTT shifted Yangian \eqref{equ on rtt yangian}.
\end{Theorem}

The $\mu=0$ case of Theorem \ref{thm ex Jordan_main} follows from Theorem \ref{thm compare with MO yangian} and the fact that $\mathsf Y^{\mathrm{MO}}_Q\cong Y_0(\mathfrak{gl}_2)$ \cite[Prop.\,11.3.1]{MO}. In \S \ref{sect on proof of thm jordan}, we give a uniform proof of Theorem \ref{thm ex Jordan_main} for all $\mu\in \bZ_{\leqslant 0}$.

We also study several representations of shifted Yangian $Y_{\mu}(\mathfrak{gl}_2)$ constructed from critical cohomology, including 
higher spin, dual Verma and prefundamental modules.
We perform \textit{explicit} computations of cohomological stable envelopes and root $R$-matrices in \S \ref{sect on comp stab of gl2} and 
\S \ref{sect on comp rmatrix of gl2}. We give a detailed study of the \textit{fusion procedure} from geometric construction in \S \ref{sect on fusion gl2} (see Corollary \ref{cor fusion formula} for the fusion formula). 
We deduce the explicit formula for quantum multiplication by divisors on the framed quiver varieties (Proposition \ref{prop on qmbd on gl2}). 
%from which we show that quantum multiplication by divisors is in general not compatible with dimensional reduction (Remark \ref{rm on Q not cp with dim red}).

\subsection{Bases of \texorpdfstring{$\mathcal H_N$}{HN}, \texorpdfstring{$\mathcal V$}{V}, and \texorpdfstring{$\mathcal F$}{F}}\label{sec bases of H_N, V, F}

The following are three basic cases of framings and potentials:
\begin{enumerate}

\item (Higher spin module) $\underline{\bd}=\mathbf 1=(1,1)$, $\sW_N=B\Phi^NA$ (with $N\in \bZ_{\geqslant 0}$), and we require that $\bC^*_\hbar$ scales $B$ with weight $N$ and fixes $A$. Denote $\sw_N=\tr(\sW_N)$.
\begin{equation*}
\begin{tikzpicture}[x={(1cm,0cm)}, y={(0cm,1cm)}, baseline=0cm]
  % Nodes
  \node[draw,circle,fill=white] (Gauge) at (0,0){$\phantom{n}$}; %{$\bv$}  {$\phantom{n}$};
  \node[draw,rectangle,fill=white] (Framing) at (2,0) {1};
  \node (Z) at (-1,0) {\scriptsize $\Phi$};
  % Edges
  \draw[<-] (Gauge.340) -- (Framing.205) node[midway,below] {\scriptsize $A$};
  \draw[->] (Gauge.20) -- (Framing.155) node[midway,above] {\scriptsize $B$};

  % Loop
  \draw[->,looseness=7] (Gauge.225) to[out=225,in=135] (Gauge.135);
\end{tikzpicture}
\qquad \sW_N=B\Phi^NA.
\end{equation*}
%\begin{equation*} \xymatrix{\text{\tiny{\boxed{1}}}  \ar@/^0.3pc/[r]^{A}  & \ar@/^0.3pc/[l]^{B} \text{\textcircled{v}}  \ar@(dr,ur)_{\Phi}  }
%\end{equation*}
We define  
\begin{equation}\label{equ on HNn}
   \mathcal H_N(n):=H^{\bC^*_\hbar}(\cM(n,\mathbf 1),\sw_N), \quad  \mathcal H_N:=H^{\bC^*_\hbar}(\cM(\mathbf 1),\sw_N)=\bigoplus_{n\in \bZ_{\geqslant 0}} \mathcal H_N(n).
\end{equation}

\item (Dual Verma module) $\underline{\bd}=\mathbf 1=(1,1)$, $\sW=0$, and take another torus $\bC^*_t$ which scales $B$ with weight $1$ and fixes other quiver data.
    \begin{equation*}
\begin{tikzpicture}[x={(1cm,0cm)}, y={(0cm,1cm)}, baseline=0cm]
  % Nodes
  \node[draw,circle,fill=white] (Gauge) at (0,0){$\phantom{n}$}; % {$\phantom{n}$};
  \node[draw,rectangle,fill=white] (Framing) at (2,0) {1};
  \node (Z) at (-1,0) {\scriptsize $\Phi$};
  % Edges
  \draw[<-] (Gauge.340) -- (Framing.205) node[midway,below] {\scriptsize $A$};
  \draw[->] (Gauge.20) -- (Framing.155) node[midway,above] {\scriptsize $B$};

  % Loop
  \draw[->,looseness=7] (Gauge.225) to[out=225,in=135] (Gauge.135);
\end{tikzpicture}
\qquad \sW=0.
\end{equation*}
We define 
\begin{align*}
    \mathcal V(n):=H^{\bC^*_\hbar\times \bC^*_t}(\cM(n,\mathbf 1)), \quad \mathcal 
    V:=H^{\bC^*_\hbar\times \bC^*_t}(\cM(\mathbf 1))=\bigoplus_{n\in \bZ_{\geqslant 0}}\mathcal V(n).
\end{align*}

\item (Prefundamental module) $\underline{\bd}=\underline{\mathbf 0}^{\mathbf 1}=(1,0)$, $\sW=0$.
    \begin{equation*}
\begin{tikzpicture}[x={(1cm,0cm)}, y={(0cm,1cm)}, baseline=0cm]
  % Nodes
  \node[draw,circle,fill=white] (Gauge) at (0,0){$\phantom{n}$}; % {$\phantom{n}$};
  \node[draw,rectangle,fill=white] (Framing) at (2,0) {1};
  \node (Z) at (-1,0) {\scriptsize $\Phi$};
  % Edges
  \draw[<-] (Gauge.0) -- (Framing.180) node[midway,below] {\scriptsize $A$};
  % \draw[->] (Gauge.20) -- (Framing.155) node[midway,above] {\scriptsize $B$};

  % Loop
  \draw[->,looseness=7] (Gauge.225) to[out=225,in=135] (Gauge.135);
\end{tikzpicture}
\qquad \sW=0.
\end{equation*}
We define 
\begin{equation}\label{equ on prefund}
\mathcal F(n):=H^{\bC^*_\hbar}(\cM(n,\underline{\mathbf 0}^{\mathbf 1})), \quad 
\mathcal F:=H^{\bC^*_\hbar}(\cM(\underline{\mathbf 0}^{\mathbf 1}))=\bigoplus_{n\in \bZ_{\geqslant 0}}\mathcal F(n).
\end{equation}
\end{enumerate}
Note that there is a natural isomorphism $\cV\cong \cF[t]$ since $\cM(\mathbf 1)$ is a vector bundle on $\cM(\underline{\mathbf 0}^{\mathbf 1})$, but we will see that their stable envelopes and $R$-matrices are not extensions by scalars.

To present stable envelopes and $R$-matrices, we need bases of $\mathcal H_N$, $\mathcal V$, and $\mathcal F$. 

As $\cM(n,\mathbf 1)$ is a vector bundle on $\cM(n,\underline{\mathbf 0}^{\mathbf 1})$ with fibers given by the descent of $B\in \Hom(\bC^n,\bC)$ to the GIT quotient,
by dimensional reduction \cite[Thm.~A.1]{Dav}, there is an isomorphism
\begin{align*}
    H^{\bC^*_\hbar}(\cM(n,\mathbf 1),\sw_N)\cong H^{\bC^*_\hbar}(Z'_n),
\end{align*}
where $Z'_n\subset \cM(n,\underline{\mathbf 0}^{\mathbf 1})$ is the zero locus of the section $\Phi^NA\in \Hom(\bC,\bC^n)$. Note that 
$$\cM(n,\underline{\mathbf 0}^{\mathbf 1})\cong \Hilb^n(\bC)\cong S^n\bC\cong \bC^n, $$ 
and the equation $\Phi^NA=0$ translates to requiring that $n$ points on $\bC$ are supported on $\Spec \bC[z]/(z^N)$; therefore
\begin{align*}
    Z'_n=\begin{cases}
        \text{the zero of $\bC^n$},& n\leqslant N,\\
       \quad \quad \quad \emptyset, & n>N.
    \end{cases}
\end{align*}
In particular, $H^{\bC^*_\hbar}(Z'_n)=\bC[\hbar]\cdot[Z'_n]$. Let $Z_n$ be the preimage of $Z'_n$ in $\cM(n,\mathbf 1)$, then we have:
\begin{Lemma}
\begin{align*}
    \mathcal H_N(n)=H^{\bC^*_\hbar}(\cM(n,\mathbf 1),\sw_N)\cong \begin{cases}
        \bC[\hbar]\cdot[Z_n], & n\leqslant N,\\
       \quad\quad  0, & n>N,
    \end{cases}
\end{align*}
where the isomorphism is induced by the canonical map $H^{\bC^*_\hbar}(Z_n)\to H^{\bC^*_\hbar}(Z(\sw_N))\xrightarrow{\can}H^{\bC^*_\hbar}(\cM(n,\mathbf 1),\sw_N)$.
\end{Lemma}
For $\mathcal V$ and $\mathcal F$, the natural choices of bases are given by fundamental classes:
\begin{align*}
    \mathcal V(n)\cong \bC[\hbar,t]\cdot[\cM(n,\mathbf 1)],\quad \mathcal F(n)\cong \bC[\hbar]\cdot[\cM(n,\underline{\mathbf 0}^{\mathbf 1})].
\end{align*}
\begin{Lemma}\label{lem on injec sp higher spin}
There is an injective specialization map 
\begin{align*}
    \mathsf{sp}\colon \mathcal H_N\to \mathcal V\big|_{t=N\hbar}, \quad\mathsf{sp}([Z_n])=n!\hbar^n [\cM(n,\mathbf 1)] \quad ( n\leqslant N).
\end{align*}
\end{Lemma}
\begin{proof}
Note that $\mathcal V|_{t=N\hbar}$ is $H^{\bC^*_\hbar}(\cM(n,\mathbf 1))$, where $\bC^*_\hbar$ scales $B$ with weight $N$ and $\bC^*_\hbar$ fixes $A$. By 
\cite[Ex.~7.19]{COZZ}, the specialization map exits, and the explicit formula uses $e^{\bC^*_\hbar}(N_{Z_n/\cM(n,\mathbf 1)})=n!\hbar^n$. In particular, it is injective.
\end{proof}
For later use, we introduce identifications of vector spaces 
\begin{align*}
    \mathcal V\cong \bC[\hbar,t][\mathbf x],\quad \mathcal F\cong \bC[\hbar][\mathbf x],
\end{align*}
such that bases are identified by
\begin{align}\label{bases of V and F}
    [\cM(n,\mathbf 1)]\;\longleftrightarrow \;\frac{\mathbf x^n}{n!}\;\longleftrightarrow \;[\cM(n,\underline{\mathbf 0}^{\mathbf 1})].
\end{align}

\begin{Remark}
The critical cohomologies $\cH_N(n)$, $\cV(n)$ and $\cF(n)$ are pure, then by Remark \ref{rmk purity and tensor prod}, Thom-Sebastiani Theorem applies to the product $\prod_{i=1}^a\cM(n_i,\mathbf 1)\times \prod_{j=1}^b\cM(m_j,\underline{\mathbf 0}^{\mathbf 1})$ with potential $=\boxplus_{i=1}^{a'} \sw_{N_i}$ $(a'\leqslant a)$ without the need of localization.
\end{Remark}

%\subsection{Stable envelope for \texorpdfstring{$\cM(n,r\underline{\bd})$}{M(n,kd)}}

\subsection{Computations of stable envelopes}\label{sect on comp stab of gl2}

Consider the quiver variety $\cM(n,\mathbf r)$ depicted as below.
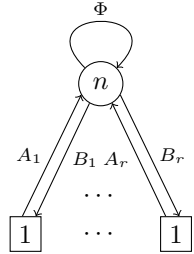
\begin{equation*}
\begin{tikzpicture}[x={(1cm,0cm)}, y={(0cm,1cm)}, baseline=0cm]
  % Nodes
  \node[draw,circle,fill=white] (Gauge) at (0,0) {$n$};
  \node[draw,rectangle,fill=white] (Framing1) at (-1,-2) {1};
  \node (Framingdot) at (0,-2) {$\cdots$};
  \node (Arrowdot) at (0,-1.5) {$\cdots$};
  \node[draw,rectangle,fill=white] (Framingk) at (1,-2) {1};
  \node (Z) at (0,1) {\scriptsize $\Phi$};
  % Edges
  \draw[<-] (Gauge.210) -- (Framing1.100) node[midway,left] {\scriptsize $A_1$};
  \draw[->] (Gauge.240) -- (Framing1.60) node[midway,right] {\scriptsize $B_1$};

  \draw[->] (Gauge.330) -- (Framingk.80) node[midway,right] {\scriptsize $B_r$};
  \draw[<-] (Gauge.300) -- (Framingk.120) node[midway,left] {\scriptsize $A_r$};

  % Loop
  \draw[->,looseness=7] (Gauge.135) to[out=135,in=45] (Gauge.45);
\end{tikzpicture}
% \sw=\sum_{i=1}^r B_i\Phi
\end{equation*}
Define an $\sA\cong (\bC^*)^r$ action by $\mathbf r=\sum_{i=1}^r z_i\cdot\mathbf 1^{(i)}$. Let $\sT=\sA\times \bC^*_\hbar\times \prod_{i=1}^r\bC^*_{t_i}$, where $(q,p_1,\cdots,p_r)\in \bC^*_\hbar\times \prod_{i=1}^r\bC^*_{t_i}$ acts on quiver data by
\begin{align*}
    (\Phi,A_1,B_1,\ldots,A_k,B_k)\mapsto (q^{-1}\Phi,A_1,p_1 B_1,\ldots,A_r,p_rB_r).
\end{align*}
We have
\begin{align*}
    \cM(n,\mathbf r)^\sA=\bigsqcup_{\substack{\vec{n}=(n_1,\cdots,n_r)\\n_1+\cdots+n_r=n}}F_{\vec{n}}, \quad \text{where}\quad F_{\vec{n}}= \prod_{i=1}^r \cM(n_i,\mathbf 1).
\end{align*}
Let $\fC$ be the chamber $\{z_1>\cdots>z_r\}$. Denote $H_{\GL_n}(\pt)=\bC[x_1,\ldots,x_r]$. 
Then \cite[Prop.~9.7]{COZZ} implies:
\begin{Proposition}
We have the following explicit formula for the stable envelope:
$$\Stab_\fC\colon H^\sT(\cM(n,\mathbf r)^\sA)\to H^\sT(\cM(n,\mathbf r)),$$
\begin{equation}\label{stab for higher spin_w=0}
\begin{split}
    &\Stab_\fC([F_{\vec{n}}])=\mathbf W_{\fC,\vec{n}}(\{x_\alpha\},\{z_i\},\{t_i\})\cdot[\cM(n,\mathbf r)],\\
    \mathbf W_{\fC,\vec{n}}(\{x_\alpha\},\{z_i\},\{t_i\})&=\sum_{\sigma\in\mathrm{Sh}\{S_1,\ldots,S_r\}}\sigma\left[\prod_{1\leqslant i<j\leqslant r}\left(\prod_{\alpha\in S_j} (x_\alpha-z_i)\right)\left(\prod_{\beta\in S_i} (z_j+t_j-x_\beta)\right)\left(\prod_{\substack{\alpha\in S_j\\\beta\in S_i}} \frac{x_\alpha-x_\beta-\hbar}{x_\alpha-x_\beta}\right)\right],
\end{split} 
\end{equation}
where $S_1\sqcup S_2 \sqcup \cdots \sqcup S_r=\{1,2,\ldots,n\}$ is a fixed decomposition with $|S_i|=n_i$, and $\mathrm{Sh}\{S_1,\ldots,S_r\}$ is the set of permutations $\sigma\in \mathfrak{S}_n$ that preserves the order of elements in each $S_i$, that is, $\gamma<\nu\Leftrightarrow \sigma(\gamma)<\sigma(\nu)$ for all pairs $\gamma,\nu\in S_i$. The matrix elements of $\Stab_\fC$ is given by
\begin{align*}
    \Stab_\fC([F_{\vec{n}}])\bigg|_{[F_{\vec{n}'}]}=\mathbf W_{\fC,\vec{n}}(\{x_\alpha\},\{z_i\},\{t_i\})\bigg|_{x_\alpha\mapsto z_{p_{\vec{n}'}(\alpha)}+l_{\vec{n}'}(\alpha)\hbar}[F_{\vec{n}'}].
\end{align*}
Here the notation $p_{\vec{n}'}(\alpha)$ and $l_{\vec{n}'}(\alpha)$ is determined by the conditions that $\alpha$ belongs to the subset $S'_{p(\alpha)}$ and there are exactly $l(\alpha)$ elements in $S'_{p(\alpha)}$ which are smaller than $\alpha$, where $S'_1\sqcup S'_2 \sqcup \cdots \sqcup S'_r=\{1,2,\ldots,n\}$ is a fixed decomposition with $|S'_i|=n'_i$. 
\end{Proposition}
\begin{Remark}
The $K$-theoretic version of formula \eqref{stab for higher spin_w=0} first appears in \cite[(60), (61)]{Ta}. The function 
$\mathbf W_{\fC,\vec{n}}$ is called the ``weight function'' in \cite{YZJ}.
\end{Remark}
To get the stable envelope for the opposite chamber $-\fC=\{z_1<\cdots<z_r\}$, we change 
$\mathbf W_{\fC,\vec{n}}(\{x_\alpha\},\{z_i\},\{t_i\})$ in \eqref{stab for higher spin_w=0} to
\begin{align*}
    \mathbf W_{-\fC,\vec{n}}(\{x_\alpha\},\{z_i\},\{t_i\})=\sum_{\sigma\in\mathrm{Sh}\{S_1,\ldots,S_r\}}\sigma\left[\prod_{1\leqslant i<j\leqslant r}\left(\prod_{\alpha\in S_i} (x_\alpha-z_j)\right)\left(\prod_{\beta\in S_j} (z_i+t_i-x_\beta)\right)\left(\prod_{\substack{\alpha\in S_i\\\beta\in S_j}} \frac{x_\alpha-x_\beta-\hbar}{x_\alpha-x_\beta}\right)\right].
\end{align*}

\subsubsection{Transpose of stable envelope}
The transpose $\Stab_\fC^{\tau}\colon H^\sT(\cM(n,\mathbf r))_\loc\to H^\sT(\cM(n,\mathbf r)^\sA)_\loc$ is the map induced by the transpose of stable envelope correspondence (\cite[\S 3.4.2]{COZZ}),~i.e.
\begin{align*}
    \Stab_\fC^{\tau}([F_{\vec{n}'}])\bigg|_{[F_{\vec{n}}]}&=\frac{\int_{[F_{\vec{n}'}]}1}{\int_{[F_{\vec{n}}]}1} \Stab_\fC([F_{\vec{n}}])\bigg|_{[F_{\vec{n}'}]}\\
    &=\left(\prod_{i=1}^r \frac{n_i!\prod_{k=0}^{n_i-1}(t_i-k\hbar)}{n'_i!\prod_{m=0}^{n'_i-1}(t_i-m\hbar)}\right)\mathbf W_{\fC,\vec{n}}(\{x_\alpha\},\{z_i\},\{t_i\})\bigg|_{x_\alpha\mapsto z_{p_{\vec{n}'}(\alpha)}+l_{\vec{n}'}(\alpha)\hbar}\cdot[F_{\vec{n}}].
\end{align*}

\subsubsection{$R$-matrix}

The $R$-matrix $R_{-\fC,\fC}\in \End(H^\sT(\cM(n,\mathbf r)^\sA)_\loc)$ is given by
\begin{align*}
    R_{-\fC,\fC}([F_{\vec{n}}])\bigg|_{F_{\vec{n}'}}=\sum_{\vec{n}''}\frac{1}{e^\sT(N_{F_{\vec{n}''}/\cM(n,\mathbf r)})} \Stab_\fC^{\tau}([F_{\vec{n}''}])\bigg|_{[F_{\vec{n}'}]} \Stab_\fC([F_{\vec{n}}])\bigg|_{[F_{\vec{n}''}]}\cdot[F_{\vec{n}'}].
\end{align*}

\subsubsection{Including potential \texorpdfstring{$\sw=\sum_{i=1}^rB_i\Phi^{N_i}A_i$}{for higher spin modules}}\label{sec higher spin R-matrix} 
To preserve the potential, we set $t_i=N_i\hbar$ with $\sT'=\sA\times \bC^*_\hbar$. 

Thom-Sebastiani isomorphism implies 
\begin{align*}
    H^{\sT'}(F_{\vec{n}},\sw)\cong \bigotimes_{i=1}^rH^{\sT'}(\cM(n_i,\mathbf 1),\sw_{N_i})\cong\begin{cases}
        \bC[\hbar,z_1,\ldots,z_r]\cdot [Z_{n_1}]\otimes\cdots \otimes [Z_{n_r}],& \text{if }\:\forall\, i, n_i\leqslant N_i,\\
        \quad \quad \quad\quad \quad\quad 0, & \text{otherwise}.
    \end{cases} 
\end{align*}
By \cite[Prop.~7.18]{COZZ} and Lemma \ref{lem on injec sp higher spin}, we have: 
\begin{Proposition}
The stable envelope 
$$\Stab_\fC\colon H^{\sT'}(\cM(n,\mathbf r)^\sA,\sw)\to H^{\sT'}(\cM(n,\mathbf r),\sw)$$ is the restriction of the zero potential case \eqref{stab for higher spin_w=0} to the image of specialization map. Namely,
\begin{align}\label{stab for higher spin}
\mathsf{sp}\circ \Stab_\fC([Z_{n_1}]\otimes\cdots \otimes [Z_{n_r}])=\hbar^n\left(\prod_{i=1}^r n_i!\right)\mathbf W_{\fC,\vec{n}}(\{x_\alpha\},\{z_i\},\{t_i\})\bigg|_{t_i\mapsto N_i\hbar}\cdot[\cM(n,\mathbf r)].
\end{align}
The matrix elements of $\Stab_\fC$ are given by
\begin{align*}
    \Stab_\fC([Z_{n_1}]\otimes\cdots \otimes [Z_{n_r}])\bigg|_{F_{\vec{n}'}}=\left(\prod_{i=1}^r \frac{n_i!}{n'_i!}\right) \mathbf W_{\fC,\vec{n}}(\{x_\alpha\},\{z_i\},\{t_i\})\Bigg|_{\substack{t_i\mapsto N_i\hbar\\x_\alpha\mapsto z_{p_{\vec{n}'}(\alpha)}+l_{\vec{n}'}(\alpha)\hbar}}\cdot [Z_{n'_1}]\otimes\cdots \otimes [Z_{n'_r}].
\end{align*}
And the transpose is given by
\begin{align*}
    \Stab^\tau_\fC([Z_{n'_1}]\otimes\cdots \otimes [Z_{n'_r}])\bigg|_{F_{\vec{n}}}=\left(\prod_{i=1}^r \frac{\prod_{k=0}^{n_i-1}(N_i-k)}{\prod_{m=0}^{n'_i-1}(N_i-m)}\right) \mathbf W_{\fC,\vec{n}}(\{x_\alpha\},\{z_i\},\{t_i\})\Bigg|_{\substack{t_i\mapsto N_i\hbar\\x_\alpha\mapsto z_{p_{\vec{n}'}(\alpha)}+l_{\vec{n}'}(\alpha)\hbar}}\cdot [Z_{n_1}]\otimes\cdots \otimes [Z_{n_r}].
\end{align*}
\end{Proposition}

\subsection{Root \texorpdfstring{$R$}{R}-matrices}\label{sect on comp rmatrix of gl2} 
We compute several root $R$-matrices explicitly. 
%using the formula in the previous section.

\subsubsection{\texorpdfstring{$R$}{R}-matrix for \texorpdfstring{$\mathcal H_1\otimes\mathcal V$}{fund(x)V}}
Consider the following quiver with potential
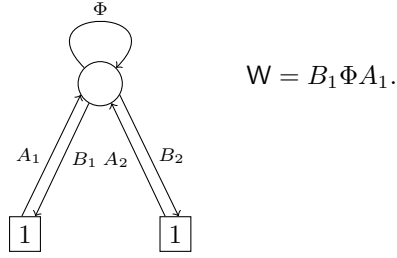
\begin{equation*}
\begin{tikzpicture}[x={(1cm,0cm)}, y={(0cm,1cm)}, baseline=0cm]
  % Nodes
  \node[draw,circle,fill=white] (Gauge) at (0,0) {$\phantom{n}$};
  \node[draw,rectangle,fill=white] (Framing1) at (-1,-2) {1};
  % \node (Framingdot) at (0,-2) {$\cdots$};
  % \node (Arrowdot) at (0,-1.5) {$\cdots$};
  \node[draw,rectangle,fill=white] (Framingk) at (1,-2) {1};
  \node (Z) at (0,1) {\scriptsize $\Phi$};
  % Edges
  \draw[<-] (Gauge.210) -- (Framing1.100) node[midway,left] {\scriptsize $A_1$};
  \draw[->] (Gauge.240) -- (Framing1.60) node[midway,right] {\scriptsize $B_1$};

  \draw[->] (Gauge.330) -- (Framingk.80) node[midway,right] {\scriptsize $B_2$};
  \draw[<-] (Gauge.300) -- (Framingk.120) node[midway,left] {\scriptsize $A_2$};

  % Loop
  \draw[->,looseness=7] (Gauge.135) to[out=135,in=45] (Gauge.45);
\end{tikzpicture}
\qquad
\sW=B_1\Phi A_1.
\end{equation*}
Let $\sA\cong\bC^*_u$ act on the framing by $u\mathbf 1^{(1)}+\mathbf 1^{(2)}$, and $\sT=\sA\times \bC^*_\hbar\times \bC^*_{t}$, where $(q,p)\in \bC^*_\hbar\times \bC^*_{t}$ acts on quiver data by
\begin{align*}
    (\Phi,A_1,B_1,A_2,B_2)\mapsto (q^{-1}\Phi,A_1, qB_1,A_2,pB_2).
\end{align*}
In the basis $[Z_0]\otimes[\cM(n,\mathbf 1)]$ and $[Z_1]\otimes[\cM(n-1,\mathbf 1)]$, $\Stab_{u>0}$ is given by the matrix
\begin{align*}
\begin{pmatrix}
\prod_{j=0}^{n-1}(j\hbar-u) & \,\, n \hbar\,(t-(n-1)\hbar)\prod_{j=0}^{n-2}(j\hbar-u) \\
0 & (t-u)\prod_{j=0}^{n-2}((j-1)\hbar-u) 
\end{pmatrix},
\end{align*}
$\Stab_{u<0}$ is given by the matrix
\begin{align*}
\begin{pmatrix}
\prod_{i=0}^{n-1}(u-(i-1)\hbar) & 0 \\
\prod_{i=0}^{n-2}(u-(i-1)\hbar) & \quad (u-(n-1)\hbar)\prod_{i=0}^{n-2}(u-(i-1)\hbar)
\end{pmatrix}.
\end{align*}
To properly normalize the diagonal elements of the $R$-matrix, we introduce the normalizers (\cite[Def.~3.5]{COZZ}):
\begin{align}\label{normalizer_higher spin}
    \epsilon|_F=(-1)^{\rk \mathcal P|_F^+}.
\end{align}
Here the vector bundle $\mathcal P$ on $\cM(2)$ is the descent of $\Hom(V,D_{\Out})$, where $D_{\Out}$ is the out-going framing vector space and $V$ is the gauge vector space. $\mathcal P|_F^+$ is the attracting subbundle of $\mathcal P$ restricted on a fixed component $F$. 

Then in the basis $[Z_0]\otimes[\cM(n,\mathbf 1)]$ and $[Z_1]\otimes[\cM(n-1,\mathbf 1)]$, we have 
\begin{align*}
    R_{\mathcal H_1,\mathcal V}(u)=\Stab_{u<0,\epsilon}^{-1}\circ\Stab_{u>0,\epsilon}&=\begin{pmatrix}
(-1)^n & 0 \\
0 & (-1)^{n-1}
\end{pmatrix}
\Stab_{u<0}^{-1}\, \Stab_{u>0}
\begin{pmatrix}
1 & 0 \\
0 & -1
\end{pmatrix}
\\
&=\frac{1}{u+\hbar}\left(
\begin{array}{cc}
 u-(n-1)\hbar & \, \, n\hbar\left(t-(n-1)\hbar\right) \\
 1 & u+n\hbar-t \\
\end{array}
\right).
\end{align*}
Using the basis $[\cM(n,\mathbf 1)]=\frac{\mathbf x^n}{n!}$ in \eqref{bases of V and F}, we can write $R_{\mathcal H_1,\mathcal V}(u)$ in the following compact form
\begin{align}\label{R(H_1,V)}
    R_{\mathcal H_1,\mathcal V}(u)=\id+\frac{1}{u+\hbar}\left(
\begin{array}{cc}
 -\mathbf x\mathbf y &\,\, \hbar\mathbf x (t-\mathbf x\mathbf y) \\
 \mathbf y/\hbar & \mathbf x\mathbf y-t \\
\end{array}
\right),
\end{align}
where $\mathbf y=\hbar\frac{\partial}{\partial\mathbf x}$ acts on $\bC[\hbar][\mathbf x]$ as a differential operator.

\subsubsection{\texorpdfstring{$R$}{R}-matrix for \texorpdfstring{$\mathcal H_1\otimes\mathcal H_N$}{fund(x)Sym}}

Consider the following quiver with potential
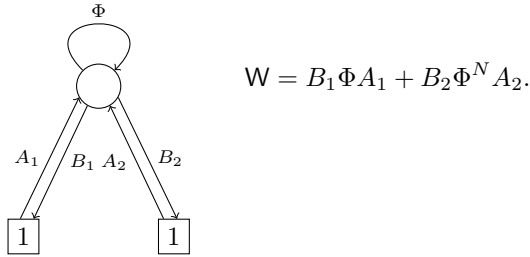
\begin{equation*}
\begin{tikzpicture}[x={(1cm,0cm)}, y={(0cm,1cm)}, baseline=0cm]
  % Nodes
  \node[draw,circle,fill=white] (Gauge) at (0,0) {$\phantom{n}$};
  \node[draw,rectangle,fill=white] (Framing1) at (-1,-2) {1};
  % \node (Framingdot) at (0,-2) {$\cdots$};
  % \node (Arrowdot) at (0,-1.5) {$\cdots$};
  \node[draw,rectangle,fill=white] (Framingk) at (1,-2) {1};
  \node (Z) at (0,1) {\scriptsize $\Phi$};
  % Edges
  \draw[<-] (Gauge.210) -- (Framing1.100) node[midway,left] {\scriptsize $A_1$};
  \draw[->] (Gauge.240) -- (Framing1.60) node[midway,right] {\scriptsize $B_1$};

  \draw[->] (Gauge.330) -- (Framingk.80) node[midway,right] {\scriptsize $B_2$};
  \draw[<-] (Gauge.300) -- (Framingk.120) node[midway,left] {\scriptsize $A_2$};

  % Loop
  \draw[->,looseness=7] (Gauge.135) to[out=135,in=45] (Gauge.45);
\end{tikzpicture}
\qquad
\sW=B_1\Phi A_1+B_2\Phi^N A_2.
\end{equation*}
Let $\sA\cong\bC^*_u$ act on the framing by $u\mathbf 1^{(1)}+\mathbf 1^{(2)}$, and $\sT=\sA\times \bC^*_\hbar$, where $q\in \bC^*_\hbar$ acts on quiver data by
\begin{align*}
    (\Phi,A_1,B_1,A_2,B_2)\mapsto (q^{-1}\Phi,A_1, qB_1,A_2,q^NB_2).
\end{align*}
Using the normalizer \eqref{normalizer_higher spin}, we define 
\begin{align*}
    R_{\mathcal H_1,\mathcal H_N}(u):=\Stab_{u<0,\epsilon}^{-1}\circ\Stab_{u>0,\epsilon}\in \End(\mathcal H_1\otimes\mathcal H_N)(u).
\end{align*}

\begin{Proposition}\label{prop res from V to H_N}
$R_{\mathcal H_1,\mathcal H_N}(u)$ is the restriction of $R_{\mathcal H_1,\mathcal V}(u)_{t=N\hbar}$ to $\mathcal H_1\otimes\mathcal H_N$ that is identified with its image under the specialization $\id\otimes\:\mathsf{sp}\colon \mathcal H_1\otimes\mathcal H_N\to \mathcal H_1\otimes\mathcal V\big|_{t=N\hbar}$.
\end{Proposition}

\begin{proof}
Take an additional $\bC^*$ action on $\cM(n,\mathbf{2})$ which acts on $B_2$ with weight $1$ and fixes all other arrows, this $\bC^*$-action can be used to construct specialization map (as in \cite[Def.~7.14]{COZZ}):
$$\mathsf{sp}\colon H^{\sA\times \bC^*_\hbar}(\cM(\mathbf{2}),B_1\Phi A_1+B_2\Phi^N A_2)\to H^{\sA\times \bC^*_\hbar}(\cM(\mathbf{2}),B_1\Phi A_1). $$ 
Passing to a fixed locus, the induced specialization map is given by
\begin{align*}
    \id\otimes \:\mathsf{sp}'\colon H^{\sA\times \bC^*_\hbar}(\cM(\mathbf{1}),B_1\Phi A_1)\otimes H^{\sA\times \bC^*_\hbar}(\cM(\mathbf{1}),B_2\Phi^N A_2)\longrightarrow H^{\sA\times \bC^*_\hbar}(\cM(\mathbf{1}),B_1\Phi A_1)\otimes H^{\sA\times \bC^*_\hbar}(\cM(\mathbf{1})),
\end{align*}
where $\mathsf{sp}'\colon H^{\sA\times \bC^*_\hbar}(\cM(\mathbf{1}),B_2\Phi^N A_2)\to H^{\sA\times \bC^*_\hbar}(\cM(\mathbf{1}))$ is the specialization map induced from the aforementioned $\bC^*$-action. By \cite[Ex.~7.19]{COZZ}, we also have
$$\mathsf{sp}\colon H^{\sA\times \bC^*_\hbar}(\cM(\mathbf{1}),B_2\Phi^N A_2)\to H^{\sA\times \bC^*_\hbar}(\cM(\mathbf{1})), $$
which is induced by the $\bC^*$-action that scales every arrow with weight $1$. We claim that $\mathsf{sp}=\mathsf{sp}'$. In fact, by 
\cite[Rmk.~7.16]{COZZ}, we have $\mathsf{sp}\circ \can=\mathsf{sp}'\circ \can=$ pushforward map from $Z(B_2\Phi^N A_2)$ to $\cM(\mathbf{1})$. Note that 
$$\can\colon H^{\sA\times \bC^*_\hbar}(Z(B_2\Phi^N A_2))\to H^{\sA\times \bC^*_\hbar}(\cM(\mathbf{1}),B_2\Phi^N A_2)$$ is surjective because the composition $$H^{\sA\times \bC^*_\hbar}(Z_n)\to H^{\sA\times \bC^*_\hbar}(Z(B_2\Phi^N A_2))\xrightarrow{\can} H^{\sA\times \bC^*_\hbar}(\cM(n,\mathbf{1}),B_2\Phi^N A_2)$$ is surjective. This proves the claim. Then the result follows from \cite[Prop.~7.18]{COZZ}.
\end{proof}

Write $\bar{\mathbf x}:=\hbar\mathbf x$, and let the Lie algebra $\mathfrak{gl}_2=\Span_\bC\{\mathbf e_{ij}\}_{i,j\in \{1,2\}}$ act on $\mathcal H_N=\Span_{\bC[\hbar]}\{1,\bar{\mathbf x},\ldots,\bar{\mathbf x}^{N-1},\bar{\mathbf x}^N\}$ by
\begin{align*}
    \mathbf e_{11}\mapsto 1-\bar{\mathbf x}\frac{\partial}{\partial \bar{\mathbf x}},\quad  \mathbf e_{12}\mapsto \frac{\partial}{\partial \bar{\mathbf x}} ,\quad  \mathbf e_{21}\mapsto \bar{\mathbf x}\left(N-\bar{\mathbf x} \frac{\partial}{\partial \bar{\mathbf x}}\right),\quad \mathbf e_{22}\mapsto 1-N+\bar{\mathbf x} \frac{\partial}{\partial \bar{\mathbf x}}.
\end{align*}
Then $R_{\mathcal H_1,\mathcal H_N}(u)$ takes the form:
\begin{align}\label{R(H_1,H_N)}
    R_{\mathcal H_1,\mathcal H_N}(u)=\frac{u\id+\hbar\:\Omega}{u+\hbar},\quad\,\, \Omega=\sum_{i,j}\mathbf e_{ij}\otimes \mathbf e_{ji}.
\end{align}
In particular, 
\begin{equation}\label{equ on RH1H1}
    R_{\mathcal H_1,\mathcal H_1}(u)=\frac{u\id+\hbar P}{u+\hbar},\quad \, \, P(\bar{\mathbf x}^i\otimes \bar{\mathbf x}^j)=\bar{\mathbf x}^j\otimes \bar{\mathbf x}^i.
\end{equation}

\begin{Remark}
Another way to obtain $R_{\mathcal H_1,\mathcal H_1}(u)$ is by dimensional reduction (\cite[Thm.~6.10]{COZZ}) to the Nakajima quiver with one node and without edge. The corresponding $R$-matrix is computed in \cite[\S4.1.2]{MO}.
\end{Remark}

\subsubsection{\texorpdfstring{$R$}{R}-matrix for \texorpdfstring{$\mathcal H_1\otimes\mathcal F$}{fund(x)F}}
Consider the following quiver with potential
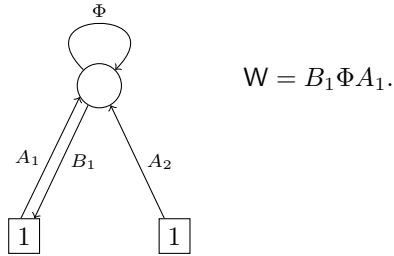
\begin{equation*}
\begin{tikzpicture}[x={(1cm,0cm)}, y={(0cm,1cm)}, baseline=0cm]
  % Nodes
  \node[draw,circle,fill=white] (Gauge) at (0,0) {$\phantom{n}$};
  \node[draw,rectangle,fill=white] (Framing1) at (-1,-2) {1};
  % \node (Framingdot) at (0,-2) {$\cdots$};
  % \node (Arrowdot) at (0,-1.5) {$\cdots$};
  \node[draw,rectangle,fill=white] (Framingk) at (1,-2) {1};
  \node (Z) at (0,1) {\scriptsize $\Phi$};
  % Edges
  \draw[<-] (Gauge.210) -- (Framing1.100) node[midway,left] {\scriptsize $A_1$};
  \draw[->] (Gauge.240) -- (Framing1.60) node[midway,right] {\scriptsize $B_1$};

  %\draw[->] (Gauge.330) -- (Framingk.80) node[midway,right] {\scriptsize $B_2$};
  \draw[<-] (Gauge.300) -- (Framingk.120) node[midway,right] {\scriptsize $A_2$};

  % Loop
  \draw[->,looseness=7] (Gauge.135) to[out=135,in=45] (Gauge.45);
\end{tikzpicture}
\qquad
\sW=B_1\Phi A_1.
\end{equation*}
Let $\sA\cong\bC^*_u$ act on the framing by $u\underline{\mathbf 1}+\underline{\mathbf 0}^{\mathbf 1}$.
In the basis $[Z_0]\otimes[\cM(n,\underline{\mathbf 0}^{\mathbf 1})]$ and $[Z_1]\otimes[\cM(n-1,\underline{\mathbf 0}^{\mathbf 1})]$, $\Stab_{u>0}$ is given by the matrix
\begin{align*}
\begin{pmatrix}
\prod_{j=0}^{n-1}(j\hbar-u) & n\hbar\prod_{j=0}^{n-2}(j\hbar-u) \\
0 & \prod_{j=0}^{n-2}((j-1)\hbar-u) 
\end{pmatrix},
\end{align*}
$\Stab_{u<0}$ is given by the matrix
\begin{align*}
\begin{pmatrix}
\prod_{i=0}^{n-1}(u-(i-1)\hbar) & 0 \\
\prod_{i=0}^{n-2}(u-(i-1)\hbar) & \quad (u-(n-1)\hbar)\prod_{i=0}^{n-2}(u-(i-1)\hbar)
\end{pmatrix}.
\end{align*}
Using the normalizer \eqref{normalizer_higher spin}, we have
\begin{align*}
    R_{\mathcal H_1,\mathcal F}(u)=\Stab_{u<0,\epsilon}^{-1}\circ\Stab_{u>0,\epsilon}=\frac{1}{u+\hbar}\left(
\begin{array}{cc}
 u-(n-1)\hbar & \, \, n\hbar \\
 -1 & \, \, 1 \\
\end{array}
\right).
\end{align*}

\begin{Remark}
$\sA$-action on the affine quotient $\cM_0(n,\underline{\mathbf 1}^{\mathbf 1})$ is attracting. By \cite[Prop.~3.35]{COZZ}, $\Stab_{u>0}^{-1}$ is defined for integral cohomology classes. Thus $R_{\mathcal H_1,\mathcal F}(u)^{-1}\in \End(\mathcal H_1\otimes\mathcal F)[u]$ is a polynomial in $u$. Direct computations give
\begin{align*}
    R_{\mathcal H_1,\mathcal F}(u)^{-1}= \left(
\begin{array}{cc}
 1 & \,\, -n\hbar \\
 1 & \,\,u-(n-1)\hbar \\
\end{array}
\right).
\end{align*}
\end{Remark}
Using the basis $[\cM(n,\underline{\mathbf 0}^{\mathbf 1})]=\frac{\mathbf x^n}{n!}$ in \eqref{bases of V and F}, we can write $R_{\mathcal H_1,\mathcal F}(u)$ in the following compact form
\begin{align}\label{R(H_1,F)}
    R_{\mathcal H_1,\mathcal F}(u)=\frac{1}{u+\hbar}\left(
\begin{array}{cc}
 u+\hbar-\mathbf x\mathbf y & \,\, \hbar\mathbf x \\
 -\mathbf y/\hbar & \,\, 1 \\
\end{array}
\right),
\end{align}
where $\mathbf y=\hbar\frac{\partial}{\partial\mathbf x}$ acts on $\bC[\hbar][\mathbf x]$ as a differential operator.

\subsection{Frassek-Pestun-Tsymbaliuk shifted Yangian}\label{sec yangian action}

Let 
\begin{equation}\label{equ on R(z)}R(z)=z\id+\hbar P\in \End(\bC^2)^{\otimes 2}[z,\hbar]. \end{equation}
Here $P=\sum_{1\leqslant i,j\leqslant 2}E_{ij}\otimes E_{ji}$ is the permutation operator, where $E_{ij}$ is the matrix with $1$ at $(i,j)$ place and zero elsewhere.

\begin{Definition}[{\cite[\S 2.3]{FPT}}]\label{def of shifted Y(gl_2)}
Fix $\mu\in \bZ_{\leqslant 0}$. The Frassek-Pestun-Tsymbaliuk $\mu$-\textit{shifted Yangian} $Y_{\mu}(\mathfrak{gl}_2)$ is a $\bC[\hbar]$ algebra generated by $\left\{e_{r},f_r,g_{1,r},g_{2,s}\right\}_{r\in \bZ_{\geqslant0},s\in \bZ_{\geqslant-\mu}}$, subject to relations
\begin{align}\label{RTT}
    R(u-v)\,T_1(u)\,T_2(v)=T_2(v)\,T_1(u)\,R(u-v),
\end{align}
with $T(z)\in Y_{\mu}(\mathfrak{gl}_2)[\![z,z^{-1}]\!]\otimes \End(\bC^2)$ defined via
\begin{align*}
    T(z)=F(z)\,G(z)\,E(z),\,\,\, \text{where}\,\,\, E(z)=\id_2+e(z)\otimes E_{12},\,\,\,  F(z)=\id_2+f(z)\otimes E_{21},\,\,\,  G(z)=\sum_{i=1}^2 g_i(z)\otimes E_{ii},
\end{align*}
and 
\begin{align}\label{gen current}
    e(z)=\sum_{r\geqslant 0} e_r z^{-r-1},\quad f(z)=\sum_{r\geqslant 0} f_r z^{-r-1},\quad g_1(z)=1+\sum_{r\geqslant 0}g_{1,r}z^{-r-1},\quad g_2(z)=z^{\mu}+\sum_{s\geqslant-\mu}g_{2,s}z^{-s-1}\,.
\end{align}
Here \eqref{RTT} is an equation in $Y_{\mu}(\mathfrak{gl}_2)[\![u,v,u^{-1},v^{-1}]\!]\otimes \End(\bC^2)^{\otimes 2}$. 

The $\mu$-\textit{shifted Yangian} $Y_{\mu}(\mathfrak{sl}_2)$ is the $\bC[\hbar]$-subalgebra of $Y_{\mu}(\mathfrak{gl}_2)$ generated by $\left\{e_r,f_r\right\}_{r\in \bZ_{\geqslant0}}$ and the coefficients of $h(z)=g_1(z)^{-1}g_2(z)$.
When $\mu=0$, we also write 
$Y(\mathfrak{sl}_2)=Y_{0}(\mathfrak{sl}_2)$, $Y(\mathfrak{gl}_2)=Y_{0}(\mathfrak{gl}_2)$.

% The special case $\mu=(0,0)$ gives the definition of ordinary (unshifted) Yangians $Y(\mathfrak{gl}_2)$ and $Y(\mathfrak{sl}_2)$, respectively.
\end{Definition}
% \begin{Remark}
% For any $m\in \mathbb{Z}$, $T(z)\mapsto z^m\,T(z)$ induces an algebra isomorphism 
% $$Y_{\mu}(\mathfrak{gl}_2)\cong Y_{\mu-(m,m)}(\mathfrak{gl}_2).$$
% The isomorphism class of $Y_{\mu}(\mathfrak{gl}_2)$ only depends on $\mu$. 

% \end{Remark}

%\subsubsection{Shifted Yangian modules via stable envelopes}
Next we explain how to describe modules of Frassek-Pestun-Tsymbaliuk shifted Yangian via stable envelopes. 
Consider the following framed quiver with $\bd_{\In}\geqslant\bd_{\Out}\in \bN$:
\begin{equation}\label{equ on m poten}
\begin{tikzpicture}[x={(1cm,0cm)}, y={(0cm,1cm)}, baseline=0cm]
  % Nodes
  \node[draw,circle,fill=white] (Gauge) at (0,0) {$\phantom{n}$};
  \node[draw,rectangle,fill=white] (Framing1) at (-1,-2) {$\bd_{\In}$};
  % \node (Framingdot) at (0,-2) {$\cdots$};
  % \node (Arrowdot) at (0,-1.5) {$\cdots$};
  \node[draw,rectangle,fill=white] (Framingk) at (1,-2) {$\bd_{\Out}$};
  \node (Z) at (0,1) {\scriptsize $\Phi$};
  % Edges
  \draw[<-] (Gauge.210) -- (Framing1.100) node[midway,left] {\scriptsize $A$};
  % \draw[->] (Gauge.240) -- (Framing1.60) node[midway,right] {\scriptsize $B_1$};

  \draw[->] (Gauge.330) -- (Framingk.80) node[midway,right] {\scriptsize $B$};
  % \draw[<-] (Gauge.300) -- (Framingk.120) node[midway,right] {\scriptsize $A_2$};

  % Loop
  \draw[->,looseness=7] (Gauge.135) to[out=135,in=45] (Gauge.45);
\end{tikzpicture}
\qquad\sW\in \bC Q^{\underline{\bd}}/[\bC Q^{\underline{\bd}},\bC Q^{\underline{\bd}}],
\end{equation}
where $\bC Q^{\underline{\bd}}$ is the path algebra of the Crawley-Boevey quiver $Q^{\underline{\bd}}$ defined in 
beginning of \S\ref{sec double of COHA}. We assume that $\sW$ is $\bC^*_\hbar\times \sA^{\mathrm{fr}}$-invariant, where $\sA^{\mathrm{fr}}$ is a torus that acts on $\bC^{\bd_{\In}}$ and $\bC^{\bd_{\Out}}$ (i.e. framing torus), 
$\bC^*_\hbar$ scales
$\Phi$ with weight $-1$ (its action on $A$, $B$ is arbitrary as long as preserving $\sw$). This implies that 
$$\sW\big|_{A=0,B=0}=0,$$ 
since for any $k>0$, $\tr(\Phi^k)$ is not $\bC^*_\hbar$-invariant. Denote $\sw=\tr(\sW)$.

Consider the state space
\begin{align*}
\mathcal H^{\sW}_{\underline{\bd},\sA^{\mathrm{fr}}}:=H^{\bC^*_\hbar\times \sA^{\mathrm{fr}}}(\cM(\underline{\bd}),\sw).
\end{align*}
When $\sA^{\mathrm{fr}}$ is the identity group, we write $\mathcal H^{\sW}_{\underline{\bd}}=\mathcal H^{\sW}_{\underline{\bd},\sA^{\mathrm{fr}}}$.

We can ``combine'' $\mathcal H_1$ and $\mathcal H^{\sW}_{\underline{\bd},\sA^{\mathrm{fr}}}$ by considering the following quiver with potential:
\begin{equation}\label{quiver for H_1 times general H}
\begin{tikzpicture}[x={(1cm,0cm)}, y={(0cm,1cm)}, baseline=0cm]
  % Nodes
  \node[draw,circle,fill=white] (Gauge) at (0,0) {$\phantom{n}$};
  \node[draw,rectangle,fill=white] (Framing1) at (-2,-2) {$1$};
  \node[draw,rectangle,fill=white] (Framingdot) at (0,-2) {$\bd_{\In}$};
  % \node (Arrowdot) at (0,-1.5) {$\cdots$};
  \node[draw,rectangle,fill=white] (Framingk) at (2,-2) {$\bd_{\Out}$};
  \node (Z) at (0,1) {\scriptsize $\Phi$};
  % Edges
  \draw[<-] (Gauge.210) -- (Framing1.115) node[midway,left] {\scriptsize $A_1$};
  \draw[->] (Gauge.240) -- (Framing1.60) node[midway,right] {\scriptsize $B_1$};

  \draw[<-] (Gauge.270) -- (Framingdot.90) node[midway,right] {\scriptsize $A$};

  \draw[->] (Gauge.330) -- (Framingk.80) node[midway,right] {\scriptsize $B$};
  % \draw[<-] (Gauge.300) -- (Framingk.120) node[midway,right] {\scriptsize $A_2$};

  % Loop
  \draw[->,looseness=7] (Gauge.135) to[out=135,in=45] (Gauge.45);
\end{tikzpicture}
\qquad\sW'=B_1\Phi A_1+\sW.
\end{equation}
Let $\sA\cong \bC^*_u$ act on the framing by $u\underline{\mathbf 1}+\underline{\bd}$ and $\sT=\sA\times \bC^*_\hbar\times \sA^{\mathrm{fr}}$. Applying the deformed dimensional reduction (\cite[Thm.~C.1]{COZZ}) to $\tr(B_1\Phi A_1)\boxplus\sw$ along $B_1$, we get
\begin{align*}
H^{\sT}(\cM(m,\underline{\mathbf 1})\times\cM(n,\underline{\bd}),\tr(B_1\Phi A_1)\boxplus\sw)\cong H^{\sT}(\cM(n,\underline{\bd})^{},\sw),\quad m\in\{0,1\}
\end{align*}
and this implies that
$$\mathcal H_1\otimes_{\bC[\hbar]} \mathcal H^{\sW}_{\underline{\bd},\sA^{\mathrm{fr}}}\otimes_{\bC}\bC[u]\cong H^{\sT}(\cM(\underline{\mathbf 1}+\underline{\bd})^{\sA},\sw'^{\sA}), \; \text{ for }\; \sw'=\tr(\sW')$$ 
In the basis $[Z_0]$, $[Z_1]$, the stable envelope with normalizer \eqref{normalizer_higher spin} in the $u>0$ chamber takes the matrix form:
\begin{align*}
\Stab_{u>0,\epsilon}=\left(
\begin{array}{cc}
 (-1)^n e^{\sT}(N^-_{\scriptscriptstyle \cM(0,\mathbf 1)\times \cM(n,\underline{\bd})/\cM(n,\underline{\bd}+\underline{\mathbf 1})}) & \star_1 \\
 0 & (-1)^{n-1}e^{\sT}(N^-_{\scriptscriptstyle\cM(1,\mathbf 1)\times \cM(n-1,\underline{\bd})/\cM(n,\underline{\bd}+\underline{\mathbf 1})}) \\
\end{array}
\right),
\end{align*}
with $\deg_u\star_1<\rk N^-_{\scriptscriptstyle\cM(0,\mathbf 1)\times \cM(n,\underline{\bd})/\cM(n,\underline{\bd}+\underline{\mathbf 1})}$. 
In the ample partial order (\cite[Rmk.~3.8]{COZZ}), we have 
$$\cM(1,\mathbf 1)\times \cM(n-1,\underline{\bd})< \cM(0,\mathbf 1)\times \cM(n,\underline{\bd}), $$
so the 
$21$-component of $\Stab_{u>0,\epsilon}$ is zero due to axiom (i) in \cite[Def.~3.4]{COZZ}. Similarly
\begin{align*}
\Stab_{u<0,\epsilon}=\left(
\begin{array}{cc}
e^{\sT}(N^+_{\scriptscriptstyle \cM(0,\mathbf 1)\times \cM(n,\underline{\bd})/\cM(n,\underline{\bd}+\underline{\mathbf 1})}) & 0 \\
 \star_2 & (-1)^{\bd_{\Out}}e^{\sT}(N^+_{\scriptscriptstyle\cM(1,\mathbf 1)\times \cM(n-1,\underline{\bd})/\cM(n,\underline{\bd}+\underline{\mathbf 1})}) \\
\end{array}
\right),
\end{align*}
with $\deg_u\star_2<\rk N^+_{\scriptscriptstyle\cM(1,\mathbf 1)\times \cM(n-1,\underline{\bd})/\cM(n,\underline{\bd}+\underline{\mathbf 1})}$. 

\begin{Proposition}\label{prop shifted Y(gl(2)) action}
For arbitrary $\underline{\bd}$ and $\sW$ in \eqref{equ on m poten}, the map 
\begin{equation}\label{image of generators}
\begin{split}
e(u)\mapsto \frac{(-1)^n\star_1}{  e^{\sT}(N^-_{\scriptscriptstyle \cM(0,\mathbf 1)\times \cM(n,\underline{\bd})/\cM(n,\underline{\bd}+\underline{\mathbf 1})})}, &\qquad f(u)\mapsto \frac{(-1)^{\bd_{\Out}-1}\star_2}{e^{\sT}(N^+_{\scriptscriptstyle\cM(1,\mathbf 1)\times \cM(n-1,\underline{\bd})/\cM(n,\underline{\bd}+\underline{\mathbf 1})})},\\
g_1(u)\mapsto \frac{(-1)^ne^{\sT}(N^-_{\scriptscriptstyle \cM(0,\mathbf 1)\times \cM(n,\underline{\bd})/\cM(n,\underline{\bd}+\underline{\mathbf 1})})}{ e^{\sT}(N^+_{\scriptscriptstyle \cM(0,\mathbf 1)\times \cM(n,\underline{\bd})/\cM(n,\underline{\bd}+\underline{\mathbf 1})}) },&\qquad g_2(u)\mapsto \frac{(-1)^{n+\bd_{\Out}-1}e^{\sT}(N^-_{\scriptscriptstyle\cM(1,\mathbf 1)\times \cM(n-1,\underline{\bd})/\cM(n,\underline{\bd}+\underline{\mathbf 1})})}{e^{\sT}(N^+_{\scriptscriptstyle\cM(1,\mathbf 1)\times \cM(n-1,\underline{\bd})/\cM(n,\underline{\bd}+\underline{\mathbf 1})})},
\end{split}
\end{equation}
gives an action of $Y_{\mu}(\mathfrak{gl}_2)$ on $\mathcal H^{\sW}_{\underline{\bd},\sA^{\mathrm{fr}}}$ with $\mu=\bd_{\Out}-\bd_{\In}$. Here all rational functions are expanded in $u\to \infty$.
\end{Proposition}

\begin{proof}
$e(u)$ and $f(u)$ have the correct form of power series in $u$ since 
$$\lim_{u\to \infty}e(u)=0=\lim_{u\to \infty}f(u). $$ 
The leading terms in the $u\to \infty$ of $g_1(u)$ and $g_2(u)$ are $1$ and $u^{\bd_{\Out}-\bd_{\In}}$ respectively. Thus the image of $e(u)$, $f(u)$, and $g_i(u)$ are power series of the form \eqref{gen current}. By a direct calculation, under the map \eqref{image of generators}, we have
$$T(u)=F(u)\,G(u)\,E(u)\mapsto R_{\mathcal H_1,\mathcal H^{\sW}_{\underline{\bd}}}(u)=\left(\Stab_{u<0,\epsilon} \right)^{-1} \circ \Stab_{u>0,\epsilon}. $$ 
By the braid relation \eqref{braid}, the RTT relation holds:
\begin{align*}
    R_{\mathcal H^{(1)}_1,\mathcal H^{(2)}_1}(u-v)\,R_{\mathcal H^{(1)}_1,\mathcal H^{\sW}_{\underline{\bd}}}(u)\,R_{\mathcal H^{(2)}_1,\mathcal H^{\sW}_{\underline{\bd}}}(v)=R_{\mathcal H^{(2)}_1,\mathcal H^{\sW}_{\underline{\bd}}}(v)\,R_{\mathcal H^{(1)}_1,\mathcal H^{\sW}_{\underline{\bd}}}(u)\,R_{\mathcal H^{(1)}_1,\mathcal H^{(2)}_1}(u-v).
\end{align*}
By \eqref{equ on RH1H1}, $R_{\mathcal H^{(1)}_1,\mathcal H^{(2)}_1}(z)$ coincides with \eqref{equ on R(z)} (up to a constant). 
Then the proposition follows from Definition \ref{def of shifted Y(gl_2)}.
\end{proof}

\begin{Remark}\label{rmk on Y(gl_2) map}
For the Jordan quiver $Q$ with zero potential, $\mathcal H_1$ is admissible by Definition \ref{def on adm pot}, so the maps in \eqref{image of generators} induce a surjective algebra map from $Y_{\mu}(\mathfrak{gl}_2)\otimes_{\bC[\hbar]}\bC(\hbar)$ to $\mathsf Y_{\mu}(Q,0)$ for the Jordan quiver $Q$, by Proposition \ref{prop shifted Y(gl(2)) action} and Theorem \ref{thm admissible}.
\end{Remark}

\begin{Example}\label{ex H_N, V, F}
In the examples of $\mathcal V$, $\mathcal H_N$ and $\mathcal F$, we can explicitly pin down their shifted Yangian module structures by comparing the $R$-matrices \eqref{R(H_1,V)}, \eqref{R(H_1,H_N)} and \eqref{R(H_1,F)} with the known Lax matrices \cite{FPT}. For simplicity, let us
restrict to the subalgebra $Y_{\mu}(\mathfrak{sl}_2)$, which is also denoted 
as $Y(\mathfrak{sl}_2)$ when $\mu=0$. Then we have
\begin{itemize}
\item $\tau_{\hbar}^*\mathcal V$ is isomorphic to the universal dual Verma module with highest weight $t$ induced by the evaluation map $Y(\mathfrak{sl}_2)\to U(\mathfrak{sl}_2)[\hbar]$;
\item $\tau_{\hbar}^*\mathcal H_N$ is isomorphic to the $N$-th symmetric power of the fundamental representation $\bC[\hbar]^{\oplus 2}$ of $Y(\mathfrak{sl}_2)$ induced by the evaluation map $Y(\mathfrak{sl}_2)\to U(\mathfrak{sl}_2)[\hbar]$;
\item $\tau_{\hbar}^*\mathcal F$ is isomorphic to the predundamental module $L^{-}_0$ of $Y_{-1}(\mathfrak{sl}_2)$.
\end{itemize}
Here the automorphism is given by: 
$$\tau_{a}\colon Y_{\mu}(\mathfrak{gl}_2)\cong Y_{\mu}(\mathfrak{gl}_2), \quad T(z)\mapsto T(z-a). $$ 
It induces a pullback map on Yangian modules: 
\begin{equation}\label{equ on tau pb}\tau_{a}^*\colon Y_{\mu}(\mathfrak{gl}_2)\mathrm{-Mod}\cong Y_{\mu}(\mathfrak{gl}_2)\mathrm{-Mod}. \end{equation}
The subalgebra $Y_{\mu}(\mathfrak{sl}_2)$ is $\tau_a$ invariant, so $\tau_{a}^*$ restricts to an automorphism $Y_{\mu}(\mathfrak{sl}_2)\mathrm{-Mod}\cong Y_{\mu}(\mathfrak{sl}_2)\mathrm{-Mod}$.
\end{Example}

\begin{Remark}
The images of $g_1(u)$ and $g_2(u)$ can be written as tautological classes on $\cM(\underline{\bd})$:
\begin{align*}
    g_1(u)\mapsto c_{-1/u}\left((1-\hbar^{-1})\mathsf V\right),\quad g_2(u)\mapsto u^{\bd_{\Out}-\bd_{\In}}\cdot c_{-1/u}\left((1-\hbar)\mathsf V+\mathsf D_{\Out}-\mathsf D_{\In}\right),
\end{align*}
where $\mathsf V$ is the descent of $\bC^n$ to the GIT quotient $\cM(n,\underline{\bd})$, $\mathsf D_{\In}$ and $\mathsf D_{\Out}$ are the pullback of the in-coming and out-going framing vector spaces to $\cM(n,\underline{\bd})$ respectively. $\mathsf V$, $\mathsf D_{\In}$, and $\mathsf D_{\Out}$ are $\sT_0$-equivariant vector bundles on $\cM(n,\underline{\bd})$, and $c_{-1/u}(\cdots)$ is the $\sT_0$-equivariant Chern polynomial with variable $-1/u$. For example,
\begin{align*}
    g_{1,0}\mapsto  -\hbar\bv,\quad g_{2,-\mu} \mapsto c_1(\mathsf D_{\In})-c_1(\mathsf D_{\Out})+\hbar\bv\,.
\end{align*}
In the case of $\mu=0$, comparison with the Lie algebra generators in Proposition \ref{prop g(Q,w)} gives:
$g_{1,0}\mapsto -\hbar\mathsf v$, $g_{2,0}\mapsto \mathsf d+\hbar\mathsf v$, and
\begin{align}\label{r-mat gl2}
-g_{2,0}\otimes g_{1,0}-g_{1,0}\otimes g_{2,0}-2g_{1,0}\otimes g_{1,0}+e_0\otimes f_0+f_0\otimes e_0\mapsto \hbar\,\pmb r\,.
\end{align}
\end{Remark}

\begin{Remark}
Proposition \ref{prop shifted Y(gl(2)) action} is obtained by \cite{YZJ} in the special case when 
$\bd_{\In}=\bd_{\Out}=\bd$ and $\sW=\sum_{i=1}^{\bd}B_i\Phi^{N_i}A_i$.
%The case $\bd_{\In}=\bd_{\Out}=\bd$ with $\sw=\sum_{i=1}^{\bd}B_i\Phi^{N_i}A_i$ has been discussed in \cite{YZJ}. 
It is proven in \textit{loc.\,cit.}~that $Y(\mathfrak{sl}_2)$ acts on $\cH^{\sW}_{{\bd}}$, and that the action makes $\cH_N$ the $N$-th symmetric power of the fundamental module through the evaluation map $Y(\mathfrak{sl}_2)\to U(\mathfrak{sl}_2)[\hbar]$. The method in \textit{loc.\,cit.}~is different from ours. They use dimensional reduction to the Borel-Moore homology of zero locus of $\{\Phi^{N_i}A_i\}_{i=1}^{\bd}$, construct correspondences for the generators and check relations. 
\end{Remark}

\subsection{Coproducts and stable envelopes}

Frassek-Pestun-Tsymbaliuk shifted Yangians possess the following coproduct maps \cite[Prop.~2.316]{FPT}:
\begin{align*}
    \Delta_{\mu',\mu''}\colon Y_{\mu'+\mu''}(\mathfrak{gl}_2)\to Y_{\mu'}(\mathfrak{gl}_2)\otimes Y_{\mu''}(\mathfrak{gl}_2),\qquad T(z)\mapsto T(z)\otimes T(z).
\end{align*}
The coproduct induces a tensor product on the module categories:
\begin{align*}
    \left(Y_{\mu'}(\mathfrak{gl}_2)\mathrm{-Mod}\right)\;\times \; \left(Y_{\mu''}(\mathfrak{gl}_2)\mathrm{-Mod}\right)\;\xrightarrow{\quad\otimes\quad} \; Y_{\mu'+\mu''}(\mathfrak{gl}_2)\mathrm{-Mod}.
\end{align*}
On the geometry side, we have stable envelope:
\begin{align*}
    \Stab_{a>0,\epsilon}\colon \mathcal H^{\sW'}_{\underline{\bd}'}\otimes \mathcal H^{\sW''}_{\underline{\bd}''}\longrightarrow \mathcal H^{\sW}_{\underline{\bd}'+\underline{\bd}''},
\end{align*}
which is defined using the torus $\bC^*_a$ action on the framing: $a\underline{\bd}'+\underline{\bd}''$ with normalizer \eqref{normalizer_higher spin}. 
% Write $$\sw'=\tr(\mathsf m'), \quad \sw''=\tr(\mathsf m''), \quad \sw=\tr(\mathsf m'+\mathsf m''). $$ 
By \S \ref{sect on coprod}, we have: 
\begin{align*}
    \Stab_{a>0,\epsilon}^{-1}\,R_{\mathcal H_1,\mathcal H^{\sW}_{\underline{\bd}'+\underline{\bd}''}}(u)\,\Stab_{a>0,\epsilon}=  R_{\mathcal H_1,\mathcal H^{\sW''}_{\underline{\bd}''}}(u)\,R_{\mathcal H_1,\mathcal H^{\sW'}_{\underline{\bd}'}}(u-a).
\end{align*}
Thus $\Stab_{a>0,\epsilon}$ induces a $Y_{\mu'+\mu''}(\mathfrak{gl}_2)$-module map (with $\mu'=(0,\bd_{\In}'-\bd_{\Out}')$, $\mu''=(0,\bd_{\In}''-\bd_{\Out}'')$): 
$$\mathcal H^{\sW''}_{\underline{\bd}''}\otimes\tau_{a}^*\mathcal H^{\sW'}_{\underline{\bd}'}\to \mathcal H^{\sW}_{\underline{\bd}'+\underline{\bd}''}, $$
where $\tau_{a}^*$ is the Yangian module pullback map \eqref{equ on tau pb}. 

Similarly, $\Stab_{a<0,\epsilon}$ induces a $Y_{\mu'+\mu''}(\mathfrak{gl}_2)$-module map
$$\tau_{a}^*\mathcal H^{\sW'}_{\underline{\bd}'}\otimes \mathcal H^{\sW''}_{\underline{\bd}''}\to \mathcal H^{\sW}_{\underline{\bd}'+\underline{\bd}''}. $$

\begin{Proposition}
The $R$-matrix
\begin{align*}
    R_{\scriptscriptstyle\mathcal H^{\sW'}_{\underline{\bd}'},\mathcal H^{\sW''}_{\underline{\bd}''}}(a):= \left(\Stab_{a<0,\epsilon}\right)^{-1}\circ \Stab_{a>0,\epsilon}\colon \mathcal H^{\sW''}_{\underline{\bd}''}\otimes\tau_{a}^*\mathcal H^{\sW'}_{\underline{\bd}'}\longrightarrow\tau_{a}^*\mathcal H^{\sW'}_{\underline{\bd}'}\otimes \mathcal H^{\sW''}_{\underline{\bd}''}
\end{align*}
is a $Y_{\mu'+\mu''}(\mathfrak{gl}_2)$-module map. 

Moreover, if $\bd_{\Out}'=0$ then $R_{\scriptscriptstyle\mathcal H^{\sW'}_{\underline{\bd}'},\mathcal H^{\sW''}_{\underline{\bd}''}}(a)$ is a polynomial in $a$; if $\bd_{\Out}''=0$ then $R_{\scriptscriptstyle\mathcal H^{\sW'}_{\underline{\bd}'},\mathcal H^{\sW''}_{\underline{\bd}''}}(a)^{-1}$ is a polynomial in $a$.
\end{Proposition}

\begin{proof}
The statement on intertwining property follows from the above discussions. The statement on polynomiality follows from 
\cite[Prop.~3.35]{COZZ}.
\end{proof}

\begin{Remark}
Similarly as \cite[Prop.~4.5.1]{MO}, the above intertwiners are unitary:
\begin{align*}
    R_{\scriptscriptstyle\mathcal H^{\sW'}_{\underline{\bd}'},\mathcal H^{\sW''}_{\underline{\bd}''}}(a)^{-1}=R_{\scriptscriptstyle\mathcal H^{\sW''}_{\underline{\bd}''},\mathcal H^{\sW'}_{\underline{\bd}'}}(-a).
\end{align*}
%The proof is essentially the same as .
\end{Remark}

% \begin{Remark}
% We expect 
% \end{Remark}

\begin{Example}
Write $\mathcal F^{(i)}=\bC[\hbar][\mathbf x_i]$ and $\mathbf y_i=\hbar\frac{\partial}{\partial\mathbf x_i}$ for $i\in \{1,2\}$, so we have $ \mathcal F^{(1)}\otimes\mathcal F^{(2)}=\bC[\hbar][\mathbf x_1,\mathbf x_2]$. Then the $R$-matrix takes the following explicit form
\begin{equation}\label{equ on rff}
    R_{\mathcal F^{(1)},\mathcal F^{(2)}}(a)=\left(1-\mathbf x_2\mathbf y_1\right)^{a/\hbar}S_{12},
\end{equation}
where $S_{12}(f(\mathbf x_1,\mathbf x_2))=f(\mathbf x_2,\mathbf x_1)$. 
In fact, by direct computations, $\left(1-\mathbf x_2\mathbf y_1\right)^{a/\hbar}S_{12}$ solves the intertwining equation:
\begin{align*}
    \left(1-\mathbf x_2\mathbf y_1\right)^{\frac{u-v}{\hbar}}\,S_{12}\,R_{\mathcal F^{(1)},\mathcal H_1}(u)\,R_{\mathcal F^{(2)},\mathcal H_1}(v)=R_{\mathcal F^{(2)},\mathcal H_1}(v)\,R_{\mathcal F^{(1)},\mathcal H_1}(u)\left(1-\mathbf x_2\mathbf y_1\right)^{\frac{u-v}{\hbar}}\,S_{12}.
\end{align*}
An intertwiner that maps $1\in \bC[\hbar][\mathbf x_1,\mathbf x_2]$ to itself is unique by \cite[Thm.~5.2\,(i)]{HZ}, so \eqref{equ on rff} holds.
%$$R_{\mathcal F^{(1)},\mathcal F^{(2)}}(a)=\left(1-\mathbf x_2\mathbf y_1\right)^{a/\hbar}S_{12}.$$
\end{Example}

\subsection{Proof of Theorem \ref{thm ex Jordan_main}}\label{sect on proof of thm jordan}

Theorem \ref{thm ex Jordan_main} follows from Remark \ref{rmk on Y(gl_2) map} and the following Proposition \ref{prop embed Y into end}.

\begin{Proposition}\label{prop embed Y into end}
For arbitrary $\mu\in \bZ_{\leqslant 0}$, the natural map 
\begin{align*}
    \iota\colon Y_{\mu}(\mathfrak{gl}_2)\to \prod_{\underline{\bd},\sW,\sA^{\mathrm{fr}} }\End(\cH^{\sW}_{\underline{\bd},\sA^{\mathrm{fr}}})
\end{align*}
is injective, where the product is for all $\underline{\bd}$ and $\sW$ in \eqref{equ on m poten} such that $\bd_{\Out}-\bd_{\In}=\mu$.
\end{Proposition}

\begin{proof}
We prove this by induction on $\mu$. When $\mu=0$, we can restrict to representations of the form 
$$\cH^{\sW}_{\underline{\bd},\sA^{\mathrm{fr}}}\cong \tau^*_{a_1}\cH_1\otimes\cdots \otimes\tau^*_{a_n}\cH_1,$$ 
then the induced map
\begin{align*}
    Y_{0}(\mathfrak{gl}_2)\to \prod_{a_1,\ldots,a_n}\End(\tau^*_{a_1}\cH_1\otimes\cdots \otimes\tau^*_{a_n}\cH_1)
\end{align*}
is injective by \cite[\S 11.3.2]{MO}. 

For $\mu<0$, assume that $\iota$ is injective for $\mu'=\mu+1$, we take representations of the form 
$$\cH^{\sW}_{\underline{\bd},\sA^{\mathrm{fr}}}\cong \tau^*_{a_1}\cF\otimes \tau^*_{a_2}\cH^{\sW'}_{\underline{\bd}'},$$ where $\bd_{\Out}'-\bd_{\In}'=\mu'$ and $\cF$ is prefundamental module \eqref{equ on prefund}. Here the tensor product is induced by the stable envelope as in the previous subsection. Then the action of $Y_{\mu}(\mathfrak{gl}_2)$ on $\cH^{\sW}_{\underline{\bd},\sA^{\mathrm{fr}}}$ factors through 
\begin{align*}
    Y_{\mu}(\mathfrak{gl}_2)\xrightarrow{\Delta} Y_{(0,1)}(\mathfrak{gl}_2)\otimes Y_{\mu'}(\mathfrak{gl}_2)\twoheadrightarrow D\otimes Y_{\mu'}(\mathfrak{gl}_2)
\end{align*}
where $D=\bC[\hbar]\langle\mathbf x,\mathbf y\rangle/([\mathbf y,\mathbf x]-\hbar)$ is the Weyl algebra. By induction and the fact that $D$-action on $\mathcal F$ is faithful, the induced map
\begin{align*}
    D\otimes Y_{\mu'}(\mathfrak{gl}_2)\to \prod_{a_1,a_2,\underline{\bd}',\sw' }\End(\tau^*_{a_1}\cF\otimes \tau^*_{a_2}\cH^{\sW'}_{\underline{\bd}'})
\end{align*}
is injective. Then it is enough to show that $Y_{\mu}(\mathfrak{gl}_2)\to D\otimes Y_{\mu'}(\mathfrak{gl}_2)$ is injective. 

By the flatness over $\bC[\hbar]$, it suffices to prove the injectivity after passing to the classical limit $\hbar=0$. It is known that $Y_{\mu}(\mathfrak{gl}_2)/(\hbar)$ is a commutative algebra, which is the function ring of a pro-finite type affine variety $\cW_{\mu}$. $\cW_{\mu}$ is the closed subscheme of the group ind-scheme $\GL_2(\!(z^{-1})\!)$ consisting of matrices of the form
\begin{align*}
\left(
\begin{array}{cc}
 1+a(z) & \,\, b(z) \\
 c(z) & \,\,*\\
\end{array}
\right)\text{ with determinant } z^{-m}(1+d(z)),
\end{align*}
where $a(z),b(z),c(z),d(z)\in z^{-1}\bC[\![z^{-1}]\!]$ are power series. The classical limit of $D$ becomes polynomial ring of two variables, which is naturally identified with the subscheme $\cW^{(1,0)}_{(0,1)}$ in $\GL_2(\!(z^{-1})\!)$ consisting of matrices of the form
\begin{align*}
\frac{1}{z}
\left(
\begin{array}{cc}
 z+xy & \,\, x \\
 y & \,\, 1\\
\end{array}
\right),\quad (x,y)\in \bC^2.
\end{align*}
The $\hbar=0$ limit of the map $Y_{\mu}(\mathfrak{gl}_2)\to D\otimes Y_{\mu'}(\mathfrak{gl}_2)$ is the map between function rings induced by the multiplication in $\GL_2(\!(z^{-1})\!)$:
\begin{align*}
    \cW^{(1,0)}_{(0,1)}\times \cW_{\mu'}\to\cW_{\mu}\:.
\end{align*}
It is elementary to see that the above map is surjective. Since $\bC[\cW_{\mu}]$ is an integral domain, the induced map $\bC[\cW_{\mu}]\to \bC\left[\cW^{(1,0)}_{(0,1)}\right]\otimes \bC[\cW_{\mu'}]$ is injective.
\end{proof}

\subsection{Specialization maps as modules intertwiners}

We generalize the construction of Proposition \ref{prop res from V to H_N}, 
which will be used in \S \ref{sect on fusion gl2}.
Take a $\bG_m$-action on $\cM(n,\underline{\bd})$ that scales on $B$ in the quiver 
\eqref{equ on m poten} with weight $1$ and fixes all other arrows.  
Let $\sw=\tr(\sW)$ and $\sT_0$ be as in \eqref{equ on m poten}. This $\bG_m$-action fits into the $\bG_m$-action in \cite[Setting~7.1]{COZZ}, 
and it can be used to construct specialization map (as in \cite[Def.~7.14]{COZZ}):
\begin{align}\label{sp from w to 0}
    \mathsf{sp}\colon \mathcal H^{\sW}_{\underline{\bd}}=H^{\sT_0}(\cM(\underline{\bd}),\sw)\to H^{\sT_0}(\cM(\underline{\bd}))=\mathcal H^{0}_{\underline{\bd}}.
\end{align}

\begin{Proposition}\label{prop sp as mod map}
We have $$R_{\mathcal H_1,\mathcal H^{0}_{\underline{\bd}}}(u)\circ (\id\otimes\:\mathsf{sp})=(\id\otimes\:\mathsf{sp})\circ R_{\mathcal H_1,\mathcal H^{\sW}_{\underline{\bd}}}(u).$$ In particular, $\mathsf{sp}$ is a $Y_{\mu}(\mathfrak{gl}_2)$-module map.
\end{Proposition}

\begin{proof}
Consider the quiver \eqref{quiver for H_1 times general H} and its moduli space $\cM(n,\underline{\bd}+\underline{\mathbf 1})$ with $\sT=\sA\times \sT_0$, where 
$\sA\cong \bC^*_u$ acts on the framing by 
$u\underline{\mathbf 1}+\underline{\bd}$. 
Take another $\bC^*$-action which acts on $B$ in the quiver \eqref{quiver for H_1 times general H} with weight $1$ and fixes all other arrows. 
As in \cite[Def.~7.14]{COZZ}, there is a specialization map:
\begin{align*}
    H^\sT(\cM(\underline{\bd}+\underline{\mathbf 1}),B_1\Phi A_1+\tr(\mathsf m))\to H^\sT(\cM(\underline{\bd}+\underline{\mathbf 1}),B_1\Phi A_1).
\end{align*}
Passing to $\sA$-fixed loci, the induced specialization map is given by
\begin{align*}
    \id\otimes\:\mathsf{sp}\colon H^{\sT}(\cM(\mathbf 1),B_1\Phi A_1)\otimes H^{\sT}(\cM(\underline{\bd}),\tr(\mathsf m))\to H^{\sT}(\cM(\mathbf 1),B_1\Phi A_1)\otimes H^{\sT}(\cM(\underline{\bd}))
\end{align*}
where $\mathsf{sp}$ is the map defined in \eqref{sp from w to 0}. Then the result follows from 
\cite[Prop.~7.18]{COZZ}.
\end{proof}

\begin{Remark}
Same argument shows that 
\begin{align}\label{sp respect R-mat}
    R_{\mathcal H^{\sW'}_{\underline{\bd}'},\mathcal H^{0}_{\underline{\bd}}}(u)\circ (\id\otimes\:\mathsf{sp})=(\id\otimes\:\mathsf{sp})\circ R_{\mathcal H^{\sW'}_{\underline{\bd}'},\mathcal H^{\sW}_{\underline{\bd}}}(u)
\end{align}
for arbitrary $\underline{\bd}'$ and $\sW'$ in \eqref{equ on m poten}.
\end{Remark}

\subsection{Fusion procedure from geometry}\label{sect on fusion gl2}

It is known that an intertwiner between $\mathcal H_N$ and $\mathcal H_M$ can be obtained from the elementary $R$-matrix 
$$R(u)=u\id+\hbar P, $$ via the \textit{fusion procedure} \cite{KRS}. In this section, we derive the fusion procedure geometrically.

Consider the following framed quiver with potential:
\begin{equation*}
\begin{tikzpicture}[x={(1cm,0cm)}, y={(0cm,1cm)}, baseline=0cm]
  % Nodes
  \node[draw,circle,fill=white] (Gauge) at (0,0){$n$}; %{$\bv$}  {$n$};
  \node[draw,rectangle,fill=white] (Framing) at (2,0) {$N$};
  \node (Z) at (-1,0) {\scriptsize $\Phi$};
  % Edges
  \draw[<-] (Gauge.340) -- (Framing.200) node[midway,below] {\scriptsize $A$};
  \draw[->] (Gauge.20) -- (Framing.160) node[midway,above] {\scriptsize $B$};

  % Loop
  \draw[->,looseness=7] (Gauge.225) to[out=225,in=135] (Gauge.135);
\end{tikzpicture}
\qquad \sW_*=B\Phi A-\Xi B A, \quad \sw_*=\tr(\sW_*),
\end{equation*}
where $\Xi\in \End(\C^N)$ is a fixed  nilpotent matrix of rank $(N-1)$. Let $q\in \bC^*_\hbar$ act on quiver data by 
\begin{align}\label{hbar action on M(N)}
    (\Phi,B,A)\mapsto (q^{-1}\Phi,q \cdot\sigma(q)B,A\sigma(q)^{-1}),
\end{align}
where $\sigma\colon \bC^*_\hbar\to \GL_N$ is a homomorphism such that $\Xi$ is eigenvector under the adjoint action of $\sigma(\bC^*_\hbar)$ with eigenvalue $1$. Given a basis $\bC^N=\Span_\bC\{e_1,\ldots,e_N\}$ such that $\Xi(e_i)=e_{i+1}$, we can take $\sigma(q)=\diag(1,q,\ldots,q^{N-1})$.

% The function $\sw$ fits into conditions of $\tr(\mathsf m)$ \eqref{equ on m poten}.
By Proposition \ref{prop shifted Y(gl(2)) action}, the state space 
\begin{align*}
    \mathcal H_{\underline{\mathbf N}}^{\sW_*}:=H^{\bC^*_\hbar}(\cM(\mathbf N),\sw_*)=\bigoplus_{n\in \bZ_{\geqslant 0}} H^{\bC^*_\hbar}(\cM(n,\mathbf N),\sw_*)
\end{align*}
is equipped with a natural Yangian algebra $Y(\mathfrak{gl}_2)$ action coming from stable envelopes.
Similarly, the space \eqref{equ on HNn}: 
$$\mathcal H_N=\bigoplus_{n\in \bZ_{\geqslant 0}} H^{\bC^*_\hbar}(\cM(n,\mathbf 1),\sw_N), \,\,\, \mathrm{with} \,\,\, \sw_N=\tr(B\Phi^NA) $$
is also a Yangian module. 

\begin{Proposition}\label{prop comp dim red for higher spin}
There is a $Y(\mathfrak{gl}_2)$ module isomorphism 
\begin{equation}\label{equ on delt}\delta\colon \mathcal H_N\cong \mathcal H_{\underline{\mathbf N}}^{\sW_*}, \end{equation} 
such that 
\begin{align*}
    R_{\mathcal H^{\sW}_{\underline{\bd}},\mathcal H_{\underline{\mathbf N}}^{\sW_*}}(u)\circ(\id\otimes \:\delta)=(\id\otimes \:\delta)\circ R_{\mathcal H^{\sW}_{\underline{\bd}},\mathcal H_N}(u)
\end{align*}
for arbitrary $\underline{\bd}$ and $\sW$ in \eqref{equ on m poten}.
\end{Proposition}

%We will explicitly construct the isomorphism $\delta\colon \mathcal H_N\cong \mathcal H_{\underline{\mathbf N}}^{\sW_*}$ using compatible dimensional reduction data (Definition \ref{compatible dim red data}) later in this section. Assume the above proposition for now. 
We postpone the proof to the end of this section and use it for now. 

Consider the same quiver variety $\cM(\mathbf N)$ with the \textit{canonical cubic potential} $\sw_1=\tr(B\Phi A)$.
By the dimensional reduction, we have an isomorphism
\begin{align}\label{dim red for w_1}
 \mathcal H_{\underline{\mathbf N}}^{\sW_1}:=H^{\bC^*_\hbar}(\cM(\mathbf N),\sw_1) \cong H^{\bC^*_\hbar}(\cN(\mathbf N)), 
\end{align}
where $\cN(\mathbf N)=\bigsqcup_{m\in \mathbb{N}}\mathcal N(m,\mathbf N)$ is the union of Nakajima varieties of $A_1$ quiver \cite{Nak1}.
\begin{Lemma}
\eqref{dim red for w_1} is an isomorphism of $Y(\mathfrak{gl}_2)$ modules.
\end{Lemma}
\begin{proof}
 By Proposition \ref{prop shifted Y(gl(2)) action}, $\mathcal H_{\underline{\mathbf N}}^{\sW_1}$ is a $Y(\mathfrak{gl}_2)$ module. 
 The $Y(\mathfrak{gl}_2)$ module structure on $H^{\bC^*_\hbar}(\cN(\mathbf N))$ is given by the stable envelope $\mathbf{Stab}_{u>0}$ of \cite{MO} for Nakajima varieties\,\footnote{$\mathbf{Stab}_{u>0}=\Stab_{u>0,\epsilon'}$ where $\epsilon'$ is the normalizer $\epsilon'\big|_{F}=(-1)^{\rk T^{1/2}|^{+}_{F/\cN(\mathbf N+\mathbf 1)}}$ for the polarization $T^{1/2}=\Hom(\bC^{N+1},V)-\End(V)$.}.
 Then the statement follows from commutative diagram:
 \begin{equation*}
\xymatrix{
H^{\bC_\hbar^*\times \bC^*_u}(\cM(n,\mathbf 1)\times \cM(m,\mathbf N),B_1\Phi A_1+\tr(B\Phi A)) \ar[rr]^{\Stab_{u>0,\epsilon}} & & H^{\bC_\hbar^*\times \bC^*_u}(\cM(n+m,\mathbf 1+\mathbf N),B_1\Phi A_1+\tr(B\Phi A))\\
H^{\bC_\hbar^*\times \bC^*_u}(\cN(n,\mathbf 1)\times \cN(m,\mathbf N)) \ar[rr]^{\mathbf{Stab}_{u>0}} \ar[u]^{\cong} & & H^{\bC_\hbar^*\times \bC^*_u}(\cN(n+m,\mathbf 1+\mathbf N)), \ar[u]_{\cong}
}
\end{equation*}
where $\bC^*_u$ acts on the framing by $u\underline{\mathbf 1}+\mathbf N$, and vertical isomorphisms are dimensional reduction maps 
(see \cite[Thm.~6.10]{COZZ}, Example \ref{ex doubled vs tripled} and \cite[Rmk.~9.14]{COZZ}). 
\end{proof}
We relate $Y(\mathfrak{gl}_2)$ modules $ \mathcal H_{\underline{\mathbf N}}^{\sW_1}$ and $\cH_{1}$. 
\begin{Lemma}\label{lem tensor N H_1}
There is a $Y(\mathfrak{gl}_2)\otimes_{\bC[\hbar]} \bC(\hbar)$ module isomorphism
\begin{align*}
    \mathcal H_{\underline{\mathbf N},\loc}^{\sW_1}\cong \tau^*_{(N-1)\hbar}\cH_{1,\loc}\otimes \cdots \otimes \tau^*_{\hbar}\cH_{1,\loc}\otimes \cH_{1,\loc}\:.
\end{align*}
Here ``loc'' means $-\otimes_{\bC[\hbar]} \bC(\hbar)$, and tensor products are taken over the base field $\bC(\hbar)$.
\end{Lemma}
\begin{proof}
Consider the Nakajima variety associated to the following quiver
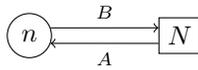
\begin{equation*}
\begin{tikzpicture}[x={(1cm,0cm)}, y={(0cm,1cm)}, baseline=0cm]
  % Nodes
  \node[draw,circle,fill=white] (Gauge) at (0,0){$n$}; %{$\bv$}  {$n$};
  \node[draw,rectangle,fill=white] (Framing) at (2,0) {$N$};
  % \node (Z) at (-1,0) {\scriptsize $\Phi$};
  % Edges
  \draw[<-] (Gauge.340) -- (Framing.200) node[midway,below] {\scriptsize $A$};
  \draw[->] (Gauge.20) -- (Framing.160) node[midway,above] {\scriptsize $B$};

  % Loop
  % \draw[->,looseness=7] (Gauge.225) to[out=225,in=135] (Gauge.135);
\end{tikzpicture}
\end{equation*}
and let $(q,p)\in \bC^*_\hbar\times \bC^*_t$ act on quiver data by
\begin{align*}
    (B,A)\mapsto (q \cdot\sigma(p)B,A\sigma(p)^{-1}),
\end{align*}
where $\sigma$ is the cocharacter of $\GL_N$ in \eqref{hbar action on M(N)}. The stable envelope
\begin{align*}
    \mathbf{Stab}_{t<0}\colon H^{\bC^*_\hbar\times \bC^*_t}(\cN(\mathbf N)^{\bC^*_t})\to H^{\bC^*_\hbar\times \bC^*_t}(\cN(\mathbf N))
\end{align*}
is a $Y(\mathfrak{gl}_2)$-module map by the argument in Section \ref{sect on coprod}. Evaluation at $t=\hbar$ amounts to taking equivariant cohomologies with respect to the diagonal subtorus $\bC^*_\hbar\xrightarrow{\diag}\bC^*_\hbar\times \bC^*_t$. Since $\cN(\mathbf N)^{\bC^*_t}=\cN(\mathbf 1)^{\times N}$, with diagonal $\bC^*_\hbar$ weights $\{0,1,\ldots,N-1\}$ on the one-dimensional framing vector spaces, 
% torus weight $(-i,i+1)$ and $(0,1)$ are equivalent by gauge actions. 
we get a $Y(\mathfrak{gl}_2)$-module map:
\begin{align*}
    \mathbf{Stab}_{t<0}\colon \tau^*_{(N-1)\hbar}\cH_{1}\otimes \cdots \otimes \tau^*_{\hbar}\cH_{1}\otimes \cH_{1}\to H^{\bC^*_\hbar}(\cN(\mathbf N)).
    % \cH_{1}\otimes \tau^*_{\hbar}\cH_{1}\otimes \cdots\otimes \tau^*_{(N-1)\hbar}\cH_{1}
\end{align*}
It is elementary to see that the induced $\bC^*_t$ action on the diagonal $\bC^*_\hbar$ fixed locus $\cN(\mathbf N)^{\bC^*_{\hbar,\diag}}$ is repelling. Then by \cite[Prop.~3.36]{COZZ}, the above map is injective after localization.
By dimension counting: $$\dim_{\bC(\hbar)}\cH_{1,\loc}^{\otimes N}=2^N=\dim_{\bC(\hbar)}H^{\bC^*_\hbar}(\cN(\mathbf N))_\loc,$$ 
the map 
\begin{align*}
    \mathbf{Stab}_{t<0}\colon \tau^*_{(N-1)\hbar}\cH_{1,\loc}\otimes \cdots \otimes \tau^*_{\hbar}\cH_{1,\loc}\otimes \cH_{1,\loc}\to H^{\bC^*_\hbar}(\cN(\mathbf N))_\loc
    % \cH_{1,\loc}\otimes \tau^*_{\hbar}\cH_{1,\loc}\otimes \cdots\otimes \tau^*_{(N-1)\hbar}\cH_{1,\loc}
\end{align*}
is an isomorphism. Combing with dimensional reduction \eqref{dim red for w_1}, we get the desired module isomorphism. 
%$\mathcal H_{\underline{\mathbf N},\loc}^{\sW_1}\cong \tau^*_{(N-1)\hbar}\cH_{1,\loc}\otimes \cdots \otimes \tau^*_{\hbar}\cH_{1,\loc}\otimes \cH_{1,\loc}$.
\end{proof}

\begin{Remark}
Since stable envelopes are compatible with dimensional reductions (\cite[Thm.~6.10]{COZZ}), the map
\begin{align*}
    \Stab_{t<0,\epsilon}\colon H^{\bC^*_\hbar\times \bC^*_t}(\cM(\mathbf N)^{\bC^*_t},\sw_1)\big|_{t=\hbar}\to H^{\bC^*_\hbar\times \bC^*_t}(\cM(\mathbf N),\sw_1)\big|_{t=\hbar}
\end{align*}
is an isomorphism after localization. The braid relation \eqref{braid} implies:
\begin{align*}
    \Stab_{t<0,\epsilon}^{-1}\,R_{\mathcal H^{\sW}_{\underline{\bd}},\mathcal H^{\sW_1}_{\underline{\mathbf N}}}(u)\,\Stab_{t<0,\epsilon}=R_{\mathcal H^{\sW}_{\underline{\bd}},\mathcal H_1^{(1)}}(u-(N-1)\hbar)\cdots  R_{\mathcal H^{\sW}_{\underline{\bd}},\mathcal H_1^{(N-1)}}(u-\hbar)\,R_{\mathcal H^{\sW}_{\underline{\bd}},\mathcal H_1^{(N)}}(u)
\end{align*}
for arbitrary $\underline{\bd}$ and $\sW$ in \eqref{equ on m poten}.
\end{Remark}

Consider specialization maps \eqref{sp from w to 0} for $\sW_*$ and $\sW_1$:
\begin{align*}
    \mathsf{sp}_*\colon \mathcal H_{\underline{\mathbf N}}^{\sW_*}\to \mathcal H_{\underline{\mathbf N}}^{0},\qquad \mathsf{sp}_1\colon \mathcal H_{\underline{\mathbf N}}^{\sW_1}\to \mathcal H_{\underline{\mathbf N}}^{0}.
\end{align*}
\begin{Lemma}\label{lem sp* and sp1 inj}
The maps $\mathsf{sp}_*$ and $\mathsf{sp}_1$ are injective. Moreover, the image of $\mathsf{sp}_*$ is contained in the image of $\mathsf{sp}_1$.
\end{Lemma}

\begin{proof}
$\bullet$ (Injectivity of $\mathsf{sp}_*$). By Proposition \ref{prop comp dim red for higher spin} and Example \ref{ex H_N, V, F}, $\mathcal H_{\underline{\mathbf N}}^{\sW_*} \otimes_{\bC[\hbar]}\bC(\hbar)$ is an irreducible $Y(\mathfrak{gl}_2) \otimes_{\bC[\hbar]}\bC(\hbar)$ module because it is already irreducible as an $\mathfrak{sl}_2$ module. $\mathsf{sp}_*$ is a $Y(\mathfrak{gl}_2)$ module map by Proposition \ref{prop sp as mod map}, and the degree zero component $\mathsf{sp}_*\colon \mathcal H_{\underline{\mathbf N}}^{\sW_*}(0)\to \mathcal H_{\underline{\mathbf N}}^{0}(0)$ is the identity map; thus $\mathsf{sp}_*$ must be injective after localization. $\mathcal H_N$ is a free $\bC[\hbar]$ module as we have seen in Section \ref{sec bases of H_N, V, F}, so $\mathsf{sp}_*$ is injective.

$\bullet$ (Injectivity of $\mathsf{sp}_1$). By the dimensional reduction argument in the proof of Lemma \ref{lem tensor N H_1}, $\mathcal H_{\underline{\mathbf N}}^{\sW_1}$ is isomorphic to the BM homology of Nakajima variety $\cN(\mathbf N)$, and the latter is a free $\bC[\hbar]$ module by \cite[Thm.~7.3.5]{Nak3}. Then the proof of injectivity of $\mathsf{sp}_1$ boils down to the proof of injectivity after localization. 

Since $\sw_1$ is linear in $B$, and $B$ does not enter the stability condition, by dimensional reduction along $B$, we get 
$$H^{\bC^*_\hbar}(Z(\Phi A))\cong H^{\bC^*_\hbar}(\cM(\mathbf N),\sw_*), $$ 
where $Z(\Phi A)$ is the zero loci of $\Phi A$ in $\cM(\mathbf N)$, and the isomorphism is induced by the canonical map 
$$H^{\bC^*_\hbar}(Z(\Phi A))\to H^{\bC^*_\hbar}(Z(\sw_*))\xrightarrow{\can} H^{\bC^*_\hbar}(\cM(\mathbf N),\sw).$$ 
By \cite[Rmk.~7.16]{COZZ}, it suffices to show that the pushforward map $H^{\bC^*_\hbar}(Z(\Phi A))\to H^{\bC^*_\hbar}(\cM(\mathbf N))$ is injective after localization. Equivalently, we need to show that the pushforward map $H^{\bC^*_\hbar}(Z'(\Phi A))\to H^{\bC^*_\hbar}(\cM(\underline{\mathbf 0}^{\mathbf N}))$ is injective after localization, where $\cM(\underline{\mathbf 0}^{\mathbf N})$ is the following quiver variety:
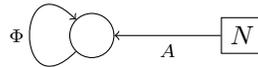
\begin{equation*}
\begin{tikzpicture}[x={(1cm,0cm)}, y={(0cm,1cm)}, baseline=0cm]
  % Nodes
  \node[draw,circle,fill=white] (Gauge) at (0,0){$\phantom{n}$}; %{$\bv$}  {$n$};
  \node[draw,rectangle,fill=white] (Framing) at (2,0) {$N$};
  \node (Z) at (-1,0) {\scriptsize $\Phi$};
  % Edges
  \draw[<-] (Gauge.0) -- (Framing.180) node[midway,below] {\scriptsize $A$};
  % \draw[->] (Gauge.20) -- (Framing.160) node[midway,above] {\scriptsize $B$};

  % Loop
  \draw[->,looseness=7] (Gauge.225) to[out=225,in=135] (Gauge.135);
\end{tikzpicture}
\end{equation*}
and $Z'(\Phi A)$ is the zero of $\Phi A$ in $\cM(\underline{\mathbf 0}^{\mathbf N})$. It is easy to see that the fixed locus $Z'(\Phi A)^{\bC^*_\hbar}$ is discrete, and that every connected component of $\cM(\underline{\mathbf 0}^{\mathbf N})^{\bC^*_\hbar}$ is proper. Note that a point class $[x]$ in any proper variety $X$ is nonzero in its 
Borel-Moore homology. Therefore, we only need to show that every connected component $F$ of $\cM(\underline{\mathbf 0}^{\mathbf N})^{\bC^*_\hbar}$ contains at most one point in $Z'(\Phi A)$,~i.e.~the equation $\Phi A=0$ has at most one solution on $F$. Explicitly, $F$ is isomorphic to the following quiver variety for some $(n_1,\ldots,n_N, n_{N+1},\ldots)$ with cyclic stability:
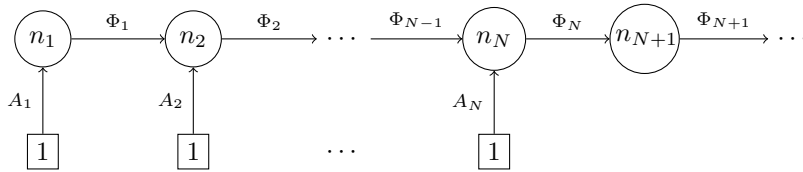
\begin{equation*}
\begin{tikzpicture}[x={(1cm,0cm)}, y={(0cm,1cm)}, baseline=0cm]
  % Nodes
  \node[draw,circle,fill=white] (Gauge1) at (0,0) {$n_1$};
  \node[draw,rectangle,fill=white] (Framing1) at (0,-1.5) {$1$};
  
  \node[draw,circle,fill=white] (Gauge2) at (2,0) {$n_2$};
  \node[draw,rectangle,fill=white] (Framing2) at (2,-1.5) {$1$};

  \node (dot1) at (4,0) {$\cdots$};
  \node (Framingdot1) at (4,-1.5) {$\cdots$};
  \node[draw,circle,fill=white] (GaugeN) at (6,0) {$n_N$};
  \node[draw,rectangle,fill=white] (FramingN) at (6,-1.5) {$1$};
  \node[draw,circle,fill=white,inner sep=0pt] (GaugeN+1) at (8,0) {$\,n_{N+1}$};

  \node (dot2) at (10,0) {$\cdots$};

  % \node (Z) at (0,1) {\scriptsize $\Phi$};
  % Edges
  \draw[<-] (Gauge1.270) -- (Framing1.90) node[midway,left] {\scriptsize $A_1$};
  \draw[<-] (Gauge2.270) -- (Framing2.90) node[midway,left] {\scriptsize $A_2$};
  \draw[<-] (GaugeN.270) -- (FramingN.90) node[midway,left] {\scriptsize $A_N$};

  \draw[->] (Gauge1.0) -- (Gauge2.180) node[midway,above] {\scriptsize $\Phi_1$};
  \draw[->] (Gauge2.0) -- (dot1.180) node[midway,above] {\scriptsize $\Phi_2$};
  \draw[->] (dot1.0) -- (GaugeN.180) node[midway,above] {\scriptsize $\Phi_{N-1}$};
  \draw[->] (GaugeN.0) -- (GaugeN+1.180) node[midway,above] {\scriptsize $\Phi_N$};
  \draw[->] (GaugeN+1.0) -- (dot2.180) node[midway,above] {\scriptsize $\Phi_{N+1}$};

  % Loop
  % \draw[->,looseness=7] (Gauge.135) to[out=135,in=45] (Gauge.45);
\end{tikzpicture}
\end{equation*}
we see that $\Phi A=0$ has at most one solution on $F$.

$\bullet$ (The inclusion $\im(\mathsf{sp}_*)\subseteq \im(\mathsf{sp}_1)$). Since $\sw_*$ is linear in $B$, and $B$ does not enter the stability condition, we can apply the dimensional reduction along $B$ and the composition of the following maps is isomorphism
$$H^{\bC^*_\hbar}(Z(\Phi A-A\Xi))\to H^{\bC^*_\hbar}(Z(\sw_*))\xrightarrow{\can} H^{\bC^*_\hbar}(\cM(\mathbf N),\sw_*),$$ 
where $Z(\Phi A-A\Xi)$ is the zero loci of $\Phi A-A\Xi$ in $\cM(\mathbf N)$. In particular, $H^{\bC^*_\hbar}(Z(\sw_*))\xrightarrow{\can} H^{\bC^*_\hbar}(\cM(\mathbf N),\sw_*)$ is surjective. Using an auxiliary $\C^*$ that scales $B$ with weight $1$, $\Phi$ with weight $-1$, and fixes $A$, then this $\C^*$ action fits into 
\cite[Setting~7.1]{COZZ} and we have a specialization map\,\footnote{However, this auxiliary $\C^*$-action does not satisfy \cite[Assumption~7.12]{COZZ}, so we can not use it to define a specialization map $H^{\bC^*_\hbar}(\cM(\mathbf N),\sw_*)\to H^{\bC^*_\hbar}(\cM(\mathbf N),\sw_1)$ between critical cohomologies.} $$\mathsf{sp}_Z\colon H^{\bC^*_\hbar}(Z(\sw_*))\to H^{\bC^*_\hbar}(Z(\sw_1)).$$
By \cite[Rmk.~7.16]{COZZ}, $\mathsf{sp}_*\circ \can=$ pushforward from $Z(\sw_*)$ to $\cM(\mathbf N)$, and $\mathsf{sp}_1\circ \can=$ pushforward from $Z(\sw_1)$ to $\cM(\mathbf N)$. Combining with \cite[Rmk.~7.9]{COZZ}, the following diagram
\begin{equation*}
\xymatrix{
H^{\bC^*_\hbar}(Z(\sw_*)) \ar[rr]^{\mathsf{sp}_Z} \ar[d]^{\can} & & H^{\bC^*_\hbar}(Z(\sw_1)) \ar[d]^{\can}\\
H^{\bC^*_\hbar}(\cM(\mathbf N),\sw_*) \ar[r]^{\mathsf{sp}_*} & H^{\bC^*_\hbar}(\cM(\mathbf N)) & \ar[l]_{\mathsf{sp}_1} H^{\bC^*_\hbar}(\cM(\mathbf N),\sw_1)
}
\end{equation*}
commutes. Then the inclusion $\im(\mathsf{sp}_*)\subseteq \im(\mathsf{sp}_1)$ follows from the surjectivity of canonical maps.
\end{proof}
By Lemma \ref{lem sp* and sp1 inj}, we have an induced injective map
\begin{align*}
    \overline{\mathsf{sp}}\colon \mathcal H_{\underline{\mathbf N}}^{\sW_*}\hookrightarrow \mathcal H_{\underline{\mathbf N}}^{\sW_1}.
\end{align*}
%to be the injective map induced from $\mathrm{sp_*}$, using Lemma \ref{lem sp* and sp1 inj}. 
The equation \eqref{sp respect R-mat} implies that
\begin{align}
    R_{\mathcal H^{\sW}_{\underline{\bd}},\mathcal H_{\underline{\mathbf N}}^{\sW_1}}(u)\circ(\id\otimes \:\overline{\mathsf{sp}})=(\id\otimes \:\overline{\mathsf{sp}})\circ R_{\mathcal H^{\sW}_{\underline{\bd}},\mathcal H_{\underline{\mathbf N}}^{\sW_*}}(u)
\end{align}
for arbitrary $\underline{\bd}$ and $\sW$ in \eqref{equ on m poten}. In particular, $\overline{\mathsf{sp}}$ is a $Y(\mathfrak{gl}_2)$ module map. 

To state the fusion formula, we define $\eta$ to be the composition 
\begin{align}\label{equ on eta map}
    \eta\colon \cH_{N,\loc}\xrightarrow[\cong]{\eqref{equ on delt}} \cH^{\sW_*}_{\underline{\mathbf N},\loc}\overset{\overline{\mathsf{sp}}}{\hookrightarrow}\cH^{\sW_1}_{\underline{\mathbf N},\loc}\xrightarrow[\cong]{\text{Lem.\,}\ref{lem tensor N H_1}} \tau^*_{(N-1)\hbar}\cH_{1,\loc}\otimes \cdots \otimes \tau^*_{\hbar}\cH_{1,\loc}\otimes \cH_{1,\loc}\:.
\end{align}
Then $\eta$ is a $Y(\mathfrak{gl}_2)\otimes_{\bC[\hbar]}\bC(\hbar)$ module embedding.

\begin{Theorem}\label{thm fusion}
$\eta$ induces an isomorphism 
$$\eta\colon \cH_{N,\loc}\cong S^N(\cH_{1,\loc})$$
to the $N$-th symmetric tensor $S^N(\cH_{1,\loc})$ of $\cH_{1,\loc}$. Moreover, we have
\begin{align*}
    (\id\otimes \:\eta)\circ R_{\mathcal H^{\sW}_{\underline{\bd}},\cH_N}(u)= R_{\mathcal H^{\sW}_{\underline{\bd}},\mathcal H_1^{(1)}}(u-(N-1)\hbar)\cdots R_{\mathcal H^{\sW}_{\underline{\bd}},\mathcal H_1^{(N-1)}}(u-\hbar)\,R_{\mathcal H^{\sW}_{\underline{\bd}},\mathcal H_1^{(N)}}(u)\circ (\id\otimes \:\eta)\,,
\end{align*}
for arbitrary $\underline{\bd}$ and $\sW$ in \eqref{equ on m poten}. 
Here $\mathcal H_1^{(i)}$ denotes the $i$-th component of the RHS of \eqref{equ on eta map}.
% Here $P_N\colon \mathcal H_1^{\otimes N}\to S^N(\mathcal H_1)$ is the projection to the symmetric subspace.
\end{Theorem}

\begin{proof}
As $Y(\mathfrak{sl}_2)\otimes_{\bC[h]}\bC(\hbar)$ modules, $\cH_{N,\loc}$ is abstractly isomorphic to $S^N(\bC(\hbar)^{\oplus 2})$ induced by the evaluation map $Y(\mathfrak{sl}_2)\otimes_{\bC[h]}\bC(\hbar)\to U(\mathfrak{sl}_2)\otimes\bC(\hbar)$ (Example \ref{ex H_N, V, F}). So the image of $\eta$ must be $S^N(\cH_{1,\loc})$ as it is the unique summand of $\cH_{1,\loc}^{\otimes N}$ which is isomorphic to $S^N(\bC(\hbar)^{\oplus 2})$ as $\mathfrak{sl}_2$ module. The rest follows from what we have established.
\end{proof}

Taking $\mathcal H^{\sW}_{\underline{\bd}}=\cH_M$ in the above theorem, we obtain:

\begin{Corollary}\label{cor fusion formula}
We have the \textit{fusion formula} for the $R$-\textit{matrix} $R_{\cH_M,\cH_N}(u)$:
\begin{align*}
    (\eta\otimes \eta)\circ R_{\cH_M,\cH_N}(u)= \left(\prod_{\alpha=1}^{\substack{\longleftarrow\\ M}}\prod_{\beta=1}^{\substack{\longrightarrow\\ N}}\frac{(u+(\beta-\alpha+M-N)\hbar)\id+\hbar P_{\alpha\beta}}{u+(\beta-\alpha+M-N+1)\hbar}\right)\circ (\eta\otimes \eta)
\end{align*}
where $P_{\alpha\beta}\in \End\left(\cH_1^{\otimes M}\otimes \cH_1^{\otimes N}\right)$ is the operator that permutes the $\alpha$-th component in $\cH_1^{\otimes M}$ and $\beta$-th component in $\cH_1^{\otimes N}$, and the product is ordered such that index increases along the arrows.
\end{Corollary}

\begin{proof}[Proof of Proposition \ref{prop comp dim red for higher spin}]

%For the purpose of a clean presentation, we introduce the following change of notations: 
To facilitate presentations, 
quiver data associated to $\cM(\mathbf 1)$ will be denoted by lowercase letters: $(\varphi,a,b)$; quiver data associated to $\cM(\mathbf N)$ will still be denoted by $(\Phi,A,B)$.

Recall the notation $\underline{\mathbf 0}^{\mathbf k}=(k,0)$, and let 
\begin{equation*}
    Y'=\cM(n,\underline{\mathbf 0}^{\mathbf 1}),\quad X'=\mathrm{Tot}(E')= \cM(n,\mathbf 1)
\end{equation*}
$$s'=\varphi^N a\in \Gamma(Y',E'^\vee),\quad \phi'=0,\quad \sw_N=\langle s', b\rangle,$$
where $E'$ is the vector bundle that descends from $\Hom(\bC^n,\bC)$ with coordinate given by $b$, and $\langle-,-\rangle$ is the natural pairing. Let
\begin{equation}
Y:=\cM(n,\underline{\mathbf 0}^{\mathbf N}),\quad X=\mathrm{Tot}(E)= \cM(n,\mathbf N), \end{equation}
$$s=\Phi A-A\Xi \in \Gamma(Y,E^\vee), \quad \phi=0, \quad \sw_*=\langle s, B\rangle. $$
where $E$ is the vector bundle that descends from $\Hom(\bC^n,\bC^N)$ with coordinate given by $B$. Note that cohomological dimensional reduction hold for $(X',Y',s',\phi')$ and $(X,Y,s,\phi)$ (see \cite[Def.~6.1]{COZZ}).

We show that these two data are compatible in the sense of \cite[Def.~6.2]{COZZ}. 
\begin{Lemma}\label{lem on dr higher spin}
(1) The map 
$$i\colon R(n,\mathbf 1) \to R(n,\mathbf N), \quad (\varphi,a,b)\mapsto (\Phi,A,B) $$
%$$(x,\xi,y,a,b)\mapsto (x,\xi,y,A,B), $$
with $\Phi=\varphi$, $A=(a,\varphi a,\ldots, \varphi^{N-1}a)$, 
$B=\begin{pmatrix}
b\varphi^{N-1} \\ 
\vdots \\
b\varphi \\
b \end{pmatrix}$ 
is a $(\GL_n\times \bC^*_\hbar)$-equivariant closed immersion such that 
$$i^*\sw_*=\sw_{N}, \quad \Crit(\sw_{N}) \stackrel{i}{\cong} \Crit(\sw_*). $$
This induces a closed embedding $i\colon X' \hookrightarrow X$ whose restriction to $Y'$ gives $i\colon Y' \hookrightarrow Y$.

(2) The $(\GL_n\times \bC^*_\hbar)$-equivariant projection 
$$\pr\colon \Hom(\C^n,\C^N)\to \Hom(\C^n,\C), \quad B=\begin{pmatrix}
b_1 \\ 
\vdots \\
b_N \end{pmatrix}\mapsto b_N$$
induces a surjective vector bundle map 
\begin{equation}\label{equ on pr on higher spin}\pr\colon E|_{Y'}\twoheadrightarrow E' \end{equation}
such that conditions in \cite[Def.~6.2]{COZZ} holds. Moreover, there is an isomorphism of vector bundles
$$\ker(\pr)\cong N^\vee_{Y'/Y}. $$
%\begin{itemize} \item the potential function coincides under pullback: $$i^*\sw_N=\sw_{\mathrm{sym},N}, $$
 % \item the critical loci are isomorphic under inclusion: $$ \Crit(\sw_{\mathrm{sym},N}) \stackrel{i}{\cong} \Crit(\sw_N). $$\end{itemize}    
\end{Lemma}
\begin{proof}
(1) It is straightforward to check that the map $i$ is $(\GL_n\times \bC^*_\hbar)$-equivariant and 
$$i^*\tr(\Phi A B )=N\cdot \tr(b\varphi^N a), \quad   i^*\tr(A\Xi B)=(N-1)\cdot \tr(b\varphi^N a), $$
therefore the first claim holds. To show the second, note that 
$\frac{\partial \sw_*}{\partial A}=0,  \frac{\partial \sw_*}{\partial B}=0$ gives  
$$B  \Phi=\Xi  B, \quad \Phi A=A \Xi. $$
This implies that $A$ is of the form $(a,\varphi a,\ldots, \varphi^{N-1}a)$ and $B$ is of form $\begin{pmatrix}
b\varphi^{N-1} \\ 
\vdots \\
b\varphi \\
b \end{pmatrix}$ 
such that 
$$\varphi^{N}a=0, \quad  b\varphi^{N}=0. $$
This coincides exactly with $\frac{\partial \sw_{N}}{\partial a}=0,  \frac{\partial \sw_{N}}{\partial b}=0$. 
It is easy to check remaining equations coincide. 

(2) By a direct calculations, we have $\phi|_{Y'}=\phi'$, $s|_{Y'}=\pr^\vee\circ s'$.  
As shown in (1), elements in the zero locus $Z(s)$ are of form $(a,\varphi a,\cdots, \varphi^{N-1}a)$, 
so it sits inside $Y'$.
There is an obvious isomorphism $$\ker(\pr)\cong N^\vee_{Y'/Y}$$ by construction. 
By \cite[Prop.~6.4]{COZZ}, all conditions in \cite[Def.~6.2]{COZZ} are satisfied. 
\end{proof}
Combining Lemma \ref{lem on dr higher spin} and \cite[Prop.~6.9]{COZZ}, the map 
\begin{align*}
    \delta_H=(Y'\times_Y X\hookrightarrow X)_*\circ \pr^*\colon H^{\bC^*_\hbar}(\cM(n,\mathbf 1),\sw_N)\to H^{\bC^*_\hbar}(\cM(n,\mathbf N),\sw_*)
\end{align*}
is an isomorphism, where $\pr$ is the projection \eqref{equ on pr on higher spin}. 

Moreover, for arbitrary $\underline{\bd}$ and $\sW$ in \eqref{equ on m poten}, we can repeat the above construction and obtain a map
\begin{align*}
    \delta_{H,\underline{\bd}}\colon H^{\sT}(\cM(n,\underline{\bd}+\mathbf 1),\sw+\sw_N)\to H^{\sT}(\cM(n,\underline{\bd}+\mathbf N),\sw+\sw_*),
\end{align*}
which is induced by a compatible dimensional reduction pair. Here $\sT$ is a torus that contains $\bC^*_\hbar\times \bC^*_u$ and that the framings have $\bC^*_u$ weights $u\underline{\bd}+\underline{\mathbf 1}$ and $u\underline{\bd}+\underline{\mathbf N}$ respectively. $\delta_{H,\underline{\bd}}$ is an isomorphism by \cite[Prop.~6.9]{COZZ}. Let $$\delta=(-1)^{(N-1)n}\cdot\delta_H,\quad \delta_{\underline{\bd}}=(-1)^{(N-1)n}\cdot\delta_{H,\underline{\bd}},$$
then by \cite[Thm.~6.10]{COZZ}, the diagram
\begin{equation*}
\xymatrix{
 H^{\sT}(\cM(\underline{\bd}),\sw)\otimes H^{\sT}(\cM(\mathbf N),\sw_*) \ar[rr]^-{\Stab_{u\gtrless 0,\epsilon}} & & H^{\sT}(\cM(\underline{\bd}+\mathbf N),\sw+\sw_*)  \\
H^{\sT}(\cM(\underline{\bd}),\sw)\otimes H^{\sT}(\cM(\mathbf 1),\sw_N)  \ar[rr]^-{\Stab_{u\gtrless 0, \epsilon}} \ar[u]^{\id\otimes\delta } & & H^{\sT}(\cM(\underline{\bd}+\mathbf 1),\sw+\sw_N) \ar[u]_{\delta_{\underline{\bd}}} .
} 
\end{equation*}
commutes, with normalizer $\epsilon$ given by \eqref{normalizer_higher spin}. It follows that
\begin{align*}
    R_{\mathcal H^{\sW}_{\underline{\bd}},\mathcal H_{\underline{\mathbf N}}^{\sW_*}}(u)\circ(\id\otimes \:\delta)=(\id\otimes \:\delta)\circ R_{\mathcal H^{\sW}_{\underline{\bd}},\mathcal H_N}(u).
\end{align*}
Taking $\mathcal H^{\sW}_{\underline{\bd}}=\cH_1$ implies that $\delta$ is a $Y(\mathfrak{gl}_2)$ module isomorphism. This finishes the proof of Proposition \ref{prop comp dim red for higher spin}.
\end{proof}

%\yl{Mention $K$-theoretic version}

\subsection{Quantum multiplication by divisors}\label{sec qm by div_Jordan}

Consider the framing dimension $\underline{\bd}=(r_1,r_2)$ with $r_1\geqslant r_2$ in the following framed quiver with zero potential:
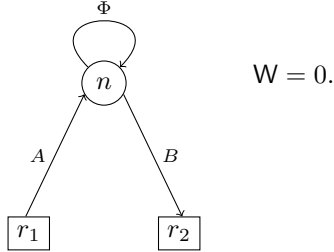
\begin{equation*}
\begin{tikzpicture}[x={(1cm,0cm)}, y={(0cm,1cm)}, baseline=0cm]
  % Nodes
  \node[draw,circle,fill=white] (Gauge) at (0,0) {$n$};
  \node[draw,rectangle,fill=white] (Framing1) at (-1,-2) {$r_1$};
  % \node (Framingdot) at (0,-2) {$\cdots$};
  % \node (Arrowdot) at (0,-1.5) {$\cdots$};
  \node[draw,rectangle,fill=white] (Framingk) at (1,-2) {$r_2$};
  \node (Z) at (0,1) {\scriptsize $\Phi$};
  % Edges
  \draw[<-] (Gauge.210) -- (Framing1.100) node[midway,left] {\scriptsize $A$};
  % \draw[->] (Gauge.240) -- (Framing1.60) node[midway,right] {\scriptsize $B_1$};

  \draw[->] (Gauge.330) -- (Framingk.80) node[midway,right] {\scriptsize $B$};
  % \draw[<-] (Gauge.300) -- (Framingk.120) node[midway,left] {\scriptsize $A_2$};

  % Loop
  \draw[->,looseness=7] (Gauge.135) to[out=135,in=45] (Gauge.45);
\end{tikzpicture}
\qquad
\sW=0. 
\end{equation*} 
Denote $X_{n,r_1,r_2}=\cM(n,\underline{\bd})$. Apply Theorem \ref{thm qm div for asym_sp inj} to $X_{n,r_1,r_2}$ and we get an explicit formula of modified quantum multiplication by divisor on $X_{n,r_1,r_2}$. 
Note that $X_{n,1,r_2}$ are affine spaces, so we shall focus on the $r_1>1$ cases.

\begin{Proposition}\label{prop on qmbd on gl2}
Assume $r_1>1$. Let $\mathcal O(1)$ be the determinant of tautological bundle $\mathsf V$ on $X_{n,r_1,r_2}$, then 
%the modified quantum multiplication by $c_1(\mathcal O(1))$ is the following
\begin{align*}
c_1(\mathcal O(1))\,\widetilde{\star}\,\cdot =
\begin{cases}
c_1(\mathcal O(1))\cup \cdot - \hbar^{-1}\frac{z}{1-z}e_0f_0, &\text{ if $r_1=r_2$ },\\
c_1(\mathcal O(1))\cup \cdot -\hbar^{-1}z\,e_0f_0, &\text{ if $r_1>r_2$ },
\end{cases}
\end{align*}
where $e_0,f_0$ are part of the generators of $Y_{r_2-r_1}(\mathfrak{gl}_2)$ in Definition \ref{def of shifted Y(gl_2)}.
\end{Proposition}

\begin{proof}
By Theorem \ref{thm compare with MO yangian} and the formula for the classical $R$-matrix \eqref{r-mat gl2}, the Casimir operators for $\mathsf Y_0(Q,0)$ are $$\mathrm{Cas}_{1}=\hbar^{-1}e_0f_0, \quad \mathrm{Cas}_{i}=0, \,\, \mathrm{if}\, \, i>1.$$ 
Then by Theorem \ref{thm on qm div for sym}, we have 
$$c_1(\mathcal O(1))\,\widetilde{\star}\,\cdot=c_1(\mathcal O(1))\cup \cdot - \hbar^{-1}\frac{z}{1-z}e_0f_0+\mathrm{const}. $$ 
The constant is given by $\hbar^{-1}\frac{z}{1-z}e_0f_0$ acting on $1\in H^{\sT_0}(X_{n,r_1,r_2})$. According to Theorem \ref{cor shifted yangian action}, $f_0(1)$ is given by the pushforward of the fundamental class $[\overline{\mathfrak{P}}(n,n-1,\underline{\bd})]^{\mathrm{t}}$ to $X_{n-1,r_1,r_2}$. It follows from \eqref{two complexes} that $\overline{\mathfrak{P}}(n,n-1,\underline{\bd})$ is a $\bP^{r_1-1}$-bundle over $X_{n-1,r_1,r_2}$, so $f_0(1)=0$. This proves the $r_1=r_2$ case.

For the $r_1>r_2$ case, it suffices to show that 
$$\mathrm{Cas}_{1,r_2-r_1}= \hbar^{-1}e_0f_0, $$ 
where $e_0,f_0$ are part of the generators of $Y_{r_2-r_1}(\mathfrak{gl}_2)$. We put bar on the generators of $Y_{r_2-r_1}(\mathfrak{gl}_2)$ to distinguish from their counterparts of $Y_{0}(\mathfrak{gl}_2)$. From the proof of Lemma \ref{lem shift map}, we have equations of operators: 
$$(p^{*})^{-1}e(u)\,p^{*}=H_{z}\,\bar e(u)\,H_{z}^{-1}=\frac{(z-c_1(\mathcal L))^{r_1-r_2}}{u-c_1(\mathcal L)}\,[\overline{\mathfrak{P}}(n+1,n,\underline{\bd})]
=z^{r_1-r_2}\bar e(u)+\text{lower order in $z$}, $$
$$(p^{*})^{-1}f(u)\,p^{*}=(-1)^{r_1-r_2}\bar f(u).$$
By expanding in the $u$-power, we obtain 
$$(p^{*})^{-1}e_r\,p^{*}=z^{r_1-r_2}\, \bar e_r+\text{lower order in $z$}, \quad (p^{*})^{-1}f_r\,p^{*}=(-1)^{r_1-r_2}\bar f_r. $$ 
This implies that 
\begin{equation*}
(p^{*})^{-1}\mathrm{Cas}_1\,p^{*}=\hbar^{-1}(p^{*})^{-1}e_0f_0\,p^{*}=(-z)^{r_1-r_2}\,\hbar^{-1}\, \bar e_0\bar f_0+\text{lower order in $z$},
\end{equation*}
therefore $\mathrm{Cas}_{1,r_2-r_1}=\lim_{z\to \infty}(-z)^{r_2-r_1}(p^{*})^{-1}\mathrm{Cas}_1\,p^{*}=\hbar^{-1}\bar e_0\bar f_0$ by Definition \ref{def of shifted Casimir}.
\end{proof}

\begin{Remark}\label{rm on Q not cp with dim red}
Take $r_1=r_2=r$, and let $\sw=\tr(\Phi AB)$, by dimensional reduction, we have $$H^{\sT_0}(X_{n,r,r},\sw)\cong H^{\sT_0}(T^*\Gr(n,r)). $$
The modified quantum multiplication by $c_1(\mathcal O(1))$ on $T^*\Gr(n,r)$ is computed in \cite[Prop.~11.2.2]{MO}\footnote{The generators $e_0$ and $f_0$ are identified with $\hbar e$ and $-\hbar f$ in \cite[\S 11]{MO}, respectively.}, given by 
\begin{align*}
c_1(\mathcal O(1))\,\widetilde{\star}\,\cdot =
c_1(\mathcal O(1))\cup \cdot -\hbar^{-1}\frac{z}{1-z}e_0f_0+\mathrm{const},
\end{align*}
for some nonzero constant. On the other hand, we have an injective specialization map 
$$H^{\sT_0}(X_{n,r,r},\sw)\to H^{\sT_0}(X_{n,r,r}). $$ Therefore by Theorem \ref{thm qm div for asym_sp inj}, the modified quantum multiplication by $c_1(\mathcal O(1))$ on $H^{\sT_0}(X_{n,r,r},\sw)$ is given by 
$$c_1(\mathcal O(1))\cup \cdot -\hbar^{-1}\frac{z}{1-z}e_0f_0. $$ This example shows that the quantum multiplication by divisor defined using stable maps \eqref{equ on Q map} is in general \textit{not compatible} with dimensional reduction. 
We expect a compatibility holds when the quantum multiplication is defined using twisted maps, see Remark \ref{rmk on two qbet}.
\end{Remark}

\section{Example: tripled Jordan quiver and modules of shifted affine Yangians of \texorpdfstring{$\mathfrak{gl}_1$}{gl(1)}}\label{sect on hilbc3}

% \section{Example\colon Hilbert schemes of points on \texorpdfstring{$\mathbb{C}^3$}{C3}}

% In this section, we apply the technology developed above to examples. 

In this section, we focus on the tripled Jordan quiver with potential:
\begin{equation}\label{equ on trip jord qui}
Q:
\begin{tikzpicture}[x={(1cm,0cm)}, y={(0cm,1cm)}, baseline=0cm]
  % Nodes
  \node[draw,circle,fill=white] (Gauge) at (0,0) {$\phantom{n}$};
  % \node[draw,rectangle,fill=white] (Framing) at (1,0) {};
  \node (Z) at (0,1) {\scriptsize $Z$};
  \node (X) at (-0.86,-0.5) {\scriptsize $X$};
  \node (Y) at (0.86,-0.5) {\scriptsize $Y$};
  % Edges
 % \draw[<-] (Gauge.0) -- (Framing.180) node[midway,below] {\scriptsize $\varphi$};

  % Loop
  \draw[->,looseness=7] (Gauge.120) to[out=120,in=60] (Gauge.60);
  \draw[->,looseness=7] (Gauge.240) to[out=240,in=180] (Gauge.180);
  \draw[->,looseness=7] (Gauge.0) to[out=0,in=300] (Gauge.300);
\end{tikzpicture}\qquad
\sW=Z[X,Y].
\end{equation}
% We focus on the cohomology constructions.
Let $\bC^*_{q_1},\bC^*_{q_2}$ and $\bC^*_{q_3}$ be the tori that scale the loops $X,Y$ and $Z$ with weights $-1$ respectively. 
Set $$\mathsf S=\bC^*_{q_1}\times\bC^*_{q_2}\times\bC^*_{q_3}, $$ 
and $$\sT_0=\ker(\mathsf S\to\bC^*),\quad (t_1,t_2,t_3)\mapsto t_1t_2t_3.$$
Let $\hbar_i$ be the equivariant parameter for $\bC^*_{q_i}$, then 
$$\bC[\mathsf t_0]=\bC[\hbar_1,\hbar_2,\hbar_3]/(\hbar_1+\hbar_2+\hbar_3). $$
We fix a cyclic stability $\theta<0$ in this section when we consider associated quiver varieties. 
%$$\cM(n,\underline{\bd})=\cM_\theta(n,\underline{\bd}), \quad \cM(\underline{\bd})=\bigsqcup_{n\in \bN} \cM(n,\underline{\bd}). $$
%Recall that for symmetric framing $\mathbf d_{\In}=\mathbf d_{\Out}=\bd$, we write 
%$$\cM(n,\bd)=\cM(n,\underline{\bd}), \quad \cM(\bd)=\cM(\underline{\bd}).$$ 

We define $\ell$-shifted affine Yangian $Y_\ell(\widehat{\mathfrak{gl}}_1)$ 
of $\mathfrak{gl}_1$ (Definition \ref{def of shifted affine Y(gl_1)}) using explicit generators with relations. 
When $\ell\leqslant 0$, we construct a natural surjective map from it to the Reshetikhin type shifted Yangian $\mathsf Y_{\ell}(Q,\sW)$ of the above $(Q,\sW)$, which we prove to be an isomorphism when $\ell=0$
(Theorem \ref{thm ex tripled Jordan_main}).
Building on this, we derive an explicit formula for quantum multiplication by divisors for Hilbert schemes of points on $\C^3$ (Corollary \ref{cor on qmd hilb3}). 

We also study various representations of shifted affine Yangians, including 
Kirillov-Reshetikhin modules, Vacuum MacMahon modules, prefundamental modules
(Theorem \ref{prop on 3 modules of gl1hat}) and modules related to Gaiotto-Rapčák's corner W-algebras (Proposition \ref{prop on more ex}).

\subsection{Shifted affine Yangians of \texorpdfstring{$\mathfrak{gl}_1$}{gl(1)}}\label{sec affine yangian}

% Consider the $\sT_0$-\textit{equivariant} 
% \textit{critical CoHA} associated to $(Q,\sW)$, denoted by $\mathcal{H}_{Q,\sW}$. 

\begin{Definition}\label{def of shifted affine Y(gl_1)}
For an integer $\ell\in \bZ$, we define the $\ell$-shifted affine Yangian of $\mathfrak{gl}_1$, denoted by $Y_\ell(\widehat{\mathfrak{gl}}_1)$, to be the $\bC(\mathsf t_0)=\bC(\hbar_1,\hbar_2)$ algebra generated by $\left\{e_j,f_j,g_j,c_j\right\}_{j\in \mathbb Z_{\geqslant 0}}$ subject to the following relations:
\begin{gather*}
[g_i,g_j]=0,\quad c_j\text{ is central},\tag{Y0}\\
[e_{i+3}, e_j]-3[e_{i+2}, e_{j+1}] + 3[e_{i+1}, e_{j+2}]-[e_i, e_{j+3}] + \sigma_2([e_{i+1}, e_j]-[e_i, e_{j+1}]) = \sigma_3\{e_i, e_j\},\tag{Y1}\\
[f_{i+3}, f_j]-3[f_{i+2}, f_{j+1}] + 3[f_{i+1}, f_{j+2}]-[f_i, f_{j+3}] + \sigma_2([f_{i+1}, f_j]-[f_i, f_{j+1}]) = -\sigma_3\{f_i, f_j\},\tag{Y2}\\
[e_i,f_j]=\sigma_3 h_{i+j},\tag{Y3}\\
[g_i, e_j]= e_{i+j},\tag{Y4}\\
% [\psi_0,e_j]=0,\quad[\psi_1,e_j]=0,\quad[\psi_2,e_j]=2e_j,\tag{Y$4'$} \\
[g_i, f_j]= -f_{i+j},\tag{Y5}\\
% [\psi_0,f_j]=0,\quad[\psi_1,f_j]=0,\quad[\psi_2,f_j]=-2f_j,\tag{Y$5'$} \\
\Sym_{\mathfrak{S}_3}[e_{i_1},[e_{i_2},e_{i_3+1}]]=0,\quad \Sym_{\mathfrak{S}_3}[f_{i_1},[f_{i_2},f_{i_3+1}]]=0,\tag{Y6}
\end{gather*}
where $\sigma_2=\hbar_1\hbar_2+\hbar_2\hbar_3+\hbar_3\hbar_1$, $\sigma_3=\hbar_1\hbar_2\hbar_3$, and the RHS of (Y3) is given by
\begin{align*}
1+\sum_{n\geqslant 0}h_{n-\ell}x^{-n-1}=\exp\left(\sum_{j\geqslant 1}\frac{c_{j-1}}{j\cdot x^j}+\sum_{k\geqslant 0}g_{k}\psi_k(x)\right),\quad h_{-\ell-1}=1,\quad \text{and }\;h_j=0\;\text{ for }\;j<-\ell-1,
\end{align*} 
where $(\psi_k(x))_{k\in \bZ_{\geqslant 0}}$ is a sequence of functions such that $\exp\left(\sum_{k\geqslant 0} a^k\psi_k(x)\right)=\prod_{s=1}^3\frac{x+\hbar_s-a}{x-\hbar_s-a}$ for a formal parameter $a$.

We set
\begin{align*}
    e(u)=\sum_{j\geqslant 0}e_ju^{-j-1},\quad f(u)=\sum_{j\geqslant 0}f_ju^{-j-1}.
\end{align*}
\end{Definition}

\begin{Remark}\label{rmk triality}
The relations (Y1)-(Y6) of $Y_\ell(\widehat{\mathfrak{gl}}_1)$ only depend on $\sigma_2$ and $\sigma_3$ which are invariant under the $\mathfrak{S}_3$ permutation of $\{\hbar_1,\hbar_2,\hbar_3\}$, so $Y_\ell(\widehat{\mathfrak{gl}}_1)$ has an $\mathfrak{S}_3$ automorphism that permutes $\{\hbar_1,\hbar_2,\hbar_3\}$.
\end{Remark}

\begin{Remark}\label{rmk positive part}
Let $Y(\widehat{\mathfrak{gl}}_1)^+$ be the $\bC(\mathsf t_0)$-algebra generated by $\{e_j\}_{j\in \bZ_{\geqslant0}}$ subject to (Y1) and the first equation of (Y6). 
By \cite[Thm.~5.11]{Dav3}, the assignment 
$$e_j\mapsto c_1(\mathcal L)^{j}\cap [\fM(1,\mathbf 0)]$$ induces an isomorphism: 
$$Y(\widehat{\mathfrak{gl}}_1)^+\cong \widetilde{\mathcal {H}}_{Q,\sW}\otimes_{\bC[\mathsf t_0]}\bC(\mathsf t_0), $$ 
where $\widetilde{\mathcal {H}}_{Q,\sW}$ is the CoHA of $(Q,\sW)$ with twisted product \eqref{twi prod}.
\end{Remark}

\begin{Lemma}\label{lem shift map affine Y(gl_1)}
For any $n\in \bZ_{\geqslant 0}$, the assignment 
\begin{align*}
    e_j\mapsto \sum_{k=0}^n(-1)^kz^{n-k} \binom{n}{k}e_{j+k},\quad f_j\mapsto (-1)^n f_j,\quad g_j\mapsto g_j,\quad c_j\mapsto c_j-nz^{j+1}
\end{align*}
generates a $\bC(\mathsf t_0)$-algebra map $\bar{\sS}_{\ell,\ell-n;z}\colon Y_{\ell}(\widehat{\mathfrak{gl}}_1)\to Y_{\ell-n}(\widehat{\mathfrak{gl}}_1)[z]$.
\end{Lemma}

\begin{proof}
We need to show that the assignment respects the relations (Y0)-(Y6).
Note that all relations except for (Y1) and the first equation in (Y6) can be checked directly. In principle, one can directly verify (Y1) and (Y6) as well, here we provide another strategy using CoHA. By Remark \ref{rmk positive part}, it is enough to show that 
\begin{align*}
c_1(\mathcal L)^{j}\cap [\fM(1,\mathbf 0)]\mapsto (z-c_1(\mathcal L))^nc_1(\mathcal L)^{j}\cap [\fM(1,\mathbf 0)]
\end{align*}
extends to a $\bC[\mathsf t_0]$-algebra homomorphism $\widetilde{\mathcal {H}}_{Q,\sW}\to \widetilde{\mathcal {H}}_{Q,\sW}[z]$. Let $\mathsf V$ be the rank $r$ tautological bundle on $\fM(r,\mathbf 0)$, then we claim that the homogeneous linear map $\phi=(\phi_r\colon \widetilde{\mathcal {H}}_{Q,\sW}(r)\to \widetilde{\mathcal {H}}_{Q,\sW}(r)[z])_{r\in \bZ_{\geqslant 0}} $ with 
\begin{align*}
\phi_r(\alpha)= z^{nr}c_{-1/z}(\mathsf V)^n\cdot\alpha 
\end{align*}
generates a $\bC[\mathsf t_0]$-algebra homomorphism $\widetilde{\mathcal {H}}_{Q,\sW}\to \widetilde{\mathcal {H}}_{Q,\sW}[z]$. In fact, the CoHA multiplication is preserved because 
\begin{align*}
z^{nr}c_{-1/z}(\mathsf V)^n=z^{nr'}c_{-1/z}(\mathsf V')^n\cdot z^{nr''}c_{-1/z}(\mathsf V'')^n\;\text{ for an SES }\; 0\longrightarrow \mathsf V'\longrightarrow \mathsf V\longrightarrow\mathsf V'''\longrightarrow 0.
\end{align*}
Since $\phi_1(c_1(\mathcal L)^{j}\cap [\fM(1,\mathbf 0)])=(z-c_1(\mathcal L))^nc_1(\mathcal L)^{j}\cap [\fM(1,\mathbf 0)]$, the lemma follows.
\end{proof}

% \begin{Remark}
% It is known that a completion of $Y_0(\widehat{\mathfrak{gl}}_1)$ in certain topology is isomorphic to $\mathcal U(\mathcal W_{1+\infty})$, the mode algebra of $\mathcal W_{1+\infty}$ vertex algebra \cite[App. F]{SV1}
% \end{Remark}

\begin{Remark}
Our shifted affine Yangian $Y_\ell(\widehat{\mathfrak{gl}}_1)$ of $\mathfrak{gl}_1$ differs from the one defined in \cite[\S 4.1]{RSYZ}.

In the zero shifted case, $Y_0(\widehat{\mathfrak{gl}}_1)$ is isomorphic to the algebra $\mathbf{SH}^{\mathbf c}$ in \cite[\S 1.8]{SV1}, which is a $\mathbb C(\xi)$ algebra  generated by $\{D_{0,j},D_{1,k},D_{-1,k},\mathbf c_k\}_{j\in \mathbb Z_{\geqslant 1},k\in \mathbb Z_{\geqslant 0}}$, and a complete set of relations given in \cite{AS}. The generators are identified by
\begin{align}\label{iso Y(gl_1) SH}
\xi\mapsto -\frac{\hbar_3}{\hbar_1},\quad D_{1,j}\mapsto \frac{-e_j}{\hbar_3\:\hbar_1^j},\quad D_{-1,j}\mapsto \frac{-f_j}{\sigma_3\:\hbar_1^j},\quad D_{0,j+1}\mapsto \frac{g_j}{\hbar_1^j},\quad \sum_{i\geqslant 0}(-1)^{i+1}\mathbf c_i\phi_i(s)\mapsto \sum_{j\geqslant 1}\frac{c_{j-1}}{j\:\hbar_1^j}s^j,
\end{align}
where $\phi_i(s)$ is defined in \cite[(1.68)]{SV1}, that is, $\exp\left(\sum_{i\geqslant 0}(-1)^{i+1}a^i\phi_i(s)\right)=\frac{1+\xi s+as}{1+as}$ for all $a\in \bC$.
\end{Remark}

\begin{Remark}\label{rmk coproduct affine Y}
By \cite[Thm.\,7.9]{SV1}, $Y_0(\widehat{\mathfrak{gl}}_1)$ admits a coproduct which is uniquely determined by 
\begin{align}\label{eqn:coproduct affine Y}
\begin{split}
    &\mathbf{\Delta}(e_0) =\square(e_0) ,\quad \mathbf{\Delta}(f_0) =\square(f_0) ,\quad \mathbf{\Delta}(g_0) =\square(g_0) ,\\
    &\mathbf{\Delta}(c_j) =\square(c_j), \;j\in\bZ_{\geqslant 0}, \\
        & \mathbf{\Delta}(g_1) = \square(g_1)  +\sigma_3 \sum_{n>0} n J_n \otimes J_{-n},
\end{split}
\end{align}
where $\square(x)=x\otimes 1+1\otimes x$, $J_n=\frac{\sigma_3^{-n}}{(n-1)!} \text{ad}^{n-1} _{-f_1} f_0$ and $J_{-n}=-\frac{\sigma_3^{-n}}{(n-1)!} \text{ad}^{n-1} _{e_1} e_0$ for $n>0$. 
% This coproduct can be characterized as follows. 
\end{Remark}

Apply the Gauss decomposition \eqref{gauss decomp of R} to the tripled Jordan quiver $Q$, and we get
\begin{align*}
\begin{pmatrix}
\langle\mathbf 0|R^{\mathrm{sup}}(u)|\mathbf 0\rangle & \langle\mathbf 0|R^{\mathrm{sup}}(u)|\delta_i\rangle \\
\langle\delta_i|R^{\mathrm{sup}}(u)|\mathbf 0\rangle & \langle\delta_i|R^{\mathrm{sup}}(u)|\delta_i\rangle
\end{pmatrix}=
\begin{pmatrix}
1 & 0 \\
\frac{-t}{\sigma_3}\mathsf f(u) & 1
\end{pmatrix}
\begin{pmatrix}
\mathsf g(u) & 0 \\
0 & \mathsf g(u)\mathsf h(u) 
\end{pmatrix}
\begin{pmatrix}
1 & \mathsf e(u)\\
0 & 1 
\end{pmatrix},
\end{align*}
where $t$ is the $\sT_0$ weight of the out-going framing of the auxiliary space $\cM(\delta_i)$, and 
$\mathsf e(u), \mathsf f(u), \mathsf g(u), \mathsf h(u)$ are elements in $\mathsf Y_{\mu}(Q,\sW)$ given by 
\begin{gather*}
\mathsf e(u)=\frac{(-1)^{n+1}}{u-c_1(\mathcal L)}\,[\overline{\mathfrak{P}}(n+1,n,\underline{\bd})],\quad \mathsf f(u)=\frac{(-1)^{n+\bd_{\Out}}}{u-c_1(\mathcal L)}\,[\overline{\mathfrak{P}}(n,n-1,\underline{\bd})]^{\mathrm{t}},\\
\mathsf g(u)=c_{-1/u}((1-t^{-1})\mathsf V),\quad
\mathsf h(u)=u^{\mu}c_{-1/u}\left((q_1^{-1}+q_2^{-1}+q_3^{-1}-q_1-q_2-q_3)\mathsf V+\mathsf D_{\Out}-\mathsf D_{\In}\right).
\end{gather*}
Here $\overline{\mathfrak{P}}(n+1,n,\underline{\bd})\subset \cM(n+1,\underline{\bd})\times \cM(n,\underline{\bd})$ is the correspondence defined below the equation \eqref{sph coha act}, $\mathcal L$ is the tautological line bundle on $\overline{\mathfrak{P}}(n+1,n,\underline{\bd})$, $\mathsf V$ is the tautological vector bundle of rank $n$ on $\cM(n,\underline{\bd})$.

\begin{Theorem}\label{thm ex tripled Jordan_main}
Let $Q$ be the tripled Jordan quiver, equipped with the canonical cubic potential $\sW=Z[X,Y]$. For arbitrary $\mu\in \bZ_{\leqslant 0}$, the assignment 
\begin{align}\label{map rho_m}
e(u)\mapsto \mathsf e(u),\quad f(u)\mapsto \mathsf f(u),\quad g_j\mapsto j!\:\mathrm{ch}_j(\mathsf V),\quad c_j\mapsto (j+1)!\:\mathrm{ch}_{j+1}(\mathsf D_{\In}-\mathsf D_{\Out}),
\end{align}
generates a surjective $\bC(\mathsf t_0)$-algebra map 
\begin{align*}
    \varrho_\mu\colon Y_{\mu}(\widehat{\mathfrak{gl}}_1)\twoheadrightarrow \mathsf Y_{\mu}(Q,\sW).
\end{align*}
Moreover, the family of maps $\varrho_\mu$ satisfy the following properties:
\begin{enumerate}
    \item the map $\varrho_0$ is a bialgebra isomorphism,
    \item for $\mu'<\mu\leqslant 0$, the parametrized shift homomorphisms $\sS_{\mu,\mu';z}$ \S \ref{sec shift map} and $\bar \sS_{\mu,\mu';z}$ (defined in Lemma \ref{lem shift map affine Y(gl_1)}) are compatible with $\varrho_\mu$ and $\varrho_{\mu'}$, that is, 
    \begin{align*}
        \sS_{\mu,\mu';z}\circ\varrho_\mu=\varrho_{\mu'}\circ\bar \sS_{\mu,\mu';z}\:.
    \end{align*}
%     the following diagram commutes:
% \begin{equation*}
% \xymatrix{
% Y_{\mu}(\widehat{\mathfrak{gl}}_1) \ar[d]_{\bar s_{m,m'}} \ar[r]^{\varrho_\mu} & \mathsf Y_{\mu}(Q,\sW) \ar[d]^{s_{m,m'}} \\
% Y_{\mu'}(\widehat{\mathfrak{gl}}_1) \ar[r]^{\varrho_{\mu'}} & \mathsf Y_{\mu'}(Q,\sW)
% }
% \end{equation*}
    % \item for any $m\in \bZ_{\geqslant 0}$, the image of $Y_{-m}(\widehat{\mathfrak{gl}}_1)$ and the central elements $\left\{\mathrm{ch}_k(\mathsf D_{\In}-\mathsf D_{\Out})\right\}_{k\in \bZ_{\geqslant 1}}$ generate $\mathsf Y_{-m}(Q,\sW)$.
\end{enumerate}
\end{Theorem}

\begin{proof}
\textbf{Step 1.} We show that \eqref{map rho_m} defines a $\bC(\mathsf t_0)$-algebra map. Let $Y(\widehat{\mathfrak{gl}}_1)^+$ be the $\bC(\mathsf t_0)$-algebra generated by $\{e_j\}_{j\in \bZ_{\geqslant0}}$ subject to (Y1) and the first equation of (Y6). We claim that 
\begin{align*}
    e_j\mapsto c_1(\mathcal L)^{j}\cap [\fM(1,\mathbf 0)^{\mathrm{nil}}]\in \widetilde{\mathcal {SH}}_{Q,\sW}^{\mathrm{nil}}
\end{align*}
gives a $\bC(\mathsf t_0)$-algebra isomorphism 
\begin{equation}\label{equ on ygl1+}Y(\widehat{\mathfrak{gl}}_1)^+\cong \widetilde{\mathcal {SH}}_{Q,\sW}^{\mathrm{nil}}\otimes_{\bC[\mathsf t_0]}\bC(\mathsf t_0). \end{equation} 
% Dimensional reduction gives an isomorphism 
% $$\widetilde{\mathcal {H}}_{Q,\sW}\cong \mathcal A^{\sT_0}_{\bC^2}, $$ 
% where $\mathcal A^{\sT_0}_{\bC^2}$ is the $\sT_0$-equivariant preprojective CoHA on $\bC^2$ studied in \cite{Dav3}. 
By Remark \ref{rmk positive part}, the assignment 
$$e_j\mapsto c_1(\mathcal L)^{j}\cap [\fM(1,\mathbf 0)]$$ induces an isomorphism 
$$Y(\widehat{\mathfrak{gl}}_1)^+\cong \widetilde{\mathcal {H}}_{Q,\sW}\otimes_{\bC[\mathsf t_0]}\bC(\mathsf t_0). $$ 
Since $e_j\mapsto -\sigma_3e_j$ is an automorphism of $Y(\widehat{\mathfrak{gl}}_1)^+$, we only need to show that the natural map
\begin{align*}
    H^{\sT_0}(\fM(n,\mathbf 0),\sw)_{\fM(n,\mathbf 0)^{\mathrm{nil}}}\to H^{\sT_0}(\fM(n,\mathbf 0),\sw)
\end{align*}
is an isomorphism after tensoring with $\bC(\mathsf t_0)$. Consider the semisimplification map 
$$\mathsf{JH}\colon \fM(n,\mathbf 0)\to \cM_0(n,\mathbf 0),$$ 
where the target is the affine quotient space, then $\fM(n,\mathbf 0)^{\mathrm{nil}}=\mathsf{JH}^{-1}(0)$, where $0\in \cM_0(n,\mathbf 0)$ corresponds to the trivial quiver representation. Then it boils down to show that
\begin{align}\label{loc after JH}
    H^{\sT_0}(\cM_0(n,\mathbf 0),i_*i^!\mathsf{JH}_*\varphi_\sw\omega_{\fM(n,\mathbf 0)})\to H^{\sT_0}(\cM_0(n,\mathbf 0),\mathsf{JH}_*\varphi_\sw\omega_{\fM(n,\mathbf 0)})
\end{align}
is an isomorphism after tensoring with $\bC(\mathsf t_0)$. Note that the vanishing cycle sheaf $\varphi_\sw\omega_{\fM(n,\mathbf 0)}$ is supported on the critical locus $\Crit(\sw)$, and $\mathsf{JH}$ maps $\Crit(\sw)$ to its affine quotient $\Crit(\sw)/\!\!/ \GL(n)$, which is isomorphic to $\bC^{3n}/\mathfrak{S}_n$. Since $\sT_0$-fixed locus of $\bC^{3n}/\mathfrak{S}_n$ consists of only one point $0$, applying $\sT_0$-localization to $\mathsf{JH}_*\varphi_\sw\omega_{\fM(n,\mathbf 0)}$ shows that \eqref{loc after JH} is an isomorphism after tensoring with $\bC(\mathsf t_0)$. This proves the claimed isomorphism \eqref{equ on ygl1+}.

Combining the above claim with Theorem \ref{cor shifted yangian action}, we see that \eqref{map rho_m} respects the relations (Y1) and the first equation of (Y6). Taking the opposite relations, we see that \eqref{map rho_m} also respects (Y2) and the second equation of (Y6). It is easy to see that (Y0), (Y4), and (Y5) are respected by \eqref{map rho_m}. (Y3) follows from Theorem \ref{cor shifted yangian action}. This proves that \eqref{map rho_m} defines a $\bC(\mathsf t_0)$-algebra map $\varrho_\mu$.

\textbf{Step 2.} We prove Property (2). Comparing the formulas for $\bar \sS_{\mu,\mu';z}$ in Lemma \ref{lem shift map affine Y(gl_1)} and for $\sS_{\mu,\mu';z}$ in Remark \ref{rmk form of shift}, it is straightforward to see that $\sS_{\mu,\mu';z}\circ\varrho_\mu=\varrho_{\mu'}\circ\bar \sS_{\mu,\mu';z}$.

\textbf{Step 3.} We prove Property (1). We will use the isomorphism (Theorem \ref{thm compare with MO yangian}):
$$\mathsf Y_0(Q,\sW)\cong \mathsf Y^{\mathrm{MO}}_{Q'}\otimes_{\bC[\mathsf t_0]}\bC(\mathsf t_0), $$ 
where $Q'$ is the Jordan quiver. This isomorphism is induced by dimensional reduction, which we may assume is along the $Z$-loop in quiver $Q$ without loss of generality. 
And operators $e(u), f(u), g_j, c_j$ are mapped to
\begin{equation}\label{rep inst moduli}
\begin{gathered}
e(u)\mapsto\frac{-\sigma_3}{u-c_1(\mathcal L)}\,[\mathfrak{B}(n+1,n,r)],\quad f(u)\mapsto\frac{(-1)^{r}\sigma_3}{u-c_1(\mathcal L)}\,[\mathfrak{B}(n,n-1,r)]^{\mathrm{t}},\\
g_j\mapsto j!\:\mathrm{ch}_j(\mathsf V),\quad
c_j\mapsto (j+1)!\:\mathrm{ch}_{j+1}((1-q_3)\mathsf D_{\In}).
\end{gathered}
\end{equation}
Here $\mathfrak{B}(n+1,n,r)\subset \cN(n+1,r)\times \cN(n,r)$ is the following correspondence: if we identify $\cN(n,r)$ with moduli space of framed rank $r$ torsion-free sheaves with second Chern number $-n$ on $\bP^2$, then $\mathfrak{B}(n+1,n,r)$ consists of pairs of sheaves $(\mathcal F,\mathcal F')$ such that $\mathcal F\subset\mathcal F'$ and $\mathrm{length}\:\mathcal F'/\mathcal F=1$. $\mathcal L$ is the line bundle on $\mathfrak{B}(n+1,n,r)$ whose fiber at $(\mathcal F,\mathcal F')$ is $\mathcal F'/\mathcal F$, $\mathsf V$ is the tautological vector bundle of rank $n$ on $\cN(n,r)$.

In view of the isomorphism $Y_0(\widehat{\mathfrak{gl}}_1)\cong \mathbf{SH}^{\mathbf c}$ \eqref{iso Y(gl_1) SH}, the induced representation \eqref{rep inst moduli} of $\mathbf{SH}^{\mathbf c}$ on $H^{\sT_0\times\sA^{\mathrm{fr}}}(\cN(r))$ agrees with the one given in \cite[Thm.~3.2]{SV1}. Here $\sA^{\mathrm{fr}}$ acts on the framing $\bC^{r}$ with weights $a_1,\ldots,a_{r}$. By \cite[Thm.~3.2]{SV1}, the induced map $\mathbf{SH}^{\mathbf c}\to \prod_{r}\End(H^{\sT_0\times\sA^{\mathrm{fr}}}(\cN(r)))$ is injective. It follows that $\varrho_0\colon Y_0(\widehat{\mathfrak{gl}}_1)\to \mathsf Y_0(Q,\sW)$ is injective.

For the surjectivity of $\varrho_0$, we notice that $e_0$ (resp. $f_0$) is the Baranovsky operator $\beta_{-1}(-\sigma_3)$ (resp.~$\beta_{1}(-\sigma_3)$) \cite[\S 12.2]{MO}, and that tautological classes $\{\mathrm{ch}_j(\mathsf V),\mathrm{ch}_j(\mathsf D_{\In})\}_{j\in \bZ_{\geqslant0}}$ are in the image of $\varrho_0$. Then $\varrho_0$ is surjective by \cite[Thm.~18.1.1]{MO}. This implies that $\varrho_0$ is a bijection.

Finally, by Lemma \ref{lem Baranovsky op} below and \cite[\S 12.4.9]{MO}, the classical $R$-matrix $\mathbf r$ of $\mathsf Y_0(Q,\sW)$ is\,\footnote{The classical $R$-matrix $\mathbf r$ of $\mathsf Y_0(Q,\sW)$ is $-\hbar_3$ times the classical $R$-matrix of $\mathsf Y_{Q'}^{\mathrm{MO}}$ due to different normalizations.}
\begin{equation}\label{equ on r matrix gl1hat}
    \mathbf r=\mathsf v\otimes\mathsf d+\mathsf d\otimes\mathsf v-\sigma_3\sum_{n\neq 0}\varrho_0(J_n)\otimes \varrho_0(J_{-n}).
\end{equation}
Then the coproduct formula \eqref{coproduct for 1st Chern class} reads:
\begin{align*}
    \Delta c_1(\mathsf V)=\square \,c_1(\mathsf V)+\sigma_3\sum_{n>0}n\varrho_0(J_n)\otimes \varrho_0(J_{-n}),
\end{align*}
that is, 
$$\Delta(\varrho_0(g_1))=(\varrho_0\otimes \varrho_0)\mathbf\Delta(g_1). $$ 
It is easy to see that for $x=e_0,f_0,g_0,c_j$ ($j\in \bZ_{\geqslant0}$), we have
$$\Delta(\varrho_0(x))=(\varrho_0\otimes \varrho_0)\mathbf\Delta(x). $$ Thus $\varrho_0$ is a bialgebra map.

\textbf{Step 4.} We prove the surjectivity of $\varrho_\mu$. By Lemma \ref{lem shift generation lemma}, $\mathsf Y_\mu(Q,\sW)$ is generated by $\bigcup_{a\in \bC}\sS_{0,\mu;z}(\mathsf Y_0(Q,\sW))\big|_{z=a}$. Since $\varrho_0$ is an isomorphism, we have
\begin{align*}
\sS_{0,\mu;z}(\mathsf Y_0(Q,\sW))=\sS_{0,\mu;z}(\varrho_0(Y_0(\widehat{\mathfrak{gl}}_1)))=\varrho_\mu(\bar\sS_{0,\mu;z}(Y_0(\widehat{\mathfrak{gl}}_1)))\subseteq \varrho_\mu(Y_\mu(\widehat{\mathfrak{gl}}_1))[z],
\end{align*}
and it follows that $\bigcup_{a\in \bC}\sS_{0,\mu;z}(\mathsf Y_0(Q,\sW))\big|_{z=a}\subseteq \varrho_\mu(Y_\mu(\widehat{\mathfrak{gl}}_1))$; therefore $\varrho_\mu$ is surjective.
\end{proof}

\begin{Lemma}\label{lem Baranovsky op}
Let $J_{\pm n}$ be the operators defined in Remark \ref{rmk coproduct affine Y}, then $\varrho_0(J_{\pm n})=\beta_{\pm n}(1)$, where $\beta_{\pm n}(1)$ are the Baranovsky operators defined in \cite[\S 12.2]{MO}.
\end{Lemma}

\begin{proof}
Apply \cite[Prop.~5.2]{Dav3} iteratively to $e_j\propto \alpha_1^{(j)}$\,\footnote{Here $\alpha_1^{(j)}$ is given in \cite[\S 5]{Dav3} and $\propto$ means equal up to multiplying a scalar.}, and we see that the image of $J_{-n}=-\frac{\sigma_3^{-n}}{(n-1)!} \text{ad}^{n-1} _{e_1} e_0$ under $\varrho_0$ is in the BPS Lie algebra (isomorphic to the MO Lie algebra $\mathfrak{g}_{Q'}^{\mathrm{MO}}$ by \cite{BD}); therefore $\varrho_0(J_{\pm n})\propto \beta_{\pm n}(1)$ because $\mathfrak{g}_{Q'}^{\mathrm{MO}}$ has a basis $\{\mathsf v,\mathsf d,\beta_n(1)\:|\: n\neq 0\}$ \cite[\S 12.4.10]{MO}. 

We determine the coefficients as follows. Both $\varrho_0(J_{-n})$ and $\beta_{-n}(1)$ are primitive, that is, 
$$\Delta\varrho_0(J_{\pm n})=\square \varrho_0(J_{\pm n}), \quad \Delta\beta_{-n}(1)=\square \beta_{-n}(1). $$ 
Then, to show that $\varrho_0(J_{-n})=\beta_{-n}(1)$, it is enough to show that the actions of $\varrho_0(J_{-n})$ and $\beta_{-n}(1)$ agree on $H^{\sT_0}(\Hilb\bC^2)$. $\beta_{-n}(1)$ acts on $H^{\sT_0}(\Hilb\bC^2)$ as Nakajima's raising operator $\alpha_{-n}$, and $\varrho_0(J_{-n})$ acts on $H^{\sT_0}(\Hilb\bC^2)$ as $\frac{1}{(n-1)!} \text{ad}^{n-1} _{-\alpha'_{-1}} \alpha_{-1}$ where $\alpha'_{-1}$ is the derivative of $\alpha_{-1}$ defined by Lehn in \cite[Def.~3.8]{Leh}. 

By \cite[Thm.~3.10]{Leh}, we have $$\alpha_{-n}=\frac{1}{(n-1)!} \text{ad}^{n-1} _{-\alpha'_{-1}} \alpha_{-1}.$$ This shows that $\varrho_0(J_{-n})=\beta_{-n}(1)$. Comparing the commutators $[J_n,J_{-n}]=\frac{n}{\sigma_3}c_0$ (by direct computation) and 
$$[\beta_n(1),\beta_{-n}(1)]=\frac{-rn}{\hbar_1\hbar_2}=\frac{n}{\sigma_3}\varrho_0(c_0), $$ 
(by \cite[(12.6)]{MO}), we conclude that $\varrho_0(J_{n})=\beta_{n}(1)$.
\end{proof}

% \begin{Remark}
% The core Yangian $\mathbb Y_{Q'}^{\text{MO}}\otimes_{\bC[\mathsf t_0]}\bC(\mathsf t_0)\subset \mathsf Y_{Q'}^{\text{MO}}\otimes_{\bC[\mathsf t_0]}\bC(\mathsf t_0)$ is isomorphic to the $\bC(\mathsf t_0)$ subalgebra of  generated by 
% \end{Remark}

We conjecture that the isomorphism $\varrho_0$ extends to any $\mu\leqslant 0$.
\begin{Conjecture}
For arbitrary $\mu\in \bZ_{\leqslant 0}$, the map $\varrho_\mu\colon Y_{\mu}(\widehat{\mathfrak{gl}}_1)\to \mathsf Y_{\mu}(Q,\sW)$ is an isomorphism. 
\end{Conjecture}

\subsection{Various representations from framed tripled Jordan quiver with potentials}\label{subsec modules of Ygl1hat}

Let us consider the following framings with potentials:

\begin{enumerate}

\item (Kirillov-Reshetikhin module) $\underline{\bd}=\mathbf 1=(1,1)$, the potential is $\sW_N=\sW+BZ^NA$ (with $N\in \bZ_{\geqslant 0}$), and we require that $\sT_0$ scales $B$ by the character $q_3^N$ and fixes $A$.  
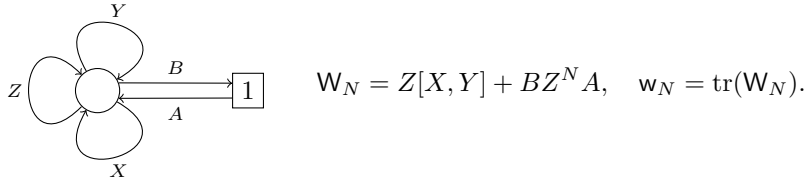
\begin{equation*}
\begin{tikzpicture}[x={(1cm,0cm)}, y={(0cm,1cm)}, baseline=0cm]
  % Nodes
  \node[draw,circle,fill=white] (Gauge) at (0,0){$\phantom{n}$}; %{$\bv$}  {$\phantom{n}$};
  \node[draw,rectangle,fill=white] (Framing) at (2,0) {1};
  \node (Z) at (-1.1,0) {\scriptsize $Z$};
  \node (Y) at (0.28,1.06) {\scriptsize $Y$};
  \node (X) at (0.28,-1.06) {\scriptsize $X$};
  % Edges
  \draw[<-] (Gauge.340) -- (Framing.205) node[midway,below] {\scriptsize $A$};
  \draw[->] (Gauge.20) -- (Framing.155) node[midway,above] {\scriptsize $B$};

  % Loop
  \draw[->,looseness=8] (Gauge.225) to[out=225,in=135] (Gauge.135);
  \draw[->,looseness=8] (Gauge.120) to[out=120,in=30] (Gauge.30);
  \draw[->,looseness=8] (Gauge.330) to[out=330,in=240] (Gauge.240);
\end{tikzpicture}
\qquad \sW_N=Z[X,Y]+BZ^NA, \quad \sw_N=\tr(\sW_N).
\end{equation*}
%\begin{equation*} \xymatrix{\text{\tiny{\boxed{1}}}  \ar@/^0.3pc/[r]^{A}  & \ar@/^0.3pc/[l]^{B} \text{\textcircled{v}}  \ar@(dr,ur)_{\Phi}  }
%\end{equation*}
We define  
\begin{equation*}
   \mathcal H_N(n):=H^{\sT_0}(\cM(n,\mathbf 1),\sw_N), \quad  \mathcal H_N:=H^{\sT_0}(\cM(\mathbf 1),\sw_N)=\bigoplus_{n\in \bZ_{\geqslant 0}} \mathcal H_N(n).
\end{equation*}

\item (Vacuum MacMahon module) $\underline{\bd}=\mathbf 1=(1,1)$, the potential is $\sW$, and take another torus $\bC^*_t$ which scales $B$ with weight $1$ and fixes other quiver data.
    
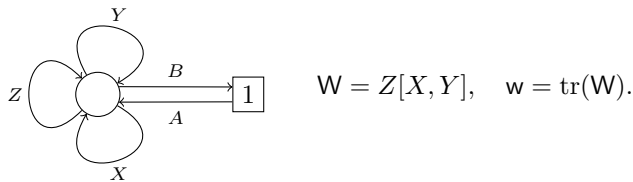
\begin{equation*}
\begin{tikzpicture}[x={(1cm,0cm)}, y={(0cm,1cm)}, baseline=0cm]
  % Nodes
  \node[draw,circle,fill=white] (Gauge) at (0,0){$\phantom{n}$}; % {$\phantom{n}$};
  \node[draw,rectangle,fill=white] (Framing) at (2,0) {1};
  \node (Z) at (-1.1,0) {\scriptsize $Z$};
  \node (Y) at (0.28,1.06) {\scriptsize $Y$};
  \node (X) at (0.28,-1.06) {\scriptsize $X$};
  % Edges
  \draw[<-] (Gauge.340) -- (Framing.205) node[midway,below] {\scriptsize $A$};
  \draw[->] (Gauge.20) -- (Framing.155) node[midway,above] {\scriptsize $B$};

  % Loop
  \draw[->,looseness=8] (Gauge.225) to[out=225,in=135] (Gauge.135);
  \draw[->,looseness=8] (Gauge.120) to[out=120,in=30] (Gauge.30);
  \draw[->,looseness=8] (Gauge.330) to[out=330,in=240] (Gauge.240);
\end{tikzpicture}
\qquad \sW=Z[X,Y], \quad \sw=\tr(\sW).
\end{equation*}
We define 
\begin{align*}
    \mathcal V(n):=H^{\sT_0\times \bC^*_t}(\cM(n,\mathbf 1),\sw), \quad \mathcal 
    V:=H^{\sT_0\times \bC^*_t}(\cM(\mathbf 1),\sw)=\bigoplus_{n\in \bZ_{\geqslant 0}}\mathcal V(n).
\end{align*}

\item (Prefundamental module) $\underline{\bd}=\underline{\mathbf 0}^{\mathbf 1}=(1,0)$, the potential is $\sW$.
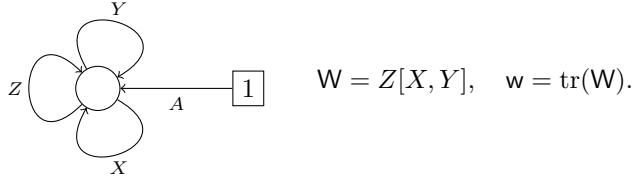
    \begin{equation*}
\begin{tikzpicture}[x={(1cm,0cm)}, y={(0cm,1cm)}, baseline=0cm]
  % Nodes
  \node[draw,circle,fill=white] (Gauge) at (0,0){$\phantom{n}$}; % {$\phantom{n}$};
  \node[draw,rectangle,fill=white] (Framing) at (2,0) {1};
  \node (Z) at (-1.1,0) {\scriptsize $Z$};
  \node (Y) at (0.28,1.06) {\scriptsize $Y$};
  \node (X) at (0.28,-1.06) {\scriptsize $X$};
  % Edges
  \draw[<-] (Gauge.0) -- (Framing.180) node[midway,below] {\scriptsize $A$};
  % \draw[->] (Gauge.20) -- (Framing.155) node[midway,above] {\scriptsize $B$};

  % Loop
  \draw[->,looseness=8] (Gauge.225) to[out=225,in=135] (Gauge.135);
  \draw[->,looseness=8] (Gauge.120) to[out=120,in=30] (Gauge.30);
  \draw[->,looseness=8] (Gauge.330) to[out=330,in=240] (Gauge.240);
\end{tikzpicture}
\qquad \sW=Z[X,Y], \quad \sw=\tr(\sW).
\end{equation*}
We define 
\begin{equation*}
\mathcal F(n):=H^{\sT_0}(\cM(n,\underline{\mathbf 0}^{\mathbf 1}),\sw), \quad 
\mathcal F:=H^{\sT_0}(\cM(\underline{\mathbf 0}^{\mathbf 1}),\sw)=\bigoplus_{n\in \bZ_{\geqslant 0}}\mathcal F(n).
\end{equation*}
\end{enumerate}

\begin{Remark}\label{rmk purity for V and F}
The mixed Hodge structures on derived global sections $\mathrm{R\Gamma}(\cM(n,\mathbf 1),\sw)$ and $\mathrm{R\Gamma}(\cM(n,\underline{\mathbf 0}^{\mathbf 1}),\sw)$ are pure. This can be seen as follows. $\mathrm{R\Gamma}(\cM(n,\mathbf 1),\sw)$ is a direct summand of $\mathrm{R\Gamma}(\fM(n,\mathbf 1),\sw)$ (using the nonabelian stable envelope, see \cite[\S 9.1]{COZZ}). By a vector bundle pullback isomorphism followed by dimensional reduction, $\mathrm{R\Gamma}(\fM(n,\mathbf 1),\sw)$ is isomorphic to $\mathrm{R\Gamma}(\fX(n,\mathbf 0))$ up to degree shift and Tate twist, where $\fX(n,\mathbf 0)$ is the preprojective stack of the Jordan quiver with zero framing. $\mathrm{R\Gamma}(\fX(n,\mathbf 0))$ is pure by \cite[Thm.~A]{Dav2}, and this shows the purity for $\mathrm{R\Gamma}(\cM(n,\mathbf 1),\sw)$. The purity for $\mathrm{R\Gamma}(\cM(n,\underline{\mathbf 0}^{\mathbf 1}),\sw)$ follows from vector bundle pullback isomorphism 
$$\mathrm{R\Gamma}(\cM(n,\mathbf 1),\sw)\cong\mathrm{R\Gamma}(\cM(n,\underline{\mathbf 0}^{\mathbf 1}),\sw). $$
As a consequence of the purity, the equivariant cohomologies $\cV(n)$ and $\cF(n)$ are formal, and in particular they are free over $H_{\sT_0\times\bC^*_t}(\pt)$ and over $H_{\sT_0}(\pt)$ respectively.
\end{Remark}

\begin{Remark}
If we introduce an additional $\bG_m$ action in the example (1) that scales $B$ with weight $1$ and fixes all the other quiver data, then it fits into the assumption for the $\bG_m$-action in \cite[\S 7.2]{COZZ}, and the construction in \textit{loc.\,cit.} gives a specialization map 
\begin{equation}\label{equ on sp HN}
    \mathsf{sp}\colon \cH_{N}=H^{\sT_0}(\cM(\mathbf 1),\sw_N)\longrightarrow H^{\sT_0}(\cM(\mathbf 1),\sw)=\cV|_{t=N\hbar_3}\:,
\end{equation}
which is a $\mathsf Y_{0}(Q,\sW)$-module map because it commutes with stable envelopes and therefore it commutes with $R$-matrices. We will see shortly that $\mathsf{sp}$ is injective after localization.
\end{Remark}
\textbf{Notations.}  In below, we use the following notations. 
\begin{itemize} 
\item The $\sT_0$-localization is denoted as $(\cdots)_\loc=(\cdots)\otimes_{\bC[\mathsf t_0]}\bC(\mathsf t_0)$. 
\item $\cH^{\sW^{\mathrm{fr}}}_{\underline{\bd},\sA^{\mathrm{fr}}}[a]$ denotes the cohomology $H^{\sT_0\times \sA^{\mathrm{fr}}\times \bC^*_a}(\cM(\underline{\bd}),\sw^{\mathrm{fr}})$, which is obtained from the original state space $\cH^{\sW^{\mathrm{fr}}}_{\underline{\bd},\sA^{\mathrm{fr}}}$ by including a torus $\bC^*_a$ that scales the framing vector space with weight $1$.
\end{itemize}

\begin{Theorem}\label{prop on 3 modules of gl1hat}
In the above examples, we have the following.
\begin{enumerate}
    \item[(a)] $\cH_1$ is isomorphic to the Fock module in \cite[\S 13.1]{MO}, and $\cH_{N,\loc}$ is isomorphic to the $\mathsf Y_{0}(Q,\sW)$-submodule of $\cH_{1,\loc}[(N-1)\hbar_3]\otimes\cdots\otimes \cH_{1,\loc}[\hbar_3]\otimes \cH_{1,\loc}$ that is generated by vacuum $|\mathbf 0\rangle\otimes\cdots\otimes |\mathbf 0\rangle$.
    \item[(b)] $\cH_{N,\loc}$ and $\cV_\loc|_{t=\lambda_1\hbar_1+\lambda_2\hbar_2}$ are irreducible $\mathsf Y_{0}(Q,\sW)$-modules for $(\lambda_1,\lambda_2)\in \bC^2\setminus\bZ^2$, and $\cF_\loc$ is an irreducible $\mathsf Y_{-1}(Q,\sW)$-module.
    \item[(c)] The localized specialization map $\mathsf{sp}\colon \cH_{N,\loc}\to \cV_\loc|_{t=N\hbar_3}$ of \eqref{equ on sp HN}  
    is an injective $\mathsf Y_{0}(Q,\sW)$-module map.
\end{enumerate}
\end{Theorem}

\begin{proof}
(a) By dimensional reduction, we have $\cH_1\cong H^{\sT_0}(\cN(\mathbf 1))$, where $\cN(\mathbf 1)$ is the Nakajima variety associated to the doubled Jordan quiver with framing dimension one. In view of the isomorphism 
$$\mathsf Y_0(Q,\sW)\cong \mathsf Y_{Q'}^{\text{MO}}\otimes_{\bC[\mathsf t_0]}\bC(\mathsf t_0), $$ 
$\cH_1\cong H^{\sT_0}(\cN(\mathbf 1))$ is a $\mathsf Y_0(Q,\sW)$ module isomorphism, and the latter is the Fock module in \cite[\S 13.1]{MO}. 

The statement concerning $\cH_{N,\loc}$ is proven using a similar idea as Theorem \ref{thm fusion}. We define $\eta$ to be the composition
\begin{align}\label{the map eta}
    \eta\colon \cH_{N,\loc}\xrightarrow{\delta_1} \cH_{\underline{\mathbf N},\loc}^{\sW_*}\xrightarrow{\delta_2^{-1}} H^{\sT_0}(\cN(\mathbf N),\phi)_{\loc}\xrightarrow{\mathsf{sp}} H^{\sT_0}(\cN(\mathbf N))_{\loc}\xrightarrow{\mathbf{Stab}_{s<0}^{-1}} \cH_{1,\loc}[(N-1)\hbar_3]\otimes\cdots\otimes \cH_{1,\loc}[\hbar_3]\otimes \cH_{1,\loc},
\end{align}
where notations in the middle steps are explained as follows. 

Let $\cH_{\underline{\mathbf N}}^{\sW_*}=H^{\sT_0}(\cM(\mathbf N),\sw_*)$ (with $\sw_*=\tr(\sW_*)$) be the critical cohomology of the following quiver with potential
\begin{equation*}
\begin{tikzpicture}[x={(1cm,0cm)}, y={(0cm,1cm)}, baseline=0cm]
  % Nodes
  \node[draw,circle,fill=white] (Gauge) at (0,0){$\phantom{n}$}; %{$\bv$}  {$\phantom{n}$};
  \node[draw,rectangle,fill=white] (Framing) at (2,0) {$N$};
  \node (Z) at (-1.1,0) {\scriptsize $Z$};
  \node (Y) at (0.28,1.06) {\scriptsize $Y$};
  \node (X) at (0.28,-1.06) {\scriptsize $X$};
  % Edges
  \draw[<-] (Gauge.340) -- (Framing.200) node[midway,below] {\scriptsize $A$};
  \draw[->] (Gauge.20) -- (Framing.160) node[midway,above] {\scriptsize $B$};

  % Loop
  \draw[->,looseness=8] (Gauge.225) to[out=225,in=135] (Gauge.135);
  \draw[->,looseness=8] (Gauge.120) to[out=120,in=30] (Gauge.30);
  \draw[->,looseness=8] (Gauge.330) to[out=330,in=240] (Gauge.240);
\end{tikzpicture}
\qquad \sW_*=Z[X,Y]+ZAB-A\Xi B,
\end{equation*}
where $\Xi\in \End(\C^N)$ is the regular nilpotent matrix such that $\Xi(e_i)=e_{i+1}$ for a given bases $\bC^N=\Span_\bC\{e_1,\ldots,e_N\}$. We set $(t_1,t_2,t_3)\in \sT_0\subset\bC^*_{q_1}\times \bC^*_{q_2}\times \bC^*_{q_3}$ action on quiver data to be
\begin{align}\label{T_0 action}
    (X,Y,Z,B,A)\mapsto (t_1^{-1}X,t_2^{-1}Y,t_3^{-1}Z,t_3 \cdot\sigma(t_3)B,A\sigma(t_3)^{-1}),
\end{align}
where $\sigma\colon \bC^*_{q_3}\to \GL_N$ is the homomorphism $\sigma(t_3)=\diag(1,t_3,\ldots,t_3^{N-1})$ in the aforementioned basis. $\cN(\mathbf N)$ is the Nakajima variety associated to the doubled Jordan quiver (with loops $X$ and $Y$), with $\sT_0$ action induced by omitting $Z$ in \eqref{T_0 action}. $\phi$ is the function on $\cN(\mathbf N)$ given by
\begin{align*}
    \phi=\tr(A\Xi B).
\end{align*}
The maps involved in \eqref{the map eta} are given as follows.
\begin{itemize}
    \item $\delta_1$ is the interpolation map $\delta_H$ constructed in \cite[Prop.~6.7]{COZZ} with respect to the following compatible dimensional reduction data. Recall the notation $\underline{\mathbf 0}^{\mathbf k}=(k,0)$, and let 
\begin{equation*}
    \mathcal Y'=\cM(n,\underline{\mathbf 0}^{\mathbf 1}),\quad \mathcal X'=\mathrm{Tot}(E')= \cM(n,\mathbf 1),
\end{equation*}
$$s'=Z^N A\in \Gamma(\mathcal Y',E'^\vee),\quad \varphi'=\tr(Z[X,Y]),\quad \sw_N=\varphi'+\langle s', B\rangle,$$
where $E'$ is the vector bundle that descends from $\Hom(\bC^n,\bC)$ with coordinate given by $B$, and $\langle-,-\rangle$ is the natural pairing. Let
\begin{equation*}
\mathcal Y:=\cM(n,\underline{\mathbf 0}^{\mathbf N}),\quad \mathcal X=\mathrm{Tot}(E)= \cM(n,\mathbf N), \end{equation*}
$$s=Z A-A\Xi \in \Gamma(\mathcal Y,E^\vee), \quad \varphi=\tr(Z[X,Y]), \quad \sw_*=\varphi+\langle s, B\rangle, $$
where $E$ is the vector bundle that descends from $\Hom(\bC^n,\bC^N)$ with coordinate given by $B$. The compatibility between the dimensional reduction data $(\mathcal X,\mathcal Y,s,\varphi)$ and $(\mathcal X',\mathcal Y',s',\varphi')$ is proven similarly to Lemma \ref{lem on dr higher spin} and we shall not repeat it. By \cite[Prop.~6.9]{COZZ}, $\delta_1$ is an isomorphism. 

\item $\delta_2$ is the deformed dimensional reduction along $Z$ (which is an isomorphism by \cite[Thm.~C.1]{COZZ}).

\item $\mathsf{sp}$ is the specialization map in \cite[\S 7.2]{COZZ} constructed using the additional $\bG_m$ action that scales $B$ with weight $1$ and fixes other quiver data.

\item The stable envelope map $\mathbf{Stab}_{s<0}$ is constructed as follows. Let $(t_1,t_2,t_3,p)\in \sT_0\times \bC^*_s$ act on $\cN(\mathbf N)$ by
\begin{align*}
    (X,Y,B,A)\mapsto (t_1^{-1}X,t_2^{-1}Y,t_3 \cdot\sigma(p)B,A\sigma(p)^{-1}).
\end{align*}
The stable envelope
\begin{align*}
    \mathbf{Stab}_{s<0}\colon H^{\sT_0\times \bC^*_s}(\cN(\mathbf N)^{\bC^*_s})\to H^{\sT_0\times \bC^*_s}(\cN(\mathbf N))
\end{align*}
is a $\mathsf Y_{0}(Q,\sW)$-module map by the argument in \S \ref{sect on coprod}. Evaluation at $s=\hbar_3$ amounts to taking equivariant cohomologies with respect to the subtorus $$\sT_0\xrightarrow{(\id,q_3)}\sT_0\times \bC^*_s, $$ which will be called the ``diagonal'' $\sT_0$. Since $\cN(\mathbf N)^{\bC^*_s}=\cN(\mathbf 1)^{\times N}$, with $\bC^*_{q_3}$ weights $\{0,1,\ldots,N-1\}$ on the one-dimensional framing vector spaces, 
% torus weight $(-i,i+1)$ and $(0,1)$ are equivalent by gauge actions. 
we get a $\mathsf Y_{0}(Q,\sW)$-module map:
\begin{align*}
    \mathbf{Stab}_{s<0}\colon \cH_{1}[(N-1)\hbar_3]\otimes \cdots \otimes \cH_{1}[\hbar_3]\otimes \cH_{1}\to H^{\sT_0}(\cN(\mathbf N)).
    % \cH_{1}\otimes \tau^*_{\hbar}\cH_{1}\otimes \cdots\otimes \tau^*_{(N-1)\hbar}\cH_{1}
\end{align*}
It is elementary to see that the induced $\bC^*_s$ action on the diagonal $\sT_0$ fixed locus $\cN(\mathbf N)^{\sT_0}$ is repelling. Then by \cite[Prop.~3.36]{COZZ}, the above map is injective after localization. Moreover, the map 
\begin{align*}
    \mathbf{Stab}_{s<0}\colon \cH_{1,\loc}[(N-1)\hbar_3]\otimes\cdots\otimes \cH_{1,\loc}[\hbar_3]\otimes \cH_{1,\loc}\to H^{\sT_0}(\cN(\mathbf N))_\loc
    % \cH_{1,\loc}\otimes \tau^*_{\hbar}\cH_{1,\loc}\otimes \cdots\otimes \tau^*_{(N-1)\hbar}\cH_{1,\loc}
\end{align*}
is an isomorphism by a graded dimension counting argument.
\end{itemize}

We have seen that all intermediate maps but the specialization map $\mathsf{sp}$ are isomorphisms in \eqref{the map eta}, and we claim that $\mathsf{sp}$ is injective. To prove the claim, it is equivalent to show that the specialization map at the $\sT_0$-fixed locus:
$$H(\cN(\mathbf N)^{\sT_0},\phi)\xrightarrow{\mathsf{sp}} H(\cN(\mathbf N)^{\sT_0})$$ 
is injective. Note that $\cN(\mathbf N)^{\sT_0}$ is the same as the $\sT_0$-fixed locus of an $A_{\infty}$ quiver, so the claim follows from Lemma \ref{lem inj sp for type A} below. 
The maps $\delta_1$, $\delta_2$, and $\mathsf{sp}$ are $\mathsf Y_{0}(Q,\sW)$-module maps because they commute with stable envelopes and therefore they commute with $R$-matrices. $\mathbf{Stab}_{s<0}$ is a $\mathsf Y_{0}(Q,\sW)$-module map as we explained above. We conclude that $\eta$ is an injective $\mathsf Y_{0}(Q,\sW)$-module map. We also note that $\eta$ maps the vacuum $|\mathbf 0\rangle\in \cH_{N,\loc}$ to vacuum 
$$|\mathbf 0\rangle\otimes\cdots\otimes |\mathbf 0\rangle\in \cH_{1,\loc}[(N-1)\hbar_3]\otimes\cdots\otimes \cH_{1,\loc}[\hbar_3]\otimes \cH_{1,\loc}. $$ 
To show that $\cH_{N,\loc}$ is isomorphic to the $\mathsf Y_{0}(Q,\sW)$-submodule of $\cH_{1,\loc}[(N-1)\hbar_3]\otimes\cdots\otimes \cH_{1,\loc}[\hbar_3]\otimes \cH_{1,\loc}$ that is generated by vacuum $|\mathbf 0\rangle\otimes\cdots\otimes |\mathbf 0\rangle$, it is enough to show that $\cH_{N,\loc}$ is irreducible, which we prove below.

The subgroup $\{(x,x^{-1},1)\:|\: x\in \bC^*\}\subset \sT_0$ satisfies the assumption for the $\mathsf S$ torus in Corollary \ref{cor cogen_nak}, so $H^{\sT_0}(\cN(\mathbf N))_\loc$ is cogenerated by vacuum. As a submodule, $\cH_{N,\loc}$ is also cogenerated by vacuum. 

Let us show that $\cH_{N,\loc}$ is compact. By \cite[Prop.~6.5]{COZZ}, we have $\sT_0$-equivariant isomorphism between schemes 
$$\Crit_{\cM(\mathbf 1)}(\sw_N)\cong \Crit_{\cM(\mathbf N)}(\sw_*)\cong \Crit_{\cN(\mathbf N)}(\phi). $$ 
$\Crit_{\cN(\mathbf N)}(\phi)$ can be explicitly described as follows: it is the moduli space 
$$M=\{(X,Y,A,B)\:|\:[X,Y]+AB=0\}, $$ 
whose points represent stable representations of framed preprojective algebra $\Pi_{\overline{Q'}}^{\mathrm{fr}}$ of the doubled Jordan quiver $\overline{Q'}$, together with a $\Pi_{\overline{Q'}^{\mathrm{fr}}}$-module endomorphism $\cE$ of $M$ such that $\mathcal E$ acts on the framing vector space by $\Xi$. Note that the cyclic stability condition constrains such $\cE$ to be uniquely determined by $\Xi$. $\mathcal E$ induces a $\Pi_{\overline{Q'}}^{\mathrm{fr}}$-module filtration on $M$ such that each associated graded piece has framing dimension one. Under the Jordan-H\"older map 
$$\mathsf{JH}\colon \cN(n,\mathbf N)\to \cN_0(n,\mathbf N), $$ 
$\mathsf{JH}(\Crit_{\cN(n,\mathbf N)}(\phi))$ is contained in the union of the images of the direct sum map $\prod_{i=1}^N\cN_0(n_i,\mathbf 1)\to \cN_0(n,\mathbf N)$, 
where the union is taken for all decompositions $n=\sum_{i=1}^Nn_i$. Since $\cN_0(n_i,\mathbf 1)^{\sT_0}=\{0\}$, we see that $\mathsf{JH}(\Crit_{\cN(n,\mathbf N)}(\phi)^{\sT_0})=\{0\}$, and this implies that $\Crit_{\cN(n,\mathbf N)}(\phi)^{\sT_0}$ is compact.

As $\cH_{N,\loc}$ is compact and cogenerated by vacuum, it is irreducible by Proposition \ref{prop on irr mod}. This proves (a).  % \bigskip

(b) We have shown that $\cH_{N,\loc}$ is irreducible in the proof of (a). We claim that $\cF_\loc$ is compact. In fact, $\cM(n,\underline{\mathbf 0}^{\mathbf 1})\cong \Hilb^n(\bC^3)$, and $\sT_0\subset \bC^*_{q_1}\times \bC^*_{q_2}\times \bC^*_{q_3}$ acts on $\bC^3$ by weight $(q_1,q_2,q_3)$. It is well known that $\Hilb^n(\bC^3)^{\sT_0}$ is a finite number of points and, in particular, compact. 

The $(\sT_0\times \bC^*_t)$-weights that appear in the normal bundle $N_{\cM(\mathbf 1)^{\sT_0\times \bC^*_t}/\cM(\mathbf 1)}$ are of the form 
$$\{a_1\hbar_1+a_2\hbar_2,\: t+b_1\hbar_1+b_2\hbar_2\:|\: a_1,a_2,b_1,b_2\in \bZ\}. $$ 
So for $t=\lambda_1\hbar_1+\lambda_2\hbar_2$ such that $(\lambda_1,\lambda_2)\in \bC^2\setminus\bZ^2$, the equivariant localization implies that the embedding $\cM(\mathbf 1)^{\sT_0\times \bC^*_t}\hookrightarrow\cM(\mathbf 1)$ induces an isomorphism 
$$H^{\sT_0}(\cM(\mathbf 1)^{\sT_0\times \bC^*_t},\sw)_\loc\cong H^{\sT_0\times \bC^*_t}(\cM(\mathbf 1),\sw)_\loc\big|_{t=\lambda_1\hbar_1+\lambda_2\hbar_2}. $$ 
We note that $\cM(n,\mathbf 1)^{\sT_0\times \bC^*_t}=\cM(n,\underline{\mathbf 0}^{\mathbf 1})^{\sT_0}$, and the latter is compact. Therefore, $\cV_{\loc}|_{t=\lambda_1\hbar_1+\lambda_2\hbar_2}$ admits a $\mathsf Y_0(Q,\sW)$-invariant nondegenerate bilinear form.

By Proposition \ref{prop on irr mod}, to prove that $\cF_\loc$ and $\cV_{\loc}|_{t=\lambda_1\hbar_1+\lambda_2\hbar_2}$ are irreducible, it remains to show that they are cogenerated by vacuum. By Theorem \ref{thm vac cogen_general}, it is enough to check the conditions (i) and (ii) in \textit{loc.\,cit.}. Condition (i) is obvious for both $\cF_\loc$ and $\cV_{\loc}|_{t=\lambda_1\hbar_1+\lambda_2\hbar_2}$. For condition (ii), we claim that specialization maps
\begin{align*}
    \mathsf{sp}_{\mathbf 1}\colon H^{\sT_0\times \bC^*_t}(\fM(\mathbf 1),\sw)\big|_{t=a_1\hbar_1+a_2\hbar_2}\to H^{\sT_0\times \bC^*_t}(\fM(\mathbf 1))\big|_{t=a_1\hbar_1+a_2\hbar_2}\:,\quad \mathsf{sp}_{\underline{\mathbf 0}^{\mathbf 1}}\colon H^{\sT_0}(\fM(\underline{\mathbf 0}^{\mathbf 1}),\sw)\to H^{\sT_0}(\fM(\underline{\mathbf 0}^{\mathbf 1})),
\end{align*}
are injective, for \textit{all} $(a_1,a_2)\in \bC^2$. In fact, since $\fM(\mathbf 1)$ and $\fM(\underline{\mathbf 0}^{\mathbf 1})$ are vector bundles over the stack of unframed representations $\fM(\mathbf 0)$ and the potential $\sw$ is the pullback one, we have natural isomorphisms
\begin{equation*}
\begin{gathered}
H^{\sT_0\times \bC^*_t}(\fM(\mathbf 1),\sw)\cong H^{\sT_0}(\fM(\mathbf 0),\sw)\otimes\bC[t],\quad H^{\sT_0\times \bC^*_t}(\fM(\mathbf 1))\cong H^{\sT_0}(\fM(\mathbf 0))\otimes\bC[t],\\
H^{\sT_0}(\fM(\underline{\mathbf 0}^{\mathbf 1}),\sw)\cong H^{\sT_0}(\fM(\mathbf 0),\sw),\quad  H^{\sT_0}(\fM(\underline{\mathbf 0}^{\mathbf 1}))\cong H^{\sT_0}(\fM(\mathbf 0)).
\end{gathered}
\end{equation*}
We note that $\mathsf{sp}_{\mathbf 1}$ and $\mathsf{sp}_{\underline{\mathbf 0}^{\mathbf 1}}$ are obtained by the specialization map 
$$\mathsf{sp}_{\mathbf 0}\colon H^{\sT_0}(\fM(\mathbf 0),\sw)\to H^{\sT_0}(\fM(\mathbf 0))$$ on the base stack  via the aforementioned isomorphisms. 
By \cite[Rem.~7.16]{COZZ}, we have 
$$\mathsf{sp}_{\mathbf 0}=j_*\circ \delta^{-1}, $$ 
where $\delta\colon H^{\sT_0}(Z(\mu))\cong H^{\sT_0}(\fM(\mathbf 0),\sw)$ is the dimensional reduction isomorphism ($\mu(X,Y,Z)=[X,Y]$ is the moment map morphism), $j\colon Z(\mu)\hookrightarrow \fM(\mathbf 0)$ is the natural embedding. The pushforward map $j_*$ is injective by \cite[Thm.~10.2]{Dav2}, so $\mathsf{sp}_{\mathbf 0}$, $\mathsf{sp}_{\mathbf 1}$, and $\mathsf{sp}_{\underline{\mathbf 0}^{\mathbf 1}}$ are injective. This proves (b). %\bigskip

(c) The specialization $\mathsf{sp}\colon \cH_{N,\loc}\to \cV_\loc|_{t=N\hbar_3}$ is a $\mathsf Y_{0}(Q,\sW)$-module map because it commutes with stable envelope and therefore it commutes with $R$-matrices. Obviously, $\mathsf{sp}(|\mathbf 0\rangle)=|\mathbf 0\rangle$ and, in particular, $\mathsf{sp}$ is nonzero. The injectivity of $\mathsf{sp}$ follows from the irreducibility of $\cH_{N,\loc}$.
\end{proof}

\begin{Lemma}\label{lem inj sp for type A} 
Let $\Gamma$ be a quiver of $A_r$ type, and $i$ be an integer such that $1\leqslant i\leqslant r$. Consider the following framed representation of the double quiver 
$\overline{\Gamma}$:
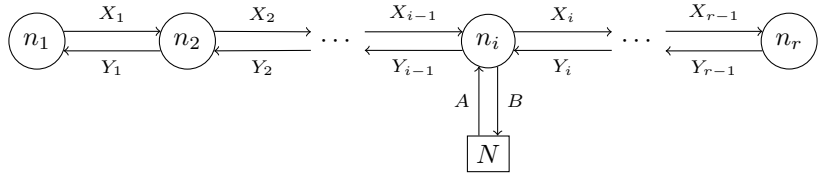
\begin{equation*}
\begin{tikzpicture}[x={(1cm,0cm)}, y={(0cm,1cm)}, baseline=0cm]
  % Nodes
  \node[draw,circle,fill=white] (Gauge1) at (0,0) {$n_1$};
  % \node[draw,rectangle,fill=white] (Framing1) at (0,-1.5) {$1$};
  
  \node[draw,circle,fill=white] (Gauge2) at (2,0) {$n_2$};
  % \node[draw,rectangle,fill=white] (Framing2) at (2,-1.5) {$1$};

  \node (dot1) at (4,0) {$\cdots$};
  % \node (Framingdot1) at (4,-1.5) {$\cdots$};
  \node[draw,circle,fill=white] (GaugeN) at (6,0) {$n_i$};
  \node[draw,rectangle,fill=white] (FramingN) at (6,-1.5) {$N$};
  \node (dot2) at (8,0) {$\cdots$};
  \node[draw,circle,fill=white] (GaugeN+1) at (10,0) {$n_r$};

  % \node (Z) at (0,1) {\scriptsize $\Phi$};
  % Edges
  % \draw[<-] (Gauge1.270) -- (Framing1.90) node[midway,left] {\scriptsize $A_1$};
  % \draw[<-] (Gauge2.270) -- (Framing2.90) node[midway,left] {\scriptsize $A_2$};
  \draw[<-] (GaugeN.250) -- (FramingN.117) node[midway,left] {\scriptsize $A$};
  \draw[->] (GaugeN.290) -- (FramingN.63) node[midway,right] {\scriptsize $B$};

  \draw[->] (Gauge1.20) -- (Gauge2.160) node[midway,above] {\scriptsize $X_1$};
  \draw[<-] (Gauge1.340) -- (Gauge2.200) node[midway,below] {\scriptsize $Y_1$};
  \draw[->] (Gauge2.20) -- (dot1.161) node[midway,above] {\scriptsize $X_2$};
  \draw[<-] (Gauge2.340) -- (dot1.199) node[midway,below] {\scriptsize $Y_2$};
  \draw[->] (dot1.19) -- (GaugeN.160) node[midway,above] {\scriptsize $X_{i-1}$};
  \draw[<-] (dot1.341) -- (GaugeN.200) node[midway,below] {\scriptsize $Y_{i-1}$};
  \draw[->] (GaugeN.20) -- (dot2.161) node[midway,above] {\scriptsize $X_i$};
  \draw[<-] (GaugeN.340) -- (dot2.199) node[midway,below] {\scriptsize $Y_i$};
  \draw[->] (dot2.20) -- (GaugeN+1.160) node[midway,above] {\scriptsize $X_{r-1}$};
  \draw[<-] (dot2.340) -- (GaugeN+1.200) node[midway,below] {\scriptsize $Y_{r-1}$};

  % Loop
  % \draw[->,looseness=7] (Gauge.135) to[out=135,in=45] (Gauge.45);
\end{tikzpicture}
\end{equation*}
with $\sT_0$ acted on by 
\begin{align*}
    (t_1,t_2,t_3)\in \sT_0\:\colon \: (X_\bullet,Y_{\bullet},B,A)\mapsto (t_1^{-1}X_\bullet,t_2^{-1}Y_\bullet,t_3\cdot\sigma(t_3)B,A\sigma(t_3)^{-1}),
\end{align*}
where $\sigma(t_3)=\diag(1,t_3,\cdots,t_3^{N-1})\in \GL_N$ in a given basis $\bC^N=\Span_\bC\{e_1,\ldots,e_N\}$. The Nakajima variety of the above quiver is denoted $\cN(N\delta_i)=\bigsqcup_{\mathbf n=(n_1,\ldots,n_r)}\cN(\mathbf n,N\delta_i)$. Define a potential $\phi$ on $\cN(N\delta_i)$ by $\phi(X_\bullet,Y_{\bullet},B,A)=\tr(\Xi BA)$ where $\Xi$ is the nilpotent matrix such that $\Xi(e_k)=e_{k+1}$ in the aforementioned basis. Then the specialization map $\mathsf{sp}\colon H^{\sT_0}(\cN(N\delta_i),\phi)\to H^{\sT_0}(\cN(N\delta_i))$ is injective after $\sT_0$-localization.
\end{Lemma}

\begin{proof}
It is enough to show that $H^{\sT_0}(\cN(N\delta_i),\phi)_\loc$ is an irreducible $\mathsf Y_{\Gamma}^{\text{MO}}$-module. By \cite[Cor.\,6.2.3]{McB}, the operators $\{e_i(u),f_i(u),h_i(u)\}_{i\in Q_0}$ in the Gauss decomposition \eqref{gauss decomp of R} generate a subalgebra of $\mathsf Y_{\Gamma}^{\text{MO}}$ that is isomorphic to the Yangian algebra $Y_{\hbar}(\mathfrak{sl}_{r+1})$ defined in \cite{V}. Here parameters are identified by $2\hbar=\hbar_3$. By the construction of the critical convolution algebra action in \cite[\S 3]{VV2}, the $Y_{\hbar}(\mathfrak{sl}_{r+1})$-module structure on $H^{\sT_0}(\cN(N\delta_i),\phi)_\loc$ agrees with the one in \cite[Thm.~4.3(e)]{VV2}, which is irreducible (a Kirillov-Reshetikhin module) by \cite[Thm.~4.5]{VV2}. This implies the irreducibility of $H^{\sT_0}(\cN(N\delta_i),\phi)_\loc$ as a $\mathsf Y_{\Gamma}^{\text{MO}}$-module.
\end{proof}

\begin{Remark}
Lemma \ref{lem inj sp for type A} still holds if we replace $\Gamma$ by an arbitrary Dynkin quiver and let $i\in \Gamma_0$ be any node.
\end{Remark}

\begin{Remark}
$\cV(n)_\loc$ (and also $\cF(n)_\loc$) has a basis indexed by 3d Young diagrams of $n$ boxes. A \textit{3d Young diagram} of $n$ boxes is a set $\mathcal D$ consisting of $n$ triples $\{(i,j,k)\in \bZ_{\geqslant 0}^3\}$ such that $(i,j,k)\in \mathcal D$ implies $(i^-,j,k),(i,j^-,k),(i,j,k^-)\in \mathcal D$, where $x^-=\max\{0,x-1\}$. The image of the specialization map 
$$\mathsf{sp}\colon \cH_N(n)_\loc\to \cV(n)_\loc|_{t=N\hbar_3}$$ (which is injective by Theorem \ref{prop on 3 modules of gl1hat}) is spanned by those 3d Young diagrams $\mathcal D=\{(i,j,k)\in \bZ_{>0}^3\}$ such that $k<N$ for all $(i,j,k)\in \mathcal D$.
\end{Remark}

\begin{Remark}
We expect that the $K$-theoretic analog of the discussions in this section will provide a geometric incarnation of the constructions of MacMahon modules, resonance locus, and submodules for the quantum toroidal algebra $U_{q_1,q_2,q_3}(\widehat{\mathfrak{gl}}_1)$ in \cite{FJMM}.
\end{Remark}

\subsection{More examples: spiked instantons and related geometries}

Let us consider the following framings with potentials.

\begin{enumerate}

\item[(4)] (Free-field realization) We fix 
$$\underline{\bd}=\mathbf L+\mathbf M+\mathbf N=(L+M+N,L+M+N), $$ the potential is $\sW_{\mathrm{cub}}$ given below, and we require that $\sT_0$ scales $B_i$ with weight $q_i$ and fixes $A_i$, for all $i=1,2,3$. Let $\sA^{\mathrm{fr}}=(\bC^*)^{L+M+N}$ that acts on the framing vector space $\bC^{L+M+N}$ diagonally,
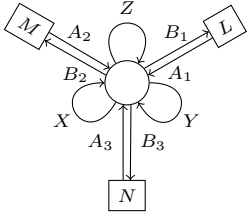
\begin{equation*}
\begin{tikzpicture}[x={(1cm,0cm)}, y={(0cm,1cm)}, baseline=0cm]
  % Nodes
  \node[draw,circle,fill=white] (Gauge) at (0,0){$\phantom{n}$}; %{$\bv$}  {$\phantom{n}$};
  \node[draw,rectangle,fill=white] (FramingZ) at (0,-1.5) {\scriptsize $N$};
  \node[draw,rectangle,fill=white,rotate=30] (FramingX) at (1.3,0.75) {\scriptsize $L$};
  \node[draw,rectangle,fill=white,rotate=-30] (FramingY) at (-1.3,0.75) {\scriptsize $M$};
  \node (Z) at (0,1) {\scriptsize $Z$};
  \node (X) at (-0.86,-0.5) {\scriptsize $X$};
  \node (Y) at (0.86,-0.5) {\scriptsize $Y$};
  % Edges
 % \draw[<-] (Gauge.0) -- (Framing.180) node[midway,below] {\scriptsize $\varphi$};

  % Loop
  \draw[->,looseness=7] (Gauge.120) to[out=120,in=60] (Gauge.60);
  \draw[->,looseness=7] (Gauge.240) to[out=240,in=180] (Gauge.180);
  \draw[->,looseness=7] (Gauge.0) to[out=0,in=300] (Gauge.300);
  % Edges
  \draw[<-] (Gauge.260) -- (FramingZ.104) node[midway,left] {\scriptsize $A_3$};
  \draw[->] (Gauge.280) -- (FramingZ.78) node[midway,right] {\scriptsize $B_3$};

  \draw[<-] (Gauge.20) -- (FramingX.193) node[midway,below] {\scriptsize $A_1$};
  \draw[->] (Gauge.40) -- (FramingX.168) node[midway,above] {\scriptsize $B_1$};

  \draw[<-] (Gauge.140) -- (FramingY.13) node[midway,above] {\scriptsize $A_2$};
  \draw[->] (Gauge.160) -- (FramingY.-12) node[midway,below] {\scriptsize $B_2$};
\end{tikzpicture}
\quad \sW_{\mathrm{cub}}=Z[X,Y]+XA_1B_1+YA_2B_2+ZA_3B_3, \,\, \sw_{\mathrm{cub}}=\tr(\sW_{\mathrm{cub}}).
\end{equation*}
%\begin{equation*} \xymatrix{\text{\tiny{\boxed{1}}}  \ar@/^0.3pc/[r]^{A}  & \ar@/^0.3pc/[l]^{B} \text{\textcircled{v}}  \ar@(dr,ur)_{\Phi}  }
%\end{equation*}
We define  
\begin{equation*}
   \cS_{L,M,N}(n):=H^{\sT_0\times \sA^{\mathrm{fr}}}(\cM(n,\mathbf L+\mathbf M+\mathbf N),\sw_{\mathrm{cub}}), \quad  \cS_{L,M,N}:=H^{\sT_0\times \sA^{\mathrm{fr}}}(\cM(\mathbf L+\mathbf M+\mathbf N),\sw_{\mathrm{cub}})=\bigoplus_{n\in \bZ_{\geqslant 0}} \cS_{L,M,N}(n).
\end{equation*}

\item[(5)] (A degenerate module) $\underline{\bd}=\mathbf 3=(3,3)$, the potential is $\sW_{L,M,N}$ given below, and we require that $\sT_0$ scales $(B_1,B_2,B_3)$ with weights $(q_1^L,q_2^M,q_3^N)$ and fixes $A_i$,  for all $i=1,2,3$. Let $\sA^{\mathrm{fr}}=(\bC^*)^{3}$ that acts on the framing vector space $\bC^{3}$ diagonally,
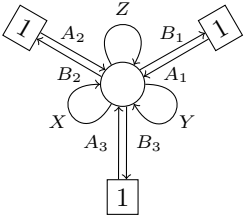
\begin{equation*}
\begin{tikzpicture}[x={(1cm,0cm)}, y={(0cm,1cm)}, baseline=0cm]
  % Nodes
  \node[draw,circle,fill=white] (Gauge) at (0,0){$\phantom{n}$}; %{$\bv$}  {$\phantom{n}$};
  \node[draw,rectangle,fill=white] (FramingZ) at (0,-1.5) {$1$};
  \node[draw,rectangle,fill=white,rotate=30] (FramingX) at (1.3,0.75) {$1$};
  \node[draw,rectangle,fill=white,rotate=-30] (FramingY) at (-1.3,0.75) {$1$};
  \node (Z) at (0,1) {\scriptsize $Z$};
  \node (X) at (-0.86,-0.5) {\scriptsize $X$};
  \node (Y) at (0.86,-0.5) {\scriptsize $Y$};
  % Edges
 % \draw[<-] (Gauge.0) -- (Framing.180) node[midway,below] {\scriptsize $\varphi$};

  % Loop
  \draw[->,looseness=7] (Gauge.120) to[out=120,in=60] (Gauge.60);
  \draw[->,looseness=7] (Gauge.240) to[out=240,in=180] (Gauge.180);
  \draw[->,looseness=7] (Gauge.0) to[out=0,in=300] (Gauge.300);
  % Edges
  \draw[<-] (Gauge.260) -- (FramingZ.103) node[midway,left] {\scriptsize $A_3$};
  \draw[->] (Gauge.280) -- (FramingZ.78) node[midway,right] {\scriptsize $B_3$};

  \draw[<-] (Gauge.20) -- (FramingX.193) node[midway,below] {\scriptsize $A_1$};
  \draw[->] (Gauge.40) -- (FramingX.168) node[midway,above] {\scriptsize $B_1$};

  \draw[<-] (Gauge.140) -- (FramingY.13) node[midway,above] {\scriptsize $A_2$};
  \draw[->] (Gauge.160) -- (FramingY.-12) node[midway,below] {\scriptsize $B_2$};
\end{tikzpicture}
\quad \sW_{L,M,N}=Z[X,Y]+B_1X^LA_1+B_2Y^MA_2+B_3Z^NA_3, \,\, \sw_{L,M,N}=\tr(\sW_{L,M,N}).
\end{equation*}
We define 
\begin{align*}
    \cH_{L,M,N}(n):=H^{\sT_0\times \sA^{\mathrm{fr}}}(\cM(n,\mathbf 3),\sw_{L,M,N}), \quad \cH_{L,M,N}:=H^{\sT_0\times \sA^{\mathrm{fr}}}(\cM(\mathbf 3),\sw_{L,M,N})=\bigoplus_{n\in \bZ_{\geqslant 0}}\cH_{L,M,N}(n).
\end{align*}

\end{enumerate}

By dimensional reduction, we have 
$$\cS_{0,0,N}\cong H^{\sT_0\times \sA^{\mathrm{fr}}}(\cN(\mathbf N)), $$ 
where $\cN(\mathbf N)$ is the Nakajima variety of the doubled Jordan quiver with quiver data $(X,Y,B,A)$. The stable envelope $$\mathbf{Stab}\colon \cH_1[a_1]\otimes\cdots\otimes \cH_1[a_N]\to H^{\sT_0\times \sA^{\mathrm{fr}}}(\cN(\mathbf N))$$
is a $\mathsf Y_{Q'}^{\text{MO}}$-module isomorphism after $\sA^{\mathrm{fr}}$-localization. Here $(a_1,\ldots,a_N)$ are equivariant parameters in $H_{\sA^{\mathrm{fr}}}(\pt)=\bC[a_1,\ldots,a_N]$. According to \cite[Prop.~19.2.2]{MO}, the action of $\mathsf Y_{Q'}^{\text{MO}}$ on $\cH_1[a_1]\otimes\cdots\otimes \cH_1[a_N]$ factors through the mode algebra $\mathfrak{U}(\cW(\mathfrak{gl}_N))$ of the W-algebra $\cW(\mathfrak{gl}_N)$. It is also known that the map $\mathsf Y_{Q'}^{\text{MO}}\to \mathfrak{U}(\cW(\mathfrak{gl}_N))$ has a dense image by \cite[Prop.~19.2.3]{MO}. We define a $\bC(\mathsf t_0)$-algebra map 
%$\Psi^{(3)}_N\colon Y_0(\widehat{\mathfrak{gl}}_1)\to \mathfrak{U}(\cW(\mathfrak{gl}_N))_\loc$ 
to be the composition
$$\Psi^{(3)}_N\colon Y_0(\widehat{\mathfrak{gl}}_1)\xrightarrow[\cong]{\varrho_0}\mathsf Y_0(Q,\sW)\cong \mathsf Y_{Q',\loc}^{\text{MO}}\to \mathfrak{U}(\cW(\mathfrak{gl}_N))_\loc, $$
and denote $Y_{0,0,N}$ to be the quotient algebra 
$$Y_{0,0,N}:=Y_0(\widehat{\mathfrak{gl}}_1)/\ker(\Psi^{(3)}_N). $$ 
For a pair $(i,j)$ of distinct elements in $\{1,2,3\}$, let $s_{ij}\in \Aut\bC(\mathsf t_0)$ be the automorphism that swaps $\hbar_i\leftrightarrow\hbar_j$ and fixes the third one. By Remark \ref{rmk triality}, $s_{ij}$ extends to a $\bC$-algebra automorphism $s_{ij}\colon Y_0(\widehat{\mathfrak{gl}}_1)\cong Y_0(\widehat{\mathfrak{gl}}_1)$. For $i\in\{1,2\}$, we define a $\bC(\mathsf t_0)$-skew linear map 
%$Y_0(\widehat{\mathfrak{gl}}_1)\to \mathfrak{U}(\cW(\mathfrak{gl}_N))_\loc$ 
to be the composition 
\begin{align*}
\Psi^{(i)}_N\colon Y_0(\widehat{\mathfrak{gl}}_1)\xrightarrow[\cong]{s_{i3}} Y_0(\widehat{\mathfrak{gl}}_1)\xrightarrow{\Psi^{(i)}_N} \mathfrak{U}(\cW(\mathfrak{gl}_N))_\loc\:,
\end{align*}
and denote the quotient algebras 
$$Y_{N,0,0}:=Y_0(\widehat{\mathfrak{gl}}_1)/\ker(\Psi^{(1)}_N), \quad Y_{0,N,0}:=Y_0(\widehat{\mathfrak{gl}}_1)/\ker(\Psi^{(2)}_N).$$ 
We note that the $Y_0(\widehat{\mathfrak{gl}}_1)$-module structure on $\cS_{N,0,0}$ is obtained from $\cS_{0,0,N}$ by pullback along the $\bC$-algebra automorphism $s_{13}\colon Y_0(\widehat{\mathfrak{gl}}_1)\cong Y_0(\widehat{\mathfrak{gl}}_1)$, so the $Y_0(\widehat{\mathfrak{gl}}_1)$-action on $\cS_{N,0,0}$ factors through $Y_{N,0,0}$. Similarly, the $Y_0(\widehat{\mathfrak{gl}}_1)$-action on $\cS_{0,N,0}$ factors through $Y_{0,N,0}$.

\begin{Definition}
We define the \textit{Y-algebra} $Y_{L,M,N}$ to be the quotient of $Y_0(\widehat{\mathfrak{gl}}_1)$ by the kernel of the compositions
\begin{align}\label{truncation L,M,N}
Y_0(\widehat{\mathfrak{gl}}_1)\xrightarrow{(\Delta\otimes \id)\circ \Delta} Y_0(\widehat{\mathfrak{gl}}_1)\widehat\otimes Y_0(\widehat{\mathfrak{gl}}_1)\widehat\otimes Y_0(\widehat{\mathfrak{gl}}_1)\xrightarrow{\Psi^{(1)}_L\otimes \Psi^{(2)}_M\otimes\Psi^{(3)}_N} Y_{L,0,0}\widehat\otimes Y_{0,M,0}\widehat\otimes Y_{0,0,N}.
\end{align}
\end{Definition}

\begin{Proposition}\label{prop on more ex}
The $Y_0(\widehat{\mathfrak{gl}}_1)$-actions on $\cS_{L,M,N,\loc}$ and on $\cH_{L,M,N,\loc}$ factor through $Y_{L,M,N}$. Moreover, $\cS_{L,M,N,\loc}$ and $\cH_{L,M,N,\loc}$ are irreducible $Y_{L,M,N}$ modules after $\sA^{\mathrm{fr}}$-localizations.
\end{Proposition}

\begin{proof}
The stable envelope 
$$\Stab\colon \cS_{L,0,0,\loc}\otimes \cS_{0,M,0,\loc}\otimes \cS_{0,0,N,\loc}\to \cS_{L,M,N,\loc}$$ 
is a $Y_0(\widehat{\mathfrak{gl}}_1)$-module map and is an isomorphism after $\sA^{\mathrm{fr}}$-localization. This implies that the $Y_0(\widehat{\mathfrak{gl}}_1)$-action on $\cS_{L,M,N,\loc}$ factors through the map \eqref{truncation L,M,N}, so the $Y_0(\widehat{\mathfrak{gl}}_1)$-actions on $\cS_{L,M,N,\loc}$ factors through $Y_{L,M,N}$. 

$\cH_{0,0,N}$ is the same as $\cH_N$ in \S \ref{subsec modules of Ygl1hat}, and $\cH_{N,\loc}$ is isomorphic to a $Y_0(\widehat{\mathfrak{gl}}_1)$-submodule of $\cH_{1,\loc}[(N-1)\hbar_3]\otimes\cdots\otimes \cH_{1,\loc}[\hbar_3]\otimes \cH_{1,\loc}$ by Theorem \ref{prop on 3 modules of gl1hat}. In particular, $Y_0(\widehat{\mathfrak{gl}}_1)$-action on $\cH_{N,\loc}$ factors through $Y_{0,0,N}$. Using the automorphism $s_{13}$ 
(resp.~$s_{23}$), we see that $Y_0(\widehat{\mathfrak{gl}}_1)$-action on $\cH_{N,0,0}$ (resp.~on $\cH_{0,N,0}$) factors through $Y_{0,0,N}$ (resp. $Y_{0,N,0}$). The stable envelope 
$$\Stab\colon \cH_{L,0,0,\loc}\otimes \cH_{0,M,0,\loc}\otimes \cH_{0,0,N,\loc}\to \cH_{L,M,N,\loc}$$ is a $Y_0(\widehat{\mathfrak{gl}}_1)$-module map and is an isomorphism after $\sA^{\mathrm{fr}}$-localization. Therefore, the $Y_0(\widehat{\mathfrak{gl}}_1)$-actions on $\cH_{L,M,N,\loc}$ factors through $Y_{L,M,N}$.

By Proposition \ref{prop on irr mod}, to prove that $\cS_{L,M,N,\loc}$ and $\cH_{L,M,N,\loc}$ are irreducible, it is enough to show that they are compact and cogenerated by vacuum.

$\cS_{L,M,N}$ is compact because 
$$\cM(\mathbf L+\mathbf M+\mathbf N)^{\sA^{\mathrm{fr}}}=\prod_{i=1}^{L+M+N}\cM(\mathbf 1),$$ 
and the critical locus decomposes accordingly 
$$\Crit_{\cM(n,\mathbf L+\mathbf M+\mathbf N)}(\sw_{L,M,N})^{\sA^{\mathrm{fr}}}=\bigsqcup_{\sum n_i=n}\prod_{i=1}^{L+M+N}\Crit_{\cM(n_i,\mathbf 1)}(\sw_1), $$ 
and the latter has compact $\sT_0$-fixed locus as we have seen in the proof of Theorem \ref{prop on 3 modules of gl1hat}(a). A similar argument shows that $\cH_{L,M,N}$ is compact.

In both $\cS_{L,M,N}$ and $\cH_{L,M,N}$ cases, consider the $\bG_m$ action that scales $B_{i}$ with weight $1$ for all $i=1,2,3$ and fixes other quiver data. By \cite[\S 7.2]{COZZ}, this $\bG_m$ action induces specializations
\begin{align*}
    \mathsf{sp}_{\cS}\colon \cS_{L,M,N}\to H^{\sT_0\times \sA^{\mathrm{fr}}}(\cM(\mathbf L+\mathbf M+\mathbf N),\sw),\quad \mathsf{sp}_{\cH}\colon \cH_{L,M,N}\to H^{\sT_0\times \sA^{\mathrm{fr}}}(\cM(\mathbf 3),\sw),
\end{align*}
which are $Y_0(\widehat{\mathfrak{gl}}_1)$-module maps. After $(\sT_0\times \sA^{\mathrm{fr}})$-localization, $\mathsf{sp}_{\cS}$ is a tensor product of $(L+M+N)$ copies of the specialization map $H^{\sT_0}(\cM(\mathbf 1),\sw_1)_\loc\to H^{\sT_0}(\cM(\mathbf 1),\sw)_\loc$, which is injective by Theorem \ref{prop on 3 modules of gl1hat}(c). Therefore, $\mathsf{sp}_{\cS}$ is injective after localization. Similarly, $\mathsf{sp}_{\cH}$ is injective after localization. Then, it is enough to show that $H^{\sT_0\times \sA^{\mathrm{fr}}}(\cM(\mathbf L+\mathbf M+\mathbf N),\sw)$ and $H^{\sT_0\times \sA^{\mathrm{fr}}}(\cM(\mathbf 3),\sw)$ are cogenerated by vacuum. By Theorem \ref{thm vac cogen_general}, it is enough to check the conditions (i) and (ii) in \textit{loc.\,cit.}. Condition (i) is obvious for both $H^{\sT_0\times \sA^{\mathrm{fr}}}(\cM(\mathbf L+\mathbf M+\mathbf N),\sw)$ and $H^{\sT_0\times \sA^{\mathrm{fr}}}(\cM(\mathbf 3),\sw)$. For condition (ii), we claim that specialization maps
\begin{equation*}
\begin{gathered}
\mathsf{sp}_{\mathbf L+\mathbf M+\mathbf N}\colon H^{\sT_0\times \sA^{\mathrm{fr}}}(\fM(\mathbf L+\mathbf M+\mathbf N),\sw)\to H^{\sT_0\times \sA^{\mathrm{fr}}}(\fM(\mathbf L+\mathbf M+\mathbf N)),\\ \mathsf{sp}_{\mathbf 3}\colon H^{\sT_0}(\fM(\mathbf 3),\sw)\to H^{\sT_0}(\fM(\mathbf 3)),
\end{gathered}
\end{equation*}
are injective. In fact, since $\fM(\mathbf 3)$ and $\fM(\mathbf L+\mathbf M+\mathbf N)$ are vector bundles over the stack of unframed representations $\fM(\mathbf 0)$ and the potential $\sw$ is the pullback one, we have natural isomorphisms
\begin{equation*}
\begin{gathered}
H^{\sT_0\times \sA^{\mathrm{fr}}}(\fM(\mathbf L+\mathbf M+\mathbf N),\sw)\cong H^{\sT_0}(\fM(\mathbf 0),\sw)\otimes H_{\sA^{\mathrm{fr}}}(\pt), \quad H^{\sT_0\times \sA^{\mathrm{fr}}}(\fM(\mathbf L+\mathbf M+\mathbf N))\cong H^{\sT_0}(\fM(\mathbf 0))\otimes H_{\sA^{\mathrm{fr}}}(\pt),\\
H^{\sT_0\times \sA^{\mathrm{fr}}}(\fM(\mathbf 3),\sw)\cong H^{\sT_0}(\fM(\mathbf 0),\sw)\otimes H_{\sA^{\mathrm{fr}}}(\pt),\quad H^{\sT_0\times \sA^{\mathrm{fr}}}(\fM(\mathbf 3))\cong H^{\sT_0}(\fM(\mathbf 0))\otimes H_{\sA^{\mathrm{fr}}}(\pt).
\end{gathered}
\end{equation*}
Moreover, $\mathsf{sp}_{\mathbf L+\mathbf M+\mathbf N}$ and $\mathsf{sp}_{\mathbf 3}$ are obtained by the specialization map on the base stack $\mathsf{sp}_{\mathbf 0}\colon H^{\sT_0}(\fM(\mathbf 0),\sw)\to H^{\sT_0}(\fM(\mathbf 0))$ via the aforementioned isomorphisms. As we have seen in the proof of Theorem \ref{prop on 3 modules of gl1hat}(c), $\mathsf{sp}_{\mathbf 0}$ is injective. This implies that $\mathsf{sp}_{\mathbf L+\mathbf M+\mathbf N}$ and $\mathsf{sp}_{\mathbf 3}$ are injective; hence $\cS_{L,M,N}$ and $\cH_{L,M,N}$ are cogenerated by vacuum.
\end{proof}

\begin{Remark}
Since the $\cW_N$ vertex algebra is a vertex subalgebra of tensor product of $N$ copies of Heisenberg algebras, we have
\begin{align*}
    Y_{0,0,N}\subset \widehat{\mathsf H}^{\widehat{\otimes} N}_{-\hbar_3/\sigma_3},
\end{align*}
where $\widehat{\mathsf H}_{\lambda}$ is a completion of the Heisenberg algebra $\mathsf H_{\lambda}$ with level $\lambda$, which is a $\bC(\mathsf t_0)$-algebra generated by $\{b_n\}_{n\in \bZ}$ with relations $[b_n,b_m]=\lambda n\delta_{n,-m}$. Then it follows from the definition of $Y_{L,M,N}$ that
\begin{align*}
    Y_{L,M,N}\subset \widehat{\mathsf H}^{\widehat{\otimes} L}_{-\hbar_1/\sigma_3}\widehat{\otimes} \widehat{\mathsf H}^{\widehat{\otimes} M}_{-\hbar_2/\sigma_3}\widehat{\otimes}\widehat{\mathsf H}^{\widehat{\otimes} N}_{-\hbar_3/\sigma_3}.
\end{align*}
The computations in \cite[\S 10]{RSYZ1} show that the image of $Y_{L,M,N}$ in $\widehat{\mathsf H}^{\widehat{\otimes} L}_{-\hbar_1/\sigma_3}\widehat{\otimes} \widehat{\mathsf H}^{\widehat{\otimes} M}_{-\hbar_2/\sigma_3}\widehat{\otimes}\widehat{\mathsf H}^{\widehat{\otimes} N}_{-\hbar_3/\sigma_3}$ is contained in a subalgebra $\mathfrak{U}(\cW_{L,M,N})$ defined by the intersection of kernels of $L+M+N-1$ screening operators \cite[\S 9]{RSYZ1}. $\cW_{L,M,N}$ is the vertex algebra which is the annihilator of the aforementioned screening operators. It is conjectured by Proch\'azka and Rap\v{c}\'ak \cite{PR} that $\cW_{L,M,N}$ is isomorphic to Gaiotto-Rap\v{c}\'ak's corner vertex algebra $\mathcal Y_{L,M,N}$ in \cite{GR}.
\end{Remark}

\subsection{Quantum multiplication by divisors on $\Hilb^n(\C^3)$}\label{sect on qm by div on hilb}

Consider quiver with potential $(Q,\sW)$ given by \eqref{equ on trip jord qui}, and quiver varieties 
$$X:=\cM(n,(1,0)), \quad Y:=\cM(n,(1,1)).  $$
We trivially extend $\sw$ to a potential on $X$ and $Y$. 
%where $$W_{\mathrm{sym}}=W_{\Hilb(\C^3)}\oplus \Hom(\C^n,\C), \quad W_{\Hilb(\C^3)}:=\Hom(\C^n,\C^n)^3\oplus \Hom(\C,\C^n). $$ 
For any $d\in \Hom(\mathrm{char}(\mathrm{GL}_n),\mathbb{Z})\cong \mathbb{Z}$, take 
%Consider the action of $\sT:=\C^*_t\times \C^*_q\times \C^*_{\hbar}$, which 
%scales loops $x,\xi,y$ in \eqref{equ on Qsym} by weight $t$, $q^{-1}$, $t^{-1}q$ respectively, and acts on $a$ trivially and on $b$ by $\hbar$.   
$$[Z]:=M_{d}(\gamma)\in H^{\sT_0}(X^2,\sw\boxminus \sw) $$
to be \eqref{equ on Mgamma}.
There is a projection 
$$p\colon Y\to X$$ and we denote $\hbar$ to be the equivariant parameter of its fiber $\C^*$-action. 
Let $$\sT:=\sT_0\times \C^*_{e^\hbar}.$$
There is an isomorphism 
$$p^*\colon H^{\sT_0}(X,\sw)[\hbar]\cong H^{\sT}(X,\sw) \stackrel{\cong}{\to} H^{\sT}(Y,p^*\sw). $$

\begin{Corollary}\label{cor on qm of sym of hilb}
For an $\sT$-equivariant line bundle $L$ on $Y$, we have 
$$c_1(L)\, \widetilde{\star}\, \cdot =c_1(L)\cup \cdot+\sigma_3\deg(L)\sum_{i\in \mathbb{Z}_{>0}}\frac{iz^i}{1-z^i}\,J_{-i}\, J_{i} \,+\mathrm{const}, $$
where $\sigma_3=\hbar_1\hbar_2\hbar_3$, $J_i=\frac{\sigma_3^{-i}}{(i-1)!} \text{ad}^{i-1} _{-f_1} f_0$, $J_{-i}=-\frac{\sigma_3^{-i}}{(i-1)!} \text{ad}^{i-1} _{e_1} e_0$ for $i>0$,
and $e_i$, $f_i$ are part of the generators of Yangian $Y_{0}(\widehat{\mathfrak{gl}_1})$ (Definition \ref{def of shifted affine Y(gl_1)}), which act on the critical cohomology via Theorem \ref{thm ex tripled Jordan_main}. 
\end{Corollary}
\begin{proof}
We denote $Y_n:=\cM(n,(1,1))$. By \cite[Ex.~7.19]{COZZ}, there is a specialization map 
$$H^{\sT}(Y_n,p^*\sw)\to H^{\sT}(Y_n). $$
Taking direct sum over $n$, we obtain $\mathsf Y_{0}(Q,\sW)$-module map 
$$\bigoplus_{n\geqslant 0} H^{\sT}(Y_n,p^*\sw)_{\loc}\to \bigoplus_{n\geqslant 0}  H^{\sT}(Y_n)_{\loc}, $$
such that the vacuum (i.e.~$H^{\sT}(Y_0,p^*\sw)_{\loc}$) maps to vacuum (i.e.~$H^{\sT}(Y_0)_{\loc}$).
By Theorem \ref{prop on 3 modules of gl1hat}\.(b), the LHS is an irreducible $\mathsf Y_{0}(Q,\sW)$-module. 
So the specialization maps are injective.
By applying Corollary \ref{cor qm div for sym_sp inj} to the symmetric quiver variety $Y$ with a function $\sT$-invariant function $\sw\colon Y\to \bC$, and comparing \eqref{explicit r matrix} and \eqref{equ on r matrix gl1hat}, we are done.
\end{proof}

%\yl{1. prove two pt fun and dim red, 2. add cubic potential, by dim red, can fix the scalar operator, $\hbar z(1-z)^{-1}n$. by taking limit 
%$z\cdot n \cdot \deg(L)$}

We deduce the following formula of quantum multiplication by divisors for $\Hilb^n(\C^3)$ using Theorem \ref{thm qm div for asym_sp inj}. 
%Proposition \ref{prop on cmp corr}.

\begin{Corollary}\label{cor on qmd hilb3}
For the above $(X,\sw)$ and any $\sT_0$-equivariant divisor class $L$ on $X$, we have 
\begin{equation}\label{equ on MdL of hilb}
c_1(L)\, \widetilde{\star}_{d}\, \cdot =\sigma_3\deg(L)\cdot d\cdot J_{-d}\, J_{d}+\mathrm{const}\,, \quad \forall\,\,d>0.
\end{equation}
Here $\sigma_3=\hbar_1\hbar_2\hbar_3$, $J_i=\frac{\sigma_3^{-i}}{(i-1)!} \text{ad}^{i-1} _{-f_1} f_0$, $J_{-i}=-\frac{\sigma_3^{-i}}{(i-1)!} \text{ad}^{i-1} _{e_1} e_0$ for $i>0$,
and $e_i$, $f_i$ are part of the generators of shifted Yangian $Y_{-1}(\widehat{\mathfrak{gl}_1})$ \eqref{def of shifted affine Y(gl_1)}, which act on the critical cohomology via Theorem \ref{thm ex tripled Jordan_main}. 
\end{Corollary}

\begin{proof}
By Theorem \ref{thm qm div for asym_sp inj}, it is enough to show $\mathrm{Cas}_{i,-1}=\sigma_3 J_{-i}\, J_{i}$ for all $i>0$. We denote $e(u), f(u)$, $e(u)', f(u)'$ to be the elements in the Yangian $Y_{0}(\widehat{\mathfrak{gl}_1})$, $Y_{-1}(\widehat{\mathfrak{gl}_1})$ respectively. 

From the proof of Lemma \ref{lem shift map}, we have equations of operators on critical cohomology: 
$$(p^{*})^{-1}e(u)\,p^{*}=H_{\hbar}\,e(u)'\,H_{\hbar}^{-1}=\frac{\hbar-c_1(\mathcal L)}{u-c_1(\mathcal L)}\,[\overline{\mathfrak{P}}(n+1,n,\underline{\bd})]
=(\hbar-u)e(u)'+e_0', $$
$$(p^{*})^{-1}f(u)\,p^{*}=-f(u)'.$$
By expanding in the $u$-power, we obtain 
$$(p^{*})^{-1}e_k\,p^{*}=\hbar e'_k-e_{k+1}', \quad (p^{*})^{-1}f_k\,p^{*}=-f_k'. $$ 
This implies that 
\begin{equation}\label{equ on jnn}
(p^{*})^{-1}\mathrm{Cas}_i\,p^{*}=\sigma_3(p^{*})^{-1}J_{-i}J_i\,p^{*}=\sigma_3(-1)^i\hbar^i J_{-i}'J_i'+\mathrm{l.o.t.},
\end{equation}
where $J_i'$ differs from $J_i$ by replacing $e_k, f_k$ by $e_k', f_k'$ respectively, and $\mathrm{l.o.t.}$ denotes the lower order terms of $\hbar$. Then $\mathrm{Cas}_{i,-1}=\sigma_3J_{-i}\, J_{i}$ by Defintion \ref{def of shifted Casimir}.
\end{proof}

\begin{Remark}
In this case, we have $\widetilde{\star}_d=(-1)^d \star_d$.
\end{Remark}

\begin{Remark}
By degree counting, we have 
$$c_1(L)\, \widetilde{\star}_{d}\,\cdot\colon H^{\sT_0}_*(X,\sw) \to H^{\sT_0}_{*+2-2d}(X,\sw), $$ 
so the constant term in \eqref{equ on MdL of hilb} vanishes when $d>1$.
\end{Remark}

\appendix

\section{Positive part of several Yangians}

In this appendix, we consider three classes of examples of (shifted, affine) Yangians, 
whose positive parts are described using explicit generators and relations.  
We construct surjective algebra homomorphisms from them to the 
spherical nilpotent CoHA of corresponding symmetric quivers with potentials (Propositions \ref{prop Y^+(sym KM) to COHA}, \ref{prop Y^+(aff sl(n|m)) to COHA}, \ref{prop Y^+(D(2,1;a)) to COHA}). These are used in the constructions of surjective algebra homomorphisms from 
the entire (shifted, affine) Yangians to the Drinfeld double Yangians (Examples \ref{ex on gen symi}, \ref{ex on vv ex}, \ref{ex on dd ex}).

\subsection{Positive part of the Yangian of a symmetrizable Kac-Moody algebra}\label{sec Y(sym KM)} 

\begin{Definition}
Let the pair $(A,D)$ be a \textit{generalized Cartan matrix with symmetrizer}, that is,
\begin{itemize}
    \item $A=(a_{ij})_{1\leqslant i,j\leqslant n}\in \mathrm{Mat}_{n\times n}(\bZ)$ such that:~$a_{ii}=2$ for all $1\leqslant i\leqslant n$; $a_{ij}\leqslant 0$ for all $1\leqslant i\neq j\leqslant n$; $a_{ij}\neq 0\Leftrightarrow a_{ji}\neq 0$.
    \item $D=\diag(d_1,\ldots,d_n)$ is a diagonal matrix with positive integer entries such that $DA$ is symmetric.
\end{itemize}
\end{Definition}
From the pair $(A,D)$, following \cite{GLS}, we construct a quiver $\widetilde{Q}$ with potential $\sW$. We start with a quiver $Q$ with $Q_0=\{1,\ldots,n\}$, and for any $1\leqslant i<j\leqslant n$, the number of arrows from $i$ to $j$ is 
\begin{align*}
    n_{ij}=|\gcd(a_{ij},a_{ji})|\:.
\end{align*}
Let $\widetilde{Q}$ be the \textit{tripled quiver} of $Q$, then the set of edges in $\widetilde{Q}$ decomposes as $\widetilde{Q}_1=Q_1\sqcup Q_1^*\sqcup \mathcal E$ where $\mathcal E\cong Q_0$ is the set of added edge loops. Let $\{X_a\}_{a\in Q_1}$, $\{X_{a^*}\}_{a^*\in Q^*_1}$, $\{\Phi_i\}_{i\in Q_0}$ be the edges in $Q_1$, $Q_1^*$, and $\mathcal E$ respectively. 

When $i<j$, we define: 
\begin{align*}
    l_{ij}=\bigg|\frac{a_{ij}}{n_{ij}}\bigg|\:, \quad l_{ji}=\bigg|\frac{a_{ji}}{n_{ij}}\bigg|\:,
    \end{align*}
and the \textit{potential}\footnote{Our convention of potentials is the same as \cite[Eqn.~(1)]{YZ} which differs from \cite{GLS} by the transpose of $A$, see \cite[Rmk.~4.2]{YZ}.}:
\begin{align*}
    \sW:=\sum_{(a:i\to j)\in Q_1}\Phi_i^{l_{ij}}X_{a^*}X_a-\Phi_j^{l_{ji}}X_aX_{a^*}, \quad \sw:=\tr(\sW).
\end{align*}
Note that the Jacobian algebra of $\sW$ is the generalized preprojective algebra of \cite{GLS}. 
In the finite dimensional simple Lie algebra case, the potential is also constructed by \cite{CD}.

Let us label the set of arrows in $Q_1$ from $i$ to $j$ by $\{a_1,\ldots,a_{n_{ij}}\}$. Take $\sT_0=\bC^*_\hbar$ which scales $X_{a_p}$ with weight $(n_{ij}/2+1-p)d_il_{ij}$, scales $X_{a^*_p}$ with weight $(p-n_{ij}/2)d_il_{ij}$, and scales $\Phi_i$ with weight $-d_i$. Then $\sW$ is $\bC^*_\hbar$-invariant, 
as $DA$ is symmetric. 
\begin{Definition}
The \textit{Kac-Moody Lie algebra} associated to $(A,D)$, denoted $\mathfrak{g}_{A,D}$, is generated by $\{e_i,f_i\:|\:i\in Q_0\}$ with relations
\begin{equation*}
\begin{gathered}
\ad_{e_i}^{1-a_{ij}}(e_j)=0=\ad_{f_i}^{1-a_{ij}}(f_j),\quad [e_i,f_j]=0,\;i\neq j,\\
[[e_i,f_i],e_j]=a_{ij}e_j,\quad [[e_i,f_i],f_j]=-a_{ij}f_j,\;\forall \,\, i,j\in Q_0.
\end{gathered}
\end{equation*}
The \textit{positive part} $Y^+(\mathfrak{g}_{A,D})$ of the Yangian of $\mathfrak{g}_{A,D}$ \cite{GTL,YZ3} is generated by $\{x^+_{i,r}\}_{i\in Q_0}^{r\in \bZ_{\geqslant 0}}$, subject to relations
\begin{equation}\label{relation for Y^+(sym KM)}
\begin{gathered}
(u-v)[x^+_i(u),x^+_j(v)]-d_ia_{ij}\hbar/2\{x^+_i(u),x^+_j(v)\}=[x^+_{i,0},x^+_j(v)]-[x^+_i(u),x^+_{j,0}],\;\forall \, i,j\in Q_0,\\
[x^+_i(u),x^+_j(v)]=0,\;\text{if }a_{ij}=0,\\
\sum_{\sigma\in \mathfrak{S}_{1-a_{ij}}} [x^+_i(u_{\sigma(1)}),[x^+_i(u_{\sigma(2)}),[\cdots,[x^+_i(u_{\sigma(1-a_{kl})}),x^+_j(v)]]]]=0,\; \text{for } i\neq j.
\end{gathered}
\end{equation}
Here $x^+_i(z):=\sum_{r\geqslant 0}x^+_{i,r}z^{-r-1}$.
\end{Definition}
We establish a surjective algebra map from the positive part of Yangian to the spherical \textit{nilpotent} CoHA of $(\widetilde{Q},\sW)$ as defined in \S \ref{sec shifted yangian_sym quiver}. A related result when the target is the spherical shuffle algebra was proven in \cite[Thm.~5.1]{YZ3}.
\begin{Proposition}\label{prop Y^+(sym KM) to COHA}
There is a surjective map of $\bC(\hbar)$-algebras:
$$Y^+(\mathfrak{g}_{A,D})\otimes_{\bC[\hbar]}\bC(\hbar)\twoheadrightarrow \widetilde{\mathcal{SH}}^{\mathrm{nil}}_{\widetilde{Q},\sW}\otimes_{\bC[\hbar]}\bC(\hbar),\quad x^+_i(z)\mapsto e_i(z).$$
\end{Proposition}

\begin{proof}
It is enough to show that $x^+_i(z)\mapsto e_i(z)$ gives a $\bC(\hbar)$-algebra homomorphism, then the surjectivity follows from the definition that $\{e_i(z)\}_{i\in Q_0}$ generates $\widetilde{\mathcal{SH}}^{\mathrm{nil}}_{\widetilde{Q},\sW}$. Let $\widetilde{\mathcal{SH}}_{\widetilde{Q},0}$ be the spherical CoHA with zero potential, which is isomorphic to the spherical shuffle algebra. By \cite[Thm.~5.1]{YZ3}, there is a $\bC[\hbar]$-algebra homomorphism  
$$Y^+(\mathfrak{g}_{A,D})\to \widetilde{\mathcal{SH}}_{\widetilde{Q},0}, \quad x^+_i(z)\mapsto \tilde{e}_i(z), $$ 
where $\tilde{e}_i(z)=\sum_{r\geqslant 0}\tilde{e}_{i,r}z^{-r-1}$ is the generating function of the CoHA elements
$$\tilde{e}_{i,r}:=c_1(\mathcal L_{i})^{r}\cap\left[\fM(\delta_i,\mathbf 0)\right], \quad \forall \, r\in \bZ_{\geqslant 0},$$
and $\mathcal L_{i}$ is the universal line bundle on $\fM(\delta_i,\mathbf 0)\cong [\bC/\bC^*]$. The specialization map 
$$\mathsf{sp}\colon \widetilde{\mathcal{H}}_{\widetilde{Q},\sW}\to \widetilde{\mathcal{H}}_{\widetilde{Q},0}$$ 
is a $\bC[\hbar]$-algebra homomorphism. In Lemma \ref{lem sp inj for Y(sym KM)} below, we show that $\mathsf{sp}$ is injective after $\bC(\hbar)$-localization. Let us regard $\tilde{e}_i(z)$ as generators of the spherical CoHA with potential $\sW$, denoted $\widetilde{\mathcal{SH}}_{\widetilde{Q},\sW}$, then $x^+_i(z)\mapsto \tilde{e}_i(z)$ gives a $\bC(\hbar)$-algebra homomorphism  
$$Y^+(\mathfrak{g}_{A,D})\otimes_{\bC[\hbar]}\bC(\hbar)\to \widetilde{\mathcal{SH}}_{\widetilde{Q},\sW}\otimes_{\bC[\hbar]}\bC(\hbar).$$ 
The pushforward along the closed embedding 
$$\iota\colon \widetilde{\mathcal{SH}}^{\mathrm{nil}}_{\widetilde{Q},\sW}\to \widetilde{\mathcal{SH}}_{\widetilde{Q},\sW}$$ maps $e_i(z)$ to $d_i\hbar\: \tilde{e}_i(z)$. Note that $x^+_i(z)\mapsto d_i\hbar\:x^+_i(z)$ gives a $\bC(\hbar)$-algebra automorphism 
$$Y^+(\mathfrak{g}_{A,D})\otimes_{\bC[\hbar]}\bC(\hbar)\cong Y^+(\mathfrak{g}_{A,D})\otimes_{\bC[\hbar]}\bC(\hbar).$$ 
Then it is enough to check that $\iota$ is injective after $\bC^*_\hbar$-localization for those degrees that appear in the relations \eqref{relation for Y^+(sym KM)}. This is checked in Lemma \ref{lem inj nil} below.
\end{proof}

\begin{Lemma}\label{lem sp inj for Y(sym KM)}
The $\bC^*_\hbar$-localized specialization map
\begin{align*}
\mathsf{sp}\colon H^{\bC^*_\hbar}(\fM(\bv,\mathbf 0),\sw)\otimes_{\bC[\hbar]}\bC(\hbar)\to H^{\bC^*_\hbar}(\fM(\bv,\mathbf 0))\otimes_{\bC[\hbar]}\bC(\hbar)
\end{align*}
is injective.
\end{Lemma}

\begin{proof}
Consider an auxiliary quiver $Q^\sharp$ which is obtained from $\widetilde Q$ by removing all the arrows in $Q_1^*$. Then $\fM(\widetilde{Q},\bv,\mathbf 0)$ is a vector bundle $E$ over $\fM(Q^\sharp,\bv,\mathbf 0)$ with natural projection $\pi$, and $\sw$ can be written as 
\begin{align*}
    \sw=\langle e,s\rangle,\text{ $e$ is the fiber coordinate of $E$, }\quad s=(X_a\Phi_i^{l_{ij}}-\Phi_j^{l_{ji}}X_a)_{(a:i\to j)\in Q_1}\in \Gamma(\fM(Q^\sharp,\bv,\mathbf 0),E^\vee).
\end{align*}
Then we have dimensional reduction isomorphism 
$$H^{\bC^*_\hbar}\left(\fM(\widetilde{Q},\bv,\mathbf 0),\sw\right)\cong H^{\bC^*_\hbar}(\mathfrak{Z}(s)), $$ 
where $\mathfrak{Z}(s)$ is the zero locus of $s$ in $\fM(Q^\sharp,\bv,\mathbf 0)$. By \cite[Rem.\,\,7.16]{COZZ}, we have 
$$\mathsf{sp}=\tilde j_*\circ \delta^{-1}, $$ 
where $\delta\colon H^{\bC^*_\hbar}(\pi^{-1}\mathfrak{Z}(s))\cong H^{\bC^*_\hbar}(\fM(\widetilde{Q},\bv,\mathbf 0),\sw)$ is the dimensional reduction isomorphism, and $\tilde j\colon \pi^{-1}\mathfrak{Z}(s)\hookrightarrow \fM(\widetilde{Q},\bv,\mathbf 0)$ is the closed embedding. Then, it is equivalent to show that $\tilde j_*$ is injective after $\bC^*_\hbar$-localization, or by smooth pullback isomorphism, equivalent to show that $j_*$ is injective after $\bC^*_\hbar$-localization, where $j\colon \mathfrak{Z}(s)\hookrightarrow \fM(Q^\sharp,\bv,\mathbf 0)$ is the closed embedding.

Consider Jordan-H\"older map 
$$\mathsf{JH}\colon \fM\left(Q^\sharp,\bv,\mathbf 0\right)\to \cM_0\left(Q^\sharp,\bv,\mathbf 0\right)=\mathrm{Rep}_{Q^\sharp}(\bv,\mathbf 0)/\!\!/\GL(\bv). $$ 
Since all the cycles in $Q^\sharp$ are compositions of edge loops, we have 
$$\cM_0\left(Q^\sharp,\bv,\mathbf 0\right)\cong \mathfrak{gl}(\bv)/\!\!/\GL(\bv). $$ 
Then we have $\cM_0\left(Q^\sharp,\bv,\mathbf 0\right)^{\bC^*_\hbar}=\{0\}$. Set $\fM(Q^\sharp,\bv,\mathbf 0)^{\mathrm{nil}}:=\mathsf{JH}^{-1}(\{0\})$, then the pushforward map 
$$H^{\bC^*_\hbar}\left(\fM(Q^\sharp,\bv,\mathbf 0)^{\mathrm{nil}}\right)\to H^{\bC^*_\hbar}\left(\fM(Q^\sharp,\bv,\mathbf 0)\right)$$ is obtained by
\begin{align*}
H^{\bC^*_\hbar}\left(\fM(Q^\sharp,\bv,\mathbf 0)^{\mathrm{nil}}\right)&\cong  H^{\bC^*_\hbar}\left(\cM_0(Q^\sharp,\bv,\mathbf 0), i_*i^!\mathsf{JH}_*\omega_{\fM(Q^\sharp,\bv,\mathbf 0)}\right) \\
&\to H^{\bC^*_\hbar}\left(\cM_0(Q^\sharp,\bv,\mathbf 0\right), \mathsf{JH}_*\omega_{\fM(Q^\sharp,\bv,\mathbf 0)})\cong  H^{\bC^*_\hbar}\left(\fM(Q^\sharp,\bv,\mathbf 0)\right),
\end{align*}
where $i\colon \{0\}\hookrightarrow \cM_0(Q^\sharp,\bv,\mathbf 0)$ is the natural map, which becomes an isomorphism after $\bC^*_\hbar$-localization. 

Similarly, if we set $\mathfrak{Z}(s)^{\mathrm{nil}}=\mathfrak{Z}(s)\cap \mathsf{JH}^{-1}(\{0\})$, then the pushforward map $H^{\bC^*_\hbar}(\mathfrak{Z}(s)^{\mathrm{nil}})\to H^{\bC^*_\hbar}(\mathfrak{Z}(s))$ becomes an isomorphism after $\bC^*_\hbar$-localization. Therfeore, it suffices to show that $j^{\mathrm{nil}}_*$ is injective after $\bC^*_\hbar$-localization, where $j^{\mathrm{nil}}\colon \mathfrak{Z}(s)^{\mathrm{nil}}\hookrightarrow \fM(Q^\sharp,\bv,\mathbf 0)^{\mathrm{nil}}$ is the closed embedding.

We have 
$$H^{\bC^*_\hbar}(\mathfrak{Z}(s)^{\mathrm{nil}})\cong H^{\bC^*_\hbar\times \GL(\bv)}(Z(s)^{\mathrm{nil}}), \quad 
H^{\bC^*_\hbar}(\fM(Q^\sharp,\bv,\mathbf 0)^{\mathrm{nil}})\cong H^{\bC^*_\hbar\times \GL(\bv)}(\mathrm{Rep}_{Q^\sharp}(\bv,\mathbf 0)^{\mathrm{nil}}), $$ where $Z(s)^{\mathrm{nil}}$ and $\mathrm{Rep}_{Q^\sharp}(\bv,\mathbf 0)^{\mathrm{nil}}$ are the preimages of $\mathfrak{Z}(s)^{\mathrm{nil}}$ and $\fM(Q^\sharp,\bv,\mathbf 0)^{\mathrm{nil}}$ in $\mathrm{Rep}_{Q^\sharp}(\bv,\mathbf 0)$ respectively. 

We claim that
\begin{itemize}
    \item[(a)] $H^{\bC^*_\hbar\times \GL(\bv)}(Z(s)^{\mathrm{nil}})$ and $ H^{\bC^*_\hbar\times \GL(\bv)}(\mathrm{Rep}_{Q^\sharp}(\bv,\mathbf 0)^{\mathrm{nil}})$ are free modules of $H_{\bC^*_\hbar\times \GL(\bv)}(\pt)$, and
    \item[(b)] the natural map $H^{\bC^*_\hbar\times \GL(\bv)}(Z(s)^{\mathrm{nil}})\to H^{\bC^*_\hbar\times \GL(\bv)}(\mathrm{Rep}_{Q^\sharp}(\bv,\mathbf 0)^{\mathrm{nil}})$ is injective.
\end{itemize}
The Claim (b) implies that $j^{\mathrm{nil}}_*$ is injective without $\bC^*_\hbar$-localization, and this will imply the lemma.

Notice that $\mathrm{Rep}_{Q^\sharp}(\bv,\mathbf 0)^{\mathrm{nil}}=\mathrm{Rep}_{Q}(\bv,\mathbf 0)\times \cN_{\mathfrak{gl}(\bv)}$, where $\cN_{\mathfrak{gl}(\bv)}$ is the nilpotent cone of $\mathfrak{gl}(\bv)$. Note that $\cN_{\mathfrak{gl}(\bv)}$ is stratified by finitely many $\GL(\bv)$-orbits
\begin{align*}
\cN_{\mathfrak{gl}(\bv)}=\bigsqcup_{\pmb{\lambda}=(\lambda_i)_{i\in Q_0}}\cN_{\pmb{\lambda}},\quad \cN_{\pmb{\lambda}}=\prod_{i\in Q_0}\cN_{\lambda_i},
\end{align*}
where $\lambda_i$ is a partition of $\bv_i$ and $\cN_{\lambda_i}$ is the set of $\bv_i\times \bv_i$ nilpotent matrices of type $\lambda_i$. 
Denote the natural projection 
$$\pr\colon \mathrm{Rep}_{Q^\sharp}(\bv,\mathbf 0)^{\mathrm{nil}}\to \cN_{\mathfrak{gl}(\bv)}. $$
To prove the Claim (a), using the excision sequence, it is enough to prove that 
$$H^{\bC^*_\hbar\times \GL(\bv)}\left(Z(s)^{\mathrm{nil}}\cap \pr^{-1}\cN_{\pmb{\lambda}}\right), \quad 
H^{\bC^*_\hbar\times \GL(\bv)}\left(\pr^{-1}\cN_{\pmb{\lambda}}\right)$$ 
are free modules of $H_{\bC^*_\hbar\times \GL(\bv)}(\pt)$ and concentrated in even degree for all $\pmb{\lambda}$. Note that 
$$\pr^{-1}\cN_{\pmb{\lambda}}\cong \mathrm{Rep}_{Q}(\bv,\mathbf 0)\times \cN_{\pmb{\lambda}}$$ is an equivariant vector bundle on $\cN_{\pmb{\lambda}}$, and $\cN_{\pmb{\lambda}}\cong \GL(\bv)/C_{\GL(\bv)}(x)$, where $C_{\GL(\bv)}(x)$ is the centralizer of a point $x\in \cN_{\pmb{\lambda}}$. Then we have 
$$H^{\bC^*_\hbar\times \GL(\bv)}\left(\pr^{-1}\cN_{\pmb{\lambda}}\right)\cong H_{\bC^*_\hbar\times C_{\GL(\bv)}(x)}(\pt), $$ 
and the latter is a free module of $H_{\bC^*_\hbar\times \GL(\bv)}(\pt)$ and concentrated in even degree. Let $x=(\Phi_i)_{i\in Q_0}\in \cN_{\pmb{\lambda}}$, then $\pr^{-1}(x)\cap Z(s)^{\mathrm{nil}}$ is the subspace of $\mathrm{Rep}_{Q}(\bv,\mathbf 0)$ consisting of those $(X_a)_{a\in Q_1}$ which solve the equations 
$$X_a\Phi_i^{l_{ij}}-\Phi_j^{l_{ji}}X_a=0, \quad \forall \,\,(a:i\to j)\in Q_1. $$ 
Since the equation is linear in $X$, $\pr^{-1}(x)\cap Z(s)^{\mathrm{nil}}$ is a linear subspace of $\mathrm{Rep}_{Q}(\bv,\mathbf 0)$. Transporting along the $\GL(\bv)$-action, we see that $\pr^{-1}(x)\cap Z(s)^{\mathrm{nil}}$ is a vector subbundle of $\mathrm{Rep}_{Q}(\bv,\mathbf 0)\times \cN_{\pmb{\lambda}}\to \cN_{\pmb{\lambda}}$. In particular, we have 
$$H^{\bC^*_\hbar\times \GL(\bv)}\left(Z(s)^{\mathrm{nil}}\cap \pr^{-1}\cN_{\pmb{\lambda}}\right)\cong H^{\bC^*_\hbar\times \GL(\bv)}(\cN_{\pmb{\lambda}})\cong H_{\bC^*_\hbar\times C_{\GL(\bv)}(x)}(\pt), $$ 
which is a free module of $H_{\bC^*_\hbar\times \GL(\bv)}(\pt)$ and concentrated in even degree. This proves the Claim (a). 

Since $\mathrm{Rep}_{Q^\sharp}(\bv,\mathbf 0)^{\bC^*_\hbar}\subset Z(s)^{\mathrm{nil}}$, the $I$-localized pushforward map 
$$H^{\bC^*_\hbar\times \GL(\bv)}\left(Z(s)^{\mathrm{nil}}\right)_I\to  H^{\bC^*_\hbar\times \GL(\bv)}\left(\mathrm{Rep}_{Q^\sharp}(\bv,\mathbf 0)^{\mathrm{nil}}\right)_I$$ is an isomorphism, where $I\subset H_{\bC^*_\hbar\times \GL(\bv)}(\pt)$ is the prime ideal defined as the kernel of the projection $ H_{\bC^*_\hbar\times \GL(\bv)}(\pt)\twoheadrightarrow  H_{\bC^*_\hbar}(\pt)$. Then the Claim (b) follows from the injectivity after $I$-localization and the freeness (Claim (a)).
% $\pr_{\pmb{\lambda}}\colon \mathrm{Rep}_{Q}(\bv,\mathbf 0)\times \cN_{\pmb{\lambda}}\to \cN_{\pmb{\lambda}}$
\end{proof}

\begin{Lemma}\label{lem inj nil}
The pushforward along the closed embedding $$\iota\colon H^{\bC^*_\hbar}(\fM(\bv,\mathbf 0),\sw)_{\fM(\bv,\mathbf 0)^{\mathrm{nil}}}\otimes_{\bC[\hbar]}\bC(\hbar)\to H^{\bC^*_\hbar}(\fM(\bv,\mathbf 0),\sw)\otimes_{\bC[\hbar]}\bC(\hbar)$$ is injective when $\bv=n\delta_i,n\in \bZ_{\geqslant 1}$ or $\bv=m\delta_i+\delta_j,i\neq j,m\in \bZ_{\geqslant 1}$.
\end{Lemma}

\begin{proof}
Consider the Jordan-H\"older map $\mathsf{JH}\colon \fM(\bv,\mathbf 0)\to \cM_0(\bv,\mathbf 0)$ where the target is the affine quotient 
$$\cM_0(\bv,\mathbf 0)=\mathrm{Rep}_{\widetilde{Q}}(\bv,\mathbf 0)/\!\!/\GL(\bv). $$ 
Then $\iota$ is induced by applying the hypercohomology $H^{\bC^*_\hbar}(\cM_0(\bv,\mathbf 0),-)$ to the natural map between complexes:
\begin{align*}
    i_*i^!\mathsf{JH}_*\varphi_{\sw}\omega_{\fM(\bv,\mathbf 0)}\to \mathsf{JH}_*\varphi_{\sw}\omega_{\fM(\bv,\mathbf 0)},
\end{align*}
where $i\colon \{0\}\hookrightarrow \cM_0(\bv,\mathbf 0)$ is the closed embedding given by the trivial quiver representation. Note that $\mathsf{JH}_*\varphi_{\sw}\omega_{\fM(\bv,\mathbf 0)}$ is supported on 
$$\mathsf{JH}(\Crit_{\fM(\bv,\mathbf 0)}(\sw))\cong \Crit_{\mathrm{Rep}_{\widetilde{Q}}(\bv,\mathbf 0)}(\sw)/\!\!/\GL(\bv). $$ 
Then it suffices to show that the $\bC^*_\hbar$-fixed point of $\Crit_{\mathrm{Rep}_{\widetilde{Q}}(\bv,\mathbf 0)}(\sw)/\!\!/\GL(\bv)$ is $\{0\}$ in the following two cases:
\begin{enumerate}
    \item[(i)] $\bv=n\delta_i$, $n\in \bZ_{\geqslant 1}$,
    \item[(ii)] $\bv=m\delta_i+\delta_j$, $i\neq j$, $m\in \bZ_{\geqslant 1}$.
\end{enumerate}
In the case (i), $\sw$ is zero on $\mathrm{Rep}_{\widetilde{Q}}(\bv,\mathbf 0)$, and we have 
$$\cM_0(\bv,\mathbf 0)\cong \Spec\bC[\tr(\Phi_i),\tr(\Phi_i^2),\ldots,\tr(\Phi_i^n)]. $$ 
Since $\bC^*_\hbar$ scales $\Phi_i$ with nontrivial weight, we have $\cM_0(\bv,\mathbf 0)^{\bC^*_\hbar}=\{0\}$. 

In the case (ii), without loss of generality, we can assume $i<j$, and the $i>j$ case can be proven in the same way. Consider the natural map $$\Crit_{\mathrm{Rep}_{\widetilde{Q}}(\bv,\mathbf 0)}(\sw)/\!\!/\GL(\bv)\to \mathfrak{gl}(\bv)/\!\!/\GL(\bv)$$
which is given by sending $(\Phi_k,X_a,X_{a^*})$ to $(\Phi_k)_{k\in Q_0}\in \mathfrak{gl}(\bv)$. The $\bC^*_\hbar$ fixed point of $\mathfrak{gl}(\bv)/\!\!/\GL(\bv)$ is $\{0\}$, so for any representative $(\Phi_k,X_a,X_{a^*})$ of a $\bC^*_\hbar$ fixed point in $\Crit_{\mathrm{Rep}_{\widetilde{Q}}(\bv,\mathbf 0)}(\sw)/\!\!/\GL(\bv)$, $(\Phi_k)_{k\in Q_0}$ must be nilpotent. Nilpotency implies that $\Phi_j=0$ as $\bv_j=1$, and the functions $\tr(\Phi_i^s)$ vanishes on $\Crit_{\mathrm{Rep}_{\widetilde{Q}}(\bv,\mathbf 0)}(\sw)/\!\!/\GL(\bv)$ for all $s>0$. Then, by invariant theory, $\bC\left[\left(\Crit_{\mathrm{Rep}_{\widetilde{Q}}(\bv,\mathbf 0)}(\sw)/\!\!/\GL(\bv)\right)^{\bC^*_\hbar}\right]$ is generated by $\{X_b\Phi_i^rX_{a^*}\}_{(a,b:i\to j)\in Q_1,r\in \bZ_{\geqslant0}}$. On the other hand, $\partial\sw/\partial X_a=0$ gives $\Phi_i^{l_{ij}}X_{a^*}=0$ for all $(a:i\to j)\in Q_1$ (here we have used $\Phi_j=0$). Then, only those $X_b\Phi_i^rX_{a^*}$ with $r<l_{ij}$ can be nonzero. However, $X_b\Phi_i^rX_{a^*}$ has $\bC^*_\hbar$ weight $d_i(r-l_{ij})$ which is nonzero if $r\neq l_{ij}$. Therefore, $$\bC\left[\left(\Crit_{\mathrm{Rep}_{\widetilde{Q}}(\bv,\mathbf 0)}(\sw)/\!\!/\GL(\bv)\right)^{\bC^*_\hbar}\right]=\bC,$$
and we conclude that $\left(\Crit_{\mathrm{Rep}_{\widetilde{Q}}(\bv,\mathbf 0)}(\sw)/\!\!/\GL(\bv)\right)^{\bC^*_\hbar}=\{0\}$.
\end{proof}

\subsection{Positive part of the affine Yangian of \texorpdfstring{$\mathfrak{sl}_{n|m}$}{sl(n|m)}}\label{sec super affine Yangian}

Let $m,n\in \bN$ with $\max\{mn,m,n\}\geqslant3$. A \textit{parity sequence} of type $(m,n)$ is a tuple $\mathbf s=(s_1,s_2,\ldots,s_{m+n})$ of $\pm 1$ 
with $m$ entries equal to $1$. Let $Q_0=\bZ/(m+n)\bZ$. 

Given a parity sequence $\mathbf s$, we define matrices $A=(a_{i,j})_{i,j\in Q_0}$ and $M=(m_{i,j})_{i,j\in Q_0}$ by
\begin{equation}\label{aij and mij}
\begin{gathered}
a_{ij}=(s_i+s_{i+1})\delta_{ij}-s_i\delta_{i,j+1}-s_{j}\delta_{i+1,j},\\
m_{i,i+1}=-m_{i+1,i}=-s_{i+1},\\
m_{i,j}=0\text{ if }|i-j|\neq 1.
\end{gathered}
\end{equation}
Define a \textit{symmetric quiver} $Q=(Q_0,Q_1)$ by setting $Q_1=Q_1^+\sqcup Q_1^-\sqcup \mathcal E$ with
\begin{align*}
Q_1^+=\{X_i:i\to i+1\:|\:i\in Q_0\},\quad Q_1^-=\{Y_i:i+1\to i\:|\:i\in Q_0\},\quad \mathcal E=\{\Phi_i:i\to i\:|\:s_i+s_{i+1}\neq 0\}.
\end{align*}
Note that the odd nodes are $Q_0^{\mathrm{f}}=\{i\in Q_0:s_i+s_{i+1}=0\}$. Define the \textit{potential}
\begin{align*}
\sW=\sum_{i\in Q_0}\Phi_i(Y_iX_i-X_{i-1}Y_{i-1})+\sum_{j\in Q_0^{\mathrm{f}}}Y_jX_jX_{j-1}Y_{j-1},
\end{align*}
and denote $\sw=\tr(\sW)$. Set $\sT_0=\bC^*_\hbar\times \bC^*_{\varepsilon}$ and let it act on $Q$ by 
\begin{align*}
(q,t)\in \bC^*_\hbar\times \bC^*_{\varepsilon}:(X_i,Y_i,\Phi_j)\mapsto (q^{s_{i+1}}t^{-s_{i+1}}X_i,q^{s_{i+1}}t^{s_{i+1}}Y_i,q^{-2s_j}\Phi_j),\text{ for }i\in Q_0,j\in Q_0^{\mathrm{f}}.
\end{align*}
Then $\sW$ is $\sT_0$-invariant.

\begin{Definition}[{\cite{U,VV1}}]
The positive part of the affine Yangian of $\mathfrak{sl}_{n|m}$, denoted $Y^+(\widehat{\mathfrak{sl}}_{n|m})$, is the $\bC[\hbar,\varepsilon]$-algebra generated by $\{x^+_{i,r}\}_{i\in Q_0,r\in \bZ_{\geqslant0}}$ subject to relations
\begin{equation}\label{relation for Y^+(aff sl(n|m))}
\begin{gathered}
(u-v+m_{ij}\varepsilon)[x^+_i(u),x^+_j(v)]-a_{ij}\hbar\{x^+_i(u),x^+_j(v)\}=[x^+_{i,0},x^+_j(v)]-[x^+_i(u),x^+_{j,0}],\;\forall \, i,j\in Q_0,\\
[x^+_i(u),x^+_j(v)]=0,\;\text{if }a_{ij}=0,\\
\sum_{\sigma\in \mathfrak{S}_2} [x^+_i(u_{\sigma(1)}),[x^+_i(u_{\sigma(2)}),x^+_j(v)]]=0,\; \text{for } i\in Q_0^{\mathrm{b}}\text{ and }|i-j|=1,\\
\sum_{\sigma\in \mathfrak{S}_2} [x^+_i(u_{\sigma(1)}),[x^+_{i+1}(v),[x^+_i(u_{\sigma(2)}),x^+_{i-1}(w)]]]=0,\; \text{for } i\in Q_0^{\mathrm{f}}\text{ and }i\pm 1\in Q_0^{\mathrm{b}}.
\end{gathered}
\end{equation}
Here $[a,b]=ab-(-1)^{|a|\cdot|b|}ba$, $\{a,b\}=ab+(-1)^{|a|\cdot|b|}ba$, and $|x^+_{i,r}|=|i|$
\end{Definition}

\begin{Proposition}\label{prop Y^+(aff sl(n|m)) to COHA}
There is a surjective $\bC(\hbar,\varepsilon)$-algebra map
\begin{align*}
    Y^+(\widehat{\mathfrak{sl}}_{n|m})\otimes_{\bC[\hbar,\varepsilon]}\bC(\hbar,\varepsilon)\twoheadrightarrow \widetilde{\mathcal{SH}}^{\mathrm{nil}}_{Q,\sW}\otimes_{\bC[\hbar,\varepsilon]}\bC(\hbar,\varepsilon),\quad x^+_{i}\mapsto e_{i}(z).
\end{align*}
\end{Proposition}

\begin{proof}
It is enough to show that $x^+_i(z)\mapsto e_i(z)$ gives a $\bC(\hbar)$-algebra homomorphism, then the surjectivity follows from the definition that $\{e_i(z)\}_{i\in Q_0}$ generates $\widetilde{\mathcal{SH}}^{\mathrm{nil}}_{Q,\sW}$. Let $\widetilde{\mathcal{SH}}_{Q,\sW}$ be the spherical CoHA with potential $\sW$. 

By \cite[Thm.~C.2]{VV1}, there is a $\bC[\hbar,\varepsilon]$-algebra homomorphism  
$$Y^+(\widehat{\mathfrak{sl}}_{n|m})\to \widetilde{\mathcal{SH}}_{Q,\sW}, \quad x^+_i(z)\mapsto \tilde{e}_i(z), $$ 
where $\tilde{e}_i(z)=\sum_{r\geqslant 0}\tilde{e}_{i,r}z^{-r-1}$ is the generating function of the CoHA elements
$$\tilde{e}_{i,r}:=c_1(\mathcal L_{i})^{r}\cap\left[\fM(\delta_i,\mathbf 0)\right], \quad \forall \, r\in \bZ_{\geqslant 0},$$
and $\mathcal L_{i}$ is the universal line bundle on $\fM(\delta_i,\mathbf 0)\cong [\bC^{\mathsf g_i}/\bC^*]$. Here $\mathsf g_i=1$ if $i$ is even and $\mathsf g_i=0$ if $i$ is odd. 

The pushforward along the closed embedding 
$$\iota\colon \widetilde{\mathcal{SH}}^{\mathrm{nil}}_{Q,\sW}\to \widetilde{\mathcal{SH}}_{Q,\sW}$$ maps $e_i(z)$ to $-2s_i\hbar\: \tilde{e}_i(z)$ if $i$ is even and to $\tilde{e}_i(z)$ if $i$ is odd. 
Note that $$x^+_i(z)\mapsto \begin{cases}
-2s_i\hbar x^+_i(z),& i\in Q_0^{\mathrm{b}},\\ 
x^+_i(z),& i\in Q_0^{\mathrm{f}},
\end{cases}$$
gives a $\bC(\hbar,\varepsilon)$-algebra automorphism 
$$Y^+(\widehat{\mathfrak{sl}}_{n|m})\otimes_{\bC[\hbar,\varepsilon]}\bC(\hbar,\varepsilon)\cong Y^+(\widehat{\mathfrak{sl}}_{n|m})\otimes_{\bC[\hbar,\varepsilon]}\bC(\hbar,\varepsilon).$$ 
Then it is enough to check that $\iota$ is injective after $\sT_0$-localization for those degrees that appear in the relations \eqref{relation for Y^+(aff sl(n|m))}. This is checked in Lemma \ref{lem inj nil super} below.
\end{proof}

\begin{Lemma}\label{lem inj nil super}
The pushforward along the closed embedding $$\iota\colon H^{\sT_0}(\fM(\bv,\mathbf 0),\sw)_{\fM(\bv,\mathbf 0)^{\mathrm{nil}}}\otimes_{\bC[\hbar,\varepsilon]}\bC(\hbar,\varepsilon)\to H^{\sT_0}(\fM(\bv,\mathbf 0),\sw)\otimes_{\bC[\hbar,\varepsilon]}\bC(\hbar,\varepsilon)$$ is injective when $\bv=n\delta_i,n\in \bZ_{\geqslant 1}$, or $\bv=m\delta_i+\delta_j,i\neq j,m\in \bZ_{\geqslant 1}$, or $\bv=\delta_{i-1}+2\delta_i+\delta_{i+1}$, $i$ odd and $i\pm 1$ even.
\end{Lemma}

\begin{proof}
Similar to the proof of Lemma \ref{lem inj nil}, we need to show that $\sT_0$ fixed point of $\Crit_{\mathrm{Rep}_{Q}(\bv,\mathbf 0)}(\sw)/\!\!/\GL(\bv)$ is $\{0\}$ in the following three cases:
\begin{enumerate}
    \item[(i)] $\bv=n\delta_i$, $n\in \bZ_{\geqslant 1}$,
    \item[(ii)] $\bv=m\delta_i+\delta_j$, $i\neq j$, $m\in \bZ_{\geqslant 1}$,
    \item[(iii)]  $\bv=\delta_{i-1}+2\delta_i+\delta_{i+1}$, $i$ odd and $i\pm 1$ even.
\end{enumerate}
The first two cases are proven similarly as Lemma \ref{lem inj nil} and we shall not repeat. For case (iii), we have a natural map $$\Crit_{\mathrm{Rep}_{Q}(\bv,\mathbf 0)}(\sw)/\!\!/\GL(\bv)\to \bC^*\times \bC^*,$$
which is given by sending $(X_i,X_{i-1},Y_i,Y_{i-1},\Phi_{i+1},\Phi_{i-1})$ to $(\Phi_{i-1},\Phi_{i+1})\in \bC^*\times \bC^*$. The $\sT_0$ fixed point of $\bC^*\times \bC^*$ is $\{0\}$, so for any representative $(X_i,X_{i-1},Y_i,Y_{i-1},\Phi_{i+1},\Phi_{i-1})$ of a $\sT_0$ fixed point in $\Crit_{\mathrm{Rep}_{\widetilde{Q}}(\bv,\mathbf 0)}(\sw)/\!\!/\GL(\bv)$, we have $\Phi_{i\pm 1}=0$. Then, by invariant theory, $R=\bC\left[\left(\Crit_{\mathrm{Rep}_{Q}(\bv,\mathbf 0)}(\sw)/\!\!/\GL(\bv)\right)^{\bC^*_\hbar}\right]$ is generated by trace of cycles formed by $\{X_i,X_{i-1},Y_i,Y_{i-1}\}$. On the other hand, $\partial\sw/\partial X_i=0$ implies that $X_{i-1}Y_{i-1}Y_i=0$, so the nontrivial generators of $R$ are $\tr(X_iY_i)^m$ and $\tr(X_{i-1}Y_{i-1})^m$ for $m\in \bZ_{\geqslant 1}$. Both of the two types of generators have nonzero $\sT_0$ weights, so $R=\bC$, and we conclude that $\left(\Crit_{\mathrm{Rep}_{Q}(\bv,\mathbf 0)}(\sw)/\!\!/\GL(\bv)\right)^{\sT_0}=\{0\}$.
\end{proof}

\subsection{Positive part of the Yangian of \texorpdfstring{$D(2,1;\lambda)$}{D(2,1;a)}}\label{sec Yangian for D(2,1;a)}

Recall the \textit{exceptional simple Lie superalgebra} $\mathfrak{g=}D(2,1;\lambda)$, which has two presentations $D_\lambda$ and $D^\circ_\lambda$ with nonconjugate root systems. Explicitly, $\mathfrak{g}$ is generated by $\{e_i,f_i\}_{i=1,2,3}$, with parity given by $|e_i|=|f_i|=i+1\pmod 2$ if $\mathfrak{g}=D_\lambda$ and $|e_i|=|f_i|=1$ if $\mathfrak{g}=D^\circ_\lambda$, and subject to the relations 
\begin{align*}
    [e_i,f_j]=0,\:i\neq j, \quad [[e_i,f_i],e_j]=a_{ij}e_j,\quad [[e_i,f_i],f_j]=-a_{ij}f_j,\;\forall \,\, i,j\in Q_0,
\end{align*}
and 
\begin{gather*}
\ad_{e_i}^{1+\delta(a_{ij}\neq 0)}(e_j)=0=\ad_{f_i}^{1+\delta(a_{ij}\neq 0)}(f_j),\; \text{ if }\mathfrak{g}=D_\lambda,\\
a_{13}[[e_1,e_2],e_3]=a_{12}[[e_1,e_3],e_2],\; a_{13}[[f_1,f_2],f_3]=a_{12}[[f_1,f_3],f_2],\; \text{ if }\mathfrak{g}=D^\circ_\lambda.
\end{gather*}
Here $\delta(-)=1$ if $(-)$ holds and $\delta(-)=0$ otherwise, and the Cartan matrix $(a_{ij})$ is given by
\begin{align}\label{Cartan matrix for D(2,1)}
(a_{ij})=
\begin{pmatrix}
2 & -1 & 0\\
1 & 0 & \lambda\\
0 & -1 & 2
\end{pmatrix}\; \text{ if }\;
\mathfrak{g}=D_\lambda, \quad (a_{ij})=
\begin{pmatrix}
0 & -1 & 1+\lambda\\
-1 & 0 & -\lambda\\
1+\lambda & -\lambda & 0
\end{pmatrix}\;\text{ if }\;\mathfrak{g}=D^\circ_\lambda.
\end{align}
Define diagonal matrix $(d_{ij})$ which is given by
\begin{align*}
(d_{ij})=\diag(1,-1,\lambda)\; \text{ if }\;
\mathfrak{g}=D_\lambda, \quad (d_{ij})=\mathrm{id}_{3\times 3}\;\text{ if }\;\mathfrak{g}=D^\circ_\lambda.
\end{align*}
\begin{Definition}\label{def Y^+(D(2,1;a))}
The positive part of the Yangian of $D(2,1;\lambda)$, denoted $Y^+(\mathfrak{g})$ for $\mathfrak{g}=D_\lambda$ or $\mathfrak{g}=D^\circ_\lambda$, is the $\bC[\hbar,\lambda]$-algebra generated by $\{x^+_{i,r}\}_{i\in \{1,2,3\},r\in \bZ_{\geqslant 0}}$ subject to relations
\begin{equation}\label{relation for Y^+(D(2,1;a))}
\begin{gathered}
(u-v)[x^+_i(u),x^+_j(v)]-d_{ii}a_{ij}\hbar/2\{x^+_i(u),x^+_j(v)\}=[x^+_{i,0},x^+_j(v)]-[x^+_i(u),x^+_{j,0}],\;\forall \, i,j,\\
[x^+_i(u),x^+_j(v)]=0,\;\text{if }a_{ij}=0,\\
\sum_{\sigma\in \mathfrak{S}_2} [x^+_i(u_{\sigma(1)}),[x^+_i(u_{\sigma(2)}),x^+_j(v)]]=0,\; \text{if }\mathfrak{g}=D_\lambda,\; a_{ij}\neq 0,\;\text{and}\;i\neq j,\\
a_{13}[[x^+_1(v_1),x^+_2(v_2)],x^+_3(v_3)]=a_{12}[[x^+_1(v_1),x^+_3(v_3)],x^+_2(v_2)],\; \text{if }\mathfrak{g}=D^\circ_\lambda.
\end{gathered}
\end{equation}
Here $x^+_i(z):=\sum_{r\geqslant 0}x^+_{i,r}z^{-r-1}$, $[a,b]=ab-(-1)^{|a|\cdot|b|}ba$, $\{a,b\}=ab+(-1)^{|a|\cdot|b|}ba$, and
\begin{align*}
|x^+_{i,r}|=\begin{cases}
i+1 \pmod{2},&\text{ if }\mathfrak{g}=D_\lambda,\\
1, &\text{ if }\mathfrak{g}=D_\lambda^\circ.
\end{cases}
\end{align*}
\end{Definition} 

\subsubsection{\texorpdfstring{$Y^+(D_\lambda)$}{Y(D)} and CoHA}\label{sec quiver for D(2,1)}

Let $Q=(Q_0,Q_1)$ be the \textit{symmetric quiver} with $Q_0=\{1,2,3\}$ and $Q_1=Q_1^{+}\sqcup Q_1^{-}\sqcup \mathcal E$, where 
\begin{align*}
    Q_1^{+}=\{X_i:i\to i+1\:|\: i=1,2\},\quad Q_1^{-}=\{Y_i:i+1\to i\:|\: i=1,2\},\quad \mathcal E=\{\Phi_j:j\to j\:|\: j=1,3\},
\end{align*}
depicted as follows:
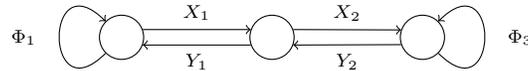
\begin{equation*}
\begin{tikzpicture}[x={(1cm,0cm)}, y={(0cm,1cm)}, baseline=0cm]
  % Nodes
  \node[draw,circle,fill=white] (Gauge1) at (0,0) {$\phantom{n}$};
  % \node[draw,rectangle,fill=white] (Framing1) at (0,-1.5) {$1$};
  
  \node[draw,circle,fill=white] (Gauge2) at (2,0) {$\phantom{n}$};
  % \node[draw,rectangle,fill=white] (Framing2) at (2,-1.5) {$1$};

  \node[draw,circle,fill=white] (Gauge3) at (4,0) {$\phantom{n}$};
  
  \node (Z1) at (-1.3,0) {\scriptsize $\Phi_1$};
  \node (Z3) at (5.3,0) {\scriptsize $\Phi_3$};
  % Edges

  \draw[->] (Gauge1.20) -- (Gauge2.160) node[midway,above] {\scriptsize $X_1$};
  \draw[<-] (Gauge1.340) -- (Gauge2.200) node[midway,below] {\scriptsize $Y_1$};
  \draw[->] (Gauge2.20) -- (Gauge3.161) node[midway,above] {\scriptsize $X_2$};
  \draw[<-] (Gauge2.340) -- (Gauge3.199) node[midway,below] {\scriptsize $Y_2$};
 
  % Loop
  \draw[->,looseness=7] (Gauge1.225) to[out=225,in=135] (Gauge1.135);
  \draw[->,looseness=7] (Gauge3.315) to[out=315,in=45] (Gauge3.45);
\end{tikzpicture}
\end{equation*}
Let $\sT_0=\bC^*_{\hbar_1}\times \bC^*_{\hbar_2}$ act on $Q$ by
\begin{align*}
    (t_1,t_2)\in \bC^*_{\hbar_1}\times \bC^*_{\hbar_2}\colon (X_1,Y_1,X_2,Y_2,\Phi_1,\Phi_3)\mapsto (t_1X_1,t_1Y_1,t_2X_2,t_2Y_2,t_1^{-2}\Phi_1,t_2^{-2}\Phi_3).
\end{align*}
Take the \textit{potential} $\sW$ to be
\begin{align*}
    \sW=\Phi_1Y_1X_1+\Phi_3X_2Y_2,
\end{align*}
and denote $\sw=\tr(\sW)$.
\begin{Proposition}\label{prop Y^+(D(2,1;a)) to COHA}
There is a surjective $\bC$-algebra map
\begin{align*}
    Y^+(D_\lambda)\otimes_{\bC[\hbar,\lambda]}\bC(\hbar,\lambda)\twoheadrightarrow \widetilde{\mathcal{SH}}^{\mathrm{nil}}_{Q,\sW}\otimes_{\bC[\mathsf t_0]}\bC(\mathsf t_0),\quad x^+_i(z)\mapsto e_i(z), \quad \hbar\mapsto 2\hbar_1,\quad \lambda\mapsto \hbar_2/\hbar_1.
\end{align*}
\end{Proposition}

\begin{proof}
We notice that the relations \eqref{relation for Y^+(D(2,1;a))} for $\mathfrak{g}=D_\lambda$ only involve at most two nodes in the quiver $Q$. The relation $[e_1(u), e_3(v)]=0$ holds in $\widetilde{\mathcal{SH}}^{\mathrm{nil}}_{Q,\sW}$ because the node $1$ and $3$ are disconnected. So to verify \eqref{relation for Y^+(D(2,1;a))} on the CoHA side, it is enough to focus on the quiver representations which are supported on two neighboring nodes, that is, the following two subquivers with potentials restricted from $\sW$:
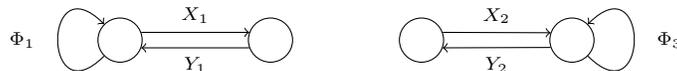
\begin{equation*}
\begin{tikzpicture}[x={(1cm,0cm)}, y={(0cm,1cm)}, baseline=0cm]
  % Nodes
  \node[draw,circle,fill=white] (Gauge1) at (0,0) {$\phantom{n}$};
  % \node[draw,rectangle,fill=white] (Framing1) at (0,-1.5) {$1$};
  
  \node[draw,circle,fill=white] (Gauge2) at (2,0) {$\phantom{n}$};
  % \node[draw,rectangle,fill=white] (Framing2) at (2,-1.5) {$1$};
  
  \node (Z1) at (-1.3,0) {\scriptsize $\Phi_1$};

  % Edges

  \draw[->] (Gauge1.20) -- (Gauge2.160) node[midway,above] {\scriptsize $X_1$};
  \draw[<-] (Gauge1.340) -- (Gauge2.200) node[midway,below] {\scriptsize $Y_1$};

  % Loop
  \draw[->,looseness=7] (Gauge1.225) to[out=225,in=135] (Gauge1.135);
\end{tikzpicture}
\qquad\qquad
\begin{tikzpicture}[x={(1cm,0cm)}, y={(0cm,1cm)}, baseline=0cm]
  % Nodes
  
  \node[draw,circle,fill=white] (Gauge2) at (2,0) {$\phantom{n}$};
  % \node[draw,rectangle,fill=white] (Framing2) at (2,-1.5) {$1$};

  \node[draw,circle,fill=white] (Gauge3) at (4,0) {$\phantom{n}$};
  
  \node (Z3) at (5.3,0) {\scriptsize $\Phi_3$};
  % Edges

  \draw[->] (Gauge2.20) -- (Gauge3.161) node[midway,above] {\scriptsize $X_2$};
  \draw[<-] (Gauge2.340) -- (Gauge3.199) node[midway,below] {\scriptsize $Y_2$};
 
  % Loop
  \draw[->,looseness=7] (Gauge3.315) to[out=315,in=45] (Gauge3.45);
\end{tikzpicture}
\end{equation*}
The above two quivers with potentials are special cases of quivers with potentials in \S\ref{sec super affine Yangian}, and the rest of relations in \eqref{relation for Y^+(D(2,1;a))} follow from Proposition \ref{prop Y^+(aff sl(n|m)) to COHA}.
\end{proof}

\subsubsection{\texorpdfstring{$Y^+(D^\circ_\lambda)$}{Y'(D)} and CoHA}\label{sec quiver for D'(2,1)}

Let $Q^\circ=(Q^\circ_0,Q^\circ_1)$ be the \textit{symmetric quiver} with $Q^\circ_0=\bZ/3\bZ$ and $Q^\circ_1=Q_1^{\circ+}\sqcup Q_1^{\circ-}$, where 
\begin{align*}
    Q_1^{\circ+}=\{X_i:i\to i+1\:|\: i\in Q_0\},\quad Q_1^{\circ-}=\{Y_i:i+1\to i\:|\: i\in Q_0\},
\end{align*}
depicted as follows:
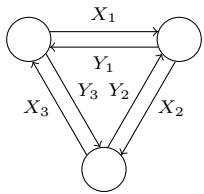
\begin{equation*}
\begin{tikzpicture}[x={(1cm,0cm)}, y={(0cm,1cm)}, baseline=0cm]
  % Nodes
  \node[draw,circle,fill=white] (Gauge1) at (0,0) {$\phantom{n}$};
  % \node[draw,rectangle,fill=white] (Framing1) at (0,-1.5) {$1$};
  
  \node[draw,circle,fill=white] (Gauge2) at (2,0) {$\phantom{n}$};
  % \node[draw,rectangle,fill=white] (Framing2) at (2,-1.5) {$1$};

  \node[draw,circle,fill=white] (Gauge3) at (1,-1.73) {$\phantom{n}$};

  % Edges

  \draw[->] (Gauge1.20) -- (Gauge2.160) node[midway,above] {\scriptsize $X_1$};
  \draw[<-] (Gauge1.340) -- (Gauge2.200) node[midway,below] {\scriptsize $Y_1$};
  \draw[->] (Gauge2.260) -- (Gauge3.40) node[midway,right] {\scriptsize $X_2$};
  \draw[<-] (Gauge2.220) -- (Gauge3.80) node[pos=0.4,left] {\scriptsize $Y_2$};
  \draw[->] (Gauge1.320) -- (Gauge3.100) node[pos=0.4,right] {\scriptsize $Y_3$};
  \draw[<-] (Gauge1.280) -- (Gauge3.140) node[midway,left] {\scriptsize $X_3$};
  % Loop
\end{tikzpicture}
\end{equation*}
Let $\bC^*_{\hbar_1}\times \bC^*_{\hbar_2}\times \bC^*_{\hbar_3}$ act on $Q^\circ$ by
\begin{align*}
    (t_1,t_2,t_3)\in \bC^*_{\hbar_1}\times \bC^*_{\hbar_2}\times \bC^*_{\hbar_3}\colon (X_i,Y_i)\mapsto (t_iX_i,t_iY_i).
\end{align*}
Take the \textit{potential} $\sW^\circ$ to be
\begin{align*}
    \sW^\circ=X_3X_2X_1+Y_1Y_2Y_3,
\end{align*}
which is fixed by the subtorus 
$$\sT_0=\{(t_1,t_2,t_3)\in \bC^*_{\hbar_1}\times \bC^*_{\hbar_2}\times \bC^*_{\hbar_3}\:|\: t_1t_2t_3=1\}, $$
and denote $\sw^\circ=\tr(\sW^\circ)$.
Denote the equivariant base ring to be $$\bC[\mathsf t_0]:=H_{\sT_0}(\pt)\cong \bC[\hbar_1,\hbar_2,\hbar_3]/(\hbar_1+\hbar_2+\hbar_3). $$

\begin{Proposition}\label{prop Y^+(D'(2,1;a)) to COHA}
There is a surjective $\bC$-algebra map
\begin{align*}
    Y^+(D^\circ_\lambda)\otimes_{\bC[\hbar,\lambda]}\bC(\hbar,\lambda)\twoheadrightarrow \widetilde{\mathcal{SH}}^{\mathrm{nil}}_{Q^\circ,\sW^\circ}\otimes_{\bC[\mathsf t_0]}\bC(\mathsf t_0),\quad x^+_i(z)\mapsto e_i(z), \quad \hbar\mapsto 2\hbar_1,\quad \lambda\mapsto \hbar_2/\hbar_1.
\end{align*}
\end{Proposition}

\begin{proof}
Similarly to the proof of Proposition \ref{prop Y^+(D(2,1;a)) to COHA}, all but the last relation in \eqref{relation for Y^+(D(2,1;a))} reduce to the following quiver
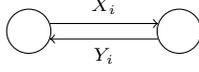
\begin{equation*}
\begin{tikzpicture}[x={(1cm,0cm)}, y={(0cm,1cm)}, baseline=0cm]
  % Nodes
  \node[draw,circle,fill=white] (Gauge1) at (0,0) {$\phantom{n}$};
  % \node[draw,rectangle,fill=white] (Framing1) at (0,-1.5) {$1$};
  
  \node[draw,circle,fill=white] (Gauge2) at (2,0) {$\phantom{n}$};
  % \node[draw,rectangle,fill=white] (Framing2) at (2,-1.5) {$1$};
  
  % \node (Z1) at (-1.3,0) {\scriptsize $\Phi_1$};

  % Edges

  \draw[->] (Gauge1.20) -- (Gauge2.160) node[midway,above] {\scriptsize $X_i$};
  \draw[<-] (Gauge1.340) -- (Gauge2.200) node[midway,below] {\scriptsize $Y_i$};

  % Loop
  % \draw[->,looseness=7] (Gauge1.225) to[out=225,in=135] (Gauge1.135);
\end{tikzpicture}
\end{equation*}
which is a special case of the quivers with potentials in Proposition \ref{prop Y^+(aff sl(n|m)) to COHA}. It remains to show that the equation
\begin{align}\label{serre rel D^circ}
\hbar_3[[e_1(v_1),e_2(v_2)],e_3(v_3)]=\hbar_1[[e_1(v_1),e_3(v_3)],e_2(v_2)]
\end{align}
holds in $\widetilde{\mathcal{SH}}^{\mathrm{nil}}_{Q^\circ,\sW^\circ}\otimes_{\bC[\mathsf t_0]}\bC(\mathsf t_0)$. Let $\bv=(1,1,1)$, then we claim that
\begin{enumerate}
    \item[(i)] the pushforward map $H^{\sT_0}(\fM(\bv,\mathbf 0),\sw^\circ)_{\fM(\bv,\mathbf 0)^{\mathrm{nil}}}\to H^{\sT_0}(\fM(\bv,\mathbf 0),\sw^\circ)$ is injective after localization,
    \item[(ii)] the specialization map $\mathrm{sp}\colon H^{\sT_0}(\fM(\bv,\mathbf 0),\sw^\circ)\to H^{\sT_0}(\fM(\bv,\mathbf 0))$ is injective.
\end{enumerate}
Assuming the claims for now, it is enough to prove that \eqref{serre rel D^circ} holds in $\widetilde{\mathcal{SH}}_{Q^\circ,0}$, the cohomological Hall algebra for $Q^\circ$ with zero potential. \eqref{serre rel D^circ} is equivalent to the following symmetric form
\begin{align}\label{serre rel D^circ_sym}
\sum_{i\in \bZ/3\bZ}\hbar_{i+1}\left(e_{i-1}(v_{i-1})e_{i}(v_{i})e_{i+1}(v_{i+1})-e_{i+1}(v_{i+1})e_{i}(v_{i})e_{i-1}(v_{i-1})\right)=0\,.
\end{align}
For the dimension vector $\bv=(1,1,1)$, each node has one dimension, and let $x_i$ be the equivariant parameter for the node $i$, then $H^{\sT_0}(\fM(\bv,\mathbf 0))\cong \bC[x_1,x_2,x_3]\otimes \bC[\mathsf t_0]$, and we have $e_i(v_i)=\sum_{n\geqslant 0}x_i^n v_i^{-n-1}=\frac{1}{v_i-x_i}$. Note that the sign twist \eqref{twist hall product} for $Q^\circ$ is trivial. Direct computation shows that
\begin{gather*}
e_{i-1}(v_{i-1})e_{i}(v_{i})e_{i+1}(v_{i+1})=\frac{(x_{i+1}-x_{i-1}+\hbar_{i+1})(x_{i}-x_{i-1}+\hbar_{i-1})(x_{i+1}-x_i+\hbar_i)}{(v_{1}-x_{1})(v_{2}-x_{2})(v_{3}-x_{3})}\,,\\
e_{i+1}(v_{i+1})e_{i}(v_{i})e_{i-1}(v_{i-1})=\frac{(x_{i-1}-x_{i+1}+\hbar_{i+1})(x_{i-1}-x_{i}+\hbar_{i-1})(x_i-x_{i+1}+\hbar_i)}{(v_{1}-x_{1})(v_{2}-x_{2})(v_{3}-x_{3})}\,.
\end{gather*}
Then
\begin{align*}
\text{LHS of \eqref{serre rel D^circ_sym}}&=2\sum_{i\in \bZ/3\bZ}\frac{-\hbar_{i+1}(x_1-x_3)(x_2-x_1)(x_3-x_2)+(x_i-x_{i-1})(\hbar_{i+1}^2\hbar_i+\hbar_{i}^2\hbar_{i+1}-\hbar_i\hbar_{i-1}\hbar_{i+1})}{(v_{1}-x_{1})(v_{2}-x_{2})(v_{3}-x_{3})}=0.
\end{align*}
In the remainder, we prove claims (i) and (ii).

The claim (i) is proven similarly to Lemma \ref{lem inj nil}, namely, we need to show that  
\begin{equation}\label{equ on crit rep is zero}\left(\Crit_{\mathrm{Rep}_{Q^\circ}(\bv,\mathbf 0)}(\sw^\circ)/\!\!/\GL(\bv)\right)^{\sT_0}=\{0\}. \end{equation}
For $i\in \bZ/3\bZ$, $\partial\sw^\circ/\partial X_i=0$ gives $X_{i-1}X_{i+1}=0$, and $\partial\sw^\circ/\partial Y_i=0$ gives $Y_{i+1}Y_{i-1}=0$. Then $\GL(\bv)$-invariant functions in  $\bC[\Crit_{\mathrm{Rep}_{Q^\circ}(\bv,\mathbf 0)}(\sw^\circ)]$ are generated by $\{X_iY_i\}_{i\in \bZ/3\bZ}$, and they have relations $X_iY_iX_jY_j=0$ for any pair $i\neq j$. Thus, $\bC[\Crit_{\mathrm{Rep}_{Q^\circ}(\bv,\mathbf 0)}(\sw^\circ)/\!\!/\GL(\bv)]$ is spanned by the monomials $\{(X_iY_i)^n\}_{i\in \bZ/3\bZ}^{n\in \bZ_{\geqslant 0}}$, with the only $\sT_0$-invariant functions being the constant ones. This shows \eqref{equ on crit rep is zero}.

By dimensional reduction, we have 
$$H^{\sT_0}(\fM(\bv,\mathbf 0),\sw^\circ)\cong H^{\sT_0\times \GL(\bv)}(Z), $$ 
where $Z$ is the locus $\{X_1X_2=0,\,Y_1Y_2=0\}$ in the space of representations $R\cong \Spec \bC[X_1,X_2,Y_1,Y_2]$ of quiver
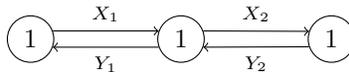
\begin{equation*}
\begin{tikzpicture}[x={(1cm,0cm)}, y={(0cm,1cm)}, baseline=0cm]
  % Nodes
  \node[draw,circle,fill=white] (Gauge1) at (0,0) {$1$};
  % \node[draw,rectangle,fill=white] (Framing1) at (0,-1.5) {$1$};
  
  \node[draw,circle,fill=white] (Gauge2) at (2,0) {$1$};
  % \node[draw,rectangle,fill=white] (Framing2) at (2,-1.5) {$1$};

  \node[draw,circle,fill=white] (Gauge3) at (4,0) {$1$};

  % Edges

  \draw[->] (Gauge1.20) -- (Gauge2.160) node[midway,above] {\scriptsize $X_1$};
  \draw[<-] (Gauge1.340) -- (Gauge2.200) node[midway,below] {\scriptsize $Y_1$};
  \draw[->] (Gauge2.20) -- (Gauge3.161) node[midway,above] {\scriptsize $X_2$};
  \draw[<-] (Gauge2.340) -- (Gauge3.199) node[midway,below] {\scriptsize $Y_2$};
 
  % Loop
  
\end{tikzpicture} 
\end{equation*}
By the commutativity of the specialization map and the canonical map \cite[Rem.~7.16]{COZZ}, claim (ii) is equivalent to the claim that the 
pushforward $H^{\sT_0\times \GL(\bv)}(Z)\to H^{\sT_0\times \GL(\bv)}(R)$ is injective. Note that $\bC[R/\!\!/ \GL(\bv)]$ is generated by $X_1X_2, Y_1Y_2, X_1Y_1, X_2Y_2$ which are all positive with respect to the diagonal subtorus $\bC^*\hookrightarrow \bC^*_{\hbar_1}\times \bC^*_{\hbar_2}\cong \sT_0$, so $(R/\!\!/ \GL(\bv))^{\sT_0}=\{0\}$. Since $\{0\}\in Z/\!\!/ \GL(\bv)$, $H^{\sT_0\times \GL(\bv)}(Z)\to H^{\sT_0\times \GL(\bv)}(R)$ is injective after $\sT_0$-localization. 

We can stratify $Z$ by $Z=Z_0\cup Z_1\cup Z_2\cup Z_3$ with 
\begin{align*}
    Z_0=Z\cap \{X_1\neq 0, Y_2\neq 0\},\; Z_1=Z\cap \{X_1\neq 0, Y_2=0\},\; Z_2=Z\cap \{X_1=0, Y_2\neq 0\},\; Z_4=Z\cap \{X_1=0, Y_2=0\}.
\end{align*}
Write $\GL(\bv)=\bC^*_{x_1}\times \bC^*_{x_2}\times \bC^*_{x_3}$ and let the equivariant parameters of $\GL(\bv)$ be $x_1,x_2,x_3$, then
\begin{gather*}
H^{\sT_0\times \GL(\bv)}(Z_0)\cong\bC[\hbar_1,\hbar_2,x_1,x_2,x_3]/(x_2-x_1+\hbar_1, x_2-x_3+\hbar_2),\\
H^{\sT_0\times \GL(\bv)}(Z_1)\cong\bC[\hbar_1,\hbar_2,x_1,x_2,x_3]/(x_2-x_1+\hbar_1),\\
H^{\sT_0\times \GL(\bv)}(Z_2)\cong\bC[\hbar_1,\hbar_2,x_1,x_2,x_3]/(x_2-x_3+\hbar_2),\\
H^{\sT_0\times \GL(\bv)}(Z_3)\cong\bC[\hbar_1,\hbar_2,x_1,x_2,x_3].
\end{gather*}
All of the cohomologies in the above are free $\bC[\hbar_1,\hbar_2]$-modules and are concentrated in even degrees, and applying excision sequence to the stratification $Z=Z_0\cup Z_1\cup Z_2\cup Z_3$ shows that $H^{\sT_0\times \GL(\bv)}(Z)$ is a free $\bC[\hbar_1,\hbar_2]$-module. Thus $H^{\sT_0\times \GL(\bv)}(Z)\to H^{\sT_0\times \GL(\bv)}(R)$ is injective without localization.
\end{proof}

\section{\texorpdfstring{$R$}{R}-matrices for quivers with potentials associated to \texorpdfstring{$\mathfrak{gl}_{2|1}$}{gl(2|1)}}\label{sect on Lie superalgebra}
It is well-known that $\mathfrak{gl}_{2|1}$ has the following three kinds of Kac-Dynkin diagrams:
\begin{equation*}
\begin{tikzpicture}[cross/.style={path picture={ 
  \draw[black]
(path picture bounding box.south east) -- (path picture bounding box.north west) (path picture bounding box.south west) -- (path picture bounding box.north east);
}}]

 \node [draw,circle,minimum width=0.5cm](A){};
 \node [draw,circle,cross,minimum width=0.5cm](B) at (2,0){}; 
  
\draw[-] (A) -- (B);

\end{tikzpicture}
\qquad\phantom{\sW=X\Phi Y,}
\qquad
\begin{tikzpicture}[cross/.style={path picture={ 
  \draw[black]
(path picture bounding box.south east) -- (path picture bounding box.north west) (path picture bounding box.south west) -- (path picture bounding box.north east);
}}]

 \node [draw,circle,cross,minimum width=0.5cm](A){};
 \node [draw,circle,minimum width=0.5cm](B) at (2,0){}; 
  
\draw[-] (A) -- (B);

\end{tikzpicture}
\qquad\phantom{\sW=X\Phi Y,}
\qquad
\begin{tikzpicture}[cross/.style={path picture={ 
  \draw[black]
(path picture bounding box.south east) -- (path picture bounding box.north west) (path picture bounding box.south west) -- (path picture bounding box.north east);
}}]

 \node [draw,circle,cross,minimum width=0.5cm](A){};
 \node [draw,circle,cross,minimum width=0.5cm](B) at (2,0){}; 
  
\draw[-] (A) -- (B);

\end{tikzpicture}
\phantom{=0}
\end{equation*}
which we will refer to as the even-odd, the odd-even, and the odd-odd cases. We associate the following three quivers with potentials to the above three Kac-Dynkin diagrams:
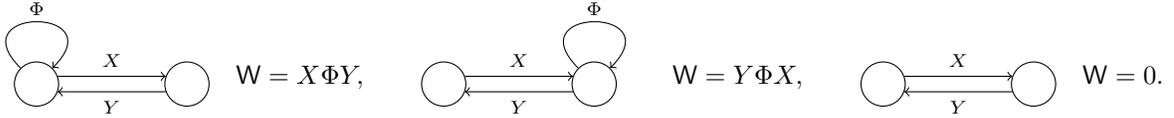
\begin{equation*}
\begin{tikzpicture}[x={(1cm,0cm)}, y={(0cm,1cm)}, baseline=0cm]
  % Nodes
  \node[draw,circle,fill=white] (Gauge1) at (0,0){$\phantom{n}$}; %{$\bv$}  {$\phantom{n}$};
  \node[draw,circle,fill=white] (Gauge2) at (2,0) {$\phantom{n}$};
  \node (Z) at (0,1) {\scriptsize $\Phi$};
  % Edges
  \draw[<-] (Gauge1.340) -- (Gauge2.200) node[midway,below] {\scriptsize $Y$};
  \draw[->] (Gauge1.20) -- (Gauge2.160) node[midway,above] {\scriptsize $X$};

  % Loop
  \draw[->,looseness=7] (Gauge1.135) to[out=135,in=45] (Gauge1.45);
\end{tikzpicture}
\quad\sW=X\Phi Y,
\qquad
\begin{tikzpicture}[x={(1cm,0cm)}, y={(0cm,1cm)}, baseline=0cm]
  % Nodes
  \node[draw,circle,fill=white] (Gauge1) at (0,0){$\phantom{n}$}; %{$\bv$}  {$\phantom{n}$};
  \node[draw,circle,fill=white] (Gauge2) at (2,0) {$\phantom{n}$};
  \node (Z) at (2,1) {\scriptsize $\Phi$};
  % Edges
  \draw[<-] (Gauge1.340) -- (Gauge2.200) node[midway,below] {\scriptsize $Y$};
  \draw[->] (Gauge1.20) -- (Gauge2.160) node[midway,above] {\scriptsize $X$};

  % Loop
  \draw[->,looseness=7] (Gauge2.135) to[out=135,in=45] (Gauge2.45);
\end{tikzpicture}
\sW=Y\Phi X,
\qquad
\begin{tikzpicture}[x={(1cm,0cm)}, y={(0cm,1cm)}, baseline=0cm]
  % Nodes
  \node[draw,circle,fill=white] (Gauge1) at (0,0){$\phantom{n}$}; %{$\bv$}  {$\phantom{n}$};
  \node[draw,circle,fill=white] (Gauge2) at (2,0) {$\phantom{n}$};
  % \node (Z) at (0,1) {\scriptsize $\Phi$};
  % Edges
  \draw[<-] (Gauge1.340) -- (Gauge2.200) node[midway,below] {\scriptsize $Y$};
  \draw[->] (Gauge1.20) -- (Gauge2.160) node[midway,above] {\scriptsize $X$};

  % Loop
  % \draw[->,looseness=7] (Gauge.135) to[out=135,in=45] (Gauge.45);
\end{tikzpicture}
\quad\sW=0.
\end{equation*}
In all three cases, let $\sT_0=\bC^*_\hbar$ act on the arrows $X$, $Y$, and $\Phi$ (if exists) with weights $0$, $1$, and $-1$, respectively. As in the beginning of \S \ref{sec shifted yangian_sym quiver}, we fix a splitting of edges $Q_1=\mathcal A\sqcup \mathcal A^*\sqcup \mathcal E$, where $\mathcal A=\{X\}$, $\mathcal A^*=\{Y\}$, $\mathcal E=\{\Phi\}$ or empty.

In the following, we work with the cyclic stability condition ($\theta<0$) when defining quiver varieties. The framing dimension vector is take to be
$$\bd=(\bd_1,\bd_2)=(1,0). $$ We compute the $R^{\mathrm{sup}}(u)$ matrix for the symmetric quiver variety $\cM(2\bd)$ with torus $\sA=\bC^*_u$ action on the framing as $u\bd+\bd$. Here $R^{\mathrm{sup}}(u)$ is defined in Theorem \ref{thm e f h as R matrix elements}. The framed potential will be specified case-by-case. 

% Plan: prove some of the conjectures in \cite{IMRY} about super Yangian actions for quiver diagrams of following type:

\subsection{The even-odd case}
%Let $\cH^{\mathrm{eo}}_{\bd}$ be the $\sT_0$ equivariant critical cohomology of
Consider the following framed quiver with potential:
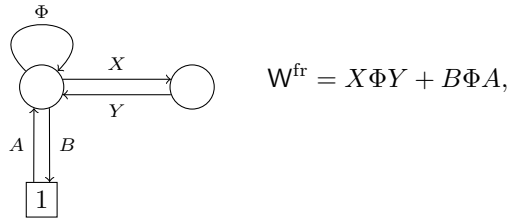
\begin{equation*}
\begin{tikzpicture}[x={(1cm,0cm)}, y={(0cm,1cm)}, baseline=0cm]
  % Nodes
  \node[draw,circle,fill=white] (Gauge1) at (0,0){$\phantom{n}$}; %{$\bv$}  {$\phantom{n}$};
  \node[draw,circle,fill=white] (Gauge2) at (2,0) {$\phantom{n}$};
  \node (Z) at (0,1) {\scriptsize $\Phi$};
  \node[draw,rectangle,fill=white] (FramingZ) at (0,-1.5) {$1$};
  % Edges
  \draw[<-] (Gauge1.340) -- (Gauge2.200) node[midway,below] {\scriptsize $Y$};
  \draw[->] (Gauge1.20) -- (Gauge2.160) node[midway,above] {\scriptsize $X$};
  \draw[<-] (Gauge1.250) -- (FramingZ.114) node[midway,left] {\scriptsize $A$};
  \draw[->] (Gauge1.290) -- (FramingZ.66) node[midway,right] {\scriptsize $B$};
  % Loop
  \draw[->,looseness=7] (Gauge1.135) to[out=135,in=45] (Gauge1.45);
\end{tikzpicture}
\qquad\sW^{\mathrm{fr}}=X\Phi Y+B\Phi A,
\end{equation*}
and the $\sT_0$-equivariant critical cohomology 
$$\cH^{\mathrm{eo}}_{\bd}=\bigoplus_{\bv\in \bZ^{2}_{\geqslant0}}\cH^{\mathrm{eo}}_{\bd}(\bv), \quad \cH^{\mathrm{eo}}_{\bd}(\bv)=H^{\sT_0}(\cM(\bv,\bd),\tr \sW^{\mathrm{fr}}). $$ Here we let $\sT_0=\bC^*_\hbar$ act on $A$ and $B$ with weights $0$ and $1$ respectively. 

The argument in \cite[Lem.~6.3]{Dav2} can be applied to the above quiver variety with potential, and we have an explicit description of the critical locus
\begin{align*}
    \Crit_{\cM(\bv,\bd)}(\tr \sW^{\mathrm{fr}})=\{XY+AB=0,\; \Phi=0\}/\!\!/_\theta \GL(\bv)=: \cM'(\bv,\bd).
\end{align*}
Applying dimensional reduction along $\Phi$, we get an isomorphism
\begin{align}\label{dim red eo}
\cH^{\mathrm{eo}}_{\bd}(\bv)\cong H^{\sT_0}(\cM'(\bv,\bd)).
\end{align}
Note that
\begin{align*}
\cM'(\bv,\bd)\cong \begin{cases}
\pt, & \bv=(0,0),\\
\pt, & \bv=(1,0),\\
\bC, & \bv=(1,1),\\
\, \emptyset, & \text{otherwise}.
\end{cases}
\end{align*}
Let $b_{\bv}$ be the image of the fundamental class $[\cM'(\bv,\bd)]$ in $\cH^{\mathrm{eo}}_{\bd}(\bv)$ via the isomorphism \eqref{dim red eo}. We also use the short-hand notation $b_{0}$, $b_{1}$, and $b_{2}$ to denote $b_{(0,0)}$, $b_{(1,0)}$, and $b_{(1,1)}$ respectively. Then $\cH^{\mathrm{eo}}_{\bd}$ is a free $\bC[\hbar]$ module with a basis $\{b_0,b_1,b_2\}$.

Consider the specialization map \cite[Def.~7.14]{COZZ}: 
$$\mathrm{sp}\colon H^{\sT_0}(\cM(\bv,\bd),\tr \sW^{\mathrm{fr}})\to H^{\sT_0}(\cM(\bv,\bd)). $$ 
By \cite[Rem.~7.16]{COZZ}, $\mathrm{sp}(b_\bv)$ agrees with the pushforward of the fundamental class of $\{XY+AB=0\}/\!\!/_\theta \GL(\bv)$ into $\cM(\bv,\bd)$, that is,
\begin{align*}
\mathrm{sp}(b_0)=[\cM((0,0),\bd)],\quad \mathrm{sp}(b_1)=\hbar\,[\cM((1,0),\bd)],\quad \mathrm{sp}(b_2)=\hbar\,[\cM((1,1),\bd)].
\end{align*}
In particular, $\mathrm{sp}$ is injective.

\begin{Proposition}\label{prop gl(2|1) even-odd}
In the basis $\{b_i\otimes b_j\}_{0\leqslant i,j\leqslant 2}$, the $R^{\mathrm{sup}}(u)$ matrix for $\cH^{\mathrm{eo}}_{\bd}\otimes \cH^{\mathrm{eo}}_{\bd}$ is
\begin{align*}
R^{\mathrm{sup}}(u)=\frac{u+\hbar P^{\mathrm{eo}}}{u+\hbar}.
\end{align*}
Here $P^{\mathrm{eo}}$ is the super permutation matrix on $\bC^{2|1}\otimes \bC^{2|1}$, where $\bC^{2|1}=\Span_\bC\{b_0,b_1,b_2\}$ with $b_0$, $b_1$ even and $b_2$ odd.
\end{Proposition}

\begin{proof}
Since specializations commute with stable envelopes by \cite[Prop.~7.18]{COZZ}, the stable envelopes 
$$\Stab_{\pm}\colon H^{\sT_0\times \sA}(\cM(\bv,2\bd)^{\sA},\tr\sW^{\mathrm{fr}})\to H^{\sT_0\times \sA}(\cM(\bv,2\bd),\tr\sW^{\mathrm{fr}})$$ agree with the restriction of their zero potential counterparts $\Stab_{\pm}\colon H^{\sT_0\times \sA}(\cM(\bv,2\bd)^{\sA})\to H^{\sT_0\times \sA}(\cM(\bv,2\bd))$ to the image of specialization map. To compute the stable envelopes and $R^{\mathrm{sup}}(u)$ matrices, we have the following observation: if $\bv_2=0$, then the quiver variety $\cM(\bv,2\bd)$ is a special case in \S \ref{sect on comp stab of gl2}, we can directly apply the results there and get
\begin{align*}
R^{\mathrm{sup}}(u)(b_i\otimes b_j)=\frac{1}{u+\hbar}\left(u\,b_i\otimes b_j+\hbar \,b_j\otimes b_i \right),\;\text{ for }\; 0\leqslant i,j\leqslant 1.
\end{align*}
As $\cM((2,2),2\bd)$ has only one fixed component $\cM((1,1),\bd)\times \cM((1,1),\bd)$, so the $R^{\mathrm{sup}}(u)$ matrix in the $\bv=(2,2)$ case is a scalar. It is easy to compute that
\begin{align*}
R^{\mathrm{sup}}(u)(b_2\otimes b_2)=\frac{u-\hbar}{u+\hbar}\, b_2\otimes b_2\,.
\end{align*}
Note that $\cM((1,1),2\bd)\cong \mathrm{Tot}_{\bP^1}(\mathcal O(-1)^{\oplus 2})\times \bC^2$, where the $\sA$ action on $\bC^2$ is trivial, and $\mathrm{Tot}_{\bP^1}(\mathcal O(-1)^{\oplus 2})$ is the quiver variety with gauge rank $1$ defined at the beginning of \S \ref{sec fund R-mat gl(1|1)}. We can directly apply the results there and get
\begin{align*}
R^{\mathrm{sup}}(u)(b_i\otimes b_j)=\frac{1}{u+\hbar}\left(u\,b_i\otimes b_j+\hbar \,b_j\otimes b_i \right),\;\text{ for $(i,j)=(0,2)$ or $(2,0)$}\; .
\end{align*}
Let $X=\cM((2,1),2\bd)$ and $X_{\mathrm{iso}}$ be the open subset of $X$ on which $A$ is an isomorphism. Then $X^\sA\subset X_{\mathrm{iso}}$, and we have $X_{\mathrm{iso}}\cong \mathrm{Tot}_{\bP^1}(\mathcal O(-1)^{\oplus 2})\times \End(\bC^2)_\hbar\times \End(\bC^2)_{-\hbar}$, where the $\bC^2$'s are $\sA$-equivariantly identified with the framing vector space, the subscripts $\hbar$ and $-\hbar$ stand for the $\bC^*_\hbar$ weights (on the whole endomorphism space). Denote $\overline{\Attr}_{\pm}(\Delta_{X^\sA})_{Y}$ the closure of attracting set of $\Delta_{X^\sA}$ in $Y$ for $Y=X_{\mathrm{iso}}\times X^\sA$ or $X\times X^\sA$. Note that $$\overline{\Attr}_{\pm}(\Delta_{X^\sA})_{X_{\mathrm{iso}}\times X^\sA}=\overline{\Attr}_{\pm}(\Delta_{X^\sA})_{X\times X^\sA}\,\bigcap\, (X_{\mathrm{iso}}\times X^\sA).$$
So the composition $H^{\sT_0}(X^\sA)\xrightarrow{\Stab_{\pm}} H^{\sT_0}(X)\xrightarrow{\text{restriction}} H^{\sT_0}(X_{\mathrm{iso}})$ is equal to the stable envelope for $X_{\mathrm{iso}}$. Then $R^{\mathrm{sup}}(u)$ for $X$ is equal to the product of $R^{\mathrm{sup}}(u)$ for $\mathrm{Tot}_{\bP^1}(\mathcal O(-1)^{\oplus 2})$, $\End(\bC^2)_\hbar$, and $\End(\bC^2)_{-\hbar}$:
\begin{align*}
R^{\mathrm{sup}}(u)=\frac{1}{u+\hbar}\begin{pmatrix}
u & \hbar\\
\hbar & u
\end{pmatrix}\times \frac{u-\hbar}{u+\hbar}\times \frac{u+\hbar}{u-\hbar}=\frac{1}{u+\hbar}\begin{pmatrix}
u & \hbar\\
\hbar & u
\end{pmatrix},
\end{align*}
and we get 
\begin{equation*}
R^{\mathrm{sup}}(u)(b_i\otimes b_j)=\frac{1}{u+\hbar}\left(u\,b_i\otimes b_j+\hbar \,b_j\otimes b_i \right),\;\text{ for $(i,j)=(1,2)$ or $(2,1)$}\; . \qedhere
\end{equation*}
\end{proof}

\begin{Remark}
One can also use the explicit formula in \cite[Cor.~4.25]{COZZ} to compute the stable envelopes and $R^{\mathrm{sup}}(u)$.
\end{Remark}

% Let $x_1,x_2$ be the equivariant variables for the first gauge group $\GL(2)$, and let $y$ be the equivariant variable for the second gauge group $\GL(1)$. Using the explicit formula in \cite[Cor.~4.25]{COZZ}, we get
% \begin{align*}
% \Stab_+([\cM((1,0),\bd)\times \cM((1,1),\bd)])&=\frac{(\hbar-x_1)(x_2-u)(x_2-x_1-\hbar)(y-x_1)}{x_2-x_1}[\cM((2,1),2\bd)]+\{x_2\leftrightarrow x_1\}\,,\\
% \Stab_+([\cM((1,1),\bd)\times \cM((1,0),\bd)])&=\frac{(\hbar-x_1)(x_2-u)(x_2-x_1-\hbar)(x_2-y+\hbar)}{x_2-x_1}[\cM((2,1),2\bd)]+\{x_2\leftrightarrow x_1\}\,.
% \end{align*}
% In the basis $\{b_1\otimes b_2,b_2\otimes b_1\}$, the matrix form of $\Stab_+$ is
% \begin{align*}
% (\hbar^2-u^2)\begin{pmatrix}
% u & -\hbar \\
% 0 & u-\hbar
% \end{pmatrix}.
% \end{align*}
% A similar computation gives the matrix form of $\Stab_-$:
% \begin{align*}
% (\hbar^2-u^2)\begin{pmatrix}
% u & -\hbar \\
% 0 & u-\hbar
% \end{pmatrix}.
% \end{align*}

\subsection{The odd-even case} 
Consider the following framed quiver with potential:
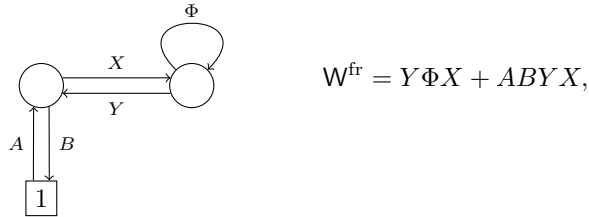
\begin{equation*}
\begin{tikzpicture}[x={(1cm,0cm)}, y={(0cm,1cm)}, baseline=0cm]
  % Nodes
  \node[draw,circle,fill=white] (Gauge1) at (0,0){$\phantom{n}$}; %{$\bv$}  {$\phantom{n}$};
  \node[draw,circle,fill=white] (Gauge2) at (2,0) {$\phantom{n}$};
  \node (Z) at (2,1) {\scriptsize $\Phi$};
  \node[draw,rectangle,fill=white] (FramingZ) at (0,-1.5) {$1$};
  % Edges
  \draw[<-] (Gauge1.340) -- (Gauge2.200) node[midway,below] {\scriptsize $Y$};
  \draw[->] (Gauge1.20) -- (Gauge2.160) node[midway,above] {\scriptsize $X$};
  \draw[<-] (Gauge1.250) -- (FramingZ.114) node[midway,left] {\scriptsize $A$};
  \draw[->] (Gauge1.290) -- (FramingZ.66) node[midway,right] {\scriptsize $B$};
  % Loop
  \draw[->,looseness=7] (Gauge2.135) to[out=135,in=45] (Gauge2.45);
\end{tikzpicture}
\qquad\sW^{\mathrm{fr}}=Y\Phi X+ABYX,
\end{equation*}
and the $\sT_0$-equivariant critical cohomology
$$\cH^{\mathrm{oe}}_{\bd}=\bigoplus_{\bv\in \bZ^{2}_{\geqslant0}}\cH^{\mathrm{oe}}_{\bd}(\bv), \quad 
\cH^{\mathrm{oe}}_{\bd}(\bv)=H^{\sT_0}(\cM(\bv,\bd),\tr \sW^{\mathrm{fr}}). $$ 
Here we let $\sT_0=\bC^*_\hbar$ act on $A$ and $B$ with weights $0$ and $-1$ respectively. 

Solving the equation $d\tr\sW^{\mathrm{fr}}=0$ under the cyclic stability condition, we get
\begin{align*}
\Crit_{\cM(\bv,\bd)}(\tr\sW^{\mathrm{fr}})\neq \emptyset \;\Longleftrightarrow\; \text{$\bv=(0,0)$, $(1,0)$, or $(1,1)$}.
\end{align*}
When $\Crit_{\cM(\bv,\bd)}(\tr\sW^{\mathrm{fr}})$ is nonempty, dimension reduction along $\Phi$ and $B$ yields isomorphism:
\begin{align}\label{dim red oe}
\cH^{\mathrm{oe}}_{\bd}(\bv)\cong H^{\sT_0}(\cM'(\bv,\bd)),\;\text{ where }\; \cM'(\bv,\bd)=\{B=0,\Phi=0,Y=0\}/\!\!/_\theta\GL(\bv).
\end{align}
Let $b_{\bv}$ be the image of the fundamental class $[\cM'(\bv,\bd)]$ in $\cH^{\mathrm{oe}}_{\bd}(\bv)$ via the isomorphism \eqref{dim red oe}. We also use the short-hand notation $b_{0}$, $b_{1}$, and $b_{2}$ to denote $b_{(0,0)}$, $b_{(1,0)}$, and $b_{(1,1)}$ respectively. Then $\cH^{\mathrm{oe}}_{\bd}$ is a free $\bC[\hbar]$ module with a basis $\{b_0,b_1,b_2\}$.

Similarly, we have a specialization map 
$$\mathrm{sp}\colon H^{\sT_0}(\cM(\bv,\bd),\tr \sW^{\mathrm{fr}})\to H^{\sT_0}(\cM(\bv,\bd)), $$ 
%By \cite[Rem.~7.16]{COZZ}, the specialization of $b_\bv$ agrees with the pushforward of the fundamental class of $\{Y=0\}/\!\!/_\theta \GL(\bv)$ into $\cM(\bv,\bd)$, that is,
which satisfies
\begin{align*}
\mathrm{sp}(b_0)=[\cM((0,0),\bd)],\quad \mathrm{sp}(b_1)=[\cM((1,0),\bd)],\quad \mathrm{sp}(b_2)=\hbar\,[\cM((1,1),\bd)].
\end{align*}
In particular, $\mathrm{sp}$ is injective.
As in Proposition \ref{prop gl(2|1) even-odd}, we have:
\begin{Proposition}\label{prop gl(2|1) odd-even}
In the basis $\{b_i\otimes b_j\}_{0\leqslant i,j\leqslant 2}$, the $R^{\mathrm{sup}}(u)$ matrix for $\cH^{\mathrm{oe}}_{\bd}\otimes \cH^{\mathrm{oe}}_{\bd}$ is
\begin{align*}
R^{\mathrm{sup}}(u)=\frac{u-\hbar P^{\mathrm{oe}}}{u-\hbar}.
\end{align*}
Here $P^{\mathrm{oe}}$ is the super permutation matrix on $\bC^{1|2}\otimes \bC^{1|2}$, where $\bC^{1|2}=\Span_\bC\{b_0,b_1,b_2\}$ with $b_0$ even and $b_1$, $b_2$ odd.
\end{Proposition}

\subsection{The odd-odd case} Consider the following framed quiver with potential:
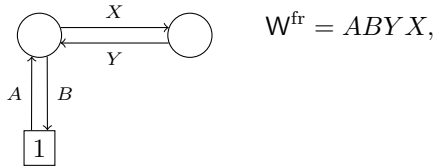
\begin{equation*}
\begin{tikzpicture}[x={(1cm,0cm)}, y={(0cm,1cm)}, baseline=0cm]
  % Nodes
  \node[draw,circle,fill=white] (Gauge1) at (0,0){$\phantom{n}$}; %{$\bv$}  {$\phantom{n}$};
  \node[draw,circle,fill=white] (Gauge2) at (2,0) {$\phantom{n}$};
  \node[draw,rectangle,fill=white] (FramingZ) at (0,-1.5) {$1$};
  % Edges
  \draw[<-] (Gauge1.340) -- (Gauge2.200) node[midway,below] {\scriptsize $Y$};
  \draw[->] (Gauge1.20) -- (Gauge2.160) node[midway,above] {\scriptsize $X$};
  \draw[<-] (Gauge1.250) -- (FramingZ.114) node[midway,left] {\scriptsize $A$};
  \draw[->] (Gauge1.290) -- (FramingZ.66) node[midway,right] {\scriptsize $B$};
  % Loop
\end{tikzpicture}
\qquad\sW^{\mathrm{fr}}=ABYX,
\end{equation*}
and the $\sT_0$-equivariant critical cohomology: 
$$\cH^{\mathrm{oo}}_{\bd}=\bigoplus_{\bv\in \bZ^{2}_{\geqslant0}}\cH^{\mathrm{oo}}_{\bd}(\bv), \quad \cH^{\mathrm{oo}}_{\bd}(\bv)=H^{\sT_0}(\cM(\bv,\bd),\tr \sW^{\mathrm{fr}}).$$ Here we let $\sT_0=\bC^*_\hbar$ act on $A$ and $B$ with weights $0$ and $-1$ respectively.

Solving the equation $d\tr\sW^{\mathrm{fr}}=0$ under the cyclic stability condition, we get
\begin{align*}
\Crit_{\cM(\bv,\bd)}(\tr\sW^{\mathrm{fr}})\neq \emptyset \;\Longleftrightarrow\; \text{$\bv=(0,0)$, $(1,0)$, or $(1,1)$}.
\end{align*}
When $\Crit_{\cM(\bv,\bd)}(\tr\sW^{\mathrm{fr}})$ is nonempty, dimension reduction along $B$ yields isomorphism:
\begin{align}\label{dim red oo}
\cH^{\mathrm{oo}}_{\bd}(\bv)\cong H^{\sT_0}(\cM'(\bv,\bd)),\;\text{ where }\; \cM'(\bv,\bd)=\{B=0,Y=0\}/\!\!/_\theta\GL(\bv).
\end{align}
Let $b_{\bv}$ be the image of the fundamental class $[\cM'(\bv,\bd)]$ in $\cH^{\mathrm{oo}}_{\bd}(\bv)$ via the isomorphism \eqref{dim red oo}. We also use the short-hand notation $b_{0}$, $b_{1}$, and $b_{2}$ to denote $b_{(0,0)}$, $b_{(1,0)}$, and $b_{(1,1)}$ respectively. Then $\cH^{\mathrm{oo}}_{\bd}$ is a free $\bC[\hbar]$ module with a basis $\{b_0,b_1,b_2\}$.

As before, we have a specialization map which satisfies: 
%Let $\mathrm{sp}\colon H^{\sT_0}(\cM(\bv,\bd),\tr \sW^{\mathrm{fr}})\to H^{\sT_0}(\cM(\bv,\bd))$ be the specialization map \cite[Def.~7.14]{COZZ}. By \cite[Rem.~7.16]{COZZ}, the specialization of $b_\bv$ agrees with the pushforward of the fundamental class of $\{Y=0\}/\!\!/_\theta \GL(\bv)$ into $\cM(\bv,\bd)$, that is,
\begin{align*}
\mathrm{sp}(b_0)=[\cM((0,0),\bd)],\quad \mathrm{sp}(b_1)=[\cM((1,0),\bd)],\quad \mathrm{sp}(b_2)=\hbar\,[\cM((1,1),\bd)].
\end{align*}
In particular, $\mathrm{sp}$ is injective.
As in Proposition \ref{prop gl(2|1) even-odd}, we have:
\begin{Proposition}\label{prop gl(2|1) odd-odd}
In the basis $\{b_i\otimes b_j\}_{0\leqslant i,j\leqslant 2}$, the $R^{\mathrm{sup}}(u)$ matrix for $\cH^{\mathrm{oo}}_{\bd}\otimes \cH^{\mathrm{oo}}_{\bd}$ is
\begin{align*}
R^{\mathrm{sup}}(u)=\frac{u-\hbar P^{\mathrm{oo}}}{u-\hbar}.
\end{align*}
Here $P^{\mathrm{oo}}$ is the super permutation matrix on $\bC^{2|1}\otimes \bC^{2|1}$, where $\bC^{2|1}=\Span_\bC\{b_0,b_1,b_2\}$ with $b_0$, $b_2$ even and $b_1$ odd.
\end{Proposition}

% \section{Example: handsaw/chainsaw quiver with equal framing dimension}

\end{document}